\documentclass[a4paper,11pt,leqno,twoside]{article}%
\usepackage[top=2.5cm, bottom=2.5cm, left=2.5cm, right=2.5cm]{geometry}	
\usepackage{graphicx}        
\usepackage[dvipsnames]{xcolor}
\usepackage[utf8]{inputenc}
\usepackage{longtable}
\usepackage{amsmath}
\usepackage{amssymb}
\usepackage{amsthm}
\usepackage{mathrsfs}
\usepackage{mathtools}  
\usepackage[font=small,labelfont=bf]{caption}
\usepackage{mdwlist} 
\usepackage{imakeidx}
\usepackage{hyperref} 
\hypersetup{
colorlinks=true,
linkcolor=Blue,
citecolor=Blue
}
\usepackage{enumitem}
\usepackage{chngcntr}
\usepackage{tocloft}
\usepackage{bm} 
\usepackage{algorithmic,algorithm}
\usepackage[hang,flushmargin]{footmisc} 
\usepackage[toctitles]{titlesec}
\usepackage{tocvsec2}

\usepackage{datetime} 

\newdateformat{monthyeardate}{\monthname[\THEMONTH] \THEYEAR}

\usepackage{titleps}

\newpagestyle{emptyjk}{}
\newpagestyle{preface}{\setfoot[\thepage][][] 
				   {}{}{\thepage}}
				   
\newpagestyle{premain}{\sethead[][][\sectiontitle]  
				   {\sectiontitle}{}{}
				   \setfoot[\thepage][][] 
				   {}{}{\thepage}}

\newpagestyle{FUN}{\sethead[Sec.~0][][\sectiontitle]  
				   {\sectiontitle}{}{Sec.~0}
				   \setfoot[\thepage][][] 
				   {}{}{\thepage}}
				   
\newpagestyle{main}{\sethead[Sec.~\thesubsection][][\sectiontitle]  
				   {\MakeUppercase{\subsectiontitle}}{}{Sec.~\thesubsection}			                  
				   \setfoot[\thepage][][]  
				   {}{}{\thepage}}

\usepackage{etoolbox}
\AtBeginEnvironment{thebibliography}{\linespread{1}\selectfont}

\usepackage{chngcntr}
\counterwithout{subsection}{section}
\numberwithin{equation}{subsection}

\renewcommand\theequation{\maybe{\arabic{subsection}}\arabic{equation}}
\DeclareRobustCommand\maybe[1]{\ifnum#1=\value{subsection}\relax\else#1.\fi}

\usepackage[round]{natbib}

\usepackage[nopostdot,style=long]{glossaries}
 {\begin{longtable}[l]{lp{0.75\linewidth}}}
 {\end{longtable}}

\indexsetup{level=\section*,toclevel=section,noclearpage,firstpagestyle=premain}
\makeindex

\titleformat{\subsection}[runin]{\large\bfseries}{\thesubsection}{0em}{. }[.]

\titleformat{\subsubsection}[runin]{\normalfont\bfseries}{\thesubsubsection}{0em}{}[.]

\titleformat{\paragraph}[runin]{\normalfont\bfseries}{\theparagraph}{0em}{}[.]

\newtheoremstyle{mytheoremstyle} 
    {\topsep}                    
    {\topsep}                    
    {\itshape}                   
    {}                           
    {}                   
    {.}                          
    {0.5em}                       
    {(#2) #1\thmnote{ (#3)}}    
    
\theoremstyle{mytheoremstyle}

\newtheorem{theorem}[equation]{THEOREM}
\newtheorem{example}[equation]{EXAMPLE}
\newtheorem{definition}[equation]{DEFINITION}
\newtheorem{corollary}[equation]{COROLLARY}
\newtheorem{lemma}[equation]{LEMMA}

\newtheorem{remark}[equation]{REMARK}

\newtheorem{exercise}[equation]{EXERCISE}
\newtheorem{proposition}[equation]{PROPOSITION}

\newcommand{\MyQuote}[1]{\vspace{1cm}\refstepcounter{equation}\noindent(\theequation)\hspace*{0.5cm}
     \parbox{\textwidth-3cm}{\em #1}\\[1cm]}

\renewcommand{\cal}[1]{\mathcal{#1}}
\newcommand{\bb}[1]{\mathbb{#1}}

\newglossary[glg]{gen}{gld}{gln}{}
\newglossary[blg]{both}{bld}{bln}{}
\newglossary[dlg]{dt}{dld}{dln}{}
\newglossary[clg]{ct}{cld}{cln}{}

\makeglossaries

\newcommand{\C}{\mathcal{C}}

\newglossaryentry{vee}
{
  name={$a\vee b$},
  description={maximum of two real numbers $a$ and $b$},
  sort=c,
  type=gen
}

\newglossaryentry{wedge}
{
  name={$a\wedge b$},
  description={minimum of two real numbers $a$ and $b$},
  sort=c,
  type=gen
}

  \newglossaryentry{floor}
{
  name={$\lfloor a\rfloor$},
  description={largest integer no greater than $a$},
  sort=c,
  type=gen
}
  
  \newglossaryentry{ceil}
{
  name={$\lceil a\rceil$},
  description={smallest integer no lesser than $a$},
  sort=c,
  type=gen
}

\newglossaryentry{1}
{
  name={$=$},
  description={the left-hand side equals the right-hand side},
  sort=1,
  type=gen
}

  \newglossaryentry{3}
{
  name={$:=$},
  description={the left-hand side is defined to be the right-hand side},
  sort=1,
  type=gen
}
  
  \newglossaryentry{4}
{
  name={$=:$},
  description={the right-hand side is defined to be the left-hand side},
  sort=1,
  type=gen
}

\renewcommand{\r}{\mathbb{R}}
\newglossaryentry{r}
{
  name={$\mathbb{R}$},
  description={set of real numbers},
  sort=a,
  type=gen
}
\newglossaryentry{re}
{
  name={$\mathbb{R}_E$},
  description={set of extended real numbers},
  sort=b,
  type=gen
}
\newglossaryentry{rp}
{
  name={$[0,\infty)$},
  description={set of non-negative real numbers},
  sort=a,
  type=gen
}
\newglossaryentry{rpe}
{
  name={$[0,\infty]$},
  description={set of extended non-negative real numbers},
  sort=b,
  type=gen
}

\newglossaryentry{rpp}
{
  name={$(0,\infty)$},
  description={set of positive real numbers},
  sort=a,
  type=gen
}
\newglossaryentry{rppe}
{
  name={$(0,\infty]$},
  description={set of extended positive real numbers},
  sort=b,
  type=gen
}

\newcommand{\n}{\mathbb{N}}
\newglossaryentry{n}
{
  name={$\mathbb{N}$},
  description={set of natural numbers},
  sort=a,
  type = gen
}
  
\newglossaryentry{ne}
{
  name={$\mathbb{N}_E$},
  description={set of extended natural numbers},
  sort=b,
  type = gen
}
\newcommand{\z}{\mathbb{Z}}
\newglossaryentry{z}
{
  name={$\z$},
  description={set of integers},
  sort=a,
  type=gen
}
\newglossaryentry{ze}
{
  name={$\z_E$},
  description={set of extended integers},
  sort=b,
  type=gen
}

\newcommand{\zp}{\mathbb{Z}_+}
\newglossaryentry{zp}
{
  name={$\zp$},
  description={set of positive integers},
  sort=a,
  type=gen
}

\newglossaryentry{rat}
{
  name={$\mathbb{Q}$},
  description={set of rational numbers},
  sort=a,
  type=gen
}
  
\newglossaryentry{inda}
{
  name={$1_A$},
  description={indicator function of $A$},
  sort=x,
  type=gen
}
  
  \newglossaryentry{preimage}
{
  name={$\{f\in B\}$},
  description={pre-image of $B$ through $f$},
  sort=x,
  type=gen
}

\newcommand{\mmag}[1]{\left|#1\right|}
\newglossaryentry{2om}
{
  name={$2^A$},
  description={power set of $A$},
  sort=d,
  type=gen
}

\newglossaryentry{borel}
{
  name={$\cal{B}(A)$},
  description={Borel sigma-algebra on $A$},
  sort=d,
  type=gen
}

\newcommand{\supp}[1]{\operatorname{supp}\left(#1\right)}
\newglossaryentry{supp}
{
  name={$\supp{\cdot}$},
  description={support of $\cdot$},
  sort=z,
  type=gen
}

\newcommand{\norm}[1]{\left|\left|{#1}\right|\right|}

\newglossaryentry{normtv}
{
  name={$\norm{\cdot}$},
  description={total variation norm of $\cdot$},
  sort=z,
  type=gen  
}

\newcommand{\s}{\mathcal{S}}

\newglossaryentry{s}
{
  name={$\s$},
  description={countable state space},
  sort=a,
  type=both
}

\newcommand{\tws}{2^\mathcal{S}}

\newglossaryentry{gamma}
{
  name={$\gamma$},
  description={initial distribution $(\gamma(x))_{x\in\s}$},
  sort=a,
  type=both
}

\newglossaryentry{of}
{
  name={$(\Omega,\cal{F})$},
  description={underlying measurable space},
  sort=a,
  type=both
}
  
\newglossaryentry{Pl}
{
  name={$\mathbb{P}_\gamma$},
  description={probability measure describing the chain's statistics if its starting state is sampled from $\gamma$},
  sort=a,
  type=both
}
  
\newglossaryentry{El}
{
  name={$\mathbb{E}_\gamma$},
  description={expectation with respect to $\mathbb{P}_\lambda$},
  sort=a,
  type=both
}

\newglossaryentry{Px}
{
  name={$\mathbb{P}_x$},
  description={probability measure on describing the chain's statistics if its starting state is $x$},
  sort=a,
  type=both
}
  
  \newglossaryentry{Ex}
{
  name={$\mathbb{E}_x$},
  description={expectation with respect to $\mathbb{P}_x$},
  sort=a,
  type=both
}

  \newglossaryentry{Pspace}
{
  name={$\cal{P}$},
  description={path space},
  sort=b,
  type=both
}
    \newglossaryentry{E}
{
  name={$\cal{E}$},
  description={cylinder sigma-algebra on $\cal{P}$},
  sort=b,
  type=both
}

  \newglossaryentry{Lgam}
{
  name={$\mathbb{L}_\gamma$},
  description={path law if the chain's starting state is sampled from $\gamma$},
  sort=b,
  type=both
}

  \newglossaryentry{Lx}
{
  name={$\mathbb{L}_x$},
  description={path law if the chain's starting state is $x$},
  sort=b,
  type=both
}

\newglossaryentry{accessibility}
{
  name={$x\to y$},
  description={state $y$ is accessible from state $x$},
  sort=c,
  type=both
}
  \newglossaryentry{accessibility2}
{
  name={$A\to B$},
  description={set $B$ is accessible from set $A$},
  sort=c,
  type=both
}

  \newglossaryentry{D}
{
  name={$\cal{D}$},
  description={domain},
  sort=d,
  type=both
}

  \newglossaryentry{mu}
{
  name={$\mu$},
  description={stopping/exit distribution},
  sort=d,
  type=both
}
  
    \newglossaryentry{nu}
{
  name={$\nu$},
  description={occupation measure},
  sort=d,
  type=both
}

\newglossaryentry{mus}
{
  name={$\mu_S$},
  description={space marginal of stopping/exit distribution},
  sort=d,
  type=both
}
  
\newglossaryentry{nus}
{
  name={$\nu_S$},
  description={space marginal of occupation measure},
  sort=d,
  type=both
}

\newglossaryentry{rec}
{
  name={$\cal{R}$},
  description={set of recurrent states},
  sort=f,
  type=both
}

\newglossaryentry{recp}
{
  name={$\cal{R}_+$},
  description={set of positive recurrent states},
  sort=f,
  type=both
}

\newglossaryentry{nrec}
{
  name={$\cal{T}$},
  description={set of states not positive recurrent},
  sort=f,
  type=both
}
  
\newglossaryentry{ci}
{
  name={$\cal{C}_i$},
  description={$i$th positive recurrent closed communicating class},
  sort=f,
  type=both
}

\newglossaryentry{pi}
{
  name={$\pi$},
  description={stationary distribution},
  sort=g,
  type=both
}
  
\newglossaryentry{pii}
{
  name={$\pi_i$},
  description={ergodic distribution associated with $\cal{C}_i$},
  sort=g,
  type=both
}

\newglossaryentry{pigamma}
{
  name={$\pi_\gamma$},
  description={limit of the time-varying law},
  sort=g,
  type=both
}
\newglossaryentry{epinf}
{
  name={$\epsilon_\infty$},
  description={limit of the empirical distribution},
  sort=g,
  type=both
}

\newglossaryentry{P}
{
  name={$P$},
  description={one-step matrix $(p(x,y))_{x,y\in\s}$},
  sort=a,
  type=dt
}

\newglossaryentry{Un}
{
  name={$(U_n)_{n\in\zp}$},
  description={i.i.d. sequence of $(0,1)$-valued uniform random variables used in the chain's definition},
  sort=b,
  type=dt
}
  
\newglossaryentry{Xn}
{
  name={$X$},
  description={discrete-time chain $(X_n)_{n\in\n}$},
  sort=b,
  type=dt
}

  \newglossaryentry{filtrationdt}
{
  name={$(\cal{F}_n)_{n\in\n}$},
  description={filtration generated by the chain},
  sort=c,
  type=dt
}

      \newglossaryentry{pn}
{
  name={$p_n$},
  description={time-varying law $(p_n(x))_{x\in\s}$},
  sort=d,
  type=dt
}

\newglossaryentry{Pn}
{
  name={$P_n$},
  description={$n$-step matrix $(p_n(x,y))_{x,y\in\s}$},
  sort=d,
  type=dt
}
  
\newglossaryentry{epn}
{
  name={$\epsilon_N$},
  description={empirical distribution $(\epsilon_N(x))_{x\in\s}$},
  sort=d,
  type=dt
}

\newglossaryentry{varsigma}
{
  name={$\varsigma$},
  description={$(\cal{F}_n)_{n\in\n}$-stopping time},
  sort=c,
  type=dt
}

\newglossaryentry{Fsig}
{
  name={$\cal{F}_\varsigma$},
  description={pre-$\varsigma$ sigma-algebra},
  sort=c,
  type=dt
}

\newglossaryentry{Xvarsig}
{
  name={$X^\varsigma$},
  description={$\varsigma$-shifted chain},
  sort=c,
  type=dt
}

\newglossaryentry{sigmaa}
{
  name={$\sigma_A$ ($\sigma_x$)},
  description={hitting time of a set $A$ (resp.~state $x$)},
  sort=f,
  type=dt
}

\newglossaryentry{sigma}
{
  name={$\sigma$},
  description={exit time from the domain $\cal{D}$},
  sort=f,
  type=dt
}

\newglossaryentry{sigmar}
{
  name={$\sigma_r$},
  description={exit time from the truncation $\s_r$},
  sort=f,
  type=dt
}

\newglossaryentry{phiA}
{
  name={$\phi_A$ ($\phi_x$)},
  description={first entrance time to a set $A$ (resp. state $x$)},
  sort=f,
  type=dt
}

\newglossaryentry{phiAk}
{
  name={$\phi_A^k$ ($\phi_x^k$)},
  description={$k$th entrance time to a set $A$ (resp. state $x$)},
  sort=f,
  type=dt
}

\newglossaryentry{Q}
{
  name={$Q$},
  description={rate matrix $(q(x,y))_{x,y\in\s}$ of a continuous-time chain},
  sort=a,
  type=ct
}

\newglossaryentry{Pct}
{
  name={$P$},
  description={jump matrix $(p(x,y))_{x,y\in\s}$},
  sort=a,
  type=ct
}
\newglossaryentry{lambda}
{
  name={$\lambda$},
  description={jump rates $(\lambda(x))_{x\in\s}$},
  sort=a,
  type=ct
} 
  
\newglossaryentry{Unct}
{
  name={$(U_n)_{n\in\zp}$},
  description={i.i.d. sequence of $(0,1)$-valued uniform random variables used in the chain's definition},
  sort=b,
  type=ct
}
  
\newglossaryentry{xi}
{
  name={$(\xi_n)_{n\in\zp}$},
  description={i.i.d. sequence of unit-mean exponential random variables used in the chain's definition},
  sort=b,
  type=ct
}
  
\newglossaryentry{Xt}
{
  name={$X$},
  description={continuous-time chain $(X_t)_{t\geq0}$},
  sort=b,
  type=ct
}

\newglossaryentry{Y}
{
  name={$Y$},
  description={jump chain $(Y_n)_{n\in\n}$},
  sort=b,
  type=ct
}

\newglossaryentry{Sn}
{
  name={$S_n$},
  description={$n$th waiting time},
  sort=b,
  type=ct
}

  \newglossaryentry{Tinf}
{
  name={$T_\infty$},
  description={explosion time},
  sort=b,
  type=ct
}

\newglossaryentry{Tn}
{
  name={$T_n$},
  description={$n$th jump time},
  sort=b,
  type=ct
}

\newglossaryentry{filtrationct}
{
  name={$(\cal{F}_t)_{t\geq0}$},
  description={filtration generated by the chain},
  sort=c,
  type=ct
}

\newglossaryentry{eta}
{
  name={$\eta$},
  description={$(\cal{F}_t)_{t\geq0}$-stopping time a continuous-time chain},
 sort=c,
  type=ct
}

\newglossaryentry{Feta}
{
  name={$\cal{F}_\eta$},
  description={pre-$\eta$ sigma-algebra},
  sort=c,
  type=ct
}

\newglossaryentry{Xeta}
{
  name={$X^\eta$},
  description={$\eta$-shifted chain},
  sort=c,
  type=ct
}

\newglossaryentry{pt}
{
  name={$p_t$},
  description={time-varying law $(p_t(x))_{x\in\s}$},
  sort=d,
  type=ct
}

\newglossaryentry{Pt}
{
  name={$P_t$},
  description={$t$-transition matrix $(p_t(x,y))_{x\in\s}$},
  sort=d,
  type=ct
}

\newglossaryentry{ept}
{
  name={$\epsilon_T$},
  description={empirical distribution $(\epsilon_T(x))_{x\in\s}$},
  sort=d,
  type=ct
}

\newglossaryentry{taua}
{
  name={$\tau_A$ ($\tau_x$)},
  description={hitting time of set a $A$ (resp.~state $x$)},
  sort=e,
  type=ct
}
  
  \newglossaryentry{tau}
{
  name={$\tau$},
  description={exit time from the domain $\cal{D}$},
 sort=e,
  type=ct
} 

  \newglossaryentry{taur}
{
  name={$\tau_r$},
  description={exit time from the truncation $\s_r$ of a continuous-time chain},
 sort=e,
  type=ct
} 

\newglossaryentry{varphiA}
{
  name={$\varphi_A$ ($\varphi_y$)},
  description={first entrance time to a set $A$ (resp. state $x$)},
  sort=e,
  type=ct
}
\newglossaryentry{varphiAk}
{
  name={$\varphi_A^k$ ($\varphi_y^k$)},
  description={$k$th entrance time to a set $A$ (resp. state $x$)},
  sort=e,
  type=ct
}

  \newglossaryentry{filtrationdtct}
{
  name={$(\cal{G}_n)_{n\in\n}$},
  description={filtration generated by the jump times and jump chain},
  sort=f,
  type=ct
}

\newglossaryentry{varsigmact}
{
  name={$\varsigma$},
  description={$(\cal{G}_n)_{n\in\n}$-stopping time},
  sort=f,
  type=ct
}
  
\newglossaryentry{sigmact}
{
  name={$\sigma$},
  description={exit time from the domain $\cal{D}$ of the jump chain $Y$},
  sort=f,
  type=ct
}
  
\newglossaryentry{sigmaact}
{
  name={$\sigma_A$ ($\sigma_x$)},
  description={hitting time of a set $A$ (resp.~state $x$) of the jump chain $Y$},
  sort=f,
  type=ct
}

\newglossaryentry{sigmarct}
{
  name={$\sigma_r$},
  description={exit time from the truncation $\s_r$ of the jump chain $Y$},
  sort=f,
  type=ct
}

\newglossaryentry{phiAct}
{
  name={$\phi_A$ ($\phi_x$)},
  description={first entrance time to a set $A$ (resp. state $x$) of the jump chain $Y$},
  sort=f,
  type=ct
}

\newglossaryentry{phiAkct}
{
  name={$\phi_A^k$ ($\phi_x^k$)},
  description={$k$th entrance time to a set $A$ (resp. state $x$) of the jump chain $Y$},
  sort=f,
  type=ct
}

\newglossaryentry{skeleton}
{
  name={$X^\delta$},
  description={skeleton chain $(X^\delta_n)_{n\in\n}$},
  sort=g,
  type=ct
}

\newglossaryentry{sy:space1}{name=$\qquad\qquad$, description={}, sort=zzzzzz, type=gen}
\newglossaryentry{sy:space2}{name=$\qquad\qquad$, description={}, sort=zzzzzz, type=both}
\newglossaryentry{sy:space3}{name=$\qquad\qquad$, description={}, sort=zzzzzz, type=dt}
\newglossaryentry{sy:space4}{name=$\qquad\qquad$, description={}, sort=zzzzzz, type=ct}

\newcommand{\Pb}{\mathbb{P}}
\newcommand{\Eb}{\mathbb{E}}

\newcommand{\Pbb}[1]{\mathbb{P}\left(#1\right)}
\newcommand{\Ebb}[1]{\mathbb{E}\left[#1\right]}

\newcommand{\Pbx}[1]{\mathbb{P}_x\left(#1\right)}
\newcommand{\Ebx}[1]{\mathbb{E}_x\left[#1\right]}

\newcommand{\Pby}[1]{\mathbb{P}_y\left(#1\right)}
\newcommand{\Eby}[1]{\mathbb{E}_y\left[#1\right]}

\newcommand{\Pbz}[1]{\mathbb{P}_z\left(#1\right)}
\newcommand{\Ebz}[1]{\mathbb{E}_z\left[#1\right]}

\newcommand{\Pbl}[1]{\mathbb{P}_\gamma\left(#1\right)}
\newcommand{\Ebl}[1]{\mathbb{E}_\gamma\left[#1\right]}

\newcommand{\Pbp}[1]{\mathbb{P}_\pi\left(#1\right)}
\newcommand{\Ebp}[1]{\mathbb{E}_\pi\left[#1\right]}
\newcommand{\Lbp}[1]{\mathbb{L}_\pi\left(#1\right)}

\newcommand{\q}{\mathbb{Q}}

\renewcommand{\a}{\mathcal{A}}

\newif\ifdraft
\draftfalse

\begin{document}\pagenumbering{roman}

\pagestyle{emptyjk}
~
\vspace{25pt}
\begin{center}\textbf{\LARGE{Markov chains revisited}}\end{center}\vspace{-10pt}

\begin{center}\LARGE{Juan Kuntz}\end{center}



\newpage
\pagestyle{preface}
\addtocontents{toc}{\protect\thispagestyle{preface}} 

%
%

\vspace*{\fill}
\begin{center}\large{Copyright$\enskip$\copyright$\enskip$2020 Juan Kuntz Nussio}\end{center}
\vspace*{\fill}

\newpage

\vspace{200pt}
%


\vspace{160pt}

\newpage

\section*{Preface --- please read!}
\addcontentsline{toc}{section}{\protect\numberline{}Preface -- please read!}
%
%
%
%
\paragraph{What this book is intended to be} Although I've disguised it very well, 
 this book is actually written for those interested in the applications of  Markov chains (discrete-time and continuous-time processes that satisfy the Markov property and take values in countable sets). I have two main goals in writing it:
\begin{enumerate}
\item To equip graduate students and researchers with the tools and intuition necessary to overcome the theoretical headaches encountered in the modelling of real-world phenomena using Markov chains with infinite  state spaces.
%
%
My aim here is not so much to introduce novel results as to gather scattered ones and make the existing theory more accessible.
\begin{quotation}``Although only a special case of the strong Markov
property is involved in the culminating Theorem 8, a ``prodigious''
amount of preparation, it seems, enters into the proof of this intuitively
obvious result. One wonders if the present theory of stochastic process
is not still too difficult for applications.''\end{quotation}
\begin{quotation}\emph{Kai Lai Chung in \citep{Chung1967}}.\end{quotation}
%
%
\item To survey and consolidate the disparate numerical techniques available for the practical use of these type of models (this material is currently absent from the manuscript).
\end{enumerate}
I also have several other aims:
\begin{enumerate}
\item Help abridge the gap between undergraduate-level texts, such as Norris's fine book \citep{Norris1997}, and more advanced treatments of the type in \citep{Chung1967,Rogers2000a}. To this end, I introduce early on the path space and path law of a chain and give the  strong Markov property in its most applicable (albeit ugly) form which involves these concepts.
\item Give a comprehensive account of the long-term behaviour of chains 
and the associated Foster-Lyapunov criteria similar to that given in \citep{Meyn2009} for general discrete-time Markov processes (but in the much simplified and more accessible setting of countable state spaces). In particular, an account free of the irreducibility assumptions pervasive in the introductory Markov chain literature. 
These assumptions are generally not necessary and, unfortunately, can hinder the practical use of the results that feature them.
Not just because they limit the applicability of the stated result to the irreducible case, but also because, frustratingly,  arguing in practice that irreducible chains are indeed irreducible sometimes proves very challenging.


\item Highlight the edges of the theory where unresolved issues and open questions creep. 
\begin{quotation}
``The mistakes and unresolved difficulties of the past in mathematics have always been the opportunities of its future; and should analysis ever appear to be without flaw or blemish, its perfection might only be that of death. ''
\end{quotation}
\begin{quotation}
\emph{Eric~Temple~Bell in \citep{Bell1945}.}
\end{quotation}
\end{enumerate}

\paragraph{What this book currently is} A rigorous and largely self-contained account of (a) the bread-and-butter concepts and techniques in Markov chain theory and (b)  the long-term behaviour of chains. As much as possible, the treatment is probabilistic instead of analytical (I stay away from semigroup theory). Personally, I tend to find that the intuition lies with the former and not the latter.  As stated above, this manuscript is geared towards those interested in the use of Markov chains as models of real-life phenomena. For this reason, I focus on the type of chains most commonly encountered in practice (time-homogeneous, minimal, and right-continuous in the discrete topology) and choose a starting point very familiar to this audience: the (Kendall-) Gillespie Algorithm commonly used to simulate these chains. In order to keep the prerequisite knowledge and technical complications to a minimum, I take a `bare-bones' approach that keeps the focus on chains (instead of more general processes) to an almost pathological degree: I use  the `jump and hold' structure of chains extensively; almost no martingale theory; minimal coupling; no stochastic  calculus; and, even though regeneration and renewal (of course!) feature in the book, they do so exclusively in the context of chains. I have also taken some extra steps to avoid imposing certain assumptions encountered in other texts that sometimes prove to be stumbling blocks in practice (e.g., irreducibility of the state space, boundedness of test functions and of stopping times). 
\begin{quotation}``The traveller often has the choice between climbing a peak or using
a cable car.''\end{quotation} 
\begin{quotation}\emph{William Feller in \citep{Feller1971}.}\end{quotation}

To facilitate the cable car experience for those who do not wish to drag themselves through the proofs, I try to give the main results and ideas at the beginning of each section and leave the more technical  arguments until the end (some sections are more amenable to this than others). This said, to those with the luxury of time, I heartily recommend climbing.

\ifdraft
\paragraph{The types of Markov chains typically encountered in practice}F

I have several aims in writing this book:

\begin{enumerate}
\item\textbf{Find a more seamless way to dealing with explosions than expanding the state space with a ``dead'' or ``infinity'' state:} While there's nothing wrong with this approach, it always felt like a somewhat artificial and cumbersome\footnote{In this approach, we end up having to do things like extending any real-valued function $f$ on the state $\s$ to the \emph{extended state space} $\s\cup\{\Delta\}$ by setting $f(\Delta)=0$ (and similarly for rate matrices, initial distributions, one-step matrices, etc.).} way of dispatching with the issue of explosions. It can trip up inexperienced theoreticians (even some experienced ones as well) and, in this minimal case I consider where no returns from infinity are allowed, it is unnecessary.

\item
\item\textbf{Giving a thorough accounting of Foster-Lyapunov criteria:} In practice, ruling out unstable behaviours such as explosivity or transience and establishing stable ones such as positive recurrence (or equivalently, the existence of stationary distributions), is a difficult task. Similarly, when it comes to establishing finiteness of exit times and their means. If we scan over the results of the previous two sections, a pattern emerges: very few of these results are constructive in the sense that they do not give us hints on how to verify in practice that their premises hold for a chain of interest. Very often, we need to resort to Foster-Lyapunov criteria to test for these properties. These are sufficient (and sometimes necessary too) conditions for these properties that amount to finding a function that satisfies various inequalities. They are named jointly after Aleksandr Lyapunov who first introduced these types of conditions in his study of the solutions of ordinary differential equations and F. Gordon Foster who first ported them to a stochastic setting (in particular, that of discrete-time chains) in \citep{Foster1953}. These type of criteria are also known as \emph{drift conditions}. When it comes to this type of criteria, we like the works of Sean P. Meyn and Richard L. Tweedie \citep{Meyn1993b,Meyn2009} and Rafail Z. Khasminskii \citep{Khasminskii2012}. A secondary reason for our interest in Foster-Lyapunov functions is that they allow us to establish the integrability with respect to the various distributions introduced in the previous sections of unbounded test functions; something that will be important in Chapters \ref{mcmom} and \ref{sdemom} where we focus on the moments of these distributions. We now survey a few of these criteria that we will use in this thesis.  Most of the criteria we give are borrowed from elsewhere, however those related to the exit times of continuous-time chains we have had to derive ourselves.

\item\textbf{Delve into the methods used in practice to analyse these chains.}

\item\textbf{Make rigour more accessible.} Heavy use of jump times (Chung's words \emph{taming the continuum} come to mind).
\end{enumerate}

\fi

\paragraph{This is not a finished product! Read at your own peril} This book is  a work in progress to which I plan to add material over time. 
%
To facilitate its use in the interim, I will post updated drafts on arXiv. 
If you find any mistakes or typos in the latest draft posted, or have some general thoughts on the book, please let me know about them. I can be reached at
\begin{center}
juan.kuntz-nussio@warwick.ac.uk
\end{center} 
\paragraph{Prerequisites} The only prerequisites are a competence with the measure theoretic foundations of modern probability and a familiarity with conditional expectation, its properties, and discrete-time martingales; all to a level no further than that of \citep[Chapters~1--10]{Williams1991}. See the ``Some important notation and preliminaries'' section for details.

\paragraph{How to use this book}
This paragraph is for the non-expert reader, the expert reader of course will be able to judge this matter for him-or-herself! The manuscript, as it stands, consists of two parallel developments of the theory, the first for discrete-time chains and the second for continuous-time ones. Each is intended to be read linearly, although sections marked with an `*' are of little importance to later sections. Readers not interested in their contents can skip these sections without jeopardising their future comprehension. Readers whose interests lie entirely within the continuous-time case  may be tempted to skip Sections~\ref{sec:dtdef}--\ref{sec:flgeodt} dealing with the discrete-time one. I, however, would advise against this. Not only are the proofs for the continuous-time case often reliant on, or analogous to, those for the discrete-time one, but discrete-time chains are technically simpler and the underlying ideas are exactly the same (ignore, however, any mention of periodicity in the discrete-time treatment). To minimise the amount of repetition, I often give the intuitive discussion only for the discrete-time case (for instance, compare Sections~\ref{sec:dtmarkov}~and~\ref{sec:dtstrmarkov} with Section~\ref{sec:markovprop}).
%

\paragraph{Referencing within the book}Within individual sections, objects (equations, theorems, examples, etc.) are labelled sequentially with a single number. Outside a section, the section's number is prefixed to the labels of that section's objects. For instance, Equation~(7) in Section~1 would be referenced as ``(7)'' throughout Section~1 but as ``(1.7)'' throughout Sections 2, 3,$\dots$

\paragraph{On the book's format}Readers acquainted with David Williams's books might find the formatting of this manuscript suspiciously familiar and they would be right to do so. You know what they say about imitation and flattery.
%
%
%
%
\paragraph{Bland assurances and exercises} I, too, will occasionally adhere to the following convention (a convention I used to find despicable as a student, it appears that I may have lived long enough to become the villain):
\begin{quotation}``Convention dictates that It\^o's formula should be proved for $n=1$, the general case being left as an exercise, amid bland assurances that only the notation is any more difficult.''\end{quotation}\begin{quotation}
\emph{Chris Rogers and David Williams in \citep{Rogers2000b}}.
\end{quotation}
The first few applications in the book of a given result tend to be quite detailed while, as is perhaps inevitable, later applications are more cavalier or shamelessly left as an exercise for the reader. 
If you find yourself hopelessly lost in one of these bland assurances or offhand applications, 
you can always email me at the address given above and I'll try my best to get back to you. Who knows, you may have found a mistake (in which case I would be in your debt!). Indeed, in my experience, such applications and assurances tend to be the weakest points of books and articles.
\begin{quotation}
```Obvious' is the most dangerous word in mathematics.''
\end{quotation}
\begin{quotation}Eric Temple Bell in \citep{Bell1952}.
\end{quotation}


\paragraph{I claim no originality of the results presented in this book}
Even though I may have ironed out certain corners of the theory while writing this book, by and large everything I have (so far!) covered in the manuscript is well-known within the community. 
\begin{quotation}``Nothing of me is original. I am the combined effort of everybody I've ever known.''\end{quotation}
\begin{quotation}
\emph{Chuck Palahniuk in Invisible Monsters.}\end{quotation}
%
%
%
Markov chain theory is a classical and very mature subject; the literature on it is vast. 
 Many of the results I give in this manuscript can be found verbatim, or close to, elsewhere. 
Whenever   the history of a given result (concept, approach, etc.) seems clear to me, 
I append to the end of the section containing the result (concept, approach, etc.) a `Notes and references' blurb commenting on this history. However, regretfully, these are mostly lacking, at least for the time being. 
%
%
\begin{quotation}``We hope that we always have told the truth, but realize that it is seldom the whole truth. It is not our intention to give anyone less than his full measure of credit and we apologize in advance to anyone who may feel slighted.''
\end{quotation}
\begin{quotation}
\emph{Robert~M.~Blumenthal and Ronald~K.~Getoor in \citep{Blumenthal2007}.}
\end{quotation}

\paragraph{Acknowledgements}
There are many I should thank here, and I will in due course.
\begin{flushright} Juan Kuntz 

London, \monthyeardate\today\end{flushright}

\ifdraft
\paragraph{Analytical characterisations} In particular, we introduce the distributions relevant to Questions Q1--3
(the time-varying law, the exit distributions, and the stationary distributions, respectively) and we \emph{characterise} each of them \emph{analytically}. In other words, we take each of these objects that are defined \emph{probabilistically} in terms of the chain and the underlying probability measure and we derive an equivalent description of the object in terms of the solutions of various equations and inequalities (in particular, descriptions that involve no probabilistic notions). These analytical characterisations are the starting point of many of the deterministic numerical schemes discussed in Chapter \ref{introthe} that are used in practice to compute these distributions (or approximations thereof should exact computation not be feasible). For example, the schemes of Chapters \ref{mcmom}--\ref{dists} focus on approximating the solutions of these equations and inequalities (and essentially forget about the origin of these equations and inequalities), while those of Chapter \ref{fspchap} modify the original chain in such a way that these equations and inequalities become computationally tractable, solve them, and then relate the distributions of the modified chain to those of the original chain. 

Our focus on these analytical characterisations and our use of the words ``analytical characterisation" are motivated by analogous linear programming characterisations pervasive in the optimal control literature, see \citep{Kurtz1998,Cho2002,Lasserre2008,Hernandez-Lerma1999,Hernandez-Lerma1996,Manne1960,Wolfe1962,Borkar1996} and the many references therein.  These types of characterisations and their use are a cornerstone of optimal control, have been one of the main driving forces behind the development of high-quality LP solvers\footnote{Solvers these days are advertised as being able to handle hundreds of thousands, if not millions, of variables and constraints. In our experience, this very much depends on the linear program at hand and if it is well-conditioned.}, and date back to Alan Manne \citep{Manne1960}, George Dantzig \citep{Wolfe1962}, and their contemporaries. Indeed, it is not difficult to see that each of the characterisations we give in this chapter can be recast in terms of a linear program (in the case of the characterisations involving a ``minimal solution'' one only needs to note that this minimal solution is also the solution with minimum mass). The characterisations of the stationary distributions in particular can be viewed as a special case of those used in  long-run average cost control problems \citep{Echeverra1982,Hernandez-Lerma1999,Hernandez-Lerma2003,Lasserre1994,Kurtz1998}. Those of the time-varying law and of the (joint) exit distributions can be viewed as special cases of those used in finite-horizon cost and first-passage cost control problems \citep{Kurtz1998}.

\fi

%

%
%
%
%
%

\newpage
\tableofcontents
%
%

\ifdraft
\newpage

\section*{Introduction: An overview: Markov chains, their analysis, and their numerical approximations}\addcontentsline{toc}{section}{\protect\numberline{}An overview: Markov chains, their analysis, and their numerical approximations}
\pagestyle{premain}\sectionmark{\MakeUppercase{Introduction}}

{\color{red}BRING THIS BACK ONCE WE HAVE SOME NUMERICAL METHODS DONE}

This thesis is a monograph on Markov chains, the problem of their analysis, and deterministic approximations schemes developed to tackle this problem. Let's take this sentence apart:
\subsubsection*{Markov chains:}By a Markov chain, we mean a time-homogeneous Markov process that takes values in a countable state space. We consider both the discrete-time and the continuous-time cases. In the latter, we restrict our attention to minimal chains with stable and conservative rate matrices. 

Our motivation behind the study of these chains is their real-life use. They are employed as models in an astonishingly wide range of scientific and engineering disciplines. These include chemical physics \citep{Schlogl1972,kampen2007,Goutsias2013,Gardiner2009}, computer science \citep{Kleinrock1976}, cellular biology \citep{Goutsias2013,Anderson2015,Kimmel2002}, queueing theory \citep{Asmussen2003,Meyn2008}, neuroscience \citep{Amit1992}, economics \citep{Aoki2002}, epidemiology \citep{Goutsias2013,Haccou2005}, quantum optics \citep{Lugiato1987}, bioinformatics \citep{Kimmel2002}, operations research \citep{Meyn2008}, pharmacokinetics \citep{Goutsias2013}, medicine \citep{Kimmel2002}, machine learning \citep{Watkins1992}, and ecology \citep{Gardiner2009,Goutsias2013,Haccou2005}. More generally, mathematical models are a pillar of both the scientific method and modern engineering techniques. Regarding the former, models allow us to test our hypotheses inexpensively and to infer information regarding the unobservable aspects of physical phenomena from the observable ones. Regarding the latter, models drastically cut down the monetary, time, and human cost of design cycles and, by doing so, make technical ventures feasible when they otherwise would not be. 
\begin{quotation}``Essentially, all models are wrong, but some are useful.''\end{quotation}
\begin{quotation}\emph{George Box in Empirical Model-Building and Response Surfaces.}\end{quotation}

A model is no more useful than it is tractable. If a model is so complex that we cannot extract any information from it, then its utility in guiding our real-life decisions is null. A trade-off between the practitioner's desire for faithful representations of real life with the theoretician's desire for tractable models is always necessary. Markov processes are often deemed to be a good comprise in this respect. While still flexible enough to be appropriate in many contexts, their basic theory is very mature (at least when it comes to the time-homogeneous case), we know what range of behaviours they exhibit, and, in many cases, we know how to discriminate between the possible behaviours in practice. However, \emph{quantifying} these behaviours is a task that often still proves elusive to this day.

\subsubsection*{The problem of analysis:}Broadly speaking, \emph{analysing} a model amounts to predicting its behaviour (or a particular aspect of interest thereof) from the model's definition. It is the foundation on which any real-life use of a model builds on. Such a use involves a \emph{modelling} process and, often, one of \emph{control} too. Modelling, or \emph{fitting}, consists of amending the model's definition so that its behaviour is consistent with what we have observed of the phenomenon that we wish to study. It reduces to iterated analysis: we predict the model's behaviour given the current definition, compare the predictions to the data gathered in real life and, based on this comparison, we alter the model so that it better emulates the phenomenon; repeating this process until the predictions match the data satisfactorily well. We then hope that the model not only behaves like the phenomenon in the aspects that we have observed but also in those that we have not. In this way, a good model enables us to extrapolate the data we have gathered so far and to predict how the phenomenon will behave in untested circumstances (especially, circumstances that are costly, dangerous, or outright infeasible to investigate).

In engineering applications, we are often not interested in predicting how the phenomenon will act in unfamiliar circumstances or filling in gaps in our knowledge as much as we are in \emph{controlling} the phenomenon. By controlling we mean influencing the physical process so that it  behaves in a desired manner. Using an (already fitted) model to this end consists of modifying the model's definition to reflect a potential real-life action that we could take to control the phenomenon, predicting the modified model's behaviour, and comparing these predictions to the desired behaviour; repeating this procedure until we find an action that makes the  predicted behaviour match the desired one. At this point, we undertake this action in real life and see if it works. In other words, control problems also consist of iterated analysis problems.
\begin{figure}[h!]
	\begin{center}
	\includegraphics[width=0.75\textwidth]{./Figures/g321.png}
	\vspace{-10pt}
	\end{center}
\caption[Modelling and control as iterated analysis]{}\label{fig:structure}
\end{figure}

Markov chains are defined in terms of a \emph{one-step matrix} (in discrete-time) or a \emph{rate matrix} (in continuous-time) and an \emph{initial condition}. The one-step/rate matrix describe where the chain will go next given where it is now while the initial condition describes where the chain is at the initial time. The analysis problems that we consider in this thesis amount to answering questions of the following type:
\begin{itemize}
\item[Q1:] Where will the chain be at a prescribed point in time?
\item[Q2:] When does the chain exit a given set, or \emph{domain}, of interest, and what part of domain's boundary does it cross while exiting?
\item[Q3:] In the long-run, which parts of the state space will the chain visit and how often?
\end{itemize}
In the field's jargon:
\begin{itemize}
\item[Q1:] What is the law of the chain at the prescribed point in time\footnote{For the lack of a better name, we refer to this distribution as the \emph{time-varying law} so not to conflate it with the distribution on the path-space induced by the chain sometimes referred to as the law of the process. Some texts refer to the time-varying law as the ``one-dimensional law'' of the process.}?
\item[Q2:] What is the exit distribution associated with the exit from the domain?
\item[Q3:] What are the stationary distributions of the chain?
\end{itemize}
As we discuss in Chapter \ref{Markovpre}, answering each of these questions amounts to finding the solutions of a set of linear equations and inequalities of one type or another. Unfortunately, no analytical expressions of these solutions can be obtained for the majority of chains, and we must resort to numerical methods. These equations and inequalities involve as many variables as there are states in the state space and a comparable number of equations and inequalities. If the state space is infinite, or even just large, direct numerical computation of these solutions is infeasible, and we must settle for approximations thereof. The end goal of this thesis is to study and develop reliable and systematic schemes that yield approximations of the distributions featuring in Questions Q1--3. In doing so, we aim to equip practitioners with the tools required to analyse these models and, consequently, further enable their use in real-life applications. This said, this thesis is dedicated to the theoretical aspects of chains and these  approximation schemes; we have largely left their applications and implementations as future work.
\subsubsection*{Deterministic approximation schemes:} Schemes that yield approximations of the distributions featuring in Questions Q1--3 posed above can be grouped into two types: stochastic schemes and deterministic ones. Stochastic schemes are those for which the approximation computed is a random object itself; if the scheme is run twice under identical conditions, the approximations it returns are not  identical (although, hopefully, the difference between them is negligible). The absence of this characterises deterministic schemes; if the scheme is run twice under the same conditions, then it yields the same approximation. This thesis is dedicated to the latter variety.

By and large, stochastic schemes are simulation and Monte Carlo methods. They are recipes for constructing a random variable known as an \emph{estimator} that takes values close to the object of interest. Often, the estimator consists of the empirical average of a number other random variables whose mean is the object of interest (in which case the estimator is said to be \emph{unbiased}), or something close to it. The random variables averaged are sampled independently of each other whenever possible, or something close to that otherwise. Various laws of large numbers and generalisations thereof imply that the empirical average converges almost surely to something close to the object of interest as more and more variables are added to it. In the unbiased and independent case, it converges to exactly the object of interest and the estimator is said to be \emph{consistent}. Of course, in practice only finitely many samples can be drawn, and so the object computed is only approximate. However, various central limit theorems and generalisations thereof can be employed to quantify the approximation error. For more on these schemes in the context of stochastic reaction networks (a type of chain that we discuss in Section \ref{SRNs}) see \citep[Chap.5]{Anderson2015}. In broader contexts, we find Gareth O. Roberts, Andrew M. Stuart, and their co-authors to be particularly eloquent orators on these matters \citep{Roberts2004,Stuart2010}.

There are two main types of deterministic schemes. Schemes of the first kind, known as \emph{system size expansions}, apply to chains which involve large numbers of individuals and large numbers of interactions between them. Even though the individual (or \emph{microscopic}) interactions are random, these tend to cancel each other out, and the overall collective (or \emph{macroscopic}) behaviour is approximately deterministic. In these cases, the distributions featuring in Questions Q1--3 tend to be unimodal and narrow so that their means, on their own, provide accurate depictions of the distribution. If this assumption holds exactly, then the means obey deterministic laws (recurrence equations or ordinary differential equations). The schemes involve replacing the distributions with their means and computing these using the appropriate deterministic law\footnote{Here, we are giving a very simplistic perspective of system size expansions. There are intermediate levels of these approximations in which the chains are approximated by other stochastic processes.}. We do not investigate these schemes in this thesis. On this matter, we like Thomas G. Kurtz \citep{Kurtz1970,Kurtz1971,Kurtz1976,Kurtz1978,Ethier1986}. 

Deterministic schemes of the second type are the subject of this thesis. In our opinion, they are formalisations of the following simple idea:
%
%

\MyQuote{\label{eq:idea}\hspace{10pt} Most Markov chains have many states. However, it is often the case that only a small fraction of these states are of any significant importance.
}
For instance, consider the classical gambler's ruin problem: a gambler repeatedly bets one pound, each time a coin is tossed, if it lands on heads she gets two pounds back, otherwise she losses her pound. Suppose that the coin is not biased in her favour: the probability that it lands on heads is at most one half. The classical model for this scenario is a biased random walk on the set of natural numbers $\n$: the walk being in a state $x\in\n$ represents the gambler having $x$ pounds in her possession. Famously, regardless of her initial wealth, the gambler will eventually go bankrupt with probability one. Intuitively, if the gambler starts betting with ten pounds in her pocket, it is unlikely that she will ever amass a thousand pounds or more. For this reason, it is ostensibly reasonable to expect that the states past $999$ are irrelevant to any question we may ask about the gambler's financial prospects. In other words, instead of dealing with the infinitely many states of the walk's state space, we only have to focus on $1000$ of them, a far more manageable number. Formally, we say that we are replacing the state space $\n$ with its \emph{truncation} $\{0,1,\dots,999\}$.

The difficulty in exploiting \eqref{eq:idea} lies in deciding which states are important for the question at hand. Even if something seems reasonable based on physical arguments, models are abstractions of real life and can be far from faithful representations of it. This is particularly prominent in the context of modelling where the model is not yet fitted. Indeed, many fitting algorithms involve automated sweeps of parameter sets, and the algorithms can easily step through parameter values that lead to poor models. Even if the model is well fitted, our intuition regarding the physical phenomenon of interest can often be  lacklustre. For instance, as we will verify in Example \ref{gam1000} of Chapter \ref{Markovpre}, if the coin is unbiased, then the gambler has a $1\%$ chance of ever amassing $1000$ pounds, a perhaps not so negligible number. The intuition here is that game is fair, so the gambler has a $1/100$ chance of multiplying her starting wealth a $100$-fold. This example underscores the importance taking care when dismissing states as inconsequential. In this thesis, we address this issue using one or both of the following:
\begin{enumerate}
\item \textbf{Bounds on the object of interest and computable errors:} The schemes we study do not just yield approximations on the object of interest, but also bounds. Combining a lower and an upper bound on a number of interest quantifies how far each of these is from the actual number (they are no further away than the difference of the two bounds): an observation that we exploit in Chapters \ref{mcmom}--\ref{sdemom}. Alternatively, by computing lower bounds on every entry of a probability distribution, it is possible to calculate exactly the total variation distance between the unknown distribution and the collection of lower bounds. We make repeated use of this in Chapters \ref{fspchap}, \ref{infdimlpsec}, and \ref{dists}. This error control does not directly tell us which states to pay attention to; however, if we have made poor choices, then it alerts of this fact \emph{a posteriori} (that is, after running the scheme and computing the bounds).
\item \textbf{Moment and tail bounds:} The approximation error of the schemes discussed in Chapter \ref{fspchap} is the mass of the tail of a distribution. By tail we mean the part of the distribution that lies outside of the truncation we have chosen. Suppose that we know one of the moments of this distribution. Choosing a truncation related to this moment, we can use Markov's inequality to bound the mass of the tail (we call these type of bounds \emph{tail bounds}). Unfortunately, these moments are typically not computable. However, in Chapter \ref{mcmom} we explain how to obtain upper (and lower) bounds on them. Replacing such an upper bound for the moment in Markov's inequality we recover a tail bound. In this manner, we can bound the approximation error \emph{a priori} (that is, before running the scheme) and our truncation choice is guided by the moment bound. Of course, this a priori error control comes at the expense of having to compute a moment bound. Along similar lines, the approximation error of the schemes in Chapter \ref{infdimlpsec} is at least the mass of the tail of the distribution being computed while that of the schemes in Chapter \ref{dists} is at least the tail bound (these schemes themselves employ a moment bound and a corresponding tail bound).
\end{enumerate}

As an aside, simulation and Monte-Carlo methods frequently work in circumstances that truncation-based methods fail because they tend to automatically resolve this issue of which states are important and which are not. By construction, simulations only visit states that the chain itself is likely to visit: these are typically the important states. 

The task of computing numerical approximations of these typically infinite-dimensional objects is daunting and involves several levels of approximation. A division of labour is required, and each of us must choose the level they work on. We draw the line by assuming that systems of linear ordinary differential equations (ODEs), linear programs (LPs), and semidefinite programs (SDPs) can be solved exactly. None of these things are true. Philosophically(-ish), they cannot be; at the end  of the day, we cannot carry out the real number arithmetic they rely on \emph{exactly}: we are always limited by machine precision. However, the theory behind ODE, LP, and SDP solvers is very rich: that behind ODE and LP solvers is classical, see \citep{Moler2003,Ben-Tal2001,Boyd2004} and references therein, while that behind SDP solvers was settled twenty-five years ago in the renown work \citep{Nesterov1994} of A. Nemirovski and Y. Nesterov on polynomial-time primal-dual interior point methods for convex optimisation problems. Moreover, ODE and LP solvers are mature technologies and large efforts to improve all three are ongoing. In our experience, issues at this level only become noticeable when the problem data is ill-conditioned. Unfortunately, when it comes to probability, the numbers involved change orders of magnitude very quickly, and, in our opinion, this ensures that ill-conditioning is not a rare occurrence. Consequently, this lower level of approximation that we by and large ignore (aside from the remarks made in Sections \ref{momprobthe} and \ref{apriori}) is very important for the success of these schemes and their practical use.

\subsubsection*{Thesis structure:}In the ensuing section we gather notation that we use throughout the thesis. Chapter \ref{Markovpre} introduces the chains and discusses the theory we need to answer Questions Q1--3 in practice. In particular, we review descriptions in terms of linear equations and inequalities of the distributions featuring in these questions. These \emph{analytical} descriptions are the starting point of the deterministic approximation schemes discussed in the thesis used to answer Q1--3.  In Chapter \ref{fspchap}, we capitalise on the descriptions and develop truncation-based schemes that yield approximations with controlled errors of the distributions in Questions Q1 and Q2. Chapter \ref{mcmom} discusses semidefinite programming schemes that yield bounds on the moments of the distributions in Questions Q2 and Q3 (although, we focus on the latter). In Chapter \ref{infdimlpsec}, we introduce truncation-based schemes that can be used to solve countably infinite linear programs (CILPs) and study the schemes' theoretical properties. In Chapter \ref{dists}, 
we exploit the fact that the analytical descriptions of Chapter \ref{Markovpre} can be rephrased in terms of CILPs and apply the schemes of Chapter \ref{infdimlpsec} to obtain approximations (accompanied by computable error bounds) of the distributions in Questions Q2 and Q3. Once again, we focus on Q3. Chapter \ref{sdemom} is a short foray out of the countable state space case that demonstrates that (at least some of) the schemes discussed in this thesis apply to more general Markov processes than chains. In particular, we present a scheme analogous to those in Chapter \ref{mcmom} that yields bounds on the moments of the distributions of Question Q3 for diffusion processes generated by stochastic differential equations. We conclude the thesis with some parting comments in Chapter \ref{conclusion}.

Although this thesis was written to be read linearly, this is not necessary: each chapter is presented of a self-contained unit beginning with an introduction, a review of the relevant literature, and delineation of the chapter's contributions and concluded with a discussion of its content and future work. We only recommend reading Chapter \ref{Markovpre} before Chapter \ref{fspchap}, \ref{mcmom}, or \ref{dists}. This said, readers well-versed in Markov chain theory or that are more concerned with the statement of the results of later chapters than their proofs may wish to only skim Chapter \ref{Markovpre} (in particular, just to get acquainted with the notation introduced in it). To facilitate the notation/results look-up process, we have appended  an index  and a glossary of symbols to the end of the thesis and a  glossary of acronyms to its start. 

Throughout the thesis, we assume that the reader is familiar with measure theory, the measure theoretic foundations of modern probability theory including conditional expectation and discrete-time martingales (the last two are only important in Chapter \ref{Markovpre}). If not, we recommend \citep{Tao2011} and \citep{Williams1991}. We also assume that the reader is familiar with the basics of linear programming and semidefinite programming. If not, we recommend \citep{Ben-Tal2001,Vandenberghe1996,Boyd2004,Renegar2001}. In Chapter \ref{sdemom}, we assume that the reader is familiar with the stochastic calculus for continuous semimartingales and euclidean-valued stochastic differential equations. Here, we recommend \citep{Rogers2000b}.

\fi

\newpage

\pagenumbering{arabic}
%
\sectionmark{\MakeUppercase{Some important notation and preliminaries}}
\titleformat{\subsection}[runin]{\large\bfseries}{\thesubsection}{0em}{. }[.]
\settocdepth{subsection}
\subsection{Some important notation and preliminaries}\label{notationthe}
\sectionmark{\MakeUppercase{Some important notation and preliminaries}}
\pagestyle{FUN}

The symbol 
$$\text{`}:=\text{'}$$
is\glsadd{3} used to define whatever precedes as whatever follows it, and vice versa for $\text{`}=:\text{'}$.\glsadd{4} 
\subsubsection*{Numbers}
\begin{itemize}[label={},leftmargin=*]
\item $\n$ denotes the set of natural numbers (including zero), \glsadd{n}
\item $\z$ that of integers, \glsadd{z}
\item $\zp$ that of positive integers, \glsadd{zp}
\item $\r$ that of real numbers, \glsadd{r}
\item $[0,\infty)$ that of non-negative real numbers, \glsadd{rp}
\item $(0,\infty)$ that of positive real numbers, \glsadd{rpp}
\item and $\q$ that of rational numbers.\glsadd{rat}
\end{itemize}
\subsubsection*{Extended numbers}
\begin{itemize}[label={},leftmargin=*]
\item $\n_E:=\n\cup\{\infty\}$ denotes the set of extended natural numbers\glsadd{ne}, 
\item $\z_E:=\z\cup\{-\infty,\infty\}$ that of extended integers\glsadd{ze},
\item $\r_E:=\r\cup\{-\infty,\infty\}$ that of extended real numbers\glsadd{re}, 
\item $[0,\infty]:=[0,\infty)\cup \{\infty\}$ that of non-negative numbers\glsadd{rpe}, 
\item $(0,\infty]:=(0,\infty)\cup \{\infty\}$ that of positive real numbers\glsadd{rppe}, 
\end{itemize}

Whenever I say that a number (function, random variable, etc.) is `non-negative', I mean that it takes values in, $[0,\infty]$ instead of $[0,\infty)$, while `non-negative and finite' means that it takes values in, $[0,\infty)$. Similarly with `positive', `non-positive', and `negative'.

We order the extended numbers as usual:  for all $a,b\in \r$, $a\geq b$ if and only if $a-b\in[0,\infty)$ and 
$$-\infty\leq c\leq \infty \qquad\forall c\in\r_E.$$
The rules we use to manipulate them are also the usual:
\begin{itemize}[label={},leftmargin=*]
\item infinity times $-1$ is minus infinity: $\infty\cdot (-1)=-\infty$,
\item infinity times a positive number $a\in(0,\infty]$ is itself: $\infty\cdot a=\infty$,
\item any real number $a\in\r$  divided by infinity is zero: $a/\infty=0$,
\item infinity plus a number $a\neq-\infty$ is itself: $a+\infty=\infty$, 
\item and infinity times zero is zero: $\infty\cdot 0=0$.
\end{itemize}

\subsubsection*{Inf/sup/min/max/floor/ceiling}
The infimum (resp. supremum) of the empty set $\emptyset$ is plus infinity (resp. minus infinity):
$$\inf\emptyset=\infty,\qquad\sup\emptyset=-\infty.$$
For any numbers $a,b\in\r_E$,\glsadd{vee}\glsadd{wedge}
$$a\wedge b:=\min\{a,b\}\quad a\vee b:=\max\{a,b\},$$
$\lfloor a\rfloor$ denotes the largest  number\glsadd{floor} in $\z_E$ no greater than $a$, and $\lceil a\rceil$ the smallest\glsadd{ceil} no smaller than $a$.

\subsubsection*{Sigma-algebras}
\begin{itemize}[label={},leftmargin=*]
\item $2^\Omega$ denotes power set $\{A:A\subseteq\Omega\}$ of any given set $\Omega$.\glsadd{2om}
\item $\cal{B}(\r)$ denotes the Borel sigma-algebra on $\r$ (topologised with the euclidean topology): the smallest sigma-algebra on $\r$ that contains all of its open subsets (i.e., the sigma-algebra \emph{generated} by the collection of open subsets).
\item $\cal{B}(\Omega):=\{A\cap\Omega:A\in\cal{B}(\r)\}$ denotes the Borel sigma-algebra on $\Omega$, for any Borel subset $\Omega$ in $\cal{B}(\r)$ of $\r$.\glsadd{borel}
\item $\cal{B}(\r_E)$ denotes the sigma-algebra on $\r_E$ generated by $\cal{B}(\r)\cup\{\{-\infty\},\{\infty\}\}$.
\end{itemize}
\subsubsection*{Functions} Throughout the following, let $f$ be a function mapping from $\Omega$ to $\cal{X}$ and let $\cal{F}$ and $\cal{G}$ be sigma-algebras on $\Omega$ and $\cal{X}$, respectively. 
\begin{itemize}[label={},leftmargin=*]
\item $f=g$ means\glsadd{1} that $f(\omega)=g(\omega)$ for all $\omega$ in $\Omega$ and similarly for $f\geq g$ and $f\leq g$.
\item For any subset $B$ of the co-domain $\cal{X}$, $\{f\in B\}$ denotes the pre-image\glsadd{preimage} of $B$: 
$$\{f\in B\}:=\{\omega\in\Omega:f(\omega)\in B\}.$$
\item $f$ is \emph{$\cal{F}/\cal{G}$-measurable} if it is a measurable mapping from $(\Omega,\cal{F})$ to $(\cal{X},\cal{G})$:
$$\{f\in B\}\in\cal{F}\quad\forall B\in\cal{G}.$$
If $B$ is a singleton $\{x\}$, then we write $\{f=x\}$ instead of $\{f\in B\}$. Similarly for $\{f\geq x\}$, $\{f\leq x\}$, and $\r_E$-valued functions $f$.
\item For any a subset $A$ of $\Omega$, $1_A$ denotes the indicator function\glsadd{inda} of $A$:
$$1_A(\omega):=\left\{\begin{array}{ll}1&\text{if }\omega\in A\\0&\text{otherwise}\end{array}\right.$$
For a single point $\omega$ in $\Omega$, we use the shorthand $1_\omega:=1_{\{\omega\}}$.
\end{itemize}

Sometimes, we will encounter $\r_E$-valued functions $f$ that are only defined  on a subset $A$ of $\Omega$. In these cases, I will use $1_Af$ to denote the function on $\Omega$ defined as
\begin{equation}\label{eq:partdef}(1_Af)(x):=\left\{\begin{array}{ll}f(x)&\text{if }x\in A\\0&\text{if }x\not\in A\end{array}\right.\end{equation}
\begin{lemma}\label{lem:partdefmeas}Suppose that $A$ belongs to $\cal{F}$ and that $f$ maps from $A$ to $\r_E$. Then $1_Af$ in \eqref{eq:partdef} is $\cal{F}/\cal{B}(\r_E)$-measurable if 
$$\{\omega\in A: f(\omega)\in B\}\in \cal{F},\quad\forall B\in\cal{B}(\r_E).$$
\end{lemma}

\begin{proof}Fix any $B$ in $\cal{B}(\r_E)$. If $0$ does not belong to $B$, then 
$$\{1_Af\in B\}=\{\omega\in A:f(\omega)\in B\}\in \cal{F}.$$
Otherwise,
$$\{1_Af\in B\}=(\Omega\backslash A) \cup\{\omega\in A:f(\omega)\in B\},$$
and it follows that $\{1_Af\in B\}$ belongs to $\cal{F}$ as sigma-algebras are closed under unions and intersections.
\end{proof}

\subsubsection*{Well-defined integrals and sums} For any unsigned measure $\rho$ on $(\Omega,\cal{F})$ and $\r_E$ valued function $f$ on $\Omega$, I say that the integral
$$\rho(f):=\int f(x)\rho(dx)$$
is \emph{well-defined} if
$$\rho(f\wedge 0)=\int (f(x)\wedge 0 )\rho(dx)>-\infty\quad\text{or}\quad\rho(f\vee 0)=\int (f(x)\vee 0 )\rho(dx)<\infty.$$
The terminology also applies to sums: setting $\rho$ to be the counting measure on a countable set $\s$, we have that
$$\sum_{\omega\in\Omega}f(\omega)$$
is well-defined if
$$\sum_{\omega\in\Omega}f(\omega)\wedge 0 >-\infty \quad\text{or}\quad \sum_{\omega\in\Omega}f(\omega)\vee 0<\infty.$$
\subsubsection*{Countable sets and vector notation}

\begin{itemize}[label={},leftmargin=*]
\item I always use the power set $\tws$ of a countable set $\s$ as a sigma-algebra on $\s$: whenever I say a measure or a distribution on $\s$, we really mean a measure or distribution on $(\s,\tws)$.

\item Given any measure $\rho$ on $\s$, I commit the slight abuse of notation of using $\rho$ to denote both the measure and its density with respect to the counting measure on $(\s,\tws)$. That is, I write $\rho(x):=\rho(\{x\})$ for each $x$ in $\cal{S}$, so that
$$\rho(A)=\sum_{x\in A}\rho(x)$$
for each subset $A$ of $\cal{S}$, and 
$$\rho(f)=\sum_{x\in\s}f(x)\rho(x),$$
for every $\r_E$-valued function $f$ on $\s$ such that the above is well-defined. 

\item I treat $\r_E$-valued functions $f$  on $\s$ as `column vectors'  $(f(x))_{x\in\s}$, 
and  measures $\rho$ on $\s$ as `row vectors', of (extended) real numbers indexed by the elements of $\s$. In particular, given a matrix $A:=(a(x,y))_{x,y\in\s}$ of (extended) real numbers indexed by $\s$, I use $Af$ to denote the $\r_E$-valued function defined by
$$Af(x):=\sum_{y\in\cal{S}}a(x,y)f(y)\quad\forall x\in\s,$$
and $\rho A$ to denote the measure defined by
$$\rho A(x):=\sum_{y\in\cal{S}}\rho(y)a(y,x)\quad\forall x\in\s,$$
assuming, of course, that the above sums are all well-defined.
\end{itemize}
\subsubsection*{Preliminaries}Throughout the text, I assume that the reader is familiar with the following:
\begin{itemize}[label={},leftmargin=*]
\item Sigma-algebras, generating sets, measurable functions, and the operations that preserve measurability (e.g., the composition of two measurable functions are measurable).
\item The Lebesgue integral, properties thereof (linearity, etc.), and the standard machine.
\item Fatou's lemma and the monotone, dominated, and bounded convergence theorems.
\item Product measures, Tonelli's theorem, and Fubini's theorem.
\item Probability triplets and the definition of random variables as measurable functions.
\item The modes of convergence of random variables (almost surely, in probability, etc.).
\item The definition of conditional expectation as a random variables and its properties (see below in particular).
\item The definition of a discrete-time martingale.
\end{itemize}
For those who are not, I recommend \citep{Williams1991,Tao2011}.
\subsubsection*{Properties of conditional expectation}I will make repeated use of the following:
\begin{theorem}\label{thrm:condexpprops}Suppose that $X$ is non-negative ($X(\omega)\geq0$ for all $\omega\in\Omega$) or integrable ($\Ebb{\mmag{X}}<\infty$) $\r_E$-valued random variable on some probability triplet $(\Omega,\cal{F},\Pb)$.
\begin{enumerate}[label=(\roman*),noitemsep]
\item If $X$ is $\cal{G}$-measurable, then $\Ebb{X|\cal{G}}=X$, a.s.
\item {(Linearity)} If $X_1$ and $X_2$ are integrable random variables and $a_1,a_2\in\r$, 
$$\Ebb{a_1X_1+a_2X_2|\cal{G}}=a_1\Ebb{X_1|\cal{G}}+a_2\Ebb{X_2|\cal{G}},\quad\text{a.s.}$$
If $X_1,X_2,\dots$ are non-negative random variables and $a_1,a_2,\dots$ non-negative constants, then 
$$\Ebb{\sum_{n=1}^\infty a_nX_n|\cal{G}}=\sum_{n=1}^\infty a_n\Ebb{X_n|\cal{G}},\quad\text{a.s.}$$
\item {(Positivity)} If $X$ is non-negative, then $\Ebb{X|\cal{G}}$ is non-negative, a.s.
\item {(Tower property)} If $\cal{H}$ is a sigma-algebra contained in $\cal{G}$, then
$$\Ebb{\Ebb{X|\cal{G}}|\cal{H}}=\Ebb{X|\cal{H}},\quad\text{a.s.}$$
\item {(Take out what is known)} Suppose that $Z$ is $\cal{G}$-measurable. If $Z$ is bounded, or if both $Z$ and $X$ are non-negative, then
$$\Ebb{ZX|\cal{G}}=Z\Ebb{X|\cal{G}},\quad\text{a.s.}$$
%
\end{enumerate}
\end{theorem}

\begin{proof}The above properties are proven in \citep[Section~9.7]{Williams1991} for the $\r$-valued integrable case. However, the proofs given therein hold almost verbatim our slightly more general case.
\end{proof}


%

%

\newpage

\part{Theory}\label{part:theory}
\pagestyle{premain}

\begin{quotation}``We have not succeeded in answering all our problems. The answers we have found only serve to raise a whole set of new questions. In some ways we feel we are as confused as ever, but we believe we are confused on a higher level and about more important things.''\end{quotation}
\begin{quotation}
\emph{Posted outside the mathematics reading room in Troms\o ~University, according to Bernt \O ksendal \citep{Oksendal2003}.}\end{quotation}

\section*{Introduction: Markov processes and their long-term behaviour}\addcontentsline{toc}{section}{\protect\numberline{}Introduction: Markov processes and their long-term behaviours}
\sectionmark{\MakeUppercase{Introduction: Markov processes and their long-term behaviour}}

\label{sec:qualbeh}

This book revolves around  Markov processes: collections of random variables $(X_t)_{t\in\mathbb{T}}$ (referred to as \emph{processes}) that are indexed by an ordered set $\bb{T}$ of time points, take values in some \emph{state space}, and satisfy the Markov property. Informally, the Markov property says that we are equally well-equipped to predict what will happen in the future if we only know what is going on right now (that is, the current \emph{state} of the process) or if we know
what is going on right now \emph{and} what happened in the past. In our case, the indexing set will either be the natural numbers, in which case we describe the process as \emph{discrete-time}, or the non-negative real numbers, in which case we classify it as \emph{continuous-time}. If the state space is a countable set, we refer to the process as a \emph{chain} (beware, it is equally as common for authors to call a Markov process a `chain' if it is a discrete-time process regardless of whether of the state space is countable or not). A Markov process is said to be \emph{time-homogeneous} if the probability that it transitions from any given subset of the state space to any other over a given interval of time depends only on the interval's length and not on it's endpoints.  Using the words `a Markov process' when one means `a time-homogeneous Markov process' is nearly a tradition in the field as John F. C. Kingman amusingly pointed out in a lovely talk 
%
he gave   at Imperial College in 2013. We will adhere to this `tradition' (throughout the book, we only treat the time-homogeneous case).

One of the most well-understood and extensively studied aspects of Markov processes is their \emph{long-term} (or \emph{asymptotic} or \emph{long-run}) behaviour. Without getting into technicalities (we will have plenty of these later!), the possible long-term behaviours can be roughly grouped into four types (see also Fig.~\ref{fig:cartoon}):
\begin{enumerate}
\item The sample paths of the process diverge to infinity in a finite amount of time. Because discrete-time processes can only change state once per step and single step transitions of infinite size are usually not allowed, only continuous-time processes exhibit this behaviour. In this case, the process is said to be \emph{explosive} and the time at which this event occurs is called the \emph{explosion time}. 
\item The process diverges to infinity in an infinite amount of time. In this case, the process is said to be \emph{transient}. Processes that are not explosive nor transient are said to be \emph{recurrent}.
\item The paths do not tend to infinity and keep revisiting certain small subsets  of the state space (for euclidean-valued processes `small' typically means compact in the usual topology, while for chains it usually means finite). However, the excursions from these states get larger and larger over time and the frequency of the visits decays to zero. In this case, the mass of the distribution describing process's location spreads out over the entire state space and the probability that the process lies inside any given small set tends to zero as time progresses.  In this case, the process is said to be \emph{null recurrent}.
\item None of the above. In this case, for each path, there are small subsets of the state space for which the frequency of the path's visits does not degenerate to zero. Consequently, the  mass of the distribution of the process's location does not vanish over time and instead concentrates around these small subsets. The process is then said to be \emph{positive recurrent}.
\end{enumerate}

The above can be viewed as progressive levels of stability of Markov processes. In this book, we often focus on the fourth type of behaviour (which we refer to as \emph{stable}) since it is the one typically exhibited by Markovian models of real-life phenomena. Or at least, it is often hoped to be exhibited by these models! I find the argument `the physical phenomenon exhibits no such unstable behaviour, and thus this model of it cannot either' sometimes given to be flawed. Models are abstractions of real life and can be far from faithful representations of it. An unstable model of stable phenomena is not something impossible; it is merely an indication of poor modelling choices. Indeed, modelling is a difficult task and mistakes happen. Moreover, many fitting algorithms involve automated sweeps of parameter sets, and these algorithms can easily step through sets that lead to poor models.
For these reasons, I believe care must be taken to rule out unstable behaviours even in models of stable phenomena. This can be done in practice using the Foster-Lyapunov criteria discussed throughout the book.

\begin{figure}
	\begin{center}
	\includegraphics[origin=c,width=1\textwidth]{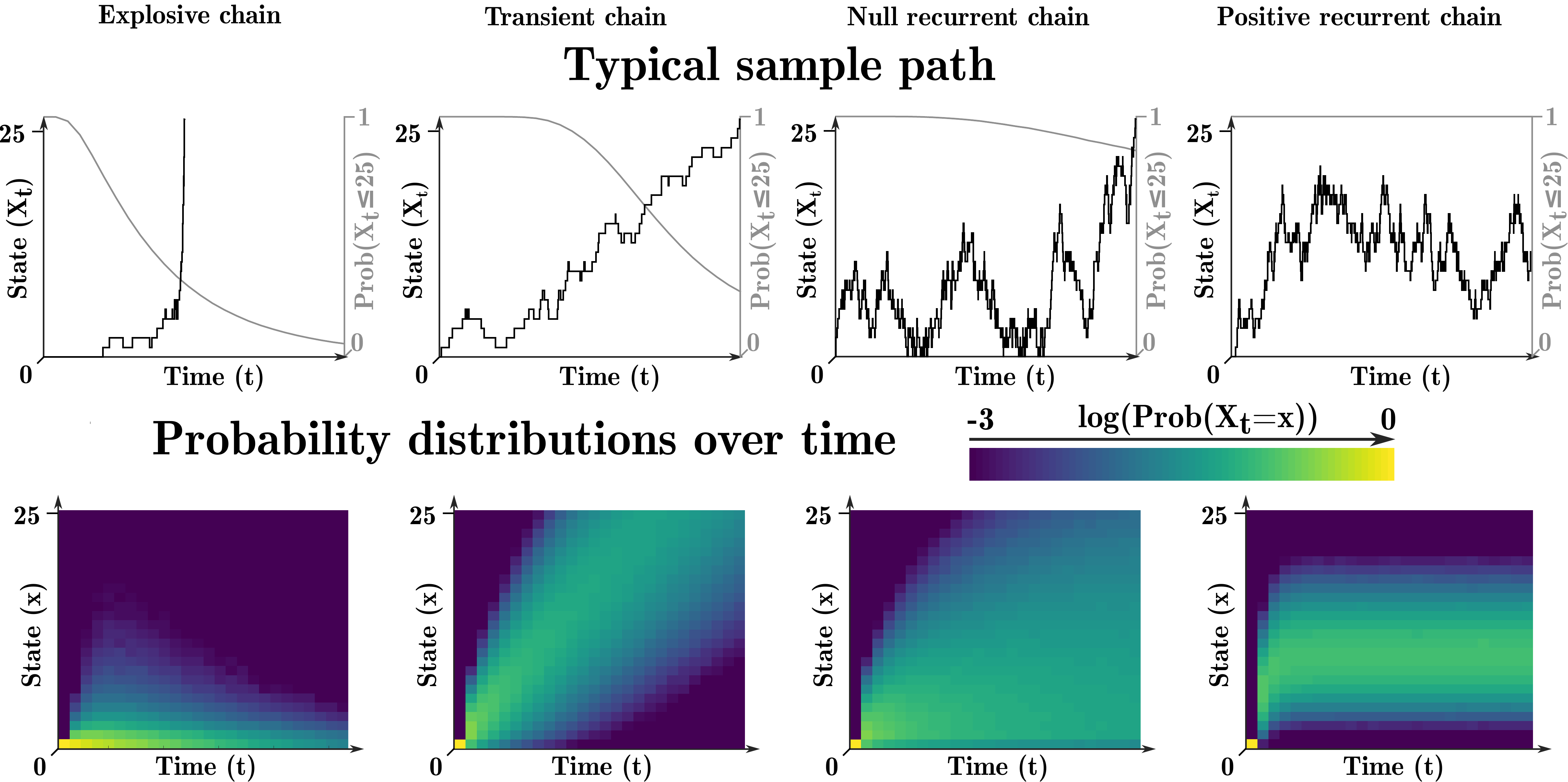}
	\vspace{-15pt}
	\end{center}
\caption[The long-term behaviours of continuous-time Markov chains.]{\textbf{The long-term behaviours of continuous-time Markov chains.} (Top, black) Typical sample paths $t\mapsto X_t(\omega)$ of an explosive chain, a transient chain, a null recurrent chain, and a positive recurrent chain. For all four chains, the state space is $\n$ but the plot only displays states $0$--$25$. (Top, shaded) The probability $Prob(X_t\leq 25)$ of the chains being within the plotted states at time $t$. (Bottom) Corresponding probability distributions (in log-scale) as a function of time; chains started in state $0$.

In the explosive case, the chain lingers awhile in certain states ($\sim0$--$5$). However, once it ventures out of these states, its path shoots off towards infinity. Thus, the probability mass quickly escapes towards infinity with any mass remaining in the plotted region being concentrated in states $0$--$5$. In the transient case, the paths steadily march towards infinity. Correspondingly, the probability mass moves away from the lower states and heads towards infinity. In the null recurrent case, the paths keep returning certain states (e.g.\ $0$), however these visits become ever rarer as time progresses. Consequently, the probability mass does not distinctly move away from these states but rather spreads out over an ever growing region in such a way that the probability of being  in any given state tends to zero.  In the positive recurrent case, the chain keeps visiting certain states (e.g.\ $10$) and the frequency of these visits does not decay to zero. After a bit of time, the probability distribution stabilises around a \emph{stationary distribution}.}\label{fig:cartoon}
\end{figure}

If the process is positive recurrent, then it has at least one \emph{stationary} (or \emph{invariant} or \emph{steady-state}) distribution. That is, a probability distribution $\pi$ on the state space such that sampling the process's starting location from $\pi$ ensures that the process is stationary (i.e.\ that its finite-dimensional distributions are unchanged by shifts in time). 

In the stable case, stationary distributions play an important role for all other starting distributions too. Various ergodic theorems, or generalised laws of large numbers, tell us that the \emph{time} (or \emph{empirical}) averages of the process converge to averages with respect to stationary distributions. In particular, in the continuous-time case, we have that for each sample path $t\mapsto X_t(\omega)$, there exists a stationary distribution $\pi$ such that the fraction of time the path spends in any given region $A$ of the state space converges to $\pi$'s mass in $A$:
\begin{equation}\label{eq:ergconvgen}\lim_{T\to\infty}\frac{1}{T}\int_0^T1_A(X_t(\omega))dt=\pi(A).\end{equation}
%
If $\bb{T}=\n$, then replace $\int_0^T1_A(X_t(\omega))dt$ with $\sum_{t=0}^{T-1}1_A(X_t(\omega))$. The stationary distribution in question generally depends on the sample path, but not on $A$. Furthermore, if the process is \emph{aperiodic}\index{aperiodic}\footnote{Note that continuous-time Markov chains are always aperiodic, and hence why there often is no mention of periodicity in texts discussing them, see Section \ref{sec:entct} for details.} (meaning that the distribution of the chain's location does not oscillate in time), then the distribution of the chain's location at time $t$ converges to some stationary distribution as $t$ approaches infinity. Which stationary distribution in particular depends on the starting distribution. If there is only one stationary distribution, then the probability that the chain lies in any particular region of the state space converges to the total fraction of time that any given path spends in the region.
%
%
That is, after enough time has passed
\begin{equation}\label{eq:time-space}
\text{time averages }\approx\text{ space averages}.
\end{equation}
For this reason, aperiodic positive recurrent Markov processes with a unique stationary distribution are said to be \emph{ergodic}.


The stationary distributions are intimately tied to the \emph{closed irreducible sets} contained in the state space. These are sets that the process, once inside, cannot escape from (hence, `closed') and that do not contain any proper subsets also fulfilling this property (hence `irreducible', these sets cannot be `reduced' further). All states not contained within a closed irreducible set are transient in the sense that the process will eventually leave each such state and never return. This implies that the paths  either diverge to infinity or get absorbed into one of the closed irreducible sets. If the process is positive recurrent, then the former is not possible and each sample path $t\mapsto X_t(\omega)$ will eventually enter a closed irreducible set $\cal{I}$. Because the path proceeds to spend the remainder of time inside $\cal{I}$, it must be the case that the stationary distribution $\pi$ featuring in~\eqref{eq:ergconvgen} has support contained in $\cal{I}$. The irreducibility of these sets further implies that there is no more than one such distribution and that it has support on the entire set. It is known as the \emph{ergodic distribution} associated with the closed irreducible set.


As mentioned before, if the process is positive recurrent and aperiodic, then the distribution of its location converges to a stationary distribution. This limiting distribution is just the weighted combination of the ergodic distributions where the weight given to each one of these distributions is the  probability that the process ever enters the irreducible set associated with that ergodic distribution. Consequently, if the starting distribution is confined to a closed irreducible set, we recover \eqref{eq:time-space}, hence justifying the name `ergodic distribution'.

Even though we have so far discussed Markov processes in general, this book is not intended to be a survey of general Markov process theory. For this, there is an abundance of far better resources:  \citep{Ethier1986,Kallenberg2001,Meyn1992,Meyn1993a,Meyn1993b,Rogers2000a,Rogers2000b} are all great. Instead, the main aim of the following Sections~\ref{sec:dtdef}--\ref{sec:expcri} is to formalise the above discussion for the cases of discrete-time  and continuous-time chains. By focusing on countable state spaces, we avoid a host of technical questions that arise with uncountable ones and often prove to only be a distraction from important ideas guiding the theory.  
In fact, virtually every aspect of the behaviour of Markov processes is also encountered in the study of chains. For this reason, possessing a mastery of theory pertinent to the chains proves to be a great boon even when dealing with more complicated cases (I am a firm believer of  that paradigm pervasive throughout science and mathematics that goes `a thorough understanding of simple situations is vital to tackle more complicated ones').

\begin{quotation}``Seulement, alors qu'autrefois il suffisait de deux heures d'expos\'e pour traiter l'int\'egrale d'It\^o, et qu'ensuite les belles applications commen\c caient, il faut \`a pr\'esent un cours de six mois sur les d\'efinitions. Que peut on y faire? Les math\'ematiques et les math\'ematiciens ont pris cette tournure. Il est temps de commencer.''\end{quotation}
\begin{quotation}
\emph{Paul-Andr\'e Meyer \citep{Meyer1976} on the dilemma of modern math education.}\end{quotation} 

\newpage
\thispagestyle{premain}
\section*{Discrete-time chains I: the basics}\addcontentsline{toc}{section}{\protect\numberline{}Discrete-time chains I: the basics}
\sectionmark{\MakeUppercase{Discrete-time chains I: the basics}}

The least technically involved type of Markov processes are discrete-time Markov processes on countable state spaces, or \emph{discrete-time Markov chains}, or \emph{discrete-time chains} for short. Our motivation behind their study is threefold:
\begin{enumerate}
\item Discrete-time chains are of interest in their own right and serve as models for a host 
of real life phenomena.
\item Many proofs in the continuous-time case require discrete-time chains.
\item By avoiding the (mostly technical) complications arising from an uncountable state space or time axis, discrete-time chains provide us with the ideal probabilistic playground to $(a)$ develop our understanding of the Markov property and its consequences;  $(b)$ acquire an intuition for the veracity of statements regarding more general Markov processes; and $(c)$ get exposure to the types of proofs and arguments pervasive in Markov chain theory. Moreover, as we will see 
when studying the long-term behaviour of continuous-time chains (Sections~\ref{sec:entct}--\ref{sec:expcri}), many proofs in the continuous-time case are analogous to their discrete-time counterparts (to the point of being a verbatim copy except for  some symbol and reference-call changes).
\end{enumerate} 
The aim of this chapter is to introduce discrete-time chains, their associated statistics (time-varying law, path law, stopping distributions, occupation measures, etc.), and the bread-and-butter tools (Chapman-Kolmogorov equation, martingale property, stopping times, Dynkin's formula, strong Markov property, etc.)  that form the foundation upon which the rest of discrete-time chain theory rests. 

The treatment here is geared towards those interested in proving results on Markov chains instead of those just interested in using them to study particular models of interest and, for this reason, is a bit technical at points. In particular, the measure-theoretic machinery required to introduce, and to comfortably
 use, the strong Markov property in its full splendour is heavier than that needed throughout the rest of the chapter. 
Past a certain point, working on Markov chain theory without the strong Markov property is like trying to wade your way through neck deep water instead of learning how to swim: sometimes do-able but needlessly tiring and frustrating. 
My advice is to bite the bullet and work your way through the details, even if only once in your life.

\subsubsection*{Overview of the chapter} We begin by defining the chains through an algorithm used to simulate them in practice (Section~\ref{sec:dtdef}). Next, we formalise the idea of the chain's history-to-date with the filtration it generates
and prove our first version of the Markov property stating that, conditioned on the present, the chain's past and future are independent (Section~\ref{sec:dtmarkov}). We then pause 
for a moment to introduce a very simple, but very instructive, example (Section~\ref{sec:gamblers}). We resume our treatment of the theory in Section~\ref{sec:dtlawsintro} where we introduce  the time-varying law (describing the chain's location as a function of time) and use the Markov property to derive a recursion expressing the time-varying law in terms of the \emph{one-step matrix} (describing the chain's next position conditioned on its present position) and the \emph{initial distribution} (describing the chain's starting location). 

In Section~\ref{sec:pathlaw}, we take a very important step: we stop viewing the chain as a sequence of random variables each taking values in the state space and instead regard it as a single random variable taking values in the space of possible sequences of states (the \emph{path space}). We then take a moment to give an equivalent description of the chains in terms of martingales (Section~\ref{sec:martin}) which, as we will see in Section~\ref{sec:dynkindt}, proves key when studying the statistics of the chain at random points in time (e.g.\ the first time that the chain enters some set of interest)  instead of deterministic one (e.g.\ time-step number $10$). Capitalising on our new path space and martingale perspectives, we show in Section~\ref{sec:otherdef} that we are not doing anything wrong by using the particular definition of the chain given in Section~\ref{sec:dtdef} instead other definitions found elsewhere: all these definitions are equivalent in the sense that they lead to the same distribution, or \emph{path law}, on the path space.


Next up are stopping times (Section~\ref{sec:stop}), which are the points in time that events of a certain type occur. In particular, events whose occurrence we can deduce the instant they happen by continuously monitoring the chain  (e.g.\ the first time the chain visits a given set, but not the last time it does). These events are non-predictive: their occurrence does not tell us anything about the future of the chain  that cannot be gleaned from the chain's position at the moment of occurrence. For this reason, stopping times combine seamlessly with the Markov property and we are able to generalise many of the results featuring deterministic times so that they also apply to stopping times. In particular,  we prove  Dynkin's formula (Section~\ref{sec:dynkindt}), a generalisation of the recurrence satisfied by the time-varying (Section~\ref{sec:dtlawsintro}), and the strong Markov property (Section~\ref{sec:dtstrmarkov}), a generalisation of the Markov property (Section~\ref{sec:dtmarkov}). Lastly, given any stopping time $\varsigma$, we introduce (Section~\ref{sec:morehitdt}) its \emph{stopping distribution} (describing the joint statistics of $\varsigma$ and of the chain's location at $\varsigma$) and its \emph{occupation measure} (describing the states visited by the chain prior to $\varsigma$ and the times of visit) and we characterise (Section~\ref{sec:exit}) them in terms of sets linear equations for the special, but important, case of an \emph{exit time} (i.e.\ the point in time when the chain first ventures outside of a given set of interest).

\subsection{The chain's definition}\label{sec:dtdef}
\pagestyle{main}
A discrete-time Markov chain (or \emph{discrete-time chain} for short)\index{discrete-time (Markov) chain} is a sequence $X=(X_n)_{n\in\n}$ of random variables that take values in a countable set $\s$ (the \emph{state space}\index{state space})\glsadd{s} and satisfy the Markov property. Out of convenience, we will assume throughout the book that the chain has been constructed in a particular way (in Section~\ref{sec:otherdef}, we will show that we lose nothing by picking this particular construction). The aim of this section is to describe the construction.

\subsubsection*{One-step matrices}For the definition, we need a \emph{one-step matrix} $P=(p(x,y))_{x,y\in\s}$\glsadd{P}\index{one-step matrix} describing where the chain will go next given its current location. In particular, $p(x,y)$ denotes the probability that the chain next visits $y$ if it currently lies at $x$. For this reason, the one-step matrix must satisfy:
\begin{equation}\label{eq:1step}p(x,y)\geq0,\quad\forall x,y\in\cal{S},\qquad \sum_{y\in\cal{S}}p(x,y)=1,\quad\forall x\in\cal{S}.\end{equation}
\subsubsection*{The technical set-up}Next, we require a measurable space\glsadd{of} $(\Omega,\cal{F})$, a collection $(\Pb_x)_{x\in\s}$\glsadd{Px} of probability measures on $(\Omega,\cal{F})$, an $\cal{F}/\tws$-measurable function $X_0$ from $\Omega$ to $\s$, and a sequence of $\cal{F}/\cal{B}((0,1))$-measurable functions\glsadd{Un} $U_1,U_2,\dots$ from $\Omega$ to $(0,1)$ such that, under $\Pb_x$,
\begin{enumerate}
\item $\{X_0=x\}$ occurs almost surely:  $\Pbx{\{X_0=x\}}=1$;
\item for each $n>0$, the random variable $U_n$ is uniformly distributed on $(0,1)$: $\Pbx{\{U_n\leq u\}}=u$, for all $u$ in $(0,1)$;
\item and the random variables $X_0,U_1,U_2,\dots$ are independent:
$$\Pbx{\{X_0=y,U_1\leq u_1,\dots,U_n\leq u_n\}}=\Pbx{\{X_0=y\}}\Pbx{\{U_1\leq u_1\}}\dots\Pbx{\{U_n\leq u_n\}}$$
for all $y$ in $\s$, $u_1,\dots,u_n$ in $(0,1)$, and $n>0$;
\end{enumerate}
for each $x$ in $\s$. For those of you inclined to such things, there's a detailed construction of this space, the measures, and the random variables at the end of the section.

Below, we will set the chain's starting location to be $X_0$. For this reason, $\Pb_x$ will describe the statistics of the chain were it to start from $x$. To instead describe its statistics were its starting location be sampled from a given probability distribution $\gamma=(\gamma(x))_{x\in\s}$\glsadd{gamma} on $\s$ (the \emph{initial distribution}), we define the probability measure $\Pb_\gamma$\glsadd{Pl} on $(\Omega,\cal{F})$ via
\begin{equation}\label{eq:pld}\Pb_\gamma(A):=\sum_{x\in\cal{S}}\gamma(x)\Pb_x(A)\quad\forall A\in\cal{F}.\end{equation}
Under this measure, the starting location of the chain has law $\gamma$:
$$\Pbl{\{X_0=y\}}=\sum_{x\in\cal{S}}\gamma(x)\Pbx{\{X_0=y\}}=\sum_{x\in\cal{S}}\gamma(x)1_y(x)=\gamma(y)\quad\forall y\in\s.$$
However, the random variables $U_1,U_2,\dots$   are still uniformly distributed:
$$\Pbl{\{U_n\leq u\}}=\sum_{x\in\cal{S}}\gamma(x)\Pbx{\{U_n\leq u\}}=u\left(\sum_{x\in\cal{S}}\gamma(x)\right)=u\quad\forall u\in(0,1),\enskip n>0;$$
and $X_0,U_1,U_2,\dots$ are still independent:
\begin{align*}
\Pbl{\{X_0=y,U_1\leq u_1,\dots,U_n\leq u_n\}}&=\sum_{x\in\cal{S}}\gamma(x)\Pbx{\{X_0=y,U_1\leq u_1,\dots,U_n\leq u_n\}}&\\
&=\sum_{x\in\cal{S}}\gamma(x)\Pbx{\{X_0=y\}}\Pbx{\{U_1\leq u_1\}}\dots\Pbx{\{U_n\leq u_n\}}\\
&=\left(\sum_{x\in\cal{S}}\gamma(x)\Pbx{\{X_0=y\}}\right)u_1\dots u_n\\
&=\Pbl{\{X_0=y\}}\Pbl{\{U_1\leq u_1\}}\dots\Pbl{\{U_n\leq u_n\}},
\end{align*}
for all $y$ in $\s$, $u_1,\dots,u_n$ in $(0,1)$, and $n>0$. Throughout the book, we use $\Eb_x$\glsadd{Ex} (resp. $\Eb_\gamma$\glsadd{El}) to denote expectation with respect to $\Pb_x$ (resp. $\Pb_\gamma$).

\subsubsection*{An algorithmic definition}
Armed with the one-step matrix and the random variables $X_0,U_1,U_2,\dots$, we define  our chain $X=(X_n)_{n\in\n}$ \glsadd{Xn}recursively by running Algorithm~\ref{dtmcalg} below. It goes as follows:  sample  a state $x$ from the initial distribution $\gamma$ and start the chain at $x$ (i.e., set $X_0:=x$). Next, sample $y$ from the probability distribution $p(x,\cdot)$ in \eqref{eq:1step} and update the chain's state to $y$. Repeat these steps starting from $y$ instead of $x$. All random variables sampled must be independent of each other. 

\begin{algorithm}[h]
\begin{algorithmic}[1]
\FOR{$n=1,2,\dots$}
\STATE{sample $U_{n}\sim\operatorname{uni}_{(0,1)}$ independently of $(X_0,U_1,\dots,U_{n-1})$}
\STATE{$i:=0$}
\WHILE{$U_n> \sum_{j=0}^ip(X_{n-1},x_j)$}
\STATE{$i:=i+1$}
\ENDWHILE
\STATE{$X_n:=x_i$}
\ENDFOR
 \end{algorithmic}
 \caption{An algorithm to construct discrete-time chains on $\s=\{x_0,x_1,\dots\}$}\label{dtmcalg}
\end{algorithm}

\subsubsection*{The underlying space}

In the above, we glossed over the details of the space $(\Omega,\cal{F})$, the measures $(\Pb_x)_{x\in\s}$, and the random variables $X_0,U_1,U_2,\dots$ We take a moment here to argue their existence. To do so, we require the following.
\begin{theorem}\label{seqindp}For each natural number $n$, let $\rho_n$ be a probability measure on a measurable space $(\Omega_n,\cal{G}_n)$. Define the sequence space
$$\Omega=\prod_{n=0}^\infty\Omega_n,$$
the coordinate functions 
$$W_n:\Omega\to\Omega_n\qquad W_n(\omega):=\omega_n\qquad\forall \omega\in\Omega,\enskip n\geq0,$$
and let $\cal{F}$ be the sigma-algebra on $\Omega$ generated by $\prod_{n=0}^\infty \cal{G}_n$. There exists a unique probability measure $\Pb$ on $(\Omega,\cal{F})$ such that 
$$\Pbb{\left(\prod_{n=0}^k A_n\right)\times\left(\prod_{n=k+1}^\infty \Omega_n\right)}=\prod_{n=0}^k\rho_n(A_n),$$
for all $k\geq0$ and $A_0$ in $\cal{G}_0$, $\dots$,  $A_k$ in $\cal{G}_k$. In particular, under $\Pb$, $W_0,W_1,\dots$ is a sequence of independent random variables such that $W_n$ has law $\rho_n$  for each $n\geq0$. 
\end{theorem}
\begin{proof} This a consequence of Carath\'eodory's extension theorem, see \citep[Appendix~A9]{Williams1991}. The proof given therein is phrased for the case that $\Omega_n:=\r$ and $\cal{F}_n:=\cal{B}(\r)$ for all $n\geq0$, but it holds identically in our slightly more general setting.
\end{proof}

To construct the random variables in Algorithm \ref{dtmcalg}, let $(\Omega,\cal{F})$ be as in the theorem's premise after defining
$$\Omega_0:=\s,\quad \cal{G}_0:=2^\s,\quad\Omega_n:=(0,1),\quad \cal{G}_n:=\cal{B}((0,1)),\quad\forall n>0,$$
and set $X_0$ to be $W_0$ and $U_n$ to be $W_n$ for each $n>0$. Given any state $x$, set $\rho_0$ to be the point mass  $1_x$ at $x$ and $\rho_n$ to be the uniform measure on $(0,1)$ for each $n>0$. The theorem shows that a measure $\Pb_x$ on $(\Omega,\cal{F})$  satisfying the desired properties (1--3 listed after \eqref{eq:1step}) exists.

\subsubsection*{A word of warning} None of the probability spaces in $\{(\Omega,\cal{F},\Pb_x)\}_{x\in\s}$ will be complete. To see why, pick any (possible non-measurable) subset $A$ of $(0,1)$. If $(\Omega,\cal{F},\Pb_x)$ were to be complete for a given $x$, then, for any other $y\neq x$, the set $B:=\{y\}\times A\times \prod_{n=2}^\infty(0,1)$ would belong to $\cal{F}$ because $B$ is a subset of $\{y\}\times \prod_{n=1}^\infty(0,1)$ and the latter has $\Pb_x$ measure of zero:
$$\Pbx{\{y\}\times \prod_{n=1}^\infty(0,1)}=1_x(y)\prod_{n=1}^\infty \operatorname{uni}_{(0,1)}((0,1))=1_x(y)=0.$$
Define a real-valued function $\rho$ on $2^{(0,1)}$ via
$$\rho(A):=\Pby{\{y\}\times A\times \prod_{n=2}^\infty(0,1)},\qquad \forall A\subseteq (0,1)$$
It is straightforward to check that $\rho$ is a probability measure on $((0,1),2^{(0,1)})$ and that it coincides with the uniform distribution on $\cal{B}((0,1))$. However, this is impossible as it would contradict the existence of non-Lebesgue-measurable sets, e.g., see \citep{Tao2011}. 

The moral here is that, if we wish to work with collections of probability measures (one per possible starting state of the chain) on a \emph{single} measurable space, then we must resign ourselves to working with incomplete measures. In particular, this means that we must take some care when applying bounded and dominated convergence to sequences of random variables that only converge  $\Pb_\gamma$-almost surely (instead of pointwise). The way to mitigate this issue is to note that, for any $\cal{F}/\cal{B}(\r)$-measurable real-valued functions $V_1,V_2,\dots$  on a (possibly incomplete) space $(\Omega,\cal{F},\Pb)$, the set
\begin{equation}\label{eq:convmeas}A:=\{\omega\in\Omega:(V_n(\omega))_{n\in\zp}\text{ converges}\}=\bigcap_{k=1}^\infty\bigcup_{l=1}^\infty\bigcap_{n=1}^l\bigcap_{m=1}^l\left\{\omega\in\Omega:\mmag{V_n(\omega)-V_m(\omega)}\leq \frac{1}{k}\right\}\end{equation}
is measurable ($A\in\cal{F}$). Thus, even though the almost sure limit of $(V_n)_{n\in\zp}$ may not be measurable (or even defined everywhere), the sequence $(1_AV_n)_{n\in\zp}$ does converge everywhere and its limit is measurable (the pointwise limit of measurable functions is measurable). 
We then apply bounded or dominated convergence to $(1_AV_n)_{n\in\zp}$ instead of $(V_n)_{n\in\zp}$ and exploit that $\Pbl{A}=1$ to compute the expectation of the limit:
$$\lim_{n\to\infty}\Ebl{V_n}=\lim_{n\to\infty}\Ebl{1_AV_n}=\Ebl{\lim_{n\to\infty}1_AV_n},$$
under the usual integrability conditions for dominated convergence. In fact, many authors tacitly write `$\lim_{n\to\infty} V_n$' when they really mean `$\lim_{n\to\infty}1_AV_n$'.

\subsection{The Markov property}\label{sec:dtmarkov}
%
The Markov property states that conditioned on the present, the chain's past and future are independent. The `present' is described by the position $X_n$ of the chain $X$ at the current time-step $n$. For the time being, we limit the chain's `future' to be its position $X_{n+1}$ at the next time step $n+1$ (we will overcome this restriction in Section~\ref{sec:dtstrmarkov}). To formalise the notion of its `past', we employ the filtration generated by the chain.
\subsubsection*{The filtration generated by the chain}Throughout the book, we use $(\cal{F}_n)_{n\in\n}$ to denote the \emph{filtration}\index{filtration generated by the chain} generated by $X$\glsadd{filtrationdt}:

\begin{definition}[Filtration generated by the chain]\label{def:filt} The filtration  $(\cal{F}_n)_{n\in\n}$ generated by $X$ is the increasing sequence of sigma-algebras defined by
$$\cal{F}_n:=\text{ the sigma-algebra generated by }(X_0,X_1,\dots,X_n)\quad\forall n\geq0,$$
meaning that $\cal{F}_n$ is the smallest sigma-algebra such that $X_0,\dots,X_n$ are $\cal{F}_n$-measurable random variables. 
\end{definition}

An event $A$ belongs to $\cal{F}_n$ if and only if observing the chain up until time $n$ is sufficient to deduce whether $A$ has occurred. Formally, this is the case if and only if we are able to express the event's indicator function in terms of $X_0,X_1,\dots,X_n$:
$$1_A(\omega)=g(X_0(\omega),X_1(\omega),\dots,X_n(\omega))\quad\forall \omega\in\Omega,$$
for some $g:\s^n\to\{0,1\}$. Similarly, a random variable $Y$ is $\cal{F}_n$-measurable if and only if it can be expressed in terms of $X_0,\dots,X_n$ (i.e.\ if we can compute $Y(\omega)$ from $X_0(\omega),\dots,X_n(\omega)$, for each $\omega$ in $\Omega$). For these reasons, $\cal{F}_n$ represents all information that can be gleaned from observing the chain up until time $n$ and the filtration is viewed as the history of the chain.

\subsubsection*{The Markov property}The process $X$ defined via Algorithm \ref{dtmcalg} is referred to as a Markov chain because it satisfies the Markov property:\index{Markov property}
\begin{theorem}[The Markov property]\label{markovd} For any initial distribution $\gamma$,
\begin{equation}\label{eq:markovd0}\Pbl{\left.\{X_{n+1}=y\}\right|\cal{F}_n}=\Pbl{\left.\{X_{n+1}=y\}\right|X_n}=p(X_n,y)\quad\forall y\in\s,\enskip n\geq0,\enskip  \Pb_\gamma\text{-a.s.},\end{equation}
where $(\cal{F}_n)_{n\in\n}$ denotes the filtration generated by the chain (Definition~\ref{def:filt}).
\end{theorem}
\begin{proof} Enumerate the state space $\{x_0,x_1,\dots\}$ as in Algorithm~\ref{dtmcalg}. We will show that
\begin{equation}\label{eq:fja70h3n8ayfnya3w}\Pbl{\left.\{X_{n+1}=x_i\}\right|\cal{F}_n}=p(X_n,x_i)\quad\forall i\geq0,\quad \Pb_\gamma\text{-almost surely}.\end{equation}
The other inequality in \eqref{eq:markovd0} then follows by taking expectations conditional on $X_n$ and using the tower property of conditional expectation (Theorem~\ref{thrm:condexpprops}$i,iv$).  Because the collection of events 
$$\{\{X_0=z_0,\dots,X_{n-1}=z_{n}\}:z_0,\dots,z_n\in\s\}$$
generates $\cal{F}_{n}$, to prove \eqref{eq:fja70h3n8ayfnya3w} it suffices to show that
$$\Pbl{\{X_0=z_0,\dots X_n=z_n,X_{n+1}=x_i\}}=\Ebl{1_{\{X_0=z_0,\dots X_n=z_n\}}p(X_n,x_i)}\quad\forall z_0,\dots,z_n\in\s,\enskip i\geq0.$$
The chain's definition in Algorithm~\ref{dtmcalg} implies that 
$$\{X_n=z_n,X_{n+1}=x_i\}=\left\{X_{n}=z_n,\sum_{j=0}^{i-1}p(z_n,x_j)\leq U_{n+1}<\sum_{j=0}^ip(z_n,x_j)\right\}\quad\forall z_n\in\s,\enskip i\geq0.$$
Because $X_1,\dots,X_n$ are functions of $(X_0,U_1,\dots,U_{n})$ and $U_{n+1}$ is independent of $(X_0,U_1,\dots,U_{n})$, $X_n$ and $U_{n+1}$ are independent. For this reason, \eqref{eq:fja70h3n8ayfnya3w} follows from the above:
\begin{align*}\Pb_\gamma(\{X_0=z_0&,\dots X_n=z_n,X_{n+1}=x_i\})\\
&=\Pbl{\{X_0=z_0,\dots X_n=z_n\}}\Pbl{\left\{\sum_{j=0}^{i-1}p(z_n,x_j)\leq U_{n+1}<\sum_{j=0}^ip(z_n,x_j)\right\}}\\
&=\Pbl{\{X_0=z_0,\dots X_n=z_n\}}p(z_n,x_i)\\
&=\Ebl{1_{\{X_0=z_0,\dots X_n=z_n\}}p(X_n,x_i)}\quad\forall z_0,\dots,z_n\in\s,\enskip i\geq0.\end{align*}
\end{proof}
In this time-homogeneous setting\footnote{($b$) falters in the time-inhomogeneous case.}, Markov property really means two things: conditioned on the present, ($a$)  the chain's future and past are independent  and ($b$)  the chain starts `afresh' from its current location. 

To see  ($a$), note that $\cal{F}_n$ is generated by $\cal{F}_{n-1}$ and $X_n$. For this reason, the first equality in \eqref{eq:markovd0} tells us that the chain's future (for now, modelled by $X_{n+1}$) conditioned on its past ($\cal{F}_{n-1}$) and present ($X_n$) equals that conditioned on only its present. In other words, if we are aware of the chain's present,  knowledge of its past does not improve our ability to predict its future. Because 
\begin{align*}\Pbl{\text{future}|\text{past and present}}&=\Pbl{\text{future}|\text{present}}\\
&\Leftrightarrow\\
\Pbl{\text{past and future}|\text{present}}&=\Pbl{\text{past}|\text{present}}\Pbl{\text{future}|\text{present}},\end{align*}
it follows that the chain's past and future are independent ($a$). For a formal proof, use
\begin{proposition}[Characterising conditional independence,~{\citealp[Proposition~6.6]{Kallenberg2001}}]\label{prop:condind} For any probability triplet $(\Omega,\cal{F},\Pb)$ and sigma-algebras $\cal{G},\cal{H},\cal{I}\subseteq\cal{F}$,
\begin{align*}\Pbb{I|\cal{GH}}=\Pbb{I|\cal{H}}&\quad\enskip\Pb\text{-almost surely,}\enskip\forall I\in\cal{I}\\
&\Leftrightarrow\\
\Pbb{G\cap I|\cal{H}}=\Pbb{G|\cal{H}}\Pbb{I|\cal{H}}&\quad\enskip\Pb\text{-almost surely,}\enskip\forall G\in\cal{G},\enskip I\in\cal{I},\end{align*}
where $\cal{GH}$ denotes the sigma-algebra generated by $\cal{G}$ and $\cal{H}$.
\end{proposition}
\noindent to show that \eqref{eq:markovd0} holds for all $x$ in $\s$ if and only if $\cal{F}_{n-1}$ and $X_{n+1}$ are conditionally independent:
\begin{equation}\label{eq:fneway8fne6a7f}\Pbl{A\cap\{X_{n+1}=y\}|X_n}=\Pbl{A|X_n}\Pbl{\{X_{n+1}=y\}|X_n}\quad\Pb_\gamma\text{-almost surely},\end{equation}
for any given $A$ in $\cal{F}_{n-1}$ and $y$ in $\s$. 

To see ($b$), notice that setting $n:=0$ and $\gamma:=1_x$ in \eqref{eq:markovd0} and taking expectations with respect to $\Pb_x$, we find that $p(x,y)=\Pbx{\{X_1=y\}}$. Plugging this back into \eqref{eq:markovd0}, we find that 
\begin{equation}\label{eq:markovstats}\Pbl{\left.\{X_{n+1}=y\}\right|X_n}=\Pb_{X_n}(\{X_{1}=y\})\quad\forall y\in\s,\enskip\Pb_\gamma\text{-almost surely}\end{equation}
That is, the statistics of the chain's future ($X_{n+1}$) conditioned on $X_n$ are identical to those of the chain started at $X_n$. Of course, $X_{n+1}$ is only a small sliver of the chain's future. As we will show in Section~\ref{sec:dtstrmarkov}, ($a$) and ($b$) also hold for the entirety of the chain's future. For the time being however, $X_{n+1}$ is all we need.

Before proceeding, it is worth pointing out that the second equality in \eqref{eq:markovd0} justifies the name `one-step matrix' afforded to $P$: $p(x,y)$ is the probability that the chain next transitions to $y$ if it currently lies at $x$, regardless of when this transition occurs and of the chain's history up until this point. 

We finish the section with one final re-writing of \eqref{eq:markovd0} that will prove of great use later on: For any bounded  real-valued function $f$ on $\s$ and natural number $n$,
\begin{equation}\label{eq:markovd}\Ebl{\left.f(X_{n+1})\right|\cal{F}_n}=\Ebl{\left.f(X_{n+1})\right|X_{n}}=Pf(X_{n})\quad \Pb_\gamma\text{-almost surely},\end{equation}
where
$$Pf(x):=\sum_{y\in\s}p(x,y)f(y)\quad\forall x\in\s.$$
\begin{exercise}\label{ex:markovd}Show that \eqref{eq:markovd0} holds if and only if \eqref{eq:markovd} holds for all bounded real-valued functions $f$ on $\s$. Hint: express $f$ as $f\vee0-f\wedge0$ and use Theorem~\ref{thrm:condexpprops}$ii$.
\end{exercise}
\ifdraft

\subsubsection*{Notes}The idea of filtrations is due to Joseph L. Doob \citep{Doob1953}.

\fi

\subsection{Gambler's ruin}\label{sec:gamblers} At various points in the book, we illustrate aspects of the theory using the famous gambler's ruin problem: A gambler, say her name is Alice, plays a game with multiple rounds. In each round, a coin is tossed. If it lands on heads, Alice wins a pound. Otherwise, she loses one. The coin may be biased so that the probability that it lands on heads is $a\in(0,1)$. Alice has bad credit; if she goes broke, no one will lend her money and she will no longer be allowed to play. Let $X_n$ denote Alice's wealth right after the $n$th round of betting and $X_0$ denote her initial wealth. Clearly, $(X_n)_{n\in\n}$ is a Markov chain with one-step matrix\index{gambler's ruin}
\begin{align*}p(0,y)=1_0(y)\quad\forall y\geq0,\qquad p(x,y)=(1-a)1_{x-1}(y)+a1_{x+1}(y),\quad\forall x>0,\enskip y\geq0.\end{align*}

\ifdraft
{\color{red} Move after time-varying law and work out time-varying law?}
\fi

\subsection{The time-varying law and its difference equation}\label{sec:dtlawsintro}

For any state $x$ in $\s$, let 
\begin{equation}\label{eq:timevardt}p_n(x):=\Pbl{\{X_n=x\}}\end{equation}
denote\glsadd{pn} the probability that the chain in $x$ at time $n$ if its initial position was sampled from  $\gamma$. Taking expectations of \eqref{eq:markovd0}, we find that the \emph{time-varying law}\index{time-varying law} $p=(p_n)_{n\in\n}=((p_n(x))_{x\in\s})_{n\in\n}$ of $X$ is the only solution to the difference equation
\begin{equation}\label{eq:dtlaw}p_{n+1}(x)=\sum_{x'\in\s}p_n(x')p(x',x)\quad\forall x\in\s,\enskip n\in\n,\qquad p_0(x)=\gamma(x)\quad\forall x\in\s,\end{equation}
or, in matrix notation,
$$p_{n+1}=p_n P\quad\forall n\in\n,\qquad p_0=\gamma.$$
This is the first of the \emph{analytical characterisations} that we will encounter in this book. We have taken an object ($p$) defined probabilistically \eqref{eq:timevardt} in terms of the chain ($X$) and the underlying probability measure ($\Pb_\gamma$) and derived an equivalent description of the object involving  analytical equations and inequalities but no probability. This non-probabilistic description is what we call the analytical characterisation of $p$. As we will see throughout the book, these types of characterisations are a recurring theme  in the Markov chain theory. For ease of reference, I re-state the above as:
\begin{theorem}[Analytical characterisation of the time-varying law]\label{dtlawchar} The time-varying law $p=(p_n)_{n\in\n}$ of $X$ defined by \eqref{eq:timevardt} is the unique solution of \eqref{eq:dtlaw}.
\end{theorem}

Iterating \eqref{eq:markovd0} and using the tower and take-out-what-is-known properties of conditional expectation (Theorem~\ref{thrm:condexpprops}$iv,v$), we obtain an expression for the joint distribution of $X_0$, $X_1$, $\dots$, $X_n$:
\begin{equation}\label{eq:nstepthe}\Pbl{\{X_0=x_0,X_1=x_1,\dots,X_n=x_n\}}=\gamma(x_0)p(x_0,x_1)\dots p(x_{n-1},x_n),\end{equation}
for all  natural numbers $n$ and states $x_0,x_1,\dots,x_n$. Marginalising the above, we obtain expressions for the joint law of any finite subset of $(X_n)_{n\in\n}$ (the collection of these distributions is often referred to as the \emph{finite-dimensional distributions}\index{finite-dimensional distributions}). In particular, we find
$$p_n(x)=\sum_{x'\in\s}\gamma(x')p_n(x',x)\quad\forall x\in\s,$$
or $p_n=\gamma P_n$ in matrix notation, for the distribution $p_n=(p_n(x))_{x\in\s}$ of the chain at time $n$ in terms the initial distribution $\gamma$ and of the \emph{$n$-step matrix}\index{n-step matrix} $P_n=(p_n(x,y))_{x,y\in\s}$\glsadd{Pn} defined by:
\begin{equation}\label{eq:nstepdef}p_n(x,y):=\sum_{x_1\in\s}\dots\sum_{x_{n-1}\in\s}p(x,x_1)\dots p(x_{n-1},y)\quad\forall x,y\in\s\enskip n>0.\end{equation}
Out of notational convenience, we use $P_0$ to denote the identity matrix $(1_x(y))_{x,y\in\s}$ on $\s$. The nomenclature here is motivated by the equation
\begin{equation}\label{eq:fme7a80fjawfeawnmiufa}\Pbl{\{X_{n+m}=y\}|X_n}=p_m(X_n,y)\quad\forall y\in\s,\enskip n,m\geq 0,\enskip\Pb_\gamma\text{-almost surely.}\end{equation}
obtained by replacing in \eqref{eq:nstepthe} `$n$'  with `$n+m$' and `$x_{n+m}$' with `$y$', multiplying both sides by $1_{\{X_n=x_n\}}$, summing over all $x_0,\dots,x_{n-1},x_{n+1},\dots,x_{n+m-1}$ in $\s$, and taking expectations. In short, $p_m(x,y)$ is the probability that the chain will be at $y$ in $m$ steps if it currently lies in $x$.
\begin{exercise}Prove \eqref{eq:fme7a80fjawfeawnmiufa}.\end{exercise}

A good point to end this section is the celebrated \emph{Chapman-Kolmogorov equation}\index{Chapman-Kolmogorov equation}: for any natural numbers $n$ and $m$,
\begin{equation}
\label{eq:chap-kol}p_{n+m}(x,y)=\sum_{z\in\s}p_n(x,z)p_m(z,y)\quad\forall x,y\in\s,
\end{equation}
or $P_{n+m}=P_nP_m$ in matrix notation. It follows directly from the definition of the $n$-step matrix in \eqref{eq:nstepdef}.
\subsection{The path space and path law}\label{sec:pathlaw}Up until now, we have viewed the chain as a sequence of random variables each taking values in the state space $\s$. Sometimes, it is very useful to instead think of it  as a single random variable taking values in the space of all possible sequences of states (the \emph{path space})\index{path space}\glsadd{Pspace}:
$$\cal{P}:=\s^\n=\s\times\s\times\dots,$$
formally, the space of all functions from $\n$ to $\s$. To do so, we need to assign a sigma-algebra to $\s^\n$. We pick the \emph{cylinder sigma-algebra}\index{cylinder sigma-algebra}\glsadd{E} $\cal{E}$ generated by the collection of subsets of $\cal{P}$ of the form
\begin{equation}\label{eq:gensets}\{x_0\}\times\{x_1\}\times\dots\times\{x_n\}\times \s\times\s\times\dots,\end{equation}
where $n$ is any natural number and $x_0,x_1,\dots,x_n$ are any states in $\s$. Let $Y:=(Y_n)_{n\in\n}$ be a sequence of random variables $Y_n$ taking values in $\s$ and defined on some underlying measurable space $(\Omega,\cal{F})$. A simple exercise shows that the collection $Y$ (viewed as a function from $\Omega$ to $\cal{P}$) is $\cal{F}/\cal{E}$-measurable (i.e.\ a random variable taking values in $\cal{P}$) if and only if each $Y_n$ is $\cal{F}/\tws$-measurable where $2^\s$ denotes the power set of $\s$. In particular, the chain $X:=(X_n)_{n\in\n}$ defined in Section~\ref{sec:dtdef} is $\cal{F}/\cal{E}$-measurable. For this reason,
\begin{equation}\label{eq:pathlaw}\mathbb{L}_\gamma(A):=\Pbl{\{X\in A\}}\quad\forall A\in\cal{E},\end{equation}
is a well-defined probability measure on $(\cal{P},\cal{E})$ known as the \emph{path law}\index{path law}\glsadd{Lgam} of $X$ (some authors simply say the \emph{law} of $X$), where 
$$\{X\in A\}:=\{\omega\in\Omega:X(\omega)\in A\}$$
denotes the preimage of $A$ under $X$. Just as with $\Pb_x$ and $\Eb_x$, we write throughout the book $\mathbb{L}_x$\glsadd{Lx} as a shorthand for $\mathbb{L}_\gamma$ with $\gamma=1_x$.
\begin{exercise}Convince yourself that $\mathbb{L}_\gamma$ is a probability measure on $(\cal{P},\cal{E})$. \end{exercise}
For example, \eqref{eq:nstepthe} shows that
\begin{equation}
\label{eq:fme78ahfnea78fheau9fhna}\mathbb{L}_\gamma(\{x_0\}\times\{x_1\}\times\dots\times\{x_n\}\times \s\times\s\times\dots)=\gamma(x_0)p(x_0,x_1)\dots p(x_{n-1},x_n)
\end{equation}
for all natural numbers $n$ and states $x_0,x_1,\dots,x_n$. Remarkably, the above is all we need to fully specify $\mathbb{L}_\gamma$:
\begin{theorem}\label{pathlawuni}The path law $\mathbb{L}_\gamma$ defined in \eqref{eq:pathlaw} is the only measure on $(\cal{P},\cal{E})$ satisfying~\eqref{eq:fme78ahfnea78fheau9fhna} for all natural numbers $n$ and states $x_0,x_1,\dots,x_n$.
\end{theorem}

To argue the above, we need the following consequence of Dynkin's $\pi$-$\lambda$ lemma: 
\begin{lemma}[We lose nothing by working with $\pi$-systems that generate a sigma-algebra instead of the entire sigma-algebra,~{\citealp[Lemma~1.6]{Williams1991}}]\label{lem:dynkinpl} Suppose that $\cal{H}$ is a $\pi$-system on a set $\Omega$, that is, $\cal{H}$ is set of subsets of $\Omega$ that is closed under finite intersections:\index{$\pi$-system}
$$A,B\in\cal{H}\Rightarrow A\cap B\in\cal{H}.$$
If $\cal{H}$ generates a sigma-algebra $\cal{G}$ and $\mu_1,\mu_2$ are two measures on $(\Omega,\cal{G})$ satisfying $\mu_1(\Omega)=\mu_2(\Omega)<\infty$ and $\mu_1(A)=\mu_2(A)$ for all $A$ in $\cal{H}$, then $\mu_1(A)=\mu_2(A)$ for all $A$ in $\cal{G}$.
\end{lemma}

\begin{exercise}Using Lemma~\ref{lem:dynkinpl} prove Theorem~\ref{pathlawuni}. Hint: show that 
$$\{\{x_0\}\times\{x_1\}\times\dots\times\{x_n\}\times \s\times\s\times\dots: x_0,x_1,\dots,x_n\in\s,\enskip n\in\n\}\cup\emptyset$$
is a $\pi$-system.
\end{exercise}

\subsection{The martingale characterisation} \label{sec:martin}

Martingales play an important role in most areas of modern probability and Markov theory is no exception. Indeed, if we look at them from the right angle, it is not too difficult to see that Markov chains can be equivalently defined in terms of martingales. Here is that angle:
\begin{theorem}[Martingale characterisation]\label{thrm:martin}Suppose\index{martingale characterisation} that $X=(X_n)_{n\in\n}$ is a sequence of random-variables on a probability triplet $(\Omega,\cal{F},\Pb)$ that take values in $\s$ and let $(\cal{F}_n)_{n\in\n}$ denote the filtration generated by the sequence (Definition~\ref{def:filt}). The sequence satisfies the Markov property~\eqref{eq:markovd0} (with $\Pb$ replacing $\Pb_\gamma$) if and only if 
\begin{equation}\label{eq:martthe}M_n:=g(n)f(X_n)-g(0)f(X_0)-\sum_{m=0}^{n-1}(g(m+1)Pf(X_m)-g(m)f(X_m))\quad\forall n\geq0,\end{equation}
defines an $(\cal{F}_n)_{n\in\n}$-adapted $\Pb$-martingale, for every bounded real-valued function $f$ on $\s$ and real-valued function $g$ on $\n$. In particular, if $X$ is the chain introduced in Section~\ref{sec:dtdef}, then \eqref{eq:martthe} defines an $(\cal{F}_n)_{n\in\n}$-adapted $\Pb_\gamma$-martingale for every bounded $f,g$ and initial distribution $\gamma$.
\end{theorem}
\begin{proof}Clearly, $M_n$ is $\cal{F}_n$-adapted, and there are no integrability issues since $g$, $f$, and, consequently, $Pf$ are bounded functions. 
Suppose that $(X_n)_{n\in\n}$ satisfies the Markov property. To prove that \eqref{eq:martthe} defines a martingale we need to show that
$$\Ebb{\left.M_n\right|\cal{F}_{n-1}}=M_{n-1},\quad\forall n\geq0,\enskip \Pb\text{-almost surely.}$$
Applying  \eqref{eq:markovd}, we obtain
$$\Ebb{\left.g(n)f(X_n)\right|\cal{F}_{n-1}}=g(n)\Ebb{\left.f(X_n)\right|\cal{F}_{n-1}}=g(n)Pf(X_{n-1}),\quad \Pb_\gamma\text{-almost surely;}$$
and so,
$$\Ebb{\left.M_n\right|\cal{F}_{n-1}}=\Ebl{\left.g(n)f(X_n)\right|\cal{F}_{n-1}}-g(n)Pf(X_{n-1})+M_{n-1}^f=M_{n-1}^f,\quad \Pb_\gamma\text{-almost surely.}$$

Conversely, suppose that \eqref{eq:martthe} holds for every bounded $f,g$.  Pick any $y$ in $\s$ and $n$ in $\n$ and set $f:=1_y$ and $g:=1$ in \eqref{eq:martthe} so that
$$M_{n+1}-M_{n}=1_y(X_{n+1})-p(X_n,y).$$
Conditioning on $\cal{F}_n$ and taking expectations, we find that
$$\Pbb{\{X_{n+1}=y\}|\cal{F}_n}=p(X_n,y)\quad\Pb\text{-almost surely}.$$
The Markov property~\eqref{eq:markovd0} follows by conditioning on $X_n$, taking expectations, and applying the tower property of conditional expectation (in particular, Theorem~\ref{thrm:condexpprops}$i,iv$).
\end{proof}

\subsection{Other definitions of the chain*}\label{sec:otherdef}Armed with the path space, path law, and martingale characterisation, we can now answer a question we have so far avoided: our definition of the chain via Algorithm~\ref{dtmcalg} seems rather specific; does this matter?

In short, no, it does not. In long, depending on what text you pick up, you will find a discrete-time chain $(X_n)_{n\in\n}$ on a countable state space $\s$ with one-step matrix $P$ and initial distribution $\gamma$ defined as a sequence of random variables taking values in $\s$, defined on a probability triplet $(\Omega,\cal{F},\Pb_\gamma)$, and satisfying $\Pbl{\{X_0=\cdot\}}=\gamma(\cdot)$ and
\begin{enumerate}[label=(\alph*)]
\item the Markov property \eqref{eq:markovd0}, where $(\cal{F}_n)_{n\in\n}$ denotes the filtration generated by the sequence (Definition \ref{def:filt}),
\item or \eqref{eq:nstepthe} for every  $n$ in $\n$ and all states $x_0,\dots,x_n$ in $\s$,
\item or the martingale property: \eqref{eq:martthe} defines an $\cal{F}_n$-adapted $\Pb_\gamma$-martingale for every bounded real-valued function $f$ on $\s$ and real-valued function $g$ on $\n$.
\end{enumerate}
The above three statements are equivalent:
\begin{theorem}[Equivalent definitions of discrete-time chains]\label{samedef}If $X:=(X_n)_{n\in\n}$ is a sequence of $\s$-valued random-variables on a probability triplet $(\Omega,\cal{F},\Pb_\gamma)$ with $\Pbl{\{X_0=x\}}=\gamma(x)$ for all $x$ in $\s$, then statements $(a)$, $(b)$, or $(c)$ are equivalent.
\end{theorem}
\begin{proof} Previously, we already showed that $(a)\Rightarrow (b)$ (Section~\ref{sec:dtlawsintro}) and $(a)\Leftrightarrow (c)$ (Theorem~\ref{thrm:martin}) and, so, it suffices to show here that $(b)\Rightarrow (a)$. If we can argue that
\begin{equation}\label{eq:iureasohfyuia2}\Ebb{f(X_{n+1})|\cal{F}_n}=Pf(X_n)\qquad\Pb_\gamma\text{-almost surely}\end{equation}
for any bounded real-valued function $f$ on $\s$, then the rest of the $(a)$ follows by conditioning on $X_n$, taking expectations, and applying the tower property (Theorem~\ref{thrm:condexpprops}$i,iv$). Picking any $x_0,\dots,x_n$ in $\s$ and applying \eqref{eq:nstepthe} we obtain
\begin{align*}\Ebb{f(X_{n+1})1_{\{X_0=x_0\}}\dots1_{\{X_n=x_n\}}}&=\sum_{x\in\s}f(x)\Ebb{1_{\{X_0=x_0\}}\dots1_{\{X_n=x_n\}}1_{\{X_{n+1}=x\}}}\\
&=\sum_{x\in\s}f(x)\gamma(x_0)p(x_0,x_1)\dots p(x_n,x)\\
&=\gamma(x_0)p(x_0,x_1)\dots p(x_{n-1},x_n)Pf(x_n)\\&=\Ebb{Pf(X_{n})1_{\{X_0=x_0\}}\dots1_{\{X_n=x_n\}}},\end{align*}
and \eqref{eq:iureasohfyuia2} follows because the collection of events $\{\{X_0=x_0,\dots,X_{n}=x_{n}\}:x_0,\dots,x_n\in\s\}$ generates $\cal{F}_n$.
\end{proof}

Because the path law is fully specified by~\eqref{eq:fme78ahfnea78fheau9fhna} (Theorem~\ref{pathlawuni}), Theorem \ref{samedef} shows that the path law~\eqref{eq:pathlaw} of a discrete-time chain with one-step matrix $P$ and initial distribution $\gamma$ is the same  regardless of whether we use (a), (b), or (c) above in its definition. For this reason, the particular construction of the chain does not matter as long as we are only interested in questions that can be answered by observing the entire path of the chain. 
%
%
I focus on the construction given in Algorithm~\ref{dtmcalg} because I find it straightforward to work with and because the algorithm itself is an easy-to-implement recipe for simulating chains in practice.

\subsubsection*{A final observation: the time-varying law does not characterise the path law}In contrast with (a--c) above, the time-varying law $(p_n)_{n\in\n}$ of the chain in \eqref{eq:timevardt} is \emph{not} enough to uniquely define a distribution on the path space. For instance, using Theorem \ref{seqindp} we can build a sequence of independent random variables $(W_n)_{n\in\n}$ such that the law of $W_n$ is $p_n$ for each $n$ in $\n$---something absurd for all but the most trivial of chains (think of \eqref{eq:markovd0}).

\subsection{Stopping times}\label{sec:stop}
\index{stopping time}\glsadd{varsigma} Given the filtration $(\cal{F}_n)_{n\in\n}$ generated by the chain (Definition \ref{def:filt}), a random variable is said to be an $(\cal{F}_n)_{n\in\n}$-stopping time if it is the time at which some  event occurs with the event being such that we are able to deduce whether it has occurred by any given point in time if we have observed the chain's path up until (and including) said point. For example, the  first (or second, or $k$th) time that the chain visits a given set of interest is a stopping time, but the last time it visits the set is not. Formally:

\begin{definition}[Discrete-time stopping times]\label{def:stopdt}Given the filtration $(\cal{F}_n)_{n\in\n}$ generated by the chain, a random variable $\varsigma:\Omega\to\n_E$ is said to be an $(\cal{F}_n)_{n\in\n}$-stopping time if and only if the event $\{\varsigma\leq n\}$ belongs to $\cal{F}_n$, for each $n$ in $\n$. With a stopping time $\varsigma$, we associate the sigma-algebra $\cal{F}_\varsigma$ defined by
$$\forall A\in\cal{F},\quad A\in\cal{F}_\varsigma\Leftrightarrow A\cap\{\varsigma\leq n\}\in\cal{F}_n \quad \forall n\geq0.$$
\end{definition}

\begin{exercise}Convince yourself that $\cal{F}_\varsigma$ is a sigma-algebra.
\end{exercise}

The name `stopping time' afforded to these random variables stems from the occurrence such an event being used in practice to signal that one should stop what they are doing and take a certain action. For instance, our gambler Alice (Section~\ref{sec:gamblers}) would probably be wise in walking away from the game the moment her winnings reached an amount she finds satisfactory.
In this vein, the chain is often said to \emph{stop} at $\varsigma$.

The \emph{pre-$\varsigma$ sigma-algebra}\index{pre-$\varsigma$ sigma-algebra} $\cal{F}_\varsigma$\glsadd{Fsig} formalises the chain's history up until (and including) $\varsigma$: it is the collection of events whose occurrence we are able to deduce by tracking the chain's position up until the stopping time $\varsigma$. For instance, if $\varsigma$ is the second time that the chain visits a given subset $A$, then the event that the chain visits $A$ once (or at least once, or twice, or at least twice) belongs to $\cal{F}_\varsigma$  but the event that the chain visits $A$ three (or four, or at least four, or $\dots$) does not belong to $\cal{F}_\varsigma$.  From this vantage point, the following useful lemma is almost trivial:

\begin{lemma}\label{lem:2stop}If $\varsigma$ and $\vartheta$ are two $(\cal{F}_n)_{n\in\n}$-stopping times, the 
\begin{enumerate}[label=(\roman*),noitemsep]
\item The events $\{\varsigma\leq \vartheta\}$, $\{\varsigma< \vartheta\}$, and $\{\varsigma= \vartheta\}$ belong to $\cal{F}_\vartheta$.
\item Given any event $A$ in $\cal{F}_\varsigma$, the event $A\cap\{\varsigma\leq \vartheta\}$ belongs to $\cal{F}_\vartheta$.
\item In particular, if $\varsigma\leq\vartheta$, then $\cal{F}_\varsigma\subseteq\cal{F}_\vartheta$.
\end{enumerate}
\end{lemma}

Part~$(i)$ states that we are able to decide whether the event associated with $\varsigma$ has occurred by (or before, or at) $\vartheta$ if we have observed the chain up until $\vartheta$. Part~$(ii)$ states that anything we are able to deduce from observing the chain up until $\varsigma$, we are able to deduce from observing the chain up until  $\vartheta$ as long as $\varsigma$ is no greater than $\vartheta$. Part~$(iii)$ then follows directly from $(ii)$: if $\vartheta$ is always greater than $\varsigma$, then we are able to deduce by time $\vartheta$ anything we can deduce by time $\varsigma$. 

\begin{proof}$(i)$ Fix any natural number $n$ and note that
\begin{align*}\{\varsigma<\vartheta\}\cap\{\vartheta\leq n\}&=\bigcap_{m=0}^{n-1}\{\varsigma\leq m\}\cap\{m<\vartheta\}\cap\{\vartheta\leq n\}=\bigcap_{m=0}^{n-1}\{\varsigma\leq m\}\cap(\Omega\backslash\{\vartheta\leq m\})\cap\{\vartheta\leq n\}.
\end{align*}
Because sigma-algebras are closed under finite intersections and complements, the above belongs to $\cal{F}_n$ and it follows that  $\{\varsigma<\vartheta\}$ belongs to $\cal{F}_\vartheta$. An analogous argument shows that $\{\vartheta<\varsigma\}$ also belongs to $\cal{F}_\vartheta$. Given that sigma-algebras are closed under complements and intersections, it follows that $\{\vartheta\leq\varsigma\}=\Omega\backslash\{\varsigma<\vartheta\}$ and $\{\varsigma=\vartheta\}=\{\vartheta\leq\varsigma\}\cap (\Omega\backslash\{\vartheta<\varsigma\})$ also belong to $\cal{F}_\vartheta$.

$(ii)$ For any natural number $n$,
$$A\cap\{\varsigma\leq \vartheta\}\cap\{\vartheta\leq n\}=(A\cap\{\varsigma\leq n\})\cap(\{\varsigma\leq \vartheta\}\cap\{\vartheta\leq n\}).$$
Because $(i)$ shows that $\{\varsigma\leq \vartheta\}$ belongs to $\cal{F}_{\vartheta}$, and because sigma-algebras are closed under intersections, the above implies that $A\cap\{\varsigma\leq \vartheta\}\cap\{\vartheta\leq n\}$ belongs to $\cal{F}_n$, and the result follows.

$(iii)$ Follows directly from $(ii)$ as, in this case, $\{\varsigma\leq\vartheta\}=\Omega$.
\end{proof}

\begin{exercise}\label{exe:xvarsigma}Show that the event $\{\varsigma<\infty,X_\varsigma=x\}$ belongs to $\cal{F}_\varsigma$, for any $x$ in $\s$. Hint: notice that Lemma~\ref{lem:2stop}$i$ implies that $\{\varsigma=n\}$ belongs to $\cal{F}_n$ for all $n\geq0$.\end{exercise}
\subsubsection*{An important example: hitting times}

The \emph{hitting time}\index{hitting time}\index{first passage time|seealso{hitting time}} (or \emph{first passage time})\glsadd{sigmaa} $\sigma_A$ of a subset $A$ of the state space is the first time that the chain $X$ visits the subset (or infinity if it never does). Formally,
\begin{equation}\label{eq:hitd}\sigma_{A}(\omega):=\inf\{n\in\n:X_n(\omega)\in A\}\qquad\forall \omega\in\Omega,\end{equation}
where we are adhering to our convention that $\inf\emptyset=\infty$. We use the shorthand $\sigma_x:=\sigma_{\{x\}}$ for any $x$ in $\cal{S}$. 

Because we are able to deduce whether the chain has visited a set by observing the chain up until the (and including) the moment it does, hitting times are $(\cal{F}_n)_{n\in\n}$-stopping times:

\begin{proposition}\label{prop:hitisstop} Let $A$ be any subset of $\s$ and $(\cal{F}_n)_{n\in\n}$ denote the filtration generated by the chain (Definition~\ref{def:filt}). The hitting time $\sigma_A$ of $A$ defined in~\eqref{eq:hitd} is an $(\cal{F}_n)_{n\in\n}$-stopping time.
\end{proposition}
\begin{proof}This follows immediately from the hitting time's definition as it implies that
$$\{\sigma_A\leq 0\}=\{X_0\in A\}\in\cal{F}_0,\quad \{\sigma_A\leq n\}=\bigcup_{m=0}^n\{X_m\in A\}\in\cal{F}_n\enskip\forall n>0.$$
\end{proof}

\subsection{Dynkin's formula}\label{sec:dynkindt}As we will see in both this section and Section~\ref{sec:dtstrmarkov}, the non-predictive nature of stopping times allows us to generalise  the results we have gathered this far so that they apply to stopping times instead of just deterministic times. Here, we make use of the martingale characterisation (Theorem~\ref{thrm:martin}) to prove \emph{Dynkin's formula}: a generalisation of the recurrence~\eqref{eq:dtlaw} satisfied by the time-varying law. In particular, summing over $n$ in \eqref{eq:dtlaw}, we obtain
$$\Ebl{1_{\{X_n=x\}}}=p_n(x)=\gamma(x)+\sum_{m=0}^{n-1}(p_{m+1}P(x)-p_m(x))=\gamma(x)+\Ebl{\sum_{m=0}^{n-1}(p(X_m,x)-1_{\{X_m=x\}})}.$$
If $f$ is bounded real-valued function $\s$, then multiplying both sides by $f(x)$ and summing over $x$ in $\s$ we obtain the following the integral version of~\eqref{eq:dtlaw}:
$$\Ebl{f(X_n)}=\gamma(f)+\Ebl{\sum_{m=0}^{n-1}(Pf(X_m)-f(X_m))}.$$
Dynkin's formula\index{Dynkin's formula} tells us that the same equation holds if we replace $n$ with a stopping time.

\begin{theorem}[Dynkin's Formula]\label{thrm:dynkin} Let $f$ be a bounded real-valued function on $\s$, $g$ be a real-valued function on $\n$, and $(\cal{F}_n)_{n\in\n}$ denote the filtration generated by the chain (Definition~\ref{def:filt}). If $\varsigma$ is an $(\cal{F}_n)_{n\in\n}$-stopping time and $\varsigma_n:=\varsigma\wedge n$ for any given $n\geq0$,
$$\Ebl{g(\varsigma_n)f(X_{\varsigma_n})}=g(0)\Ebl{f(X_0)}+\Ebl{\sum_{m=0}^{\varsigma_n-1}g(m+1)Pf(X_m)-g(m)f(X_m)}.$$

\end{theorem}

The proof of Dynkin's formula consists of combining the martingale characterisation of the chain (Theorem~\ref{thrm:martin}) with Doob's optional stopping theorem:

\begin{theorem}[Doob's optional stopping]\label{thrm:doobsopt}Let $(\Omega,\cal{G},\Pb)$ be a probability triplet, $(\cal{G}_n)_{n\in\n}$ be a filtration contained in $\cal{G}$, and $(M_n)_{n\in\n}$ be a $(\cal{G}_n)_{n\in\n}$-adapted $\Pb$-martingale. For any bounded $(\cal{G}_n)_{n\in\n}$-stopping time $\varsigma$, $\Ebb{M_{\varsigma}}=\Ebb{M_0}$.
\end{theorem}

\begin{proof}Note that for any natural number $n$, $\varsigma\wedge n\neq \varsigma\wedge(n-1)$ only if $\varsigma>n-1$ in which case $\varsigma\geq n$. For this reason,
\begin{align*}\Ebb{M_{\varsigma\wedge n}-M_{\varsigma\wedge (n-1)}}&=\Ebb{1_{\{\varsigma>n-1\}}(M_{n}-M_{n-1})}=\Ebb{\Ebb{1_{\{\varsigma>n-1\}}(M_{n}-M_{n-1})|\cal{G}_n}}\\
&=\Ebb{1_{\{\varsigma>n-1\}}(\Ebb{M_{n}|\cal{G}_n}-M_{n-1})}=0.\end{align*}
where the final equality follows from the martingale property of $(M_n)_{n\in\n}$ and the others from the tower rule and take-out-what-is-known properties of conditional expectation (Theorem~\ref{thrm:condexpprops}$i,iv,v$). Thus,
$$\Ebb{M_{\varsigma\wedge n}}=\Ebb{M_{\varsigma\wedge (n-1)}}=\Ebb{M_{\varsigma\wedge (n-2)}}=\dots=\Ebb{M_{0}}\quad\forall n\geq0.$$
Because $\varsigma$ is bounded, there exists an $n$ such that $\varsigma\wedge n=\varsigma$ and the result follows.
\end{proof}

\begin{proof}[Proof of Theorem~\ref{thrm:dynkin}]Because $f$ and $g\circ \varsigma_n$ are bounded,  all random variables in the equation are well-defined and  Theorem \ref{thrm:martin} shows that $M=(M_n)_{n\in\n}$ in \eqref{eq:martthe} is an $(\cal{F}_n)_{n\in\n}$-adapted $\Pb_\gamma$-martingale. Because $\varsigma_n$ is a bounded $(\cal{F}_n)_{n\in\n}$-stopping time, Doob's optional stopping (Theorem \ref{thrm:doobsopt}) then tells us that $\Ebl{M_{\varsigma_n}}=\Ebl{M_0}=0$.
\end{proof}

\ifdraft 

\subsubsection*{Notes and references} This ``martingale problem'' approach to Markov processes, pioneered by Daniel W. Stroock and S. R. Srinivasa Varadhan in their study of diffusion processes, was popularised (at least partly) due to the efforts of Thomas G. Kurtz and his co-authors and has yielded many fruitful results. See \citep{Ethier1986} and references therein.

\fi

\subsection{Stopping distributions and occupation measures}\label{sec:morehitdt} With a stopping time $\varsigma$ we associate a \emph{stopping distribution}\index{stopping distribution}\glsadd{mu} $\mu$ and \emph{occupation measure}\index{occupation measure}\glsadd{nu} $\nu$ defined by
\begin{align}\mu(n,x):&=\Pbl{\{\varsigma=n,X_\varsigma=x\}},\qquad \forall n\geq0,\quad x\in\s,\label{eq:edisdefd}\\
\nu(n,x):&=\Ebl{\sum_{m=0}^{\varsigma-1} 1_n(m)1_{x}(X_m)}=\Ebl{1_{x}(X_n)\sum_{m=0}^{\varsigma-1} 1_n(m)}\label{eq:eoccdefd}\\
&=\Pbl{\{\varsigma>n,X_n=x\}},\qquad\forall n\geq0,\quad x\in\s.\nonumber\end{align}
In other words, $\mu(n,x)$ is the probability that the chain stops at time $n$ while in state $x$ and $\nu(n,x)$ is the probability that the chain is in state $x$ at time $n$ and that it has not yet stopped.

The mass of the stopping distribution is simply the probability that the stopping time is finite:
\begin{equation}\mu(\n,\s)=\sum_{n=0}^\infty\sum_{x\in\s} \Pbl{\{\varsigma=n,X_\varsigma=x\}}=\Pbl{\{\varsigma<\infty\}};\label{eq:mumassd}\end{equation}
while that of the occupation measure is the mean stopping time:
\begin{align}\nu(\n,\s)&=\sum_{n=0}^\infty\sum_{x\in\s}\Ebl{\sum_{m=0}^{\varsigma-1}1_n(m)1_{x}(X_m)}=\Ebl{\left(\sum_{m=0}^{\varsigma-1}\sum_{n=0}^\infty1_n(m)\right)\left(\sum_{x\in\s}1_{x}(X_m)\right)}\nonumber\\
&=\Ebl{\sum_{m=0}^{\varsigma-1} 1} =\Ebl{\varsigma}.\label{eq:numassd}\end{align}

The stopping distribution and the occupation measure are tied together by a set of linear equations:
\begin{lemma}\label{eqnsd}
The pair $(\mu,\nu)$ satisfies 
\begin{equation}\label{eq:eoed}\begin{array}{ll} \mu(0,x)+\nu(0,x)=\gamma(x),\quad &\forall x\in\s,  \\ \mu(n,x)+\nu(n,x)=\sum_{y\in\s}\nu(n-1,y)p(y,x), \quad&\forall n>0, \enskip x\in\s .\end{array}\end{equation}
\end{lemma}
\begin{proof}The first set of equations follow directly from the definitions of $\mu$ and $\nu$. The second set requires a bit more work. Pick any positive integer $n$ and state $x$ in $\s$. Setting $g:=1_n$ and $f:=1_x$  in Dynkin's formula (Theorem~\ref{thrm:dynkin}) yields
\begin{equation}\label{eq:genintbd}\Ebl{1_n(\varsigma_k)1_x(X_{\varsigma_k})}+\Ebl{\sum_{m=0}^{\varsigma_k-1}1_n(m)1_x(X_m)}=\Ebl{\sum_{m=0}^{\varsigma_k-1}1_n(m+1)p(X_m,x)}\quad\forall k>0,\end{equation}
where  $\varsigma_k$ denotes the minimum $\varsigma\wedge k$ of $\varsigma$ and   $k$.  For each $\omega$ in $`\Omega$, the sequence $(\varsigma_k(\omega))_{k\in\zp}$ is increasing and has limit $\varsigma(\omega)$. Thus, monotone convergence implies that
\begin{align}\label{eq:lim2d}&\lim_{k\to\infty}\Ebl{\sum_{m=0}^{\varsigma_k-1}1_n(m)1_x(X_m)}=\Ebl{\sum_{m=0}^{\varsigma-1}1_n(m)1_x(X_m)}, \\ \label{eq:lim3d}&\lim_{k\to\infty}\Ebl{\sum_{m=0}^{\varsigma_k-1}1_n(m+1)P(X_m,x)}=\Ebl{\sum_{m=0}^{\varsigma-1}1_n(m+1)P(X_m,x)}.\end{align}
Similarly, bounded convergence yields
\begin{align}\label{eq:lim1d}\lim_{n\to\infty}\Ebl{1_{n}(\varsigma_k)1_{x}(X_{\varsigma_k})}=\Ebl{1_{n}(\varsigma)1_{x}(X_{\varsigma})}=\Pbl{\{\varsigma=n,X_\varsigma=x\}}.\end{align}

Putting \eqref{eq:genintbd}--\eqref{eq:lim1d} together we have that
$$\Pbl{\{\varsigma=n,X_\varsigma=x\}}+\Ebl{\sum_{m=0}^{\varsigma-1}1_n(m)1_x(X_m)}=\Ebl{\sum_{m=0}^{\varsigma-1}1_n(m+1)P(X_m,x)}.$$
Using Tonelli's Theorem, we obtain the second equation in \eqref{eq:eoed}.
\end{proof}

\subsubsection*{The marginals}  The \emph{time marginal}\index{stopping distribution (time marginal)} $\mu_T$ of the stopping distribution  
\begin{equation}\label{eq:mut}\mu_T(n):=\sum_{x\in\s}\mu(n,x)=\Pbl{\bigcup_{x\in\s}\{\varsigma=n,X_\varsigma=x\}}=\Pbl{\{\varsigma=n\}},\end{equation}
is the distribution of the exit time itself. Technically, it is the distribution of $\varsigma$ restricted to $\n$ because $\varsigma$ may take the value $\infty$ (indicating that the chain never stops). However,  we recover the full distribution of $\varsigma$ from the above using $\Pbl{\{\varsigma=\infty\}}=1-\Pbl{\{\varsigma<\infty\}}=1-\mu_T(\n)$. Similarly, it is straightforward to check that the time marginal $\nu_T(n):=\nu(n,\s)$ is equal to one minus the cumulative distribution function of $\varsigma$.

The \emph{space marginals} $\mu_S$\glsadd{mus} and $\nu_S$\glsadd{nus} of the stopping distribution\index{stopping distribution (space marginal)} and occupation measure\index{occupation measure (space marginal, stopping time)},
\begin{align}
\label{eq:mus}\mu_S(x)&:=\sum_{n=0}^\infty\mu(n,x)=\Pbl{\bigcup_{n=0}^\infty\{\varsigma=n,X_\sigma=x\}}=\Pbl{\{\varsigma<\infty,X_\varsigma=x\}},\\
\label{eq:nus}\nu_S(x)&:=\sum_{n=0}^\infty\nu(n,x)=\Ebl{\sum_{m=0}^{\varsigma-1} \left(\sum_{n=0}^\infty 1_n(m)\right)1_{x}(X_m)}=\Ebl{\sum_{m=0}^{\varsigma-1} 1_{x}(X_m)},
\end{align}
tell us where the chain stops and where it spends time before stopping, respectively. Explicitly, $\mu_S(x)$ is the probability that the chain stops in $x$, while $\nu_S(x)$ denotes the expected number of visits  the chain makes to $x$ before stopping. Summing over $n$ in \eqref{eq:eoed}, we find that the space marginals of the exit distribution and occupation measure also satisfy  a set of linear equations:
\begin{corollary}\label{eqnsd2} 
The pair $(\mu_S,\nu_S)$ satisfies 
$$\mu_S(x)+\nu_S(x)=\gamma(x)+\sum_{y\in\s}\nu_S(y)p(y,x) \quad\forall  x\in\cal{S}.$$
\end{corollary}

\ifdraft
\subsubsection*{Notes and references} Corollary \ref{eqnsd2} is Proposition 4.1 in \citep{Pitman1977}.
\fi

\subsection{Exit times*}\label{sec:exit}

In applications, we are often interested in how long the chain takes to exit a given subset of the state space and what part of the domain's boundary does the chain cross to exit. To study this problem, we single out a subset $\cal{D}$ of the state space $\s$ and refer to it as the \emph{domain}\index{domain}\glsadd{D}. The \emph{exit time}\index{exit time} $\sigma$ from $\cal{D}$ is the first instant that chain first lies outside of $\cal{D}$\glsadd{sigma}:
\begin{equation}\label{eq:eddef}\sigma(\omega):=\inf{\{n\geq0:X_n(\omega)\not\in\cal{D}\}}\quad\forall\omega\in\Omega.\end{equation}
Equivalently, the exit time is the hitting time $\sigma_{\cal{D}^c}$~\eqref{eq:hitd} of the domain's complement $\cal{D}^c$. Proposition~\ref{prop:hitisstop} shows that $\sigma$ is an $(\cal{F}_n)_{n\in\n}$-stopping time~\eqref{def:stopdt}, where $(\cal{F}_n)_{n\in\n}$ denotes the filtration generated by the chain (Definition~\ref{def:filt}).

\subsubsection*{The exit distribution and occupation measure} Let $\mu$ and $\nu$ be as in~\eqref{eq:edisdefd}--\eqref{eq:eoccdefd} with $\sigma$ replacing $\varsigma$\glsadd{mu}\glsadd{nu}\index{exit distribution}\index{occupation measure (exit time)}. In this case, the $\mu(n,x)$ is the probability that the chain first exits the domain at time $n$ by moving to state $x$, and we refer to $\mu$ as the \emph{exit distribution}. Similarly, $\nu(n,x)$ denotes the probability that the chain is in state $x$ at time $n$ and that it has not yet exited the domain. Because the chain lies outside the domain at the time of exit, the support of the exit distribution $\mu$ is contained outside of the domain. Similarly, because before exiting the chain lies inside the domain, the support of the occupation measure $\nu$ is contained outside of the domain. In other words,
\begin{equation}\label{eq:suppstd}\mu(\n,\cal{D})=0,\qquad \nu(\n,\cal{D}^c)=0.\end{equation}
Combining the above with  Lemma \ref{eqnsd}, we obtain  an equivalent description of the exit distribution and occupation measure in terms of a linear recursion:
\begin{theorem}[Analytical characterisation of $\mu$ and $\nu$]\label{characttd} Let $\sigma$ denote the exit time~\eqref{eq:eddef} from the domain $\cal{D}$. Its  exit distribution $\mu$~\eqref{eq:edisdefd} and occupation measure $\nu$~\eqref{eq:eoccdefd}  are given by
\begin{equation}\label{eq:eoed2t}\begin{array}{lll}\nu(n,x)&=1_{\cal{D}}(x)\hat{p}_n(x),&\forall n>0,\quad\nu(0,x)=1_{\cal{D}}(x)\gamma(x),\quad\forall x\in\s,\\
\mu(n,x)&=1_{\cal{D}^c}(x)(\hat{p}_n(x)-\hat{p}_{n-1}(x)),&\forall n>0,\quad \mu(0,x)=1_{\cal{D}^c}(x)\gamma(x),\quad\forall  x\in\s,\end{array}
\end{equation}
where $\hat{p}_n=(\hat{p}_n(x))_{x\in\s}$ is defined by the difference equation
\begin{equation}\label{eq:hatdtimelaw}\hat{p}_{n+1}(x)=1_{\cal{D}^c}(x)\hat{p}_n(x)+\sum_{x'\in\cal{D}}\hat{p}_n(x')p(x',x)\quad\forall x\in\s,\enskip n\geq0,\qquad \hat{p}_0(x)=\gamma(x)\quad\forall x\in\s.\end{equation}
\end{theorem}

The ideas behind  Theorem~\ref{characttd} are simple: consider a second chain $\hat{X}$ which is identical to $X$ except that every state outside of the domain $\cal{D}$ has been turned into an absorbing state (i.e.\ a state the chain cannot leave). In other words, $\hat{X}$ has one-step matrix  $\hat{P}:=(\hat{p}(x,y))_{x,y\in\s}$ with
\begin{equation}\label{eq:phat}\hat{p}(x,y):=\left\{\begin{array}{ll} p(x,y)&\text{if }x\in\cal{D} \\ 1_x(y)&\text{if }x\not\in\cal{D} \end{array}\right.\quad\forall x,y\in\s.\end{equation}
Theorem~\ref{dtlawchar} shows that the time-varying law of $\hat{X}$ is $\hat{p}=(\hat{p}_n)_{n\in\n}$ in \eqref{eq:hatdtimelaw}. Suppose that we built $\hat{X}$ by running Algorithm~\ref{dtmcalg} with $\hat{P}$ replacing $P$ but keeping the same $X_0,U_1,U_2,\dots$ Because $X$ and $\hat{X}$ are updated using the same rules for as long as they remain inside the domain $\cal{D}$, the two chains are identical up until (and including) the moment that they simultaneously exit $\cal{D}$. Thus, $X$ and $\hat{X}$ leave $\cal{D}$ at the same time and via the same state. For this reason, the probability $\mu(\{0,1,\dots,n\},x)$ that $X$ has exited the domain by time $n$  via state $x$ is also the probability that $\hat{X}$ exited by $n$ via $x$. Because $\hat{X}$ gets stuck in the first state it visits outside of $\cal{D}$, it follows that $\mu(\{0,1,\dots,n\},x)$ is the probability $\hat{p}_n(x)$ that $\hat{X}$ is in state $x$ at time $n$. Using $\mu(n,x)=\mu(\{0,1,\dots,n\},x)-\mu(\{0,1,\dots,n-1\},x)$, we obtain the second equation in \eqref{eq:eoed2t}. The equation for the  occupation measure follows from similar reasons. The key here is that $\hat{X}$ is in a state belonging to the domain only if it has not yet exited. Thus, the probability $\nu(n,x)$ that $\hat{X}$ is in a state $x$ belonging to $\cal{D}$ at time $n$ and that it has not yet exited $\cal{D}$ is simply the probability $\hat{p}_n(x)$ that it is in state $x$ at time $n$.

\subsubsection*{The space marginals}Just as with the joint exit distribution $\mu$ and occupation measure $\nu$ of an exit time,  their space marginals $\mu_S$\glsadd{mus}\index{exit distribution (space marginal)} and $\nu_S$\glsadd{nus}\index{occupation measure (space marginal, exit time)} are fully specified by a set of linear equations:

\begin{theorem}[Analytical characterisation of $\mu_S$ and $\nu_S$]\label{charactd}Let $\sigma$ denote the exit time~\eqref{eq:eddef} from the domain $\cal{D}$. The space marginal $\mu_S$ in \eqref{eq:mus} of its  exit distribution  is given by 
\begin{equation}\label{eq:mueqsmd}\mu_S(x)=\left\{\begin{array}{ll}\gamma(x)+\sum_{z\in\cal{D}}\nu_S(z)p(z,x)&\forall x\not\in\cal{D}\\
0& \forall x\in\cal{D}\end{array}\right..\end{equation}
The space marginal $\nu_S$ in \eqref{eq:nus} of its occupation measure is the minimal non-negative solution to the equations
\begin{equation}\label{eq:nueqsmd}
\nu_S(x)=\left\{\begin{array}{ll}\gamma(x)+\sum_{z\in\cal{D}}\nu_S(z)p(z,x)&\forall x\in\cal{D}\\0&\forall x\not\in\cal{D}\end{array}\right..
\end{equation}
By minimal I mean that, if $\rho=(\rho(x))_{x\in\s}$ is any non-negative solution of \eqref{eq:nueqsmd}, then $\rho(x)\geq \nu_S(x)$ for all $x$ in $\cal{D}$.
\end{theorem}

\begin{proof} Given that Corollary \ref{eqnsd2} and \eqref{eq:suppstd} imply that $\mu_S$ and $\nu_S$ satisfy (\ref{eq:mueqsmd},\ref{eq:nueqsmd}), all that remains to be shown is the minimality of $\nu_S$.  Let $\rho=(\rho(x))_{x\in\s}$ be any other non-negative solution of \eqref{eq:nueqsmd}. If $x$ lies outside of the domain, $\nu_S(x)=\rho(x)=0$ and the result is trivial. Otherwise, 
\begin{align}\rho(x)&=\gamma(x)+\sum_{z_0\in\cal{D}}\rho(z_0)p(z_0,x)=\gamma(x)+\sum_{z_0\in\cal{D}}\gamma(z_0)p(z_0,x)\nonumber\\
&+\sum_{z_0\in\cal{D}}\sum_{z_1\in\cal{D}}\rho(z_0)p(z_0,z_1)p(z_1,x)\nonumber\\
&=\dots=\gamma(x)+\sum_{l=1}^{m}\sum_{z_0\in\cal{D}}\dots\sum_{z_{l-1}\in\cal{D}}\gamma(z_0)p(z_0,z_1)\dots p(z_{l-1},x)\nonumber\\
&+\sum_{z_0\in\cal{D}}\dots\sum_{z_m\in\cal{D}}\rho(z_0)p(z_0,z_1)p(z_1,z_2)\dots p(z_m,x)\nonumber\\
&\geq \gamma(x)+\sum_{l=1}^{m}\sum_{z_0\in\cal{D}}\dots\sum_{z_{l-1}\in\cal{D}}\gamma(z_0)p(z_0,z_1)\dots p(z_{l-1},x)\qquad\forall x\in\cal{D}.\label{preq:dnhyuabfesh}\end{align}
However, it follows from \eqref{eq:nstepthe} and the definition of the exit time $\sigma$~\eqref{eq:eddef} that
\begin{align*}\sum_{z_0\in\cal{D}}\dots\sum_{z_{l-1}\in\cal{D}}\gamma(z_0)p(z_0,z_1)\dots p(z_{l-1},x)&=\Pbl{\{X_0\in\cal{D},\dots,X_{l-1}\in\cal{D},X_l=x\}}\\
&=\Pbl{\{X_l=x,\sigma>l\}}\quad\forall x\in\cal{D}.\end{align*}
Using the above, we rewrite \eqref{preq:dnhyuabfesh} as
$$\rho(x)\geq \sum_{l=0}^m \Pbl{\{X_l=x,\sigma>l\}}\quad\forall l>0,\enskip x\in\cal{D}.$$
Taking the limit $m\to\infty$ and comparing the above with the definitions of the occupation measure $\nu$ and its marginal $\nu_S$ in \eqref{eq:eoccdefd} and \eqref{eq:nus}, we have the desired $\rho(x)\geq\nu_S(x)$.
\end{proof}

You may be wondering whether we can shed the minimality requirement from the characterisation \eqref{charactd}. In other words, can the equations \eqref{eq:nueqsmd} have other non-negative solutions aside from $\nu_S$? They can:

\begin{example}\label{eq:jd8wndwua8ndu8wada} Consider the exit from $\{1,2\}$ of a chain with state space $\{0,1,2\}$, initial distribution $1_1$, and one-step matrix
$$\begin{bmatrix}1&0&0\\ 1&0&0\\ 0&0&1\end{bmatrix},$$
The chain starts at $1$, jumps to $0$ in the first step, and then remains at $0$ forever. Thus, we have that the space marginal of the occupation measure is $\nu_S=1_1$. Simple algebraic manipulations tell us that $(\rho(x))_{x\in\s}$ solves \eqref{eq:nueqsmd} if and only if
\begin{equation}\rho=1_1+\alpha1_2\quad\text{for some }\alpha\in[0,\infty].\end{equation}
In other words, \eqref{eq:nueqsmd} has infinitely many solutions including $\nu_S$. However, out of all these solutions, $\nu_S$ is the minimal one:
$$\nu_S(x)=1_1(x)\leq 1_1(x)+\alpha1_2(x)=\rho(x),\qquad\forall x\in\{0,1,2\}.$$
\end{example}

The issue in the above example is that we have chosen to include in our state space the closed set $\{2\}$ which is not accessible from the support of the initial distribution $1_1$. For this reason, the state $2$ is irrelevant to any question regarding the exit of the chain from $\{1,2\}$. By ruling out such an extraneous closed set, it is straightforward to  show that there are no other non-negative solutions of \eqref{eq:nueqsmd} also satisfying the condition
$$\gamma(\s\backslash \cal{D})+\sum_{z\in\cal{D}}\sum_{x\not\in\cal{D}}\rho(z)p(z,x)=\Pbx{\{\sigma<\infty\}}$$
obtained by summing \eqref{eq:mueqsmd} over $x$, see \citep[Corollary~2.11]{Kuntzthe}. However, the question remains...
\subsubsection*{An open question} When  does \eqref{eq:nueqsmd} have multiple non-negative solutions?

\subsubsection*{Ruin probabilities}We close this chapter by revisiting our gambler Alice (Section~\ref{sec:gamblers}) and applying Theorem~\ref{charactd} to work out her ruin probabilities. In particular, suppose that Alice aims to leave the game with $K$ pounds in her pocket and will carry on playing until she either accumulates this amount of money or goes broke trying (at which point, she is excluded from the game). Suppose that $X_0$ belongs $\{1,2,\dots,K-1\}$ and consider the problem of computing the probability that Alice is successful in her venture. The time of exit $\sigma$~\eqref{eq:eddef} from the domain $\{1,2,\dots,K-1\}$  is the betting round that Alice goes broke or accumulates $K$ pounds, whichever comes first. For this reason, $\mu_S(K)$ is the probability that Alice is successful, while $\mu_S(0)$ is the probability she is not. The equations \eqref{eq:nueqsmd} satisfied by $\nu_S$ read
%
\begin{equation}\label{eq:gm}\left\{
\begin{array}{l}
\gamma(1)=\nu_S(1)-(1-a)\nu_S(2)\\
\gamma(x)=\nu_S(x)-a\nu_S(x-1)-(1-a)\nu_S(x+1),\qquad\forall x=2,\dots,K-2\\
\gamma(K-1)=\nu_S(K-1)-a\nu_S(K-2)\\
\end{array}\right.,
\end{equation}
while those \eqref{eq:mueqsmd} satisfied by $\mu_S$ read
\begin{equation}\label{eq:gm2}\mu_S(0)=(1-a)\nu_S(1),\qquad \mu_S(K)=a\nu_S(K-1).\end{equation}
If the coin is unbiased ($a=1/2$), then multiplying \eqref{eq:gm} by $x$ and summing over $x\in\{1,\dots,K-1\}$ yields
\begin{equation}\label{eq:fn8eanaf8wnfwe8a1}\Ebl{X_0}=\sum_{x=1}^{K-1}x\gamma(x)=\frac{K\nu_S(K-1)}{2}.\end{equation}
The above and \eqref{eq:gm2} imply that the probability $\mu_S(K)$ that Alice is successful in her strategy is $\Ebl{X_0}/K$ (i.e.\ her \emph{ruin probability}\index{ruin probabilities} is $1-\Ebl{X_0}/K$). 

If the coin is biased ($a\neq1/2$), we may re-arrange the middle equations in \eqref{eq:gm} into
$$\nu_S(x)-\nu_S(x-1)=\alpha (\nu_S(x+1)-\nu_S(x))+\frac{\gamma(x)}{a},\qquad\forall x=2,\dots,K-2,$$
where $\alpha:=(1-a)/a$. Iterating the above, we have that
\begin{equation}\label{eq:jufiawnfai}\nu_S(2)-\nu_S(1)=\alpha^{K-3}(\nu_S(K-1)-\nu_S(K-2))+\frac{1}{a}\sum_{y=2}^{K-2}\alpha^{y-2}\gamma(y).\end{equation}
The remaining equations in \eqref{eq:gm} and \eqref{eq:gm2} imply that 
$$\nu_S(2)-\nu_S(1)=\frac{1}{1-a}(\alpha^{-1}\mu_S(0)-\gamma(1)),\qquad \nu_S(K-1)-\nu_S(K-2)=\frac{1}{a}(\gamma(K-1)-\alpha\mu_S(K)).$$
Combining the above with \eqref{eq:jufiawnfai} yields
\begin{equation}\label{eq:jufiawnfai2}\mu_S(0)+\alpha^K\mu_S(K)=\sum_{y=1}^{K-1}\alpha^y\gamma(y)=\Ebl{\alpha^{X_0}}.\end{equation}
\noindent Adding up the equations \eqref{eq:gm}, we find that
$$1=\Pbl{\{X_0\in\{1,\dots,K-1\}\}}=\sum_{x=1}^{K-1}\gamma(x)=(1-a)\nu_S(1)+a\nu_S(K-1)=\mu_S(0)+\mu_S(K),$$
where the last equation follows from \eqref{eq:gm2}. In other words, Alice will eventually either amass $K$ pounds or go broke (with probability one). Combining this fact with \eqref{eq:jufiawnfai2}, we find the probability that Alice is successful is
\begin{equation}\label{eq:fn8eanaf8wnfwe8a2}\mu_S(K)=\frac{1-\Ebl{\alpha^{X_0}}}{1-\alpha^K}.\end{equation}
In other words, in the biased case, her ruin probability is $1-\frac{1-\Ebl{\alpha^{X_0}}}{1-\alpha^K}$.
%
%

Alice develops a full-blown gambling addiction and abandons her strategy: the only thing  stopping her now from betting is bankruptcy. 
Taking the limit $K\to\infty$ in \eqref{eq:fn8eanaf8wnfwe8a1} and \eqref{eq:fn8eanaf8wnfwe8a2}, we recover the classical result that the probability that Alice will now go broke is one unless the coin is biased in her favour ($a>1/2$), in which case it is $\Ebl{\alpha^{X_0}}$.

\ifdraft
{\color{red}NOTES:  Theorem \ref{charactd} is mentioned in \citep[p.112]{Kemeny1976}, but no proof is provided therein.}
\fi
\subsection{The strong Markov property}\label{sec:dtstrmarkov} To finish the chapter, we prove the \emph{strong Markov property}, a generalisation the Markov property~\eqref{eq:markovd0} that applies to $(\cal{F}_n)_{n\in\n}$-stopping times $\varsigma$ (instead of just deterministic times $n$). However, if we try to replace $n$ with $\varsigma$ in \eqref{eq:markovd0} we immediately run into difficulties: $X_{\varsigma}$ may only be partially defined as $\varsigma$ may be infinite, so what does it mean to condition on $X_{\varsigma}$? 

To circumvent this issue, we re-write~\eqref{eq:markovd0} in an uglier, but easier-to-manipulate form:  
%
%
recall \eqref{eq:markovstats}, i.e.\
$$\Pbl{\{X_{n+1}=y\}|\cal{F}_n}=\Pb_{X_n}(\{X_1=y\})\quad\forall y\in\s,\enskip \Pb_\gamma\text{-almost surely}.$$
For a (say, bounded) real-valued function $f$ on $\s$, multiply both sides by $f(y)$ and sum over $y$ in $\s$ to find that
$$\Ebl{f(X_{n+1})|\cal{F}_n}=\Eb_{X_n}[f(X_{1})],\quad\Pb_\gamma\text{-almost surely}.$$
Pick any $x$ in $\s$ and (for the time being, bounded) $\cal{F}_n/\cal{B}(\r)$-measurable random variable $Z$, multiply both sides by $Z1_{\{n<\infty,X_n=x\}}$ (note that $\{n<\infty\}=\Omega$), apply the take-out-what-is-known property of conditional expectation (Theorem~\ref{thrm:condexpprops}$v$), take expectations, and apply the tower property (Theorem~\ref{thrm:condexpprops}$iv$) to obtain
\begin{equation}\label{eq:markovrewrite}\Ebl{Z1_{\{n<\infty,X_n=x\}}f(X_{n+1})}=\Ebl{Z1_{\{n<\infty,X_n=x\}}}\Ebx{f(X_1)}\qquad\forall x\in\s.\end{equation}
%
%
Setting $Z:=1$ and $f:=1_y$ and re-arranging we recover~\eqref{eq:markovd0}, proving the equivalence of \eqref{eq:markovrewrite} and \eqref{eq:markovd0}. However, replacing $n$ with $\varsigma$ in~\eqref{eq:markovrewrite} does not raise any red flags (our convention~\eqref{eq:partdef} regarding partially-defined functions is important here). We obtain
\begin{equation}\label{eq:markovrewrite2}\Ebl{Z1_{\{\varsigma<\infty,X_\varsigma=x\}}f(X_{\varsigma+1})}=\Ebl{Z1_{\{\varsigma<\infty,X_\varsigma=x\}}}\Ebx{f(X_1)}\qquad\forall x\in\s\end{equation}
for $\cal{F}_\varsigma/\cal{B}(\r)$-measurable $Z$ (c.f.~Definition~\ref{def:stopdt}) and every bit of the equation makes sense.

All of this re-writing aside, \eqref{eq:markovrewrite2} is still just a fancy way of saying that ($a$) the chain's future conditioned on its past and present equals that conditioned on only the present (i.e.\ the past and future are conditionally independent given the present, see Proposition~\ref{prop:condind}) and ($b$)  the chain is `born anew' from its current location. The only difference is that we have moved the dividing line between past and the future from $n$ to $\varsigma$.

In \eqref{eq:markovrewrite2}, we are still using the location $X_{\varsigma+1}$ of the chain one-step after the present to represent the chain's future, while in reality $X_{\varsigma+1}$ captures only a sliver of the future. We also take the opportunity here to rectify this sorry state of affairs: we show that~\eqref{eq:markovrewrite2} holds for every appropriately-measurable functional $F$ of the chain's present and future positions ($X_{\varsigma},X_{\varsigma+1},\dots$)  instead of just $f(X_{n+1})$. To this end, we need the \emph{$\varsigma$-shifted chain}:

\subsubsection*{Shifting the chain left by $\varsigma$} To describe the chain's future from $\varsigma$ onwards, we use the \emph{shifted chain}\index{shifted chain} $X^\varsigma$\glsadd{Xvarsig} defined as the function mapping from $\{\varsigma<\infty\}$ to the path space $\cal{P}$ (c.f. Section~\ref{sec:pathlaw}) with
\begin{equation}\label{eq:shiftdt}X^{\varsigma}(\omega):=(X_{\varsigma(\omega)}(\omega),X_{\varsigma(\omega)+1}(\omega),X_{\varsigma(\omega)+2}(\omega),\dots)\quad\forall \omega\in\{\varsigma<\infty\}.\end{equation}
This $\cal{P}$-valued function $X^\varsigma$ describes the chain \emph{shifted leftwards} in time by $\varsigma$:
$$X^\varsigma_{n}(\omega)=X_{\varsigma(\omega)+n}(\omega)\quad\forall \omega\in\{\varsigma<\infty\},\enskip n\geq0.$$
\subsubsection*{The strong Markov property}\index{strong Markov property}
We are now in great shape to state and prove the strong Markov property. Note that whenever we write $1_{\{\varsigma<\infty\}}F(X^\varsigma)$ in what follows, we are using the notation for partially-defined functions introduced in~\eqref{eq:partdef}.\index{strong Markov property}
\begin{theorem}[The strong Markov property]\label{thrm:strmkvpath} Let $(\cal{F}_n)_{n\in\n}$ be the filtration generated by the chain~(Definition~\ref{def:filt}), $\cal{E}$ be the cylinder sigma-algebra on the path space $\cal{P}$ (Section~\ref{sec:pathlaw}), $\varsigma$ be an $(\cal{F}_n)_{n\in\n}$-stopping time~(Definition~\ref{def:stopdt}), and $\cal{F}_\varsigma$ be the pre-$\varsigma$ sigma-algebra~(Definition~\ref{def:stopdt}). Suppose that $Z$ is an $\cal{F}_\varsigma/\cal{B}(\r_E)$-measurable random variable, $F$ is a $\cal{E}/\cal{B}(\r_E)$-measurable function, and that $Z$ and $F$ are both non-negative, or both bounded. If $x$ is any state in $\s$, then $Z1_{\{\varsigma<\infty,X_\varsigma=x\}}F(X^\varsigma)$ is $\cal{F}/\cal{B}(\r_E)$-measurable, where $X^\varsigma$ denotes the $\varsigma$-shifted chain~\eqref{eq:shiftdt}. Moreover,
\begin{equation}\label{eq:strmkvd}\Ebl{Z1_{\{\varsigma<\infty,X_\varsigma=x\}}F(X^\varsigma)}=\Ebl{Z1_{\{\varsigma<\infty,X_\varsigma=x\}}}\Ebx{F(X)}\quad\forall x\in\s.\end{equation}
%
\end{theorem}


We do this proof in two steps, beginning by proving the measurability  of $Z1_{\{\varsigma<\infty,X_\varsigma=x\}}F(X^\varsigma)$ then moving on to \eqref{eq:strmkvd}.

\begin{proof}[Step 1: $Z1_{\{\varsigma<\infty,X_\varsigma=x\}}F(X^\varsigma)$ is $\cal{F}/\cal{B}(\r_E)$-measurable] Because of  Exercise~\ref{exe:xvarsigma}, and because the product of measurable functions is measurable, all we need to argue here is that $1_{\{\varsigma<\infty\}}F(X^\varsigma)$ is $\cal{F}/\cal{B}(\r_E)$-measurable. Due to Lemma~\ref{lem:partdefmeas}, this consists of proving that
$$\{\varsigma<\infty, F(X^{\varsigma})\in A\}\in \cal{F},\quad\forall A\in\cal{B}(\r_E).$$
However,
$$\{\varsigma<\infty, F(X^{\varsigma})\in A\}\in \cal{F}=\bigcup_{n=0}^\infty\{\varsigma=n, \Theta_n(X)\in F^{-1}(A)\},$$ 
where, for each $n\geq0$, $\Theta_n:\cal{P}\to\cal{P}$ denotes \emph{n-shift operator} defined by
$$\Theta_n((x_m)_{m\in\n}):=(x_n,x_{n+1},x_{n+2},\dots)\quad\forall (x_m)_{m\in\n}\in\cal{P}.$$
Given that $\{\varsigma=n\}$ belongs to $\cal{F}_n\subseteq\cal{F}$ (Lemma~\ref{lem:2stop}$i$) and $F$ is $\cal{E}/\cal{B}(\r_E)$-measurable by assumption, it suffices to show that  $\Theta_n$ is $\cal{E}/\cal{E}$ measurable. Because the cylinder sets~\eqref{eq:gensets} generate $\cal{E}$, this follows easily.
\end{proof}

\begin{proof}[Step 2: \eqref{eq:strmkvd} holds] Given Step 1 and our assumption that  $Z$ and $F$ are either both non-negative or both bounded, every term in \eqref{eq:strmkvd} is well-defined. Because $\Ebx{F(X)}=\mathbb{L}_x(F)$, where $\mathbb{L}_x$ denotes the path law~\eqref{eq:pathlaw} starting from $x$, and because of the definition of the Lebesgue integral, it suffices to show that  
\begin{equation}\label{eq:dja7dany83d}\Pbl{A\cap\{\varsigma<\infty,X_\varsigma=x,X^\varsigma\in B\}}=\Pbl{A\cap\{\varsigma<\infty,X_\varsigma=x\}}\mathbb{L}_x(B)\end{equation}
for all $A$ in $\cal{F}_\varsigma$, $B$ in $\cal{E}$, and $x$ in $\s$---if you are not following this, go read about the \emph{standard machine} in \citep[Section~5.12]{Williams1991}. 

Fix any $A$ in $\cal{F}_{\varsigma}$ and $x$ in $\s$. If $c:=\Pbl{A\cap\{\varsigma<\infty,X_\varsigma=x\}}=0$, then both sides of \eqref{eq:dja7dany83d} are zero and we are done. Suppose that $c\neq0$ and consider the map  $\tilde{\mathbb{L}}:\cal{E}\to\r_E$ defined by
$$\tilde{\mathbb{L}}(B):=c^{-1}\Pbl{\{A\cap\{\varsigma<\infty,X_\varsigma=x,X^\varsigma\in B\}}\quad\forall B\in\cal{E}.$$
Clearly, $\tilde{\mathbb{L}}$ is non-negative and $\tilde{\mathbb{L}}(\cal{P})=c/c=1$. Furthermore, Tonelli's Theorem implies that $\tilde{\mathbb{L}}$ is countably additive. In other words, $\tilde{\mathbb{L}}$ is a probability measure. Suppose that we are able to show that
\begin{equation}\label{eq:fne8awne78awfhaw}\tilde{\mathbb{L}}(\{x_0\}\times\{x_1\}\times\dots\times\{x_m\}\times \s\times\s\times\dots)=\mathbb{L}_x(\{x_0\}\times\{x_1\}\times\dots\times\{x_m\}\times \s\times\s\times\dots)\end{equation}
for all natural numbers $m$ and states $x_0,x_1,\dots,x_n$ in $\s$. Theorem \ref{pathlawuni} then shows that $\tilde{\mathbb{L}}$ is $\mathbb{L}_x$ and \eqref{eq:dja7dany83d} follows.

To prove \eqref{eq:fne8awne78awfhaw}, re-write this equation as 
\begin{align*}&\Pbl{A\cap\{\varsigma<\infty,X_{\varsigma}=x,X_{\varsigma}=x_0,X_{\varsigma+1}=x_1,\dots,X_{\varsigma+m}=x_m\}}\\
&=\Pbl{A\cap\{\varsigma<\infty,X_{\varsigma}=x\}}\Pbx{\{X_{0}=x_0,X_{1}=x_1,\dots,X_{m}=x_m\}}.\end{align*}
If $x\neq x_0$ then both sides of the above are zero and we are done. Suppose that $x=x_0$ and fix any $l>0$. Using the definitions in Algorithm \ref{dtmcalg}, we can express $X_{n+l}$ as some measurable function of $X_n$ and $U_{n+1},\dots,U_{n+l}$:
$$X_{n+l}=f_{l}(X_n,U_{n+1},\dots,U_{n+l})\qquad \forall j>0,\enskip n\geq0$$
and this function depends on $l$ but not on $n$.  Because $(X_1,\dots, X_n)$ is a function of $X_0$ and $(U_1,\dots, U_n)$ and these are independent of $(U_{n+1},\dots,U_{n+j})$, the random variable
$$f_{l}(x,U_{n+1},\dots,U_{n+l})$$
is independent of $(X_0,\dots, X_n)$ and, consequently, of $\cal{F}_n$. Because $A\cap\{\varsigma=n,X_n=x\}$ belongs to $\cal{F}_n$ (Lemma~\ref{lem:2stop}$(i)$) and $(U_n)_{n\in\zp}$ is i.i.d., it follows that
\begin{align}&\Pbl{A\cap\{\varsigma=n,X_{n}=x,X_{n+1}=x_1,\dots,X_{n+m}=x_m\}}\nonumber\\
&=\Pbl{A\cap\{\varsigma=n,X_{n}=x,f_{1}(x,U_{n+1})=x_1,\dots,f_{m}(x,U_{n+1},\dots,U_{n+m})=x_m\}}\nonumber\\
&=\Pbl{A\cap\{\varsigma=n,X_{n}=x\}}\Pbl{\{f_{1}(x,U_{n+1})=x_1,\dots,f_{m}(x,U_{n+1},\dots,U_{n+m})=x_m\}}\nonumber\\
&=\Pbl{A\cap\{\varsigma=n,X_{n}=x\}}\Pbl{\{f_{1}(x,U_{1})=x_1,\dots,f_{m}(x,U_{1},\dots,U_{l})=x_m\}}\nonumber\\
&=\Pbl{A\cap\{\varsigma=n,X_{n}=x\}}\Pbx{\{f_{1}(X_0,U_{1})=x_1,\dots,f_{m}(X_0,U_{1},\dots,U_{m})=x_m\}}\nonumber\\
&=\Pbl{A\cap\{\varsigma=n,X_{n}=x\}}\Pbx{\{X_{1}=x_1,\dots,X_{m}=x_m\}},\nonumber\end{align}
Summing both sides over $n$ in $\n$ then yields \eqref{eq:fne8awne78awfhaw}, so completing the proof.
\end{proof}
\subsubsection*{Measurable functions on $\cal{P}$}
We finish this section with the following technical lemma that will cover our measurability-checking necessities for functionals on $\cal{P}$ throughout the remainder of our treatment of discrete-time chains.
\begin{lemma}\label{lem:pathspmeas}The following are $\cal{E}/\cal{B}(\r_E)$-measurable functions:
\begin{enumerate}[label=(\roman*),noitemsep]
\item For any real-valued function $f$ on $\s$,
$$F_-(x):=\liminf_{n\to\infty}\frac{1}{n}\sum_{m=0}^{n-1}f(x_m),\quad F_+(x):=\limsup_{n\to\infty}\frac{1}{n}\sum_{m=0}^{n-1}f(x_m),\quad\forall x\in\cal{P}.$$
\item For any subset $A$ of $\s$ and positive integer $k$,
$$F(x):=\inf\left\{n>0:\sum_{m=1}^n1_A(x_m)=k\right\}\quad\forall x\in\cal{P}.$$
%
%
\item For any non-negative function $f$ on $\s$,
$$G(x):=\sum_{n=0}^{F((x_m)_{m\in\n})-1}f(x_n)\quad\forall x\in\cal{P},$$
where $F$ is as in~$(ii)$.
\end{enumerate} 
\end{lemma}

\begin{proof}

$(i)$ The definition of the cylinder sigma-algebra $\cal{E}$ in Section~\ref{sec:pathlaw} implies that $(x_m)_{m\in\n}\mapsto f(x_n)$ is $\cal{E}/\cal{B}(\r)$-measurable for each $n\in\n$. Because finite sums and products of measurable functions are measurable,  we have that $(x_m)_{m\in\n}\mapsto\frac{1}{n}\sum_{m=0}^{n-1}f(x_m)$ is $\cal{E}/\cal{B}(\r)$-measurable for each $n\in\zp$. Because the limit infimum and supremum of measurable functions are measurable, the result follows.

$(ii)$ By definition, $F(x)$ denotes the time that the path $x$ enters the set $A$ for the $k$th time and so
\begin{align*}F&=\infty\cdot 1_{\{g_\infty(x)<k\}}+\sum_{n=1}^\infty n1_{\{g_{n-1}(x)<k\}}1_{\{g_n(x)=k\}}\\
&=\infty\cdot \lim_{N\to\infty}1_{\{g_N(x)<k\}}+\lim_{N\to\infty}\sum_{n=1}^Nn1_{\{g_{n-1}(x)<k\}}1_{\{g_n(x)=k\}},\end{align*}
where $g_n(x):=\sum_{m=1}^n1_A(x_m)$ denotes the number of times the path $x=(x_m)_{m\in\n}$ has entered $A$ by time $n$ (with $g_0:=0$). Because any finite sum of measurable functions is measurable, $g_n$ is $\cal{E}/\cal{B}(\r_E)$-measurable, for each $n\in\zp$. The result then follows from the above by exploiting once again that the sum, product, limit superior, and limit inferior of measurable functions are all measurable. 
%
%

$(iii)$ Because any finite sum of measurable functions is measurable, $g_n((x_m)_{m\in\n}):=\sum_{m=0}^nf(x_m)$ is $\cal{E}/\cal{B}(\r_E)$-measurable for any $n\in\n$. Furthermore, as the limit of measurable functions is measurable, $g_\infty:=\lim_{n\to\infty}g_n$ is also $\cal{E}/\cal{B}(\r_E)$-measurable. Pick any $A\in\cal{B}(\r_E)$ and note that
$$\left\{(n,(x_m)_{m\in\n})\in\n_E\times\cal{P}:\sum_{m=0}^nf(x_m)\in A\right\}=(\{\infty\}\times\{g_\infty\in A\})\cup\left(\bigcup_{n=0}^\infty\{n\}\times\{g_n\in A\}\right).$$
Because the right-hand side belongs to the product sigma-algebra $2^{\n_E}\times\cal{E}$, it follows that 
$$H(n,(x_m)_{m\in\n}):=\sum_{m=0}^nf(x_m)$$
defines an $2^{\n_E}\times\cal{E}/\cal{B}(\r_E)$-measurable function, for each $n\in\n$. Because $G((x_m)_{m\in\n})=H(F((x_m)_{m\in\n})-1,(x_n)_{n\in\n})$, with $F$ as in $(ii)$, the result then follows because the composition of measurable functions is measurable.
\end{proof}

\newpage
\thispagestyle{premain}
\sectionmark{\MakeUppercase{Discrete-time chains II: the long-term behaviour}}
\section*{Discrete-time chains II: the long-term behaviour}\addcontentsline{toc}{section}{\protect\numberline{}Discrete-time chains I: the long-term behaviour}

Armed with the tools introduced throughout the previous chapter (Sections~\ref{sec:dtdef}--\ref{sec:dtstrmarkov}), we are ready to start chipping away at the long-term behaviour of discrete-time chains. The aim of this chapter (Sections~\ref{sec:entrance}--\ref{sec:flgeodt}) is to formalise the discussion given in the introduction to Part~\ref{part:theory} for the case of discrete-time chains.

\subsubsection*{Overview of the chapter} We begin in Section~\ref{sec:entrance} by introducing entrance times and using them to  (a) describe whether the chain is able to travel between two given states; and (b) define the \emph{closed communicating classes}: subsets of the state space that the chain, once inside, is unable to leave but is free to explore within. 
We then describe (Section~\ref{sec:rec}) the states that the chain will keep revisiting (the \emph{recurrent} states) and those that it not  (the \emph{transient} states). Upon each visit to a recurrent state, the chain starts `afresh' in the sense that the segments of the chain's path delimited by the visits form an i.i.d. sequence. We prove this important \emph{regenerative property} of discrete-time chains in Section~\ref{sec:regen}. Combining this property with the law of large numbers, we find (Section~\ref{sec:empdist}) that the fraction of time-steps each path spends in any given state settles down as the number of time-steps approaches infinity. Of course, for transient states the fraction converges to zero as the visits to these states eventually stop. For recurrent states, the limiting value may be zero---in which case the frequency of the visits degenerates to zero and the state is said to be \emph{null recurrent}---or non-zero---in which case the frequency remains bounded away from zero and the state is said to be \emph{positive recurrent}. As we will see, the average time it takes the chain to return to a recurrent state differentiates these two cases: it is infinite for null recurrent states and finite for positive recurrent ones. 

In Section~\ref{sec:statdists}, we introduce the \emph{stationary distributions} of the chain: the starting distributions that 
turn the chain into a stationary process (i.e.\ one whose finite-dimensional statistics do not change with time). We also show that there exists a stationary distribution with support on a given state if and only if that state is positive recurrent and that positive recurrence is a \emph{class property}: a state within a closed communicating class is positive recurrent if and only if every state within the class is positive recurrent. Moreover, as we will see, there is only one stationary distribution with support contained inside any given positive recurrent class (known as the \emph{ergodic distribution} associated with that class) and the set of stationary distributions is simply the convex hull of these ergodic distributions. In Section~\ref{sec:timeaved}, we revisit the limiting value of the fraction of time-steps that the chain spends in any given positive recurrent state $x$ and express it in terms of the ergodic distribution $\pi$ associated with the class $\cal{C}$ that $x$ belongs to (it is $\pi(x)$ if the chain's path ever enters $\cal{C}$ and zero otherwise). Using these facts, we then see why chains whose paths all enter the positive recurrent closed communicating classes (i.e.\ \emph{positive Tweedie recurrent} chains) are typically viewed as the \emph{stable} chains.

We then turn our attention to the long-term behaviour of the chain's statistics and focus on the question `does the time-varying law converge and to what?'. Here, we require the notion of a \emph{periodic chain} (Section~\ref{sec:periodicity}): one for which the possibility of return to at least one positive recurrent state is a periodic phenomenon (e.g.\ the chain may return to the state after an even number of steps but not an odd number). We then settle the question in Section~\ref{sec:ensembleaved}: if the chain is periodic, the time-varying law does not converge but oscillates instead. Otherwise, the time-varying law converges to a convex combination of the ergodic distributions, where the weight awarded to each distribution is the probability that the chain ever enters the closed communicating class associated with it. This marks the  end of our treatment on the  basics of the long-term behaviour of discrete-time chains. In my experience, the material presented throughout Sections~\ref{sec:dtdef}--\ref{sec:ensembleaved} forms a good base camp of sorts from which one can approach other more advanced or specialised summits of  Markov chain theory.

At this point there are many directions we could  take: further investigate the limiting behaviour of the pathwise averages and derive pertinent  central limit theorems and laws of the iterated logarithm, study the relationship between the rate of convergence the time-varying law and the tails of the return times,  delve into the notion of  quasistationarity describing the medium-term behaviour of the chain, figure out how to compute  in practice time-varying laws, occupation measures, and stopping/exit/stationary distributions introduced throughout the sections, start working on the theory for the continuous-time case or the general state space case, etc. The remainder of this chapter is dedicated to two such subjects:
\begin{enumerate}
\item (Geometric recurrence and convergence: Section~\ref{sec:kendall}) We study the question of when does the time-varying law converge geometrically fast and prove Kendall's theorem: a key result in this area showing that the answer lies in the tails of the return time distributions.
\item (Foster-Lyapunov criteria: Sections~\ref{sec:flrecdt}--\ref{sec:flgeodt}). We derive a series of analytical inequalities involving the one-step matrix that yield information on the entrance and return times of a chain and, consequently, on its  recurrence properties and stationary distributions. Before we start with these however, we introduce in Section~\ref{sec:geometrictrial} certain \emph{geometric trials arguments} that will be important for the proofs of  Sections~\ref{sec:flrecdt}--\ref{sec:flgeodt}.
\end{enumerate}

\subsectionmark{\MakeUppercase{Entrance times; accessibility; closed communicating classes}}
\subsection{Entrance times, accessibility, and closed communicating classes}\label{sec:entrance}
\subsectionmark{\MakeUppercase{Entrance times; accessibility; closed communicating classes}}

Before discussing which bits of the state space the chain might visit in the long-run, we need to discuss which bits it will ever visit in the first place. Here, we require \emph{entrance times}\index{entrance time}\glsadd{phiA}\glsadd{phiAk}:
%
%
%
\begin{definition}[Entrance times]\label{def:entrance} Given any subset $A$ of $\s$ and positive integer $k$, the $k$th entrance time $\phi^k_A$ is the time-step when the chain enters the set $A$ for the $k$th time or infinity if it never does:
$$\phi^0_A(\omega):=0\enskip\forall \omega\in\Omega,\quad \phi^k_A(\omega):=\inf\left\{n>0:\sum_{m=1}^n 1_A(X_m(\omega))=k\right\}\enskip\forall \omega\in\Omega,\enskip k>0,$$
where we are adhering to our convention that $\inf\emptyset=\infty$. For the first entrance time, we write $\phi_A$ instead of $\phi_A^1$. Additionally, we use the shorthands $\phi_x:=\phi_{\{x\}}$ (resp. $\phi_x^k:=\phi_{\{x\}}^k$) to denote the first (resp. $k$th) entrance time of a state $x$.
\end{definition}
You may be more familiar with the recursive definition of entrance times:
\begin{equation}\label{eq:gnr8e7a0gja789jga78}\phi^0_A(\omega):=0\enskip\forall\omega\in\Omega,\quad\phi_A^{k}(\omega):=\inf\{n>\phi^{k-1}_A(\omega):X_n(\omega)\in A\}\enskip\forall\omega\in\Omega,\enskip k>0.\end{equation}
However, it is not difficult to check that the two definitions are equivalent:
\begin{exercise}\label{ex:entdef}Use induction to show that \eqref{def:entrance} and \eqref{eq:gnr8e7a0gja789jga78} are equivalent. Hint: the chain has visited $A$ exactly $k$ times by $\phi_A^k$ and so $\sum_{m=1}^{\phi_A^k} 1_A(X_m)=k$ whenever $\phi_A^k$ is finite.
\end{exercise}

The difference between entrance times and hitting times (Section~\ref{sec:stop}) is that---somewhat inconsistently with our nomenclature for exit times (Section~\ref{sec:exit})---we do not consider starting inside $A$ as `entering' $A$ (hence, $\phi_A>0$ by definition). If the chain starts in $x$, then $\phi_x$  is the first  time that the chain \emph{returns} to $x$. For this reason, $\phi_x$ is often referred to as the \emph{first return time}\index{return time}. Similarly, $\phi_x^k$ is known as the $k$th \emph{return time}. Of course, entrance times are stopping times because we are able to deduce whether the chain has entered a given set by continuously monitoring it:
\begin{exercise}\label{ex:entstop}Using similar arguments to those given in the proof of Proposition~\ref{prop:hitisstop}, show that entrance times are $(\cal{F}_n)_{n\in\n}$-stopping times, where $(\cal{F}_n)_{n\in\n}$ denotes the filtration generated by the chain~(Definition~\ref{def:filt}).
\end{exercise}

\subsubsection*{Accessibility, communicating classes, and closed sets} Equipped with entrance times, we now are able to consider the question of whether the chain may reach any region of the state space from any other or if it is limited in some way. In particular, if the chain is able to travel from one subset of the state $A$ to another $B$, we say that $B$ is \emph{accessible}\index{accessible}\glsadd{accessibility}\glsadd{accessibility2} from $A$. Formally:

%
\begin{definition}[Accessibility]\label{def:accesible}A state $y$ is accessible from a state $x$ (or $x\to y$ for short) if and only if $\Pb_x\left(\{\phi_y<\infty\}\right)>0$. Similarly, we say that a subset $B$ of the state space is accessible from another subset $A$ (or $A\to B$) if for every $x$ in $A$ there exists a $y$ in $B$ such that $x\to y$ (equivalently, if $\Pbl{\{\phi_B<\infty\}}>0$ for all initial distributions $\gamma$ with support in $A$). 
\end{definition}
It is straightforward to characterise accessibility in terms of the one-step matrix:
\begin{lemma}\label{lem:accdt} $x\to y$ if and only if $p(x,y)>0$ or there exists some states $x_1,\dots,x_l$ such that 
$$p(x,x_1)p(x_1,x_2)\dots p(x_l,y)>0.$$\end{lemma}
\begin{proof}By the definition of entrance times,
$$\{\phi_y<\infty\}=\{X_1\in A\}\cup \{X_1\not\in A,X_2\in A\}\cup \{X_1\not\in A,X_2\not\in A,X_3\in A\}\cup\dots$$
and the lemma follows by taking expectations and applying \eqref{eq:nstepthe}.
\end{proof}
It follows from the lemma that $\to$ is a transitive relation on both  $\cal{S}$ and its power set. Two sets are said to \emph{communicate} if each is accessible from the other. A set to be a \emph{communicating class}\index{communicating class} if all its subsets communicate:
\begin{definition}[Communicating class]A subset $\cal{C}$ of $\cal{S}$ is a communicating class if and only if $x\to y$ for each pair $x,y$ in $\cal{C}$ (equivalently $A\to B$ for all subsets $A,B$ of $\cal{C}$).
\end{definition}
A subset of the state space is said to be \emph{closed}\index{closed set} if no state outside the set is accessible from a state inside the set. 
\begin{definition}[Closed set]A subset $\cal{C}$ of $\s$ is closed if and only if for each $x$ in $\cal{C}$ and $y$ in $\s$, $x\to y$ implies that $y$ belongs to $\cal{C}$.\end{definition}
Closed sets are those from which the chain may never escape. 
\begin{proposition}\label{prop:closed}If the chains starts in a closed set $\cal{C}$ (i.e.\ the initial distribution $\gamma$ satisfies $\gamma(\cal{C})=1$), then the chain remains in $\cal{C}$:
$$\Pbl{\{X_n\in\cal{C}\enskip\forall n\in\n\}}=1.$$
\end{proposition}
\begin{proof}Because Lemma~\ref{lem:accdt} implies that $\sum_{y\in\cal{C}}p(x,y)=\sum_{y\in\s}p(x,y)=1$ for all $x$ in $\cal{C}$, \eqref{eq:nstepthe} implies that
\begin{align*}&\Pbl{\{X_0\in\cal{C},X_1\in\cal{C},\dots,X_{n-2}\in\cal{C},X_{n-1}\in\cal{C},X_n\in\cal{C}\}}\\
&=\sum_{x_0\in\cal{C}}\sum_{x_1\in\cal{C}}\dots\sum_{x_{n-2}\in\cal{C}}\sum_{x_{n-1}\in\cal{C}}\sum_{x_n\in\cal{C}}\Pbl{\{X_0=x_0,X_1=x_1,\dots,X_{n-2}=x_{n-2},X_{n-1}=x_{n-1},X_n=x_n\}}\\
&=\sum_{x_0\in\cal{C}}\sum_{x_1\in\cal{C}}\dots\sum_{x_{n-2}\in\cal{C}}\sum_{x_{n-1}\in\cal{C}}\sum_{x_n\in\cal{C}}\gamma(x_0)p(x_0,x_1)\dots p(x_{n-2},x_{n-1})p(x_{n-1},x_n)\\
&=\sum_{x_0\in\cal{C}}\sum_{x_1\in\cal{C}}\dots\sum_{x_{n-2}\in\cal{C}}\sum_{x_{n-1}\in\cal{C}}\gamma(x_0)p(x_0,x_1)\dots p(x_{n-2},x_{n-1})\left(\sum_{x_n\in\cal{C}}p(x_{n-1},x_n)\right)\\
&=\sum_{x_0\in\cal{C}}\sum_{x_1\in\cal{C}}\dots\sum_{x_{n-2}\in\cal{C}}\sum_{x_{n-1}\in\cal{C}}\gamma(x_0)p(x_0,x_1)\dots p(x_{n-2},x_{n-1})=\dots=\sum_{x_0\in\cal{C}}\gamma(x_0)=1\quad\forall n\in\n.\end{align*}
Taking the limit $n\to\infty$ in the above and applying downwards monotone convergence then completes the proof.
\end{proof}
Closed communicating classes in particular are very important for Markov chain theory: they are the irreducible sets mentioned in the introduction to Part~\ref{part:theory}. For this reason, the chain and its one-step matrix are said to be \emph{irreducible}\index{irreducible} if the whole state space $\s$ is a single closed communicating class (they are said to be \emph{$\varphi$-irreducible} if there is a single closed communicating class and it is accessible from all states in the state space). 

\subsubsection*{Two useful consequences of the strong Markov property}We finish the section by proving two propositions that will be of use later. The first concerns the \emph{hitting probabilities} of the chain meaning the chance  $\Pbx{\{\phi_A<\infty\}}$ that the chain ever enters a set $A$ for the various possible starting states $x$.
\begin{proposition}[The hitting probabilities satisfy a set of linear equations]\label{prop:hitprobeqs} For all states $x$ in $\s$ and subsets $A$ of $\s$,
$$\Pbx{\{\phi_A<\infty\}}=\sum_{y\in A}p(x,y)+\sum_{z\not\in A}p(x,z)\Pbz{\{\phi_A<\infty\}}.$$
\end{proposition}
The second proposition formalises the intuitive idea that the chain may visit at most one of two  disjoint closed sets (for it enters one of these sets, it will never leave the set, and, consequently, never visit the other one):
\begin{proposition}[The chain visits at most one of two disjoint closed sets]\label{prop:closedis} If $\cal{C}_1$ and $\cal{C}_2$ are two disjoint closed sets, then
$$\Pbl{\{\phi_{\cal{C}_1}<\infty,\phi_{\cal{C}_2}<\infty\}}=0.$$
\end{proposition}

The proofs of both propositions require the following consequence of the strong Markov property.
\begin{lemma}\label{lem:dtenttimestr3}For any subset $A$ of the state space and $(\cal{F}_n)_{n\in\n}$-stopping time $\varsigma$,
\begin{align*}\Pbl{\{\varsigma<\phi_A<\infty\}}&=\sum_{x\in \s}\Pb_\gamma(\{\varsigma<\phi_A,X_{\varsigma}=x\})\Pbx{\{\phi_A<\infty\}},\\
\Pbl{\{\varsigma<\phi_A=\infty\}}&=\sum_{x\in \s}\Pb_\gamma(\{\varsigma<\phi_A,X_{\varsigma}=x\})\Pbx{\{\phi_A=\infty\}}.\end{align*}
\end{lemma}

\begin{proof}Note that $X_n$ does not belong to $A$ for $n=1,2,\dots,\varsigma$ if $\varsigma<\phi_A$, and so
\begin{align}\phi_A&=\inf\left\{n>0:\sum_{m=1}^n1_A(X_m)=1\right\}=\inf\left\{n>\varsigma:\sum_{m=\varsigma+1}^n1_A(X_m)=1\right\}\nonumber\\
&=\varsigma+\inf\left\{n>0:\sum_{m=1}^n1_A(X^{\varsigma}_m)=1\right\}\quad\text{on}\quad\{\varsigma<\phi_A\},\label{eq:fdafnda7s8fhzs}\end{align}
where $X^{\varsigma}$ denotes the ${\varsigma}$-shifted chain in~\eqref{eq:shiftdt}.

Lemma~\ref{lem:pathspmeas}$ii$ shows that the function $F((x_m)_{m\in\n}):=\inf\left\{n>0:\sum_{m=1}^n1_A(x_{m})=1\right\}$ is $\cal{E}/\cal{B}(\r_E)$-measurable, where $\cal{E}$ denotes the cylinder sigma-algebra on the path space (Section~\ref{sec:pathlaw}). Given that we are able to deduce whether the chain has entered $A$ before time $\varsigma$ from observing the chain up until $\varsigma$, the event $\{\varsigma<\phi_A\}$  belongs to the pre-$\varsigma$ sigma-algebra $\cal{F}_{\varsigma}$ introduced in Definition~\ref{def:stopdt} (for a formal argument, use Lemma~\ref{lem:2stop}$i$ and Exercise~\ref{ex:entstop}). For these reasons, applying the strong Markov property (Theorem \ref{thrm:strmkvpath}) yields the first equality given in the premise:
\begin{align*}\Pbl{\{\varsigma<\phi_A<\infty\}}&=\Ebl{1_{\{\varsigma<\phi_A\}}1_{\{\phi_A-\varsigma<\infty\}}}=\sum_{x\in \s}\Ebl{1_{\{\varsigma<\phi_A\}}1_{\{\varsigma<\infty,X_{\varsigma}=x\}}1_{\n}(F(X^\varsigma))}\\
&=\sum_{x\in \s}\Ebl{1_{\{\varsigma<\phi_A\}}1_{\{\varsigma<\infty,X_{\varsigma}=x\}}}\Ebx{1_{\n}(F(X))}\\
&=\sum_{x\in \s}\Pbl{\{\varsigma<\phi_A,X_{\varsigma}=x\}}\Pbx{\{\phi_A<\infty\}}\end{align*}
where  the second and fourth equality follow from \eqref{eq:fdafnda7s8fhzs} and the fact that $\{\varsigma<\phi_A\}\subseteq\{\varsigma<\infty\}$ (as $\varsigma$ is strictly less than $\phi_A$ only if the $\varsigma$ is finite). The second equality in the premise follows by replacing `$<\infty$' and `$1_\n$' with `$=\infty$' and `$1_{\infty}$' in the above.
\end{proof}

\begin{proof}[Proof of Proposition~\ref{prop:hitprobeqs}]Given that
$$\Pbx{\{\phi_A<\infty\}}=\Pbx{\{\phi_A=1\}}+\Pbx{\{1<\phi_A<\infty\}}=\sum_{y\in A} p(x,y)+\Pbx{\{1<\phi_A<\infty\}},$$
the proposition follows directly by setting $\varsigma:=1$ and $\gamma:=1_x$ in Lemma~\ref{lem:dtenttimestr3} and noting that, due to $\phi_A$'s definition, $X_1\in A$ on $\{\phi_A=1\}$ and $X_1\not\in A$ on $\{\phi_A>1\}$.
\end{proof}

\begin{proof}[Proof of Proposition~\ref{prop:closedis}]Because the sets are disjoint, the chain cannot enter both at the same time:
$$\Pbl{\{\phi_{\cal{C}_1}=\phi_{\cal{C}_2}<\infty\}}=\Pb_\gamma(\{\phi_{\cal{C}_1}=\phi_{\cal{C}_2}<\infty,X_{\phi_{\cal{C}_1}}\in\cal{C}_2\})=0.$$
Thus,
\begin{align*}\Pbl{\{\phi_{\cal{C}_1}<\infty,\phi_{\cal{C}_2}<\infty\}}
&=\Pbl{\{\phi_{\cal{C}_1}<\phi_{\cal{C}_2},\phi_{\cal{C}_2}<\infty\}}+\Pbl{\{\phi_{\cal{C}_2}<\phi_{\cal{C}_1},\phi_{\cal{C}_1}<\infty\}}.\end{align*}
However, setting $A:=\cal{C}_2$ and $\varsigma:=\phi_{\cal{C}_1}$ in Lemma~\ref{lem:dtenttimestr3}, we find that
\begin{equation}\label{eq:fja78wn7a3a3wwf34fa}\Pbl{\{\phi_{\cal{C}_1}<\phi_{\cal{C}_2},\phi_{\cal{C}_2}<\infty\}}=\sum_{x\in\cal{C}_1}\Pb_\gamma(\{\phi_{\cal{C}_1}<\phi_{\cal{C}_2},X_{\phi_{\cal{C}_1}}=x\})\Pbx{\{\phi_{\cal{C}_2}<\infty\}}\end{equation}
Because the sets are closed and disjoint, 
$$\Pbx{\{\phi_{\cal{C}_2}<\infty\}}=\sum_{y\in\cal{C}_2}\Pbx{\{\phi_{\cal{C}_2}=\phi_y,\phi_y<\infty\}}\leq \sum_{y\in\cal{C}_2}\Pbx{\{\phi_y<\infty\}} =0\quad \forall x\in\cal{C}_1,$$
and it follows that the right-hand side of \eqref{eq:fja78wn7a3a3wwf34fa} is zero. Swapping $\cal{C}_1$ and $\cal{C}_2$ throughout the above completes the proof.
\end{proof}

\subsection{Recurrence and transience}\label{sec:rec}

Our study of the long-term behaviour starts in earnest with the notions of \emph{recurrence} and \emph{transience}:

\begin{definition}[Recurrent and transient states]\label{def:rectrans} A state $x$ is said to be recurrent\index{recurrent state} if the chain has probability one of returning to it (i.e.\ $\Pbx{\{\phi_x<\infty\}}=1$) and transient\index{transient state} otherwise.
\end{definition}
Recurrent states are those the chain will keep revisiting (whenever an initial visit occurs) while transient ones are those the chain will cease visiting at some point:
\begin{theorem}[Characterising recurrent and transient states]\label{thrm:rectrans}
For any initial distribution $\gamma$ and  $x$ in $\s$,
\begin{equation}\label{eq:nfuieafbua}\Pb_\gamma(\{\phi^k_x<\infty\})=\Pbl{\{\phi_x<\infty\}}\Pb_x(\{\phi_x<\infty\})^{k-1}\quad\forall k>0.\end{equation}
Moreover, if $x$ is transient, then, with probability one, the chain visits $x$ finitely many times: $\Pbl{\{R<\infty\}}=1$ where 
$$R:=\sum_{n=1}^\infty 1_x(X_n)$$
denotes the total number of returns to $x$. If $x$ instead is recurrent, then  the probability $\Pbl{\{R=\infty\}}$ that the chain visits $x$ infinitely many times equals the probability $\Pbl{\{\phi_x<\infty\}}$ that visits $x$ at least once.
\end{theorem}
\begin{proof}Because the number of entrances to $x$ by time $\phi_x^k$ is exactly $k$ ($\sum_{m=1}^{\phi_x^k}1_x(X_m)=k$) and $\phi^{k+1}_x$ is greater than $\phi^k_x$ whenever the latter is finite (Exercise~\ref{ex:entdef}), we have that
\begin{align*}\phi^{k+1}_x&=\inf\left\{n>0:\sum_{m=1}^n1_x(X_m)=k+1\right\}=\inf\left\{n>\phi^k_x:\sum_{m=\phi^k_x+1}^n1_x(X_m)=1\right\}\\
&=\phi_x^k+\inf\left\{n>0:\sum_{m=1}^n1_x(X^{\phi^k_x}_{m})=1\right\}\quad\forall k>0,\end{align*}
where $X^{\phi^k_x}$ denotes the ${\phi^k_x}$-shifted chain~\eqref{eq:shiftdt}. Given the above, an application of the strong Markov property similar to that in the proof of Lemma~\ref{lem:dtenttimestr3} yields
$$\Pb_\gamma(\{\phi^{k+1}_x<\infty\})=\Pb_\gamma(\{\phi^{k}_x<\infty\})\Pbx{\{\phi_x<\infty\}}\quad\forall k>0.$$
Iterating the above equation backwards, we obtain \eqref{eq:nfuieafbua}. 

Next, the total number of returns the chain makes to the state is equal to the number of return times that are finite:
\begin{equation}\label{eq:fmuawefbeywua}R=\sum_{k=1}^\infty 1_{\{\phi^{k}_x<\infty\}},\end{equation}
Because $\phi^k_x$ is finite only if $\phi^1_x,\dots,\phi^{k-1}_x$ are finite too, monotone convergence implies that
\begin{align}\label{eq:mdw78ahdnw78andwa}\Pbl{\{R=\infty\}}&=\Pbl{\left\{\sum_{k=1}^\infty 1_{\{\phi^{k}_x<\infty\}}=\infty\right\}}=\Pb_\gamma(\{\phi_x^k<\infty\enskip\forall k\in\zp\})\\
&=\lim_{k\to\infty}\Pb_\gamma(\{\phi_x^k<\infty\})=\Pb_\gamma(\{\phi_x<\infty\})\lim_{k\to\infty}\Pb_x(\{\phi_x<\infty\})^k\nonumber\\
&=\left\{\begin{array}{ll}\Pb_\gamma(\{\phi_x<\infty\})&\text{if $x$ is recurrent}\\0&\text{if $x$ is transient}\end{array}\right..\nonumber\end{align}
%
\end{proof}
The concepts of transience and recurrence are examples of \emph{class properties}\index{class property} in the sense that they hold for any one state in a communicating class if and only if they hold for every state in the communicating class:
\begin{theorem}[Transience and recurrence are class properties]\label{thrm:recclasspro} 
If $x$ is recurrent and $x\to y$, then $y$ is also recurrent and 
$$\Pbx{\{\phi_y<\infty\}}=\Pby{\{\phi_x<\infty\}}=1.$$
For this reason, (a) if $y$ is transient and $x\to y$, then $x$ is also transient and (b) a state in a communicating class is recurrent (resp. transient) if and only if all states in the class are recurrent (resp. transient). Moreover, a state is recurrent only if it belongs to a closed communicating class, and so the set $\cal{R}$ of all recurrent states equals the union of one or more closed communicating classes.
\end{theorem}

Hence, we say that a communicating class is \emph{recurrent} (resp. \emph{transient}) if any one (and therefore all) of its states is recurrent (resp. transient). For chains, we use the following slightly more elaborate classification:\index{recurrent class}\index{transient class}

\begin{definition}[Recurrent and transient chains]\label{def:recurrentchains}If the set $\cal{R}$\glsadd{rec} of all recurrent states is empty, then the chain is said to be transient\index{transient chain}. Otherwise, the chain is said to be Tweedie recurrent\index{Tweedie recurrent chain} if, regardless of the initial distribution, the chain has probability one of entering $\cal{R}$:
$$\Pbl{\{\phi_\cal{R}<\infty\}}=1\enskip\text{for all initial distributions}\enskip\gamma,$$
where $\phi_\cal{R}$ denotes the first entrance time to $\cal{R}$ (Definition~\ref{def:entrance}). If, additionally, there exists only one communicating class, then the chain is said to be Harris recurrent.\index{Harris recurrent chain} If the class is the entire state space, then the chain is simply called recurrent.\index{recurrent chain}
\end{definition}

We finish the section with the proof of Theorem~\ref{thrm:recclasspro} and a useful corollary thereof:
\begin{corollary}\label{cor:phixphic} If the chain ever enters a recurrent closed communicating class $\cal{C}$, it will eventually visit every state in $\cal{C}$:
$$1_{\{\phi_{\cal{C}}<\infty\}}=1_{\{\phi_x<\infty\}}\quad\forall x\in\cal{C},\quad\Pb_\gamma\text{-almost surely}.$$
\end{corollary}

\begin{proof}Because the chain must enter $\cal{C}$ to enter a state $x$ in $\cal{C}$ ($\phi_{\cal{C}}\leq \phi_x$ by definition), we need only show that
$$\Pbl{\{\phi_{\cal{C}}<\phi_x=\infty\}}=0.$$
Setting $\varsigma:=\phi_{\cal{C}}$ and $A:=\{x\}$ in Lemma~\ref{lem:dtenttimestr3},  we find that 
\begin{align*}\Pbl{\{\phi_{\cal{C}}<\phi_x=\infty\}}&=\sum_{z\in \cal{C}}\Pb_\gamma(\{\phi_{\cal{C}}<\phi_x,X_{\phi_{\cal{C}}}=z\})\Pbz{\{\phi_x=\infty\}}=0
\end{align*}
because $\Pbz{\{\phi_x=\infty\}}=1-\Pbz{\{\phi_x<\infty\}}=0$ for all $z,x$ in $\cal{C}$ (Theorem~\ref{thrm:recclasspro}),
\end{proof}
\subsubsection*{A proof of Theorem~\ref{thrm:recclasspro}}

Here, we need the following corollary of Theorem~\ref{thrm:rectrans}:

\begin{corollary}\label{cor:recnstepchar}For any initial distribution $\gamma$,
$$\sum_{n=1}^\infty p_n(x)=\frac{\Pbl{\{\phi_x<\infty\}}}{1-\Pbx{\{\phi_x<\infty\}}},$$
where $p_n$ denotes the time-varying law~\eqref{eq:dtlaw}.  
Thus, a state $x$ is recurrent if and only if 
$$\sum_{n=1}^\infty p_n(x,x)=\infty\quad\forall x\in\s,$$
where $P_n=(p_n(x,y))_{x,y\in\s}$ denotes the $n$-step matrix~\eqref{eq:nstepdef}.
\end{corollary}

\begin{proof}

Tonelli's theorem and (\ref{eq:nfuieafbua}--\ref{eq:fmuawefbeywua}) imply that
$$\sum_{n=1}^\infty p_n(x)=\Ebl{\sum_{k=1}^\infty1_x(X_n)}=\sum_{k=1}^\infty\Pbl{\left\{\phi_x^k<\infty\right\}}=\Pbl{\{\phi_x<\infty\}}\sum_{k=0}^\infty\Pbx{\{\phi_x<\infty\}}^k,$$
and the corollary follows from the geometric progression.
\end{proof}

\begin{proof}[Proof of Theorem \ref{thrm:recclasspro}] We begin by showing that if $x\to y$ and $x$ is recurrent, then $\Pby{\{\phi_x<\infty\}}=1$. The idea is that, because the chain will visit $x$ infinitely many times (Theorem \ref{thrm:rectrans}) and because it has positive probability of travelling from $x$ to $y$, it will travel to $y$ at some point. For the infinite returns to continue the chain must then travel back from $y$ to $x$. In particular, letting $x_1,\dots,x_l$ be as in Lemma~\ref{lem:accdt} and applying Proposition~\ref{prop:hitprobeqs}, we have that
\begin{align*}\Pbx{\{\phi_x<\infty\}}&=p(x,x)+\sum_{z\neq x}p(x,z)\Pb_{z}(\{\phi_x<\infty\})\leq 1-p(x,x_1)+p(x,x_1)\Pb_{x_1}(\{\phi_x<\infty\})\\
&=1-p(x,x_1)(1-\Pb_{x_1}(\{\phi_x<\infty\}).\end{align*}
Iterating the above, we find that 
$$1=\Pbx{\{\phi_x<\infty\}}\leq 1-p(x,x_1)p(x_1,x_2)\dots p(x_l,y)(1-\Pb_{y}(\{\phi_x<\infty\})),$$
and it follows that $x$ is accessible from $y$ (moreover, that $\Pb_{y}(\{\phi_x<\infty\})=1$). 

Next, we argue that, if $x$ is recurrent, then $y$ must also be recurrent. As we have just shown, $x$ is accessible from $y$. Hence, we can pick $y_1,\dots,y_k$ such that $p(y,y_1)p(y_1,y_2)\dots p(y_k,x)>0$ (Lemma~\ref{lem:accdt}). Using the definition of the $n$-step matrix in \eqref{eq:nstepdef} we obtain
$$p_{n+k+l+2}(y,y)\geq p(y,y_1)p(y_1,y_2)\dots p(y_k,x)p_n(x,x)p(x,x_1)p(x_1,x_2)\dots p(x_l,y)\quad\forall n\geq0.$$
Summing over $n$ and applying Corollary~\ref{cor:recnstepchar}, we find that
\begin{align*}\sum_{n=1}^\infty p_n(y,y)&\geq \sum_{n=1}^\infty p_{n+k+l+2}(y,y)\\
&\geq p(y,y_1)p(y_1,y_2)\dots p(y_k,x)\left(\sum_{n=1}^\infty p_n(x,x)\right)p(x,x_1)p(x_1,x_2)\dots p(x_l,y)=\infty,\end{align*}
proving the recurrence of $y$. Given that we already know that $y\to x$, we can reverse the roles of $x$ and $y$ and the argument above shows that $\Pbx{\{\phi_y<\infty\}}=1$.

Lastly, that a recurrent state must belong to a closed communicating class follows by noting that $x$ does not belong to a closed communicating class only if there exists a state $y$ accessible from $x$ that does not communicate with $x$.
\end{proof}

\subsubsection*{Notes and references} While the notion of Harris recurrence is commonplace in the literature, that of Tweedie recurrence is not. Texts that touch upon the latter leave it unnamed. There is one noticeable exception: these types of chains were called \emph{ultimately recurrent} in \citep{Tweedie1975b}. Presently (15/10/2019), a Google search for
\begin{center}``ultimately recurrent'' and ``markov''\end{center}
seems to imply that only two other articles have ever used this term. In an effort to uniformise the terminology with ``Harris recurrent chains'' (and with ``positive Harris recurrent chains'' later on)---and in recognition of R.~L.~Tweedie's exemplary work on the stability of Markov chains (work from which I learned much of what  I know on this subject)---I instead call these types of chains \emph{Tweedie recurrent}.

\subsection{The regenerative property}\label{sec:regen}
In the recurrent case, the strong Markov property implies much more than never ending returns: upon each return, the chain \emph{regenerates} in the sense that it forgets its entire past. 

The idea is that the only aspect of the chain's past and present that its future depends on is its current location. By setting the present to be the $k$th return time $\phi^k_x$ to a given state $x$, we ensure that the chain's location $X_{\phi^k_x}$ at this time is independent of its past: by definition $X_{\phi^k_x}$ is $x$ regardless of what occurred before $\phi^k_x$ (of course, if $\phi^k_x$ is infinite, it is ill-suited to represent the `present', something we avoid by requiring $x$ to be recurrent, see~Theorem~\ref{thrm:rectrans}). For these reasons, the segment the chain's path demarcated by two consecutive returns is independent of those demarcated by previous consecutive returns.  Furthermore, given the starting-afresh-from-$X_{\phi^k_x}$-when-conditioned-on-$X_{\phi^k_x}$ aspect of the  strong Markov property, the fact that $X_{\phi^k_x}$ always equals $x$ further implies that these path segments all have the same distribution (that of the $0$-to-$(\phi_x-1)$ path segment under $\Pb_x$). In summary, the sequence of these path segments is i.i.d., a very useful fact that opens the door to studying the long-term of chains using tools developed for i.i.d. sequences. For instance, Kolmogorov's classical strong law of large numbers will have a staring role in the next section.

To formalise the above discussion,  we define the random variable
\begin{equation}\label{eq:Ikdef}I_k^f:=1_{\{\phi^{k-1}_x<\infty\}}\sum_{m=\phi_x^{k-1}}^{\phi_x^{k}-1}f(X_m),\end{equation}
where $f:\s\to\r_E$ denotes any non-negative function, $k$ any positive integer, and we are using our convention for partially-defined random variables in~\eqref{eq:partdef}.
\begin{theorem}[The regenerative property]\label{thrm:rec-iid}  Let\index{regenerative property} $f:\s\to\r_E$ be any non-negative function on $\s$. For all $k>0$, the random variable $I_k^f$ in \eqref{eq:Ikdef} is $\cal{F}_{\phi_x^{k}}/\cal{B}(\r_E)$-measurable, where $\cal{F}_{\phi_x^{k}}$ denotes the pre-$\phi^k_x$ sigma algebra (Definition~\ref{def:stopdt}). If the state $x$ is recurrent, the sequence $(I_k^f)_{k\in\zp}$ is i.i.d. under $\Pb_x$.\end{theorem}
We do the proof in two steps, beginning by arguing the measurability of $I_k^f$. With this out of the way, we  then focus on showing that $(I_k^f)_{k\in\zp}$ is i.i.d.
\begin{proof}[Step 1: $I_k^f$ is $\cal{F}_{\phi_x^{k}}/\cal{B}(\r_E)$-measurable]Note that
$$I_k^f=\lim_{N\to\infty}\sum_{n=1}^N\sum_{m=0}^{n-1}\sum_{l=m}^{n-1}1_{\{\phi^k_x=n\}}1_{\{\phi^{k-1}_x=m\}}f(X_l).$$
Because limits, finite products, and finite sums of measurable functions are measurable functions, it suffices to show that each term in the sum is $\cal{F}_{\phi_x^{k}}/\cal{B}(\r_E)$-measurable. Because $\phi_{k-1}\leq \phi_{k}$ by definition, Lemma~\ref{lem:2stop}$i,iii$ shows that $1_{\{\phi^{k-1}_x=m\}}$ is $\cal{F}_{\phi_x^{k}}/\cal{B}(\r_E)$-measurable. Similarly, Lemma~\ref{lem:2stop}$i,ii$ implies that $1_{\{\phi^k_x=n\}}f(X_l)$ is  $\cal{F}_{\phi_x^{k}}/\cal{B}(\r_E)$-measurable for all $l<n$ and the result follows.
\end{proof}

\begin{proof}[Step 2: $(I_k^f)_{k\in\zp}$ is i.i.d.]To simplify the notation, we write $I_k$ for $I^f_k$ throughout this proof. Suppose that $x$ is recurrent, choose any $a_1,\dots,a_k$ in $\r_E$, and let
$$A_k:=\{I_1\geq a_1,\dots,I_{k-1}\geq a_{k-1},I_k\geq a_k\}.$$
Because the return times to $x$ are all finite with $\Pb_x$-probability one~\eqref{eq:nfuieafbua},
\begin{align*}
\Pbx{A_k}&=\Pb_x(\{I_1\geq a_1,\dots,I_{k-1}\geq a_{k-1},\phi^{k-1}_x<\infty,X_{\phi^{k-1}_x}=x,I_k\geq a_k\})\\
&=\Pb_x(\{I_1\geq a_1,\dots,I_{k-1}\geq a_{k-1},\phi^{k-1}_x<\infty,X_{\phi^{k-1}_x}=x,G(X^{\phi^{k-1}_x})\geq a_k\}),
\end{align*}
where $X^{\phi^{k-1}_x}$ denotes the $\phi^{k-1}_x$-shifted chain (Section~\ref{sec:dtstrmarkov}) and $G((x_{n})_{n\in\n}):=\sum_{m=0}^{\inf\{n>0:x_{n}=x\}-1}f(x_{m})$. 

Because we are able to deduce the values of $I_1,\dots,I_{k-1}$ by observing the chain up until the $(k-1)$th return, the event $A_{k-1}$ belongs to the pre-$\phi^{k-1}_x$ sigma-algebra $\cal{F}_{\phi^{k-1}_x}$~\eqref{def:stopdt}. Formally, Lemma~\ref{lem:2stop}$iii$ implies that 
$$\cal{F}_{\phi^{1}_x}\subseteq\cal{F}_{\phi^{k-1}_x}\subseteq\dots\subseteq\cal{F}_{\phi^{k-1}_x}$$
as $\phi^1_x\leq\phi^2_x\leq \dots\leq\phi^{k-1}_x$~(Exercise~\ref{ex:entdef}) and it follows from  Step 1 that $A_{k-1}$ belongs to $\cal{F}_{\phi^{k-1}_x}$.

For these reasons (and after checking the measurability requirement on $G$ using Lemma \ref{lem:pathspmeas}$iii$), we may apply the strong Markov property (Theorem~\ref{thrm:strmkvpath}) to obtain
$$\Pbx{A_k}=\Pb_x(A_{k-1}\cap\{\phi^{k-1}_x<\infty,X_{\phi^{k-1}_x}=x\})\Pbx{\{G(X_{n})_{n\in\n})\geq a_k\}}=\Pbx{A_{k-1}}\Pbx{\{I_1\geq a_k\}}.$$
%
%
Iterating the above backwards, we obtain 
\begin{equation}\label{eq:nfhudasfnsha}\Pbx{A_k}=\Pbx{\{I_1\geq a_1\}}\Pbx{\{I_1\geq a_2\}}\dots\Pbx{\{I_1\geq a_k\}}.\end{equation}
Setting $a_1=\dots=a_{k-1}=0$, we find that 
$$\Pbx{\{I_k\geq a_k\}}=\Pbx{\{I_1\geq a_k\}}.$$
Because $k$ and $a_1,\dots,a_k$ were arbitrary, and $\{[a,\infty]:a\in\r_E\}$ is a $\pi$-system that generates $\cal{B}(\r_E)$, the above and Lemma~\ref{lem:dynkinpl} show that the sequence $(I_k)_{k\in\zp}$ is identically distributed. For this reason, we can rewrite \eqref{eq:nfhudasfnsha} as 
$$\Pbx{A_k}=\Pbx{\{I_1\geq a_1\}}\Pbx{\{I_2\geq a_2\}}\dots\Pbx{\{I_k\geq a_k\}}.$$
The desired independence also follows from Lemma~\ref{lem:dynkinpl} as $\{[a_1,\infty]\times\dots\times[a_k,\infty]:a_1,\dots,a_k\in\r_E\}$ is a $\pi$-system that generates $\cal{B}(\r_E^k)$.
\end{proof}
\subsection{The empirical distribution and positive recurrent states}\label{sec:empdist}
Armed with the regenerative property, we can start attacking the chain's long-term behaviour more directly. Here, we use Kolmogorov's celebrated strong law of large numbers (Theorem~\ref{thrm:klln}) to study the limiting behaviour of the \emph{empirical distribution}\glsadd{epn}\index{empirical distribution} $\epsilon_N:=(\epsilon_N(x))_{x\in\s}$ whose $x$-entry $\epsilon_N(x)$ denotes the fraction of the first $N$ time-steps that the chain $X$ spends in state $x$:
\begin{equation}\label{eq:timeavedef}\epsilon_N(x):=\frac{1}{N}\sum_{n=0}^{N-1}1_x(X_n)\quad \forall x\in\s.\end{equation} 
From Theorem~\ref{thrm:rectrans}, we already know that the chain eventually stops visiting any given transient state $x$, and it follows that 
$$\lim_{N\to\infty}\epsilon_N(x)=0\quad\Pb_\gamma\text{-almost surely.}$$
In the case of recurrent states $x$, the visits do not cease and, to understand what happens with $\epsilon_N(x)$, we must study frequency of these visits. The trick here is to notice that $X$ does not visit state $x$ in between two consecutive return times, and so
$$\sum_{n=\phi^k_x}^{\phi^{k+1}_x-1}1_x(X_n)=1\quad\forall k>0.$$
Breaking down the path followed by $X$ into segments demarcated by consecutive return times and picking the return time $\phi^K_x$ closest to $N$, we find that 
$$\epsilon_N(x)\approx\frac{1}{\phi^K_x}\sum_{n=0}^{\phi^K_x-1}1_x(X_n)=\frac{1}{\phi^K_x}\sum_{k=0}^{K-1}\left(\sum_{n=\phi^{k}_x}^{\phi^{k+1}_x-1}1_x(X_n)\right)= \frac{K}{\phi^K_x}\quad\Pb_x\text{-almost surely},$$
for  $N$ large enough that $|\phi^K_x-N|/N$ is small. The regenerative property tells us that the sequence $(\phi^{k+1}_x-\phi^k_x)_{k\in\n}$ of elapsed times between consecutive visits is i.i.d. whenever the chain starts in $x$. For this reason,  the law of large numbers   implies that
$$\frac{K}{\phi^K_x}=\left(\frac{\phi^K_x}{K}\right)^{-1}=\left(\frac{1}{K}\sum_{k=0}^{K-1}(\phi^{k+1}_x-\phi^k_x)\right)^{-1}\approx (\Ebx{\phi_x})^{-1}\quad\Pb_x\text{-almost surely},$$
for large enough $K$. If the chain does not start in $x$, then noting that $\epsilon_N(x)$ is non-zero only if the chain ever visits $x$ and applying the Markov property yields the following characterisation of $\epsilon_N(x)$'s limiting behaviour (for a formal proof, see the end of the section):
\begin{theorem}\label{thrm:empdistlims}For all states $x$ in $\s$,
$$\epsilon_\infty(x):=\lim_{N\to\infty}\epsilon_N(x)=\frac{1_{\{\phi_x<\infty\}}}{\Ebx{\phi_x}}\quad\Pb_\gamma\text{-almost surely.}$$
\end{theorem}
The theorem shows that the fraction $\epsilon_\infty(x)$ of all time that the chain spends in a state $x$ is negligible  unless returns to $x$ happen quickly enough that the mean return time is finite. Or, in other words, unless $x$ is \emph{positive recurrent}:
\begin{definition}[Positive and null recurrent states]\label{def:posrec}
A state $x$ is said to be positive recurrent\index{positive recurrent state} if its mean return time is finite ($\Ebx{\phi_x}<\infty$). If $x$ is recurrent but its mean return time is infinite ($\Ebx{\phi_x}=\infty$), we say that it is null recurrent.\index{null recurrent state}
\end{definition}
%

\subsubsection*{Kolmogorov's law of large numbers and a proof of Theorem~\ref{thrm:empdistlims}}
As outlined above, the proof of Theorem~\ref{thrm:empdistlims} entails combining the regenerative property, the strong Markov property, and the law of large numbers (LLN)\index{law of large numbers}:
\begin{theorem}[Kolmogorov's strong law of large numbers]\label{thrm:klln} If $(W_n)_{n\in\zp}$ is a sequence of i.i.d. non-negative $\r_E$-valued random variables on a probability triplet $(\Omega,\cal{F},\Pb)$, then their sample average converges to $\Ebb{W_1}$ almost surely:
$$\lim_{N\to\infty}\frac{1}{N}\sum_{n=1}^N W_n= \Ebb{W_1}\quad\Pb\text{-almost surely.}$$
\end{theorem}

\begin{proof}The proof for the case of integrable finite-valued $W_n$s can be found in many books, e.g.\ \citep{Williams1991},  and we skip it. For the general case fix $M$ in $\n$, apply Kolmogorov's LLN to the sequence $(W_n\wedge M)_{n\in\zp}$ to get
$$\lim_{N\to\infty}\frac{1}{N}\sum_{n=1}^N W_n\wedge M = \Ebb{W_1\wedge M}\quad\Pb\text{-almost surely.}$$
Taking the limit $M\to\infty$ and applying monotone convergence then completes the proof.
\end{proof}
We do the proof  of Theorem~\ref{thrm:empdistlims} in two steps. We begin by using the regenerative property and the LLN to argue the limit in the case that the chain starts from the state $x$:
\begin{proof}[Proof of Theorem~\ref{thrm:empdistlims}: Step 1]For the time being, suppose that the chain starts at some state $x$. Note that
\begin{equation}\label{eq:fn8anuafnuiea}\epsilon_N(x)=\frac{1}{N}\sum_{n=0}^{N-1}1_x(X_n)= \frac{1_x(X_0)}{N}+\frac{R_N}{N}\quad\forall N>0,\end{equation}
where $R_N:=\sum_{n=1}^{N-1}1_x(X_n)$ is the total number of returns by time $N-1$ (with $R_1:=0$). If the state is transient, then Theorem~\ref{thrm:rectrans} shows that $R:=\sum_{n=1}^{\infty}1_x(X_n)$ is finite, $\Pb_x$-almost surely, and it follows  that
$$\lim_{N\to\infty}\epsilon_N(x)=\lim_{N\to\infty}\frac{R_N}{N}\leq \lim_{N\to\infty}\frac{R}{N}=0=\frac{1}{\Ebx{\phi_x}}\quad\Pb_x\text{-almost surely},$$
as $\Ebx{\phi_x}=\infty$.

Suppose instead that $x$ is recurrent. The chain's regenerative property (Theorem~\ref{thrm:rec-iid} with $f:=1$) and Kolmogorov's LLN (Theorem~\ref{thrm:klln} with $W_n:=1_{\{\phi^{n}_x<\infty\}}(\phi^{n+1}_x-\phi^{n}_x)$):
\begin{align}\label{eq:gkm84a0jg4ang7a4ahj}&\lim_{K\to\infty}\frac{\phi^K_x}{K}=\lim_{K\to\infty}\frac{1}{K}\sum_{k=0}^{K-1}(\phi^{k+1}_x-\phi^{k}_x)=\Ebx{\phi_x}\quad\Pb_x\text{-almost surely}\\
&\Rightarrow\lim_{K\to\infty}\frac{\phi^{K+1}_x-\phi^{K}_x}{K}=0\quad\Pb_x\text{-almost surely.}\label{eq:gkm84a0jg4ang7a4ahj2}\end{align}

Because $R_N$ is the number of visit made by time $N-1$, the $R_{N}$th return must occur by time $N-1$ and the next return must occur after this time:
$$\phi^{R_N}_x\leq N-1<\phi^{R_N+1}_x\quad\forall N>0,\enskip\Pb_x\text{-almost surely}.$$
It follows that,
\begin{equation}\label{eq:fm7ewa8fne78wan8fewa}\frac{\phi^{R_N}_x}{R_N}\leq\frac{N}{R_N}\leq \frac{\phi^{R_N+1}_x}{R_N}\quad\forall N>0,\enskip\Pb_x\text{-almost surely},\end{equation}
where we are using the convention that the two left-most terms are zero if $R_N=0$. 

Because $R_N$ increases at most in steps of one ($R_{N+1}-R_N\in\{0,1\}$), $R_N$ may only approach infinity by stepping through all positive integers. Because $x$ is recurrent, Theorem~\ref{thrm:rectrans} implies that $R_N\to\infty$ as $N\to\infty$, $\Pb_x$-almost surely, and it follows that
$$\Pbx{\{\cup_{N=1}^\infty \{R_N\}=\n\}}=1.$$
For this reason, (\ref{eq:gkm84a0jg4ang7a4ahj}--\ref{eq:gkm84a0jg4ang7a4ahj2}) implies that 
\begin{align*}&\lim_{N\to\infty}\frac{\phi^{R_N}_x}{R_N}=\lim_{K\to\infty}\frac{\phi^K_x}{K}=\Ebx{\phi_x},\\
&\lim_{N\to\infty}\frac{\phi^{R_N+1}_x}{R_N}=\lim_{K\to\infty}\frac{\phi^{K+1}_x}{K}=\lim_{K\to\infty}\frac{\phi^{K+1}_x-\phi^K_x}{K}+\lim_{N\to\infty}\frac{\phi^K_x}{K}=\Ebx{\phi_x},\quad\Pb_x\text{-almost surely};\quad\end{align*}
and it follows from~\eqref{eq:fm7ewa8fne78wan8fewa} that $\epsilon_N(x)\to1/\Ebx{\phi_x}$ as $N\to\infty$ with $\Pb_x$-probability one.

\end{proof}
To finish the proof, we now transfer the result from the starting-location-is-$x$ case to the arbitrary-initial-distribution-$\gamma$ case by noting that $\epsilon_\infty(x)$ is non-zero only if $\phi_x$ is finite, conditioning on $\phi_x$, and applying the strong Markov property:
\begin{proof}[Proof of Theorem~\ref{thrm:empdistlims}: Step 2]We have now left to show that, for any initial distribution $\gamma$,
\begin{equation}\label{eq:fna8wh483afnuw4dewasdea}\Pbl{\left\{\liminf_{N\to\infty}\epsilon_N(x)=\frac{1_{\{\phi_x<\infty\}}}{\Ebx{\phi_x}}\right\}}=1,\qquad\Pbl{\left\{\limsup_{N\to\infty}\epsilon_N(x)=\frac{1_{\{\phi_x<\infty\}}}{\Ebx{\phi_x}}\right\}}=1.\end{equation}
To do so, note that, by the definition of the entrance time $\phi_x$,
$$\epsilon_N(x)=\frac{1}{N}\sum_{n=0}^{N-1}1_x(X_n)=\frac{1_{x}(X_0)}{N}+\frac{1}{N}\sum_{n=1}^{N-1}1_x(X_n)=\frac{1_{x}(X_0)}{N}+\frac{1_{\{\phi_x<N\}}}{N}\sum_{n=\phi_{x}}^{N-1}1_x(X_n).$$
Given that the first term tends to zero as $N$ grows unbounded and that
$$\limsup_{N\to\infty}\frac{1_{\{\phi_{x}<N\}}}{N}\sum_{n=N}^{\phi_{x}+N-1}1_x(X_n)\leq \limsup_{N\to\infty}\frac{1_{\{\phi_{x}<N\}}\phi_{x}}{N}=\limsup_{N\to\infty}\frac{1_{\{\phi_{x}<\infty\}}\phi_{x}}{N}=0,$$
we have that
\begin{align}\label{eq:gnf98a70gh8ea4gaewa}\liminf_{N\to\infty}\epsilon_N(x)&=\liminf_{N\to\infty}\frac{1_{\{\phi_{x}<\infty\}}}{N}\sum_{n=\phi_{x}}^{N-1}1_x(X_n)=\liminf_{N\to\infty}\frac{1_{\{\phi_{x}<\infty\}}}{N}\sum_{n=\phi_{x}}^{\phi_{x}+N-1}1_x(X_n),\\
\label{eq:gnf98a70gh8ea4gaewa2}\limsup_{N\to\infty}\epsilon_N(x)&=\limsup_{N\to\infty}\frac{1_{\{\phi_{x}<\infty\}}}{N}\sum_{n=\phi_{x}}^{N-1}1_x(X_n)=\limsup_{N\to\infty}\frac{1_{\{\phi_{x}<\infty\}}}{N}\sum_{n=\phi_{x}}^{\phi_{x}+N-1}1_x(X_n).\end{align}
Given that 
$$F_-(X)=\liminf_{N\to\infty}\frac{1}{N}\sum_{n=0}^{N-1}1_x(X_n),\quad F_+(X)=\limsup_{N\to\infty}\frac{1}{N}\sum_{n=0}^{N-1}1_x(X_n),$$
where $F_-,F_+$ are as in Lemma~\ref{lem:pathspmeas}$i$ with $f:=1_x$, we may rewrite (\ref{eq:gnf98a70gh8ea4gaewa}--\ref{eq:gnf98a70gh8ea4gaewa2}) as
\begin{equation}\label{eq:dnw8anfn7m38whag9}\liminf_{N\to\infty}\epsilon_N(x)=1_{\{\phi_{x}<\infty\}}F_-(X^{\phi_{x}}),\qquad\limsup_{N\to\infty}\epsilon_N(x)=1_{\{\phi_{x}<\infty\}}F_+(X^{\phi_{x}}),\end{equation}
where $X^{\phi_x}$ denotes the $\phi_{x}$-shifted chain in~\eqref{eq:shiftdt}. Thus, 
$$\Pbl{\left\{\liminf_{N\to\infty}\epsilon_N(x)=\frac{1_{\{\phi_x<\infty\}}}{\Ebx{\phi_x}}\right\}}=\Pbl{\{\phi_x=\infty\}}+\Pbl{\left\{\phi_{x}<\infty,F_-(X^{\phi_x})=\frac{1}{\Ebx{\phi_x}}\right\}}.$$
Because we already showed that $F_-(X)= 1/\Ebx{\phi_x}$ with $\Pb_x$-probability one, applying the strong Markov property (Theorem~\ref{thrm:strmkvpath}) yields
\begin{align*}\Pbl{\left\{\phi_{x}<\infty,F_-(X^{\phi_{x}})=\frac{1}{\Ebx{\phi_x}}\right\}}&
=\Pbl{\left\{\phi_{x}<\infty,X_{\phi_{x}}=x,F_-(X^{\phi_{x}})=\frac{1}{\Ebx{\phi_x}}\right\}}\\
&=\Pbl{\left\{\phi_{x}<\infty,X_{\phi_{x}}=x\right\}}\Pbx{\left\{F_-(X)=\frac{1}{\Ebx{\phi_x}}\right\}}\\
&=\Pbl{\left\{\phi_{x}<\infty,X_{\phi_{x}}=x\right\}}=\Pbl{\{\phi_{x}<\infty\}},\end{align*}
and the leftmost equation in~\eqref{eq:fna8wh483afnuw4dewasdea} follows. For the rightmost one, replace `$\liminf$' and `$F_-$' with `$\limsup$' and `$F_+$' in all equations following~\eqref{eq:dnw8anfn7m38whag9}.
\end{proof}

\subsection[Stationary and ergodic distributions; Doeblin decomposition]{Stationary distributions, ergodic distributions, and a Doeblin-like decomposition}\label{sec:statdists}
A probability distribution $\pi$\glsadd{pi} on $\s$ is said to be a stationary distribution\index{stationary distribution} of the chain $X$ if sampling the initial condition from $\pi$ makes $X$ a stationary process. That is, one whose path law is invariant to time shifts:
\begin{equation}\label{eq:statprocdt}\Pbp{\left\{X^n\in A\right\}}=\Lbp{A}\quad\forall A\in\cal{E},\enskip n\geq0, \end{equation}
where $\mathbb{L}_\pi$ denotes the path law with initial distribution $\pi$, $\cal{E}$ denotes the sigma-algebra generated by the cylinder sets of the path space $\cal{P}$~(Section~\ref{pathlawuni} for the definitions of $\mathbb{L}_\pi,\cal{P},\cal{E}$), and $X^n$ the $n$-shifted chain~\eqref{eq:shiftdt}. Setting $A$ in the above to be the slice $\{(x_m)_{m\in\n}\in\cal{P}:x_0=x\}$ of $\cal{P}$, we find that
\begin{equation}\label{eq:lawstat}\Pbp{\{X_n=x\}}=\pi(x)\quad\forall x\in\s,\enskip n\geq0.\end{equation}
In other words, if the chain starts with distribution $\pi$, then it remains with distribution $\pi$ for all time. Applying \eqref{eq:dtlaw} to the above we find that 
\begin{equation}\label{eq:statd}\pi P(x)=\pi(x)\quad\forall x\in\s,\end{equation}
or $\pi P=\pi$ for short. Conversely, marginalising over \eqref{eq:nstepthe} and applying \eqref{eq:statd} yields \eqref{eq:statprocdt} for all sets $A$ of the form in \eqref{eq:gensets}, and it follows from Theorem~\ref{pathlawuni} that
\begin{theorem}[Analytical characterisation of the stationary distributions]\label{pstateqs} A probability distribution $\pi$ on $\s$ is a stationary distribution of $X$ if and only if it satisfies \eqref{eq:statd}.\end{theorem}

\subsubsection*{A stationary distribution has support on a state if and only if the state is positive recurrent}If the chain is a stationary process, then the average fraction of time-steps it spends in any given state remains constant throughout time. In particular, sampling the starting location from a stationary distribution $\pi$, taking expectations of the empirical distribution $\epsilon_N$~(Section~\ref{sec:empdist}), and applying~\eqref{eq:lawstat}, we obtain that
\begin{equation}\label{eq:pntimeave}\Ebp{\epsilon_N(x)}=\Ebp{\frac{1}{N}\sum_{n=0}^{N-1}1_x(X_n)}=\frac{1}{N}\sum_{n=0}^{N-1}\Pbp{\{X_n=x\}}=\pi(x)\quad\forall x\in\s,\enskip N>0.\end{equation}
Using Theorem~\ref{thrm:empdistlims}, taking the limit $N\to\infty$, and applying bounded convergence, we find that
\begin{equation}\label{eq:empdistmean}\pi(x)=\lim_{N\to\infty}\Ebp{\epsilon_N(x)}=\frac{\Pbp{\{\phi_x<\infty\}}}{\Ebx{\phi_x}}\quad\forall x\in\s.\end{equation}
Thus, there exists a stationary distribution $\pi$ with support on $x$ (i.e., with $\pi(x)>0$) only if $x$ is positive recurrent~(Definition~\ref{def:posrec}). The intuition here is that, if the initial position is sampled from $\pi$, then the states $x$ satisfying $\pi(x)>0$ are those that the chain will visit often enough that the fraction of time-steps $\epsilon_N(x)$ it spends in the state does not decay to zero. Because, regardless of the starting location, visits to any non-positive-recurrent state $x$ eventually become so rare that $\epsilon_N(x)$ collapses to zero (Theorem~\ref{thrm:empdistlims}), it follows that $x$ must be positive recurrent for $\pi(x)$ to be non-zero.

The converse is also true and the trick in arguing it involves the stopping distributions and occupation measures of the Section~\ref{sec:morehitdt}. In particular, let $\mu_S$ and $\nu_S$ be the space marginals~(\ref{eq:mus}--\ref{eq:nus}) of the stopping distribution and occupation measure associated with the entrance time $\phi_x$ of any given recurrent state $x$ and set the chain's starting location to be  $x$. At the moment of entry, the chain is at $x$. Because $x$ is recurrent, it follows that $\mu_S$ is the point mass $1_x$ at $x$. Because we fixed the initial distribution $\gamma$ to also be $1_x$,  Corollary~\ref{eqnsd2} then shows that $\nu_S=\nu_S P$. Recall that the mass of $\nu_S$ is the mean return time, c.f.~\eqref{eq:numassd}. Hence,  if $x$ is positive recurrent, then we obtain a probability distribution $\pi$ satisfying that $\pi=\pi P$ by normalising $\nu_S$ (i.e.\ $\pi:=\nu_S/\nu_S(\s)$). Moreover, 
$$\pi(x)=\frac{\nu_S(x)}{\nu_S(\s)}=\frac{\Ebx{\sum_{n=0}^{\phi_x-1}1_x(X_n)}}{\Ebx{\phi_x}}=\frac{\Ebx{1_x(X_0)}}{\Ebx{\phi_x}}=\frac{\Pbx{\{X_0=x\}}}{\Ebx{\phi_x}}=\frac{1}{\Ebx{\phi_x}}>0.$$
In summary:
\begin{theorem}\label{thrm:posrecchar} A state $x$ is positive recurrent if and only if there exists a stationary distribution $\pi$ such that $\pi(x)>0$. Moreover, if $x$ is positive recurrent, then one such stationary distribution is given by
$$\pi(y)=\frac{1}{\Ebx{\phi_x}}\Ebx{\sum_{n=0}^{\phi_x-1}1_y(X_n)}\quad\forall y\in\s.$$
\end{theorem}
%

A handy consequence of this theorem is that positive and null recurrence are class properties:
\begin{corollary}\label{cor:posrecclass} If $x$ is positive recurrent and $x\to y$, then $y$ is also positive recurrent. Moreover, a state in a communicating class is positive (resp. null) recurrent if and only if all states in the class are positive (resp. null) reccurent.
\end{corollary}
For this reason, we say that a communicating class is \emph{positive recurrent} (resp. \emph{null recurrent}) if each one (or, equivalently, all) of its states is positive recurrent (resp. null recurrent).\index{null recurrent class}\index{positive recurrent class}
\begin{proof} Let $x_1,\dots,x_l$ be as in \eqref{lem:accdt} and suppose that $x$ is positive recurrent. Let $\pi$ be the stationary distribution in Theorem \ref{thrm:posrecchar} satisfying $\pi(x)>0$. Using \eqref{eq:statd} repeatedly we find that
$$\pi(y)=\sum_{z\in\s}\pi(z)p(z,y)\geq \pi(x_l)p(x_l,y)\geq\dots\geq \pi(x)p(x,x_1)p(x_1,x_2)\dots p(x_l,y)>0.$$
Applying Theorem \ref{thrm:posrecchar} then shows that $y$ is positive recurrent. That positive recurrence is a class property follows immediately. For null recurrence, the result then follows because recurrence is a class property (Theorem \ref{thrm:recclasspro}) and because recurrent states are either positive recurrent or null recurrent.
\end{proof}

\subsubsection*{Ergodic distributions; Doeblin-like decomposition; set of stationary distributions}

We wrap up and our treatment of the stationary distributions by characterising the set thereof. Here, we require the following \emph{Doeblin-like} decomposition\index{Doeblin-like decomposition} of the state space:
\begin{equation}\label{eq:ddd}\s=\left(\bigcup_{i\in\cal{I}} \C_i\right)\cup \cal{T}=\cal{R}_+\cup\cal{T},\end{equation}
where $\{\C_i:i\in\cal{I}\}$ denotes the (necessarily countable) set of positive recurrent closed communicating classes,\glsadd{ci} which we index with some  set $\cal{I}$, $\cal{R}_+:=\cup_{i\in\cal{I}}\C_i$ the set of all positive recurrent states,\glsadd{recp} and $\cal{T}$ that of all other states.\glsadd{nrec} As we have already seen in Theorem~\ref{thrm:posrecchar}, no stationary distribution has support in $\cal{T}$ and there at least one stationary distribution per $\cal{C}_i$. Much more can be said:
%
%
\begin{theorem}[Characterising the set of stationary distributions]\label{doeblind}  Let $\{\C_i:i\in\cal{I}\}$ be the collection of positive recurrent closed communicating classes.
\vspace{5pt}
\begin{enumerate}[label=(\roman*),noitemsep]
\item For each $\cal{C}_i$, there exists a single stationary distribution $\pi_i$ with support contained in $\cal{C}_i$ (i.e., with $\pi_i(\cal{C}_i)=1$); $\pi_i$ is  known as the ergodic distribution\index{ergodic distribution}\glsadd{pii} associated with $\cal{C}_i$. 
It has support on all of $\C_i$ ($\pi_i(x)>0$ for all states $x$ in $\cal{C}_i$) and can be expressed as
\begin{equation}\label{eq:ergdistchar}\pi_i(y)=\frac{1_{\cal{C}_i}(y)}{\Eby{\phi_y}}=\frac{1}{\Ebx{\phi_x}}\Ebx{\sum_{n=0}^{\phi_x-1}1_y(X_n)}\quad\forall y\in\s,\end{equation}
where $x$ is any state in $\cal{C}_i$.
\item A probability distribution $\pi$ is a stationary distribution of the chain if and only if it is a convex combination of the ergodic distributions:
$$\pi=\sum_{i\in\cal{I}}\theta_i\pi_i$$
for some collection $(\theta_i)_{i\in\cal{I}}$ of non-negative constants satisfying $\sum_{i\in\cal{I}}\theta_i=1$. For any stationary distribution $\pi$, the weight $\theta_i$ featuring in the above is the mass that $\pi$ awards to $\cal{C}_i$ or, equivalently, the probability that the chain ever enters $\cal{C}_i$ if its starting location was sampled from $\pi$:
$$\theta_i=\pi(\cal{C}_i)=\Pbp{\{\phi_{\cal{C}_i}<\infty\}}\quad\forall i\in\cal{C}_i.$$
\end{enumerate} 
\end{theorem}

\begin{proof} $(i)$ For any state $x$ in a recurrent closed communicating class $\cal{C}_i$, Proposition~\ref{prop:closed} and Corollary~\ref{cor:phixphic} imply that
$$\Pbx{\{\phi_y<\infty\}}=1_{\cal{C}_i}(y)\quad\forall y\in\s.$$
Combining the above into \eqref{eq:empdistmean} we find that there can only exist one stationary distribution $\pi_i$ with support contained in  $\cal{C}_i$ (i.e., $\pi_i(x)=0$ for all $x\not\in\cal{C}_i$), as for any such $\pi_i$,
\begin{equation}\label{eq:nfe8a0n78aobnfygea}\pi_i(x)=\frac{\Pb_{\pi_i}(\{\phi_x<\infty\})}{\Ebx{\phi_x}}=\frac{\sum_{x'\in\cal{C}_i}\pi_i(x')\Pb_{x'}(\{\phi_x<\infty\})}{\Ebx{\phi_x}}=\frac{1_{\cal{C}_i}(x)}{\Ebx{\phi_x}}\quad\forall x\in\s.\end{equation}
On the other hand,  Theorem~\ref{thrm:posrecchar} shows that the right-hand side of \eqref{eq:ergdistchar} defines such a   $\pi_i$ as
\begin{align*}\Ebx{\phi_x}\pi_i(y)&=\Ebx{\sum_{n=0}^{\phi_x-1}1_y(X_n)}=\Ebx{\sum_{n=1}^{\phi_x-1}1_y(X_n)}\leq \Ebx{\sum_{n=1}^{\infty}1_y(X_n)}= \Ebx{\sum_{k=1}^\infty 1_{\{\phi^{k}_y<\infty\}}}\\
&=\sum_{k=1}^\infty \Pbx{\left\{\phi^{k}_y<\infty\right\}}= \Pbx{\{\phi_y<\infty\}}\sum_{k=1}^\infty\Pby{\{\phi_y<\infty\}}^{k-1}=0\quad\forall  y\not\in\cal{C}_i,\end{align*}
where $x$ is any state in $\cal{C}_i$ and we have made use of the closedness of $\cal{C}_i$ and (\ref{eq:nfuieafbua}--\ref{eq:fmuawefbeywua}).

$(ii)$ Given that $\cal{T}$ does not contain any positive recurrent states, \eqref{eq:empdistmean} shows that  $\pi(x)=0$ for all $x$ in $\cal{T}$. Because of this, we may rewrite any stationary distribution as
\begin{equation}\label{eq:nfuiweanfiaweu}\pi(x)=\sum_{i\in\cal{I}}\pi(\C_i)\left(\frac{1}{\pi(\C_i)}1_{\C_i}(x)\pi(x)\right)\quad\forall x\in\s.\end{equation}
Because $\C_i$ is a closed communicating class, we have that
$$(1_{\C_i}\pi)P(x)=\pi P(x)=\pi(x)=1_{\C_i}(x)\pi(x)\qquad\forall x\in\C_i.$$
Thus, assuming that $\pi(\C_i)>0$, Theorem~\ref{pstateqs} shows that $1_{\C_i}\pi/\pi(\C_i)$ is a stationary distribution with support contained in $\C_i$. Part~$(i)$ then shows  that  $1_{\C_i}\pi/\pi(\C_i)$ must be equal to $\pi_i$ and we can rewrite  \eqref{eq:nfuiweanfiaweu} as
$$\pi(x)=\sum_{i\in\cal{I}}\theta_i\pi_i(x)\quad\forall x\in\s,$$
where $\theta_i:=\pi(\cal{C}_i)$. Summing both sides over $x$ in $\s$ and applying Tonelli's theorem then shows that $\sum_{i\in\cal{I}}\theta_i=1$ (i.e.\ $\pi$ is a convex combination of the ergodic distributions).

To show that $\theta_i=\Pbp{\{\phi_{\cal{C}_i}<\infty\}}$ and complete the proof, note that $\Pbp{\{\phi_{\cal{C}_i}<\infty\}}=\Pbp{\{\phi_x<\infty\}}$ if $x$ belongs to $\cal{C}_i$ (Corollary~\ref{cor:phixphic}) and compare the above with \eqref{eq:empdistmean} and \eqref{eq:nfe8a0n78aobnfygea}.
\end{proof}

\subsubsection*{Notes and references}The term ``Doeblin decomposition'' traditionally refers to~\eqref{eq:ddd} with  $\{\C_i:i\in\cal{I}\}$ including all recurrent closed communicating classes instead of only the positive recurrent ones. In this case, $\cal{T}$ is the set of all transient states and, consequently, is known as the \emph{transient set}. The motivation behind my unorthodox choice is that, for reasons that we have touched upon in the last two sections and that will become completely clear in Sections~\ref{sec:timeaved} and \ref{sec:ensembleaved}, only positive recurrent states typically matter in the long-run.

\ifdraft
Named after Wolfgang Doeblin who first established it for discrete-time Markov processes with general state spaces in \citep{Doeblin1940}. Doeblin was a student of Paul L\'evy's with a fascinating story \citep{Handwerk2007}. For instance, famously, before his death at the age of $25$, several years before Kiyosi It\^o's published his renown work on SDEs, and while on the front lines of WWII, Doeblin wrote a series of notes which, among other results, included a time-change representation of one-dimensional SDEs \citep{Bru2002}. These notes were mailed to the Academy of Sciences in Paris where they remained sealed in their envelope for the ensuing $60$ years (until $`00$)!
\fi

\ifdraft
{\color{red}\subsubsection*{Notes and references.} Aside of theoretically useful (e.g., see Sec.~\ref{}), the formula \eqref{eq:} combined with the regenerative property (Theorem \ref{}) opens up an important avenue for computing stationary distributions in practice \citep{Glynn2006}}.
\fi

\ifdraft

\subsubsection*{Notes and references} To the best of our knowledge, the Theorem \ref{doeblind}, as presented here, is due to Kai Lai Chung \citep[Thrm. I.7.2]{Chung1967}. Stated therein ``Theorems 2 and 3 were .given by the author in his lectures in 1950;
see Loeve [1]''. Also about the existence of ergodic distributions: ``Theorem 1 (for a positive class of period one) seems to be
due to Feller [3], although Kolmogorov [3] proved (2) and (3)
without the intervention of (1). It should be noted, however, that it
is feasible to state the theorem in the form given here, and not merely
for a positive class. Otherwise the determination of a positive convergent
solution {«,} would not by itself prove that the class is positive. This
would be more than a nuisance in practice, since the only general
method of showing that a class is positive is precisely to solve the''

\fi

\subsection[Emipirical distribution's limits; positive recurrent chains]{Limits of the empirical distribution and positive recurrent chains}\label{sec:timeaved}
Almost without realising it, we have derived over the last few sections a complete description of the asymptotic behaviour of the empirical distribution $\epsilon_N$~(Section~\ref{sec:empdist}) tracking the fraction of time that the chain spends in each state:
\begin{theorem}[The pointwise limits]\label{thrm:pointwiselimd} Let $\epsilon_N$ denote the empirical distribution~\eqref{eq:timeavedef}, $\{\cal{C}_i:i\in\cal{I}\}$ the collection of positive recurrent closed communicating classes $\cal{C}_i$~(Section~\ref{sec:statdists}), and $\pi_i$ the ergodic distribution of $\cal{C}_i$~(Theorem~\ref{doeblind}) for each $i$ in $\cal{I}$. For any initial distribution $\gamma$, we have that\glsadd{epinf}
\begin{equation}\label{eq:reclims}\epsilon_\infty:=\lim_{N\to\infty}\epsilon_N=\sum_{i\in\cal{I}}1_{\{\phi_{\cal{C}_i}<\infty\}}\pi_i\quad\Pb_\gamma\text{-almost surely},\end{equation}
where  the convergence is pointwise and $\phi_{\C_i}$ denotes the time of first entrance to $\C_i$~(Definition~\ref{def:entrance}).
\end{theorem}
\begin{proof}Because, with $\Pb_\gamma$-probability one, the chain will enter a state $x$ (i.e.\ $\phi_x<\infty$) in a recurrent closed communicating class $\cal{C}_i$ if and only if enters the class (i.e.\ $\phi_{\cal{C}_i}<\infty$) at all (Corollary~\ref{cor:phixphic}), this follows directly from Theorems~\ref{thrm:empdistlims} and \ref{doeblind}$(i)$.
\end{proof}

You may be wondering under what circumstances the convergence in \eqref{eq:reclims} can be strengthened. The answer turns out to be surprisingly simple:
\begin{corollary}[The convergence in total variation]\label{cor:dttimeavetv} The chain enters the set $\cal{R}_+$ of positive recurrent states with probability one (i.e.\ $\Pbl{\{\phi_{\cal{R}_+}<\infty\}}=1$) if and only if the  limit \eqref{eq:reclims} holds in total variation with $\Pb_\gamma$-probability one:
$$\lim_{n\to\infty}\norm{\epsilon_N-\epsilon_\infty}=0\quad\Pb_\gamma\text{-almost surely},$$
where $\norm{\rho}$ denotes the total variation norm\glsadd{normtv} of any signed measure $\norm{\rho}$ on $\s$:\index{total variation norm}
\begin{equation}\label{eq:tvnorm}\norm{\rho}:=\sup_{A\subseteq\s}\mmag{\rho(A)}.\end{equation}
\end{corollary}

The trick in proving the above is the following variation of Scheff\'e's lemma:

\begin{lemma}[Scheff\'e's lemma]\label{lem:scheffe} Suppose that $\rho_1,\rho_2,\dots$ is a sequence of probability distributions on $\s$ that converge pointwise to a limit $\rho$. The limit $\rho$ is a probability distribution if and only if the sequence converges in total variation.
\end{lemma}

\begin{proof} If the sequence converges in total variation, then $\rho$ is a probability distribution (as $1=\rho_n(\s)\to\rho(\s)$). To prove the converse, we will show later that
 \begin{equation}\label{eq:tvl1p}\norm{\rho_n-\rho}=\sum_{x\in\s}1_{U_n}(x)(\rho(x)-\rho_n(x))=\frac{1}{2}\sum_{x\in\s}\mmag{\rho_n(x)-\rho(x)},\end{equation}
because $\rho_n$ and $\rho$ are probability distributions, where $U_n:=\{x\in\s:\rho_n(x)<\rho(x)\}$ is the set of states for which $\rho_n$ underestimates $\rho$. 
%
Because $0\leq 1_{U_n}(x)(\rho(x)-\rho_n(x))\leq \rho(x)$ for all $x$ in $\s$ and $n$ in $\n$, and because the pointwise convergence implies that
$$\lim_{n\to\infty}1_{U_n}(x)(\rho(x)-\rho_n(x))=0\quad\forall x\in\s,$$
taking limit $n\to\infty$ in both sides of the first equation in \eqref{eq:tvl1p} and applying dominated convergence then proves the convergence in total variation.

To finish this proof, we need to argue \eqref{eq:tvl1p}. Note that
\begin{align*}\rho_n(A)-\rho(A)=(\rho_n(A\cap U_n)-\rho(A\cap U_n))+(\rho_n(A\cap U_n^c)-\rho(A\cap U_n^c))\quad\forall A\subseteq\s,\end{align*}
where $ U_n^c$ denotes the complement of $U_n$. Because the bracketed terms have opposite signs, 
\begin{align}\label{eq:fne78awnfyea}\rho(U_n)-\rho_n(U_n)&\leq \sup_{A\subseteq\s}\mmag{\rho_n(A)-\rho(A)}\\
&\leq \sup_{A\subseteq\s}\max\{\rho(A\cap U_n)-\rho_n(A\cap U_n),\rho_n(A\cap U_n^c)-\rho_n(A\cap U_n^c)\}\nonumber\\
&\leq  \max\{\rho(U_n)-\rho_n(U_n),\rho_n(U_n^c)-\rho_n(U_n^c)\}.\nonumber\end{align}
However,
$$\rho(U_n)-\rho_n(U_n)=1-\rho(U_n^c)-(1-\rho_n(U_n^c))=\rho_n(U_n^c)-\rho(U_n^c)$$
and the first equation in \eqref{eq:tvl1p} follows from \eqref{eq:fne78awnfyea}. The second equation then also follows, as 
$$2\norm{\rho_n-\rho}=\rho(U_n)-\rho_n(U_n)+1-\rho(U_n^c)-(1-\rho_n(U_n^c))=\sum_{x\in\s}\mmag{\rho_n(x)-\rho(x)}.$$ 
\end{proof}

\begin{proof}[Proof of Corollary~\ref{cor:dttimeavetv}] The probability of both $\{\phi_{\cal{C}_i}<\infty\}$ and $\{\phi_{\cal{C}_j}<\infty\}$ occurring if $i\neq j$ is zero because the classes are closed sets (Proposition~\ref{prop:closedis}). Given that every positive recurrent state $x$ belongs to a positive recurrent closed communicating class~(Theorem~\ref{thrm:recclasspro}~and~Corollary~\ref{cor:posrecclass}), it follows that
$$\sum_{i\in\cal{I}}1_{\{\phi_{\cal{C}_i}<\infty\}}\pi_i(\s)=\sum_{i\in\cal{I}}1_{\{\phi_{\cal{C}_i}<\infty\}}=1_{\{\phi_{\cup_{i\in\cal{I}}\cal{C}_i}<\infty\}}=1_{\{\phi_{\cal{R}_+}<\infty\}}\quad\Pb_\gamma\text{-almost surely}.$$
For this reason, that the limit in \eqref{eq:reclims} is a probability distribution $\Pb_\gamma$-almost surely if and only if $\Pbl{\{\phi_{\cal{R}_+}<\infty\}}=1$ and the result follows from Theorem~\ref{thrm:pointwiselimd} and Lemma~\ref{lem:scheffe}. 
\end{proof}

\subsubsection*{Positive recurrent chains}
Corollary~\ref{cor:dttimeavetv} motivates the following definitions:
%
\begin{definition}[Positive recurrent chains]\label{def:posTweedie} A chain is positive Tweedie  recurrent\index{positive Tweedie recurrent chain} if 
$$\Pbl{\{\phi_{\cal{R}_+}<\infty\}}=1\text{ for all initial distributions }\gamma,$$
where $\cal{R}_+$ denotes the set of positive recurrent states. If, additionally, there is only one closed communicating class, then the chain is said to be positive Harris recurrent\index{positive Harris recurrent chain}. If this class is the entire state space, then the chain is simply said to be positive recurrent.\index{positive recurrent chain}
\end{definition}
%

The fraction of all time that the chain spends in any given finite set $F$ is given by
$$\epsilon_\infty(F)=\sum_{x\in F}\epsilon_\infty(x)=\lim_{N\to\infty}\sum_{x\in F}\epsilon_N(x)=\lim_{N\to\infty}\sum_{x\in F}\frac{1}{N}\sum_{n=0}^{N-1}1_x(X_n)=\lim_{N\to\infty}\frac{1}{N}\sum_{n=0}^{N-1}1_F(X_n),$$
$\Pb_\gamma$-almost surely. If the chain is not positive Tweedie recurrent, then Theorem~\ref{thrm:pointwiselimd} tells us that there exists at least one initial distribution $\gamma$ such that
$$\Pbl{\{\epsilon_\infty(F)=0\}}=\Pbl{\left\{\sum_{i\in\cal{I}}1_{\{\phi_{\cal{C}_i}<\infty\}}\pi_i(F)=0\right\}}\geq \Pbl{\{\phi_{\cal{R}_+}=\infty\}}>0.$$
In other words, with this initial distribution and non-zero probability, \emph{the chain will spend far  more (indeed, infinitely more) time-steps outside $F$ than inside}. Because this non-zero probability is bounded below by a positive constant independent of the finite set $F$, we have that a \emph{non-negligible fraction of $X$'s paths will spend infinitely more time-steps outside any given finite set $F$ than inside the set:}
\begin{exercise}To formally argue the above sentence introduce a sequence $(\s_r)_{r\in\zp}$ of finite subsets (or truncations) of the state space $\s$ that approach $\s$ (i.e., with $\cup_{r=1}^\infty\s_r=\s$). Applying downwards monotone convergence, show that
$$\Pbl{A}=\lim_{r\to\infty}\Pbl{\{\epsilon_\infty(\s_r)=0\}}\geq\Pbl{\{\phi_{\cal{R}_+}=\infty\}},$$
where $A:=\bigcap_{r=1}^\infty\{\epsilon_\infty(\s_r)=0\}$. Because any finite set $F$ is included in a large enough truncation $\s_r$, conclude that $\omega$ belongs to $A$ if and only if
$$\omega\in\{\epsilon_\infty(F)=0\}\text{ for all finite subsets }F\text{ of }\s.$$
\end{exercise}
Conversely, if the chain is positive Tweedie recurrent, then, for any given initial distribution $\gamma$ and $\varepsilon$ in $(0,1]$, we are always able to find a finite set $F$ large enough that the chain spends at least $(1-\varepsilon)\times100\%$ of all time-steps inside $F$:
\begin{equation}\label{eq:fn8a0nf7fdasesee8af}\epsilon_\infty(F)\geq 1-\varepsilon\quad\Pb_\gamma\text{-almost surely}.\end{equation}
For this reason, I believe that positive Tweedie recurrence precisely describes what most practitioners think of when they hear the words `a stable chain'. Moreover, the empirical distribution converges to a probability distribution if and only if the chain is positive Tweedie recurrent (Corollary~\ref{cor:dttimeavetv}). As we will see later in Section~\ref{sec:ensembleaved}, the same holds for the time-varying law if the chain is aperiodic (if it is periodic, the time-varying law will not converge to anything but instead oscillate).
\begin{exercise}\label{ex:Tweedieepinf}Using Theorem~\ref{thrm:pointwiselimd} give a formal proof that the chain is positive Tweedie recurrent if and only if, for any given initial distribution $\gamma$ and  $\varepsilon$ in $(0,1]$, there exists a finite set $F$ such that \eqref{eq:fn8a0nf7fdasesee8af} holds, $\Pb_\gamma$-almost surely.\end{exercise}

\subsubsection*{Notes and references} Just as with Harris recurrence and Tweedie recurrence (Section~\ref{sec:rec}), positive Harris recurrence is a frequently encountered name in the literature while positive Tweedie recurrence is not. In the past, the latter has occasionally been referred to as non-dissipativity \citep{Kendall1951a,Foster1952,Mauldon1957}. 

\ifdraft

\subsubsection*{Notes} {\color{red}\eqref{thrm:pointwiselimd} provides the justification for the naive monte-carlo approach of estimating an ergodic distribution.} Explain choice of jargon Tweedie. Mention the ratio limits as exercises?

Early instances of Tweedie looking at reducible chains:

Truncation Procedures for Non-Negative Matrices 

Sufficient conditions for regularity, recurrence and ergodicity of Markov processes

Maybe his thesis? R-Theory and Truncation Algorithms for Markov Chains and Processes

Obv, he's not the first to look at long-term of reducible stuff, e.g.,

THE EXPONENTIAL DECAY OF MARKOV TRANSITION PROBABILITIES

but ....
\fi

\ifdraft

\subsection{The law of large numbers}
The limit \eqref{eq:reclims} is often regarded as a \emph{law of large numbers (LLN)}. To flesh this idea out, we need the following.
\begin{proposition}[Convergence in total variation equals weak convergence for countable spaces]\label{prop:tvweak} A sequence $\rho_1,\rho_2,\dots$ of probability distributions  on $\s$ converges in total variation to a limit $\rho$ on $\s$ if and only if 
\begin{equation}\label{eq:nf7a89fn798awfnuraf}\lim_{n\to\infty}\rho_n(f)=\rho(f)\end{equation}
for all bounded real-valued functions $f$ on $\s$.
\end{proposition}

\begin{proof}Let $f$ be any bounded real-valued function $f$ on $\s$ with bounding constant $C$ (i.e., $\mmag{f(x)}\leq C$ for all $x\in\s$). If the sequence converges in total variation, then $\rho$ is a probability distribution (as $1=\rho_n(\s)\to\rho(\s)$ and $0\leq\rho_n(x)\to\rho(x)$ for all $x\in\s$). For this reason,  \eqref{eq:tvl1p} implies that
\begin{align*}\mmag{\rho_n(f)-\rho(f)}&\leq \sum_{x\in\s}\mmag{f(x)}\mmag{\rho_n(x)-\rho(x)}\leq C \sum_{x\in\s}\mmag{\rho_n(x)-\rho(x)}\leq 2C\norm{\rho_n-\rho}\end{align*}
Taking the limit $n\to\infty$ then shows that $\rho_n(f)$ converges to $\rho(f)$ if the sequence converges in total variation.

Conversely, setting choosing $f$s in \eqref{eq:nf7a89fn798awfnuraf} that are indicator functions of the states, we find that $\rho_n$ converges pointwise to $\rho$. Setting $f:=1$ in \eqref{eq:nf7a89fn798awfnuraf} then shows that $\rho$ is a probability distribution and it follows from Scheffe's lemma~\eqref{lem:scheffe} that $\rho_n$ converges to $\rho$ in total variation.
\end{proof}

In the case of  a bounded function $f$ and a positive Harris recurrent chain with stationary distribution $\pi$, Corollary~\ref{cor:dttimeavetv} and the above proposition tell us that the sample average of the sequence $(f(X_n))_{n\in\n}$ of random variables converges almost surely to $\pi(f)$:
\begin{equation}\label{eq:llnmarkov}\lim_{N\to\infty}\frac{1}{N}\sum_{n=0}^{N-1}f(X_n)=\pi(f)\quad\Pb_\gamma\text{-almost surely,}\end{equation}
for any initial distribution $\gamma$. The above equation is viewed as a generalisation of the law of large numbers as the random variables in question need not (and generally will not) be independent nor have law $\pi$. The limit \eqref{eq:llnmarkov} holds for a much wider class of functions $f$: the $\pi$-integrable ones. To curtail the length of this section, we only do the proof for the case that the chain starts on a positive recurrent state (see Exercise~\ref{ex:nf8eaw9ng7ae8whga} for more on the general case).
\begin{theorem}\label{thrm:llndt}Let $\{\cal{C}_i\}_{i\in\cal{I}}$ denote the set of positive recurrent classes and $\{\pi_i\}_{i\in\cal{I}}$ the corresponding set of stationary distributions. If the initial distribution $\gamma$ has support contained in the set $\cup_{i\in\cal{I}}\cal{C}_i$ of positive recurrent states, then
\begin{equation}\label{eq:llnmarkovgen}\lim_{N\to\infty}\epsilon_N(f)=\lim_{N\to\infty}\frac{1}{N}\sum_{n=0}^{N-1}f(X_n)=\sum_{i\in\cal{I}}1_{\{\phi_{\cal{C}_i}<\infty\}}\pi_i(f)=\epsilon_\infty(f)\quad\Pb_\gamma\text{-almost surely},\end{equation}
for all non-negative functions $f$ and all  $f$ such that $\epsilon_\infty(|f|)<\infty$ with $\Pb_\gamma$-probability one.
\end{theorem}

\begin{exercise}[The general case]\label{ex:nf8eaw9ng7ae8whga} For those readers looking for a challenge: using the above and tweaking the harmonic function approach of \citep[Prop.~17.1.6]{Meyn2009}, show that \eqref{eq:llnmarkovgen} holds for all initial conditions and functions $f$ such that $\epsilon_\infty(|f|)<\infty$ with $\Pb_\gamma$-probability one.\end{exercise}


%
%
%

\begin{proof}[Proof of Theorem~\ref{thrm:llndt}]The bulk of this proof is showing that \eqref{eq:llnmarkovgen} holds with $\gamma:=1_x$ and $x$ being any positive recurrent state. The general case then follows immediately:
\begin{align*}&\Pbl{\left\{\lim_{N\to\infty}\epsilon_N(f)=\epsilon_\infty(f)\right\}}
=\sum_{x\in\cup_{i\in\cal{I}}\cal{C}_i}\gamma(x)\Pbx{\left\{\lim_{N\to\infty}\epsilon_N(f)=\epsilon_\infty(f)\right\}}=\sum_{x\in\cup_{i\in\cal{I}}\cal{C}_i}\gamma(x)=1.\end{align*}

The $\Pb_\gamma$-almost surely $\epsilon_\infty$-integrable $f$ case then follows from the non-negative case and the usual measure-theoretic trick: re-write $\epsilon_N(f)=\epsilon_N(f\vee0)+\epsilon_N(f\wedge0)$, take the limit $N\to\infty$, and exploit that $f\vee 0$ and $(-f)\vee 0$ are non-negative and $\epsilon_\infty$-integrable ($\Pb_\gamma$-almost surely):
$$\epsilon_\infty(f\vee 0)+\epsilon_\infty((-f)\vee 0)=\epsilon_\infty(f\vee 0+(-f)\vee 0)=\epsilon_\infty(f\vee 0-f\wedge 0)=\epsilon_\infty(\mmag{f}).$$
\end{proof}

\begin{exercise}Ratio theorem\end{exercise}

\subsubsection*{A proof of Theorem~\ref{thrm:pointwiselimd}} We begin by dealing with the transient and null recurrent states:
\fi


\subsection{Periodicity and lack thereof}\label{sec:periodicity} So far, we have managed to characterise the long-term behaviour of the \emph{time averages}  taken over individual paths and described by the empirical distribution of Section~\ref{sec:empdist}. We now set our sights on the long-term behaviour of the \emph{space} (or \emph{ensemble} or \emph{population}) averages taken across the collection of the chain's paths at given points in time and described by the time-varying law of Section~\ref{sec:dtlawsintro}. To proceed, we need to introduce a notion we have managed to avoid up until now: \emph{periodicity}.\index{periodicity}\index{aperiodicity}

\begin{definition}[Periodic and aperiodic states]\label{def:period} The greatest common divisor (gcd) of a sequence $a_1,a_2,a_3,\dots$ of integers is defined as the limit of the sequence $b_1:=gcd(a_1),$ $b_2:=gcd(a_1,a_2)$, $b_3:=gcd(a_1,a_2,a_3)$, $\dots$.

The period of a state $x$ is the greatest common divisor of the sequence of time-steps after which the chain may return to $x$:
\begin{equation}\label{eq:mfe89anfae9nfeua}d(x):=\text{greatest common divisor of }\{n>0:p_n(x,x)>0\}\end{equation}
with the convention that $d(x):=1$ if the above set is empty, where $P_n=(p_n(x,y))_{x,y\in\s}$ denotes the $n$-step matrix~\eqref{eq:nstepdef}. A state is periodic if its period is greater than one. Otherwise, it is said to be aperiodic.\end{definition}

Periodic states $x$ are those that the chain can only return after a regular number of steps: $d(x)$ steps, or $2d(x)$ steps, etc. 
Aperiodic states on the other hand, are characterised as follows:
\begin{proposition}[Characterising aperiodic states]\label{prop:aperiodic} A state $x$ is aperiodic if and only if $p_n(x,x)=0$ for all $n>0$ or there exists an $n(x)$ such that $p_n(x,x)>0$ for all $n$ greater than $n(x)$.
\end{proposition}
To prove the proposition, we require the following number-theoretic fact:
\begin{theorem}[{\citealp[A.21]{Billingsley1995}}]\label{thrm:latticethe} Suppose that $S$ is a set of positive integers closed under addition and of period one. Then $S$ contains all integers past some number $n$.
\end{theorem}
\begin{proof}[Proof of Proposition~\ref{prop:aperiodic}]The reverse direction follows immediately from the definition of the state's period. If the set in the right-hand side of \eqref{eq:mfe89anfae9nfeua} is empty, then the forward direction also follows immediately from this definition. Otherwise, note that if $k,l$ are such that $p_k(x,x)>0$ and $p_l(x,x)>0$, then the Chapman-Kolmogorov equation in~\eqref{eq:chap-kol} implies that
$$p_{k+l}(x,x)\geq p_k(x,x)p_l(x,x)>0.$$
In other words, $\{n>0:p_n(x,x)>0\}$ is closed under addition and the result follows from  Theorem~\ref{thrm:latticethe} above.
\end{proof}
Perhaps unsurprisingly, the period remains constant across communicating classes:
\begin{theorem}[The period is a class property]\label{thrm:sameperiod} If $x$ and $y$ belong to the same communicating class, then they have the same period. 
\end{theorem}

\begin{proof}Because $x$ and $y$ communicate, there exists positive integers $k$ and $l$ such that $p_k(x,y)>0$ and $p_l(y,x)>0$ (this follows from the Chapman-Kolmogorov equation in~\eqref{eq:chap-kol} and Lemma~\ref{lem:accdt}). Thus, for any $n$ such that $p_n(y,y)>0$, the Chapman-Kolmogorov equation also implies that
$$p_{k+n+l}(x,x)\geq p_k(x,y)p_n(y,y)p_l(y,x)>0.$$
Similarly, we have that $p_{k+l}(x,x)\geq p_k(x,y)p_l(y,x)>0$. Thus, we have that $d(x)$ divides both $k+n+l$ and $k+l$, and so $d(x)$ must divide $n$. We have shown that $d(x)$ divides every integer in $\{n>0:p_n(y,y)>0\}$ and so $d(x)\leq d(y)$. Reversing $x$ and $y$ and repeating the same argument we find that $d(y)\leq d(x)$ and we obtain that both states have the same period ($d(x)=d(y)$).
%
\end{proof}
The theorem motivates us to introduce the following definition:
\begin{definition}[Periodic and aperiodic classes]\label{def:periodclass} The period of a communicating class is the period of any one (or, equivalently, all) of its states. The class is said to be aperiodic if its period is one and periodic otherwise.

\end{definition}

\subsubsection*{Periodic and aperiodic chains}In Section~\ref{sec:empdist}, we saw that transient and null recurrent states are immaterial to the chain's long-term behaviour because visits to any such state either cease at some point or become so rare that the fraction of time the chain spends in the state approaches zero as time progresses. For this reason, we only asks for the positive recurrent states to be aperiodic in order for a chain to be deemed aperiodic:

\begin{definition}[Periodic and aperiodic chains] A chain is said to be aperiodic if all of its positive recurrent states are aperiodic. Otherwise, it is said to be periodic.
\end{definition}

Aside from esoteric periodic chains possessing infinitely many positive recurrent classes with distinct periods, it follows from Theorem \ref{thrm:sameperiod} that the lowest common multiple of all of the positive recurrent states' periods is finite:
\begin{equation}\label{eq:periodlcm}d:=\text{lowest common multiple of }\{d(x):x\not\in\cal{T}\}<\infty,\end{equation}
where $\cal{T}$ denotes the set of all states that are not positive recurrent. It is then a simple matter to verify that the chain obtained by sampling $X$ every $d$ steps starting at step $m<d$,
\begin{equation}\label{eq:xmdef}X^{m,d}_n:=X_{m+nd}\quad\forall n\geq0,\end{equation}
is an aperiodic chain.
\begin{exercise}\label{ex:xm}Given any $d$ and $m$ in $\n$, marginalise over \eqref{eq:nstepthe} to obtain
$$\Pbl{\left\{X^{m,d}_0=x_0,X^{m,d}_1=x_1,\dots,X^{m,d}_n=x_n\right\}}=\gamma P_m(x_0)p_d(x_0,x_1)\dots p_d(x_{n-1},x_n)$$
for all $x_0,\dots,x_n$ in $\s$, $n$ in $\n$. Apply Theorems~\ref{pathlawuni} and \ref{samedef} to show that $(X^{m,d}_n)_{n\in\n}$ is a Markov chain with one-step matrix $P_d$ and initial distribution $\gamma P_m$. Convince yourself that if $d$ is a multiple of the period $d(x)$ (for the original chain $X$) of a state $x$, then $x$ is an aperiodic state for $(X^{m,d}_n)_{n\in\n}$. Conclude that if $d$ is as in \eqref{eq:periodlcm}, then $(X^{m,d}_n)_{n\in\n}$ is aperiodic.
\end{exercise}

For this reason, the long-term behaviour of a periodic chain can be pieced together from the behaviour of each of its periodic components $X^{m,d}:=(X^{m,d}_n)_{n\in\n}$  and we (mostly) focus on the aperiodic case for the remainder of our treatment of discrete-time chains. 

Before proceeding, you should take a moment to convince yourself that sampling the chain every $d(x)$ steps does not alter the recurrence properties of the state $x$:
\begin{exercise}\label{ex:xm2}For any given state $x$, let $(X^{m,d}_n)_{n\in\n}$ be as in~\eqref{eq:xmdef} with $d$ being the state's period $d(x)$. Let $\phi^{m,d}_x$ denote the first entrance time to $x$ of $(X^{m,d}_n)_{n\in\n}$:
\begin{equation}\label{eq:xmentdef}\phi^{m,d}_x:=\inf\{n>0:X^{m,d}_n=x\}=\inf\{n>0:X_{m+nd}=x\}.\end{equation}
Using the fact that
$$\Pbx{\{X_n=x\}}=p_n(x,x)=0\quad\forall n\neq 0,d,2d,\dots,$$
show that 
$$\phi_x=d\phi^{0,d}_x,\quad\Pb_x\text{-almost surely}.$$
Given that, for all $m$ in $\n$, $X^{m,d}$ is a discrete-time chain with the same one-step matrix $P_d$~\eqref{ex:xm}, use the above to conclude that $x$ is transient (null/positive recurrent) for $X$ if and only if it is transient (resptively, null/positive recurrent) for $X^{m,d}$ and $m$ in $\n$.
\end{exercise}
%
%
%
%
%
\subsubsection*{The limiting behaviour of the time-averages is indifferent to periodicity}You may be wondering why the notion of periodicity doesn't affect the limiting behaviour of the time averages while it does affect that of the ensemble averages. The answer is best understood through an example: consider a chain that alternates between two states (say $x$ and $y$) at each step. The state space ($\{x,y\}$) consists of a single periodic class with period two. If the initial distribution is concentrated on one of these two states (say $x$), then the time-varying law will keep alternating between them ($p_0=1_x$, $p_1=1_y$, $p_2=1_x$, $p_3=1_y$, $\dots$). Consequently, $p_n$ does not converge. On the other hand, the empirical distribution ($\epsilon_N$ in \eqref{eq:timeavedef}) averages out these oscillations and settles down:
\begin{align*}
\epsilon_1 &= 1_x,\quad \epsilon_2=\frac{1}{2}1_x+\frac{1}{2}1_y,\quad \epsilon_3=\frac{2}{3}1_x+\frac{1}{3}1_y,\quad \epsilon_4=\frac{1}{2}1_x+\frac{1}{2}1_y,\quad \epsilon_5=\frac{3}{5}1_x+\frac{2}{5}1_y,\\
\epsilon_6&=\frac{1}{2}1_x+\frac{1}{2}1_y,\quad\epsilon_7=\frac{4}{7}1_x+\frac{3}{7}1_y,\quad \dots\quad\to\quad \frac{1}{2}1_x+\frac{1}{2}1_y.\end{align*}
%
\subsection[Time-varying law's limits; tightness; ergodicity; coupling]{Limits of the time-varying law; tightness; ergodicity; coupling}\label{sec:ensembleaved}
%
The aim of this section is to describe the long-term behaviour of the chain's time-varying law $(p_n)_{n\in\n}$ (introduced in Section~\ref{sec:dtlawsintro}). We do this via two theorems, the proofs of which can be found at  the end of the section. The first one shows that the probability that the chain is in any  state not positive recurrent tends to zero as time progresses:
\begin{theorem}\label{thrm:pointlims0}For any state $x$ that is not positive recurrent,
$$\lim_{n\to\infty}p_n(x)=0.$$
\end{theorem}
The intuition here is that if $x$ is transient or null recurrent, visits to $x$ eventually become so rare that not only does the fraction of time any one path spends in $x$ approaches zero (Section~\ref{sec:empdist}: $\epsilon_N(x)\to0$ as $N\to\infty$, $\Pb_\gamma$-almost surely) but so does the fraction $p_n(x)$ of the entire ensemble of paths. Be careful here: if $x$ is null recurrent and, say, the chain starts at $x$, then every path will revisit $x$ infinitely many times~(Theorem~\ref{thrm:rectrans}), however the frequency of these visits drops quickly enough that the fraction of paths in $x$ at any given moment decays to zero.

The second theorem builds on Theorem~\ref{thrm:pointlims0} to  show that, if the chain is aperiodic, then the time-varying law converges to a weighted combination of the ergodic distributions~(c.f. Section~\ref{sec:statdists}), where the weight awarded to each one is the probability that the chain gets absorbed in the closed communicating class associated with it:
\begin{theorem}[The limits of the time-varying law]\label{thrm:pointlims} Let $\{\cal{C}_i:i\in\cal{I}\}$ denote the collection of positive recurrent closed communicating classes and let $\pi_i$ be the ergodic distribution of $\cal{C}_i$ for each $i$ in $\cal{I}$. 
\begin{enumerate}[label=(\roman*),noitemsep] 
\item If the chain is aperiodic, then time-varying law $(p_n)_{n\in\n}$~\eqref{eq:dtlaw} converges pointwise with limit\glsadd{pigamma}
\begin{equation}\label{eq:reclims2}\lim_{n\to\infty}p_n=\sum_{i\in\cal{I}}\Pbl{\{\phi_{\cal{C}_i}<\infty\}}\pi_i=:\pi_\gamma,\end{equation}
where  $\phi_{\C_i}$ denotes the first entrance time to $\C_i$~(Definition~\ref{def:entrance}).
\item Conversely, if the chain is periodic there exists at least one initial distribution such that the time-varying law does not converge pointwise.
\end{enumerate}
\end{theorem}
Theorem~\ref{thrm:pointlims}$(ii)$ shows that the time-varying law does not converge if the chain is periodic. However,  by sampling every $d$ steps a periodic chain with period $d$, we obtain an aperiodic chain~(Exercise~\ref{ex:xm}). Applying Theorem~\ref{thrm:pointlims}$(i)$ to the sampled chain, we obtain the following generalisation:
\begin{exercise}[The limits of the time-varying law: the periodic case]\label{ex:pointlimsp}Suppose that the period $d$~\eqref{eq:periodlcm} of the chain is finite. Let $X^{0,d}$ be the discrete-time chain with one-step matrix $P_d$ and initial distribution $\gamma$ obtained by sampling $X$ every $d$ steps~\eqref{eq:xmdef}. Show that a state $y$ is accessible from another state $x$ for $X^{0,d}$ only if it is accessible for $X$. Conclude that each positive recurrent class of $X^{0,d}$ is contained in one of $X$ and label the positive recurrent classes $\{\cal{C}^d_{i,j}:i\in\cal{I},j\in\cal{J}_i\}$ of $X^{0,d}$ such that
$$\cal{C}^d_{i,j}\subseteq\cal{C}_i\quad\forall j\in\cal{J}_i,\enskip i\in\cal{I},$$
where $\{\cal{C}_{i}:i\in\cal{I}\}$ denotes the set of positive recurrent classes of $X$. Given that $X^{0,d}$ and $X$ share the same positive recurrent states (Exercise~\ref{ex:xm2}), argue that
\begin{equation}\label{eq:nfua89efnea798nf7a9hfnawu}\bigcup_{j\in\cal{J}_i}\cal{C}^d_{i,j}=\cal{C}_i\quad\forall i\in\cal{I}.\end{equation}
Next, tweaking the arguments in Exercise~\ref{ex:xm2}, show that
$$d\Ebx{\phi^{0,d}_x}=\Ebx{\phi^{d/d(x)}_x}\quad\forall x\not\in\cal{T},$$
where $d(x)$ denotes the period~\eqref{def:period} of a state $x$, $\phi^{k}_x$ denotes to the $k$th entrance time~\eqref{def:entrance} to $x$ of $X$, $\phi^{0,d}_x$ the first entrance time~\eqref{eq:xmentdef} to $x$  of $X^{0,d}$, and $\cal{T}$ the set of  states that are not positive recurrent (for both $X$ and $X^{0,d}$). Apply the regenerative property~\eqref{thrm:rec-iid} of $X$ to the above and obtain that
$$\Ebx{\phi^{0,d}_x}=\frac{\Ebx{\phi_x}}{d(x)}\quad\forall x\not\in\cal{T}.$$
Using the characterisation of ergodic distributions in terms of mean return times (Theorem~\ref{doeblind}$(i)$) and the above, conclude that
\begin{equation}\label{eq:nfua89efnea798nf7a9hfnawu2}\pi^d_{i,j}(x)=d_i1_{\cal{C}_{i,j}^d}(x)\pi_i(x)\quad\forall x\in\s,\enskip j\in\cal{J}_i,\enskip i\in\cal{I},\end{equation}
where $d_i$ denotes the period of the class $\cal{C}_i$~(Definition~\ref{def:periodclass}), $\pi^d_{i,j}$ the ergodic distribution of $X^{0,d}$  associated with $\cal{C}^d_{i,j}$, and $\pi_i$ the ergodic distribution of $X$ associated with $\cal{C}_i$. Summing~\eqref{eq:nfua89efnea798nf7a9hfnawu2} over $x$ in $\cal{C}^d_{i,j}$, show that
\begin{equation}\label{eq:nfua89efnea798nf7a9hfnawu3}\pi_i(\cal{C}^d_{i,j})=\frac{1}{d_i}\quad\forall j\in\cal{J}_i,\enskip i\in\cal{I}.\end{equation}

Using Exercise~\ref{ex:xm}, argue that, for any given $m$, the process $X^{m,d}$ obtained by sampling $X$ every $d$ steps starting from $m$ is an aperiodic discrete-time chain possessing the  same positive recurrent closed communicating classes and ergodic distributions as  $X^{0,d}$ does. Using this fact, Theorem~\eqref{thrm:pointlims}$(i)$, and  (\ref{eq:nfua89efnea798nf7a9hfnawu}--\ref{eq:nfua89efnea798nf7a9hfnawu2}), derive the following generalisation of \eqref{eq:reclims2}:
\begin{equation}\label{eq:reclims2p}\lim_{n\to\infty}p_{m+nd}=\sum_{i\in\cal{I}}\left(\sum_{j\in\cal{J}_i}\Pbl{\left\{\phi^{m,d}_{\cal{C}^d_{i,j}}<\infty\right\}}1_{\cal{C}^d_{i,j}}\right)d_i\pi_i=:\pi_\gamma^m,\end{equation}
for all $m<d$, where $\phi^{m,d}_A$ denotes the first entrance time~\eqref{eq:xmentdef} of $X^{m,d}$ to a set $A$ and the convergence is pointwise.

Use the closedness of $\cal{R}_+=\cup_{i\in\cal{I}}\cal{C}_i$ for $X$ to argue that, with probability one, $X^{m,d}$ eventually enters $\cal{R}_+$ if and only if $X$ does (here, use Proposition~\ref{prop:closed} and the strong Markov property). Consequently, apply Proposition~\ref{prop:closedis} and \eqref{eq:nfua89efnea798nf7a9hfnawu2}--\eqref{eq:nfua89efnea798nf7a9hfnawu3} to argue that
\begin{equation}\label{eq:pigs}\pi^m_\gamma(\s)=\Pbl{\left\{\phi^{m,d}_{\cal{R}_+}<\infty\right\}}=\Pbl{\left\{\phi_{\cal{R}_+}<\infty\right\}}.\end{equation}
\end{exercise}
\subsubsection*{Convergence in total variation  and tightness}The question `under what circumstances does \eqref{eq:reclims2} hold in total variation?' has a straightforward answer involving the notion of \emph{tightness}\footnote{Here, we are topologising the state space using the discrete metric so that the compact sets are the finite sets.}:\index{tightness of the time-varying law}
\begin{definition}[Tightness]\label{def:tight}A sequence $(\rho_n)_{n\in\n}$ of probability distributions on $\s$ is tight if and only if for every $\varepsilon>0$ there exists a finite set $F$ such that $\rho_n(F)\geq 1-\varepsilon$ for all $n$ in $\n$.
\end{definition}
We then have the following corollary of Theorem~\ref{thrm:pointlims}:
\begin{corollary}\label{cor:tvtight}Suppose that the period $d$ in~\eqref{eq:periodlcm} of the chain is finite and let $\cal{R}_+$ be the set of positive recurrent states. The following are equivalent:
\begin{enumerate}[label=(\roman*),noitemsep] 
\item The chain enters $\cal{R}_+$ with probability one: $\Pbl{\{\phi_{\cal{R}_+}<\infty\}}=1$.
\item The time-varying law $(p_n)_{n\in\n}$ is tight.
\item (Aperiodic case) The limit~\eqref{eq:reclims2} holds in total variation: with $\norm{\cdot}$ as in~\eqref{eq:tvnorm},
\begin{equation}\label{eq:pntv}\lim_{n\to\infty}\norm{p_n-\pi_\gamma}=0.\end{equation}
\end{enumerate}
\end{corollary}
\begin{proof}To not turn this proof into a notational horror show, we only consider the aperiodic case. To argue the equivalence of $(i)$ and $(ii)$ for the periodic case, use~\eqref{eq:reclims2p} and \eqref{eq:pigs} analogously to how we use \eqref{eq:reclims2} and \eqref{eq:pigmassdt} below.

$(i)\Leftrightarrow (iii)$ Given that $\cal{R}_+$ equals the union $\cup_{i\in\cal{I}}\cal{C}_i$ of the positive recurrent classes and that, because the classes are closed sets, the probability of both $\{\phi_{\cal{C}_i}<\infty\}$ and $\{\phi_{\cal{C}_j}<\infty\}$ occurring is  zero if $i\neq j$   (Proposition~\ref{prop:closedis}), the mass of the limit in~\eqref{eq:reclims2} equals the probability that the chain enters the set of positive recurrent states:
\begin{align}\label{eq:pigmassdt}\pi_\gamma(\s)&=\sum_{i\in\cal{I}}\Pbl{\{\phi_{\cal{C}_i}<\infty\}}=\Pbl{\bigcup_{i\in\cal{I}}\{\phi_{\cal{C}_i}<\infty\}}=\Pbl{\{\phi_{\cup_{i\in\cal{I}}\cal{C}_i}<\infty\}}\\
&=\Pbl{\{\phi_{\cal{R}_+}<\infty\}}\nonumber.\end{align}
Scheff\'e's lemma (Lemma~\ref{lem:scheffe}) then implies that $(i)$ holds if and only if $(iii)$ does.

$(i)\Leftrightarrow (ii)$ Equations~\eqref{eq:reclims2} and \eqref{eq:pigmassdt} imply that
$$\lim_{n\to\infty}p_{n}(F)=\pi_\gamma(F)\leq\pi_\gamma(\s)= \Pbl{\{\phi_{\cal{R}_+}<\infty\}}$$
for any finite set $F$ and it follows that the time-varying law is tight only if $\Pbl{\{\phi_{\cal{R}_+}<\infty\}}=1$.  

For the converse, fix any $\varepsilon>0$ and suppose that $\Pbl{\{\phi_{\cal{R}_+}<\infty\}}=1$. Given that the limit~\eqref{eq:reclims2} holds in total variation (as we have just shown), we can find  an $N$ such that 
$$\norm{p_{n}-\pi_\gamma}\leq \frac{\varepsilon}{2}\quad\forall n> N.$$
Using~\eqref{eq:pigmassdt}, pick a finite set $F$ large enough that
$$p_n(F)\geq 1-\varepsilon\quad\forall n\leq N,\qquad\pi_\gamma(F)\geq 1-\frac{\varepsilon}{2}.$$
The above inequalities then imply that
\begin{align*}p_n(F)&\geq 1_{\{n\leq N\}}p_n(F)+1_{\{n>N\}}p_n(F)\geq 1_{\{n\leq N\}}(1-\varepsilon)+1_{\{n>N\}}(\pi_\gamma(F)-\norm{p_{n}-\pi_\gamma})\geq 1-\varepsilon,\end{align*}
for all natural numbers $n$. Because the $\varepsilon$ was arbitrary, we have that $(p_n)_{n\in\n}$ is tight.
\end{proof}
\subsubsection*{Positive Tweedie recurrence revisited} In Section~\ref{sec:timeaved}, we characterised positive Tweedie recurrent chains (Definition~\ref{def:posTweedie}) in terms of the limits of the time-varying law $\epsilon_N$ in~\eqref{eq:timeavedef}. Armed with Corollary~\ref{cor:tvtight}, it is straightforward to improve this characterisation:\index{positive Tweedie recurrent chain}
\begin{corollary}[Characterising positive Tweedie recurrent chains]\label{cor:tweedierec}Let $d$ in~\eqref{eq:periodlcm} denote the period of the chain. The following are equivalent:
\begin{enumerate}[label=(\roman*),noitemsep] 
\item The chain is positive Tweedie recurrent.
\item The empirical distribution converges in total variation to $\epsilon_\infty$ in~\eqref{eq:reclims} with $\Pb_\gamma$-probability one, for all initial distributions $\gamma$.
\item (Finite $d$ case) The time-varying law $(p_n)_{n\in\n}$ is tight, for all initial distributions $\gamma$.
\item (Aperiodic case) The time-varying law converges in total variation (i.e.\ \eqref{eq:pntv}), for all initial distributions $\gamma$.
\end{enumerate}
\end{corollary}
\begin{proof}This follows directly from Corollaries~\ref{cor:dttimeavetv} and \ref{cor:tvtight} and the definition  of positive Tweedie recurrence in~\eqref{def:posTweedie}.
\end{proof}
Corollary~\ref{cor:tweedierec} shows that a chain with finite period is positive Tweedie recurrent if and only if for each initial distribution $\gamma$ and $\varepsilon$ in $(0,1]$, we can find a large enough finite set $F$ such that the chain has at least $1-\varepsilon$ probability of being in $F$ at any given time-step (i.e., $p_n(F)\geq 1-\varepsilon$ for all $n\geq0$): further reinforcing the view that a chain is stable if and only if it is positive Tweedie recurrent.
%
%
%
%
%
%
\subsubsection*{Ergodicity and positive Harris recurrence}In the case of an aperiodic positive Harris recurrent  chain~(Definition~\ref{def:posTweedie}), Theorem~\ref{doeblind} and Corollary~\ref{cor:tweedierec} show that the chain has a unique stationary distribution $\pi$ and that, regardless of the initial distribution $\gamma$, both the time-varying law \eqref{eq:timevardt} and  the empirical distribution \eqref{eq:timeavedef} converge  to it:
$$\lim_{N\to\infty}\epsilon_N=\lim_{n\to\infty}p_n=\pi\quad\Pb_{\gamma}\text{-almost surely},$$
where the convergence is in total variation. That is, the \emph{time averages} $\epsilon_N$ converge to the \emph{space averages} $p_n$ and the chain is said to be \emph{ergodic}\index{ergodic}.
\subsubsection*{The coupling inequality}
To prove Theorems \ref{thrm:pointlims0}~and~\ref{thrm:pointlims}, we use a technique called \emph{coupling}.\index{coupling} It involves two sequences of random variables $X=(X_n)_{n\in\n}$ and $X'=(X'_n)_{n\in\n}$, defined on the same probability space $(\Omega,\cal{F},\Pb)$, for which there exists a random time $\sigma_c:\Omega\to\n_E$ such that $X$ and $X'$ are identical from $\sigma_c$ onwards:
$$X_n(\omega)=X'_n(\omega)\quad \text{if}\quad n\geq \sigma_c(\omega).$$
The time $\sigma_c$ is said to be a \emph{coupling time} of $X$ and $X'$. The key result here is the celebrated \emph{coupling inequality}:
\begin{theorem}[The coupling inequality]\label{thrm:coupling} If $\sigma_c$ denotes a coupling time of $X$ and $X'$, $p_n$ denotes the time-varying law of $X$, and $p_n'$ that of $X'$, then
$$\norm{p_n-p_n'}\leq \Pbb{\{\sigma_c>n\}}\quad\forall n\in\n.$$
\end{theorem} 
\begin{proof}For any $A\subseteq\s$ and $n\in\n$,
\begin{align*}\mmag{\Pbb{\{X_n\in A\}}-\Pbb{\left\{X_n'\in A\right\}}}&\leq\Ebb{\mmag{1_{\{X_n\in A\}}-1_{\left\{X_n'\in A\right\}}}}\leq\Ebb{1_{\left\{X_n\neq X_n'\right\}}}\leq\Ebb{1_{\{\sigma_c>n\}}}\\
&=\Pbb{\{\sigma_c>n\}}.\end{align*}
Taking the supremum over $A\subseteq\s$ in the above completes the proof.
\end{proof}


\subsubsection*{Coupling for the proofs of Theorems~\ref{thrm:pointlims0}~and~\ref{thrm:pointlims}}
In the proofs of Theorems~\ref{thrm:pointlims0}~and~\ref{thrm:pointlims} that follow, we use two independently generated copies of the chain , $X^1$ and $X^2$,  with initial distributions $\gamma_1$ and $\gamma_2$ (respectively). We   construct these copies by running Algorithm \ref{dtmcalg} with the same one-step matrix $P$ but using two independent sets of random variables $X_0^1,U_1^1,U_2^1,\dots$ and $X_0^2,U_1^2,U_2^2,\dots$ such that $X_0^1\sim \gamma_1$ and $X_0^2\sim \gamma_2$ and we denote the underlying probability space by $(\Omega,\cal{F},\Pb)$. Additionally, we set $\sigma_c$ be the first time these two chains coincide,
$$\sigma_c:=\inf\{n\geq0:X^1_n=X^2_n\},$$
and set $X$ to be $X^1$ and $X'$ to be
$$X'_n(\omega)=\left\{\begin{array}{ll} X^2_n(\omega)&\text{if }n<\sigma_c(\omega)\\X^1_n(\omega)&\text{if }n\geq \sigma_c(\omega)
\end{array}\right.\quad\forall\omega\in\Omega,\enskip n\in\n.$$
The independence of $X^1$ and $X^2$ imply that $X'$ is also Markovian:
\begin{exercise}Following steps analogous to those taken in Exercise~\ref{ex:xm}, show that $X'$  is a Markov chain with initial distribution $\gamma_2$ and one-step matrix $P$.
\end{exercise}
Because the definition of $X'$ implies that $\sigma_c$ is a coupling time for $X$ and $X'$, the coupling inequality \eqref{thrm:coupling}  shows that 
\begin{equation}\label{eq:coupling2}\norm{\gamma_1P_n-\gamma_2P_n}\leq \Pbb{\{\sigma_c>n\}}\quad\forall n\in\n.\end{equation}
To use the above to prove Theorems~\ref{thrm:pointlims0}--\ref{thrm:pointlims}, we need to argue that the coupling time is almost surely finite. To this end,   we use the \emph{product chain} $\tilde{X}:=(X^1,X^2)$:
\begin{exercise}Following steps analogous to those taken in Exercise~\ref{ex:xm}, show that $\tilde{X}:=(X^1,X^2)$  is a Markov chain taking values in $\s^2$ with initial distribution 
$$\tilde{\gamma}((x_1,x_2)):=\gamma_1(x_1)\gamma_2(x_2)\quad\forall x_1,x_2\in\s.$$
and one-step matrix 
\begin{equation}\label{eq:prodchain1step}\tilde{p}((x_1,x_2),(y_1,y_2)):= p(x_1,y_1)p(x_2,y_2)\quad\forall x_1,x_2,y_1,y_2\in\s.\end{equation}
\end{exercise}
The coupling time $\sigma_c$ of $X$ and $X'$ is the point in time that the product chain first enters the diagonal $\{(x,x):x\in\s\}$ of its state space $\s^2$. For this reason, $\sigma_c$ is bounded above by the product chain's entrance time $\phi_{(x,x)}$ to  any given state $(x,x)$ on the diagonal:
\begin{equation}\label{eq:coupltimebound}\sigma_c\leq\phi_{(x,x)}:=\inf\{n>0:\tilde{X}_n=(x,x)\}\quad\forall x \in\s.\end{equation}
In summary, our ability to prove that $\gamma_1 P_n$ and $\gamma_2 P_n$ converge to each other hinges on whether we are able to show that the time it takes the product chain to enter any given state $(x,x)$ on the diagonal is almost surely finite if its starting location is sampled from  $\gamma_1\gamma_2$. This is where aperiodicity comes into play:
%
\begin{lemma}\label{lem:prodchainclass}Let $\cal{C}$ be an aperiodic closed communicating class $\mathcal{C}$. Then, $\mathcal{C}^2$ is a closed communicating class of the product chain $\tilde{X}:=(X^1,X^2)$. If, additionally, $\mathcal{C}^2$ is recurrent (for the product chain) and the supports of $\gamma_1$ and $\gamma_2$ are contained in $\mathcal{C}$, then $\sigma_c$ is almost surely finite.\end{lemma}

\begin{proof} Iterating \eqref{eq:prodchain1step}, we find that
\begin{equation}\label{eq:nybyteabfeu8}\tilde{p}_n((x_1,x_2),(y_1,y_2)):= p_n(x_1,y_1)p_n(x_2,y_2)\quad\forall x_1,x_2,y_1,y_2\in\s,\end{equation}
where $(\tilde{p}_n(x,y))_{x,y\in\s^2}$ denotes the $n$-step matrix of the product chain (defined analogously to $(p_n(x,y))_{x,y\in\s}$ in \eqref{eq:nstepdef}). Because $p_n(x,y)=0$ for all $n\geq0$ if and only if $y$ is not accessible from $x$ (use \eqref{eq:nstepdef} and \eqref{lem:accdt} to argue this), the above implies that $\mathcal{C}^2$ is closed (for $\tilde{X}$).

Now, suppose that $x_1,x_2,y_1,y_2$ belong to $\mathcal{C}$. Because $\cal{C}$ is aperiodic, there exists an $n(x_1)$ such that $p_n(x_1,x_1)>0$ for all $n$ greater than $n(x_1)$ (Proposition~\ref{prop:aperiodic}). Because $y_2$ is accessible from $x_2$ (for $X$), we can find an $n$ larger than $n(x_1)$ such that $p_n(x_2,y_2)>0$ (once again, use \eqref{eq:nstepdef} and \eqref{lem:accdt}). Thus, \eqref{eq:nybyteabfeu8} shows that $(x_1,y_2)$ is accessible from $(x_1,x_2)$. Repeating the same argument shows $(y_1,y_2)$ is accessible from $(x_1,y_2)$ and it follows that $(y_1,y_2)$ is accessible from $(x_1,x_2)$. Because $x_1,x_2,y_1,y_2$ were arbitrary, we have that $\mathcal{C}^2$ is a communicating class as desired.

Finiteness of $\sigma_c$ for the recurrent case then follows by setting $x$ in \eqref{eq:coupltimebound} to be a state in $\cal{C}$, applying \eqref{eq:pld}, and recalling that $\Pb_{x}(\{\phi_{y}<\infty\})=1$ if $x$ and $y$ are recurrent states belonging to the same class of any chain (Theorem \ref{thrm:recclasspro}).
\end{proof}

\subsubsection*{A proof of Theorem~\ref{thrm:pointlims0}}

We do this proof in four steps, beginning by swiftly dispatching with the case of a transient state:
\begin{proof}[Step 1: $x$ is transient] In this case, Corollary~\ref{cor:recnstepchar} shows that
\begin{align*}\sum_{n=1}^\infty p_n(x)<\infty\end{align*}
and it follows that $p_n(x)$ tends to zero as $n$ approaches to infinity.
\end{proof}
The null recurrent states take a bit more work. For the time being we focus on the case that the chain starts at the state in question and that the state is aperiodic:
\begin{proof}[Step 2: $x$ is null recurrent and aperiodic and the chain starts at $x$] Let $\cal{C}$ be the aperiodic closed communicating class that $x$ belongs to (Theorems~\ref{thrm:recclasspro} and \ref{thrm:sameperiod}) and recall that $\mathcal{C}^2$ is a closed communicating class of the product chain $\tilde{X}$ (Lemma~\ref{lem:prodchainclass}). Setting $x_1=x_2=y_1=y_2=x$ in \eqref{eq:nybyteabfeu8} we have that 
$$p_n(x,x)^2=\tilde{p}_n((x,x),(x,x))\quad\forall n\in\n.$$
For this reason, applying Step 1 to $\tilde{X}$, we find that $p_n(x,x)\to0$ as $n\to \infty$ if $\cal{C}^2$ is transient.

If $\mathcal{C}^2$ instead is recurrent, then  the coupling inequality \eqref{eq:coupling2} and Lemma~\ref{lem:prodchainclass} show that
\begin{equation}
\label{eq:781h4nfesaf}\lim_{n\to\infty}\norm{\gamma_1P_n-\gamma_2P_n}=0,
\end{equation}
where $\gamma_1:=1_x$ and $\gamma_2:=\gamma_1P$ (to apply the lemma, note that $\gamma_2$ has support in $\mathcal{C}$ because $\mathcal{C}$ is closed for $X$). 

For $p_n(x,x)$ not to converge to zero, there must exist a strictly increasing sequence $(n_k)_{k\in\n}$ and a constant $c>0$ such that
$$\lim_{k\to\infty}p_{n_k}(x,x)=c.$$
Using a routine diagonal argument (see the end of the step), we find a subsequence $(n_l)_{l\in\n}$ of $(n_k)_{k\in\n}$ such that, for all $y$ in $\s$, $p_{n_l}(x,y)$ converges  as $l$ tends to infinity; we denote the limit $\rho(y)$. The non-negativeness of $p_{n_l}(x,y)$ ensures that $\rho(y)$ is non-negative for each $y$ in $\s$, while Fatou's lemma shows that 
\begin{equation}
\label{eq:nmifnuwenfa}\sum_{y\in\s}\rho(y)\leq \lim_{l\to\infty}\sum_{y\in\s}p_{n_l}(x,y)=1.
\end{equation} 

However, \eqref{eq:781h4nfesaf} shows that $p_{n_l+1}(x,y)$ also converges  to $\rho(y)$ as $n$ tends to infinity, for all $y$ in $\s$. Calling on Fatou's lemma once again, we have that
$$\rho(y)=\lim_{l\to\infty}p_{n_l+1}(y)=\lim_{l\to\infty}\left(\sum_{x'\in\s}p_{n_l}(x,x')p(x',y)\right)\geq \sum_{x'\in\s}\rho(x')p(x',y)\quad\forall y\in\s.$$
Because $\rho(x)=c$, iteratively applying the above inequality and Tonelli's theorem, we find that
\begin{align*}\sum_{y\in\s}\frac{\rho(y)}{c}&=1+\sum_{z_1\neq x}\frac{\rho({z_1})}{c}\geq 1+\sum_{z_1\neq x}\sum_{z_2\in\s}\frac{\rho({z_2})}{c}p(z_2,z_1)\\
&=1+\sum_{z_1\neq x}p(x,z_1)+\sum_{z_1\neq x}\sum_{z_2\neq x}\frac{\rho({z_2})}{c}p(z_2,z_1)\\
&\geq\dots\geq 1+\sum_{n=2}^m \bigg(\sum_{z_{n-1}\neq x}\sum_{z_{n-2}\neq x}\dots\sum_{z_1\neq x}p(x,z_{n-1})p(z_{n-1},z_{n-2})\dots p(z_{2},z_1)\\
&+\sum_{z_n\neq x}\dots\sum_{z_1\neq x}\frac{\rho(z_n)}{c}p(z_n,z_{n-1})\dots p(z_2,z_1)\bigg)\\
&\geq 1+\sum_{n=2}^m \sum_{z_{n-1}\neq x}\dots\sum_{z_1\neq x}p(x,z_{n-1})p(z_{n-1},z_{n-2})\dots p(z_{2},z_1)\\
&=\Pbx{\{\phi_x>0\}} +\sum_{n=1}^{m-1}\Pbx{\{\phi_x>n\}}\quad\forall m>0.\end{align*}
Taking the limit $m\to\infty$, we find that
$$\sum_{y\in\s}\rho(y)\geq c\sum_{n=0}^\infty\Pbx{\{\phi_x>n\}}=c\sum_{n=0}^\infty n\Pbx{\{\phi_x=n\}}=c\Ebx{\phi_x}=\infty.$$
The above contradicts \eqref{eq:nmifnuwenfa} and so it must be the case that   $p_n(x,x)$ converges to zero as $n$ tends to infinity.

\noindent\emph{The diagonal argument:} Label the elements of $\s$ as $x,x_0,x_1,\dots$, keeping $x$ as before. Because the sequence $(p_{n_k}(x_0))_{k\in\n}$ is contained in the interval $[0,1]$, the Bolzano-Weierstrass Theorem tells us that $(p_{n_k}(x_0))_{k\in\n}$ has a converging subsequence $(p_{n_{j_0}}(x_0))_{j_0\in\n}$. Repeating the same argument for $x_1$ and $(p_{n_{j_0}})_{j_0\in\n}$, we find a convergent subsequence $(p_{n_{j_1}}(x_1))_{j_1\in\n}$ of $(p_{n_{j_0}}(x_1))_{j_0\in\n}$, and so on. Letting
$$p_{n_{l}}:=p_{n_{j_j}}\quad\forall j\in\n,$$
we obtain our pointwise convergent sequence.

\end{proof}
The remainder of the proof is downhill. Next, we get rid of the aperiodicity assumption:
\begin{proof}[Step 3: $x$ is null recurrent and the chain starts at $x$]Let $d$ be $x$'s period and recall $x$ is aperiodic and null recurrent for the chain $X^{m,d}$~\eqref{eq:xmdef} obtained by sampling $X$ every $d$ steps starting from step $m\in\{0,1,\dots,d-1\}$ (Exercises~\ref{ex:xm}--\ref{ex:xm2}). For this reason, Step 2 shows that 
$$\lim_{n\to\infty}p_{m+nd}(x,x)=\Pbx{\left\{X^{m,d}_n=x\right\}}=0\quad\forall m\in\{0,1,\dots,d-1\}.$$
In other words, for each $\varepsilon>0$ and $m$ in $\{0,1,\dots,d-1\}$, there exists an $N_m$ such that $p_{m+nd}(x)\leq\varepsilon$ for all $n\geq N_m$. For this reason, 
$$p_{n}(x)\leq\varepsilon\quad\forall n\geq\max\{N_0d,1+N_1d,\dots, d-1+N_{d-1}\}.$$
As $\varepsilon$ was arbitrary, the result follows.
\end{proof}
We are nearly done: we just need to apply the strong Markov property and port Step 3 from the start-at-$x$ case to the arbitrary-initial-distribution-$\gamma$ case:
\begin{proof}[Step 4: $x$ is null recurrent]For any positive $n$, $X_n$ equals $x$ only if $\phi_{x}$ is no greater than $n$. For this reason, applying Lemma~\ref{lem:2stop}, Exercise~\ref{ex:entstop}, and the strong Markov property (Theorem~\ref{thrm:strmkvpath} with $\varsigma:=k$, $Z:=1_{\{\phi_x=k\}}$, and $F((z_m)_{m\in\n}):=1_x(z_{n-k})$), yields
\begin{align}p_n(x)&=\Pbl{\{\phi_{x}\leq n,X_n=x\}}=\sum_{k=1}^\infty1_{\{k\leq n\}}\Pbx{\{\phi_{x}=k,X_{\phi_{x}}=x,X_n=x\}}\nonumber\\
&=\sum_{k=1}^\infty1_{\{k\leq n\}}\Pbl{\{\phi_{x}=k,X_{\phi_{x}}=x\}}\Pb_{x}(\{X_{n-k}=x\})\nonumber\\
&=\sum_{k=1}^\infty1_{\{k\leq n\}}\Pbl{\{\phi_{x}=k\}}p_{n-k}(x,x)\quad\forall n>0.\label{eq:bfgyuwabf67aw}
\end{align}
Given that $p_n(x,x)\to0$ as $n\to\infty$ (Step 3) and that
\begin{align*}&1_{\{k\leq n\}}\Pbl{\{\phi_{x}=k\}}p_{n-k}(x,x)\leq \Pbl{\{\phi_{x}=k\}}\enskip\forall k>0,\\
&\sum_{k=1}^\infty\Pbl{\{\phi_{x}=k\}}=\Pbl{\{\phi_{x}<\infty\}}<\infty,\end{align*}
taking the limit $n\to\infty$ in \eqref{eq:bfgyuwabf67aw} and applying dominated convergence, we obtain the desired $p_n(x)\to0$ as $n\to\infty$.
\end{proof}

\subsubsection*{A proof of Theorem~\ref{thrm:pointlims}}

We do this proof one part at a time:

$(i)$ Given Theorem~\ref{thrm:pointlims0}, we only need to argue here that the limit~\eqref{eq:reclims2} limit holds for states $x$ belonging to a positive recurrent class $\cal{C}_i$. Suppose that we are able to show that 
\begin{equation}\label{eq:fnmeu8aw09fn7m8ahfeqwujfaw}\lim_{n\to\infty}p_n(x,x)=\pi_i(x).\end{equation}
Because the event $\{\phi_x<\infty\}$ coincides with $\{\phi_{\cal{C}_i}<\infty\}$ with $\Pb_\gamma$-probability one (Corollary~\ref{cor:phixphic}), taking the limit $n\to\infty$ in \eqref{eq:bfgyuwabf67aw} and applying dominated convergence yields
$$\lim_{n\to\infty}p_n(x)=\left(\sum_{k=1}^\infty\Pbl{\{\phi_x=k\}}\right)\pi_i(x)=\Pbl{\{\phi_x<\infty\}}\pi_i(x)=\Pbl{\{\phi_{\cal{C}_i}<\infty\}}\pi_i(x).$$
Given that no other ergodic distribution has support on $x$~(Theorem~\ref{doeblind}$(i)$), we may rewrite the above as \eqref{eq:reclims2}.

To argue~\eqref{eq:fnmeu8aw09fn7m8ahfeqwujfaw}, set  $\gamma_1:=1_x$ and $\gamma_2:=\pi_i$ and recall that $\cal{C}^2_i$ is a closed communicating class of the product chain $\tilde{X}$ (Lemma~\ref{lem:prodchainclass}). Because $\pi_i$ is a stationary distribution, $\gamma_2P_n=\pi_i$ for all natural numbers $n$ (Theorem \ref{pstateqs}). Given that $\pi_i$ has support contained in $\mathcal{C}_i$ (Theorem \ref{doeblind}$(i)$), the coupling inequality \eqref{eq:coupling2} and Lemma~\ref{lem:prodchainclass} imply that \eqref{eq:fnmeu8aw09fn7m8ahfeqwujfaw} holds if  $\mathcal{C}^2$ is recurrent. This follows easily as $(x_1,x_2)\mapsto\pi(x_1)\pi(x_2)$ is a stationary distribution of the product chain (use \eqref{pstateqs} and \eqref{eq:prodchain1step} to argue this) and Theorem \ref{thrm:posrecchar} shows that a state is positive recurrent if and only if there exists a stationary distribution with support on the state.

$(ii)$ Let $x$ be any state in a periodic positive recurrent class $\cal{C}_i$ with period $d$ so that
\begin{equation}\label{eq:fne8af8eafeafaef}p_n(x,x)=0\quad\forall n\neq 0,d,2d,\dots,\end{equation}
and initialise the chain at $x$ (i.e., $\gamma:=1_x$). The state $x$ is aperiodic and positive recurrent for the  chain $X^{0,d}$~\eqref{eq:xmdef} obtained by sampling $X$ every $d$ steps (Exercises~\ref{ex:xm}--\ref{ex:xm2}). Thus,   $(i)$ and Theorem~\ref{doeblind}$(i)$ imply that $\Pbx{\left\{X^{0,d}_n=x\right\}}$ converges to a positive constant as $n$ grows unbounded. Given that $X^{0,d}_n=X_{nd}$, it follows from~\eqref{eq:fne8af8eafeafaef} that $p_n(x,x)$ does not converge as $n$ tends to infinity.

\subsection[Kendall's theorem; Geometric recurrence and convergence*]{Kendall's theorem on geometric recurrence and convergence*}\label{sec:kendall}

A matter that has been extensively studied is under what conditions the time-varying law converges geometrically fast with the numbers of steps. Key in answering this question are the \emph{geometrically recurrent}\index{geometric convergence and recurrence} states, that is, the states  $x$ whose return time distribution has light tails: $\Ebx{\theta^{\phi_x}}<\infty$  for some real number $\theta>1$. These states can alternatively be characterised in terms of the convergence of the $n$-step matrix:\index{Kendall's theorem}
\begin{theorem}[Kendall's Theorem, {\citealp{Kendall1959}}]\label{thrm:kendall} Suppose that the chain is aperiodic.
Let $x$ denote any state belonging to a  positive recurrent class $\cal{C}$ and let $\pi$ be the ergodic distribution associated with $\cal{C}$.
\begin{enumerate}[label=(\roman*),noitemsep] 
\item $p_n(x,x)$ converges geometrically fast to $\pi(x)$: $\mmag{p_n(x,x)-\pi(x)}=\mathcal{O}(\kappa^{-n})$ for some $\kappa>1$.
\item $x$ is geometrically recurrent.
\end{enumerate}
\end{theorem}

The remainder of this section is dedicated to the proof of Kendall's theorem and to the open problem of its generalisation. We begin with the former. 
\subsubsection*{The renewal equation} The proof of Theorem~\ref{thrm:kendall} builds on the celebrated \emph{renewal equation}\index{renewal equation} linking the time-varying distribution with the return and entrance times:
\begin{align}p_n(x)&=\sum_{k=1}^\infty\Pb_\gamma(\{\phi^k_x=n\})=\sum_{k=1}^\infty\sum_{m=0}^{n-1}\Pb_\gamma(\{\phi^{k-1}_x=m,\phi^k_x=n\})\nonumber\\
&=\Pbl{\{\phi_x=n\}}+(1-1_{1}(n))\sum_{k=2}^\infty\sum_{m=1}^{n-1}\Pb_\gamma(\{\phi^{k-1}_x=m,\phi^k_x-\phi^{k-1}_x=n-m\})\nonumber\\
&=\Pbl{\{\phi_x=n\}}+(1-1_{1}(n))\sum_{k=2}^\infty\sum_{m=1}^{n-1}\Pb_\gamma(\{\phi^{k-1}_x=m\})\Pb_x(\{\phi_x=n-m\})\nonumber\\
&=\Pbl{\{\phi_x=n\}}+(1-1_{1}(n))\sum_{m=1}^{n-1}p_m(x)\Pb_x(\{\phi_x=n-m\})\quad\forall n>0,\enskip x\in\s,\label{eq:renewaleqnz}
\end{align}
where the penultimate equality follows from the strong Markov property:
\begin{exercise}Setting $\varsigma:=\phi^{k-1}$, $Z:=1_{\{\phi^{k-1}=n\}}$, and $F((z_l)_{l\in\n}):=1_m(\inf\{l>0:\sum_{l'=1}^{l}1_x(z_{l'})=1\})$ and applying Lemma~\ref{lem:2stop},  Theorem~\ref{thrm:strmkvpath}, Lemma~\ref{lem:pathspmeas}$(ii)$, and Exercise~\ref{ex:entstop}, show that
$$\Pb_\gamma(\{\phi^{k}_x=n,\phi^{k+1}_x-\phi^{k}_x=m\})=\Pb_\gamma(\{\phi^{k}_x=n\})\Pb_x(\{\phi_x=m\})\quad\forall n,m,k\geq0.$$
\end{exercise}
Multiplying \eqref{eq:renewaleqnz} by $z^n$ and summing over all natural numbers $n$ yields the $z$-transform version of the renewal equation:
\begin{equation}\label{eq:renewaleq}G(z)=H_\gamma(z)+H_x(z)G(z)\quad\forall z\in D_1,\end{equation}
where $D_1:=\{z\in\mathbb{C}:\mmag{z}<1\}$ denotes the open unit disk in the complex plane $\mathbb{C}$,
$$G(z):=\sum_{n=1}^\infty p_n(x)z^n,\quad H_\gamma(z):=\sum_{n=1}^\infty \Pb_\gamma(\{\phi_x=n\})z^n,\quad H_x(z)=\sum_{n=1}^\infty \Pb_x(\{\phi_x=n\})z^n,\enskip\forall z\in D_1,$$
and we have used the equations
\begin{align*}\sum_{n=2}^\infty\sum_{m=1}^{n-1}p_m(x)\Pb_x(\{\phi_x=n-m\})z^n&=\sum_{m=1}^\infty p_m(x)z^m\sum_{n=m+1}^\infty\Pb_x(\{\phi_x=n-m\})z^{n-m}\\
&=\sum_{m=1}^\infty p_m(x)z^m\sum_{n=1}^\infty\Pb_x(\{\phi_x=n\})z^{n}\quad\forall z\in D_1,\end{align*}

\subsubsection*{A proof of Kendall's theorem}

We do this proof in three steps:
\begin{lemma}[Step 1]\label{lem:kendall1}Suppose that the chain is aperiodic and let $x$ be any positive recurrent state and $\pi_\gamma$ be the limit of the time-varying law in \eqref{eq:reclims2}. For any $\kappa>1$, we have that $\mmag{p_n(x)-\pi_\gamma(x)}=\cal{O}(\kappa^{-n})$ if and only if $G$ in \eqref{eq:renewaleq} has an analytic extension on the disc $D_{\kappa}:=\{z\in\mathbb{C}:\mmag{z}<\kappa\}$
except for a simple pole at $z=1$.
\end{lemma}

\begin{proof}It is not difficult to check that $\mmag{p_n(x)-\pi_\gamma(x)}=\cal{O}(\kappa^{-n})$ if and only if
$$\sum_{n=0}^\infty\mmag{p_n(x)-\pi_\gamma(x)}r^{n}<\infty\quad\forall 1<r<\kappa.$$
In this case,  the triangle inequality yields
\begin{align*}&\sum_{n=2}^\infty\mmag{p_n(x)-p_{n-1}(x)}\mmag{z}^n\leq\sum_{n=0}^\infty \mmag{p_n(x)-\pi(x)}\mmag{z}^n+\mmag{z}\sum_{n=0}^\infty\mmag{\pi(x)-p_{n}(x)}\mmag{z}^{n}< \infty\quad\forall z\in D_{\kappa},\end{align*}
and it follows that the function 
\begin{equation}\label{eq:kendallF}F(z):=\sum_{n=2}^\infty(p_n(x)-p_{n-1}(x))z^n+p_1(x)z\quad\forall z\in D_{\kappa}\end{equation}
is analytic. Because, by definition, $G$ is analytic on $D_1$ and 
\begin{equation}\label{eq:dn782gtby2}F(z)=(1-z)G(z),\end{equation} 
for all $z$ in $D_1$, we can extend  $z\mapsto(1-z)G(z)$ analytically in $D_{\kappa}$. Because the above equation then holds for all $z$ in $D_\kappa$, it follows that $G$ has no singularities in $D_{\kappa}$ except for a simple pole at $z=1$.

Conversely, if,  for some $\kappa>1$, $G$ is analytic on $D_{\kappa}$ except for a simple pole at $1$, then \eqref{eq:dn782gtby2} allows us to extend $F$ (defined by replacing $D_{\kappa}$ in \eqref{eq:kendallF} with $D_1$) analytically in $D_{\kappa}$. It follows that the radius of convergence of the series in \eqref{eq:kendallF} is at least $\kappa$ and, by the  the equation holds if we replace ``$:=$'' with ``$=$''. Because the chain is aperiodic, $p_n(x)$ tends to $\pi_\gamma(x)$ as $n$ approaches infinity (Theorem \ref{thrm:pointlims}$(i)$) which implies that 
\begin{align*}\sum_{n=0}^\infty\mmag{p_n(x)-\pi_\gamma(x)}r^{n}&=\sum_{n=0}^\infty\mmag{\sum_{m=n+1}^\infty(p_m(x)-p_{m-1}(x))}r^{n}\leq\sum_{m=1}^\infty\mmag{p_m(x)-p_{m-1}(x)}\left(\sum_{n=0}^{m-1}r^{n}\right)\\
&=\sum_{m=1}^\infty\mmag{p_m(x)-p_{m-1}(x)}\frac{r^{m}-1}{r-1}\leq \frac{1}{r-1}\sum_{m=1}^\infty\mmag{p_m(x)-p_{m-1}(x)}r^m\end{align*}
for all $1<r<\kappa$. Cauchy's inequality implies that the right-hand side is finite for all $1<r<\kappa$, and it follows that $\mmag{p_n(x)-\pi_\gamma(x)}$ is $\cal{O}(\kappa^{-n})$.
\end{proof}

Setting $\gamma:=1_x$ in Lemma~\ref{lem:kendall1} shows that $\mmag{p_n(x,x)-\pi(x)}$ is $\cal{O}(\kappa^{-n})$ if and only if 
\begin{equation}\label{eq:g2}G(z)=\sum_{n=1}^\infty p_n(x,x)z^n\end{equation}
has an analytic extension on $D_\kappa$ except for a simple pole at $z=1$. To relate the existence of this extension to the tails of the return time distribution we proceed as follows:
\begin{lemma}[Step 2]\label{lem:kendall2}Suppose that the chain is aperiodic and that $x$ is a geometrically recurrent state whose entrance time distribution,  restricted to the event $\{\phi_x<\infty\}$, has light tails: 
$$\Ebx{\theta^{\phi_x}}<\infty,\quad\Ebl{1_{\{\phi_x<\infty\}}\theta^{\phi_x}}<\infty,$$
for some $\theta>1$. In this case, the function $G$ in \eqref{eq:renewaleq} is analytic on $D_\kappa$ for some $\kappa>1$ except for a simple pole at $z=1$.

\end{lemma}

\begin{proof}Because 
\begin{align*}\sum_{n=1}^\infty \Pb_\gamma(\{\phi_x=n\})\mmag{z}^n&\leq \sum_{n=1}^\infty \Pb_\gamma(\{\phi_x=n\})\theta^n=\Ebl{1_{\{\phi_x<\infty\}}\theta^{\phi_x}},\\
\sum_{n=1}^\infty \Pb_x(\{\phi_x=n\})\mmag{z}^n&\leq\sum_{n=1}^\infty \Pb_x(\{\phi_x=n\})\theta^n=\Ebx{\theta^{\phi_x}},\quad\forall z\in D_\theta\end{align*}
our premise implies that $H_\gamma$ and $H_x$ in~\eqref{eq:renewaleq} are analytic on $D_\theta$.

Because $\Pbx{\{\phi_x=n\}}\geq p_n(x,x)$ and the chain is aperiodic, Proposition~\ref{prop:aperiodic} shows that $\Pbx{\{\phi_x=n\}}>0$ for all sufficiently large $n$. For this reason,
$$Re(H_x(z))=\sum_{n=1}^\infty\Pbx{\{\phi_x=n\}}Re(z^n) <\sum_{n=1}^\infty\Pbx{\{\phi_x=n\}}=1$$
for any $z$ in the closed unit disc $\bar{D}_1:=\{z\in\mathbb{C}:\mmag{z}\leq 1\}$ except $1$. Consequently, $1$ is the only zero of $H_x(z)-1$ in $\bar{D}_1$ and, thus, there exists a $1<\kappa\leq \theta$ such that the only zero of $H_x(z)-1$ in $D_\kappa$ is $z=1$ (here, use the Bolzano-Weierstrass theorem and the principle of permanence).  This zero is simple because 
$$\lim_{z\to 1}\frac{H_x(z)-1}{z-1}=\frac{dH_x}{dz}(1)=\sum_{n=1}^\infty n\Pbx{\{\phi_x=n\}}=\Ebx{\phi_x}\geq 1.$$
In other words, $1-H_x(z)=(1-z)K(z)$ for some function $K$ that is analytic and non-zero on $D_\kappa$. The renewal equation~\eqref{eq:renewaleq} and \eqref{eq:dn782gtby2} then show that 
\begin{equation}\label{eq:hfe87afbhea8bfeyanefua}F(z)=\frac{G(z)(1-H_x(z))}{K(z)}=\frac{H_\gamma(z)}{K(z)}\quad z\in D_1.\end{equation}
Because we have chosen $\kappa$ to be no greater than $\theta$ and because $H_\gamma$ and $H_x$ are analytic on $D_\theta$, it follows that $K$ and $F$ have analytic extensions on $D_\kappa$. It then follows from  \eqref{eq:dn782gtby2}  that $G$ has an analytic extension on $D_\kappa$.
\end{proof}

Setting $\gamma:=1_x$ in Lemma~\ref{lem:kendall2}, we find that Theorem~\ref{thrm:kendall}$(ii)$ holds only if there exists some $\kappa>1$ such that  $G$ in~\eqref{eq:g2} is analytic on $D_\kappa$, except for a simple pole at $z=1$. Given Lemma~\ref{lem:kendall1}, we need only argue the converse to finish off the proof of Kendall's theorem:
\begin{lemma}[Step 3]\label{lem:kendall3}Let $x$ be a positive recurrent state. If, for some $\kappa>1$, $G$ in~\eqref{eq:g2}  has an analytic extension on $D_\kappa$, except for a simple pole at $z=1$, then Theorem~\ref{thrm:kendall}$(ii)$ is satisfied.\end{lemma}

\begin{proof}Fixing the initial distribution $\gamma$ to be $1_x$, the renewal equation~\eqref{eq:renewaleq} reduces to
$$G(z)=1+G(z)H_x(z)\quad\forall z\in D_1.$$
Multiplying both sides by $1-z$ and applying~\eqref{eq:dn782gtby2} yields 
$$F(z)=1-z+F(z)H_x(z)\quad\forall z\in D_1$$
and it follows that $F(z)\neq 0$ for all $z$ belonging to the closed unit disc $\bar{D}_1$ aside from, perhaps, $z=1$. However, 
$$F(1)=\lim_{n\to\infty}p_n(x,x)=\pi(x)$$
which is non-zero unless $\pi(x)=1$ (in which case the claim is trivial as $x$ must be an absorbing state). From the Bolzano-Weierstrass theorem and the principle of permanence it then follows that the zero $z_1$ of $F$ closest to the origin must have magnitude strictly greater than one (i.e., $\mmag{z_1}>1$). Because $F$ is analytic on $D_\kappa$ and $H_x(z)=(F(z)-1+z)/F(z)$ for all $z$ in $D_{\mmag{z_1}}$, $H_x$ in~\eqref{eq:renewaleq} has an analytical extension on $D_{\kappa\wedge \mmag{z_1}}$. In particular,
$$\Ebx{\theta^{\phi_x}}=\sum_{n=1}^\infty \Pb_x(\{\phi_x=n\})\theta^n=H_x(\theta)<\infty\quad\forall 1<\theta<\kappa\wedge\mmag{z_1},$$
as desired.
\end{proof}

\subsubsection*{An open question}

Lemmas~\ref{lem:kendall1}~and~\ref{lem:kendall2} show that for any given initial distribution $\gamma$ and positive recurrent state $x$, $p_n(x)$ converges geometrically fast to its limit $\pi_\gamma(x)$ in~\eqref{eq:reclims2},
\begin{equation}\label{eq:nde78afb8afnueanfiea}\mmag{p_n(x)-\pi_\gamma(x)}\text{ is }\cal{O}(\kappa^{-n})\text{ for some }\kappa>1,\end{equation}
if $x$ is geometrically recurrent and its  entrance time distribution (restricted to the event $\{\phi_x<\infty\}$) has light tails:
$$\Ebx{\theta^{\phi_x}}<\infty,\quad\Ebl{1_{\{\phi_x<\infty\}}\theta^{\phi_x}}<\infty\text{ for some }\theta>1.$$
Trivially, if the chain never visits $x$ (i.e., $\Pbl{\{\phi_x<\infty\}}=0$), then $p_1(x)=p_2(x)=\dots=\pi_\gamma(x)=0$ and the convergence of $p_n(x)$ is geometric. Putting these together and using our convention that $0\cdot\infty=0$, we have the following:
\begin{theorem}\label{thrm:kendallopen}If the chain is aperiodic and $x$ is a positive recurrent state satisfying
\begin{equation}\label{eq:denau8fne7fa8fbe8abfa}\Ebl{1_{\{\phi_x<\infty\}}\theta^{\phi_x^2}}=\Ebl{1_{\{\phi_x<\infty\}}\theta^{\phi_x}}\Ebx{\theta^{\phi_x}}<\infty\end{equation}
for some $\theta>1$, then $p_n(x)$ converges geometrically fast to its limit $\pi_\gamma(x)$ (i.e.\ \eqref{eq:nde78afb8afnueanfiea} holds for some $\kappa>1$).
\end{theorem}
\begin{proof}
Given Lemmas~\ref{lem:kendall1}~and~\ref{lem:kendall2}, the theorem follows directly from the strong Markov property (in particular, \eqref{eq:dtenttimestrn2} in the following section with $A:=\{x\}$, $B:=\emptyset$ and $k:=l:=1$).
\end{proof}

To the best of my knowledge, the question of whether or not \eqref{eq:denau8fne7fa8fbe8abfa} is not only sufficient but necessary for \eqref{eq:nde78afb8afnueanfiea} to hold remains open. It seems to me that it should be: I have been unable to cook up a situation in which geometric convergence of $p_n(x)$ to a positive recurrent state $x$ for an aperiodic chain would not imply geometric convergence of $p_n(x,x)$. Unless, of course, $\Pbl{\{\phi_x<\infty\}}=0$, in which case the claim is trivial. Were $p_n(x,x)$ to converge geometrically fast, then Kendall's Theorem (Theorem~\ref{thrm:kendall}) would show that $x$ is geometrically recurrent and it would from \eqref{eq:hfe87afbhea8bfeyanefua} that $H_\gamma$ has an analytic extension on a disk $D_\kappa$ with $\kappa>1$. In particular, we would have \eqref{eq:denau8fne7fa8fbe8abfa} as it would be the case that
$$\Ebl{1_{\{\phi_x<\infty\}}\theta^{\phi_x}}=H_\gamma(\theta)<\infty\quad\forall \theta<\kappa.$$
If you know the answer to this question, I'd love to hear about it.

\subsection{Geometric trials arguments*}\label{sec:geometrictrial}

As we have seen throughout Sections~\ref{sec:empdist}--\ref{sec:kendall}, we are often interested in the time elapsed until the chain enters some set $B$ (e.g.\ a particular state, a certain class, the union of the positive recurrent classes, etc.). 
The Foster-Lyapunov criteria in the following sections yield information on the chain's visits to some set $F$. However, more often than, not $F$ will not be the set we are actually interested in: $F\neq B$.
%
Fortunately, if $B$ is accessible from $F$ (Definition~\ref{def:accesible}) and $F$ is finite, then the chance that a visit to $F$ results in a visit to $B$ may be bounded below by a constant independent of which state in $F$ the chain visits. The existence of this constant implies that the probability that the chain has not yet visited $B$ decays  geometrically with the  number of visits to $F$:\index{geometric trial arguments}
\begin{lemma}[Geometric trials property]\label{lem:AFtail1} If $F$ is finite and $F\to B$ for a second set $B$, then there exists constants $m$ and $\varepsilon>0$ independent of the initial distribution $\gamma$ such that
$$\Pbl{\{\phi^{nm}_F<\phi_B\}}\leq (1-\varepsilon)^{n-1}\quad\forall n>0,$$
where $ \phi_B$ denotes the first entrance time to $B$ and $\phi^k_B$ the $k$th entrance time to $F$~(Definition~\ref{def:entrance}).
%
\end{lemma}
Lemma~\ref{lem:AFtail1} lays the foundation for the following three geometric-trials-type arguments that  allow us to turn information on the chain's visits to $F$ into information on its visits to $B$:\index{geometric trial arguments}
\begin{lemma}[Geometric trials arguments]\label{lem:geotrialarg} Given a finite subset $F$ of the state space, suppose that $F\to B$ for some other subset $B$ and let $\phi_F$ and $\phi_B$ denote the first entrance times to $F$ and $B$~(Definition~\ref{def:entrance}).
\begin{enumerate}[label=(\roman*),noitemsep] 
\item If $\Pbl{\{\phi_F<\infty\}}=1$ and $\Pbx{\{\phi_F<\infty\}}=1$ for all $x$ in $F$, then $\Pb_\gamma(\{\phi_B<\infty\})=1$.
\item If $\Ebl{\phi_F}<\infty$ and  $\Ebx{\phi_F}<\infty$ for all $x$ in $F$, then $\Ebl{\phi_B}<\infty$.
\item If there exists a constant $\theta>1$ such that $\Eb_\gamma[\theta^{\phi_F}]<\infty$ and  $\Eb_x[\theta^{\phi_F}]<\infty$ for all $x$ in $F$, then there exists second constant $\vartheta>1$ such that $\Eb_\gamma[\vartheta^{\phi_B}]<\infty$.
\end{enumerate}
\end{lemma}

\subsubsection*{Proofs of Lemmas~\ref{lem:AFtail1} and \ref{lem:geotrialarg}}

The key to these proofs are applications of the strong Markov property that express expectations involving the $(k+l)$th entrance time to a set in terms expectations involving the $k$th entry time and others involving the $l$th entrance time:
\begin{lemma}\label{lem:abkl}For any subsets $A$ and $B$ of the state space, constant $\vartheta>1$, and natural numbers $k$ and $l$,
\begin{align}
&\Pb_\gamma\left(\left\{\phi^{k+l}_A<\phi_B\right\}\right)=\sum_{x\in A}\Pbl{\left\{\phi_A^k<\phi_B,X_{\phi_A^k}=x\right\}}\Pbx{\left\{\phi^{l}_A<\phi_B\right\}},\label{eq:dtenttimestr}\\
&\Ebl{1_{\{\phi_A^{k}<\phi_B\}}(\phi_A^{k+l}-\phi_A^{k})}=\sum_{x\in A}\Pbl{\left\{\phi_A^{k}<\phi_B,X_{\phi_A^k}=x\right\}}\Ebx{\phi_A^{l}},\label{eq:dtenttimestrn1}\\
&\Ebl{1_{\{\phi_A^{k}<\phi_B\}}\vartheta^{\phi_A^{k+l}}}=\sum_{x\in A}\Ebl{1_{\{\phi_A^{k}<\phi_B,X_{\phi_A^k}=x\}}\vartheta^{\phi_A^{k}}}\Ebx{\vartheta^{\phi_A^{l}}},\label{eq:dtenttimestrn2}
\end{align}
where $\phi_B$ denotes the first entrance time to $B$ and $\phi^k_A$ the $k$th entrance time to $A$ (Definition~\ref{def:entrance}).
\end{lemma}
%


%

\begin{proof}To argue~\eqref{eq:dtenttimestr}--\eqref{eq:dtenttimestrn2}, we will show that
\begin{equation}\Ebl{Z1_{\left\{\phi_A^k<\phi_B\right\}}g(\phi^{k+l}_A-\phi^k_A,\phi_B-\phi^k_A)}=\sum_{x\in A}\Ebl{Z1_{\{\phi_A^k<\phi_B,X_{\phi_A^k}=x\}}}\Ebx{g(\phi^{l}_A,\phi_B)}\label{eq:dtenttimestrproof}\end{equation}
for any given non-negative function $g:\n_E^2\to\r_E$ and non-negative $\cal{F}_{\phi^k_A}$/$\cal{B}(\r_E)$-measurable random variable $Z$. Setting $g(n,m):=1_{\{0,\dots,m-1\}}(n)$ and $Z:=1$ in \eqref{eq:dtenttimestrproof}, we obtain \eqref{eq:dtenttimestr}. For \eqref{eq:dtenttimestrn1}, use $g(n,m):=n$ and $Z:=1$, while for \eqref{eq:dtenttimestrn2}, use $g(n,m):=\vartheta^n$ and $Z:=\vartheta^{\phi_A^k}$.

To prove \eqref{eq:dtenttimestrproof}, note that $X_n$ does not belong to $B$ for $n=1,2,\dots,\phi_A^k$ if $\phi_A^k<\phi_B$, and so
\begin{align}\phi_B&=\inf\left\{n>0:\sum_{m=1}^n1_B(X_m)=1\right\}=\inf\left\{n>\phi^k_A:\sum_{m=\phi^k_A+1}^n1_B(X_m)=1\right\}\nonumber\\
&=\phi^k_A+G_B^1(X^{\phi^k_A})\quad\text{on}\quad\{\phi^k_A<\phi_B\}
,\label{eq:dnm7893ng131dwana}\end{align}
where $X^{\phi^k_A}$ denotes the ${\phi^k_A}$-shifted chain~\eqref{eq:shiftdt} and, for any subset $S$ of $\s$ and $j>0$, $G_S^j$ denotes the $\cal{E}/2^{\n_E}$-measurable (Lemma \ref{lem:pathspmeas}$(ii)$) function defined by
$$G_S^j(x):=\inf\left\{n>0:\sum_{m=1}^n1_S(x_{m})=j\right\}\quad\forall x\in\cal{P}.$$
Similarly, because the number of entrances to $A$ by time $\phi_A^k$ is exactly $k$ ($\sum_{m=1}^{\phi_A^k}1_A(X_m)=k$) and $\phi^{k+l}_A$ is greater than $\phi^k_A$ by definition, we have that
\begin{align}\phi^{k+l}_A&=\inf\left\{n>0:\sum_{m=1}^n1_A(X_m)=k+l\right\}=\inf\left\{n>\phi^k_A:\sum_{m=\phi^k_A+1}^n1_A(X_m)=l\right\}\nonumber\\
&=\phi^k_A+G_A^l(X^{\phi^k_A}).
\label{eq:dnm7893ng131dwanb}\end{align}

The event $\{\phi^k_A<\phi_B\}$  belongs to the pre-$\phi^k_A$ sigma-algebra $\cal{F}_{\phi^k_A}$~(Definition~\ref{def:stopdt}) because we are able to deduce whether the chain has entered $B$ once by the time of the $k$th entry to $F$ if we observe the chain up until the latter moment (for a formal argument, use Lemma~\ref{lem:2stop} and Exercise~\ref{ex:entstop}). For these reasons, we can apply the strong Markov property (Theorem \ref{thrm:strmkvpath}) to obtain \eqref{eq:dtenttimestrproof}:
\begin{align*}&\Ebl{Z1_{\{\phi_A^k<\phi_B\}}g(\phi^{k+l}_A-\phi^k_A,\phi_B-\phi^k_A)}\\
&=\sum_{x\in A}\Ebl{Z1_{\{\phi_A^k<\phi_B,\phi^k_A<\infty,X_{\phi_A^k}=x\}}g(\phi^{k+l}_A-\phi^k_A,\phi_B-\phi^k_A)}\\
&=\sum_{x\in A}\Ebl{Z1_{\{\phi_A^k<\phi_B,\phi^k_A<\infty,X_{\phi_A^k}=x\}}g(G_A^l(X^{\phi^k_A}),G_B^1(X^{\phi^k_A}))}\\
&=\sum_{x\in A}\Ebl{Z1_{\{\phi_A^k<\phi_B,\phi^k_A<\infty,X_{\phi_A^k}=x\}}}\Ebx{g(G_A^l(X),G_B^1(X))}\\
&=\sum_{x\in A}\Ebl{Z1_{\{\phi_A^k<\phi_B,X_{\phi_A^k}=x\}}}\Ebx{g(\phi^{l}_A,\phi_B)},\end{align*}
where the first equality follows from the fact that $\{\phi_A^k<\phi_B\}$ is contained in $\{\phi_A^k<\infty\}$ (as $\phi^k_A$ is strictly less than $\phi_B$ only if the $\phi^k_A$ is finite), the second from \eqref{eq:dnm7893ng131dwana}--\eqref{eq:dnm7893ng131dwanb}, the third from the strong Markov property, and the fourth from the definitions of $\phi^l_A$ and $\phi_B$ and the fact that $\{\phi_A^k<\phi_B\}$ is contained in $\{\phi_A^k<\infty\}$. 
\end{proof}

With the above out of the way, we turn our attention to the proofs of Lemmas~\ref{lem:AFtail1}~and~\ref{lem:geotrialarg}:
\begin{proof}[Proof of Lemma~\ref{lem:AFtail1}]Pick any $x$ in $F$ and let $y$ in $B$ be such that $x\to y$. Because $\Pbx{\{\phi_y=\infty\}}<1$, there exists some $0<m_x<\infty$ and $\varepsilon_x>0$ such that 
$$\Pbx{\{n<\phi_B\}}\leq \Pbx{\{n<\phi_y\}}\leq 1-\varepsilon_x\quad\forall n\geq m_x.$$
Let $m:=\max\{m_x:x\in F\}$ and $\varepsilon:=\min\{\varepsilon_x:x\in F\}$. Finiteness of $F$ ensures that $m<\infty$ and $\varepsilon>0$, and it follows from the above that 
$$\Pbx{\{n<\phi_B\}}\leq 1-\varepsilon\quad\forall n\geq m,\enskip x\in F.$$
By definition, $\phi^n_F\geq n$ and the above implies that 
$$\Pbx{\{\phi_F^n<\phi_A\}}\leq 1-\varepsilon\quad\forall n\geq m,\enskip x\in F.$$
Setting $A:=F$, $k:=(n-1)m$, and $l:=m$ in \eqref{eq:dtenttimestr}, the result follows by induction:
\begin{align*}\Pbl{\left\{\phi^{nm}_F<\phi_B\right\}}&=\sum_{x\in F}\Pbl{\left\{\phi^{(n-1)m}_F<\phi_B,X_{\phi^{(n-1)m}_F}=x\right\}}\Pbx{\{\phi_F^m<\phi_B\}}\\
&\leq(1-\varepsilon)\sum_{x\in F}\Pbl{\left\{\phi^{(n-1)m}_F<\phi_B,X_{\phi^{(n-1)m}_F}=x\right\}}\\
&=(1-\varepsilon)\Pbl{\left\{\phi^{(n-1)m}_F<\phi_B\right\}}=\dots\leq(1-\varepsilon)^{n-1}\Pbl{\left\{\phi^{m}_F<\phi_B\right\}}\\
&\leq(1-\varepsilon)^{n-1}.\end{align*}

\end{proof}

\begin{proof}[Proof of Lemma~\ref{lem:geotrialarg}$(i)$]Setting $\varsigma:=\phi^{k-1}_F$ in Lemma~\ref{lem:dtenttimestr3} shows that  the chain keeps visiting $F$:
\begin{align*}\Pbl{\left\{\phi_F^{k}<\infty\right\}}&=\sum_{x\in F}\Pbl{\left\{\phi_F^{k-1}<\infty,X_{\phi_F^{k-1}}=x\right\}}\Pbx{\{\phi_F<\infty\}}\\
&
=\sum_{x\in F}\Pbl{\left\{\phi_F^{k-1}<\infty,X_{\phi_F^{k-1}}=x\right\}}=\Pbl{\left\{\phi_F^{k-1}<\infty\right\}}\\
&=\dots=\Pbl{\left\{\phi_F<\infty\right\}}=1\quad\forall k>0.\end{align*}
For this reason, and letting $m$ and $\varepsilon>0$ be as in Lemma~\ref{lem:AFtail1}, we have that
$$\Pbl{\{\phi_B=\infty\}}=\Pbl{\{\phi^{nm}_F<\phi_B,\phi_B=\infty\}}\leq \Pbl{\{\phi^{nm}_F<\phi_B\}}\leq (1-\varepsilon)^{n-1}\quad\forall n>0.$$
The result follows by taking the limit $n\to\infty$ in the above.
\end{proof}

\begin{proof}[Proof of Lemma~\ref{lem:geotrialarg}$(ii)$]Let $m$ and $\varepsilon>0$ be as in Lemma~\ref{lem:AFtail1}. Because, by definition, $\phi^0_F=0$ and $\phi_F^{nm}\geq nm\to\infty$ as $n\to\infty$, we may re-write $\phi_B$ as
$$\phi_B=\sum_{n=0}^\infty1_{\{\phi_B>\phi_F^{nm}\}}(\phi_B\wedge\phi_F^{(n+1)m}-\phi_F^{nm})\leq\sum_{n=0}^\infty1_{\{\phi_B>\phi_F^{nm}\}}(\phi_F^{(n+1)m}-\phi_F^{nm}).$$
For this reason,
\begin{align*}\Ebl{\phi_B}&\leq \Ebl{\phi_F^{m}}+\sum_{n=1}^\infty\Pbl{\{\phi_B>\phi_F^{nm}\}}\left(\max_{x\in F}\Ebx{\phi_F^{m}}\right)\\
&\leq\Ebl{\phi_F^{m}}+\left(\max_{x\in F}\Ebx{\phi_F^{m}}\right)\sum_{n=0}^\infty(1-\varepsilon)^n\\
&\leq \Ebl{\phi_F}+(m-1)\left(\max_{x\in F}\Ebx{\phi_F}\right)+m\left(\max_{x\in F}\Ebx{\phi_F}\right)\frac{1}{\varepsilon}<\infty,\end{align*}
where the first and third inequalities follow  from \eqref{eq:dtenttimestrn1} and Tonelli's theorem and the second from the Lemma~\ref{lem:AFtail1}.
\end{proof}

\begin{proof}[Proof of \ref{lem:geotrialarg}$(iii)$] Let $m$ and $\varepsilon$ be as in Lemma~\ref{lem:AFtail1} and define $M_\vartheta:=\max_{x\in F}\Eb_x[\vartheta^{\phi_F}]$ for all $\vartheta> 1$. Setting $1<\vartheta\leq\theta$ and $\delta_\vartheta:=\log(\vartheta)/\log(\theta)$, Jensen's inequality implies that
$$M_\vartheta=\max_{x\in F}\Ebx{(\theta^{\phi_F})^{\delta_\vartheta}}\leq M_\theta^{\delta_\vartheta}<\infty.$$
For this reason, $M_\vartheta$ tends to $1$ as $\vartheta$ approaches $1$ from above. 

As shown in the proof of $(i)$, 
%
the $k$th entrance time $\phi_F^k$ to $F$ is finite $\Pb_\gamma$-almost surely, for all $k>0$. For this reason, setting $B:=\emptyset$ (so that $\phi_B=\infty$) in \eqref{eq:dtenttimestrn2}, we find that
\begin{align}\Ebl{\vartheta^{\phi_F^k}}&=\sum_{x\in F}\Ebl{1_{\{\phi_F^{k-1}<\infty,X_{\phi_F^{k-1}}=x\}}\vartheta^{\phi_F^{k-1}}}\Ebx{\vartheta^{\phi_F}}\leq M_\vartheta\Ebl{\vartheta^{\phi_F^{k-1}}}\nonumber\\
&\leq\dots \leq M_\vartheta^{k-1}\Ebl{\vartheta^{\phi_F}}\quad\forall k>0.\label{eq:hd7a8hd837ahfw392}\end{align}

Because $\phi_F^n\geq n$ by definition,  $\phi_F^n\to\infty$ as $n\to\infty$ and it follows that
\begin{align}\vartheta^{\phi_B}=\sum_{n=0}^\infty1_{\{\phi_F^{nm}<\phi_B\}}\left(\vartheta^{\phi_B\wedge \phi_F^{(n+1)m}}-\vartheta^{\phi_F^{nm}}\right)\leq \sum_{n=0}^\infty1_{\{\phi_F^{nm}<\phi_B\}}\left(\vartheta^{\phi_F^{(n+1)m}}-\vartheta^{\phi_F^{nm}}\right).\label{eq:hd7a8hd837ahfw393}
\end{align}
However, putting \eqref{eq:dtenttimestrn2} and \eqref{eq:hd7a8hd837ahfw392} together, we have that
\begin{align}\Ebl{1_{\{\phi_F^{nm}<\phi_B\}}\left(\vartheta^{\phi_F^{(n+1)m}}-\vartheta^{\phi_F^{nm}}\right)}&=\sum_{x\in F}\Ebl{1_{\{\phi_F^{nm}<\phi_B,X_{\phi_F^{nm}}=x\}}\vartheta^{\phi_F^{nm}}}\Ebx{\vartheta^{\phi_F^m}-1}\label{eq:hd7a8hd837ahfw394}\\
&\leq\left(\max_{x\in F}\Ebx{\vartheta^{\phi_F^m}}-1\right)\Ebl{1_{\{\phi_F^{nm}<\phi_B\}}\vartheta^{\phi_F^{nm}}}\nonumber\\
&\leq\left(M_\vartheta^m-1\right)\Ebl{1_{\{\phi_F^{nm}<\phi_B\}}\vartheta^{\phi_F^{nm}}} \quad\forall n\geq0.\nonumber\end{align}
It then follows from \eqref{eq:hd7a8hd837ahfw393}--\eqref{eq:hd7a8hd837ahfw394} and Tonelli's theorem that
\begin{align*}\Ebl{\vartheta^{\phi_B}}&\leq \left(M_\vartheta^m-1\right)\sum_{n=0}^\infty\Ebl{1_{\{\phi_F^{nm}<\phi_B\}}\vartheta^{\phi_F^{nm}}}.\end{align*}
To bound the sum we use the inequality $ab\leq \alpha^na^2+\alpha^{-n}b^2$, for any non-negative $a,b,n,$ and $\alpha>0$. In particular, we have that 
\begin{align*}\sum_{n=0}^\infty\Ebl{1_{\{\phi_F^{nm}<\phi_B\}}\vartheta^{\phi_F^{nm}}}&\leq \sum_{n=0}^\infty\alpha^n\Ebl{\vartheta^{2\phi_F^{nm}}}+\sum_{n=0}^\infty\alpha^{-n}\Pbl{\{\phi_F^{nm}<\phi_B\}}\\
&\leq \Ebl{\vartheta^{2\phi_F}}\sum_{n=0}^\infty(\alpha M_{\vartheta^2}^{m})^n+\sum_{n=0}^\infty\alpha^{-n}\Pbl{\{\phi_F^{nm}<\phi_B\}},\end{align*}
where the second inequality follows from \eqref{eq:hd7a8hd837ahfw392}. Lemma~\ref{lem:AFtail1} then shows that
$$\sum_{n=0}^\infty\Ebl{1_{\{\phi_F^{nm}<\phi_B\}}\vartheta^{\phi_F^{nm}}}\leq \Ebl{\vartheta^{2\phi_F}}\sum_{n=0}^\infty(\alpha M_{\vartheta^2}^{m})^n+1+\frac{1}{\alpha-1+\varepsilon},$$
for any $1-\varepsilon<\alpha<1$. Choosing an $\vartheta\leq\sqrt{\theta}$ sufficiently close to $1$ so that $M_{\vartheta^2}^{m}<1/\alpha$ makes the right-hand side finite as desired.
\end{proof}

\subsection{Foster-Lyapunov criteria I: the criterion for recurrence*}\label{sec:flrecdt}In this section, we derive the criterion for recurrence: Theorem~\ref{thrm:lyarec} below. It involves a  \emph{norm-like}\index{norm-like} function meaning a $\r_E$-valued function $v$ on $\s$ that tends to infinity as its argument tends to infinity:

\begin{definition}[Norm-like function]\label{def:normlike} A function $v:\s\to\r_E$ is norm-like if its  sub-level sets are finite:
$$\s_r:=\{x\in\s:v(x)<r\}\quad \text{is finite for all }r\in\r.$$
\end{definition}
These types of functions are also sometimes called \emph{inf-compact}. A useful fact is that if $\rho$ is a probability distribution on $\s$ and $v$ is norm-like, then the expectation $\rho(v)$ is well-defined as, in this case,
$$\rho(v\wedge 0)=\sum_{x\in\s_0}v(x)\rho(x)\geq \left(\min_{x\in\s_0}v(x)\right)\rho(\s_0)\geq \min_{x\in\s_0}v(x)>-\infty.$$
The criterion goes as follows:\index{Foster-Lyapunov criteria}
\begin{theorem}[The criterion for recurrence]\label{thrm:lyarec} If there exists  a real-valued norm-like function $v$ on $\s$ and a finite set $F$ such that
\begin{equation}\label{eq:lyarec}Pv(x)\leq v(x)\quad\forall x\not\in F,\end{equation}
then the chain is Tweedie recurrent. Conversely, if the chain is Tweedie recurrent and has finitely many closed communicating classes, then there exists $(v,F)$ satisfying \eqref{eq:lyarec} with   $v$ real-valued and norm-like and  $F$ finite.
\end{theorem} 
If there are infinitely many closed communicating classes, it may be the case that no $(v,F)$ satisfying the theorem's premise exists even if the chain is Tweedie recurrent. For instance, let $K=(k(x,y))_{x,y\in\n}$ be the one-step matrix of an irreducible and recurrent $\n$-valued chain $W$ and let $X=(X^1,X^2)$ be a chain on $\n^2$ with one-step matrix
$$p((x_1,x_2),(y_1,y_2))=\left\{\begin{array}{ll}k(x_2,y_2)&\text{if }y_1=x_1\\0&\text{if }y_1\neq x_1\end{array}\right.\quad\forall x_1,x_2,y_1,y_2\in\n.$$
That is, $X$ remains in the slice $\n_i:=\{(i,x):x\in\n\}$ of $\n^2$ indexed by its first coordinate $i:=X^1_0$, moving up and down this slice in the same manner that $W$ moves up and down in $\n$. Because $W$ is irreducible and recurrent, it follows that $X$ is Tweedie recurrent with closed communicating classes $\n_1,\n_2,\dots$ Suppose that there exists a $(v,F)$ satisfying the premise of Theorem~\ref{thrm:lyarec}. If $v$ is not non-negative, then replace it by $v-(\min_{x\in\s}v(x))$ and note that the theorem's premise still holds. Because $F$ is finite, there exists infinitely many slices $\n_i$ that do not intersect with $F$. For any such $\n_i$, \eqref{eq:lyarec} implies that $v_i(x):=v(i,x)$ is \emph{superharmonic} for $K$:
$$Kv_i(x)\leq v_i(x)\quad\forall x\in\n.$$
As is well-known, c.f. \citep[Prop.~I.5.1]{Asmussen2003}, any non-negative superharmonic function is necessarily constant. Thus, $v(i,x)=v_i(x)=c$ for all $x$ in $\n$ and some real number $c$ showing that $v$ cannot be norm-like.

\subsubsection*{A proof of Theorem~\ref{thrm:lyarec}} We do each direction separately. For the forward direction, we need the following lemma showing  we lose  nothing by assuming that the set $F$ in \eqref{eq:lyarec} contains only recurrent states (in which case, $\phi_{F\cup\cal{R}}=\phi_{\cal{R}}$).
\begin{lemma}\label{lem:Fnotrans0}Let $F$ be a finite set such that $\Pbl{\{\phi_F<\infty\}}=1$ for all initial distributions $\gamma$. The same inequality holds if we remove all transient states from $F$.\end{lemma}
\begin{proof} Pick any transient state $z\in F$ and let $F_z:=F\backslash \{z\}$ denote $F$ with $z$ removed from it. It must be the case that $F_z$ is accessible from $z$ (i.e., $\{z\}\to F$ as in Definition~\ref{def:accesible}). Otherwise, $\Pbz{\{\phi_z<\infty\}}=\Pbz{\{\phi_F<\infty\}}=1$ contradicting the transience of $z$. 

Clearly, $F$ itself is accessible from any given state $x$ in $F$. Thus, either $F_z$ is accessible from $x$, or $z$ is accessible from $x$. In the latter case, $F_z$ is accessible from $x$ because it is accessible from $z$. It then follows from the transitivity of $\to$ that $F_z$ is accessible from every $x$ in $F$ (i.e., $F\to F_z$) and, so, Lemma~\ref{lem:geotrialarg}$(i)$ that $\Pbl{\{\phi_{F_z}<\infty\}}=1$ for all initial distributions $\gamma$. Because $F$ is finite, it contains at most finitely many transient states. Hence, the result follows by repeatedly replacing $F$ with $F_z$ throughout the above argument.
\end{proof}
The proof of the forward direction then goes as follows:
\begin{proof}[Proof of Theorem \ref{thrm:lyarec} (the forward direction).]
Given that $\Pb_\gamma=\sum_{x\in\s}\gamma(x)\Pb_x$, it suffices to show that
\begin{equation}\label{eq::dm8932bt3bqg8ygh39n3q4g}\Pbx{\{\phi_{\cal{R}}<\infty\}}=1\quad\forall x\in\s.\end{equation}
Because of Lemma~\ref{lem:Fnotrans0}, we may assume that $F$ is contained in $\cal{R}$. To argue the above, set $\cal{D}$ to be the set of transient states such that the exit time $\sigma$ in \eqref{eq:eddef} equals the hitting time of $\cal{R}$. 
Because $F\subseteq\cal{R}$, inequality in \eqref{thrm:lyarec} implies that
$$\hat{P}v(x)\leq v(x)\quad\forall x\in\s,$$
where $\hat{P}=(\hat{p}(x,y))_{x,y\in\s}$ is as in \eqref{eq:phat}. Recall that $\hat{P}$ is a one-step matrix and that its associated chain $\hat{X}$ (c.f. discussion after Theorem~\ref{characttd}) is identical to $X$ except that each state inside $\cal{R}$ has been turned into an absorbing state. Iterating the above, we find that
$$\hat{P}_nv(x)\leq v(x)\quad\forall x\in\s,\enskip n>0,$$
%
where $\hat{P}_n=(\hat{p}_n(x,y))_{x,y\in\s}$ denotes the $n$-step matrix of $\hat{X}$. We then have that
\begin{equation}\label{eq:nt24n3f8a97wj90wajf}\sum_{z\in\s_r}\hat{p}_n(x,z)=1-\sum_{z\not\in\s_r}\hat{p}_n(x,z)\geq 1-\frac{1}{r}\sum_{z\not\in\s_r}\hat{p}_n(x,z)v(z)\geq 1-\frac{1}{r}\hat{P}_nv(x)\geq 1-\frac{v(x)}{r}\enskip \forall x\in\s,\end{equation}
where $\s_r$ denotes the $r$th sublevel set of $v$ (Definition~\ref{def:normlike}).

Because the dynamics of $X$ and $\hat{X}$ are identical except that the states in $\cal{R}$ are absorbing for $\hat{X}$, any state that is transient for $X$ is also transient for $\hat{X}$ (to formally argue this, use Theorem~\ref{characttd}). For this reason, Theorem~\ref{thrm:pointlims0} implies that 
$$\lim_{n\to\infty}\hat{p}_n(x,z)=0\quad\forall x\in\s,\enskip z\not\in \cal{R},$$
and it follows that
$$\lim_{n\to\infty}\sum_{z\in\s_r}\hat{p}_n(x,z)=\lim_{n\to\infty}\sum_{z\in\s_r\cap\cal{R}}\hat{p}_n(x,z)\quad\forall x\in\s,\enskip r>0.$$
%
The limit on the right-hand side exists because  Theorem \ref{charactd} shows that
$$\hat{p}_n(x,z)=\Pbx{\{\sigma\leq n,X_\sigma=z\}}\quad\forall x\in\s,\enskip z\in \cal{R},\enskip n\geq0.$$
Moreover,  taking the limit $n\to\infty$ in \eqref{eq:nt24n3f8a97wj90wajf}, we find that
\begin{align*}\Pbx{\{\sigma<\infty\}}&=\lim_{n\to\infty}\sum_{z\in \cal{R}}\Pbx{\{\sigma\leq n,X_n=z\}}=\lim_{n\to\infty}\sum_{z\in \cal{R}}\Pbx{\{\sigma\leq n,X_\sigma=z\}}\\
&=\lim_{n\to\infty}\sum_{z\in \cal{R}}\hat{p}_n(x,z)\geq\lim_{n\to\infty}\sum_{z\in\cal{S}_r\cap\cal{R}}\hat{p}_n(x,z)\geq 1-\frac{v(x)}{r}\quad\forall x\in\s,\enskip r>0.\end{align*}
%
Taking the limit $r\to\infty$ in the above, we find that $\Pbx{\{\sigma<\infty\}}=1$ for all $x\in\s$. 

Consider now the entrance time $\phi_{\cal{R}}$ to ${\cal{R}}$ (for $X$). By its definition,
$$\Pbx{\{\phi_{\cal{R}}<\infty\}}=\Pbx{\{\sigma<\infty\}}=1\quad\forall x\not\in {\cal{R}}.$$ 
For the states inside ${\cal{R}}$, we use Proposition~\ref{prop:hitprobeqs} to obtain
\begin{align*}\Pbx{\{\phi_{\cal{R}}<\infty\}}=\sum_{z\in {\cal{R}}}p(x,z)+\sum_{z\not\in {\cal{R}}}\Pbx{\{X_1=z\}}\Pbz{\{\phi_{\cal{R}}<\infty\}}=\sum_{z\in {\cal{R}}}p(x,z)+\sum_{z\not\in {\cal{R}}}p(x,z)=1,\end{align*}
for all $x\in\s$, completing the proof of \eqref{eq::dm8932bt3bqg8ygh39n3q4g}.
\end{proof}
We split the proof of the reverse direction into three lemmas:
\begin{lemma}\label{lem:dtexitrinf}Let $(\s_r)_{r\in\zp}$ be a sequence of increasing truncations that approach the state space ($\cup_{r=1}^\infty\s_r=\s$) and $(\sigma_r)_{r\in\zp}$ be the corresponding sequence\glsadd{sigmar} of exit times:
\begin{equation}\label{eq:sigmar}\sigma_r(\omega):=\inf\{n\geq0:X_n(\omega)\not\in \s_r\}\quad\forall\omega\in\Omega,\enskip r>0.\end{equation}
For any initial distribution $\gamma$,
$$\lim_{r\to\infty}\sigma_r=\infty\quad\Pb_\gamma\text{-almost surely}.$$
\end{lemma}
\begin{proof}The limit $\sigma_\infty:=\lim_{r\to\infty}\sigma_r$ exists as the sequence $(\sigma_r)_{r\in\zp}$ is increasing. By definition,
$$\Pbl{\{\sigma_r\leq k\}}=\sum_{j=0}^k\Pbl{\{\sigma_r=j\}}\leq \sum_{j=0}^k\Pbl{\{X_j\not\in\s_r\}} =\sum_{j=0}^kp_j(\s_r^c)$$
and, given that $\cup_{r=1}^\infty\s_r=\s$, downwards monotone convergence shows that
$$\Pbl{\{\sigma_\infty\leq k\}}=\lim_{r\to\infty}\Pbl{\{\sigma_r\leq k\}}=\sum_{j=0}^k\lim_{r\to\infty}p_j(\s_r^c)=0.$$
Taking the limit $k\to\infty$ and applying monotone convergence completes the proof.
\end{proof}

\begin{lemma}\label{lem:dtexitrinf2}Suppose that the chain is Tweedie recurrent and has  $n\in\n$ closed communicating classes labelled $\cal{C}_1,\cal{C}_2,\dots,\cal{C}_n$. If $F=\{x_1,x_2,\dots,x_n\}$, where $x_1$ denotes a state in $\cal{C}_1$, $x_2$ one in $\cal{C}_2$, etc., then $\Pbl{\{\phi_F<\infty\}}=1$ for all initial distributions $\gamma$.
\end{lemma}

\begin{proof}By definition, $\phi_{F}$ is at most $\phi_{x_m}$, and so 
\begin{align*}\{\phi_F\neq\phi_{x_m}<\infty\}&=\{\phi_F<\phi_{x_m}<\infty\}=\bigcup_{k\neq m}\{\phi_F=\phi_{x_k}<\phi_{x_m}<\infty\}\\
&\subseteq \bigcup_{k\neq m}\{\phi_{x_{k}}<\infty,\phi_{x_m}<\infty\}\subseteq \bigcup_{k\neq m}\{\phi_{\cal{C}_{k}}<\infty,\phi_{\cal{C}_m}<\infty\}\quad\forall m=1,2,\dots,n.\end{align*}
Taking expectations and applying Proposition~\ref{prop:closedis} then yields $\Pbl{\{\phi_F\neq\phi_{x_m}<\infty\}}=0$ for all $m$. It follows that
\begin{align*}\Pbl{\{\phi_F<\infty\}}&=\sum_{m=1}^n\Pbl{\{\phi_F=\phi_{x_m}<\infty\}}=\sum_{m=1}^n\Pbl{\{\phi_{x_m}<\infty\}}\end{align*}
Using Proposition~\ref{prop:closedis}~and~Corollary~\ref{cor:phixphic} then completes the proof:
$$\Pbl{\{\phi_F<\infty\}}=\sum_{m=1}^n\Pbl{\{\phi_{\cal{C}_m}<\infty\}}=\Pbl{\bigcup_{m=1}^n\{\phi_{\cal{C}_m}<\infty\}}=\Pbl{\{\phi_{\cal{R}}<\infty\}}=1,$$
for all initial distributions $\gamma$.

\end{proof}
The following lemma then completes the proof of Theorem~\ref{thrm:fosrec}:
\begin{lemma}\label{lem:fosrecfinal}If $F$ is a finite set such that $\Pbl{\{\phi_F<\infty\}}=1$ for all initial distributions $\gamma$, then there exists a non-negative norm-like function $v$ on $\s$ such that \eqref{eq:lyarec} holds.
\end{lemma}
\begin{proof}Let $(\s_r)_{r\in\zp}$ be any sequence of increasing finite truncations containing $F$ ($F\subseteq\s_1$) that approach the state space ($\cup_{r=1}^\infty\s_r=\s$) and let $(\sigma_r)_{r\in\zp}$ be the corresponding sequence of exit times~\eqref{eq:sigmar}. Additionally,  let 
$$\sigma_F(\omega):=\inf\{n\geq0:X_n(\omega)\in F\}\quad\forall \omega\in\Omega$$
be the hitting time of $F$ and, for any given state $x$,
$$v_r(x):=\Pbx{\{\sigma_r\leq \sigma_F\}}$$
be the probability that the chain leaves $\s_r$ without hitting $F$ if it starts at $x$. By definition, $v_r(x)=1$ if  $x$ lies outside of $\s_r$ and we have that
\begin{equation}\label{eq:nfwa89nf7w8af23h8ya}Pv_r(x)=\sum_{y\in\s}p(x,y)\Pby{\{\sigma_r\leq \sigma_F\}}\leq \sum_{y\in\s}p(x,y)=1=v_r(x)\quad\forall x\not\in\s_r.\end{equation}
If instead $x$ lies inside of $\s_r$ but outside of $F$, then
$$\sigma_r=\inf\{n\geq1:X_n\not\in \s_r\},\quad\sigma_F=\inf\{n\geq1:X_n\in F\},\quad\Pb_x\text{-almost surely}.$$
For this reason, an application of the strong Markov property similar to that in the proof of Lemma~\ref{lem:dtenttimestr3} yields
\begin{align*}Pv_r(x)&=\sum_{y\in\s}p(x,y)v_r(y)=\sum_{y\in\s}\Pbx{\{X_1=y\}}\Pby{\{\sigma_r\leq \sigma_F\}}\\
&=\sum_{y\in\s}\Pbx{\{X_1=y,\inf\{n\geq1:X_n\not\in \s_r\}\leq\inf\{n\geq1:X_n\in F\}\}}\\
&=\Pbx{\{\sigma_r\leq \sigma_F\}}=v_r(x)\end{align*}
for all $x$ in $\s_r$ but outside $F$. The above and \eqref{eq:nfwa89nf7w8af23h8ya} show that $v_r$ satisfies the inequality in \eqref{thrm:lyarec}. 

However, $v_r$ is clearly not norm-like given that it is bounded. To obtain a norm-like function let $(r_k)_{k\in\zp}$ be any increasing sequence of positive integers and note that 
$$\sum_{y\in\s}p(x,y)\sum_{k=1}^\infty v_{r_k}(y)\leq \sum_{k=1}^\infty \sum_{y\in\s}p(x,y)v_{r_k}(y)\leq \sum_{k=1}^\infty v_r(x)\quad\forall x\not\in F $$
because of Fatou's lemma. That is, the function $v$ defined by
$$v(x):=\sum_{k=1}^\infty v_{r_k}(x)\quad\forall x\in\s$$
also satisfies the inequality in \eqref{thrm:lyarec}. Moreover, because $v_{r_k}(x)=1$ for all $x$ outside $\s_{r_k}$, $v(x)$ is at least $N$ if  $x$ does not belong to $\s_{r_N}$ and it follows that $v$ is norm-like. 

For $v$ to meet all the conditions in the theorem's premise, we  only have left to show that it is finite (i.e., $v(x)<\infty$ for all $x\in\s$). The trick here is to pick the sequence $(r_k)_{k\in\zp}$ carefully. In particular, our premise implies that
$$\Pbx{\{\sigma_F<\infty\}}=\Pbx{\{\phi_F<\infty\}}=1\quad\forall x\not\in F.$$
Given that $\sigma_r$ tends to $\infty$ with $\Pb_x$-probability one (Lemma \ref{lem:dtexitrinf}), it follows from the above that
$$\lim_{r\to\infty} v_r(x)=\lim_{r\to\infty}\Pbx{\{\sigma_r\leq \sigma_F\}}=\Pbx{\{\sigma_F=\infty\}}=0\quad\forall x\in\s.$$
For this reason, finiteness of $\s_1,\s_2,\dots$ allows us to find an increasing sequence $(r_k)_{k\in\zp}$ tending to infinity  such that
$$v_{r_k}(x)\leq \frac{1}{2^k}\quad \forall x\in\s_k,\enskip k>0.$$
With this choice, we have that, for any given $x\in\s$,
$$v(x)=\sum_{k=1}^\infty v_{r_k}(x)=\sum_{k=1}^{K-1} v_{r_k}(x)+\sum_{k=K}^\infty\frac{1}{2^k}<\infty,$$
where $K$ is large enough that $x$ belongs to $\s_K$.
\end{proof}

\subsubsection*{Notes and references} The sufficiency of Theorem~\ref{thrm:lyarec} was first shown in \citep{Foster1953} for irreducible chains and singleton sets $F$. This was later generalised to arbitrary finite sets in \citep{Pakes1969}. The generalisation for non-irreducible chains was proven in \citep{Tweedie1975b}. The necessity was first shown in \citep{Mertens1978} for irreducible and aperiodic chains. To the best of my knowledge, it took another fifteen years until the theorem's necessity was proven for the general case in the first edition of \citep{Meyn2009} (where it was shown for a more general class of discrete-time Markov processes).

\subsection{Foster-Lyapunov criteria II: Foster's theorem*}\label{sec:fosters}

A fantastically useful result is Theorem~\ref{thrm:fosters} below, commonly known as \emph{Foster's theorem}:\index{Foster-Lyapunov criteria}\index{Foster's theorem}  
\begin{theorem}[Foster's theorem]\label{thrm:fosters} If there exists a non-negative real-valued function $v$ on $\s$, a finite subset $F$ of $\s$, and a real number $b$ satisfying 
\begin{equation}\label{eq:fosters}Pv(x)\leq v(x)-1+b1_{F}(x)\quad \forall x\in\s,\end{equation}
then the chain is positive Tweedie recurrent and has a finite number of closed communicating classes. Conversely, if the chain is positive Tweedie recurrent and if the state space contains no transient states and only a finite number of closed communicating classes, then there exists $(v,F,b)$ satisfying \eqref{eq:fosters} with $v$ real-valued and non-negative, $F$ finite, and $b$ real.
\end{theorem}
This theorem is one of the few tools we have in practice to establish positive Tweedie recurrence of a chain (and, consequently, the existence of stationary distributions, the convergence of the empirical distribution, etc.). It is a consequence of Theorem \ref{thrm:fosrec} below which shows that for a fixed $F$, the existence of a $v$ and $b$ satisfying~\eqref{eq:fosters} is the same as the mean entrance time to $F$ being finite for all deterministic initial positions.

Notice that Theorem~\ref{thrm:fosters} shows that the existence of the $(v,F,b)$ satisfying~\eqref{eq:fosters} is equivalent to positive Tweedie recurrence in the case of no transient states and only a finite number of closed communicating classes.  The importance of a finite number of closed communicating classes is easy to see: the set $F$ must contain at least one state from each closed communicating class. Otherwise, the chain would not be able to reach $F$ whenever it starts inside of a closed communicating class that does not intersect with $F$ and the entrance time to $F$ would be infinite. As an example, the chain with one-step matrix $P$ being the identity matrix is trivially positive recurrent, however $Pv(x)-v(x)=0\geq-1$ for all states $x$ and functions $v$. Consequently, if the state space is infinite, then \eqref{eq:fosters} will never be satisfied for for a finite $F$.

The no transient states requirement is more subtle. Positive Tweedie recurrence asks that the chain reaches a positive recurrent state with probability one. The criterion instead requires that the mean amount of time it takes the chain to reach the positive recurrent states is finite. As the example below shows the latter is a stronger demand. This is a one of the unfortunately numerous corners of Markov process theory where two concepts are almost the same, but not quite (so close, and yet...).

\begin{example}\label{ex:recnofos}Consider again the gambler's ruin problem introduced in Section~\ref{sec:gamblers} and suppose that the coin is unbiased ($a=1/2$). As shown at the end of Section~\ref{sec:exit}, for any initial distribution, the chain has probability one of entering the absorbing state $0$. Because $\{0\}$ is the only closed communicating class, it follows that the chain is positive Harris recurrent. 
However, it is not difficult to use the Markov property to verify that, unless the chain starts at $0$, the average amount of steps it takes the chain to enter $0$ is infinite. The reason why is that the chain makes very long trips away from $0$ before eventually entering $0$. Furthermore, these long trips ensure that, for any given finite set $F$, there exists a state $x$ such that the average number of steps it takes the chain to enter $F$ if it starts at $x$ is infinite and it follows from Theorem~\ref{thrm:fosrec} below that no $(v,F,b)$ satisfying the premise of \eqref{eq:fosters} exist.

More directly, suppose that there exists a non-negative real-valued function $v$ satisfying 
$$Pv(x)=\frac{1}{2}v(x-1)+\frac{1}{2}v(x+1)\leq v(x)-1$$  
for all states $x$ outside some finite set $F$. Letting $z$ denote the largest state in $F$, we have that
$$v(x+1)-v(x)\leq v(x)-v(x-1)-2\quad\forall x>z.$$
Iterating the above, we find that
$$v(z+n)-v(z+n-1)\leq v(z+n-1)-v(z+n-2)-2\leq \dots\leq v(z+1)-v(z)-2(n-1).$$
For this reason,
\begin{align*}v(z+n)&=v(z)+\sum_{m=1}^n(v(z+m)-v(z+m-1))\leq v(z)+\sum_{m=1}^n(v(z+1)-v(z)-2(m-1))\\
&=v(z)+n(v(z+1)-v(z))-n(n-1)=v(z)+n(v(z+1)-v(z)+1)-n^2.\end{align*}
Taking the limit $n\to\infty$ shows that $v(x)\to-\infty$ as $x\to\infty$ contradicting the non-negativeness of $v$ and proving that no function $v$ and set $F$ satisfying the premise of \eqref{thrm:fosrec} exist even though the chain is positive Harris recurrent.
\end{example}

\subsubsection*{A proof of Foster's theorem}As mentioned before, Theorem~\ref{thrm:fosrec} really is a consequence of the following:
\begin{theorem}\label{thrm:fosrec}Given any finite set $F$, there exists $v:\s\to[0,\infty)$ and $b\in\r$ satisfying \eqref{eq:fosters} if and only if $\Ebx{\phi_F}<\infty$ for all $x\in\s$. In this case, $\Ebl{\phi_F}<\infty$ whenever the initial distribution $\gamma$ satisfies $\gamma(v)<\infty$.\end{theorem}
The key to proving the above is the following lemma.
\begin{lemma}\label{lem:fosmin}Let $F$ be a subset of the state space and $\theta>0$. The minimal non-negative solution $v:\s\to[0,\infty]$ to the inequality
\begin{align}
\label{eq:dlyain1}Pv(x)&\leq\theta^{-1}v(x)-1\quad\forall x\not\in F
\end{align}
is given by $u(x)=0$ for all $x$ in $F$ and
\begin{equation}\label{eq:umin}u(x)=\theta\sum_{n=0}^\infty\Pbx{\{\phi_F>n\}}\theta^{n}\quad\forall x\not\in F.\end{equation}
\end{lemma}

\begin{proof} For any natural number $n$, let
$$f_n(x,y):=\Pbx{\{\phi_F>n,X_n=y\}}\quad\forall x,y\not\in F,$$
so that
$$\Pbx{\{\phi_F>n\}}=\sum_{y\not\in F}f_n(x,y)\qquad\forall x\not\in F.$$
Applying the Markov property (Theorem~\ref{thrm:strmkvpath}), we find that
\begin{align*}\Pbx{\{\phi_F>n,X_n=y\}}&=\sum_{z\not\in F}\Pbx{\left\{X_{1}=z,\phi_F-1>n-1,X_{1+(n-1)}=y\right\}}\\
&=\sum_{z\not\in F}p(x,z)f_{n-1}(z,y)\end{align*}
if $n>0$. For this reason,
$$f_n(x,y)=1_x(y)1_0(n)+(1-1_0(n))\sum_{z\not\in F}p(x,z)f_{n-1}(z,y)\qquad\forall x,y\not\in F.$$
Thus, we have that $u$ in \eqref{eq:umin} satisfies
\begin{align*} \theta^{-1} u(x)&=\sum_{n=0}^\infty\sum_{y\not\in F}f_n(x,y)\theta^{n}=1+\sum_{n=1}^\infty\sum_{y\not\in F}\sum_{z\not\in F}p(x,z)f_{n-1}(z,y)\theta^n\\
&=1+\sum_{z\not\in F}p(x,z)\left(\theta\sum_{n=1}^\infty\sum_{y\not\in F}f_{n-1}(z,y)\theta^{n-1}\right)=1+Pu(x)\quad\forall x\not\in F\end{align*}
proving that $u$ is a solution to \eqref{eq:dlyain1}. To prove minimality of $u$, let $v$ be any other non-negative solution to \eqref{eq:dlyain1}. By definition, $v(x)\geq 0 =u(x)$ for all $x\in F$. For all other states, note that
\begin{align*}v(x)\geq& \theta + \theta\sum_{x_1\in\s}p(x,x_1)v(x_1)\geq\theta + \theta\sum_{x_1\not\in F}p(x,x_1)v(x_1)\\
 \geq&\theta + \theta^2\sum_{x_1\not\in F}p(x,x_1)+ \theta^2\sum_{x_1\not\in F}\sum_{x_2\in \s}p(x,x_1)p(x_1,x_2)v(x_2)\\
\geq&\theta + \theta^2\sum_{x_1\not\in F}p(x,x_1)+ \theta^2\sum_{x_1\not\in F}\sum_{x_2\not\in F}p(x,x_1)p(x_1,x_2)v(x_2)\\
\geq&\dots\\
\geq& \theta + \theta^2\sum_{x_1\not\in F}p(x,x_1)+ \theta^3\sum_{x_1\not\in F}\sum_{x_2\not\in F}p(x,x_1)p(x_1,x_2)+\dots\\
&+\theta^{l}\sum_{x_1\not\in F}\dots\sum_{x_{l-1}\not\in F}p(x,x_1)\dots p(x_{l-2},x_{l-1})
+\theta^l\sum_{x_1\not\in F}\dots\sum_{x_l\not\in F}p(x,x_1)\dots p(x_{l-1},x_l)v(x_l)
\\=&\theta\Pbx{\{\phi_F>0\}}+\theta^2\Pbx{\{\phi_F>1\}}+\theta^3\Pbx{\{\phi_F>2\}}+\dots+\theta^l\Pbx{\{\phi_F>l\}}\\
&+\theta^l\sum_{x_1\not\in F}\dots\sum_{x_{l-1}\not\in F}p(x,x_1)\dots p(x_{l-1},x_l)v(x_l)\quad\forall x\not\in F.\end{align*}
Because non-negativity of $v$ implies that the right most term is non-negative, taking the limit $l\to\infty$ then shows that $v(x)\geq u(x)$ as desired.
\end{proof}
Proving Theorem~\ref{thrm:fosrec} is now straightforward:
\begin{proof}[Proof of Theorem \ref{thrm:fosrec}] Suppose that \eqref{eq:fosters} is satisfied. Setting $\theta:=1$ in \eqref{lem:fosmin} shows that 
$$\Ebx{\phi_F}\leq v(x)\quad\forall x\not\in F.$$
For states inside $F$, an application of the strong Markov property similar to that in the proof of Propisition~\ref{prop:hitprobeqs} yields
$$\Ebx{\phi_F}=1+\sum_{z\not\in F}p(x,z)\Ebz{\phi_F}\leq1+\sum_{z\not\in F}p(x,z)v(z)\leq1+Pv(x)\leq v(x)+ b\quad\forall x\in F.$$
Given that $\Pb_\gamma=\sum_{x\in\s}\gamma(x)\Pb_x$, multiplying the above two inequalities by $\gamma(x)$ and summing over $x$ in $\s$ then shows that $\Ebl{\phi_F}$ is finite whenever $\gamma(v)$ is finite.
%

Conversely, suppose that $\Ebx{\phi_F}<\infty$ for all $x$ in $\s$. Lemma \ref{lem:fosmin} shows that $u$ in \eqref{eq:umin} (with $\theta=1$) satisfies \eqref{eq:fosters} for all states $x$ outside of $F$. For states inside, we apply the Markov property as before:
$$Pu(x)=\sum_{z\in\s}p(x,z)u(z)=\sum_{z\not\in F}p(x,z)\Ebz{\phi_F}=\Ebx{\phi_F}-1\quad\forall x\in F.$$
In other words, $u$ satisfies \eqref{eq:fosters} with $b:=\max\{\Ebx{\phi_F}:x\in F\}$.
\end{proof}

To make full use of Theorem \ref{thrm:fosrec} and prove Foster's theorem, we require one final result:
%
%
\begin{lemma}\label{lem:Fnotrans}Let $F$ be a finite set such that $\Ebx{\phi_F}<\infty$ for all $x$ in $\s$. The same inequality holds if we remove all transient states from $F$.\end{lemma}
\begin{proof}Replace Lemma~\ref{lem:geotrialarg}$(i)$ with Lemma~\ref{lem:geotrialarg}$(ii)$ in the proof of Lemma~\ref{lem:Fnotrans0}.
\end{proof}
We are now in a great position to prove Foster's theorem (Theorem \ref{thrm:fosters}):
\begin{proof}[Proof of Theorem \ref{thrm:fosters}] To prove the forward direction, suppose that there exists $(v,F,b)$ as in the premise. Theorem~\ref{thrm:fosrec} shows that $\Ebx{\phi_F}$ is finite for all $x$ in $\s$. Lemma~\ref{lem:Fnotrans} shows that, without any loss of generality, we may assume that $F$ contains only recurrent states. Thus,
$$\Ebx{\phi_{\cal{R}}}\leq\Ebx{\phi_F}<\infty\quad \forall x\in\s,$$
where $\cal{R}$ is the set of recurrent states, and it follows that the chain is positive Tweedie recurrent. Furthermore, Proposition~\ref{prop:closedis} implies that 
$$\Ebx{\phi_{F\cap\cal{C}}}=\Ebx{\phi_F}<\infty\quad\forall x\in\cal{C},$$
where $\cal{C}$ is any closed communicating class. The intersection $F\cap\cal{C}$ must then be non-empty (otherwise, $\phi_{F\cap\cal{C}}=\infty$ contradicting the above). Given that $F$ is finite and closed communicating classes are disjoint by definition, it follows that there exists only a finite number of closed communicating classes.

Conversely, suppose that only finitely many closed communicating classes $\mathcal{C}_1,\C_2,\dots,\C_l$ exist and that there are no transient states. Positive Tweedie recurrence implies that all of these classes are positive recurrent. Pick any states $x_1\in\C_1,x_2\in\C_2,\dots, x_l\in\C_l$ and set $F:=\{x_1,x_2,\dots,x_l\}$. Because $x_j$ is only accessible from states in $\C_j$,
$$\Ebx{\phi_F}=\Ebx{\phi_{x_j}}\quad\forall x\in \C_j.$$
Fix any such state $x\in \C_j$. Because $\Ebx{\phi_x}<\infty$ and $x\to x_j$, Lemma \ref{lem:geotrialarg}$(ii)$ shows that $\Eb_x[\phi_{x_j}]$ is finite. Because $\cup_{j=1}^l\C_j=\s$ (as there are no transient states), the above then shows that $\Ebx{\phi_F}$ is finite for all $x$ in $\s$ and the existence of the $(v,F,b)$ satisfying \eqref{eq:fosters} follows from Theorem \ref{thrm:fosrec}.
\end{proof}

\subsubsection*{Notes and references} The forward direction of Theorem~\ref{thrm:fosters}, for irreducible chains and general $F$, first shows up in  F.~G.~Foster's comments in the discussion of D.~G.~Kendall's famous queueing paper~\citep{Kendall1951}. Foster does not give a proof therein but instead promises that one will be included in an upcoming paper. It seems that he was referring to his well-known paper \citep{Foster1953} published two years later where he proves both the forward and reverse directions of Theorem~\ref{thrm:fosters} for irreducible chains and singleton $F$s. A version of the generalisation of the forward direction to both non-irreducible chains with multiple closed communicating classes and arbitrary $F$s was first shown in~\citep{Mauldon1957}---although the generalisation to arbitrary $F$s in the irreducible case is often attributed to~\citep{Pakes1969} (it was also stated, but not shown, in~\citep{Kingman1961}). In particular, J.~G.~Mauldon showed that the existence of $(v,F,b)$ satisfying \eqref{eq:fosters} implies that the (pointwise) limit 
$$\lim_{N\to\infty}\frac{1}{N}\sum_{n=0}^{N-1}p_n,$$
has unit mass for all initial distributions. Because bounded convergence and Theorem~\ref{thrm:pointlims} imply that this limit is $\pi_\gamma$ in~\eqref{eq:reclims}, Corollary~\ref{cor:dttimeavetv} then shows that the chain is Tweedie recurrent. However, this point seems to have been made explicit much later in~\citep{Tweedie1975b}. I have not encountered the converse given in Theorem~\ref{thrm:fosters} elsewhere, but the underlying ideas are the same as those behind the converses in \citep{Foster1953} for the irreducible case and \citep{Tweedie1975,Tweedie1975b} for the general case.


\ifdraft

\subsection{\textbf{Foster-Lyapunov crietira III: Bounding stationary averages**}}

{\color{red}leave this section for a later version of the book}

Strong duality using Proposition 4 of Glynn.

\begin{theorem}\emph{(Theorems 11.3.4 and 14.3.7 in \citep{Meyn2009})\textbf{.}}\label{lyadth} Suppose that $u:\s\to[0,\infty)$,  $h:\s\to[1,\infty)$, $c>0$, $d\geq0$, $A$ is a finite set, and
\begin{equation}\label{eq:lyaddn}Pu(x)-u(x)\leq d1_A(x)-ch(x),\qquad\forall x\in\s.\end{equation}
Then each closed communicating class of $P$ has an ergodic distribution, and $\pi(h)\leq d/c$ for each stationary distribution $\pi$. 
\end{theorem}
\begin{proof} Pick any increasing sequence of finite subsets $\s_1\subseteq\s_1\dots$ of $\s$ such that $\cup_{r=1}^\infty\s_r=\s$ and let $\sigma_r:=\inf\{n\in\n:X_n\not\in\s_r\}$ be the time of first exit from $\s_r$. Fix an $n\in\n$ and consider the martingale $\{M_n\}_{n\in\n}$ defined in \eqref{eq:martthe} with $f:=u_k:=u1_{\s_k}$ and $g:=1$. Because $M_0=0$ and $\sigma_r\wedge n$ is a bounded $\cal{F}_n$-stopping time, Doob's Optional Stopping Theorem \citep[Sec.10.10]{Williams1991} tells us that $\Ebl{M_{\sigma_r\wedge n}}=\Ebl{M_0}=0$ which reads 
$$\Ebx{\sum_{m=0}^{\sigma_r\wedge n-1}u_k(X_m)}=u_k(x)+\Ebx{\sum_{m=0}^{\sigma_r\wedge n-1}Pu_k(X_m)}.$$
Taking the limit $k\to\infty$ and applying the Monotone Convergence Theorem we obtain
\begin{equation}\label{eq:fmeuasiofnaeisu8}\Ebx{\sum_{m=0}^{\sigma_r\wedge n}u(X_m)}=u(x)+\Ebx{\sum_{m=0}^{\sigma_r\wedge n-1}Pu(X_m)}.\end{equation}
Because $u$ is non-negative and
\begin{equation}\label{eq:mfuianewai7}\Ebx{\sum_{m=0}^{\sigma_r\wedge n-1}u(X_m)}\leq \Ebx{\sum_{m=0}^{\sigma_r\wedge n-1}\max_{x\in\s_r}u(x)}\leq n\max_{x\in\s_r}u(x)<\infty,\end{equation}
we can subtract $\Ebx{\sum_{m=0}^{\sigma_r\wedge n-1}u(X_m)}$ from both sides of \eqref{eq:fmeuasiofnaeisu8} and apply \eqref{eq:lyaddn} to obtain
\begin{align*}0&\leq \Ebx{u(X_{\sigma_r\wedge n})}=u(x)+\Ebx{\sum_{m=0}^{\sigma_r\wedge n-1}Pu(X_m)-u(X_m)}\\
&\leq u(x)+\Ebx{\sum_{m=0}^{\sigma_r\wedge n-1}d1_A(X_m)-ch(X_m)}.\end{align*}
The same reasoning as in \eqref{eq:mfuianewai7} shows that $\sum_{m=0}^{\sigma_r\wedge n-1}ch(X_m)$ is integrable, and so we can subtract $\Ebx{\sum_{m=0}^{\sigma_r\wedge n-1}ch(X_m)}$ from both sides to obtain 
$$c\Ebx{\sigma_r\wedge n}\leq c\Ebx{\sum_{m=0}^{\sigma_r\wedge n-1}h(X_m)}\leq u(x)+d\Ebx{\sum_{m=0}^{\sigma_r\wedge n-1}1_A(X_m)},$$
where the first inequality follows from our assumption that $h\geq 1$. {\color{red}Because $\sigma_r$ is an increasing sequence with $\Pb_x$-almost sure limit $\infty$}, the Monotone Convergence Theorem implies that 
$$cn\leq u(x)+d\Ebx{\sum_{m=0}^{n-1}1_A(X_m)}=u(x)+d\sum_{m=0}^{n-1}\Ebx{1_A(X_m)}=u(x)+d\sum_{m=0}^{n-1}{\color{red}p_m(x,A).}$$
Dividing through by $n$ and taking the limit $n\to\infty$ shows that
$$\liminf_{n\to\infty}\frac{1}{n}\sum_{m=0}^{n-1}p_m(x,A)\geq c/d>0.$$
\end{proof}
In the case of irreducible chains, the above has a partial converse, see \citep[Thrm.11.3.15]{Meyn2009}. Additionally, this same condition implies the convergence of the time averages and, if the chain is aperiodic, also of the space averages, both discussed in the chapter's introduction, see \citep{Meyn2009} for more on this.

{\color{red} FOR STRONG DUALITY BETWEEN STATIONARY EQUATIONS AND LYAPUNOV BOUND (AS IN IN GLYNN 2008), SEE THEOREM 2 IN Simplex Algorithm for Countable-State Discounted Markov
Decision Processes}

FOR DT CHAINS CONVERSES SEE COR 2.1 OF ADDITIVE FUNCTIONALS FOR DISCRETE-TIME MARKOV CHAINS WITH APPLICATIONS TO
BIRTH-DEATH PROCESSES

\begin{lemma}\label{rollerfin} For all $C,K,\theta>0$, the chain with one-step matrix \eqref{eq:rollerostep} has a stationary distribution $\pi$ and its moment generating function $t\mapsto\pi(e^{tx})$ is finite for all $t\in\r$.
\end{lemma}
\begin{lemma}\label{inarmgf} Suppose that the input distribution  $\gamma$ of the chain with one-step matrix \eqref{eq:inratrans} has a finite moment generating function for some $t>0$ (that is, $\gamma(e^{tx})<\infty$) and that $a\in[0,1)$. The chain has at least one stationary distribution and every stationary distribution $\pi$ has a finite moment generating function $\pi(e^{tx})$ for some $t>0$.
\end{lemma}

\fi

\subsection[Foster-Lyapunov criterion for geometric convergence*]{Foster-Lyapunov criteria III: the criterion for geometric convergence*}\label{sec:flgeodt}

Lemma \ref{lem:fosmin} begs us to study the more general inequality
\begin{equation}\label{eq:fostersgeo}Pv(x)\leq \theta^{-1} v(x)-1+b1_{F}(x)\quad \forall x\in\s.\end{equation}
instead of the special case $\theta=1$ considered in Section~\ref{sec:fosters}.  The lemma shows that, for any given $\theta>0$, the inequality's minimal (possibly infinite-valued) non-negative solution is given by $u(x)=0$ for all $x\in F$ and
\begin{equation}\label{eq:ugeo1}u(x)=\theta\sum_{n=0}^\infty\Pbx{\{\phi_F>n\}}\theta^n=\theta\Ebx{\sum_{n=0}^\infty1_{\{\phi_F>n\}}\theta^n}=\theta\Ebx{\sum_{n=0}^{\phi_F-1}\theta^n}\quad\forall x\not\in F.\end{equation}
For $\theta\neq 1$, we have that
\begin{equation}\label{eq:ugeo2}u(x)=\theta\Ebx{\frac{\theta^{\phi_F}-1}{\theta-1}}=\frac{\theta}{\theta-1}(\Eb_x[\theta^{\phi_F}]-1)\quad\forall x\not\in F.\end{equation}
The $\theta>1$ analogue of then Theorem \ref{thrm:fosrec} follows immediately:
\begin{theorem}\label{thrm:fosgeorec}Given any finite set $F$ and $\theta>1$, there exists $v:\s\to[0,\infty)$ and $b\in\r$ satisfying \eqref{eq:fostersgeo} if and only if $\Ebx{\theta^{\phi_F}}<\infty$ for all $x\in\s$. In this case, $\Ebl{\theta^{\phi_F}}<\infty$  whenever the initial distribution satisfies $\gamma(v)<\infty$.
\end{theorem}
\begin{proof}Given \eqref{eq:ugeo1}--\eqref{eq:ugeo2}, the proof is entirely analogous to that of Theorem \ref{thrm:fosrec}.\end{proof}

The above theorem shows that \eqref{eq:fostersgeo} is satisfied for some $\theta>1$, then the return time distribution of $F$ has a geometrically decaying tail (i.e., a \emph{light tail}) for any deterministic starting state. Conversely, if the tails are heavy for any one starting state, then \eqref{eq:fostersgeo} will not be satisfied. Kendall's theorem (Theorem~\ref{thrm:kendall}) then relates the light tails with the geometric convergence of the time-varying law, and we obtain:\index{Foster-Lyapunov criteria}\index{drift conditions|seealso{Foster-Lyapunov criteria}}
\begin{theorem}[The geometric criterion]\label{thrm:fostersgeo} If there exists a real-valued non-negative function $v$ on $\s$, a finite set $F$, and a real number $b$ satisfying \eqref{eq:fostersgeo} for some $\theta>1$, then the chain is positive Tweedie recurrent. Moreover, if the chain is aperiodic and $\gamma(v)<\infty$, then  the time varying law converges geometrically fast: there exists some $\kappa>1$ such that
\begin{equation}\label{eq:tvgeoconv}\norm{p_n-\pi_\gamma}=\cal{O}(\kappa^{-n}),\end{equation}
where $\pi_\gamma$ denotes the limiting stationary distribution in~\eqref{eq:reclims2} and $\norm{\cdot}$ denotes the total variation norm in~\eqref{eq:tvnorm}.

Conversely, if the chain is aperiodic and positive Tweedie recurrent, there exists no transient states and only a finite number of closed communicating classes, and \eqref{eq:tvgeoconv} holds whenever the chain starts at a fixed state (i.e., whenever $\gamma=1_x$ for some $x$ in $\s$), then there exists $(v,F,b,\theta)$ satisfying \eqref{eq:fosters} with $v$ real-valued and non-negative, $F$ finite, $b$ real, and $\theta>1$.
\end{theorem}

It is straightforward to see that the finitely-many-closed-communicating-classes requirement in the converse is necessary: consider a chain on an infinite state space whose one-step matrix is the identity matrix.
\subsubsection*{An open question} To the best of my knowledge, whether or not the no-transient-states requirement in the converse is necessary remains an open question. I believe that the answer here lies in that of the open question discussed in Section~\ref{sec:kendall}.

\subsubsection*{First-entrance last-exit decomposition} For the proof of Theorem~\ref{thrm:fostersgeo} we require the so-called \emph{first-entrance last-exit decomposition} of the $n$-step matrix:\index{first-entrance last-exit decomposition}
\begin{equation}\label{eq:feled}p_n(x,y)=p_n^z(x,y)+\sum_{k=1}^{n-1}\sum_{l=1}^kp^z_l(x,z)p_{k-l}(z,z)p^z_{n-k}(z,y)\quad\forall x,y\in\s,\enskip n>0,\end{equation}
where $z$ is any given state in $\s$ and $p_n^z(x,y)$ denotes the probability that the chain transitions from $x$ to $y$ without visiting $z$ during time-steps $1,2,\dots,n-1$:
\begin{equation}\label{eq:tabood}p_n^z(x,y)=\Pbx{\{X_n=y,\phi_{z}\geq n\}}\quad \forall x,y\in\s,\enskip n>0.\end{equation}
The above  are often referred to as the \emph{taboo probabilities}.
The decomposition \eqref{eq:feled} follows by noting that
\begin{align*}\{X_n=y\}&=\{X_n=y,\text{X doesn't visit $z$ during time-steps $1,2,\dots,n-1$}\}\\
&\cup\bigcup_{k=1}^{n-1}\{X_n=y,X\text{ last visits $z$ at time $k$}\}\\
&=\{X_n=y,\text{X doesn't visit $z$ during time-steps $1,2,\dots,n-1$}\}\\
&\cup\bigcup_{k=1}^{n-1}\left(\bigcup_{l=1}^k\{X_n=y,X\text{  first visits $z$ at time $l$ and last visits $z$ at time $k$}\}\right)\end{align*}
and that these events are disjoint. 
\begin{exercise}Taking expectations of the above and applying \eqref{eq:nstepthe}--\eqref{eq:nstepdef} and Tonelli's Theorem, prove \eqref{eq:feled}.\end{exercise}

As we will in the proof of Theorem~\ref{thrm:fostersgeo}, this decomposition turns out to be the key to strengthening the convergence of $p_n(x,x)$ for individual states $x$ (as in Kendall's Theorem, Theorem~\ref{thrm:kendall}) to convergence in total variation of the entire time-varying law. In particular, suppose that $z$ belongs to a positive recurrent class $\C$ with ergodic distribution $\pi$. Because $\Ebz{1_y(X_0)}=1_y(z)=\Ebz{1_y(X_{\phi_z})}$ for any recurrent state $z$, Theorem~\ref{doeblind}$(i)$ shows that
\begin{align*}\pi(y)&=\frac{1}{\Ebz{\phi_z}}\Ebz{\sum_{n=0}^{\phi_z-1}1_y(X_n)}=\pi(z)\Ebz{\sum_{n=0}^{\phi_z-1}1_y(X_n)}=\pi(z)\Ebz{\sum_{k=1}^{\infty}1_{\{\phi_z\geq k\}}1_y(X_k)}\\
&=\pi(z)\sum_{k=1}^{\infty}p_k^z(x,y)\quad\forall z\in\C,\enskip y\in\s.\end{align*}
For this reason, \eqref{eq:tabood} gives us the following bounds on $\mmag{p_n(x,A)-\pi(A)}$ that we will need to establish the convergence in total variation of Theorem~\ref{thrm:fostersgeo}:
\begin{align}\label{eq:mfwah9w4}\mmag{p_n(x,A)-\pi(A)}\leq&\enskip p_n^z(x,A)+\sum_{k=1}^{n-1}\mmag{\sum_{l=1}^kp^z_l(x,z)p_{k-l}(z,z)-\pi(z)}p^z_{n-k}(z,A)\\
&+\pi(z)\sum_{k=n}^\infty p_k^z(x,A)\quad\forall x\in\s,\enskip z\in\cal{C},\enskip A\subseteq \s,\enskip n>0,\nonumber\end{align}
where $p_n(x,A):=\sum_{y\in A}p(x,y)$.

\subsubsection*{A proof of Theorem~\ref{thrm:fostersgeo}}For this proof we require one final ingredient: the geometric analogue of Lemma~\ref{lem:Fnotrans} which tells us that we lose nothing by assuming that the finite set $F$ in~\eqref{eq:fostersgeo} does not contain any transient states.
\begin{lemma}\label{lem:Fnotransgeo}Let $F$ be a finite set such that $\Eb_\gamma[\theta^{\phi_F}]<\infty$ and $\Eb_x[\theta^{\phi_F}]<\infty$ for all $x\in\s$ and some $\theta>1$. There exists a $\vartheta>1$ such that, after removing all transient states from $F$,  $\Eb_\gamma[\vartheta^{\phi_{F}}]<\infty$ and  $\Eb_x[\vartheta^{\phi_{F}}]<\infty$ for all $x$ in $\s$.
\end{lemma}
\begin{proof}Replace Lemma~\ref{lem:geotrialarg}$(i)$ with Lemma~\ref{lem:geotrialarg}$(iii)$ in the proof of Lemma~\ref{lem:Fnotrans0}.
\end{proof}
Armed with the above, Kendall's Theorem, the geometric trials arguments of Section~\ref{sec:geometrictrial}, Foster's Theorem, and Theorem~\ref{thrm:fosgeorec}, the proof of the forward direction in Theorem~\ref{thrm:fostersgeo} reduces to careful bookkeeping and (numerous) applications Markov's inequality:

\begin{proof}[Proof of Theorem \ref{thrm:fostersgeo} (the forward direction)]Suppose that there exists $(v,F,b)$ satisfying \eqref{eq:fostersgeo} for some $\theta>1$ and that $\gamma(v)<\infty$. Theorem \ref{thrm:fosters} shows that the chain is positive Tweedie recurrent with finitely many closed communicating classes and that each class intersects with $F$ (for this last bit, see the Theorem \ref{thrm:fosters}'s proof). 
Theorem~\ref{thrm:fosgeorec} and Lemma~\ref{lem:Fnotransgeo} imply that there exists $\vartheta>1$ such that, after removing all transient states from $F$,
\begin{equation}\label{eq:nd7e8anf78eanfadasda}\Eb_\gamma[\vartheta^{\phi_F}]<\infty,\quad\Eb_x[\vartheta^{\phi_F}]<\infty\quad\forall x\in \s.\end{equation}

Let $\{\mathcal{C}_i:i\in\cal{I}\}$ the set of closed communicating classes $\cal{C}_i$ and $F_i:=\mathcal{C}_i\cap F$ be their (non-empty) intersections with $F$. Pick any $A\subseteq \s$ and note that
\begin{equation}\label{eq:kh8953nrnfiwuerwa3r0}\mmag{p_n(A)-\pi_{\gamma}(A)}\leq \Pbl{\{\phi_F> n, X_n\in A\}}+\mmag{\Pbl{\{\phi_F\leq n, X_n\in A\}}-\pi_\gamma(A)}.\end{equation}
Using Markov's inequality, we find that
\begin{equation}\label{eq:kh8953nrnfiwuerwa3r}\Pbl{\{\phi_F> n, X_n\in A\}}\leq \Pbl{\{\phi_F> n\}}\leq \Eb_\gamma[\vartheta^{\phi_F}]\vartheta^{-n}. \end{equation}
To deal with the second term in \eqref{eq:kh8953nrnfiwuerwa3r0}, note that
\begin{equation}\label{eq:fnmuy832nyfeas}\Pbl{\{\phi_F\leq n, X_n\in A\}}=\sum_{i\in\cal{I}}\Pbl{\{\phi_{F_i}\leq n, X_n\in A\}}\end{equation}
given that $F$ contains no transient states (hence, $F=\cup_{i\in\cal{I}}F_i$) and that
\begin{equation}\label{eq:fnea78fenh6a7fga7y}
\Pbl{\{\phi_{F_i}<\infty,\phi_{F_j}<\infty\}}=0\quad\forall i\neq j,
\end{equation}
as the classes are closed (Proposition~\ref{prop:closedis}). For this reason,
\begin{align}\mmag{\Pbl{\{\phi_F\leq n, X_n\in A\}}-\pi_\gamma(A)}\leq& \sum_{i\in\cal{I}}\mmag{\Pbl{\{\phi_{F_i}\leq n, X_n\in A\}}-\Pbl{\{\phi_{\C_i}<\infty\}}\pi_i(A)}\nonumber\\
\label{eq:dnws78adnw7a8dn7aw}\leq&  \sum_{i\in\cal{I}}\mmag{\Pbl{\{\phi_{F_i}\leq n, X_n\in A\}}-\Pbl{\{\phi_{F_i}\leq n\}}\pi_i(A)}\\
&+\sum_{i\in\cal{I}}(\Pbl{\{\phi_{\C_i}<\infty\}}-\Pbl{\{\phi_{F_i}\leq n\}})\pi_i(A).\nonumber\end{align}
Using \eqref{eq:fnea78fenh6a7fga7y}, Proposition~\ref{prop:closedis}, and Markov's inequality, we easily control the second sum:
\begin{align}
&\sum_{i\in\cal{I}}(\Pbl{\{\phi_{\C_i}<\infty\}}-\Pbl{\{\phi_{F_i}\leq n\}})\pi_i(A)\leq \sum_{i\in\cal{I}}(\Pbl{\{\phi_{\C_i}<\infty\}}-\Pbl{\{\phi_{F_i}\leq n\}})\label{eq:nf7w8eabf6ya8wba8fby8awfa}\\
&=\Pbl{\{\phi_{\cal{R}_+}<\infty\}}-\Pbl{\{\phi_{F}\leq n\}}=1-\Pbl{\{\phi_{F}\leq n\}}=\Pbl{\{\phi_{F}>n\}}\leq \frac{\Eb_\gamma[\vartheta^{\phi_F}]}{\vartheta^{n}},\nonumber
\end{align}
where $\cal{R}_+$ denotes the set of positive recurrent states and $\Pbl(\{\phi_{\cal{R}_+}<\infty\})=1$ because the chain is positive Tweedie recurrent.
To control the first sum in \eqref{eq:dnws78adnw7a8dn7aw}, apply the strong Markov property similarly as in~\eqref{eq:bfgyuwabf67aw} to obtain
\begin{equation}\label{eq:fnmuy832nyfeas2}\Pbl{\{\phi_{F_i}\leq n, X_n\in A\}}=\sum_{x\in F_i}\sum_{m=0}^n\Pb_\gamma(\{\phi_{F_i}=m,X_{\phi_{F_i}}=x\})p_{n-m}(x,A)\quad\forall i\in\cal{I},\enskip A\subseteq \s,
\end{equation}
and, consequently,
\begin{align}\label{eq:fnmuy832nyfeas2551}
&\sum_{i\in\cal{I}}\mmag{\Pbl{\{\phi_{F_i}\leq n, X_n\in A\}}-\Pbl{\{\phi_{F_i}\leq n\}}\pi_i(A)}\\
&\leq \sum_{i\in\cal{I}}\sum_{x\in F_i}\sum_{m=0}^n\Pb_\gamma(\{\phi_{F_i}=m,X_{\phi_{F_i}}=x\})\mmag{p_{n-m}(x,A)-\pi_i(A)},\quad\forall A\subseteq \s.\nonumber
\end{align}
To proceed, note that  $\Eb_x[\vartheta^{\phi_{F_i}}]=\Eb_x[\vartheta^{\phi_F}]<\infty$ for each $x$ in $F_i$ and $i$ in $\cal{I}$ as the classes are closed sets. Because any $x$ in $F_i$ is accessible from $F_i$, and because $F=\cup_{i\in\cal{I}}F_i$ is finite, Lemma \ref{lem:geotrialarg}$(iii)$ and Kendall's Theorem (Theorem~\ref{thrm:kendall}) imply that there exists a $\kappa\in(1,\vartheta]$ and $C_0\in[0,\infty)$ such that
$$\Eb_x[\kappa^{\phi_{x}}]<\infty,\quad \mmag{p_n(x,x)-\pi_i(x)}\leq C_0\kappa^{-n},\quad\forall x\in F_i,\enskip i\in\cal{I},\enskip n>0,$$
where $\pi_i$ denotes the ergodic distribution associated with $\cal{C}_i$. Setting $z:=x$ in \eqref{eq:mfwah9w4}, we find that
\begin{align}\label{eq:mfwah9w44}\mmag{p_n(x,A)-\pi_i(A)}\leq&\enskip p_n^x(x,A)+\sum_{k=1}^{n-1}\mmag{\sum_{l=1}^kp^x_l(x,x)p_{k-l}(x,x)-\pi_i(x)}p^x_{n-k}(x,A)\\
&+\pi_i(z)\sum_{k=n}^\infty p_k^x(x,A)\quad\forall x\in F_i,\enskip i\in\cal{I},\enskip A\subseteq \s,\enskip n>0.\nonumber\end{align}
Markov's inequality then yields bounds on the first two terms on the right-hand side of \eqref{eq:mfwah9w44}:
\begin{align}\label{eq:mfwah9w42}p_n^x(x,A)&=\Pbx{\{\phi_x\geq n,X_n\in A\}}\leq\Pbx{\{\phi_x\geq n\}}\leq \Eb_x[\kappa^{\phi_x}]\kappa^{-n},\\
\label{eq:mfwah9w43}\sum_{k=n}^\infty p_k^x(x,A)&\leq \sum_{k=n}^\infty \Pbx{\{\phi_x\geq k\}}\leq \frac{\Ebx{\kappa^{\phi_x}}}{\kappa-1}\kappa^{-n},\quad\forall x\in F,\enskip A\subseteq \s,\enskip n>0.\end{align}
For the middle term, we have that
\begin{align*}&\mmag{\sum_{l=1}^kp^x_l(x,x)p_{k-l}(x,x)-\pi_i(x)}\leq\sum_{l=1}^kp^x_l(x,x) \mmag{p_{k-l}(x,x)-\pi_i(x)}+\pi_i(x)\mmag{1-\sum_{l=1}^kp^x_l(x,x)}\\
&\leq C_0\kappa^{-k}\sum_{l=1}^kp_l^x(x,x)\kappa^l+\pi_i(x)\Pbx{\{\phi_x\geq k\}}\leq \Eb_x[\kappa^{\phi_x}](C_0+\pi_i(x))\kappa^{-k}\quad\forall x\in F_i,\enskip i\in\cal{I},\end{align*}
for all $x$ in $F_i$ and $i$ in $\cal{I}$. Applying Markov's inequality once again,
\begin{align*}&\sum_{k=1}^{n-1}\mmag{\sum_{l=1}^kp^x_l(x,x)p_{k-l}(x,x)-\pi_i(x)}p_{n-k}^x(x,A)\leq \Eb_x[\kappa^{\phi_x}](C_0+\pi_i(x))\sum_{k=1}^{n-1}\kappa^{-k}p_{n-k}^x(x,A)\\
&\leq\Eb_x[\kappa^{\phi_x}](C_0+\pi_i(x))\kappa^{-n}\sum_{k=1}^{n-1}\kappa^{k}\Pbx{\{\phi_x\geq k\}}\leq \Eb_x[\kappa^{\phi_x}]^2(C_0+\pi_i(x))n\kappa^{-n},\quad\forall x\in F_i,\enskip i\in\cal{I}.\end{align*}
Because $\kappa$ is no greater than $\vartheta$ and $F$ is finite, the above, \eqref{eq:nd7e8anf78eanfadasda}, and \eqref{eq:mfwah9w44}--\eqref{eq:mfwah9w43} show that
\begin{equation}\label{eq:mfwah9w4end}\mmag{p_{n}(x,A)-\pi_i(A)}\leq (C_1+C_2n)\kappa^{-n}\quad\forall x\in F_i,\enskip i\in\cal{I},\enskip A\subseteq\s,\enskip n\geq0,\end{equation}
for some  constants $C_1,C_2\in[0,\infty)$. Pick any $\iota\in(1,\kappa)$ and note that $n\mapsto n(\kappa/\iota)^{-n}$ is bounded over $\n$. Hence, the above implies that there exists a third constant $C_3\in[0,\infty)$ such that
$$\mmag{p_{n}(x,A)-\pi_i(A)}\leq C_3\iota^{-n}\quad\forall x\in F_i,\enskip i\in\cal{I},\enskip A\subseteq\s,\enskip n\geq0.$$
Plugging the above into \eqref{eq:fnmuy832nyfeas2551} and using \eqref{eq:fnea78fenh6a7fga7y}, we obtain
\begin{align}\label{eq:fnmuy832nyfeas255}
\sum_{i\in\cal{I}}|&\Pbl{\{\phi_{F_i}\leq n, X_n\in A\}}-\Pbl{\{\phi_{F_i}\leq n\}}\pi_i(A)|\\
&\leq C_3\iota^{-n}\sum_{i\in\cal{I}}\sum_{x\in F_i}\sum_{m=0}^n\Pb_\gamma(\{\phi_{F_i}=m,X_{\phi_{F_i}}=x\})\iota^{m}\nonumber\\
&=C_3\iota^{-n}\sum_{m=0}^n\Pb_\gamma(\{\phi_{F}=m\})\iota^{m}\leq \Ebl{\iota^{\phi_F}}C_3\iota^{-n}\quad\forall A\subseteq\s,\enskip n\geq0.\nonumber
\end{align}
Because $\iota$ is less than $\kappa$ and $\kappa$ is no greater than $\vartheta$, the above, \eqref{eq:nd7e8anf78eanfadasda}--\eqref{eq:kh8953nrnfiwuerwa3r}, and \eqref{eq:dnws78adnw7a8dn7aw}--\eqref{eq:nf7w8eabf6ya8wba8fby8awfa} imply that
$$\mmag{p_n(A)-\pi_\gamma(A)}\leq C_4\iota^{-n}\quad\forall A\subseteq\s,\enskip n\geq0.$$
for some constant $C_4\in[0,\infty)$.  Taking the supremum over $A\subseteq \s$, completes the proof.
\end{proof}

The proof of the reverse direction is nothing new:
\begin{proof}[Proof of Theorem \ref{thrm:fostersgeo} (the reverse direction)]Given Kendall's Theorem (Theorem~\ref{thrm:kendall}) and Theorem~\ref{thrm:fosgeorec}, this proof is analogous to that of the reverse direction of Foster's theorem (Theorem~\eqref{thrm:fosters}), with Lemma~\ref{lem:geotrialarg}$(iii)$ replacing Lemma~\ref{lem:geotrialarg}$(ii)$ therein.
\end{proof}

\subsubsection*{Notes and references} Theorem~\ref{thrm:fostersgeo} dates back to \citep{Popov1977} where both directions were proven for irreducible chains. The reverse direction in Theorem~\ref{thrm:fostersgeo} follows directly from that in the irreducible case by restricting the state space to the closed communicating classes. As for the forward direction,  it was shown in \citep{Nummelin1982} that, for positive Harris recurrent chains (and potentially-uncountable state space generalisations thereof) with stationary distribution $\pi$, the existence of $(v,F,b,\theta)$ satisfying Theorem~\ref{thrm:fostersgeo}'s premise  implies that
$$\norm{p_n-\pi}=\cal{O}(\kappa^{-n})$$
for some $\kappa>1$, whenever the chain starts a state $x$ such that $\pi(x)>0$. Note that, in the countable case, this also follows easily from the irreducible case by restricting the state space to the support of $\pi$. In \citep{Tweedie1983a}, the same result is given for all states $x$ (instead of only those in $\pi$'s support), however the proof given therein cites an unpublished draft of \citep{Nummelin1982} (while the published manuscript only considers starting states in the support of $\pi$). For initial distributions $\gamma$ satisfying $\gamma(v)<\infty$, the result was proven in \citep{Meyn1992}.

The first-entrance last-exit decomposition~\eqref{eq:feled} seems to date back to \citep{Nummelin1978a} and \citep{Nummelin1978b} where it was given for potentially-uncountable state space generalisations of discrete-time chains.

\newpage
\thispagestyle{premain}
\sectionmark{\MakeUppercase{Continuous-time chains I: the basics}}

\section*{Continuous-time chains I: the basics}\label{CTMCpre}
\addcontentsline{toc}{section}{\protect\numberline{}Continuous-time chains I: the basics}

%

\index{continuous-time (Markov) chain}Continuous-time Markov processes on countable state spaces are known as \emph{continuous-time Markov chains} or \emph{continuous-time chains} for short. They are very similar to discrete-time chains except that the amount of time elapsed between any two consecutive transitions, known as the \emph{waiting time}, is an exponentially distributed random variable whose mean is a function of the current state.

The aim of this chapter is to introduce the basic tools and concepts of continuous-time chain theory. Its development parallels that of its discrete-time counterpart (Sections~\ref{sec:dtdef}--\ref{sec:dtstrmarkov}) and the chapter overview given there largely applies here as well. The main differences are:
\begin{itemize}[leftmargin=*]
\item We need to account for the possibility of continuous-time chains \emph{exploding}: the waiting time between consecutive jumps drops fast enough that infinitely many jumps accumulate in a finite amount of time and the chain \emph{explodes}. As we will see in Section~\ref{sec:FKG}, the moment of explosion is the instant in time by which the chain has left every single finite subset of the state space, hence justifying the name given to this event.
\item We need to take a bit of time close to the beginning of the chapter (Section~\ref{sec:condind}) to prove some fundamental properties specific to continuous-time chains: conditioned on the chain's history, a) the waiting time until the next jump and the state it jumps to are independent, b) the next state visited depends only the current state, and c) the waiting time is exponentially distributed and its mean depends only on the current state.
\item Using these properties, we dispatch with the proofs of the Markov and strong Markov properties one after the other and early on (Section~\ref{sec:markovprop}). In this case, the machinery required for the latter is no heavier than that required for the former. It does, however, mean that we need to introduce the notion of stopping times earlier (Section~\ref{sec:filtct}) than we did for discrete-time chains. Moreover, because the possibility of explosions makes conditioning on the chain's location $X_t$ at time $t$ using conditional expectation
a bit of a pain (the chain may have exploded by time $t$!), we directly skip to the analogues of the ugly-but-useful versions of these properties given in Section~\ref{sec:dtstrmarkov} which seamlessly overcome this issue. Of course, for this we also need to introduce the path space and path law a bit earlier (Section~\ref{sec:pathspacect}) than we did before.
\item We skip the continuous-time martingale property for it does not buy us anything in this book that its discrete-time counterpart does not already do.
\item The question of whether the particular construction of the chains matters has a more complicated answer than in the discrete-time case and is left to the very end of the chapter (Section~\ref{sec:otherchains}).
\end{itemize}

\subsection[The Kendall-Gillespie algorithm and the chain's definition]{The Kendall-Gillespie (KG) algorithm and the chain's definition}
\label{sec:FKG}
For the type of continuous-time chains we will consider throughout the book (\emph{minimal chains with sample paths that are right-continuous in the discrete topology}), it doesn't actually matter how the chain is defined (Section~\ref{sec:otherchains}) and we might as well choose a definition that we find convenient. The aim of this section is to describe the construction we will use throughout the book involving an algorithm commonly used to simulate the chain in practice .

\subsubsection*{Rate matrices, jump probabilities, and jump rates}
The starting point of our here  is a totally stable\footnote{Not to be confused with the notions of stability previously discussed. These have nothing to do with each other: a chain with a rate matrix that is not totally stable is one that will leave certain states immediately after entering them, see \citep{Rogers2000a}. This, however, does not rule out the existence of stationary distributions, convergence of the averages, etc.} and conservative rate matrix $Q:=(q(x,y))_{x,y\in\cal{S}}$ indexed by a countable state space\glsadd{s}\index{state space} $\s$. That is, a matrix of real numbers indexed by $\cal{S}$ satisfying \glsadd{Q}\index{rate matrix}\index{totally stable}\index{conservative}\index{rate matrix}
\begin{equation}\label{eq:qmatrix}q(x,y)\geq0\quad\forall x\neq y,\qquad q(x,x)=-\sum_{y\neq x}q(x,y)<\infty\quad\forall x\in\cal{S}.\end{equation}
Throughout this book, we say ``a rate matrix $Q$'' as an abbreviation for ``a totally stable and conservative rate matrix $Q$'' and employ $q(x)$ as shorthand for $-q(x,x)$. 

Using the rate matrix, we define the \emph{jump matrix}\index{jump chain, times, rates, and matrix} $P=(p(x,y))_{x,y\in\s}$,\glsadd{Pct} 
\begin{equation}\label{eq:jumpmatrix}p(x,y):=\left\{\begin{array}{ll}(1-1_x(y))q(x,y)/q(x)&\text{if }q(x)>0,\\  1_x(y)&\text{otherwise},\end{array}\right.\quad\forall x,y\in\s,\end{equation}
and a vector $(\lambda(x))_{x\in\s}$ of \emph{jump rates},\glsadd{lambda}\index{jump chain, times, rates, and matrix}
\begin{equation}\label{eq:lambda}\lambda(x):=\left\{\begin{array}{ll}q(x)&\text{if }q(x)>0\\ 1&\text{if }q(x)=0\end{array}\right.\quad\forall x\in\s.\end{equation}

\begin{exercise}Convince yourself that $P$ in \eqref{eq:jumpmatrix} is a one-step matrix of discrete-time chain (i.e., that it satisfies \eqref{eq:1step}).
\end{exercise}

\subsubsection*{The technical set-up}To define the chain, we additionally require a measurable space\glsadd{of} $(\Omega,\cal{F})$, a collection $(\Pb_x)_{x\in\s}$\glsadd{Px} of probability measures on $(\Omega,\cal{F})$, an $\cal{F}/\tws$-measurable function $X_0$ from $\Omega$ to $\s$, a sequence of $\cal{F}/\cal{B}((0,1))$-measurable functions\glsadd{Unct} $U_1,U_2,\dots$ from $\Omega$ to $(0,1)$, and one of $\cal{F}/\cal{B}((0,1))$-measurable functions\glsadd{xi} $\xi_1,\xi_2,\dots$ from $\Omega$ to $(0,1)$ such that, under $\Pb_x$,
\begin{enumerate}
\item $\{X_0=x\}$ occurs almost surely: $\Pbx{\{X_0=x\}}=1$;
\item for each $n>0$, the random variable $U_n$ is uniformly distributed on $(0,1)$: $\Pbx{\{U_n\leq u\}}=u$ for all $u\in(0,1)$);
\item for each $n>0$, the random variable $\xi_n$ is exponentially distributed on $(0,1)$ with mean one: $\Pbx{\{\xi_n> s\}}=e^{-s}$ for all $s\in(0,\infty)$;
\item and the random variables $X_0,U_1,U_2,\dots,\xi_1,\xi_2,\dots$ are independent:
\begin{align*}&\Pbx{\{X_0=y,U_1\leq u_1,\dots,U_n\leq u_n,\xi_1>s_1,\dots,\xi_m>s_m\}}\\
&=\Pbx{\{X_0=y\}}\Pbl{\{U_1\leq u_1\}}\dots\Pbl{\{U_n\leq u_n\}}\Pbl{\{\xi_1>s_1\}}\dots\Pbl{\{\xi_m>s_m\}}\end{align*}
for all $y\in\s$, $u_1,\dots,u_n\in(0,1)$, $s_1,\dots,s_m\in(0,\infty)$, and $n,m>0$;
\end{enumerate}
for each $x$ in $\s$. Technically, we can construct this space, measures, and random variables using Theorem~\ref{seqindp} similarly as we did in Section~\ref{sec:dtdef} for the discrete-time case.

Below, we will set the chain's starting location to be $X_0$. For this reason, $\Pb_x$ will describe the statistics of the chain were it to start from $x$. To instead describe situations in which the chain's starting location was sampled from a given probability distribution $\gamma=(\gamma(x))_{x\in\s}$\glsadd{gamma} on $\s$ (the \emph{initial distribution}), we define the probability measure $\Pb_\gamma$\glsadd{Pl} on $(\Omega,\cal{F})$ via
\begin{equation}\label{eq:pl}\Pb_\gamma(A):=\sum_{x\in\cal{S}}\gamma(x)\Pb_x(A)\quad\forall A\in\cal{F}.\end{equation}
Analogous arguments to those given after \eqref{eq:pld} show that $X_0$ has law $\gamma$ under $\Pb_\gamma$ and that 2--4 of the above list holds if we replace $\Pb_x$ with $\Pb_\gamma$. As in the discrete-time case, we use\glsadd{Ex} $\Eb_x$ (resp.\glsadd{El} $\Eb_\gamma$) to denote expectation with respect to $\Pb_x$ (resp. $\Pb_\gamma$).

\subsubsection*{The Kendall-Gillespie algorithm and the chain's definition}We construct our chain recursively by running the \emph{Kendall-Gillespie algorithm}\index{Kendall-Gillespie algorithm} given below as Algorithm~\ref{gilalg}. It goes as follows: sample  a state $x$ from the initial distribution $\gamma$ and start the chain at $x$. Wait an exponentially distributed amount of time with mean $1/\lambda(x)$, where $\lambda(x)$ is as in \eqref{eq:lambda}; sample $y$ from the probability distribution $p(x,\cdot)$ in \eqref{eq:jumpmatrix}; and update the chain's state to $y$ (we say that the chain \emph{jumps} from $x$ to $y$, we call the time waited the \emph{waiting time} and the instant at which it jumps the \emph{jump time}). Repeat these steps starting from $y$ instead of $x$. All random variables sampled must be independent of each other.

\begin{algorithm}[h]
\begin{algorithmic}[1]
\STATE{$Y_0:=X_0$}
\FOR{$n=1,2,\dots$}
\STATE{sample $U_{n}\sim\operatorname{uni}_{(0,1)}$ independently of $\{X_0,\xi_1,\dots,\xi_{n-1},U_1,\dots,U_{n-1}\}$}
\STATE{sample $\xi_{n}\sim\operatorname{exp}(1)$ independently of $\{X_0,\xi_1,\dots,\xi_{n-1},U_1,\dots,U_{n}\}$}
\STATE{$i:=0$}
\WHILE{$U_n>\sum_{j=0}^ip(Y_{n-1},x_i)$}
\STATE{$i:=i+1$}
\ENDWHILE
\STATE{$Y_n:=x_i$}
\STATE{$S_n:=\xi_n/\lambda(Y_{n-1})$}
\ENDFOR
 \end{algorithmic}
 \caption{The Kendall-Gillespie Algorithm on $\s=\{x_0,x_1,x_2,\dots\}$}\label{gilalg}
\end{algorithm}

Technically, the algorithm constructs the \emph{waiting times}\glsadd{Sn}\index{waiting times} $S_1,S_2,\dots$ and the \emph{jump chain}\index{jump chain, times, rates, and matrix}\glsadd{Y} (or \emph{embedded chain}\index{embedded chain|seealso{jump chain}}) $Y:=(Y_n)_{n\in\n}$ and the chain is defined to be a piecewise constant interpolation of the jump chain:\index{jump chain, times, rates, and matrix}\glsadd{Tn}\glsadd{Xt} 
\begin{equation}\label{eq:cpathdef}X_t(\omega):=Y_n(\omega)\quad\forall t\in [ T_n(\omega),T_{n+1}(\omega)),\enskip \omega\in\Omega,\enskip n\geq0,\end{equation}
where $T_n$ denotes the $n$th \emph{jump time}:
\begin{equation}\label{eq:cpathdef2}T_0(\omega):=0\quad\forall\omega\in\Omega,\quad T_n(\omega):=\sum_{m=1}^nS_m(\omega)\enskip\forall\omega\in\Omega,\enskip n>0.\end{equation}
%

\subsubsection*{Explosions and regular rate matrices}The chain is defined only up until the \emph{explosion time} (also known as the \emph{escape time} and the \emph{time of the first infinity}),\index{explosion, explosion time}\glsadd{Tinf}
\begin{equation}\label{eq:cpathdef3}T_{\infty}(\omega):=\lim_{n\to \infty}T_n(\omega)=\sum_{m=1}^\infty S_m(\omega)\quad\forall \omega\in\Omega,\end{equation}
%
In other words, $X_t$ is defined only on $\{t<T_\infty\}$ and not on all of $\Omega$. 

To understand the meaning behind $T_\infty$, let $\s_1\subseteq \s_2\subseteq\dots$ be any increasing sequence of finite subsets (or \emph{truncations}) of $\s$ that approach $\s$ (i.e., $\cup_{r=1}^\infty\s_r=\s$) and, for each $r>0$, let  $\tau_r$ be the time that the chain $X$ first leaves $\s_r$:
\begin{equation}\label{eq:taurexit}\tau_r(\omega):=\inf\{t\in[0,T_\infty(\omega)):X_t(\omega)\not\in\s_r\}\quad\forall \omega\in\Omega.\end{equation}
%
Because\glsadd{taur} our truncations are an increasing,  $(\tau_r)_{r\in\zp}$ is an increasing sequence of random variables. Therefore, the limit $\lim_{r\to\infty}\tau_r$ exists for everywhere. Moreover, this limit is $T_\infty$:
\begin{theorem}\label{tautin} If $(\s_r)_{r\in\zp}$ is an increasing sequence of finite sets such that $\cup_{r=1}^\infty\s_r=\s$, then $\tau_r$ tends to $T_\infty$, $\Pb_\gamma$-almost surely, for all initial distributions $\gamma$.
\end{theorem}
The limiting random variable $\lim_{r\to\infty}\tau_r$ is the point in time by which the chain has left all of the truncations in the sequence $(\s_r)_{r\in\zp}$. Because the truncations approach $\s$, every finite subset of the state space is contained in all truncations sufficiently far along in the sequence. Thus, the limit $\lim_{r\to\infty}\tau_r$ is the instant in time that the chain has left all finite subsets of the state space.  The above tells us that this limit is (almost surely) equal to $T_\infty$ regardless of the particular sequence of finite truncations $(\s_r)_{r\in\zp}$ in its definition. For these reasons,  we interpret $T_\infty$ as the point in time that the chain leaves the state space, or, in other words, \emph{explodes} (see Figure~\ref{fig:cartoon} Part~\ref{part:theory} introduction for an example). 
The proof of Theorem~\ref{tautin} builds on the following lemma and can be found at the end of the section.

\begin{lemma}\label{lem:tinfsum}The random variables  $T_\infty$ and $\sum_{n=0}^\infty1/\lambda(Y_n)$ are either both finite or both infinite, with $\Pb_\gamma$-probability one, for any initial distribution $\gamma$.
\end{lemma}
\begin{proof}See the end of Section~\ref{sec:condind}.\end{proof}

\subsubsection*{An example of  explosive chain} Consider a chain taking values on $\n$ that moves one state up every jump with rate $\lambda(x):=2^{x}$ for all $x$ in $\n$. That is, a chain with rate matrix
$$Q=\begin{bmatrix}-1&1&0&0&0&\dots\\
0&-2&2&0&0&\dots\\
0&0&-4&4&0&\dots\\
0&0&0&-8&8&\ddots\\
\vdots&\vdots&\vdots&\vdots&\ddots&\ddots\end{bmatrix}.$$
In this case,
$$\sum_{n=0}^\infty\frac{1}{\lambda(Y_n)}=\sum_{n=x}^\infty\frac{1}{2^n}=\frac{2}{2^{x}}\quad\Pb_x\text{-almost surely},\enskip\forall x\in\n.$$
For this reason, \eqref{eq:pl} and Lemma~\ref{lem:tinfsum}  implies that $T_\infty$ is finite, $\Pb_\gamma$ almost surely, for all initial distribution $\gamma$.

\subsubsection*{Absorbing states and fictitious jumps}

Our definition of the jump matrix in~\eqref{eq:jumpmatrix} implies that any state $x$ satisfying $q(x)=0$ is \emph{absorbing} in the sense that once inside such an $x$, the chain may never leave:
%
\begin{proposition}\label{prop:absorb}For any $x$ in $\s$, 
$$\Pbx{\{T_\infty=\infty,\enskip X_t=x,\enskip\forall t\in[0,\infty)\}}=\left\{\begin{array}{ll}1&\text{if}\enskip q(x)=0\\0&\text{if}\enskip q(x)\neq0\end{array}\right..$$
\end{proposition}
\begin{proof}See the end of Section~\ref{sec:pathspacect}.\end{proof}
Note that Algorithm~\ref{gilalg} keeps the chain `jumping'  in an absorbing state $x$. This is nothing more than a notational trick that frees us from having to keep track of the number of jumps until the chain gets absorbed: the jumps are fictitious in the sense that they take the chain from $x$ to $x$ and nothing changes.


\subsubsection*{A proof of Theorem~\ref{tautin}}
Let $\tau_\infty:=\lim_{r\to\infty}\tau_r$ be the limit of the random times. First, note that $\tau_r$'s definition in~\eqref{eq:taurexit} implies  that 
\begin{equation}\label{eq:ngf7e8a9bg8erhwag}\tau_r(\omega)\in [0, T_\infty(\omega))\cup\{\infty\}\quad\forall \omega\in\Omega\end{equation}
and that $X$ lies outside $\s_r$ at the time $\tau_r$ if $\tau_r$ is finite or, equivalently, less than $T_\infty$ (this follows from the right-continuity of the chain's paths, e.g., see Proposition~\ref{prop:hitdc} later on). Because the truncations are increasing, it follows that ${\{\tau_l<T_\infty,X_{\tau_l}\in\s_r\}}$ is empty if $l\geq r$. For this reason, bounded convergence implies that
$$\Pbl{\{\tau_\infty<T_\infty,X_{\tau_\infty}\in\s_r\}}=\lim_{l\to\infty}\Pbl{\{\tau_l<T_\infty,X_{\tau_l}\in\s_r\}}=0\quad\forall r>0.$$
In turn, monotone convergence shows that
\begin{equation}\label{eq:mismatch2}\Pbl{\{\tau_\infty<T_\infty\}}=\Pbl{\{\tau_\infty<T_\infty,X_{\tau_\infty}\in\s\}}=\lim_{r\to\infty}\Pbl{\{\tau_\infty<T_\infty,X_{\tau_\infty}\in\s_r\}}=0.\end{equation}

Next, suppose that  $\omega$ is such that  $T_\infty(\omega)$ is finite but $\tau_\infty(\omega)$ is infinite. It follows from \eqref{eq:ngf7e8a9bg8erhwag} that there exists a positive integer $R$ such that $\tau_{R}=\infty$. In other words, $X_t(\omega)$ belongs to $\s_{R}$ for all $t$ in $[0,T_\infty(\omega))$. In particular,
$$Y_n(\omega)=X_{T_n(\omega)}(\omega)\in \s_{R}\quad\forall n\geq0.$$
Finiteness of $\s_{R}$ then implies that
$$\sum_{n=0}^\infty \frac{1}{\lambda(Y_{n}(\omega))}\geq \left(\inf_{x\in\s_{R(\omega)}}\frac{1}{\lambda(x)}\right)\left(\sum_{n=0}^\infty1\right)=\infty.$$
Because Lemma~\ref{lem:tinfsum} shows that the event that $\sum_{n=0}^\infty \frac{1}{\lambda(Y_{n})}$ is infinite and $T_\infty$ is finite has zero probability of occurring, it follows from the above that
\begin{equation}\label{eq:mismatch1}\Pbl{\{T_\infty<\tau_\infty=\infty\}}=0.\end{equation}

Lastly, given that $\tau_r$ is bounded above by $\tau_\infty$, if $\omega$ is such that $\tau_\infty(\omega)$ is finite, then $\tau_r(\omega)$ is finite. In this case, \eqref{eq:ngf7e8a9bg8erhwag} implies that 
$$\tau_r(\omega)<T_\infty(\omega)\quad\forall r>0.$$
Taking the limit $r\to\infty$ shows that $\tau_\infty(\omega)$ is at most $T_\infty(\omega)$. For this reason, the event  $\{T_\infty<\tau_\infty<\infty\}$ is empty and, thus,
\begin{equation}\label{eq:mismatch3}\Pbl{\{T_\infty<\tau_\infty<\infty\}}=0.\end{equation}
Putting \eqref{eq:mismatch2}--\eqref{eq:mismatch3} together completes the proof.

\subsubsection*{Notes and references}In the physical and life sciences, Algorithm \ref{gilalg} is typically referred to as the \emph{Stochastic Simulation Algorithm} or the \emph{Gillespie Algorithm} after D.~T.~Gillespie who introduced it to chemical physics literature \citep{Gillespie1976} and argued \citep{Gillespie1977} in favour of its use in the modelling of reacting chemical species. The ideas underlying the algorithm (Properties~1--3 in the ensuing section) seem to have been first delineated in the `40s by W.~Feller~\citep{Feller1940} and J.~L.~Doob~\citep{Doob1942,Doob1945}. The first published implementation of the algorithm was carried out in 1950 by D.~G. Kendall~\citep{Kendall1950} (see also his student F.~G.~Foster in the discussion of \citep{Kendall1951}).

In applications, the possibility of an explosion is often regarded as a pathological technicality  with no relevance to real life. You may have encountered before an argument of the type ``this real life phenomena clearly cannot explode, hence neither can its model''. Even if something seems reasonable based on physical arguments, models are abstractions of real life and, unfortunately, can be far from faithful representations of it. This is particularly prominent in the context of modelling where the model is not yet fitted. Indeed, many fitting algorithms involve automated sweeps of parameter sets, and the algorithms can easily step through parameter values that lead to poor models (e.g., an exploding model of a phenomenon impossible of exhibiting this type of behaviour).

Theorem~\ref{tautin} as given here seems to have been first shown in \citep{Kuntzthe}, although I would take this with a grain of salt. Lemma~\ref{lem:tinfsum} which makes up the bulk of the theorem's proof goes back to \citep{Chung1967}, if not earlier.

\subsection{Three important properties of the jump chain and waiting times}\label{sec:condind}
\subsectionmark{Properties of the jump chain and waiting times}
The main aim of this section is to prove three important facts regarding the chain defined in previous section. Namely, when conditioned on the chain's history up until (and including) the $n$th jump time $T_n$,
\begin{enumerate}
\item the next state $Y_{n+1}$ visited by the chain  has distribution $(p(Y_n,y))_{y\in\s}$, where $Y_n$ denotes its current state;
\item the amount of time $S_{n+1}$ elapsed until the next jump is exponentially distributed with mean $1/\lambda(Y_n)$; and
\item $Y_{n+1}$ and $S_{n+1}$ are independent of each other.
\end{enumerate}
The first property implies that $Y$ is a discrete-time chain and, hence, justifies the names of \emph{jump chain} and \emph{embedded chain} afforded to it.

These properties are not an artifice of our definitions in Section~\ref{sec:FKG}: all $\s$-valued continuous-time processes  satisfying the Markov property with right-continuous sample paths that explode at most once have them (Section~\ref{sec:otherchains}). Furthermore, these properties, together with $\lambda$ and $P$, fully determine the law induced by the chain on the space of paths (Section~\ref{sec:pathspacect}). In other words,  the paths of any continuous-time chain with right-continuous paths that explode at most once, jump matrix $P$, and jump rates $\lambda$ are statistically indistinguishable from those of the chain in~\eqref{eq:cpathdef}--\eqref{eq:cpathdef2}. This is the reason why, as mentioned in Section~\ref{sec:FKG}, we lose nothing by focusing on the particular algorithmic construction of the chain introduced there.

\subsubsection*{The filtration generated by the jump chain and jump times}
To  argue 1--3 above, we must first introduce the filtration $(\cal{G}_n)_{n\in\n}$ generated by the jump chain and jump times. It describes the information that can be gleaned by observing the chain's progress from one jump time to another. In particular, for each $n$, $\cal{G}_n$ consists of all events whose occurrence can be determined by observing the chain's path up until (and including) $T_n$ and, so, $\cal{G}_n$ formalises the idea of the chain's history up until $T_n$. In full:\index{filtration generated by the jump chain and jump times}
\begin{definition}[The filtration generated by the jump chain and jump times]\label{def:filt2}Let $(\cal{G}_n)_{n\in\n}$ be the filtration defined by\glsadd{filtrationdtct}
$$\cal{G}_n:=\text{ the sigma-algebra generated by }(Y_0,Y_1,\dots,Y_n,T_0,T_1,\dots,T_n)\quad\forall n\geq0,$$
meaning that $\cal{G}_n$ is the smallest sigma-algebra  such that $Y_0,Y_1,\dots,Y_n,T_1,\dots,T_n$ are $\cal{G}_n$-measurable random variables. 
\end{definition}

Because we know the times $T_1,\dots,T_n$ at which the first $n$ jumps occur if and only if we know the amount of time $S_1$ elapsed  until the first jump and the amounts $S_2,\dots,S_n$ elapsed between the next $n-1$ jumps, 
%
%
%
\begin{equation}\label{eq:nfa789oh3872oaabrA}\cal{G}_n=\text{ the sigma-algebra generated by }(Y_0,Y_1,\dots,Y_n,S_1,\dots,S_n)\quad\forall n>0.\end{equation}

\begin{exercise}\label{ex:fme78ahfb67ebf68aybfewa} Verify that \eqref{eq:nfa789oh3872oaabrA} holds.
\end{exercise}

\subsubsection*{A proof of Properties 1--3}Using the filtration, we formalise Properties 1--3 as
\begin{enumerate}
\item $\Pbl{\{Y_{n+1}=y,S_{n+1}>s\}|\cal{G}_n}=\Pbl{\{Y_{n+1}=y\}|\cal{G}_n}\Pbl{\{S_{n+1}>s\}|\cal{G}_n}$, $\Pb_\gamma$-almost surely, for all states $y$, non-negative numbers $s$, and initial distributions $\gamma$;
\item $\Pbl{\{Y_{n+1}=y\}|\cal{G}_n}=p(Y_n,y)$, $\Pb_\gamma$-almost surely, for all $y,s,\gamma$; and
\item $\Pbl{\{S_{n+1}>s\}|\cal{G}_n}=e^{-\lambda(Y_n)s}$, $\Pb_\gamma$-almost surely, for all $y,s,\gamma$.
\end{enumerate}
We prove a slight generalisation of the above:
\begin{theorem}\label{thrm:condind}For any given natural number $n$, let $W$ be a $\cal{G}_{n}/\cal{B}(\r)$-measurable random variable, $f$ a bounded real-valued function on $\s$, and $g$ a bounded $\cal{B}(\r^2)/\cal{B}(\r)$-measurable function. For any initial distribution $\gamma$,
%
\begin{align*}&\Ebl{f(Y_{n+1})g(S_{n+1},W)|\cal{G}_{n}}=\Ebl{f(Y_{n+1})|\cal{G}_{n}}\Ebl{g(S_{n+1},W)|\cal{G}_{n}},\quad\Pb_\gamma\text{-almost surely};\\
&\Ebl{f(Y_{n+1})|\cal{G}_{n}}=Pf(Y_{n}),\quad\Pb_\gamma\text{-almost surely};\\
&\Ebl{g(S_{n+1},W)|\cal{G}_{n}}=\int_0^\infty g(s,W)\lambda(Y_{n})e^{-\lambda(Y_{n})s}ds,\quad\Pb_\gamma\text{-almost surely};\end{align*}
where $P=(p(x,y))_{x,y\in\s}$ denotes the jump matrix~\eqref{eq:jumpmatrix}, $\lambda=(\lambda(x))_{x\in\s}$ the jump rates~\eqref{eq:lambda}, and 
$$Pf(x):=\sum_{x\in\s}p(x,y)f(y)\quad\forall x\in\s.$$

\end{theorem}

\begin{proof}
The definitions of $Y_{n+1}$ and $S_{n+1}$ in the Kendall-Gillespie Algorithm (Algorithm~\ref{gilalg} in Section~\ref{sec:FKG}) imply that
$$f(Y_{n+1})=\tilde{f}(U_{n+1},Y_{n}),\qquad g(S_{n+1},W)=g(\xi_{n+1}/\lambda(Y_{n}),W),$$
where 
$$\tilde{f}(u,x):=\sum_{j=0}^\infty f(x_j)1_{\left\{\sum_{k=0}^{j-1}p(x,x_k)\leq u<\sum_{k=0}^jp(x,x_k)\right\}}\quad\forall x\in\s,\enskip u\in(0,1),$$
and we are using the enumeration of the state space introduced in the algorithm.  By the definitions of $U_{n+1}$, $\xi_{m+1}$, and $\cal{G}_n$, the first two are independent of the third. For this reason, the same arguments as those given in \citep[Section~9.10]{Williams1991} show that 
$$\Eb_\gamma[f(Y_{n+1})g(S_{n+1},W)|\cal{G}_{n}]=\Eb_\gamma[\tilde{f}(U_{n+1},Y_{n})g(\xi_{n+1}/\lambda(Y_{n}),W)|\cal{G}_{n}]=h(Y_n,W)\enskip\Pb_\gamma\text{-almost surely},$$
where
$$h(x,w)=\Eb_\gamma[\tilde{f}(U_{n+1},x)g(\xi_{n+1}/\lambda(x),w)]\quad\forall x\in\s,\enskip w\in\r.$$
Because $U_{n+1}$ and $\xi_{n+1}$ are independent, 
\begin{align*}h(x,w)=h_1(x)h_2(x,w)\quad\text{where}\quad h_1(x):=\Eb_\gamma[\tilde{f}(U_{n+1},x)],\quad h_2(x,w):=\Ebl{g(\xi_{n+1}/\lambda(x),w)},\end{align*}
for all $x$ in $\s$ and $w$ in $\r$. Similarly, the arguments in \citep[Section~9.10]{Williams1991}  imply that
$$h_1(Y_n)=\Eb_\gamma[\tilde{f}(U_{n+1},Y_{n})|\cal{G}_{n}],  \quad h_2(Y_n,W)=\Eb_\gamma[g(\xi_{n+1}/\lambda(Y_{n}),W)|\cal{G}_{n}],\quad \Pb_\gamma\text{-almost surely}.$$
The result then follows from $U_{n+1}$'s uniform distribution, $\xi_{n+1}$'s  unit-mean exponential distribution, and Fubini's theorem:
\begin{align*}h_1(x)&=\sum_{j=0}^\infty f(x_j)\Pbl{\left\{\sum_{k=0}^{j-1}p(x,x_k)\leq u<\sum_{k=0}^jp(x,x_k)\right\}}=\sum_{j=0}^\infty f(x_j)p(x,x_j)=Pf(x),\\ %
h_2(x,w)&=\int_0^\infty g(t/\lambda(x),w)e^{-t}dt=\int_0^\infty g(s,w)\lambda(x)e^{-\lambda(x)s}ds\quad \forall x\in\s,\enskip w\in\r,\end{align*}
where we have used the change of variable $s:=t/\lambda(x)$ in the last equation.
\end{proof}

\subsubsection*{Paying our dues: a proof of Lemma~\ref{lem:tinfsum}}

This proof relies on the following consequence of the discrete-time martingale convergence theorem:
\begin{theorem}\emph{(Levy's extension of the Borel-Cantelli lemmas, \citep[Section~12.15]{Williams1991}).}\label{thrm:bocan} Suppose that $W=(W_n)_{n\in\n}$ is a sequence of non-negative random variables 
bounded above by one (i.e., $W_n\leq 1$ for all $n\geq0$) and adapted to some filtration $(\cal{F}_n)_{n\in\n}$ 
(i.e., $W_n$ is $\cal{F}_n/\cal{B}(\r)$-measurable for all $n\geq0$). With probability one, the sums
$$\sum_{n=1}^\infty W_n\quad\text{and}\quad \sum_{n=1}^\infty \Ebb{W_n|\cal{F}_{n-1}},$$
are both finite or both infinite.
\end{theorem}

\begin{exercise}Prove \eqref{thrm:bocan} by generalising the proof given in \citep[Section~12.15]{Williams1991} for the special case $W_n=1_{E_n}$ with $E_n\in\cal{F}_n$.
\end{exercise}

We are now in a great position to prove Lemma~\ref{lem:tinfsum}:
\begin{proof}[Proof of Lemma~\ref{lem:tinfsum}] It is not difficult to check that 
$$T_\infty=\lim_{m\to\infty}T_m=\sum_{n=1}^\infty S_n$$
is finite if and only if 
$$\sum_{n=1}^\infty S_{n}\wedge 1$$
is finite. Theorem~\ref{thrm:bocan} implies that,  with $\Pb_\gamma$-probability one, the above sum is finite if and only if
$$\sum_{n=1}^\infty \Ebl{S_{n}\wedge 1|\cal{G}_{n-1}}$$
is finite. Because $S_n$ is exponentially distributed with mean $1/\lambda(Y_{n-1})$ (Theorem~\ref{thrm:condind}),
\begin{align*}\Ebl{S_{n}\wedge 1|\cal{G}_{n-1}}&=\int_0^1 s\lambda(Y_{n-1})e^{-\lambda(Y_{n-1})s}ds+\int_1^\infty \lambda(Y_{n-1})e^{-\lambda(Y_{n-1})s}ds\\
&=\frac{1}{\lambda(Y_{n-1})}\left(1-e^{-\lambda(Y_{n-1})}\right),\quad\Pb_\gamma\text{-almost surely},\quad\forall n>0.\end{align*} 
For this reason,
\begin{equation}
\label{eq:ndf7e8a9f3n7ab3yaf}\sum_{n=1}^\infty \Ebl{S_{n}\wedge 1|\cal{G}_{n-1}}=\sum_{n=0}^\infty \frac{1}{\lambda(Y_{n})}\left(1-e^{-\lambda(Y_{n})}\right)\leq \sum_{n=0}^\infty \frac{1}{\lambda(Y_{n})}\quad\Pb_\gamma\text{-almost surely},\end{equation}
and to complete the proof we need only to show that the rightmost sum is finite whenever the middle one is. The one-sided limit comparison test implies that, for any sequence of positive numbers $(\alpha_n)_{n\in\n}$, the sum $\sum_{n=0}^\infty \frac{1}{\alpha_n}$ is finite if $\sum_{n=0}^\infty \frac{1}{\alpha_n}(1-e^{-\alpha_n})$ is finite unless
$$\limsup_{n\to\infty}\frac{1}{1-e^{-\alpha_n}}=\infty,$$
or, equivalently, there exists a subsequence $(\alpha_{n_k})_{k\in\n}$ such that
$$\lim_{k\to\infty}\alpha_{n_k}=0.$$
However, this is impossible if $\sum_{n=0}^\infty \frac{1}{\alpha_n}(1-e^{-\alpha_n})$ is finite as L'H\^opital's rule would imply that
$$\lim_{k\to\infty} \frac{1}{\alpha_{n_k}}(1-e^{-\alpha_{n_k}})=\lim_{k\to\infty} e^{-\alpha_{n_k}}=1$$
and $\sum_{n=0}^\infty \frac{1}{\alpha_n}(1-e^{-\alpha_n})$ would be infinite.

\end{proof}

\subsubsection*{Notes and references} The proof of Lemma~\ref{lem:tinfsum} given above is a slightly more annotated version of that given for Theorem~1 in \citep[p.237]{Chung1967}.

\subsection{The filtration generated by the chain and stopping times}\label{sec:filtct}
%
%
The filtration $(\cal{F}_t)_{t\geq0}$  generated by the chain describes the information that may be gleaned from observing the chain's progress. For any point in time $t\geq0$, $\cal{F}_t$ captures the chain's history up until (and including) $t$: an event $A$ belongs to $\cal{F}_t$  if and only if we are able to determine whether or not it has occurred by time $t$ from observing the chain's position up until (and including) $t$. Formally:\glsadd{filtrationct}\index{filtration generated by the chain}
\begin{definition}[The filtration generated by the chain]\label{def:filt3}Let $(\cal{F}_t)_{t\geq0}$ be the filtration defined by
$$\cal{F}_t:=\text{ the sigma-algebra generated by }\{T_n\leq s, Y_n=x\}\quad\forall s\in[0,t],\enskip x\in\s,\enskip n\in\n,$$
meaning the smallest sigma-algebra containing $\{T_n\leq s, Y_n=x\}$ for all $s\leq t$, $x\in\s$, and $n\in\n$.
\end{definition}
In other texts, you may find $(\cal{F}_t)_{t\geq0}$ defined slightly differently:
\begin{exercise}\label{ex:filtrd}Show that, for any $t\in[0,\infty)$, the sets
%
$$\{X_{s}=x,s<T_\infty\}\quad\forall s\leq t,\enskip x\in\s$$
%
generate $\cal{F}_t$ in \eqref{def:filt3}. Hint: \eqref{eq:cpathdef}--\eqref{eq:cpathdef2} implies that 
%
$$T_{n+1}(\omega)=\inf\{r\in\mathbb{Q}\cap (T_n(\omega),\infty):X_{r}(\omega)\neq X_{T_n(\omega)}(\omega)\}\quad\forall \omega\in\Omega,\enskip n\geq0,$$
\end{exercise}

\subsubsection*{Stopping times}A non-negative random variable $\eta$ is an $(\cal{F}_t)_{t\geq0}$-stopping time if and only if it is the time that some event occurs and we are able to determine whether the event has occurred by observing the chain up until that instant. For example, $\eta$ can be the first (or second or $k$th) time that the chain visits some subset of the state space of interest, but not the last time that it enters the subset. Formally:\glsadd{eta}\glsadd{Feta}\index{stopping time}\index{pre-$\eta$ sigma-algebra}
\begin{definition}[Stopping times]\label{def:stopct}Given the filtration $(\cal{F}_t)_{t\geq0}$ generated by the chain (c.f. \eqref{def:filt3}), a random variable $\eta:\Omega\to[0,\infty]$ is said to be a $(\cal{F}_t)_{t\geq0}$-stopping time if, for each $t\geq0$, the event $\{\eta\leq t\}$ belongs to $\cal{F}_t$.
%
%
With the stopping time we associated a sigma-algebra $\cal{F}_\eta$ defined by
$$\forall A\in\cal{F},\quad A\in\cal{F}_\eta\Leftrightarrow A\cap\{\eta\leq t\}\in\cal{F}_t \quad \forall t\geq0.$$
\end{definition}

\begin{exercise}Convince yourself that $\cal{F}_\eta$ is a sigma-algebra.
\end{exercise}
The reasons behind the name `stopping time' afforded to these random variables are the same as in the discrete-time case (Section~\ref{sec:stop}). Similarly, the \emph{pre-$\eta$ sigma-algebra} $\cal{F}_\eta$ is the collection of events whose occurrence we are able to glean by tracking the chain's position up until the stopping time $\eta$. For instance, if $\eta$ is the second time that the chain visits a given state $x$, then the event that the chain visits $x$ once (or at least once, or twice, or at least twice) belongs to $\cal{F}_\eta$  but the event that the chain visits $x$ three (or four, or ...) does not belong to $\cal{F}_\eta$.  

The following lemma is very useful. It formalises three simple facts: $(i)$ from observing the chain up until a stopping time $\vartheta$, we can deduce whether or not the event associated with another stopping $\eta$ occurs no later  than (before than, or at) $\vartheta$; $(ii)$ if we can deduce whether an event $A$ occurs from observing the chain up until $\eta$, we can deduce whether $A$ and the event associated with $\eta$  both occur no later than $\vartheta$ from observing the chain up until $\vartheta$; and $(iii)$ if $\eta$ occurs no later than $\vartheta$ and we can deduce whether an event $A$ occurs by observing the chain up until $\eta$, then we can deduce whether $A$ occurs from observing the chain up until time $\vartheta$.
\begin{lemma}\label{lem:etatheta}If $\eta$ and $\vartheta$ are $(\cal{F}_t)_{t\geq0}$-stopping times, then 
\begin{enumerate}[label=(\roman*),noitemsep]
\item The events $\{\eta\leq \vartheta\}$, $\{\eta< \vartheta\}$, and $\{\eta= \vartheta\}$ belong to $\cal{F}_\vartheta$.
\item Given any event $A$ in $\cal{F}_\eta$, the event $A\cap\{\eta\leq \vartheta\}$ belongs to $\cal{F}_\vartheta$.
\item In particular, if $\eta\leq\vartheta$, then $\cal{F}_\eta\subseteq\cal{F}_\vartheta$.
\end{enumerate}
\end{lemma}
\begin{proof}Given that
\begin{align*}\{\eta<\vartheta\}\cap\{\vartheta\leq t\}&=\bigcap_{s\in[0,t)\cap\mathbb{Q}}\{\eta\leq s\}\cap\{s<\vartheta\}\cap\{\vartheta\leq t\}\in\cal{F}_t\quad\forall t\in[0,\infty),
\end{align*}
 this proof is analogous to that of the lemma's discrete-time counterpart (Lemma~\ref{lem:2stop}) and we skip it.
\end{proof}

\begin{exercise}\label{ex:xeta}Show that $(\cal{F}_t)_{t\geq0}$-stopping times $\eta$ and states $x$ in $\s$, $\{\eta<T_\infty,X_\eta=x\}$ belongs to $\cal{F}_\eta$. Hint: re-write $\{\eta<T_\infty,X_\eta=x\}$ as
$$\{\eta<T_\infty,X_\eta=x\}=\bigcup_{n=0}^\infty\{T_{n}\leq\eta<T_{n+1},Y_n=x\},$$
and use Lemma~\ref{lem:etatheta}$(i,ii)$.\end{exercise}

\subsubsection*{Does $\cal{F}_{T_n}=\cal{G}_n$? An open problem} It follows directly from Definitions~\ref{def:filt3}~and~\ref{def:stopct} that any jump time $T_n$ is a stopping time. The pre-$T_n$ sigma algebra $\cal{F}_{T_n}$  contains all events whose occurrence or non-occurrence can be deduced from observing  the chain's path up until (and including) the jump time $T_n$. Because this path segment is characterised  by the first $n+1$ states $Y_0,\dots,Y_n$ visited and first $n$ jump times $T_1,\dots,T_n$ (recall \eqref{eq:cpathdef}--\eqref{eq:cpathdef2}), it seems natural for $\cal{F}_{T_n}$ to coincide with $\cal{G}_n$ in Definition~\ref{def:filt2} (i.e., with the collection of events whose occurrence or non-occurrence can be deduced from observing $Y_0,\dots,Y_n,T_1,\dots,T_n$). Indeed, it is straightforward to show that $\cal{G}_n$ is contained in $\cal{F}_{T_n}$:
\begin{proposition}\label{prop:ftngn}$\cal{G}_n\subseteq\cal{F}_{T_n}$ for all $n\geq0$.
\end{proposition}
\begin{proof}$\cal{G}_n$ is generated by sets of the form 
$$A:=\{Y_0=x_0, Y_1=x_1,\dots,Y_n=x_n, T_0\leq t_0,T_1\leq t_1,\dots, T_n\leq t_n\},$$
where $x_0,\dots,x_n\in\s$ and $t_0,\dots,t_n\in[0,\infty)$. Fix any $t\geq0$. Because $\cal{F}_{t}$ is closed under finite intersections, it follows from its definition that
\begin{align*}A\cap \{T_n\leq t\}
=\{Y_0=x_0, Y_1=x_1,\dots,Y_n=x_n, T_0\leq t_0\wedge t,T_1\leq t_1\wedge t,\dots, T_n\leq t_n\wedge t\}\end{align*}
belongs to $\cal{F}_t$. In other words, $A$ belongs to $\cal{F}_{T_n}$ and the proposition follows.\end{proof}
I have so far been unable to find a simple argument for the converse. One possible way forward could involve \citep[Proposition~6.157(c)]{Freedman1983}. However, at least at first glance, it seems like this approach would require tinkering with the definition of the underlying space $(\Omega,\cal{F})$ (Section~\ref{sec:FKG}) to ensure that it is Borel.

\ifdraft
\subsubsection*{Notes and references} Here we used the methods of proof in https://almostsure.wordpress.com/2009/11/23/sigma-algebras-at-a-stopping-time/ and that in the proof of \citep[Lemma~6.5.3]{Norris1997}
\fi

\subsection{The path space and the path law}
\label{sec:pathspacect}
The \emph{law of the process} (or \emph{path law} for short) is the distribution induced by the chain on the \emph{path space} (i.e., the set of possible paths). 

\subsubsection*{The path space}  Recall that in \eqref{eq:cpathdef}--\eqref{eq:cpathdef2}, we defined the chain $X$ in terms of the jump chain $Y:=(Y_n)_{n\in\n}$ and the waiting times $S:=(S_n)_{n\in\zp}$. To simplify the technicalities, throughout this book we  identify 
 $X$ with $S$ and $Y$ in the sense that we view $X$ as a random variable taking values in the set
$$\cal{P}:=\s^{\n}\times [0,\infty)^{\zp}$$
of all possible sequences of states and waiting times:
\begin{equation}\label{eq:fm9awha7n3u2a}X(\omega)=(Y(\omega),S(\omega))\in\cal{P}\quad\forall \omega\in\Omega.\end{equation}
To be able to talk about random variables taking values in the \emph{path space}\glsadd{Pspace}\index{path space} $\cal{P}$, we must assign it a sigma-algebra. We choose the sigma-algebra\glsadd{E} $\cal{E}$ generated\index{cylinder sigma-algebra} by the cylinder sets, i.e., sets of the form
\begin{equation}\label{eq:nfanfewuiafnueia}\{x_0\}\times\{x_1\}\times\dots\times\{x_n\}\times\s\times\s\times\dots\times (s_1,\infty)\times(s_2,\infty)\times\dots\times (s_n,\infty)\times[0,\infty)\times[0,\infty)\times\dots,\end{equation}
where $n$ is any positive integer, $x_0,x_1,\dots,x_n$ are any states in $\s$, and $s_1,s_2,\dots,s_n$ are any non-negative real numbers. It is not difficult to check that $Y_n$ and $S_m$ are $\cal{F}/\cal{E}$-measurable functions for every $n\geq0$ and $m>0$. It follows that the map $X:\Omega\to\cal{P}$ defined by \eqref{eq:fm9awha7n3u2a} is $\cal{F}/\cal{E}$-measurable.
\begin{exercise}\label{exe:fja89wmwau}Show that $X$ is an $\cal{F}/\cal{E}$-measurable function.\end{exercise} 
\subsubsection*{The path law}
Given the above exercise,
\begin{equation}\label{eq:pathlawct}\mathbb{L}_\gamma(A):=\Pbl{\{X\in A\}}\quad\forall A\in\cal{E},\end{equation}
is a well-defined probability measure on $(\cal{P},\cal{E})$ known as the \emph{path law}\glsadd{Lgam}\index{path law} of $X$, where 
$$\{X\in A\}:=\{\omega\in\Omega:X(\omega)\in A\}$$
denotes the preimage of $A$ under $X$. Just as with $\Pb_x$ and $\Eb_x$, we write\glsadd{Lx} $\mathbb{L}_x$ as a shorthand for $\mathbb{L}_\gamma$ with $\gamma=1_x$. The path law is characterised as follows:

\begin{theorem}\label{thrm:pathlawunict}The path law $\mathbb{L}_\gamma$ defined in \eqref{eq:pathlawct} is the only measure on $(\cal{P},\cal{E})$ such that
\begin{equation}\label{eq:nf78eawh78wea}\mathbb{L}_\gamma(\text{\eqref{eq:nfanfewuiafnueia}})=\gamma(x_0)p(x_0,x_1)\dots p(x_{n-1},x_n)e^{-\lambda(x_0)s_1}e^{-\lambda(x_1)s_2}\dots e^{-\lambda(x_{n-1})s_n}\end{equation}
for all positive integers $n$, states $x_0,x_1,\dots,x_n\in\s$, and real numbers $s_1,s_2,\dots s_n\geq0$.
\end{theorem}

\begin{proof}
Because sets of the type \eqref{eq:nfanfewuiafnueia} form a $\pi$-system that generates $\cal{E}$, Lemma~\ref{lem:dynkinpl} tells us that there is at most one probability measure on $(\cal{P},\cal{E})$ satisfying \eqref{eq:nf78eawh78wea} for all $n,x_0,x_1,\dots,x_n,s_1,s_2,\dots s_n$. Given that $\mathbb{L}_{\gamma}(\cal{P})=\mathbb{P}_\gamma(\Omega)=1$, we need only to show that \eqref{eq:nf78eawh78wea} holds for all $n,x_0,x_1,\dots,x_n,s_1,s_2,\dots s_n$.

To do so, define the events
$$A_m:=\{Y_0=x_0,\dots Y_{m}=x_{m},S_1>s_1,\dots,S_{m}>s_m\},\quad\forall m=1,2,\dots,n.$$
It follows from Exercise~\ref{ex:fme78ahfb67ebf68aybfewa} that  $A_m$ belongs to $\cal{G}_m$ in (Definition~\ref{def:filt2}). For this reason, the definition of $\mathbb{L}_\gamma$ in \eqref{eq:pathlawct} implies that
\begin{align}\mathbb{L}_\gamma(\text{\eqref{eq:nfanfewuiafnueia}})
&=\Pbl{A_{n-1}\cap\{Y_{n}=x_{n},S_n>s_n\}}\nonumber\\
&=\Ebl{\Pbl{A_{n-1}\cap\{Y_{n}=x_{n},S_n>s_n\}|\cal{G}_{n-1}}}\nonumber\\
&=\Eb_\gamma[1_{A_{n-1}}\Pbl{\{Y_{n}=x_{n},S_n>s_n\}|\cal{G}_{n-1}}].\label{eq:nf8yewab8w/8f4ewaafe4a}\end{align}
where the last two equalities follow from the tower and take-out-what-is-known properties of conditional expectation (Theorem~\ref{thrm:condexpprops}$(iv,v)$). Setting $f(x):=1_{x_n}(x)$, $g(s,w):=1_{(s_n,\infty)}(s)$, and $W:=1$ in Theorem~\ref{thrm:condind} we find that
\begin{align*}\Pbl{\{Y_{n}=x_{n},S_n>s_n\}|\cal{G}_{n-1}}&=\Pbl{\{Y_{n}=x_{n}\}|\cal{G}_{n-1}}\Pbl{\{S_n>s_n\}|\cal{G}_{n-1}}\\
&=p(Y_{n-1},x_n)e^{-\lambda(Y_{n-1})s_n}\quad\Pb_\gamma\text{-almost surely},
\end{align*}
Plugging the above into \eqref{eq:nf8yewab8w/8f4ewaafe4a} and exploiting that $Y_{n-1}=x_{n-1}$ on $A_{n-1}$, we obtain
$$\mathbb{L}_\gamma(\text{\eqref{eq:nfanfewuiafnueia}})=\Eb_\gamma[1_{A_{n-l}}p(Y_{n-1},x_n)e^{-\lambda(Y_{n-1})s_n}]=\Pbl{A_{n-1}}p(x_{n-1},x_n)e^{-\lambda(x_{n-1})s_n}.$$
Iterating the above backwards, we find that
$$\mathbb{L}_\gamma(\text{\eqref{eq:nfanfewuiafnueia}})=\Pbl{A_{0}}p(x_{0},x_1)e^{-\lambda(x_{0})s_1}p(x_{1},x_2)e^{-\lambda(x_{1})s_2}\dots p(x_{n-1},x_n)e^{-\lambda(x_{n-1})s_n},$$
and \eqref{eq:nf78eawh78wea} follows by re-arranging and noting that 
$$\Pbl{A_0}=\Pbl{\{Y_0=x_0\}}=\Pbl{\{X_0=x_0\}}=\gamma(x_0).$$

\end{proof}

\subsubsection*{More dues: a proof of Proposition~\ref{prop:absorb}}Our first use of the machinery set up in this section is proving Proposition~\ref{prop:absorb} given in Section~\ref{sec:FKG}. Because the definition of the jump matrix in~\eqref{eq:jumpmatrix} implies that $p(x,x)=1$ if and only if $q(x)=0$, the chain's definition~\eqref{eq:cpathdef}--\eqref{eq:cpathdef3}, downwards monotone convergence, \eqref{thrm:pathlawunict} imply that
\begin{align}\Pbx{\{X_t=x,\enskip\forall t\in[0,T_\infty)\}}&=\Pbx{\{Y_n=x,\enskip\forall n\in\n\}}=\lim_{m\to\infty}\Pbx{\bigcap_{n=0}^m \{Y_n=x\}}\label{eq:fnfneu8abfea78bfa}\\
&=\lim_{m\to\infty}p(x,x)^m=\left\{\begin{array}{ll}1&\text{if }q(x)=0\\0&\text{if }q(x)\neq 0\end{array}\right.,\nonumber\end{align}
and it follows that $\Pbx{\{T_\infty=\infty,X_t=x,\enskip\forall t\in[0,\infty)\}}=0$ unless $q(x)=0$. If $q(x)=0$, then the above shows that $Y_n=x$ for all $n\geq0$, $\Pb_\gamma$-almost surely. Moreover, the jump rate $\lambda(x)$ is one by its definition in~\eqref{eq:lambda} and, so, Kolmogorov's strong law of large numbers (Theorem~\ref{thrm:klln}) implies that
$$T_\infty=\lim_{n\to\infty}T_n=\lim_{n\to\infty}\sum_{m=1}^{n}S_m=\sum_{m=1}^{\infty}S_m=\sum_{m=1}^{\infty}\frac{\xi_m}{\lambda(Y_n)}=\sum_{m=1}^{\infty}\frac{\xi_m}{\lambda(x)}=\sum_{m=1}^{\infty}\xi_m=\infty\quad\Pb_\gamma\text{-a.s.}$$
The above and \eqref{eq:fnfneu8abfea78bfa} then imply that $\Pbx{\{T_\infty=\infty,X_t=x,\enskip\forall t\in[0,\infty)\}}=1$, as desired.

\subsubsection*{Measurable functions on $\cal{P}$} In the next section, we will give the Markov property in terms of $\cal{E}/\cal{B}(\r_E)$-measurable functions on $\cal{P}$. 
Here, we take a moment to get acquainted with some of these functions and a feel for what the others may be. To start with, the definition of $\cal{E}$  implies that \emph{coordinate functions}
$$c^Y_n(x)=y_n,\qquad c^S_n(x)=s_n,\qquad\forall x:=(y,s)\in\cal{P}$$
are $\cal{E}/2^\s$-measurable and $\cal{E}/\cal{B}(\r_E)$-measurable, respectively. Because linear combinations and limits of real-valued measurable functions are measurable,
$$c^T_0(x):=0,\qquad c^T_n(x)=\sum_{m=1}^ns_n,\qquad c^T_\infty(x)=\lim_{m\to\infty}c^T_n(x)=\sum_{m=1}^\infty s_n,\qquad\forall x\in\cal{P},$$
are also $\cal{E}/\cal{B}(\r_E)$-measurable. We now have measurable functions whose composition with $X$ give us the $n$th state visited by the chain, the $n$th waiting time, the $n$th jump time, and the explosion time:
\begin{equation}
\label{eq:coord}Y_n(\omega)=c^Y_n(X(\omega)),\quad S_n(\omega)=c^S_n(X(\omega)),\quad T_n(\omega)=c^T_n(X(\omega)),\quad T_\infty(\omega)=c^T_\infty(X(\omega))\end{equation}
for all $\omega$ in $\Omega$. Similarly, for any $f:\s\to\r$,
\begin{equation}
\label{eq:coord2}1_{\{t<T_\infty(\omega)\}}f(X_t(\omega))=\sum_{n=0}^\infty1_{\{T_n(\omega)\leq t<T_{n+1}(\omega)\}}f(Y_n(\omega))=c^f_t(X(\omega))\quad\forall \omega\in\Omega,\end{equation}
where we are using the notation for partially-defined functions introduced in~\eqref{eq:partdef} and
$$c_t^f(x):=\sum_{n=0}^\infty1_{[0,t]}(c^T_n(x))1_{(t,\infty)}(c^T_{n+1}(x))f(c^Y_n(x))\quad\forall x\in\cal{P},$$
is a $\cal{E}/\cal{B}(\r_E)$-measurable function given that it is obtained by composing, adding, multiplying, and taking limits of measurable functions. Setting $f$ to be the indicator function $1_z$ of any given state $z$, we have that the indicator function of the event that the chain is in $z$ at time $t$, is the composition $c_t^z(X)$  of the chain $X$ and the $\cal{E}/\cal{B}(\r_E)$-measurable function 
\begin{equation}\label{eq:ctz}c_t^z(x):=\sum_{n=0}^\infty1_{[0,t]}(c^T_n(x))1_{(t,\infty)}(c^T_{n+1}(x))1_z(c^Y_n(x))\quad\forall x\in\cal{P}.\end{equation}
\begin{exercise}\label{exe:measpedantry}If you'd like a challenge in measurability-checking-pedantry(!), show that the sets
\begin{equation}\label{eq:gensetsalt}\{x\in\cal{P}:c^{z_1}_{t_1}(x)=1,c^{z_2}_{t_2}(x)=1,\dots,c^{z_n}_{t_n}(x)=1\}\end{equation}
for all $t_1\leq t_2\leq\dots\leq t_n\in[0,\infty)$, $z_1,z_2,\dots, z_n\in\s$, and $n\in\n$, generate $\cal{E}$.
\end{exercise}


We finish this section with the following rather dull lemma that will cover most of our measurability-checking necessities throughout the ensuing treatment of continuous-time chains.
\begin{lemma}\label{lem:pathspmeasct}The following are $\cal{E}/\cal{B}(\r_E)$-measurable functions
\begin{enumerate}[label=(\roman*),noitemsep]
\item For any real-valued function $f$ on $\s$,
\begin{align*}F_-(x)&:=\liminf_{n\to\infty}\frac{1}{n}\sum_{m=0}^{n-1}c^S_{m+1}(x)f(c^Y_m(x)),\\
F_+(x)&:=\limsup_{n\to\infty}\frac{1}{n}\sum_{m=0}^{n-1}c^S_{m+1}(x)f(c^Y_m(x)),\quad\forall x\in\cal{P}.\end{align*}
%
%
\item For any subset $A$ of $\s$ and positive integer $k$,
$$F(x):=\inf\left\{n>0:\sum_{m=1}^n1_A(c^Y_m(x))=k\right\},\quad\forall x\in\cal{P}.$$
%
\item For any non-negative function $f$ on $\s$, 
$$G(x)=\sum_{m=0}^{F(x)-1}c^S_{m+1}(x)f(c^Y_m(x))\quad\forall x=(y,s)\in\cal{P},$$
where $F$ is as in $(ii)$.
\item  For any given $t$ in $[0,\infty)$ and real-valued function $f$ on $\s$,
\begin{align*}F_t(x):=&\sum_{n=0}^\infty1_{[0,t]}(c_{n+1}^T(x))c^S_{n+1}(x)f(c^Y_n(x))\\
&+\sum_{n=0}^\infty1_{[0,t]}(c_{n}^T(x))1_{(t,\infty)}(c_{n+1}^T(x))f(c^Y_n(x))(t-c_{n}^T(x))\quad\forall x\in\cal{P},\end{align*}
and
$$F_-(x):=\liminf_{t\to\infty}\frac{F_t(x)}{t},\qquad F_+(x):=\limsup_{t\to\infty}\frac{F_t(x)}{t},\qquad\forall x\in\cal{P}.$$
\end{enumerate} 
\end{lemma}
\begin{proof} This proof is analogous to that of the result's discrete-time counterpart (Lemma~\ref{lem:pathspmeas}) and may be skipped.

$(i)$ 
Because finite sums and products of measurable functions are measurable,  we have that $x\mapsto \frac{1}{n}\sum_{m=0}^{n-1}c^S_{m+1}(x)f(c^Y_m(x))$ is $\cal{E}/\cal{B}(\r)$-measurable, for each $n>0$. Because the limit infimum and supremum of measurable functions are measurable, the result follows.

$(ii)$ By definition, $F(x)$ counts the number of jumps until the path $x$ enters the set $A$ for the $k$th time. For this reason,
\begin{align*}F(x)&=\infty\cdot 1_{\{g_\infty(x)<k\}}+\sum_{n=1}^\infty n1_{\{g_n(x)<k\}}1_{\{g_n(x)=k\}}\\
&=\infty\cdot \lim_{N\to\infty}1_{\{g_N(x)<k\}}+\lim_{N\to\infty}\sum_{n=1}^Nn1_{\{g_n(x)<k\}}1_{\{g_n(x)=k\}}\quad\forall x\in\cal{P},\end{align*}
where 
$$g_n(x):=\sum_{m=1}^n1_A(c^Y_m(x))\quad\forall x\in\cal{P}$$ 
denotes the number of times the path $x$ has entered $A$ by time $n>0$. Because any finite sum of measurable functions is measurable, $g_n$ is $\cal{E}/\cal{B}(\r_E)$-measurable. The result then follows from the above by exploiting once again that the sum, product, composition, limit superior, and limit inferior of measurable functions are all measurable. 

$(iii)$ Because any finite sum of measurable functions is measurable, 
$$g_n(x):=\sum_{m=0}^nc^S_{m+1}(x)f(c^Y_m(x))\quad\forall x\in\cal{P}$$
defines an $\cal{E}/\cal{B}(\r_E)$-measurable for any natural number $n$. Furthermore, as the limit of measurable functions is measurable, $g_\infty:=\lim_{n\to\infty}g_n$ is also $\cal{E}/\cal{B}(\r_E)$-measurable. Pick any $A$ in $\cal{B}(\r_E)$ and note that
\begin{align*}\left\{(n,x\in\n_E\times\cal{P}:\sum_{m=0}^nc^S_{m+1}(x)f(c^Y_m(x))\in A\right\}\\
=(\{\infty\}\times\{g_\infty\in A\})\cup\left(\bigcup_{n=0}^\infty\{n\}\times\{g_n\in A\}\right).\end{align*}
Because the right-hand side belongs to the product sigma-algebra $2^{\n_E}\times\cal{E}$, it follows that 
$$H(n,x):=\sum_{m=0}^nc^S_{m+1}(x)f(c^Y_m(x))\quad\forall x\in\cal{P}$$
defines an $2^{\n_E}\times\cal{E}/\cal{B}(\r_E)$-measurable function. Because $G(x)=H(F(x)-1,x)$ for all $x$ in $\cal{P}$, with $F$ as in $(ii)$, the result then follows as the composition of measurable functions is measurable.

$(iv)$ This follows directly from the facts that the sum, product, limit, limit inferior, and limit superior of measurable functions are measurable.
\end{proof}

\subsection{The Markov and strong Markov properties}\label{sec:markovprop} %
Stating and proving the Markov property for the continuous-time case is somewhat more involved than it is for the discrete-time case: $X_t$ is only defined on the event $\{t<T_\infty\}$ that no explosion has occurred by time $t$, and, so, conditioning on $X_t$ requires a bit of care.  We have already faced a similar issue: when dealing with the strong Markov property  of discrete time chains $(W_n)_{n\in\n}$~(c.f.~Section~\ref{sec:dtstrmarkov}),  we had to condition on the chain's state $W_\varsigma$ at  a stopping time $\varsigma$ but $W_\varsigma$ was only defined on the event $\{\varsigma<\infty\}$ that the stopping time was finite.
%
%
We follow here the steps we took there to resolve the issue. If necessary, you should brush up on Sections~\ref{sec:dtmarkov}~and~\ref{sec:dtstrmarkov}: the concepts discussed therein are equally applicable here and guide the development of this section.

%
%
%
%
%
%

%
\subsubsection*{Shifting the chain left by $\eta$} To describe the chain's future from a stopping time $\eta$ onwards, define the \emph{shifted chain} $X^\eta:=(Y^\eta,S^\eta)$ as the function mapping from $\{\eta<T_\infty\}$ to the path space $\cal{P}$ (c.f.~Section~\ref{sec:pathspacect}) given by\glsadd{Xeta}\index{shifted chain}
\begin{align}\label{eq:tshiftx1} S^\eta(\omega)&:=(S_{n+1}(\omega)-(\eta(\omega)-T_n(\omega)),S_{n+2}(\omega),S_{n+3}(\omega),\dots),\\
\label{eq:tshiftx2}Y^\eta(\omega)&:=(Y_{n}(\omega),Y_{n+1}(\omega),Y_{n+2}(\omega),\dots),\quad\forall \omega\in\{T_n\leq \eta<T_{n+1}\},\enskip n\geq0,.\end{align}
This $\cal{P}$-valued function $X^\eta$ describes the chain \emph{shifted} leftwards in time by $\eta$ amount: setting $X^\eta_{t}$ to be the piecewise-constant interpolation of $X^\eta$ (i.e., replace $Y,S$ in \eqref{eq:cpathdef}--\eqref{eq:cpathdef2} with $Y^\eta,S^\eta$), we find that
\begin{equation}\label{eq:xshift}X^\eta_{t}(\omega)=X_{\eta+t}(\omega)\quad\forall \omega\in\{\eta+t<T_\infty\},\quad\forall t\in[0,\infty).\end{equation}
Similarly, the \emph{jump times} $T^\eta_0,T^\eta_1,\dots$ and explosion time of the shifted chain $X^\eta$, defined by
$$T^\eta_0(\omega):=0,\quad T^\eta_n(\omega):=\sum_{m=1}^nS^\eta_m(\omega)\enskip\forall n>0,\quad T_\infty^\eta(\omega):=\sum_{m=1}^\infty S^\eta_m(\omega),\quad\forall \omega\in\{t<T_\infty\},$$
satisfy
\begin{align}\label{eq:tshift}T^\eta_{n'}(\omega)&=T_{n+n'}(\omega)-\eta(\omega)\quad\forall \omega\in\{T_n\leq \eta<T_{n+1}\},\enskip n,n'\geq0,\\
\label{eq:tishift}T^\eta_\infty(\omega)&=T_\infty(\omega)-\eta(\omega)\quad\forall \omega\in\{\eta<T_{\infty}\}.\end{align}
%
%

%
\begin{exercise}\label{exe:nfua9feah76a8wnfw78a}Prove \eqref{eq:xshift} and show that \eqref{eq:coord}--\eqref{eq:coord2} hold  if we replace $Y,S,T,T_\infty,X_t,t,\Omega$ therein with $Y^\eta,S^\eta,T^\eta,T_\infty^\eta,X^\eta,\eta,\{\eta<T_\infty\}$. \end{exercise}
\subsubsection*{The Markov property}We now have all we need to state and prove the Markov property. Note that whenever we write $1_{\{t<T_\infty\}}F(X^t)$ in what follows, we are using the notation for partially-defined functions introduced in~\eqref{eq:partdef}.
\index{Markov property}\begin{theorem}[The Markov property]\label{thrm:markovprop}Let $(\cal{F}_t)_{t\geq0}$ be the filtration generated by the chain~(Definition~\ref{def:filt3}), $\cal{E}$ be the cylinder sigma-algebra on the path space $\cal{P}$ (Section~\ref{sec:pathlaw}), and $t$ belongs to $[0,\infty)$. Suppose that $Z$ is an $\cal{F}_t/\cal{B}(\r_E)$-measurable random variable, $F$ is a $\cal{E}/\cal{B}(\r_E)$-measurable function, and that $Z$ and $F$ are both non-negative, or both bounded. If $x$ any state in $\s$, then $Z1_{\{t<T_\infty,X_t=x\}}F(X^t)$ is $\cal{F}/\cal{B}(\r_E)$-measurable, where $X^t$ denotes the $t$-shifted chain~\eqref{eq:tshiftx1}--\eqref{eq:tshiftx2}. Moreover,
\begin{equation}\label{eq:mkvct}\Ebl{Z1_{\{t<T_\infty,X_t=x\}}F(X^t)}=\Ebl{Z1_{\{t<T_\infty,X_t=x\}}}\Ebx{F(X)}\quad\forall x\in\s.\end{equation}
%
%
\end{theorem}

For the theorem's proof, we need the following:
\begin{lemma}\label{lem:aan}Let $(\cal{G}_n)_{n\in\n}$ be the filtration generated by the jump chain and jump times~(Definition~\ref{def:filt2}) and $(\cal{F}_t)_{t\geq0}$ that generated by the chain~(Definition~\ref{def:filt3}). For any given $t$ in $[0,\infty$, $A$ in $\cal{F}_t$, and $n$ in $\n$, there exists an $A_n$ in $\cal{G}_{n}$ such that
$$A\cap\{t<T_{n+1}\}=A_n\cap\{t<T_{n+1}\}.$$
\end{lemma}
\begin{proof}Because $\cal{G}_n$ is a sigma-algebra, 
$$\cal{G}_t:=\{A\in\cal{F}_t:A\cap\{t<T_{n+1}\}=A_n\cap \{t<T_{n+1}\}\text{ for some }A_n\in\cal{G}_n\}$$
 is also a sigma-algebra. For any $s\leq t$, $m\geq0$, and $x$ in $\s$,
$$\{T_m\leq s, Y_m=x\}\cap\{t<T_{n+1}\}$$
is either empty  (if $m>n$) and we set $A_n:=\emptyset$  or $\cal{G}_m$ is contained in $\cal{G}_n$ (if $m\leq n$) and we set $A_n:=\{T_m\leq s, Y_m=x\}$. In either case, $A_n$ belongs to $\cal{G}_n$ and
$$\{T_m\leq s, Y_m=x\}\cap\{t<T_{n+1}\}=A_n\cap\{t< T_{n+1}\}$$
showing that $\{T_m\leq s, Y_m=x\}$ belongs to $\cal{G}_t$. Because these sets generate $\cal{F}_t$ and $\cal{G}_t$'s definition implies that $\cal{G}_t\subseteq\cal{F}_t$, it follows that $\cal{G}_t=\cal{F}_t$.
\end{proof}

\begin{proof}[Proof of Theorem~\ref{thrm:markovprop}]Given Exercise~\ref{ex:xeta} and Theorem~\ref{thrm:pathlawunict}, this proof is entirely analogous that of Theorem~\ref{thrm:strmkvpath} as long as we are able to show that 
\begin{equation}\label{eq:dks8a90ahf870gnera8gnaeg}\Pbl{A\cap\{t<T_\infty,X_{t}=x,X^t\in B\}}=\Pbl{A\cap\{t<T_\infty,X_{t}=x\}}\mathbb{L}_x(B),\end{equation}
for all $A$ in $\cal{F}_t$, $x$ in $\s$, and cylinder sets $B$ (that is, sets $B$ of the form in \eqref{eq:nfanfewuiafnueia}), where $\mathbb{L}_x$ denotes the path law~\eqref{eq:pathlawct} starting from $x$. To do so, fix any such $A$, $x$, and $B$ and note that
\begin{align}\label{eq:fn6ae8s79hfwa}\Pbl{A\cap\{t<T_\infty,X_t=x,X^t\in B\}}=\sum_{m=0}^\infty\Pbl{A_m\cap\{T_m\leq  t<T_{m+1},Y_m=x\}\cap B_m},\end{align}
where $A_m$ is the set in Lemma~\ref{lem:aan} belonging to $\cal{G}_m$ and
$$B_m:=\{Y_{m}=x_0,Y_{m+1}=x_1,\dots,Y_{m+n}=x_n,S_{m+1}-(t-T_m)>s_1,S_{m+2}>s_2,\dots,S_{m+n}>s_n\}.$$
Because $B_m\subseteq \{t<T_{m+1}\}$ and $A_m\cap\{T_m\leq t,Y_m=x\}$ belongs to $\cal{G}_m$, the tower and take-out-what-is-known properties of conditional expectation (Theorem~\ref{thrm:condexpprops}$(iv,v)$) imply that
\begin{align}\label{eq:f78ajf8au94q947qfwa8}\Pbl{A_m\cap\{T_m\leq t<T_{m+1},Y_m=x\}\cap B_m}
&=\Pbl{A_m\cap\{T_m\leq t,Y_m=x\}\cap B_m}\\
&=\Ebl{\Pbl{A_m\cap\{T_m\leq t,Y_m=x\}\cap B_m|\cal{G}_m}}\\
&=\Ebl{1_{A_m\cap\{T_m\leq t,Y_m=x\}}\Pbl{ B_m|\cal{G}_m}}.\nonumber\end{align}
Conditioning on $\cal{G}_{m+n-1},\cal{G}_{m+n-2},\dots,\cal{G}_{m+1}$ and making a repeated use of Theorem~\ref{thrm:condind} similar to that in proof of Theorem~\ref{thrm:pathlawunict}, we find that
\begin{align*}\Pbl{B_m|\cal{G}_m}=1_{Y_m}(x_0)p(x_0,x_1)\dots p(x_{n-1},x_n)e^{-\lambda(x_0)(s_1+t-T_m)}e^{-\lambda(x_1)s_2}\dots e^{-\lambda(x_{n-1})s_n}\\
=e^{-\lambda(Y_m)(t-T_m)}\mathbb{L}_{Y_m}(B)=\Pbl{\{T_{m+1}>t\}|\cal{G}_m}\mathbb{L}_{Y_m}(B).\end{align*}
Plugging the above into \eqref{eq:f78ajf8au94q947qfwa8} and applying the take-out-what-is-known and tower properties, we find that
\begin{align*}\Pbl{A_m\cap\{T_m\leq t<T_{m+1},Y_m=x\}\cap B_m}&=\Ebl{1_{A_m\cap\{T_m\leq t<T_{m+1},Y_m=x\}}\mathbb{L}_{Y_m}(B)}\\
&=\Pbl{A_m\cap\{T_m\leq t<T_{m+1},Y_m=x\}}\mathbb{L}_{x}(B)\\
&=\Pbl{A\cap\{T_m\leq t<T_{m+1},X_t=x\}}\mathbb{L}_{x}(B).\end{align*}
Combining the above with \eqref{eq:fn6ae8s79hfwa}, we obtain \eqref{eq:dks8a90ahf870gnera8gnaeg} and the result follows.
\end{proof}

\subsubsection*{The strong Markov property} Just as in the discrete-time case (Section~\ref{sec:dtstrmarkov}), the Markov property holds for all stopping times $\eta$ instead of only for deterministic times $t$:
%
%

\index{strong Markov property}\begin{theorem}[The strong Markov property]\label{thstrmk} Let $(\cal{F}_t)_{t\geq0}$ be the filtration generated by the chain (Definition~\ref{def:filt2}), $\cal{E}$ be the cylinder sigma-algebra on the path space $\cal{P}$ (Section~\ref{sec:pathlaw}), $\eta$ be a $(\cal{F}_t)_{t\in[0,\infty)}$-stopping time (Definition~\ref{def:stopct}), and $\cal{F}_\eta$ be its associated sigma-algebra. Suppose that $Z$ is an $\cal{F}_\eta/\cal{B}(\r_E)$-measurable random variable, $F$ is a $\cal{E}/\cal{B}(\r_E)$-measurable function, and that $Z$ and $F$ are both non-negative, or both bounded. If $x$ any state in $\s$, then $Z1_{\{\eta<T_\infty,X_\eta=x\}}F(X^t)$ is $\cal{F}/\cal{B}(\r_E)$-measurable, where $X^\eta$ denotes the $\eta$-shifted chain~\eqref{eq:tshiftx1}--\eqref{eq:tshiftx2}. Moreover,
\begin{equation}\label{eq:strmkvct}\Ebl{Z1_{\{\eta<T_\infty,X_\eta=x\}}F(X^\eta)}=\Ebl{Z1_{\{\eta<T_\infty,X_\eta=x\}}}\mathbb{L}_x(F)\qquad\forall x\in\s.\end{equation}
%
%
\end{theorem}

\begin{proof}Given Exercise~\ref{ex:xeta} and Theorem~\ref{thrm:pathlawunict}, the proof is entirely analogous that of the theorem's discrete-time counterpart (Theorem~\ref{thrm:strmkvpath}) as long as we are able to show that 
\begin{equation}\label{eq:fn7e89hnfa89efdsafasha7}\Pbl{A\cap\{\eta<T_\infty,X_{\eta}=x,X^\eta\in B\}}=\Pbl{A\cap\{\eta<T_\infty,X_{\eta}=x\}}\mathbb{L}_x(B),\end{equation}
for all $A$ in $\cal{F}_\eta$, $x$ in $\s$, and cylinder sets $B$ (that is, sets $B$ of the form in \eqref{eq:nfanfewuiafnueia}), where $\mathbb{L}_x$ denotes the path law~\eqref{eq:pathlawct} starting from $x$. To do so, fix any such $A$, $x$, and $B$ and consider the discretisations 
\begin{equation}\label{eq:etak}\eta_k:=\sum_{l=1}^\infty\frac{l}{k}1_{\left\{(l-1)/k<\eta\leq l/k\right\}}\quad\forall k\in\zp\end{equation}
of $\eta$. It is straightforward to check that $(\eta_k)_{k\in\zp}$ is a decreasing sequence of $(\cal{F}_t)_{t\in[0,\infty)}$-stopping times with limit $\eta$. Moreover, the right-continuity (w.r.t. the discrete topology on $\s$) of the paths of $X$, see~\eqref{eq:cpathdef}--\eqref{eq:cpathdef3},  imply that
\begin{align}\label{eq:etak2}&\lim_{k\to\infty}1_{\{\eta_k<T_\infty,X_{\eta_k}=x,X^{\eta_k}\in B\}}(\omega)=1_{\{\eta<T_\infty,X_{\eta}=x,X^{\eta}\in B\}}(\omega)\quad\forall \omega\in \Omega,\\
\label{eq:etak3}&\lim_{k\to\infty}1_{\{\eta_k<T_\infty,X_{\eta_k}=x\}}(\omega)=1_{\{\eta<T_\infty,X_{\eta}=x\}}(\omega)\quad\forall \omega\in \Omega.\end{align}
Because $A$ belongs to $\cal{F}_{\eta}$ and $\eta$ is bounded above $\eta_k$, Lemma~\ref{lem:etatheta}$(iii)$ shows that $A$ belongs to $\cal{F}_{\eta_k}$. For this reason, the Markov property (Theorem~\ref{thrm:markovprop}) and Tonelli's theorem imply that
\begin{align*}\Pbl{A\cap\{\eta_k<T_\infty,X_{\eta_k}=x,X^{\eta_k}\in B\}}&=\sum_{l=0}^\infty\Pbl{A\cap\left\{\eta_k=\frac{l}{k},\frac{l}{k}<T_\infty,X_{\frac{l}{k}}=x,X^{\frac{l}{k}}\in B\right\}}\\
&=\sum_{l=0}^\infty\Pbl{A\cap\left\{\eta_k=\frac{l}{k},\frac{l}{k}<T_\infty,X_{\frac{l}{k}}=x\right\}}\mathbb{L}_x(B)\\
&=\Pbl{A\cap\{\eta_k<T_\infty,X_{\eta_k}=x\}}\mathbb{L}_x(B)\quad\forall k>0.\end{align*}
Bounded convergence and \eqref{eq:etak2}--\eqref{eq:etak3} imply \eqref{eq:fn7e89hnfa89efdsafasha7} and the result follows.

\end{proof}

\begin{exercise}Convince yourself that $\eta_k$ in \eqref{eq:etak} is an $(\cal{F}_t)_{t\geq0}$-stopping time and that \eqref{eq:etak2}--\eqref{eq:etak3} hold.
\end{exercise}

%

\subsection{The transition probabilities and the semigroup property}
All \emph{finite-dimensional distributions}\index{finite-dimensional distributions} of the chain may be expressed in terms of the initial distribution $\gamma$ and the collection $(P_t)_{t\geq0}$ of matrices\glsadd{Pt} $P_t=(p_t(x,y))_{x,y\in\s}$ whose $(x,y)$-entry is the probability that the chain is in state $y$ at time $t$ if it starts in state $x$:
\begin{equation}\label{eq:transprob}p_t(x,y):=\Pbx{\{X_t=y,t<T_\infty\}}\quad\forall x,y\in\s,\enskip t\in[0,\infty).\end{equation}
In particular:
\begin{theorem}
\label{thrm:findim} For all times positive integers $n>0$, times $0\leq t_1\leq t_2\leq\dots\leq t_n$, and states $x_0,x_1,\dots,x_n$,
\begin{align*}\Pbl{\{X_0=x_0,X_{t_1}=x_1,X_{t_2}=x_2,\dots,X_{t_n}=x_n,t_n< T_\infty\}}\\
=\gamma(x_0)p_{t_1}(x_0,x_1)p_{t_1-t_2}(x_1,x_2)\dots p_{t_{n-1}-t_n}(x_{n-1},x_n).\end{align*}
\end{theorem}
Re-arranging the equation in \eqref{thrm:findim}, we find that probability that the chain travels from $x$ to $y$ in $t$ amount of time equals $p_t(x,y)$ regardless of when this transition occurs:
\begin{align}\label{eq:transprobinterp}\Pbl{\{X_{s+t}=y,s+t<T_\infty\}|\{X_s=x,s<T_\infty\}}&=\frac{\Pbl{\{X_{s+t}=y,s+t<T_\infty\}}}{\Pbl{\{X_s=x,s<T_\infty\}}}\\
&=p_t(x,y),\nonumber\end{align}
for all $s\geq0$ such that the denominator is non-zero. For this reason, $p_t(x,y)$ is referred to as the probability that the chain \emph{transitions} from $x$ to $y$ in $t$ amount of time and $P_t$ as the \emph{$t$-transition matrix}\index{transition probabilities and matrix}. The collection $(P_t)_{t\geq0}$ of these matrices satisfies the \emph{semigroup property}\index{semigroup property},
\begin{equation}\label{eq:semigroup}p_{t+s}(x,y)=\sum_{z\in\s}p_t(x,z)p_s(z,y)\quad\forall x,y\in\s\quad t,s\geq0,\end{equation}
or $P_{t+s}=P_tP_s$ for all $t,s\geq0$ in matrix notation, and $(P_t)_{t\geq0}$ is called the \emph{semigroup of transition probabilities}.
\begin{exercise}Using Theorem~\ref{thrm:findim}, prove~\eqref{eq:semigroup}.\end{exercise}
\begin{proof}[Proof of Theorem~\ref{thrm:findim}]Let 
$$A:=\{X_0=x_0,X_{t_1}=x_1,X_{t_2}=x_2,\dots,X_{t_{n-2}}=x_{n-2},t_{n-2}< T_\infty\}$$
and $F$ be the $\cal{E}/\cal{B}(\r_E)$-measurable function $c^{z}_{t}$ in \eqref{eq:ctz} with $z:=x_n$ and $t:=t_n-t_{n-1}$. Because $A$ belongs to $\cal{F}_{t_{n-1}}$ (Exercise~\ref{ex:filtrd}), applying \eqref{eq:coord2}, Exercise~\ref{exe:nfua9feah76a8wnfw78a}, and the Markov property (Theorem~\ref{thrm:markovprop}), we find that
\begin{align*}&\Pbl{\{X_0=x_0,X_{t_1}=x_1,X_{t_2}=x_2,\dots,X_{t_n}=x_n,t_n< T_\infty\}}\\
&=\Ebl{1_{A\cap\{ X_{t_{n-1}}=x_{n-1},t_{n-1}<T_\infty\}}F(X^{t_{n-1}})}\\
&=\Pbl{A\cap\{X_{t_{n-1}}=x_{n-1},t_{n-1}<T_\infty\}}\mathbb{E}_{x_{n-1}}[F(X)]\\
&=\Pbl{\{X_0=x_0,X_{t_1}=x_1,X_{t_2}=x_2,\dots,X_{t_{n-1}}=x_{n-1},t_{n-1}< T_\infty\}}p_{t_n-t_{n-1}}(x_{n-1},x_n).\end{align*}
Iterating the above argument, we have that
\begin{align*}&\Pbl{\{X_0=x_0,X_{t_1}=x_1,X_{t_2}=x_2,\dots,X_{t_n}=x_n,t_n< T_\infty\}}\\
&=\Pbl{\{X_0=x_0,X_{t_1}=x_1,X_{t_2}=x_2,\dots,X_{t_{n-1}}=x_{n-1},t_{n-1}< T_\infty\}}p_{t_n-t_{n-1}}(x_{n-1},x_n)\\
&=\dots=\Pbl{\{X_0=x_0,X_{t_1}=x_1,t_{1}< T_\infty\}}p_{t_2-t_{1}}(x_{1},x_2)\dots p_{t_n-t_{n-1}}(x_{n-1},x_n).
\end{align*}
Because $T_\infty=\sum_{n=1}^\infty S_n>0$  and $\Pbl{\{X_0=x_0\}}=\gamma(x_0)$, applying the Markov property (Theorem~\ref{thrm:markovprop}) once again with $A:=\Omega$ and $F$ as in \eqref{eq:ctz} with $z:= x_1$ and $t:=t_1$ completes the proof.
\end{proof}

\subsection{The forward and backward equations}\label{sec:forwardbackward}

Perhaps the most celebrated result in Markov chain theory are \emph{Kolmogorov's} \emph{forward} and \emph{backward} equations: two sets of ordinary differential equations satisfied by the transition probabilities. In particular:

\begin{theorem}[The forward and backward equations]\label{thrm:forwardbackward} Suppose that $Q$ is a stable and conservative rate matrix (i.e., satisfies \eqref{eq:qmatrix}). The transition probabilities are continuously differentiable: for each $x,y\in\s$,
$$t\mapsto p_t(x,y)$$
is a continuously differentiable function on $[0,\infty)$. Moreover, the transition probabilities satisfy the forward equations:\index{forward equations}
\begin{equation}\label{eq:forward}\dot{p}_t(x,y)=\sum_{z\in\s}p_t(x,z)q(z,y)\quad \forall t\in[0,\infty),\enskip x,y\in\s,\qquad p_0(x,y)=1_x(y)\quad\forall x,y\in\s,\quad\end{equation}
and the backward equations:\index{backward equations}
\begin{equation}\label{eq:backward}\dot{p}_t(x,y)=\sum_{z\in\s}q(x,z)p_t(z,y)\quad \forall t\in[0,\infty),\enskip x,y\in\s,\qquad p_0(x,y)=1_x(y)\quad\forall x,y\in\s.\quad\end{equation}
Moreover, the transition probabilities are the minimal non-negative solution of these equations: if $(k_t(x,y))_{x,y\in\s,t\geq0}$ is a non-negative ($k_t(x,y)\geq0$ for all $x,y\in\s,$ $t\geq0$) differentiable function satisfying either \eqref{eq:forward} or \eqref{eq:backward}, then
$$k_t(x,y)\geq p_t(x,y)\quad\forall x,y\in\s,\enskip t\in[0,\infty).$$
\end{theorem}
\noindent The proof of the above theorem is delicate and we postpone it until the end of the section. The forward and backward equations are often written in matrix notation as
$$\dot{P}_t=P_tQ,\quad \dot{P}_t=QP_t,\quad P_0=I,$$
where $I:=(1_x(y))_{x,y\in\s}$ denotes the identity matrix on $\s$.

\ifdraft
\subsubsection*{The meaning of these equations}

Give weak form and explain each of the terms featuring (conditioning on first jump for backwards and on last jump for forwards).

3) Or $P_t$ is the minimal $Q$-function (th.2.2 chap 2 anderson). Thus, $Q$-process is unique (w.r.t path law) iff minimal $Q$-process is honest.

4) Otherwise we can build others using re-starts (cor.2.5 chap 2 anderson).

5) Only restarts that don't any one state in particular satisfy forward equations (give flash, reference to chap 4 anderson).

\fi

\subsubsection*{Regular rate matrices and the uniqueness of the solutions}

The rate matrix $Q$ is said to be \emph{regular}\index{regular rate matrix} if the chain does not explode regardless of the initial distribution:
\begin{definition}[Regular rate matrices]\label{def:regular} The rate matrix $Q$ is regular if, for all initial distributions $\gamma$,
\begin{equation}\label{eq:noexpl}\Pbl{\{T_\infty=\infty\}}=1.\end{equation}
\end{definition}
We can also express the regularity of $Q$ in terms of the mass of the time-varying law:
\begin{proposition}\label{prop:nonexptimevar}If the chain starting position is sampled from $\gamma$, then the chain does not explode (i.e.,~\eqref{eq:noexpl} holds) if and only if
$$p_t(\s)=\sum_{x\in\s}p_t(x)=\sum_{x\in\s}\Pbl{\{X_t=x,t<T_\infty\}}=\Pbl{\{t<T_\infty\}}=1$$
for a single $t\in[0,\infty)$, in which case the above  holds for all $t\in[0,\infty)$.
Thus, the rate matrix is regular if and only if the above holds for all initial distributions $\gamma$.
%
\end{proposition}

\begin{proof}Monotone convergence implies that $\Pbl{\{n<T_\infty\}}\to\Pbl{\{T_\infty=\infty\}}$ as $n\to\infty$ and the result follows as $t\mapsto p_t(\s)=\Pbl{\{t<T_\infty\}}$ is clearly a non-increasing function of $t$.\end{proof} 

Consequently, if $Q$ is regular, then the transition probabilities are the unique non-negative solution of either the forward or backward equations with mass no greater than $1$ (i.e., such that $\sum_{y\in\s}p_t(x,y)\leq 1$ for all $x$ and $t$). If $Q$ is not regular, then the backward equations have infinitely-many such solutions and the forward equations may also do, see \citep{Anderson1991} and references therein.
%
%

\ifdraft

\subsubsection*{The meaning of the $Q$-matrix}

Sort of the discussion in asmussen?

DISCUSS THE IMPORTANCE OF EQUATIONS BLAH FOR MODELLING --> INTERPRETATION OF Q-MATRIX (DIAGONAL RATE OF EXIT...).

\fi

\subsubsection*{A proof of Theorem~\ref{thrm:forwardbackward}} We do this proof in five steps:

\begin{itemize}
\item[(Lemma~\ref{lem:forwardweak})]We derive the so-called \emph{forward integral recursion (FIR)} and use it to prove that the transition probabilities satisfy a weak version of the forward equations.
\item[(Lemma~\ref{lem:bir})]We use the FIR to obtain the \emph{backward integral recursion (BIR)} and use it to prove that the transition probabilities satisfy a weak version of the backward equations.
\item[(Lemma~\ref{lem:transbackward})]We use the weak version of the backward equations to show that the transition probabilities are continuously differentiable and satisfy the strong form~\eqref{eq:backward} of the backward equations.
\item[(Lemma~\ref{lem:transforward})]We use the differentiability of the transition [probabilities and the weak version of the forward equation to show that the strong form~\eqref{eq:forward} of these equations holds.
\item[(Lemma~\ref{lem:minimal})]We use the FIR and BIR to show that transition probabilities are minimal among both the solutions of the forward equations and those of the backward equations.
\end{itemize}
Let's begin:
\begin{lemma}[The FIR and the weak version of the forward equations]\label{lem:forwardweak} Suppose that $Q$ satisfies \eqref{eq:qmatrix}. Given any natural number $n$, time $t$, and states $x,y$, let
\begin{equation}\label{eq:pnt}p_t^n(x,y):=\Pbx{\{X_t=y,t<T_{n+1}\}}\end{equation}
denote  the probability that, having started in $x$, the chain lies in $y$ at time $t$ and that it has not jumped more than $n$ times during the interval $[0,t]$. These probabilities satisfy the forward integral recursion (FIR):\index{forward integral recursion (FIR)}
\begin{equation}\label{eq:fir}p_t^n(x,y)=1_x(y)e^{-\lambda(y)t}+\int_0^t\sum_{z\in\s}p_s^{n-1}(x,z)\lambda(z)p(z,y)e^{-\lambda(y)(t-s)}ds\end{equation}
%
for all $x,y$ in $\s$, non-negative $t$, and positive $n$,
where $(p(x,y))_{x,y\in\s}$ denotes the one-step matrix~\eqref{eq:jumpmatrix} of the jump chain and $(\lambda(x))_{x\in\s}$ the jump rates~\eqref{eq:lambda}. Moreover, the transition probabilities satisfy the integral version of the forward equations:
\begin{equation}\label{eq:forwardweak}p_t(x,y)=1_x(y)e^{-\lambda(y)t}+\int_0^t\sum_{z\in\s}p_s(x,z)\lambda(z)p(z,y)e^{-\lambda(y)(t-s)}ds\quad\forall x,y\in\s,\enskip t\geq 0.\end{equation}
\end{lemma}

\begin{proof}Because taking the limit $n\to\infty$ in \eqref{eq:fir} and applying monotone convergence yields \eqref{eq:forwardweak}, we need only to prove \eqref{eq:fir}. To this end, we decompose $p_t^n(x,y)$ as follows:
\begin{align}\label{eq:fm7824h8a7bgn789a49a3j8g4aq00}p_t^n(x,y)&=\sum_{m=0}^{n}\Pbx{\{X_t=y,T_m\leq t<T_{m+1}\}}\\
&=\Pbx{\{Y_0=y,t<T_{1}\}}+\sum_{m=1}^{n}\Pbx{\{Y_m=y,T_m\leq t<T_{m+1}\}},\nonumber
\end{align}
By the definition of the chain in Section~\ref{sec:FKG},
\begin{equation}\label{eq:fm7824h8a7bgn789a49a3j8g4aq0}p^0_t(x,y)=\Pbx{\{Y_0=y,t<T_{1}\}}=\Pbx{\{X_0=y\}}\Pbx{\{S_1>t\}}=1_x(y)e^{-\lambda(y)t}.\end{equation}
Next, Theorem~\ref{thrm:condind} and the tower and take-out-what-is-known properties of conditional expectation (Theorem~\ref{thrm:condexpprops}$(iv,v)$) imply that
\begin{align}
\Pb_x(\{Y_m=y,T_m\leq t<T_{m+1}\})&=\Pbx{\{Y_m=y,T_m\leq t, t-T_m<S_{m+1}\}}\\
\nonumber&=\Ebx{\Pbx{\{Y_m=y,T_m\leq t,t-T_m<S_{m+1}\}|\cal{G}_m}}\\
\nonumber&=\Ebx{1_{\{Y_m=y,T_m\leq t\}}\Pbx{\{t-T_m<S_{m+1}\}|\cal{G}_m}}\\
\nonumber&=\Ebx{1_{\{Y_m=y,T_m\leq t\}}e^{-\lambda(y)(t-T_m)}}\\
\nonumber&=\sum_{z\in\s}\Ebx{1_{\{Y_{m-1}=z,Y_m=y,S_m\leq t-T_{m-1}\}} e^{-\lambda(y)(t-T_m)}},
\end{align}
where $(\cal{G}_m)_{n\in\n}$ denotes the filtration generated by the jump chain and jump times (Definition~\ref{def:filt2}). Similarly, we have that
\begin{align}
\label{eq:fm7824h8a7bgn789a49a3j8g4aq}&\Ebx{1_{\{Y_{m-1}=z,Y_m=y,S_m\leq t-T_{m-1}\}} e^{-\lambda(y)(t-T_m)}}\\
\nonumber&=\Ebx{1_{\{Y_{m-1}=z,Y_m=y,S_m\leq t-T_{m-1}\}} e^{-\lambda(y)(t-S_m)}e^{\lambda(y)T_{m-1}}}\\
\nonumber&=\Ebx{\Ebx{1_{\{Y_{m-1}=z,Y_m=y,S_m\leq t-T_{m-1}\}} e^{-\lambda(y)(t-S_m)}|\cal{G}_{m-1}}e^{\lambda(y)T_{m-1}}}\\
\nonumber&=\Ebx{1_{\{Y_{m-1}=z\}}\Pbx{\{Y_m=y\}|\cal{G}_{m-1}}\Ebx{1_{\{S_m\leq t-T_{m-1}\}} e^{-\lambda(y)(t-S_m)}|\cal{G}_{m-1}}e^{\lambda(y)T_{m-1}}}\\
\nonumber&=\Ebx{1_{\{Y_{m-1}=z\}}p(z,y)\left(\int_{0}^{t-T_{m-1}}e^{-\lambda(y)(t-s)}\lambda(z)e^{-\lambda(z)s}ds\right)e^{\lambda(y)T_{m-1}}},\quad\forall z\in\s
\end{align}
Applying the change of variables $r:=s+T_{m-1}$, we find
\begin{align}
\nonumber\eqref{eq:fm7824h8a7bgn789a49a3j8g4aq}&=\Ebx{1_{\{Y_{m-1}=z\}}\lambda(z)p(z,y)\int_{T_{m-1}}^{t}e^{-\lambda(y)(t-r)}e^{-\lambda(z)(r-T_{m-1})}dr}\\
\nonumber&=\int_{0}^{t}\Ebx{1_{\{Y_{m-1}=z, T_{m-1}\leq r\}}e^{-\lambda(z)(r-T_{m-1})}}\lambda(z)p(z,y)e^{-\lambda(y)(t-r)}dr\\
\label{eq:mf8e7rag76aibr3u2al}&=\int_{0}^{t}\Pbx{\{Y_{m-1}=z,T_{m-1}\leq r< T_m \}}\lambda(z)p(z,y)e^{-\lambda(y)(t-r)}dr,\quad\forall z\in\s.
\end{align}
Putting \eqref{eq:fm7824h8a7bgn789a49a3j8g4aq00}--\eqref{eq:mf8e7rag76aibr3u2al} together and applying Tonelli's Theorem then yields the FIR~\eqref{eq:fir}.\end{proof}

\begin{lemma}[The BIR and the weak version of the time-varying equations]\label{lem:bir} Suppose that $Q$ satisfies \eqref{eq:qmatrix}. The backwards integral recursion (BIR),\index{backward integral recursion (FIR)}
\begin{equation}\label{eq:bir}p_t^n(x,y)=1_x(y)e^{-\lambda(x)t}+\int_0^t\lambda(x)e^{-\lambda(x)(t-s)}\sum_{z\in\s}p(x,z)p_s^{n-1}(z,y)ds\end{equation}
holds for all $x,y$ in $\s$, non-negative $t$, and positive $n$, where $p^n_t$ is as in \eqref{eq:pnt}. Moreover, the transition probabilities satisfy the integral version of the backward equations:
\begin{equation}\label{eq:backwardweak}p_t(x,y)=1_x(y)e^{-\lambda(x)t}+\int_0^t\lambda(x)e^{-\lambda(x)(t-s)}\sum_{z\in\s}p(x,z)p_s(z,y)ds\quad\forall x,y\in\s,\enskip t\in[0,\infty).\end{equation}
\end{lemma}

\begin{proof}Because taking the limit $n\to\infty$ in \eqref{eq:bir} and applying monotone convergence yields \eqref{eq:backwardweak}, we need only to prove \eqref{eq:bir}. We do this  by inductively showing that 
\begin{equation}\label{eq:fjd4a0n8704ahgfa}p^n_t(x,y)=u^n_t(x,y)\quad\forall x,y\in\s,\enskip t\in[0,\infty),\enskip n\in\n,\end{equation}
where $u^n_t$ is obtained by running the BIR:
$$u^n_t(x,y):=\left\{\begin{array}{ll}1_x(y)e^{-\lambda(x)t}&\text{if }n=0\\
1_x(y)e^{-\lambda(x)t}+\int_0^t\lambda(x)e^{-\lambda(x)(t-s)}\sum_{z\in\s}p(x,z)u_s^{n-1}(z,y)ds&\text{if }n>0\end{array}\right.$$
for all $x,y\in\s$ and $t\in[0,\infty)$. Clearly,
$$u^0_t(x,y)=1_x(y)e^{-\lambda(x)t}=p^0_t(x,y)\quad \forall x,y\in\s,\enskip t\in[0,\infty).$$
Furthermore, 
\begin{align*}u^1_t(x,y)&=1_x(y)e^{-\lambda(x)t}+\int_0^t\lambda(x)e^{-\lambda(x)(t-s)}\sum_{z\in\s}p(x,z)1_z(y)e^{-\lambda(z)s}ds\\
&=1_x(y)e^{-\lambda(x)t}+\int_0^t\lambda(x)e^{-\lambda(x)(t-s)}p(x,y)e^{-\lambda(y)s}ds\\
&=1_x(y)e^{-\lambda(x)t}+\int_0^te^{-\lambda(x)s'}\lambda(x)p(x,y)e^{-\lambda(y)(t-s')}ds'\\
&=1_x(y)e^{-\lambda(x)t}+\int_0^t\sum_{z\in\s}1_x(z)e^{-\lambda(x)s'}\lambda(z)p(z,y)e^{-\lambda(y)(t-s')}ds'=p^1_t(x,y),\end{align*}
where the third equality follows from the change of variables $s':=t-s$ and the fifth from the FIR \eqref{eq:fir}.  Now, suppose that \eqref{eq:fjd4a0n8704ahgfa} holds for all $n=0,1,\dots,m-1$ with $m\geq 2$. In this case,
\begin{align*}u^m_t(x,y)=&1_x(y)e^{-\lambda(x)t}+\int_0^t\lambda(x)\sum_{z\in\s}p(x,z)p^{m-1}_s(z,y)e^{-\lambda(x)(t-s)}ds\\
=&1_x(y)e^{-\lambda(x)t}+\int_0^t\lambda(x)\sum_{z\in\s}p(x,z)\Bigg(1_z(y)e^{-\lambda(y)s}\\
&+\int_0^s\sum_{z'\in\s}p_r^{m-2}(z,z')\lambda(z')p(z',y)e^{-\lambda(y)(s-r)}dr\Bigg)e^{-\lambda(x)(t-s)}ds\\
=&1_x(y)e^{-\lambda(x)t}+\lambda(x)p(x,y)\int_0^te^{-\lambda(y)(t-s')}e^{-\lambda(x)s'}ds'\\
&+\sum_{z\in\s}\sum_{z'\in\s}\lambda(x)p(x,z)\lambda(z')p(z',y)\int_0^t\int_0^{t-s'}p_{t-s'-r'}^{m-2}(z,z')e^{-\lambda(y)r'}e^{-\lambda(x)s'}dr'ds',\end{align*}
where we've used the changes of variables $r':=s-r$ and $s':=t-s$. Similarly,
\begin{align*}p^m_t(x,y)=&1_x(y)e^{-\lambda(x)t}+\int_0^t\sum_{z\in\s}u^{m-1}_s(x,z)\lambda(z)p(z,y)e^{-\lambda(y)(t-s)}ds\\
=&1_x(y)e^{-\lambda(x)t}+\int_0^t\sum_{z\in\s}\Bigg(1_x(z)e^{-\lambda(x)s}\\
&+\int_0^s\lambda(x)\sum_{z'\in\s}p(x,z')u_r^{m-2}(z',z)e^{-\lambda(x)(s-r)}dr\Bigg)\lambda(z)p(z,y)e^{-\lambda(y)(t-s)}ds\\
=&1_x(y)e^{-\lambda(x)t}+\lambda(x)p(x,y)\int_0^te^{-\lambda(y)(t-s)}e^{-\lambda(x)s}ds\\
&+\sum_{z\in\s}\sum_{z'\in\s}\lambda(x)p(x,z')\lambda(z)p(z,y)\int_0^t\int_0^{t-s'}u_{t-s'-r'}^{m-2}(z,z')e^{-\lambda(y)r'}e^{-\lambda(x)s'}dr'ds'.\end{align*}
Comparing the above two expressions, we find that \eqref{eq:fjd4a0n8704ahgfa} holds for $n=m$ and \eqref{eq:bir} follows.
\end{proof}

\begin{lemma}[The transition probabilities are continuously differentiable and satisfy the strong version of the backward equations]\label{lem:transbackward} Suppose that $Q$  satisfies \eqref{eq:qmatrix}. For each $x,y\in\s$, 
\begin{equation}\label{eq:transprobfun} t\mapsto p_t(x,y)\end{equation}
is a continuously differentiable function on $[0,\infty)$. Moreover, the backward equations \eqref{eq:backward} hold.
\end{lemma}

\begin{proof} Let $(\s_r)_{r\in\zp}$ be any sequence of finite subsets of the state space such that $\cup_{r=1}^\infty\s_r=\s$. Because
$$\sum_{z\in\s}p(x,z)p_s(z,y)\leq \sum_{z\in\s}p(x,z)=1\quad\forall x,y\in\s,\enskip s\in[0,\infty),$$
the weak backward equations \eqref{eq:backwardweak} imply that, for each $x,y\in\s$, \eqref{eq:transprobfun} is a continuous function on $[0,\infty)$. Thus, for any given $x,y\in\s$,
$$\left(s\mapsto \sum_{z\in\s_r}p(x,z)p_s(z,y)\right)_{r\in\zp}$$
is a sequence  of continuous functions on $[0,\infty)$.  Because
$$\mmag{\sum_{z\in\s}p(x,z)p_s(z,y)-\sum_{z\in\s_r}p(x,z)p_s(z,y)}=\sum_{z\not\in\s_r}p(x,z)p_s(z,y)\leq \sum_{z\not\in\s_r}p(x,z)\quad\forall s\in[0,\infty),$$
the sequence converges uniformly over $[0,\infty)$ to 
$$s\mapsto \sum_{z\in\s}p(x,z)p_s(z,y)$$
Thus, the limiting function is continuous on $[0,\infty)$. For this reason, the fundamental theorem of calculus and the weak backward equations \eqref{eq:backwardweak} imply that \eqref{eq:transprobfun} is continuously differentiable.

Next, multiplying both sides of the weak equations \eqref{eq:backwardweak} by $e^{\lambda(x)t}$, taking derivatives, and applying the fundamental theorem of calculus yields
$$\dot{p}_t(x,y)e^{\lambda(x)t}+p_t(x,y)\lambda(x)e^{\lambda(x)t}=\lambda(x)e^{\lambda(x)t}\sum_{z\in\s}p(x,z)p_t(z,y)\quad\forall t\in[0,\infty).$$
Multiplying through by $e^{-\lambda(x)t}$, re-arranging, and using \eqref{eq:jumpmatrix}--\eqref{eq:lambda}, we obtain~\eqref{eq:backward}.
\end{proof}

\begin{lemma}[The transition probabilities satisfy the strong version of the forward equations]\label{lem:transforward}Suppose that $Q$  satisfies \eqref{eq:qmatrix}. The transition probabilities satisfy the forward equation \eqref{eq:forward}.
\end{lemma}

\begin{proof}Once we show that, for every $x,y\in\s$, the function
\begin{equation}\label{eq:integrand}s\mapsto \sum_{z\in\s}p_s(x,z)\lambda(z)p(z,y)\end{equation} 
is a continuous on $[0,\infty)$, the remainder of the proof is entirely analogous to that of Lemma~\ref{lem:transbackward} and we skip the details.

To show that \eqref{eq:integrand} is continuous on $[0,\infty)$ it suffices to show that it is continuous on $[0,t]$ for each $t\in[0,\infty)$. To do so, let $(\s_r)_{r\in\zp}$ be any sequence of finite subsets of the state space such that $\cup_{r=1}^\infty\s_r=\s$ and consider the sequence of functions
\begin{equation}\label{eq:nfd7awe8hfa8j3ad2qaa}\left(s\mapsto \sum_{z\in\s_r}p_s(x,z)\lambda(z)p(z,y)\right)_{r\in\zp}\end{equation}
converging pointwise to \eqref{eq:integrand}. Because the subsets $\s_r$ are finite and the transition probabilities are continuous (Lemma~\ref{lem:transbackward}), each of the functions in the sequence is continuous. For this reason, we need only to show that the convergence is uniform over  $s\in[0,t]$. As we show at the end of the proof, 
\begin{equation}\label{eq:fhn7e8ahfnea7w8fhne7a8enf7aw8}\dot{p}_t(x,y)\geq - \lambda(x)\quad\forall t\in[0,\infty),\enskip x,y\in\s.\end{equation}
Multiplying through by $e^{\lambda(x)t}$, applying the product rule, and integrating shows that $t\mapsto e^{\lambda(x)t}p_t(x,y)$ is a non-decreasing function on $[0,\infty)$. The uniform convergence follows as
\begin{align*}&\mmag{\sum_{z\in\s}p_s(x,z)\lambda(z)p(z,y)-\sum_{z\in\s_r}p_s(x,z)\lambda(z)p(z,y)}=\sum_{z\not\in\s_r}p_s(x,z)\lambda(z)p(z,y)\\
&\leq \sum_{z\not\in\s_r}e^{\lambda(x)s}p_s(x,z)\lambda(z)p(z,y)\leq  \sum_{z\not\in\s_r}e^{\lambda(x)t}p_t(x,z)\lambda(z)p(z,y)\quad\forall s\in[0,t]\end{align*}
and the right-hand side converges to zero as $r\to\infty$, otherwise
\begin{align*}\sum_{z\in\s}p_s(x,z)\lambda(z)p(z,y)&=e^{-\lambda(x)s}\sum_{z\in\s}e^{\lambda(x)s}p_s(x,z)\lambda(z)p(z,y)\\
&\geq e^{-\lambda(x)s}\sum_{z\in\s}e^{\lambda(x)t}p_t(x,z)\lambda(z)p(z,y)=\infty\quad\forall s\in[t,\infty)\end{align*}
contradicting the weak forward equations \eqref{eq:forwardweak}.

We have one lose end to tie up: proving \eqref{eq:fhn7e8ahfnea7w8fhne7a8enf7aw8} or, more generally, that
\begin{equation}\label{eq:dne68wa7ndeywa7hdfa8}\mmag{\dot{p}_t(x,y)}\leq q(x)\quad\forall t\in[0,\infty),\enskip x,y\in\s.\end{equation}
If $q(x)=0$, the claim is trivial as Proposition~\ref{prop:absorb} shows that $p_t(x,y)=1_x(y)$ for all $t\in[0,\infty)$. Otherwise, the semigroup property \eqref{eq:semigroup} implies that
$$p_{t+h}(x,y)-p_t(x,y)=\sum_{z\in\s}p_h(x,z)p_t(z,y)-p_t(x,y)=(p_h(x,x)-1)p_t(x,y)+\sum_{z\neq x}p_h(x,z)p_t(z,y).$$
Because $p_t(x,y)\in[0,1]$ and $\sum_{z\in\s}p_h(x,z)\leq1$ for all $x,y\in\s$, it is not difficult to see that the two terms on the right-hand side have opposite signs and are bounded by $1-p_h(x,x)$. Thus,
\begin{equation}\label{eq:fj048fn4m78aht8473aga}\mmag{p_{t+h}(x,y)-p_t(x,y)}\leq 1-p_h(x,x)\quad\forall x,y\in\s.\end{equation}
By definition,
\begin{align*}p_h(x,x)&=\Pbx{\{X_h=x,h<T_\infty\}}\geq \Pbx{\{X_h=x,h<T_1\}}=\Pbx{\{X_0=x,h<T_1\}}\\
&=\Pbx{\{h<T_1\}}=e^{-q(x)h}\geq1-hq(x),\end{align*}
and \eqref{eq:dne68wa7ndeywa7hdfa8} follows by plugging the above into \eqref{eq:fj048fn4m78aht8473aga}, dividing through by $h$, and taking the limit $h\to0$.

\end{proof}

\begin{lemma}[The transition probabilities are the minimal non-negative solution of the forward and backward equations]\label{lem:minimal}Suppose that $Q$  satisfies \eqref{eq:qmatrix}, that, for each $x,y\in\s$,
$$t\mapsto k_t(x,y)$$
is a non-negative ($k_t(x,y)\geq0$ for all $t\in[0,\infty)$) differentiable function on $[0,\infty)$. If the collection of these functions satisfies the forward equations,
$$\dot{k}_t(x,y)=\sum_{z\in\s}k_t(x,z)q(z,y)\quad\forall t\in[0,\infty),\enskip x\in\s,\quad k_0(x,y)=1_x(y)\quad \forall x,y\in\s,$$ 
or the backward equations,
$$\dot{k}_t(x,y)=\sum_{z\in\s}q(x,z)k_t(z,y)\quad\forall t\in[0,\infty),\enskip x,y\in\s,\quad k_0(x,y)=1_x(y)\quad \forall x,y\in\s,$$ 
then 
$$k_t(x,y)\geq p_t(x,y)\quad\forall x,y\in\s,\enskip t\in[0,\infty).$$

\end{lemma}

\begin{proof} Given the BIR \eqref{eq:bir}, the proof for the case of the backward  equations is entirely analogous to that for the case of the forward equations and we focus on the latter. Suppose that we are able to show that $(k_t(x,y))_{x,y\in\s}$ satisfies the weak version of the forward equations, i.e.,
\begin{equation}\label{eq:hj8f7a4na84hfa87hnf4aj8a9-}k_t(x,y)=1_x(y)e^{-\lambda(y)t}+\int_0^t\sum_{z\in\s}k_s(x,z)\lambda(z)p(z,x)e^{-\lambda(y)(t-s)}ds\quad \forall t\in[0,\infty),\enskip x\in\s,\end{equation}
and let $p^n_t$ be as in \eqref{eq:pnt}. Because the above and \eqref{eq:fm7824h8a7bgn789a49a3j8g4aq0} imply that
$$k_t(x,y)\geq 1_x(y)e^{-\lambda(y)t}=p^0_t(x,y)\quad\forall t\in[0,\infty),\enskip x,y\in\s,$$
combining the FIR \eqref{eq:fir} and \eqref{eq:hj8f7a4na84hfa87hnf4aj8a9-} we find that
\begin{align*}k_t(x,y)&=1_x(y)e^{-\lambda(y)t}+\int_0^t\sum_{z\in\s}k_s(x,z)\lambda(z)p(z,y)e^{-\lambda(y)(t-s)}ds\\
&\geq 1_x(y)e^{-\lambda(y)t}+\int_0^t\sum_{z\in\s}p_s^0(x,z)\lambda(z)p(z,y)e^{-\lambda(y)(t-s)}ds=p^1_t(x,y)\quad\forall t\in[0,\infty),\enskip x\in\s.\end{align*}
Iterating the above argument forward shows that 
$$k_t(x,y)\geq p^n_t(x,y)\quad\forall n\in\n,\enskip t\in[0,\infty),\enskip x,y\in\s.$$
Taking the limit $n\to\infty$ and applying monotone convergence then proofs the minimality of $p_t$.

We have one lose end remaining: proving \eqref{eq:hj8f7a4na84hfa87hnf4aj8a9-}. If $q(y)=0$ this is trivial. Otherwise, $\lambda(y)=q(y)$ and
\begin{align*}\frac{d}{ds}(k_s(x,y)e^{\lambda(y)s})&=\dot{k}_s(x,y)e^{\lambda(y)s}+\lambda(y)k_s(x,y)e^{\lambda(y)s}=\sum_{z\neq y}k_s(x,z)q(z,y)e^{\lambda(y)s}\\
&=\sum_{z\in\s}k_s(x,z)\lambda(z)p(z,y)e^{\lambda(y)s}\quad\forall s\in[0,\infty),\enskip x,y\in\s.\end{align*}
Integrating over $s\in[0,t]$, applying the fundamental theorem of calculus, and multiplying through by $e^{-\lambda(y)t}$ then yields \eqref{eq:hj8f7a4na84hfa87hnf4aj8a9-}.
\end{proof}

\ifdraft
\subsubsection*{Notes:} These recursions are due to Feller 1940. The proof of Lemma \ref{lem:bir} is taken from \citep[p.71]{Anderson1991}
\fi

\ifdraft

\subsection{Stochastic reaction networks}\label{SRNs}\glsadd{SRN}

USE FORWARD EQUATIONS TO EXPLAIN INTUITION.

\index{stochastic reaction network}The original motivation that spurred us to develop the methods presented in this thesis was the study of a class of continuous-time chains (and their associated jump chains) which we refer to as \emph{stochastic reaction networks} or \emph{SRNs} for short. Reaction networks play a prominent role in chemical physics \citep{kampen2007} and in a variety of biological disciplines including systems biology, synthetic biology, epidemiology, pharmacokinetics, ecology, and neuroscience \citep{Goutsias2013}. They describe the evolution of a population of interacting species, such as various types of molecules undergoing chemical reactions in a test tube. Traditionally, the dynamics of the concentration, or density, of each species is modelled with an ordinary differential equation (ODE). These deterministic models are unreliable when the copy number (that is, the number of individuals) of at least one species is low, a situation common in biological settings such as the expression of genes inside a cell's cytoplasm. In this case, there is no law of large numbers available that can be used to approximate the species' inherently stochastic behaviour with a deterministic one (see the work of Thomas G. Kurtz for more on these approximations \citep{Ethier1986,Kurtz1970,Kurtz1971,Kurtz1976,Kurtz1978}). For this reason, stochastic descriptions of reaction networks, and in particular SRNs, have become popular models of biochemical and intracellular processes. 

Motivated by the chemical physics modelling applications of SRNs we refer to the species in the network as \emph{chemical species}, individual members as \emph{molecules}, and the interactions between them as \emph{reactions}.

The starting point of a SRN is a set of $n$ chemical species $\{S_1,\dots,S_n\}$ and a set of $m$ reactions $\{R_1,\dots,R_m\}$ that convert molecules of some of the species into molecules of some the other species. The reactions are described using schematics of the sort
$$ R_j:\quad v_{1j}^-S_1+\dots+v_{nj}^-S_n \xrightarrow{a_j} v_{1j}^+S_1+\dots+v_{nj}^+S_n, \qquad j=1,\ldots,m, $$
where $v_{ij}^-$ (resp. $v_{ij}^+$) denotes the number of molecules of species $S_i$ consumed (resp. produced) every time reaction $R_j$ occurs. If no molecules are consumed (resp. produced) by the reaction, then we write $\varnothing$ in the left-hand side (resp. right-hand side) of the above. The SRN $X:=(X^1,\dots,X^n)$ models the molecule counts of each of the species; that is, $X^i_t$ denotes the number of molecules of species $S_i$ at time $t$ for each $i=1,\dots,n$. With each reaction, we associate a \emph{reaction vector} $v_j:=(v_{ij}^+-v_{ij}^-)_{i=1,\dots,n}$ whose entries are the net change in the number of molecules of each species precipitated by the reaction and a non-negative function $a_j$ on $\nn$ called the \emph{propensity}\index{propensity} or the \emph{reaction rate}. If reaction $R_j$ occurs at time $t$, then the molecule counts get updated as
$$X_t:= X_{t-}+v_j,$$
where $X_{t-}$ denotes the molecule counts right before the reaction occurred. The propensity $a_j$ describes the rate at which the reaction $R_j$ occurs. Specifically, if at the beginning of some small interval of time of length $h$ there are $x_i$ molecules of $S_i$, for each $i=1,\dots,n$, then the probability that $R_j$ occurs at some point during the interval is approximately $a_j(x)h$. The most prevalent type of propensities in the literature are of the \emph{mass-action} type:
$$a_j(x):=k_j\prod_{i=1}^n\frac{x_i(x_i-1)\dots(x_i-u_{ij}+1)}{u_{ij}!},$$
where $k_j>0$ are some known \emph{reaction constants}. We say that a SRN with the above type of propensities is endowed with \emph{mass-action kinetics}\index{mass-action kinetics/propensities}. Note that $a_j/k_j$ as defined above is merely the number of combination possible of the reactants of $j$th reaction given $x_1$ molecules of $S_1$, $x_2$ of $S_2$, ..., and $x_n$ of $S_n$. For this reason, choosing mass-action propensities amounts to assuming that the probability that the $j$th-reaction occurs over a small time period is proportional to the number of combinations possible of the reaction's reactants. If the medium (test tube, cell's cytoplasm, etc.) in which the molecules exist is spatially uniform and well-mixed, then this is a reasonable assumption. The popularity of mass-action kinetics stems from these type of physical arguments presented in \citep{Gillespie1977} in the context of chemical reaction networks.

In this thesis, we do not limit ourselves to mass-action kinetics, the schemes in discussed in Chapters \ref{fspchap} and \ref{dists} apply to arbitrary propensities while those in Chapter \ref{mcmom} apply to propensities that are rational functions, often obtained by carrying out a time-scale separation type approximation of a SRN endowed with mass-action kinetics. We do however always assume that the state space $\s\subseteq\nn$ is chosen such that for any $x\in \s$
$$x+v_j\not\in \s\Rightarrow a_j(x)=0,\qquad \forall j=1,\dots,m.$$
This guarantees that the only way $X$ can leave $\s$ is by exploding.

Assuming that, conditioned on the current molecule counts, the reactions occur independently of each other, it can be shown that the SRN $X$ is a continuous-time chain with rate matrix $Q=(q(x,y))_{x,y\in\s}$ given by
\begin{equation}
\label{eq:qmatrixsrn}
q(x,y):=\left\{\begin{array}{cc}a_1(x)&\text{if } y=x+v_1\\ a_2(x)&\text{if } y=x+v_2\\ \vdots&\vdots\\a_m(x)&\text{if } y=x+v_m\\ 0&\text{otherwise}\end{array}\right.,
\end{equation}
see \citep{Anderson2015,Gillespie1976,Gillespie1977}.
\index{toggle switch}\begin{example}[Genetic toggle switch]\label{togintro} The following SRN with state space $\n^2$ is often used to model the expression two mutually repressing genes. So called toggle-switches are common motifs in many cell fate-decision genetic circuits \citep{gardner2000,hemberg2007,perez2016}. In particular, we consider the asymmetric case 
$$
\varnothing \xrightleftharpoons[a_2]{a_1} P_1,
 \qquad \varnothing \xrightleftharpoons[a_4]{a_3} P_2,
$$
with propensities
\begin{equation}
 \label{toggle:prop}
 a_1(x) = \frac{k_1 \theta^3}{\theta^3+x_{2}^3},  \qquad
 a_2(x) = k_2 x_{1},  \qquad
 a_3(x) = \frac{k_3}{1+x_{1}}, \qquad
 a_4(x) = k_4 x_{2},
\end{equation}
where $k_1,k_2,k_3,k_4,\theta>0$, which models mutual repression via Hill functions and dilution via linear unspecific decays, where $x=(x_1,x_2)$ and $x_1$ (resp. $x_2$) denotes the copy number of protein $P_1$ (resp. $P_2$). 
\end{example}
\subsubsection*{Birth-death processes:}\index{birth-death process} Birth-death processes are SRNs with state space $\n$ that model the number of individuals in a single population over time. This number is only allowed to increase by $1$ (when a birth occurs) or to decrease by $1$ (when a death occurs). Birth/death events do not have to represent a literal birth or death of a member of the species, it could represent the immigration/emigration of a member of the species, or other events that change the population count by $\pm1$. Indeed, birth-death processes are employed to model a variety of situations aside from the demographics of a given species: for instance, the size of a queue in a classic M/M/m queueing model \citep{Asmussen2003}, a chemical species undergoing a phase transition \citep{Schlogl1972}, or gene expression inside a living cell \citep{Dattani2017}.

The rate matrix of birth-death processes has the following tri-diagonal structure
\begin{equation}\label{eq:bdqmatrx}q(x,y):=\left\{\begin{array}{l l}a_-(x)&\text{if }y=x-1,\\ -(a_-(x)+a_+(x))&\text{if }y=x,\\ a_+(x)&\text{if }y=x+1,\\0&\text{otherwise}\end{array}\right.,\end{equation}
where $a_-:\n\to[0,\infty)$ and $a_+:\n\to[0,\infty)$ denote the \emph{birth rate} and \emph{death rate} respectively. The death rate satisfies $a_-(0)=0$ formalising the idea that if there are no members of the species, there cannot be any deaths.
\index{Schl\"ogl's model}\begin{example}[Schl\"ogl's model]
\label{ex:schloegl}
The following SRN is was suggested by Schl\"ogl~\citep{Schlogl1972} as model of a chemical phase transition involving a single species $A$: 
\begin{equation}
\label{eq:smdl}
2A\xrightleftharpoons[a_2]{a_1}3A,
\qquad \varnothing\xrightleftharpoons[a_2]{a_1}A,
\end{equation}
with propensities
$$a_1(x):=k_1 x (x-1),\qquad a_2(x):=k_2 x,\qquad a_3(x):=k_3,\qquad a_4(x):=k_4 x,$$
where $k_1, k_2, k_3, k_4$ are positive constants. The model consists of a birth-death process with rates
\begin{equation}\label{eq:schoglbdr}a_+(x):= k_1 x (x-1)+k_3,\qquad a_-(x):=k_2 x (x-1) (x-2)+k_4 x.\end{equation}
\end{example}
\fi

\ifdraft
\subsection{\textbf{On uniqueness of continuous-time chains**}}

1) Path law of Markov processes on $\s$ (together with a final-explosion-time-stopping-time) is one-to-one with (substochastic semigroups) on $\s$.

2) $\dot{P}(0)$ exists and is a $Q$-matrix, introduce idea of $Q$-process

3) Or $P_t$ is the minimal $Q$-function (th.2.2 chap 2 anderson). Thus, $Q$-process is unique (w.r.t path law) iff minimal $Q$-process is honest.

4) Otherwise we can build others using re-starts (cor.2.5 chap 2 anderson).

5) Only restarts that don't any one state in particular satisfy forward equations (give flash, reference to chap 4 anderson).

6) We don't consider non-conservative, we view this as a modelling mistake 

\fi

\subsection{The time-varying law and its differential equation}\label{sec:forward}

Let\glsadd{pt} 
\begin{equation}\label{eq:ctlawdef}p_t(x):=\Pbl{\{X_t=x,t<T_\infty\}}\end{equation}
denote the probability that the chain is at state $x$ at time $t$ if its starting state was sampled from  $\gamma$. Theorem \ref{thrm:findim} gives an explicit expression for the \emph{time-varying law}\index{time-varying law} $(p_t)_{t\geq0}$ in terms of the initial distribution $\gamma$ and $t$-transition matrix $P_t=(p_t(x,y))_{x,y\in\s}$,
\begin{equation}\label{eq:tstep} p_t(x)=\sum_{x'\in\s}\gamma(x')p_t(x',x)\quad \forall x\in\s,\end{equation}
or $p_t=\gamma P_t$ in matrix notation, for all $t$ in $[0,\infty)$. While theoretically useful, \eqref{eq:tstep} is of little practical use for computing the time-varying law of a chain. A far more useful description in this respect---and one very popular among physicists, engineers, and mathematical biologists---is the  following a generalisation of the forward equations \eqref{eq:forward}.
\begin{theorem}[Analytical characterisation of the time-varying law]\label{thrm:forward} Suppose that $Q$  satisfies \eqref{eq:qmatrix} and that
%
\begin{equation}\label{eq:forwardcond}\sum_{x\in\s}\gamma(x)\left(\sup_{s\in[0,t]}\sum_{x'\neq y}p_s(x,x')q(x',y)\right)<\infty,\quad\forall y\in\s,\enskip t\in[0,\infty).\end{equation}
The time-varying law is continuously differentiable: for each $x$ in $\s$,
$$t\mapsto p_t(x)$$
is a continuously differentiable function on $[0,\infty)$. Moreover, it satisfies the equations\index{forward equations}
\begin{equation}\label{eq:master}\dot{p}_t(x)=\sum_{x'\in\s}p_t(x')q(x',x)\quad \forall t\in[0,\infty),\enskip x\in\s,\qquad p_0(x)=\gamma(x)\quad\forall x\in\s,\quad\end{equation}
or 
$$\dot{p}_t=p_tQ\quad\forall t\in[0,\infty),\quad p_0=\gamma,$$
in matrix notation. Moreover, the time-varying law is the minimal non-negative solution of these equations: if $(k_t(x))_{x\in\s,t\in[0,\infty)}$ is a non-negative ($k_t(x)\geq0$ for all $x\in\s,$ $t\in[0,\infty)$) differentiable function satisfying \eqref{eq:master}, then
$$k_t(x)\geq p_t(x)\quad\forall x\in\s,\enskip t\in[0,\infty).$$
\end{theorem}
For the theorem's proof, see the end of the section. Condition \eqref{eq:forwardcond} may seem very technical, but it is straightforward to derive conditions that cover the majority of chains encountered in practice:
\begin{proposition}[Condition \eqref{eq:forwardcond} is mild]\label{prop:forwardcondmild} Suppose that the rate matrix $Q$ satisfies \eqref{eq:qmatrix}. If the rate matrix has bounded columns:
$$\sup_{x\in\s}q(x,y)<\infty\quad\forall y\in\s,$$
or if it's diagonal is $\gamma$-integrable:
\begin{equation}\label{eq:qgamint}\Ebl{q(X_0)}=\sum_{x\in\s}q(x)\gamma(x)<\infty,\end{equation}
then \eqref{eq:forwardcond} holds.
\end{proposition}

\begin{proof}Given that $\sum_{x'\in\s}p(x,x')\leq 1$ for all $s$ in $[0,\infty)$, the case of bounded columns is trivial. To argue the case of \eqref{eq:qgamint}, recall that the transition probabilities satisfy the forward equation~\eqref{eq:forward}. For this reason,
$$
\sum_{x'\neq y}p_s(x,x')q(x',y)\leq p_s(x,y)q(y)+\mmag{\dot{p}_s(x,y)}\leq q(y)+\mmag{\dot{p}_s(x,y)}\quad\forall s\in[0,\infty),\enskip x,y\in\s.$$
Because the derivative $\dot{p}_t(x,y)$ is bounded by $q(x)$ (c.f.~\eqref{eq:dne68wa7ndeywa7hdfa8}), the result follows from the above.
\end{proof}

\subsubsection*{Explosions and uniqueness of solutions}For any given initial distribution $\gamma$ satisfying~\eqref{eq:forwardcond}, Proposition~\ref{prop:nonexptimevar} and Theorem~\ref{thrm:forward} show that the time-varying law is the unique solution $(k_t)_{t\geq0}$ of~\eqref{eq:master} with mass no greater than one (i.e., such that $p_t(\s)\leq 1$ for all $t$ in $[0,\infty)$) if sampling the starting state from $\gamma$ results in chain that does not explode:
$$\Pbl{\{T_\infty=\infty\}}=1.$$
If the above is not satisfied, then \eqref{eq:master} may have more than one solution for the same reasons that the forward equation~\eqref{eq:forward} may have more than one solution \citep{Anderson1991}.

\subsubsection*{Some open problems}To the best of my knowledge, the question of whether equations~\eqref{eq:master} hold for all initial distributions $\gamma$ and chains with stable and conservative rate matrices $Q$ 
remains unresolved. If one is willing to settle for a weak version of these equations the answer is affirmative: integrating the forward equations~\eqref{eq:forwardweak}, re-arranging, multiplying through by $\gamma(x)$, and summing over $x$ in $\s$, we find that
\begin{equation}\label{eq:ndw89afnwanfaw}p_t(x)-\gamma(x)+\int_0^tp_s(x)q(x)ds=\int_0^t\sum_{x'\neq x}p_s(x')q(x',x)ds\quad\forall t\in[0,\infty),\enskip x\in\s.\end{equation}
Because $\int_0^tp_s(x)q(x)ds\leq tq(x)<\infty$, it follows that 
$$\sum_{x'\neq x}p_t(x')q(x',x)<\infty\text{ for Lebesgue almost every }t\in[0,\infty)\text{ and all }x\in\s.$$
For this reason, the Lesbegue differentiation theorem \citep[Theorem~1.6.11]{Tao2011} and~\eqref{eq:ndw89afnwanfaw} show that, for all $x$ in $\s$, $t\mapsto p_t(x)$ is differentiable Lebesgue almost everywhere and 
\begin{equation}\label{eq:dens1}\dot{p}_t(x)=p_t Q(x)\text{ for Lebesgue almost every }t\in[0,\infty)\text{ and all }x\in\s.\end{equation}
%

What condition~\eqref{eq:forwardcond} does is ensure that the integrands in~\eqref{eq:ndw89afnwanfaw} are finite and continuous, in which case it follows from the fundamental theorem of calculus that $t\mapsto p_t(x)$ is a continuously differentiable function satisfying $\dot{p}_t(x)=p_t Q(x)$ for all $t$ in $[0,\infty)$. Indeed, were $p_t(x)$ to be a continuously differentiable function on $[0,\infty)$ for all $x$ in $\s$, it would follow that
\begin{align*}\sup_{s\in[0,t]}\sum_{x\in\s}\gamma(x)\sum_{x'\neq y}p_s(x,x')q(x',y)&\leq\sup_{s\in[0,t]}\sum_{x'\neq y}p_s(x')q(x',y)\\
&\leq \sup_{s\in[0,t]}(\dot{p}_s(y)+p_s(y)q(y))<\infty,\quad\forall y\in\s,\enskip t\in[0,\infty)\end{align*}
as continuous functions are bounded over finite intervals. This is close to~\eqref{eq:forwardcond} but not quite, and the question lingers: is \eqref{eq:forwardcond} necessary for the time-varying law to be continuously differentiable on $[0,\infty)$?

Now, notice that there is nothing stopping us from picking a $Q$ and $\gamma$ such that $\gamma Q(x)=\infty$ for at least one $x$. In this case, \eqref{eq:forwardcond} will not be satisfied and neither will the conclusions of Theorem~\ref{thrm:forward} ($p_t(x)$ cannot both be differentiable on $[0,\infty)$ and satisfy $\dot{p}_0(x)=p_0Q(x)=\gamma Q(x)  =\infty$). However, tweaking the argument in the theorem's proof, it is possible to show that the time-varying law is continuous differentiable on $(0,\infty)$ and~\eqref{eq:master} holds if we replace~\eqref{eq:forwardcond} with
\begin{equation}\label{eq:dnwy8abd6a78wda}\sum_{x\in\s}\gamma(x)\left(\sup_{s\in[1/t,t]}\sum_{x'\neq y}p_s(x,x')q(x',y)\right)<\infty,\quad\forall y\in\s,\enskip t\in[1,\infty).\end{equation}
Whether the above is actually necessary for us to be able to strengthen \eqref{eq:ndw89afnwanfaw} into `for all $x$ in $\s$, $t\mapsto p_t(x)$ is continuously differentiable (or even just differentiable) on $(0,\infty)$ and~\eqref{eq:master} holds' also remains open. It may well be that this is the case for all initial distributions and chains with stable and conservative rate matrices regardless of whether \eqref{eq:dnwy8abd6a78wda} is satisfied.

Lastly, for the cases not covered by Proposition~\ref{prop:forwardcondmild} we are lacking in tools to verify  \eqref{eq:forwardcond}~or~\eqref{eq:dnwy8abd6a78wda} in practice. If either of these conditions prove to actually be important, the question of whether they can be rephrased entirely in terms of $\gamma$ and $Q$ is also worth having a look at. If you know anything more on these issues, please get in touch.

\subsubsection*{A proof of Theorem~\ref{thrm:forward}} The argument showing that the time-varying law $p_t$ is minimal among the set of solutions of the master equation \eqref{eq:master} is entirely analogous to those in the proofs of Lemmas~\ref{lem:forwardweak}~and~\ref{lem:minimal}. Thus, we skip it and instead focus on proving that $p_t$ is continuously differentiable and satisfies \eqref{eq:master}.

Given~\eqref{eq:ndw89afnwanfaw} and the fundamental theorem of calculus, it suffices to show that
\begin{equation}\label{eq:mfd8e9awp0ngfe78aong78agbeyaw8ghbeuaw}t\mapsto p_t(x)\quad\text{and}\quad t\mapsto \sum_{x'\neq y}p_t(x')q(x',y)\end{equation}
are continuous  functions $[0,\infty)$ for all $y$ in $\s$. To do so, let $(\s_r)_{r\in\zp}$ be any given sequence of finite subsets of the state space such that $\cup_{r=1}^\infty\s_r=\s$. By equation~\eqref{eq:tstep}, we have that $t\mapsto p_t(x)$ is the pointwise limit of the sequence
$$\left(t\mapsto \sum_{x'\in\s_r}\gamma(x')p_t(x',x)\right)_{r\in\zp}.$$
Because the sets $\s_r$ are finite, the continuity of the transition probability (Theorem~\ref{thrm:forwardbackward}) implies that these functions are continuous. Because the convergence is uniform over $t$s in $[0,\infty)$,
$$\mmag{p_t(x)-\sum_{x'\in\s_r}\gamma(x)p_t(x',x)}=\sum_{x'\not\in\s_r}\gamma(x)p_t(x',x)\leq\gamma(\s_r)\quad\forall r>0,$$
it follows that $t\mapsto p_t(x)$ is continuous on $[0,\infty)$. 

To prove the continuity of the other function in~\eqref{eq:mfd8e9awp0ngfe78aong78agbeyaw8ghbeuaw}, notice that it is the pointwise limit of the sequence
$$\left(t\mapsto \sum_{x\in\s_r}\sum_{x'\neq y}\gamma(x)p_t(x,x')q(x',y)\right)_{r\in\zp}$$
of continuous functions (continuity of the above also follows from Theorem~\ref{thrm:forwardbackward}). Because~
\begin{align*}\mmag{\sum_{x'\neq y}p_s(x')q(x',y)-\sum_{x\in\s_r}\sum_{x'\neq y}\gamma(x)p_s(x,x')q(x',y)}&=\sum_{x\not\in\s_r}\sum_{x'\neq y}\gamma(x)p_s(x,x')q(x',y)\\
&\leq \sum_{x\not\in\s_r}\gamma(x)\left(\sup_{s\in[0,t]}\sum_{x'\neq y}p_s(x,x')q(x',y)\right),\end{align*}
our assumption~\eqref{eq:forwardcond} implies that the convergence is uniform over $s$ in $[0,t]$, for every interval $[0,t]$, and it follows that the rightmost function in~\eqref{eq:mfd8e9awp0ngfe78aong78agbeyaw8ghbeuaw} is continuous on $[0,\infty)$.

\subsection{Dynkin's formula} \label{sec:dynkin} Integrating the generalised forward equation~\eqref{eq:master} over $[0,t)$ we find that
$$p_t(x)=\gamma(x)+\int_0^tp_sQ(x)ds\quad\forall x\in\s.$$
If $f$ is a real-valued function $\s$ satisfying the appropriate integrability conditions, then multiplying both sides by $f(x)$ and summing over $x\in\s$, we obtain the following the integral version of~\eqref{eq:master}:
\begin{equation}\label{eq:dynkin0}\Ebl{f(X_t)1_{\{t<T_\infty\}}}=\gamma(f)+\Ebl{\int_0^{t\wedge T_\infty}Qf(X_s)ds},
\end{equation}
where $Qf(x):=\sum_{x'\in\s}q(x,x')f(x)$. The aim of this section is to prove a version of Dynkin's formula which states that \eqref{eq:dynkin0} holds not only for deterministic times $t$ for stopping times $\eta$. 

\subsubsection*{A special class of stopping times}To not complicate matters unnecessarily, 
we only prove Dynkin's formula a special type of stopping times that I call \emph{jump-time-valued}:
\begin{definition}[Jump-time-valued stopping times]\label{def:stopct2}Let $(\cal{F}_t)_{t\geq0}$ denote the filtration generated by the chain (Definition~\ref{def:filt3}) and  $(\cal{G}_n)_{n\in\n}$ that generated by the jump chain and jump times (Definition~\ref{def:filt2}). We say that a random variable $\eta:\Omega\to[0,\infty]$  is a jump-time-valued $(\cal{F}_t)_{t\geq0}$-stopping time if there exists a $(\cal{G}_n)_{n\in\n}$-stopping time $\varsigma$ (Definition~\ref{def:stopdt} with $(\cal{G}_n)_{n\in\n}$ replacing $(\cal{F}_n)_{n\in\n}$ therein) such that
\begin{equation}\label{eq:stopdc}\eta(\omega)=\left\{\begin{array}{ll}T_{\varsigma(\omega)}(\omega)&\text{if }\varsigma(\omega)<\infty\\\infty&\text{if }\varsigma(\omega)=\infty\end{array}\right.\quad\forall\omega\in\Omega.\end{equation}
\end{definition}
For any such $\eta$, 
\begin{equation}\label{eq:etatinfinf}\{\eta<\infty\}=\{\eta<T_\infty\}=\{\varsigma<\infty\}\end{equation}
because the jump times are all finite by their definition in~\eqref{eq:cpathdef2}. Moreover these random times are $(\cal{F}_t)_{t\geq0}$-stopping times (in the sense of Definition~\ref{def:stopct}):
\begin{theorem}\label{thrm:stopdc}All jump-time-valued $(\cal{F}_t)_{t\geq0}$-stopping times (Definition~\ref{def:stopct2}) are $(\cal{F}_t)_{t\geq0}$-stopping times (Definition~\ref{def:stopct}),  where $(\cal{F}_t)_{t\geq0}$ denotes the filtration generated by the chain.\end{theorem}
\begin{proof}For any $\eta$ satisfying \eqref{eq:stopdc},
$$\{\eta\leq t\}=\bigcup_{n=0}^\infty \{\varsigma=n\}\cap\{T_n\leq t\}\quad\forall t\in[0,\infty).$$
Because $\{\varsigma=n\}$ belongs to $\cal{G}_n$ (Lemma~\ref{lem:2stop}$(i)$) and $\cal{G}_n\subseteq\cal{F}_{T_n}$ (Proposition~\ref{prop:ftngn}), the above shows that $\eta$ is a $(\cal{F}_t)_{t\geq0}$-stopping time.
\end{proof}
\subsubsection*{An open problem} The idea behind jump-time-valued stopping times $\eta$ is that they are stopping times (Definition~\ref{def:stopct}) taking values in  the set of jump times (and plus infinity):
\begin{equation}\label{eq:exitjump}\eta(\omega)\in\{T_0(\omega),T_1(\omega),T_2(\omega),\dots\}\cup\{\infty\}\quad\forall \omega\in\Omega.\end{equation}
In other words, the event that occurs at such a stopping time is one that may only occur when the chain jumps. It seems to me that any stopping time satisfying the above should be of the form in Definition~\ref{def:stopct2}. However, the question lingers and its answer hinges on that to the open question of Section~\ref{sec:filtct}: `does $\cal{G}_n=\cal{F}_{T_n}$?'. Indeed, any $(\cal{F}_t)_{t\geq0}$-stopping time $\eta$ satisfying the above satisfies \eqref{eq:stopdc} with 
\begin{equation}\label{eq:exitjump2}\varsigma(\omega):=\left\{\begin{array}{ll}n&\text{if }\eta(\omega)=T_{n}(\omega)\enskip\forall n\in\n\\\infty&\text{if }\eta(\omega)=\infty\end{array}\right.\quad\forall\omega\in\Omega.\end{equation}
Were $\cal{F}_{T_n}$ to coincide $\cal{G}_n$ for each $n$, $\varsigma$ would be a $(\cal{G}_n)_{n\in\n}$-stopping time as Lemma~\ref{lem:etatheta}$(i)$ implies that  $\{\varsigma\leq n\}=\{\eta\leq T_n\}$ belongs to $\cal{F}_{T_n}$ for all $n$.
\subsubsection*{An important example: hitting times}The \emph{hitting time}\index{hitting time} (or \emph{first passage time})\glsadd{taua} $\tau_A$ of a set $A\subseteq\s$ is the first time that the chain $X$ visits the subset (or plus infinity if it never does). Formally,
\begin{equation}\label{eq:hitthec}\tau_{A}(\omega):=\inf\{t\in[0,T_\infty(\omega)):X_t(\omega)\in A\}\quad\forall \omega\in\Omega.\end{equation}
We use the shorthand $\tau_x:=\tau_{\{x\}}$ for any $x$ in $\cal{S}$.  Of course, the first time that the chain enters $A$ is the moment that it jumps into $A$ and it follows that $\tau_A$ is a jump-time-valued stopping time:
\begin{proposition}\label{prop:hitdc} The hitting time $\tau_A$ is a jump-time-valued $(\cal{F}_t)_{t\geq0}$-stopping time. Furthermore, $X_{\tau_A}$ belongs to $A$ if $\tau_A$ is finite and \eqref{eq:stopdc}--\eqref{eq:etatinfinf} hold if we replace $\eta$ with $\tau_A$ and $\varsigma$ with $\sigma_A$, where $\sigma_A$ denotes the hitting time of $A$ of the jump chain $Y$ (defined by replacing $X$ with $Y$ in  \eqref{eq:hitd}). 
\end{proposition}

\begin{proof}By the definition of $\tau_A$, 
\begin{equation}\label{eq:djwe78afn87ewna9nfuawefea}\{\tau_A<\infty\}=\{\tau_A<T_\infty\}=\bigcup_{n=0}^\infty\{T_n\leq\tau_A<T_{n+1}\}.\end{equation}
Similarly, $\omega$ belongs to $\{T_n\leq\tau_A<T_{n+1}\}$ if only if $T_n(\omega)\leq\tau_{A}(\omega)$ and there exists an $0<\varepsilon<T_{n+1}(\omega)-\tau_A(\omega)$ such that $X_{\tau_A(\omega)+\varepsilon}(\omega)$ belongs to $A$. In this case, the definition of the chain's paths~\eqref{eq:cpathdef}--\eqref{eq:cpathdef3} implies that
$$X_{\tau_A(\omega)+\varepsilon'}=X_{T_n(\omega)}(\omega)=Y_n(\omega)\in A\quad\forall \varepsilon'\in[0,\varepsilon].$$
For this reason, $\tau_A$'s definition further implies that $\tau_A=T_n$ on $\{T_n\leq\tau_A<T_{n+1}\}$. It follows from~\eqref{eq:djwe78afn87ewna9nfuawefea} that $X_{\tau_A}$ belongs to $A$ whenever $\tau_A$ is finite and that$\tau_A$ takes values in the set of jump times (i.e., that \eqref{eq:exitjump} holds if we replace $\eta$ therein with $\tau_A$). Because
\begin{align*}\{\tau_A=T_n\}&=\{X_{T_0}\not\in A,\dots,X_{T_{n-1}}\not\in A, X_{T_n}\in A\}=\{Y_0\not\in A,\dots,Y_{n-1}\not\in A,Y_n\in A\}\\
&=\{\sigma_A=n\}\quad\forall n\in\n,\\
\{\tau_A=\infty\}&=\{\tau_A<\infty\}^c=\left(\bigcup_{n=0}^\infty\{\tau_A=T_n\}\right)^c=\left(\bigcup_{n=0}^\infty\{\sigma_A=n\}\right)^c=\{\sigma_A<\infty\}^c\\
&=\{\sigma_A=\infty\},
\end{align*}
$\varsigma$ in~\eqref{eq:exitjump2} (with $\tau_A$ replacing $\eta$) coincides with $\sigma_A$ and the result follows from Proposition~\ref{prop:hitisstop} and the fact that $(\cal{G}_n)_{n\in\n}$ contains the filtration generated by $Y$.
\end{proof}

\subsubsection*{Dynkin's formula}We are now ready to tackle Dynkin's Formula\index{Dynkin's formula}:
\begin{theorem}[Dynkin's formula]\label{Dynkin} Let $t\in[0,\infty)$, $\eta$ be as in \eqref{eq:stopdc}, $\eta_n:=\eta\wedge T_n$ for each $n\in\n$, $f$ be a $\gamma$-integrable function on $\s$ such that $Qf(x)$ is absolutely convergent for all $x$ in $\s$, and $g$ be a continuously differentiable function on $[0,\infty)$. If there exists a finite set $F\subseteq\s$ such that $X_s$ belongs to $F$ for all $s\in[0,\eta\wedge T_\infty)$, then
\begin{align}\label{eq:genintb}\Ebl{g(t\wedge\eta_n)f(X_{t\wedge\eta_n})}=g(0)\gamma(f)
+\Ebl{\int_0^{t\wedge\eta_n} (g(s)Qf(X_s)+ \dot{g}(s)f(X_s))ds}\quad\forall n\in\n.\end{align}
\end{theorem}
\begin{proof}Because $F$ is finite and $g$ is continuously differentiable,
\begin{align}&\mmag{g(s)}\leq\sup_{r\in[0,t]}g(r)<\infty,\quad \mmag{\dot{g}(s)}\leq\sup_{r\in[0,t]}\dot{g}(r)<\infty,\quad\forall s\in[0,t\wedge \eta],\nonumber\\
& \mmag{f(X_s)}\leq\max_{x\in F}\mmag{f(x)}<\infty,\quad \mmag{Qf(X_s)}\leq\max_{x\in F}\sum_{y\in \s}\mmag{q(x,y)}\mmag{f(y)}<\infty,\quad\forall s\in[0,t\wedge \eta),\label{eq:ju8fawfnwa8ybfynaufjwa}\end{align}
and it follows that all of the integrals and expectations in \eqref{eq:genintb} are well-defined and finite. Suppose that we have argued for the case of a bounded $f$ and let $(\s_r)_{r\in\zp}$ be a sequence of finite truncation approaching $\s$ (i.e., $\cup_{r=1}^\infty\s_r=\s$). Then, for any unbounded $f$, we recover \eqref{eq:genintb} by replacing $f$ therein with $f1_{\s_r}$, taking the limit $k\to\infty$, and applying dominated convergence (note that the inequalities in~\eqref{eq:ju8fawfnwa8ybfynaufjwa} if we replace $f$ by $f1_{\s_r}$ in their left-hand sides). Thus, without loss of generality, we assume that $f$ is bounded.

Let $\varsigma$ be the $(\cal{G}_n)_{n\in\n}$-stopping time in \eqref{eq:stopdc} associated with $\eta$, where $(\cal{G}_n)_{n\in\n}$ denotes the filtration generated by the jump chain and the jump times (Definition~\ref{def:filt2}), and let $\varsigma_n:=\varsigma\wedge n$ for each $n$ in $\n$. It is not difficult to check that $\eta_n=T_{\varsigma_n}$. For this reason, it follows from the definition of the chain's paths in \eqref{eq:cpathdef}--\eqref{eq:cpathdef3} that
\begin{align}\int_0^{t\wedge\eta_n} (g(s)Qf(X_s)&+ \dot{g}(s)f(X_s))ds=\sum_{m=0}^{\varsigma_n-1}\left(\int_{t\wedge T_m}^{t\wedge T_{m+1}}g(s)Qf(X_s)ds+ \int_{t\wedge T_m}^{t\wedge T_{m+1}}\dot{g}(s)f(X_s)ds\right)\nonumber\\
&=\sum_{m=0}^{\varsigma_n-1}\left(Qf(Y_m)\int_{t\wedge T_m}^{t\wedge T_{m+1}}g(s)ds+ (g(t\wedge T_{m+1})-g(t\wedge T_m))f(Y_m)\right).\label{eq:jnaw98f7hfnywa}
\end{align}
Notice that
$$\int_{t\wedge T_m}^{t\wedge T_{m+1}}g(s)ds=1_{\{T_m\leq t<T_{m+1}\}}\int_{0}^{t-T_m}g(T_m+s)ds+1_{\{T_{m+1}\leq t\}}\int_{0}^{T_{m+1}-T_m}g(T_m+s)d.s$$
Because the $(m+1)$th waiting time $S_{m+1}$ conditioned on $\mathcal{G}_m$ is exponentially distributed with mean $1/\lambda(Y_m)$ (Theorem~\ref{thrm:condind}), the take-out-what-is-known property of conditional expectation (Theorem~\ref{thrm:condexpprops}$(v)$) implies that
\begin{align*}&\Ebl{1_{\{T_m\leq t<T_{m+1}\}}\int_{0}^{t-T_m}g(T_m+s)ds\Big|\mathcal{G}_m}\\
&=1_{\{T_m\leq t\}}\left(\int_{0}^{t-T_m}g(T_m+s)ds\right)\Pbl{\{t-T_m<S_{m+1}\}|\mathcal{G}_m}\\
&=1_{\{T_m\leq t\}}\left(\int_{0}^{t-T_m}g(T_m+s)ds\right) e^{-\lambda(Y_m)(t-T_m)},\quad\Pb_\gamma\text{-almost surely}.
\end{align*}
Similarly,
\begin{align*}&\Ebl{1_{\{T_{m+1}\leq t\}}\int_{0}^{T_{m+1}-T_m}g(T_m+s)ds\Big|\mathcal{G}_m}\\
&=1_{\{T_m\leq t\}}\Ebl{1_{\{S_{m+1}\leq t-T_m\}}\int_{0}^{S_{m+1}}g(T_m+s)ds\Big|\mathcal{G}_m}\\
&=1_{\{T_m\leq t\}}\int_{0}^{t-T_m} \lambda(Y_m)e^{-\lambda(Y_m)r}\left(\int_{0}^{r}g(T_m+s)ds\right)dr\\
&=1_{\{T_m\leq t\}}\int_0^{t-T_m} g(T_m+s)\left(\int_{ s}^{t-T_m} \lambda(Y_m)e^{-\lambda(Y_m)r}dr\right)ds\\
&=1_{\{T_m\leq t\}}\int_0^{t-T_m} g(T_m+s)e^{-\lambda(Y_m)s}ds-1_{\{T_m\leq t\}}\left(\int_0^{t-T_m}g(T_m+s)ds\right)e^{-\lambda(Y_m)(t-T_m)},
\end{align*}
$\Pb_\gamma$-almost surely. Putting the above three together, we have that
\begin{align*}\Ebl{\int_{t\wedge T_m}^{t\wedge T_{m+1}}g(s)ds\Big|\mathcal{G}_m}&=1_{\{T_m\leq t\}}\int_0^{t-T_m} g(T_m+s)e^{-\lambda(Y_m)s}ds\\
&=\frac{1}{\lambda(Y_m)}1_{\{T_m\leq t\}}\Ebl{1_{\{S_{m+1}\leq t-T_m\}}g(T_{m+1})|\cal{G}_m}\\
&=\frac{1}{\lambda(Y_m)}\Ebl{\tilde{g}(T_{m+1})|\cal{G}_m},\quad\Pb_\gamma\text{-almost surely}\end{align*}
where $\tilde{g}(s):=1_{\{s\leq t\}}g(s)$ for all $s$ in $[0,\infty)$.
%
%
For this reason,
$$\Ebl{Qf(Y_m)\int_{t\wedge T_m}^{t\wedge T_{m+1}}g(s)ds\Big|\mathcal{G}_m}=(Pf(Y_m)-f(Y_m))\Ebl{\tilde{g}( T_{m+1})|\cal{G}_m}\quad\Pb_\gamma\text{-almost surely},$$
where $P$ denotes the one-step matrix of the jump chain \eqref{eq:jumpmatrix}. 
Taking expectations of \eqref{eq:jnaw98f7hfnywa} and applying the take-out-what-is-known and tower properties of conditional expectation (Theorem~\ref{thrm:condexpprops}$(iv,v)$), we find that
\begin{align*}&\Ebl{\int_0^{t\wedge\eta_n} (g(s)Qf(X_s)+ \dot{g}(s)f(X_s))ds}\\
&=\Ebl{\sum_{m=0}^{\varsigma_n-1}\left(\tilde{g}(T_{m+1})Pf(Y_m)-\tilde{g}(T_m)f(Y_m)+1_{\{T_m\leq t<T_{m+1}\}}g(t)f(Y_m)\right)}\\
&=\Ebl{\sum_{m=0}^{\varsigma_n-1}\left(\tilde{g}(T_{m+1})Pf(Y_m)-\tilde{g}(T_m)f(Y_m)\right)}+\Ebl{g(t)f(X_t)\sum_{m=0}^{\varsigma_n-1}1_{\{T_m\leq t<T_{m+1}\}}}\\
&=\Ebl{\sum_{m=0}^{\varsigma_n-1}\left(\tilde{g}(T_{m+1})Pf(Y_m)-\tilde{g}(T_m)f(Y_m)\right)}+\Ebl{1_{\{t<\eta_n\}}g(t)f(X_t)}.\end{align*}

It follows from the above  that \eqref{eq:genintb} is satisfied if and only if $\Ebl{M_{\varsigma_n}}=0$, where
$$M_n:=\tilde{g}(T_n)f(Y_n)-g(0)f(Y_0)-\sum_{m=0}^{n-1}(\tilde{g}(T_{m+1})Pf(Y_m)-\tilde{g}(T_m)f(Y_m))\quad\forall n\in\n.$$
If we can argue that $M:=(M_n)_{n\in\n}$ is a $(\cal{G}_n)_{n\in\n}$-adapted $\Pb_\gamma$-martingale, then an application of Doob's optional stopping theorem of the sort in the proof of Theorem~\ref{thrm:dynkin} yields that $\Ebl{M_{\varsigma_n}}=0$, as desired. Because $\tilde{g}$ and $f$ are bounded functions, it easy to show that $M_n$ is bounded and, hence, $\Pb_\gamma$-integrable. Thus, to show that $M$ a martingale we need only to argue that $\Ebl{M_{n+1}|\cal{G}_{n}}$ equals $M_{n-1}$, $\Pb_\gamma$-almost surely, for each $n$ in $\n$. The conditional independence of $T_{n+1}$ and $Y_{n+1}$ (Theorem~\ref{thrm:condind}) implies that
\begin{align*}\Ebl{\tilde{g}(T_{n+1})f(Y_{n+1})|\mathcal{G}_{n}}&=\Ebl{\tilde{g}(T_{n+1})|\mathcal{G}_{n}}\Ebl{f(Y_{n+1})|\mathcal{G}_{n}}=\Ebl{\tilde{g}(T_{n+1})|\mathcal{G}_{n}}Pf(Y_{n})\\
&=\Ebl{\tilde{g}(T_{n+1})Pf(Y_{n})|\mathcal{G}_{n}}\quad\Pb_\gamma\text{-almost surely},\quad\forall n\in\n.\end{align*}
It follows from the above that
$$\Ebl{M_{n+1}-M_{n}|\cal{G}_{n}}=\Ebl{\tilde{g}(T_{n+1})f(Y_{n+1})-\tilde{g}(T_{n+1})Pf(Y_{n})|\cal{G}_{n}}=0,$$
implying that $M$ is a martingale and completing the proof.
\end{proof}

\ifdraft
\subsection{\textbf{The semigroup approach and the generator**}}
Domain as in Spieksma 2015.

\fi

\subsection{Stopping distributions and occupation measures}\label{sec:morehitct} Throughout this section, we use $\eta$ to denote a jump-time-value $(\cal{F}_t)_{t\geq0}$ stopping time (Definition~\ref{def:stopct2}). For the same reasons as in the discrete-time case (see Section~\ref{sec:stop}), we say that the chain \emph{stops} at time $\eta$. With the stopping time, we associate two measures: the  \emph{stopping distribution}\index{stopping distribution}\glsadd{mu} $\mu$ and \emph{occupation measure}\index{occupation measure}\glsadd{nu} $\nu$ defined by
\begin{align}
\mu(A,x)&:=\Pbl{\{\eta\in A,X_\eta= x\}}&\forall A\in\cal{B}([0,\infty)),\quad x\in\s,\label{eq:edisdef}\\
\nu(A,x)&:=\Ebl{\int_{A\cap[0,\eta\wedge T_\infty)} 1_{x}(X_t)dt}&\forall A\in\cal{B}([0,\infty)),\quad x\in\s.\label{eq:eoccdef} 
\end{align}
In other words, $\mu(A,x)$ is the probability that the chain stops at some time point in $A$ and that it lies in state $x$ when it stops. Similarly, $\nu(A,x)$ is the expected amount of time in $A$ that the chain spends in state $x$ before stopping or exploding. If you care for these things, see the end of the section for the full details of $\mu$ and $\nu$'s definition. Otherwise, only notice that, in~\eqref{eq:edisdef}, we are exploiting that $\eta$ is finite if and only if it is less than $T_\infty$ (c.f.~\eqref{eq:etatinfinf}).

The mass of the stopping distribution is simply the probability that the stopping time is finite
\begin{equation}\label{eq:mumass}\mu([0,\infty),\s)
=\Pbl{\{\eta<\infty,X_\eta\in\s\}}=\Pbl{\{\eta<\infty\}},\end{equation}
which follows by applying the monotone convergence theorem and Tonelli's theorem to \eqref{eq:edisdef}. Similarly, it follows from \eqref{eq:eoccdef} that the mass of the occupation measure is
\begin{equation}\label{eq:numass}\nu([0,\infty),\s)=\Ebl{\int_0^{\eta\wedge T_\infty} \left(\sum_{x\in\s}1_{x}(X_t)\right) dt}=\Ebl{\eta\wedge T_\infty}.\end{equation}
If no explosion may occur before the chain stops ($\Pbl{\{\eta\leq  T_\infty\}}=1$), then $\nu$'s mass is the mean stopping time.

\subsubsection*{The marginals}We now turn our attention to the marginals of $\mu$ and $\nu$. The \emph{time marginal}\index{stopping distribution (time marginal)} $\mu_T$ of the stopping distribution defined by  
\begin{equation}\label{eq:mutc}\mu_T(A):=\mu(A,\s)=\Pbl{\{\eta\in A\}},\quad\forall A\in\cal{B}([0,\infty))\end{equation}
is the distribution of the stopping time itself. Technically, the above is the distribution of $\eta$ restricted to $[0,\infty)$ because $\eta$ may take the value $\infty$. However, we recover the full distribution of $\eta$ from the above with $\Pbl{\{\eta=\infty\}}=1-\Pbl{\{\eta<\infty\}}=1-\mu_T([0,\infty))$.

The \emph{space marginals} $\mu_S$\glsadd{mus} and $\nu_S$\glsadd{nus} of the stopping distribution\index{stopping distribution (space marginal)}\index{occupation measure (space marginal, stopping time)} and occupation measure
\begin{align}
\label{eq:musc}\mu_S(x)&:=\mu([0,\infty),x)=\Pbl{\{\eta<\infty,X_\eta=x\}},\\
\label{eq:nusc}\nu_S(x)&:=\nu([0,\infty),x)=\Ebl{\int_0^{\eta\wedge T_\infty} 1_{x}(X_t)dt},
\end{align}
tell us where the chain stops and where it spends time before stopping, respectively. Explicitly, $\mu_S(x)$ is the probability that the chain stops in $x$, while $\nu_S(x)$ denotes the expected amount of time spent in $x$ before stopping. These space marginals are tied together by a set of linear equations:
\begin{theorem}\label{eqnsc}Suppose that $\eta$ is a jump-time-valued $(\cal{F}_t)_{t\geq0}$-stopping time (Definition~\ref{def:stopct2}) and that $\mu_S$ and $\nu_S$ are the space marginals of its stopping distribution and occupation in \eqref{eq:musc}--\eqref{eq:nusc}. If the stopping time is almost surely finite,
\begin{equation}\label{eq:exbst}\Pbl{\{\eta<\infty\}}=1,\end{equation}
then  the pair $(\mu_S,\nu_S)$ satisfies 
\begin{equation}\label{eq:eoe}\mu_S(x)+\nu_S(x)q(x)=\gamma(x)+\sum_{y\neq x}\nu_S(y)q(y,x),\qquad\forall x\in\cal{S}.\end{equation}
%
%
\end{theorem}

\begin{proof} Pick any state $x$ and increasing sequence $(\s_r)_{r\in\zp}$ of finite truncations with limit $\s$ (i.e., such that $\cup_{r=1}^\infty\s_r=\s$) and let $(\tau_r)_{r\in\zp}$ be the corresponding sequence of exit times in~\eqref{eq:taurexit}. It follows from Proposition~\ref{prop:hitdc} that
$$\eta(\omega)\wedge\tau_r(\omega)=\left\{\begin{array}{ll} T_{\varsigma(\omega)\wedge\sigma_r(\omega)}(\omega)&\text{if}\quad\varsigma(\omega)\wedge\sigma_r(\omega)<\infty\\
\infty &\text{if}\quad\varsigma(\omega)\wedge\sigma_r(\omega)=\infty\end{array}\right.\quad\forall \omega\in\Omega,$$
where $\varsigma$ denotes the discrete-time stopping time in~\eqref{eq:stopdc} associated with $\eta$ and $\sigma_r$ denotes the time of exit from $\s_r$ for the jump chain:
$$\sigma_r(\omega):=\inf\{n\geq0:Y_n(\omega)\not\in\s_r\}\quad\forall \omega\in\Omega,\quad r>0.$$
In other words, $\eta\wedge\tau_r$ is also jump-time-valued $(\cal{F}_t)_{t\geq0}$-stopping time. Thus, replacing $f$ by $1_x$ and $\eta$ by $\eta\wedge\tau_r$ in  in Dynkin's formula \eqref{eq:genintb}, re-arranging, taking the limit $t\to\infty$, and applying bounded convergence, monotone convergence, and Tonelli's theorem yields
\begin{align}\label{eq:genifj87aw3j87afhn38a7ntb}\Pbl{\{X_{\eta\wedge\tau_r\wedge T_n}=x\}}&+\Ebl{\int_0^{\eta\wedge\tau_r\wedge T_n} 1_x(X_s)ds}q(x)\\
&=\gamma(x)
+\sum_{y\neq x}\Ebl{\int_0^{\eta\wedge\tau_r\wedge T_n} 1_y(X_s)ds}q(y,x)\quad\forall n\geq0,\enskip r>0.\nonumber\end{align}

Because Theorem~\ref{tautin} shows that $(\tau_r)_{r\in\zp}$ is an increasing sequence with $\Pb_\gamma$-almost sure limit $T_\infty$, monotone convergence implies that
\begin{equation}\label{eq:genifj87aw3j87afhn38a7ntb2}
\lim_{n\to\infty}\lim_{r\to\infty}\Ebl{\int_0^{\eta\wedge\tau_r\wedge T_n} 1_x(X_s)ds}=\nu_s(x)\quad\forall x\in\s.
\end{equation} 
Similarly, Lemma~\ref{lem:dtexitrinf} shows that $(\sigma_r)_{r\in\zp}$ is an increasing sequence with $\Pb_\gamma$-almost sure limit $\infty$. Because   $\eta\wedge \tau_r\wedge T_n=T_{\varsigma\wedge\sigma_r\wedge n}$ (check!), \eqref{eq:exbst}, bounded convergence, and \eqref{eq:etatinfinf} imply that 
\begin{align*}\lim_{n\to\infty}\lim_{r\to\infty}\Pbl{\{X_{\eta\wedge\tau_r\wedge T_n}=x\}}&=\lim_{n\to\infty}\lim_{r\to\infty}\Pbl{\{Y_{\varsigma\wedge\sigma_r\wedge n}=x\}}\\
&=\lim_{n\to\infty}\lim_{r\to\infty}\Pbl{\{\varsigma<\infty,Y_{\varsigma\wedge\sigma_r\wedge n}=x\}}=\Pbl{\{\varsigma<\infty,Y_{\varsigma}=x\}}\\
&=\Pbl{\{\eta<\infty,X_{T_{\varsigma}}=x\}}=\Pbl{\{\eta<\infty,X_{\eta}=x\}}=\mu_S(x)\end{align*}
for all $x$ in $\s$. Putting \eqref{eq:genifj87aw3j87afhn38a7ntb}--\eqref{eq:genifj87aw3j87afhn38a7ntb2} together with the above completes the proof.
\end{proof}

You may be wondering whether  \eqref{eq:exbst} is actually necessary for \eqref{eq:eoe} to hold. It is easy to find  examples of chains violating~\eqref{eq:exbst} for which~\eqref{eq:eoe} does not hold.
\begin{example}\label{toyfdsa}Consider the chain on $\cal{S}:=\{0,1\}$, with initial law $\gamma=\alpha\delta_0+(1-\alpha)\delta_1$ where $\alpha\in[0,1]$, and rate matrix
$$Q:=\begin{bmatrix}0&0\\0&0\end{bmatrix}.$$
Both states are absorbing, and so the chain remains forever where it starts (Proposition~\ref{prop:absorb}) $T_\infty = \infty$ and $X_t=X_0$ for all $t\in[0,\infty)$, $\Pb_\gamma$ almost surely. For this reason, if $\tau$ denotes the hitting time of $\{1\}$,
$$\Pbl{\{\tau<\infty\}}=\Pbl{\{X_0=1\}}=1-\alpha.$$
Equations \eqref{eq:eoe} read
$$0=\alpha,\qquad \mu_S(1)=1-\alpha.$$
The left-most equation is satisfied if and only if $\alpha=0$, that is if and only if $\Pbl{\{\tau<\infty\}}=1$.
\end{example}
In the case of the example's chain and stopping time, \eqref{eq:exbst} is indeed necessary and sufficient for \eqref{eq:eoe} to hold. However, for other chains and stopping times, the question remains...
\subsubsection*{An open question}Is~\eqref{eq:exbst} necessary for equations~\eqref{eq:eoe} to hold?
\subsubsection*{The details of $\mu$ and $\nu$'s definition} Throughout the above, we didn't actually define $\mu$ and $\nu$, just specified in~\eqref{eq:edisdef}--\eqref{eq:eoccdef} what values these measures should take on sets of the form $A\times\{x\}$ with $A$ in $\cal{B}([0,\infty))$ and $x$ in $\s$. It turns out that this is all that's necessary to fully define $\mu$ and $\nu$. To do see this, we require Carath\'eodory's extension theorem:
\begin{theorem}[Carath\'eodory's extension theorem,~{\citep[Theorem~1.7]{Williams1991}}]\label{thrm:caratheory}Suppose that $S$ is a set and $\Sigma_0$ is an algebra on $S$. That is, $\Sigma_0$ is a collections of subsets of $S$ satisfying $S\in \Sigma_0$, $A^c\in \Sigma_0$ for all $A\in \Sigma_0$, and $A\cup B\in\Sigma_0$ for all $A,B\in \Sigma_0$. If $\rho^0:\Sigma_0\to[0,\infty]$ is a countably additive and $\Sigma$ is the sigma-algebra generated by $\Sigma_0$, then there exists an unsigned measure $\rho$ on $(S,\Sigma)$ such that
$$\rho(A)=\rho^0(A)\quad\forall A\in\Sigma_0.$$
\end{theorem}

Now, consider the algebra 
$$\Sigma_0:=\text{the collection of all finite unions of sets of the form }A\times B\text{ with }A\in\cal{B}([0,\infty)),\enskip B\subseteq\s$$
generating  the product sigma-algebra $\cal{X}:=\cal{B}([0,\infty))\times2^\s$. We can express any set in $\Sigma_0$ as
$$\bigcup_{i\in\cal{I}}A_i\times B_i$$
for some (finite) indexing set $\cal{I}$, $A_i$s in $\cal{B}([0,\infty))$, and $B_i$s in $2^\s$. Next, consider the countably additive functions on $\Sigma_0$ 
$$\mu^0\left(\bigcup_{i\in\cal{I}}A_i\times B_i\right):=\sum_{x\in\s}\mu\left(\bigcup_{i\in\cal{I}_x}A_i,x\right),\quad\nu^0\left(\bigcup_{i\in\cal{I}}A_i\times B_i\right):=\sum_{x\in\s}\nu\left(\bigcup_{i\in\cal{I}_x}A_i,x\right)$$
where $\cal{I}_x:=\{i\in\cal{I}: x\in B_i\}$ and $\mu(A,x)$ and $\nu(A,x)$ are as in~\eqref{eq:edisdef}--\eqref{eq:eoccdef}.
\begin{exercise}Check that $\Sigma_0$ is indeed an algebra that generates $\cal{X}$  and that $\mu^0$ and $\nu^0$ are countably additive on $\Sigma_0$.\end{exercise}
For this reason, Theorem~\ref{thrm:caratheory} implies that there exists measures $\mu$ and $\nu$ on $([0,\infty)\times\s,\cal{X})$ satisfying 
\begin{equation}\label{eq:mfu9eadsads0nfeanjfea}\mu(A\times\{x\})=\mu(A,x),\quad\nu(A\times\{x\})=\nu(A,x),\quad\forall A\in\cal{B}([0,\infty)),\enskip x\in\s,\end{equation}
with $\mu(A,x)$ and $\nu(A,x)$  as in~\eqref{eq:edisdef}--\eqref{eq:eoccdef}.
Because 
$$\Sigma_{-1}:=\{A\times\{x\}:A\in\cal{B}([0,\infty),\enskip x\in\s\}$$
is a $\pi$-system and $\mu([0,\infty)\times\s)=\Pbl{\{\eta<\infty\}}\leq1$, Lemma~\ref{lem:dynkinpl} implies that $\mu$ is the only measure on $([0,\infty)\times\s,\cal{X})$ satisfying the leftmost equation in~\eqref{eq:mfu9eadsads0nfeanjfea}. We call it the \emph{stopping distribution}. Furthermore, 
\begin{equation}\label{eq:Lfmeu90anf8ebnafyfnaeiwf}\nu(A)=\sum_{n=0}^\infty\nu_n(A)\quad\forall A\in \cal{X},\end{equation}
where
$$\nu_n(A):=\nu([n,n+1)\times\s\cap A)\quad\forall A\in \cal{X},\enskip n\in \n.$$
Because 
$$\nu_n([0,\infty)\times\s)=\Ebl{(n+1)\wedge\eta\wedge T_\infty-n\wedge\eta\wedge T_\infty}\leq 1\quad\forall n\in\n,$$
applying  Lemma~\ref{lem:dynkinpl} to $\nu_n$ and using~\eqref{eq:Lfmeu90anf8ebnafyfnaeiwf}, we find that that $\nu$ is the only measure on $([0,\infty)\times\s,\cal{X})$ satisfying the rightmost equation  in \eqref{eq:mfu9eadsads0nfeanjfea}. We call it the \emph{occupation measure}.
\subsection{Exit times*}\label{moh1} 
We are often interested in how long the chain takes to exit some given subset, or \emph{domain}, of the state space and what part of the domain's boundary the chain crosses to exit. To study this problem, we single out a subset $\cal{D}$ of the state space $\s$ and refer to it as the \emph{domain}\index{domain}\glsadd{D}. The \emph{exit time} $\tau$ from the domain is the instant in time that the chain leaves the domain for the first time:\glsadd{tau}\index{exit time}
\begin{equation}\label{eq:hitthec2}\tau(\omega):=\inf\{t\in[0,T_\infty(\omega)): X_t(\omega)\not\in\cal{D}\}\quad\forall \omega\in\Omega.\end{equation}
For us, starting outside of the domain counts as `exiting' the domain in which case we set $\tau$ to zero. We say that the chain \emph{exits via $x$} if $\tau$ is finite and $X_{\tau}=x$. Clearly, the exit time is just the hitting time of the domain's complement (compare with \eqref{eq:hitthec} and \eqref{eq:hitthec2}). For this reason, Proposition~\ref{prop:hitdc} implies that $\tau$ is a jump-time-value $(\cal{F}_t)_{t\geq0}$-stopping time and that  equations \eqref{eq:stopdc}--\eqref{eq:etatinfinf} hold if we replace $\eta$ with $\tau$ and $\varsigma$ with $\sigma$, where $\sigma$ is the exit time of the jump chain:
\begin{equation}\label{eq:sigma}\sigma(\omega):=\inf\{n\geq0:Y_n(\omega)\not\in\cal{D}\}.\end{equation}

Let $\mu$ be $\tau$'s \emph{exit distribution}\glsadd{mu}\index{exit distribution} and $\nu$ its \emph{occupation measure}\glsadd{nu}\index{occupation measure (exit time)} defined as in \eqref{eq:musc}--\eqref{eq:musc} with $\tau$ replacing $\eta$.  For each state $x$, the measures $\mu(dt,x)$ and $\nu(dt,x)$ have densities\footnote{Use~\eqref{eq:dens1} and \eqref{eq:mdw89amda8w9} and \eqref{eq:dens2} further down to verify this fact.} $\mu(t,x)$ and $\nu(t,x)$ with respect to the Lebesgue measure on $[0,\infty)$ (I distinguish a measure from its density by writing $dt$ or $t$ in its argument). The density $t \mapsto \mu(t,x)$ is a function such that $\mu(t,x) \, h$ is the probability that the chain first exits the domain via state $x$ during the time interval $[t,t+h]$, for any small $h>0$.  Similarly (and assuming that the chain is non-explosive), $\nu(t,x)$ is the average fraction of the interval $[t,t+h]$ that the chain spends in state $x$ before exiting the domain. Formally, the relationship between the exit distribution and occupation measure and their densities is: 
\begin{align}\label{eq:md1}
\mu(A,x)&=1_{\cal{D}^c}(x) \, \gamma(x) \, 1_A(0)+\int_A\mu(t,x)dt&\forall   A\in\cal{B}([0,\infty)), \enskip x\in\s,\\
\label{eq:md11}\nu(A,x)&=\int_A\nu(t,x)dt&\forall  A\in\cal{B}([0,\infty)), \enskip x\in\s,
\end{align}
where the term $1_{\cal{D}^c}(x) \, \gamma(x) \, 1_A(0)$
accounts for the possibility  that the chain is started outside of the domain. The densities are characterised in terms of the minimal solution to a set of linear differential equations:
\begin{theorem}[Analytical characterisation of $\mu$ and $\nu$]\label{charactt} Suppose that
\begin{equation}
\label{eq:nd7a73a8dhd}\sup_{x\in\cal{D}}q(x,y)<\infty\quad\forall y\in\s\quad\text{ or }\quad \sum_{x\in\cal{D}}q(x)\gamma(x)<\infty.
\end{equation} 
The exit distribution $\mu$ and occupation measure $\nu$ decompose as in \eqref{eq:md1}--\eqref{eq:md11} and their densities $\nu(t,x)$ and $\mu(t,x)$ are non-negative and continuous functions on $[0,\infty)$, for each $x$ in $\s$. Moreover, 
\begin{equation}\label{eq:jointchar1}\mu(t,x)=1_{\cal{D}^c}(x)\dot{\hat{p}}_t(x),\quad \nu(t,x)=1_{\cal{D}}(x)\hat{p}_t(x),\quad\forall x\in\s,\quad  t\in[0,\infty),
\end{equation}
where $\hat{p}_t$ is the minimal non-negative solution (as in Theorem~\ref{thrm:forward}) of
\begin{equation}\label{eq:jointchar2}\dot{\hat{p}}_t(x)=\sum_{y\in\cal{D}}\hat{p}_t(y)q(y,x) \quad t\in[0,\infty),\enskip x\in\s,\qquad \hat{p}_0(x)=\gamma(x)\quad\forall x\in\s.\end{equation}
%
%
%
\end{theorem}
The proof takes some doing and we leave it until the end of the section. The ideas guiding it, however, are simple: consider a second chain $\hat{X}$ obtained by replacing $Q$ in Algorithm~\ref{gilalg} with $\hat{Q}=(\hat{q}(x,y))_{x,y\in\s}$ and keeping everything else the same, where
\begin{equation}\label{eq:qhat}\hat{q}(x,y):=\left\{\begin{array}{ll} q(x,y)&\text{if }x\in\cal{D} \\ 0&\text{if }x\not\in\cal{D} \end{array}\right.\quad\forall x,y\in\s.\end{equation}
The chains $X$ and $\hat{X}$ are identical up until (and including) the moment they both simultaneously exit the domain via the same state. Therefore the probability $\mu([0,t),x)$ that $X$ has exited the domain by time $t$ via state $x\not\in\cal{D}$ is also the probability that $\hat{X}$ exited via $x$ by time $t$. However, all states outside of the domain are absorbing for $\hat{X}$ and $\hat{X}$ gets trapped in the state it enters upon leaving the truncation. For these reasons,  $\mu([0,t),x)$ is the probability $\hat{p}_t(x)$ that $\hat{X}$ is in state $x$ by time $t$ and the first equation in \eqref{eq:jointchar1} follows. Similarly, once $\hat{X}$ leaves the domain it cannot return, hence the amount of time that $\hat{X}$ spends in a state $x\in\cal{D}$ until the moment it exits the domain is the \emph{total} time it spends in that state:
$$\int_0^{t\wedge \tau\wedge T_\infty}1_x(X_s)ds=\int_0^{t\wedge \tau\wedge \hat{T}_\infty}1_x(\hat{X}_s)ds=\int_0^{t\wedge \hat{T}_\infty}1_x(\hat{X}_s)ds.$$
Taking expectations of the above we obtain the second equation in \eqref{eq:jointchar1}. Equation \eqref{eq:jointchar2} on the other hand, is just the differential equation in Theorem~\ref{thrm:forward} satisfied by the time-varying law $\hat{p}_t$ of $\hat{X}$. 

\subsubsection*{The marginals}Equations \eqref{eq:jointchar1}--\eqref{eq:jointchar2} imply that $\mu(t,x)$ is non-negative. Thus, \eqref{eq:md1} and Tonelli's theorem show that the time-marginal $\mu_T(dt)$ in~\eqref{eq:mutc} of the exit distribution\index{exit distribution (time marginal)} also has a density $\mu_T(t)$ with respect to the Lebesgue measure and that this density is given by $\mu_T(t)=\sum_{x\in\s}\mu(t,x)$:
\begin{equation}\label{eq:mutden}\mu_T(A)=\gamma(\cal{D}^c)1_A(0)+\int_A\mu_T(t)dt\quad\forall A\in\cal{B}([0,\infty)).\end{equation}
Theorem \ref{charactt} and Tonelli's theorem then give us expressions for the marginals\index{exit distribution (space marginal)}\index{occupation measure (space marginal, exit time)} in terms of the solution of \eqref{eq:jointchar2}:
\begin{align*}\mu_T(t)&=\sum_{x\not\in\cal{D}}\hat{p}_t(x)\quad\forall t\in[0,\infty),\\
 \mu_S(x)&=1_{\cal{D}^c}(x)\left(\lim_{t\to\infty}\hat{p}_t(x)\right),\quad \nu_S(x)=1_{\cal{D}}(x)\int_0^\infty \hat{p}_t(x)dt,\quad\forall x\in\s.\end{align*}

In the case of the space\glsadd{mus}\glsadd{nus} marginals $\mu_S$ and $\nu_S$ in~\eqref{eq:musc}--\eqref{eq:nusc} and an almost surely finite exit time, we have an alternative characterisation:
\begin{theorem}[Analytical characterisation of $\mu_S$ and $\nu_S$]\label{charact}If $\Pbl{\{\tau<\infty\}}=1$, then $\nu_S$ in~\eqref{eq:nusc} is the minimal non-negative solution of the equations
\begin{equation}\label{eq:nueqsm}q(x)\nu_S(x)=\gamma(x)+\sum_{z \neq x}\nu_S(z)q(z,x)\quad\forall x\in\cal{D},\qquad\nu_S(x)=0\quad\forall x\not\in\cal{D},\end{equation}
while $\mu_S$ in~\eqref{eq:musc} is given by 
\begin{equation}\label{eq:mueqsm}\mu_S(x)=\gamma(x)+\sum_{z\in\cal{D}}\nu_S(z)q(z,x)\quad\forall x\not\in\cal{D},\qquad \mu_S(x)=0\quad\forall x\in\cal{D}.\end{equation}
%
\end{theorem}

\begin{proof}Proposition~\ref{prop:hitdc} implies that $X_t(\omega)$ belongs to the domain $\cal{D}$ if $\omega$ belongs to $\{t<\tau\wedge T_\infty\}$ and that $X_{\tau(\omega)}(\omega)$ does not belong to $\cal{D}$ if $\omega$ belongs to $\{\tau<\infty\}$. For these reasons,
$$\mu_S(x)=0\quad\forall x\in\cal{D},\qquad \nu_S(x)=0\quad\forall x\not\in\cal{D},$$
and \eqref{eq:nueqsm}--\eqref{eq:mueqsm} follow from Theorem~\ref{eqnsc}. 

All that remains to be shown is that $\rho(x)\geq \nu_S(x)$ for all $x$ in $\cal{D}$ and non-negative solution $\rho=(\rho(x))_{x\in\cal{S}}$ of \eqref{eq:nueqsm}.
Because $\tau$ is jump-time-valued (Proposition~\ref{prop:hitdc}), the definition of the chain's paths implies that
\begin{align*}
\int_0^{\tau\wedge T_\infty}1_x(X_t)dt&=\sum_{n=0}^\infty1_{\{\tau>T_n\}}\int_{T_n}^{T_{n+1}}1_x(X_t)dt=\sum_{n=0}^\infty1_{\{\tau>T_n\}}(T_{n+1}-T_n)1_x(X_{T_n})\\
&=\sum_{n=0}^\infty1_{\{\tau>T_n\}}S_{n+1}1_x(Y_n)=\sum_{n=0}^\infty S_{n+1}1_{\{Y_0\in\cal{D},\dots,Y_{n-1}\in\cal{D},Y_n=x\}}\quad\forall x\in\cal{D}.
\end{align*}
Taking expectations of both sides, we have that
\begin{align}\nu_S(x)&=\sum_{n=0}^\infty \Ebl{S_{n+1}1_{\{Y_0\in\cal{D},\dots,Y_{n-1}\in\cal{D},Y_n=x\}}}=\sum_{n=0}^\infty \Ebl{\Ebl{S_{n+1}|\cal{G}_n}1_{\{Y_0\in\cal{D},\dots,Y_{n-1}\in\cal{D},Y_n=x\}}}\nonumber\\
&=\frac{1}{\lambda(x)}\sum_{n=0}^\infty\Pbl{\{Y_0\in\cal{D},\dots,Y_{n-1}\in\cal{D},Y_n=x\}}\nonumber\\
&=\frac{1}{\lambda(x)}\left(\gamma(x)+\sum_{n=1}^\infty\sum_{z_0\in\cal{D}}\dots\sum_{z_{n-1}\in\cal{D}}\gamma(z_0)p(z_0,z_1)\dots p(z_{n-1},x)\right)\quad\forall x\in\cal{D},\label{eq:nudecomp}\end{align}
where $(p(x,y))_{x,y}$ denotes the jump matrix~\eqref{eq:jumpmatrix}, $(\lambda(x))_{x\in\s}$ the jump rates~\eqref{eq:lambda}, and $\cal{G}_n$ is the sigma-algebra generated by $Y_0,\dots Y_n$ and $S_1,\dots S_n$. The first equation follows from Tonelli's theorem, the second the take-out-what-is-known and tower properties of conditional expectation (Theorem~\ref{thrm:condexpprops}$(iv,v)$), the third from $S_{n+1}$ being exponentially distributed with mean $1/\lambda(Y_{n})$ when conditioned on $\cal{G}_n$ (Theorem~\ref{thrm:condind}), and the fourth from the expression given in Theorem~\ref{thrm:pathlawunict} for the path law.

Suppose that $x$ is an absorbing state and fix $n$ in $\n$, $z_0,z_1,\dots,z_{n-1}$ in $\cal{D}$. Because the exit time is almost surely finite, downwards monotone convergence, Theorem~\ref{thrm:pathlawunict}, and Proposition~\ref{prop:hitdc} imply that
\begin{align*}0&=\Pbl{\{\tau=\infty\}}=\Pbl{\cap_{m=0}^\infty \{Y_m\in\cal{D}\}}=\lim_{m\to\infty}\Pbl{\{Y_0\in\cal{D},Y_1\in\cal{D},\dots,Y_{n+m}\in\cal{D}\}}\\
&=\lim_{m\to\infty}\sum_{x_0\in\cal{D}}\sum_{x_1\in\cal{D}}\dots\sum_{x_{n+m}\in\cal{D}}\lambda(x_0)p(x_0,x_1)\dots p(x_{n-1},x_{n+m})\\
&\geq\gamma(z_0)p(z_0,z_1)\dots p(z_{n-1},x)\left(\lim_{m\to\infty}\sum_{x_1\in\cal{D}}\sum_{x_2\in\cal{D}}\dots\sum_{x_{m}\in\cal{D}}p(x,x_1)p(x_1,x_2)\dots p(x_{m-1},x_{m})\right)\\
&=\gamma(z_0)p(z_0,z_1)\dots p(z_{n-1},x)\left(\lim_{m\to\infty}\sum_{x_2\in\cal{D}}\dots\sum_{x_{m}\in\cal{D}}p(x,x_2)\dots p(x_{m-1},x_{m})\right)\\
&=\dots=\gamma(z_0)p(z_0,z_1)\dots p(z_{n-1},x).
\end{align*}
The decomposition \eqref{eq:nudecomp} then shows that $\nu_S(x)=0$ (and so $\rho(x)\geq\nu_S(x)$ by non-negativity of $\rho$). We now only have left to show that $\rho(x)\geq \nu_S(x)$ for every non-absorbing state $x$ inside the domain (that is, $x$ in $\cal{D}_{na}:=\{x\in\cal{D}_{na}:q(x)>0\}$). For any such state $x$, we rewrite \eqref{eq:nudecomp} as 
\begin{equation}q(x)\nu_S(x)=\gamma(x)+\sum_{n=1}^\infty\sum_{z_0\in\cal{D}_{na}}\dots\sum_{z_{n-1}\in\cal{D}_{na}}\gamma(z_0)p(z_0,z_1)\dots p(z_{n-1},x),\enskip\forall x\in\cal{D}_{na}.\label{eq:nudecomp2}\end{equation}
Using the definition of the jump matrix in \eqref{eq:jumpmatrix} and repeatedly applying \eqref{eq:nueqsm}, we have that
\begin{align*}q(x)\rho(x)&=\gamma(x)+\sum_{z_0\in\cal{D}_{na}}\rho(z_0)q(z_0,x)\\
&=\gamma(x)+\sum_{z_0\in\cal{D}_{na}}\gamma(z_0)p(z_0,x)+\sum_{z_0\in\cal{D}_{na}}\sum_{z_1\in\cal{D}_{na}}\rho(z_0)q(z_0,z_1)p(z_1,x)\\
&=\dots=\gamma(x)+\sum_{n=1}^{m}\sum_{z_0\in\cal{D}_{na}}\dots\sum_{z_{n-1}\in\cal{D}_{na}}\gamma(z_0)p(z_0,z_1)\dots p(z_{n-1},x)\\
&\qquad\quad\quad+\sum_{z_0\in\cal{D}_{na}}\dots\sum_{z_{m}\in\cal{D}_{na}}\rho(z_0)q(z_0,z_1)p(z_1,z_2)\dots p(z_{m},x)\\
&\geq\gamma(x)+\sum_{l=1}^m\sum_{z_0\in\cal{D}_{na}}\dots\sum_{z_{l-1}\in\cal{D}_{na}}\gamma(z_0)p(z_0,z_1)\dots p(z_{l-1},x),\qquad\forall x\in\cal{D}_{na}.\end{align*}
Comparing the above with \eqref{eq:nudecomp2} and taking the limit $m\to\infty$ shows $\rho(x)\geq \nu_S(x)$ for all $x\in\cal{D}_{na}$ as desired. 
%
\end{proof}
The continuous-time version of Example~\ref{eq:jd8wndwua8ndu8wada}, with $\cal{D}$ being $\{1,2\}$ for a chain with state space $\{0,1,2\}$, initial distribution $1_1$, and rate matrix
$$Q:=\begin{bmatrix}0&0&0\\ 1&-1&0\\ 0&0&0\end{bmatrix},$$
shows that \eqref{eq:nueqsm} can indeed have other non-negative solutions aside from $\nu_S$ even if \eqref{eq:exbst} is satisfied. It is not difficult to come up with sufficient conditions for $\nu_S$ to be the unique non-negative satisfying the equation
$$\gamma(\s\backslash \cal{D})+\sum_{z\in\cal{D}}\sum_{z\neq x}\rho(z)q(z,x)=1$$
obtained by summing over $x$ in \eqref{eq:mueqsm} (under assumption~\eqref{eq:exbst}), see \citep[Corollary~2.39]{Kuntzthe}.
However, the continuous-time analogue of the question posed in Section~\ref{sec:exit} persists:
\subsubsection*{An open question} Assuming that~\eqref{eq:exbst} is satisfied, when  does \eqref{eq:nueqsm} have multiple non-negative solutions?

\subsubsection*{A proof of Theorem~\ref{charactt}}The particular definition of the chain we chose in Section~\ref{sec:FKG} comes quite in handy here. Consider a second chain $\bar{X}$ constructed using the Kendall-Gillespie algorithm (Algorithm~\ref{gilalg}) and the same $X_0$, $(\xi_n)_{n\in\zp}$, and $(U_n)_{n\in\zp}$ as for our original chain $X$ but a different rate matrix $\bar{Q}$. 
If the rows of the rate matrices $Q$ and $\bar{Q}$ coincide on the domain $\cal{D}$, then both chains are updated using the same rules while they remain inside the domain. For this reason, the chains are identical up until the instant that they simultaneously leave the domain for the first time. Formally:
\begin{lemma}\label{samechainsh}Suppose that $Q$ and $\bar{Q}$ coincide on $\cal{D}$:
$$q(x,y)=\bar{q}(x,y),\qquad\forall x\in\cal{D},\quad y\in\s.$$
There exists a continuous-time chain $\bar{X}$ with rate matrix $\bar{Q}$ defined on the same underlying space $(\Omega,\cal{F})$ as $X$  and with jump chain $\bar{Y}:=(\bar{Y}_n)_{n\in\n}$, waiting times $(\bar{S}_n)_{n\in\zp}$, jump times $(\bar{T}_n)_{n\in\n}$, and explosion time $\bar{T}_\infty$ such that:
\begin{enumerate}[label=(\roman*),noitemsep]
\item\label{samechainshi} $X$ and $\bar{X}$  exit the domain at the same time:
$$\tau(\omega)=\bar{\tau}(\omega), \quad \sigma(\omega)=\bar{\sigma}(\omega)\quad \forall\omega\in\Omega.$$
where $\tau$ and $\sigma$ defined in~\eqref{eq:hitthec2}--\eqref{eq:sigma} denote the respective times of exit for $X$ and its jump chain $Y$, and  $\bar{\tau}$ and $\bar{\sigma}$ those of $\bar{X}$ and $\bar{Y}$:
$$ \bar{\tau}(\omega):=\inf\{t\in[0,\bar{T}_\infty(\omega)):\bar{X}_t(\omega)\not\in\cal{D}\},\quad\bar{\sigma}(\omega):=\inf\{n\geq0:\bar{Y}_n(\omega)\not\in\cal{D}\},\quad\forall\omega\in\Omega.$$
\item\label{samechainshii} Up until (and including) this time, both chains have identical jump chains, waiting times, and jump times, and explosion times:
$$Y_n(\omega)=\bar{Y}_n(\omega),\quad S_n(\omega)=\bar{S}_n(\omega),\quad T_n(\omega)=\bar{T}_n(\omega),\quad\forall \omega\in\{n\leq \sigma\},\enskip n\in\n.$$
\item\label{samechainshiii} Either chain explodes no later than leaving the domain if and only if the other does, in which case the explosion times are the same:
$$\{T_\infty\leq \tau\}=\{\bar{T}_\infty\leq \bar{\tau}\},\quad T_\infty(\omega)=\bar{T}_\infty(\omega)\quad\forall \omega\in\{T_\infty\leq \tau\}.$$
\item\label{samechainshiv} In summary, both chains are identical up until they exit the domain for the first time:
$$\{t\leq\tau,t<T_\infty\}=\{t\leq\bar{\tau},t<\bar{T}_\infty\},\quad X_t(\omega)=\bar{X}_t(\omega)\quad\forall \omega \in\{t\leq\tau,t<T_\infty\},\quad\forall  t\in[0,\infty).$$
\end{enumerate}
%
%
\end{lemma}
\begin{proof} Let $\bar{P}:=(\bar{p}(x,y))_{x,y\in\s}$ and $\bar{\lambda}:=(\bar{\lambda}(x))_{x\in\s}$  denote the jump matrix and jump rates obtained by replacing $Q$ with $\bar{Q}$ in \eqref{eq:jumpmatrix}--\eqref{eq:lambda}. To construct $\bar{X}$ we $(a)$ run Algorithm \ref{gilalg} employing the same $X_0$, $(\xi_n)_{n\in\zp}$, and $(U_n)_{n\in\zp}$ as for $X$ but with $\bar{P}$ and $\bar{\lambda}$ replacing $P$ and $\lambda$ to obtain the chain's waiting times $(\bar{S}_{n})_{n\in\zp}$ and jump chain $\bar{Y}:=(\bar{Y}_{n})_{n\in\n}$ and $(b)$ define $\bar{X}$, $(\bar{T}_n)_{n\in\n}$, and $\bar{T}_\infty$ by replacing $(S_{n})_{n\in\zp}$ and $Y$ with $(\bar{S}_{n})_{n\in\zp}$ and $\bar{Y}$ in  \eqref{eq:cpathdef}--\eqref{eq:cpathdef3}. 
Because the rate matrices coincide on $\cal{D}$, \eqref{eq:jumpmatrix}--\eqref{eq:lambda} imply that the jump matrices and rates also coincide on $\cal{D}$:
$$(p(x,y))_{y\in\s}=(\bar{p}(x,y))_{y\in\s},\quad \lambda(x)=\bar{\lambda(x)},\quad \forall x\in\cal{D}.$$
Algorithm \ref{gilalg} and the above imply that  
\begin{equation}\label{eq:nja}\bar{Y}_{n+1}(\omega)=Y_{n+1}(\omega)\quad\forall\omega\in\{\bar{Y}_n=Y_n\in\cal{D}\}.\end{equation}
Because  $Y_0 =X_0= \bar{Y}_0$, induction and the above imply that
\begin{equation}\label{eq:nja2}\{Y_0\in\cal{D},\dots,Y_{n-1}\in\cal{D},Y_n=x\}=\{\bar{Y}_0\in\cal{D},\dots,\bar{Y}_{n-1}\in\cal{D},\bar{Y}_n=x\}\quad\forall x\in\s,\enskip n>0.\end{equation}
Due to the definition of the exit times of the jump chains, we have that
\begin{align*}\sigma&=\infty\cdot 1_{\{Y_0\in\cal{D},Y_1\in\cal{D},\dots\}}+\sum_{n=1}^\infty n1_{\{Y_0\in\cal{D},\dots,Y_{n-1}\in\cal{D},Y_n\not\in\cal{D}\}},\\
 \bar{\sigma}&=\infty\cdot 1_{\{\bar{Y}_0\in\cal{D},\bar{Y}_1\in\cal{D},\dots\}}+\sum_{n=1}^\infty n1_{\{\bar{Y}_0\in\cal{D},\dots,\bar{Y}_{n-1}\in\cal{D},\bar{Y}_n\not\in\cal{D}\}}. \end{align*}
Plugging  \eqref{eq:nja2} into the above, we obtain the rightmost equation in $(i)$. Since $ n\leq \sigma(\omega)$ only if
$$Y_0(\omega)\in\cal{D},\enskip Y_n(\omega)\in\cal{D},\enskip\dots,\enskip Y_{n-1}(\omega)\in\cal{D},$$
the leftmost equation in $(ii)$ also follows from \eqref{eq:nja}. The definition of the waiting times (Algorithm~\ref{gilalg}) and that of the jump times \eqref{eq:cpathdef2}, then yield the other two equations in  $(ii)$. The leftmost equation in $(i)$ then follows from the other equation in $(i)$, the rightmost one in $(ii)$, and Proposition~\ref{prop:hitdc}. This proposition and \eqref{eq:etatinfinf} imply that
$$\{T_\infty\leq \tau\}=\{\tau=\infty\}=\{\sigma=\infty\},\qquad\{\bar{T}_\infty\leq \bar{\tau}\}=\{\bar{\tau}=\infty\}=\{\bar{\sigma}=\infty\}$$
For this reason, the leftmost equation in $(iii)$ follows from $(i)$. Similarly, given that $\{T_\infty\leq \tau\}=\{\sigma=\infty\}$, picking any $\omega$ in this set and taking the limit $n\to\infty$ of the rightmost equation in $(ii)$ yields the rightmost one in $(iii)$.

To complete the proof, note that $\{T_\infty\leq \tau\}=\{\tau=\infty\}$ further implies that
\begin{align*}\{t\leq\tau,t<T_\infty\}=\{t<T_\infty,\tau=\infty\}\cup\{t\leq\tau<T_\infty\}=\{t<T_\infty,\tau=\infty\}\cup\{t\leq\tau<\infty\}.\end{align*}
Because the same equations hold if we replace $\tau$ and $T_\infty$ with $\bar{\tau}$ and $\bar{T}_\infty$, the leftmost equation in $(iv)$ follows from the leftmost one in $(i)$ and the rightmost one in $(iii)$. Because
$$\{t\leq\tau,t<T_\infty\}=\bigcup_{n=0}^\infty\{t\leq\tau, T_n\leq t<T_{n+1}\}$$
and $\{T_n\leq\tau\}=\{\sigma\leq n\}$ (Proposition~\ref{prop:hitdc}), the rightmost equation in $(iv)$ then follows from the leftmost, the definition of the paths of $X$ and $\bar{X}$ in~\eqref{eq:cpathdef}, and  $(ii)$.
\end{proof}

We are now ready to tackle the proof of Theorem~\ref{charactt}:
\begin{proof}[Proof of Theorem~\ref{charactt}] Let $\hat{X}$ be as in Lemma~\ref{samechainsh} with $\hat{Q}=(\hat{q}(x,y))_{x,y\in\s}$ in \eqref{eq:qhat} replacing $\bar{Q}$ in the lemma's premise. The key ingredients in this proof are the facts that $\hat{X}$ and $X$ are identical up until the moment that they simultaneously exit the domain (Lemma~\ref{samechainsh}) and that $\hat{X}$ gets stuck in an absorbing state the instant it leaves the domain. We begin by proving the latter.

If $\hat{Y}=(\hat{Y}_n)_{n\in\n}$ denotes the jump of $\hat{X}$ and $\hat{\sigma}$ its exit time from $\cal{D}$, then $\hat{Y}_{\hat{\sigma}(\omega)}(\omega)$ lies outside $\cal{D}$ for all $\omega$ in $\{\hat{\sigma}<\infty\}$ and the definition of $\hat{Q}$ in~\eqref{eq:qhat} implies that
$$\hat{q}(\hat{Y}_{\hat{\sigma}(\omega)}(\omega))=0\quad\forall \omega\in\{\hat{\sigma}<\infty\}.$$
The Kendall-Gillespie algorithm (Algorithm~\ref{gilalg}) then implies that $\hat{Y}_{\hat{\sigma}(\omega)+1}(\omega)=\hat{Y}_{\hat{\sigma}(\omega)}(\omega)$ and, hence, $\hat{q}(\hat{Y}_{\hat{\sigma}(\omega)+1}(\omega))=0$, for all $\omega$ in $\{\hat{\sigma}<\infty\}$. Iterating this argument forward, we find that
\begin{equation}\label{eq:absjump}\hat{Y}_{\hat{\sigma}(\omega)+n}(\omega)=\hat{Y}_{\hat{\sigma}(\omega)}(\omega)\not\in\cal{D}\quad\forall n\in\n,\enskip \omega\in\{\hat{\sigma}<\infty\}.\end{equation}
Next, because  Proposition~\ref{prop:hitdc} implies that
$$\hat{T}_n\vee\hat{\tau}=\hat{T}_n\vee\hat{T}_{\hat{\sigma}}=\hat{T}_{n\vee\hat{\sigma}}\quad\forall n\in\n,$$
where $\hat{\tau}$ denotes the exit time of $\hat{X}$, we have that
$$\{\tau\leq t,\hat{T}_n\leq t<\hat{T}_{n+1}\}=\{\hat{T}_n\vee\tau\leq t<\hat{T}_{n+1}\}=\{\hat{T}_{n\vee\hat{\sigma}}\leq t<\hat{T}_{n+1}\}=\{\hat{\sigma}\leq n, \hat{T}_n\leq t<\hat{T}_{n+1}\}$$
for all $n$ in $\n$ and $t$ in $[0,\infty)$. The above, Proposition~\ref{prop:hitdc}, and \eqref{eq:absjump} then show that $\hat{X}$ does indeed get stuck in the first state it enters upon leaving the domain:
\begin{align}\nonumber \{\hat{\tau}\leq t<\hat{T}_\infty,\hat{X}_t=x\}&=\bigcup_{n=0}^\infty \{\hat{\tau}\leq t,\hat{T}_n\leq t< \hat{T}_{n+1},\hat{Y}_n=x\}\\
&=\bigcup_{n=0}^\infty \{\hat{\sigma}\leq n,\hat{T}_n\leq t< \hat{T}_{n+1},\hat{Y}_n=x\}\nonumber\\
&=\bigcup_{n=0}^\infty \{\hat{\sigma}\leq n,\hat{T}_n\leq t< \hat{T}_{n+1},\hat{Y}_{\hat{\sigma}}=x\}\nonumber\\
&=\{\hat{\sigma}<\infty,\hat{T}_{\hat{\sigma}}\leq t<\hat{T}_\infty,\hat{Y}_{\hat{\sigma}}=x\}\nonumber\\
&=\{\hat{\tau}\leq t<\hat{T}_\infty,\hat{X}_{\hat{\tau}}=x\}\enskip\forall t\in[0,\infty),\enskip x\in\s.\label{eq:stuck}\end{align}

Now, on to the characterisations of $\mu$ and $\nu$ in \eqref{eq:jointchar1}--\eqref{eq:jointchar2}:
\begin{align}\mu([0,t],x)&=\Pb(\{X_{\tau}=x,\tau\leq t\})=\Pb(\{Y_\sigma=x,\sigma<\infty,T_{\sigma}\leq t\})=\Pb(\{\hat{Y}_{\hat{\sigma}}=x,\hat{\sigma}<\infty,\hat{T}_{\hat{\sigma}}\leq t\})\nonumber\\
&=\Pb(\{\hat{X}_{\hat{\tau}}=x,\hat{\tau}\leq t<\hat{T}_\infty\})=\Pb(\{\hat{X}_{t}=x,\hat{\tau}\leq t< \hat{T}_\infty\})=\hat{p}_t(x)\enskip\forall x\not\in\cal{D},\enskip t\in[0,\infty),\label{eq:mdw89amda8w9}
\end{align}
where $\hat{p}_t$ denotes the time-varying law of $\hat{X}$:
$$\hat{p}_t(x):=\Pb_\gamma(\{\hat{X}_t=x,t<\hat{T}_\infty\})\quad\forall x\in\s,\enskip t\in[0,\infty).$$
The first equation in \eqref{eq:mdw89amda8w9} follows from the definition of $\mu$, the second from Proposition~\ref{prop:hitdc}, the third from Lemma \ref{samechainsh}$(i,ii)$, the fourth from Proposition~\ref{prop:hitdc}, the fifth from \eqref{eq:stuck}. Because~\eqref{eq:nd7a73a8dhd} and Proposition~\ref{prop:forwardcondmild} imply that~\eqref{eq:forwardcond} holds for $\hat{Q}$, Theorem~\ref{thrm:forward} shows that $\hat{p}_t$ is continuously differentiable and the minimal solution of~\eqref{eq:jointchar2}. Exploiting the continuity of $t\mapsto \hat{p}_t(x)$ and applying monotone convergence, we find that
$$\mu([0,t),x)=\lim_{n\to\infty}\mu([0,t(1-1/n)],x)=\lim_{n\to\infty}\hat{p}_{t(1-1/n)}(x)=\hat{p}_t(x),\quad\forall x\not\in\cal{D},$$
and the first equation in \eqref{eq:jointchar1} follows. 

To argue the second equation in~\eqref{eq:jointchar1}, note that Lemma~\ref{samechainsh}$(i,iii,iv)$ implies that
\begin{align}\label{eq:fnhyduafns}\int_0^{t\wedge\tau\wedge T_\infty}1_x(X_s)ds=&\int_0^{t\wedge\hat{\tau}\wedge \hat{T}_\infty}1_x(\hat{X}_s)ds=1_{\{\hat{\tau}\leq t\}}\int_0^{\hat{\tau}\wedge\hat{T}_\infty}1_x(\hat{X}_s)ds\\
&+1_{\{\hat{\tau}>t\}}\int_0^{t\wedge\hat{T}_\infty}1_x(\hat{X}_s)ds\nonumber\quad\forall x\in\cal{D},\enskip t\in[0,\infty).\end{align}
However, because $\hat{X}_{\hat{\tau}}$ lies outside $\cal{D}$ whenever $\hat{\tau}$ is finite (Proposition~\ref{prop:hitdc}), \eqref{eq:stuck} shows that
\begin{align*}1_{\{\hat{\tau}\leq t\}}\int_{\hat{\tau}\wedge \hat{T}_\infty}^{t\wedge \hat{T}_\infty}1_x(\hat{X}_s)ds&=\int_{0}^{\infty}1_{\{\hat{\tau}\leq s<t\wedge\hat{T}_\infty\}}1_x(\hat{X}_s)ds=\int_{0}^{\infty}1_{\{\hat{\tau}\leq s<t\wedge\hat{T}_\infty\}}1_x(\hat{X}_{\hat{\tau}})ds\\
&=1_{\{\hat{\tau}\leq t\}}1_x(\hat{X}_{\hat{\tau}})(t\wedge\hat{T}_\infty-\hat{\tau})=0\quad\forall x\in\cal{D}.\end{align*}
Adding the left-hand side of the above to the right-hand side of \eqref{eq:fnhyduafns}, taking expectations, using Tonelli's theorem, we find that
\begin{equation}\label{eq:dens2}\nu([0,t),x)=\Ebl{\int_0^{t\wedge\tau\wedge T_\infty}1_x(X_s)ds}=\Ebl{\int_0^{t\wedge\hat{T}_\infty}1_x(\hat{X}_s)ds}=\int_0^t\hat{p}_s(x)ds\quad\forall x\in\cal{D},\end{equation}
and the second equation in~\eqref{eq:jointchar1} follows.

\end{proof}

\ifdraft

\subsubsection{A curiosity}
For the proof of Theorem \ref{charactt} we need the following description of the (for the lack of a better name) explosion-before-exit event.
\begin{lemma}\label{exbst} Let $\hat{X}$ be as in Lemma~\ref{samechainsh} with $\hat{Q}=(\hat{q}(x,y))_{x,y\in\s}$ in \eqref{eq:qhat} replacing $\bar{Q}$ in the lemma's premise. The chain $X$ does not explode before first leaving the domain if and only if $\hat{X}$ does not explode:
$$1_{\{\tau\leq T_\infty\}}=1_{\{\hat{T}_\infty=\infty\}}\quad \Pb_\gamma\text{-almost surely.}$$
%
\end{lemma}

\begin{proof}By its definition in~\eqref{eq:hitthec2}, the exit time is no greater than the explosion time if and only if the chain exits the domain  or the chain neither exits the domain nor explodes:
$$\{\tau\leq T_\infty\}=\{\tau<T_\infty\}\cup\{\tau=T_\infty=\infty\}=\{\tau<\infty\}\cup\{\tau=T_\infty=\infty\}.$$
Because these events are disjoint, we need only argue that
\begin{equation}\label{eq:1nfhdsf}1_{\{\tau<\infty\}}=1_{\{\tau<\infty,\hat{T}_\infty=\infty\}},\quad 1_{\{\tau=T_\infty=\infty\}}=1_{\{\tau=\hat{T}_\infty=\infty\}},\quad \Pb_\gamma\text{-almost surely.}\end{equation}
However, the second equation follows directly from Lemma~\ref{samechainsh}$(iii)$. For the first equation note that if $\hat{Y}=(\hat{Y}_n)_{n\in\n}$ denotes the jump of $\hat{X}$ and $\hat{\sigma}$ its exit time from $\cal{D}$, then $\hat{Y}_{\hat{\sigma}(\omega)}(\omega)$ lies outside $\cal{D}$ for all $\omega$ in $\{\hat{\sigma}<\infty\}$ and the definition of $\hat{Q}$ in~\eqref{eq:qhat} implies that
$$\hat{q}(\hat{Y}_{\hat{\sigma}(\omega)}(\omega))=0\quad\forall \omega\in\{\hat{\sigma}<\infty\}.$$
The Kendall-Gillespie algorithm (Algorithm~\ref{gilalg}) then implies that $\hat{Y}_{\hat{\sigma}(\omega)+1}(\omega)=\hat{Y}_{\hat{\sigma}(\omega)}(\omega)$ and, hence, $\hat{q}(\hat{Y}_{\hat{\sigma}(\omega)+1}(\omega))=0$, for all $\omega$ in $\{\hat{\sigma}<\infty\}$. Iterating this argument forward, we find that
\begin{equation}\label{eq:absjump}\hat{Y}_{\hat{\sigma}(\omega)+n}(\omega)=\hat{Y}_{\hat{\sigma}(\omega)}(\omega)\not\in\cal{D}\quad\forall n\geq0,\enskip \omega\in\{\hat{\sigma}<\infty\}.\end{equation}
Thus, the law of large numbers (Theorem~\ref{thrm:klln}) and the definitions of the waiting times in Algorithm~\ref{gilalg} imply that
\begin{align*}1_{\{\hat{\sigma}=m\}}\hat{T}_\infty&=1_{\{\hat{\sigma}=m\}}\sum_{n=1}^\infty \hat{S}_n\geq 1_{\{\hat{\sigma}=m\}}\sum_{n=m}^\infty \hat{S}_n\geq 1_{\{\hat{\sigma}=m\}}\sum_{n=m}^\infty \frac{\xi_{n+1}}{\hat{\lambda}(\hat{Y}_{n-1})}\\
&=1_{\{\hat{\sigma}=m\}}\sum_{n=m}^\infty \xi_{k+1}=1_{\{\hat{\sigma}=m\}}\cdot\infty\quad\forall m\geq0,\enskip\Pb_\gamma\text{-almost surely}.\end{align*}
Summing the above over $m$ in $\n$, we find that
$$1_{\{\hat{\sigma}<\infty\}}\hat{T}_\infty\geq1_{\{\hat{\sigma}<\infty\}}\cdot\infty\qquad\Pb_\gamma\text{-almost surely}.$$
The first equation in \eqref{eq:1nfhdsf} then follows because Lemma~\ref{samechainsh}$(i)$ shows that $\hat{\sigma}=\sigma$ and Proposition~\ref{prop:hitdc} shows that $\{\sigma<\infty\}=\{\tau<\infty\}$, where $\sigma$ denotes the exit time in~\eqref{eq:sigma} of $X$'s jump chain.
\end{proof}

\fi
\subsubsection*{Notes and references} The treatment in this section follows that in~\citep{Kuntzthe,Kuntz2019}. However, the simple ideas underpinning the above characterisation are classical. For instance, the following is taken from \citep[p.494]{Feller1971}:
\begin{quotation} 
``To show how the distribution of recurrence and first passage times may be calculated we number the states $0,1,2,\dots$ and use $0$ as pivotal state. Consider a new process which coincides with the original process up to the random epoch of the first visit to $0$ but with the state fixed at $0$ forever after. In other words, the new process is obtained from the old one by making $0$ an absorbing state. Denote the transition probabilities of the modified process by $^0P_{ik}(t)$. Then $^0P_{00}(t) = 1$. In terms of the original process $^0P_{i0}(t)$ is the probability of a first passage from $i\neq0$ to $0$ before epoch $t$, and $^0P_{ik}(t)$ gives the probability of a transition from $i\neq0$ to $k\neq0$ without intermediate passage through $0$. It is probabilistically clear that the matrix $^0P(t)$ should satisfy the same backward and forward equations as $P(t)$ except that $q(0)$ [in the notation of \eqref{eq:qmatrix}] is replaced by $0$.''
\end{quotation}
%
I particularly like \citep{Syski1992} on this subject. 

\ifdraft

We now pause to point out the difference with the discrete-time case discussed in Section \ref{sec:morehitdt}.

\begin{remark}\label{exitnotstop}Equations \eqref{eq:jointchar1}--\eqref{eq:jointchar2} are not analogous to \eqref{eq:eoed} given for the discrete-time case in Section \ref{sec:morehitdt}. For one, \eqref{eq:eoed} apply to general stopping times while \eqref{eq:jointchar1}--\eqref{eq:jointchar2} apply only to exit times. Equations \eqref{eq:eoed} imply the discrete-time analogue of \eqref{eq:jointchar1}--\eqref{eq:jointchar2}: equations \eqref{eq:eoed2t} given in Section \ref{sec:morehitdt}, but the implication is only one way. To obtain the direct continuous-time analogue of \eqref{eq:eoed} satisfied by the exit distribution $\mu$ and occupation measure $\nu$ of a general stopping time $\eta$, one would have to prove that for each $x\in\s$
\begin{enumerate}[label=(\alph*)]
\item $\mu(dt,x)$ decomposes into
$$\mu(dt,x)=\mu(\{0\},x)\delta_0(dt)+\mu^c(dt,x),$$
where $\mu^c(dt,x)$ has a density $\mu^c(t,x)$ with respect to the Lebesgue measure such that $\mu^c_\tau(0,x)=0$.
\item $\nu(dt,x)$ has a density $\nu(t,x)$ with respect to the Lebesgue measure and $t\mapsto\nu(t,x)$ is a differentiable function.
\end{enumerate}
Using the above, a time-space version of Dynkin's Formula, a mollifier, and integration by parts one can then show that
\begin{align*}
& \dot{\nu}(t,x)+\mu^c(t,x)=\sum_{y\in\s}\nu(t,y)q(y,x),\qquad \forall  x\in\s,\quad t\geq0,\\ 
&\nu(0,x)+\mu(\{0\},x)=\lambda(x),\qquad  \forall x\in\s.\end{align*}
These are the continuous-time version of the equations in Lemma \ref{eqnsd}, and \eqref{eq:jointchar1}--\eqref{eq:jointchar2} can be derived from the above by fixing $\eta$ to be an exit time. While we know that the exit distribution and occupation measure of exit times satisfy the (a) and (b) above, we do not know whether those of general stopping times do too. This remains an open question.
\end{remark}

\fi
%

There appears to be some confusion in the literature regarding equations~\eqref{eq:nueqsm}--\eqref{eq:mueqsm} satisfied by the marginals. Sometimes (for example\footnote{We only consider this issue for chains, while the referenced works consider more general Markov processes. However, with some work, I believe that the arguments presented in this section, and in \citep[Chapter~2]{Kuntzthe}, can be ported over to the more general case.}, \citep{Helmes2001,Helmes2002,Lasserre2004,Lasserre2006}) it is assumed that \eqref{eq:nueqsm} has no other solutions aside from $\nu_S$. This assumption turns out to be flawed: as the example given immediately after the proof of Theorem~\ref{charact} shows, \eqref{eq:nueqsm} can indeed have other solutions if the domain contains absorbing sets that are unreachable from the support of the initial distribution. 



\subsection{On the chain's definition, minimality, and right-continuity*}\label{sec:otherchains}
In this section, we address the same question we did in Section~\ref{sec:otherdef} for discrete-time chains: are all chains made equal? Or do the particulars of the chain's definition matter? 

In the discrete-time case the answer was simply `no, the details of the chain's definition do not matter': the path law of a discrete-time chain
depends on its one-step matrix $P$ and its initial distribution $\gamma$ but not on its particular construction.

In the continuous-time case, the answer is a bit more complicated. In short, if we only consider chains (i.e., continuous-time processes taking values in countable sets $\s$ and satisfying the Markov property) whose paths posses some basic regularity properties (properties that are typically taken for granted in practice), then the particulars of the definition do not matter. However, in general, they do matter for reasons stemming from the uncountability of the time-axis allowing for a lot of bad behaviour (behaviour that is tamed by the regularity assumptions).


The properties are two: \emph{right-continuity of the paths}\index{right-continuous paths} and \emph{minimality of the chain}. Right-continuity of the paths (w.r.t. the discrete topology on $\s$) means that $(a)$ upon entering a state, the chain lingers in the state for at least a short amount on time and $(b)$ the chain enters states at  well-defined points in time (for example, it cannot be the case that $X_{t'}(\omega)=x$ for all $t'$ in some interval $(t,t+\varepsilon)$ but $X_t(\omega)\neq x$). From an applied point of view, chains whose paths do not satisfy these properties can seem very odd indeed. For instance, there are chains that immediately jump out of every state they enter \citep[Sections~III.23, 35]{Rogers2000a}. What is going on in these cases is that the state space is countable but not `discrete' in the sense that its states are not `well-separated' from each other (for instance, $\s$ might be the set of rational numbers $\mathbb{Q}$). For such chains, the theory is substantially more involved as the discrete topology on $\s$ must be replaced with the more complicated \emph{Ray-Knight} topology. For a lovely account of this theory, see \citep{Rogers2000a}. On the other hand, the right-continuous case is simple: the time elapsed between two consecutive jumps is exponentially distributed with parameter $\lambda(x)$ depending on the chain's state  $x$ immediately prior the second jump and the probability $p(x,y)$ that the chain next jumps to $y$ if it lies in $x$ depends only on $x$ and $y$.

As we have seen in Section~\ref{sec:FKG}, chains with right-continuous paths can explode: infinitely many jumps accumulate in a finite amount of time and the chain diverges to infinity in the sense captured by Theorem~\ref{tautin}. In these cases, it possible to continue the chain in ways that the  Markov property is preserved. For example, at the moment of explosion sample a state from the initial distribution, re-initialise the chain at this state, and continue the sample path by running the Kendall-Gillespie algorithm (Section~\ref{sec:FKG}) with a new set of independent random variables. Thus, explosions introduce an ambiguity in the chain's definition (the process ``runs out of instructions'' at the explosion time, \citep[p.43]{Asmussen2003}) and depending on whether the chain is re-initialised and how, we obtain chains with different path laws. Of course, a chain that is restarted has probability to be in any given state $x$ at any given time $t$ no smaller than a chain that is not restarted: the probability of the former being in $x$ at $t$ equals the probability that it has not exploded by $t$ and lies in $x$ at $t$, plus the probability that it has exploded once by $t$, been re-initialised, and lies in $x$ at time $t$, etc.; while probability of the latter being in $x$ at $t$ is simply the probability that it has not exploded by $t$ and lies in $x$ at $t$. For this reason, chains that do not continue past an explosion are called \emph{minimal}\index{minimal chains}.

Imposing right-continuity and minimality removes all ambiguity that otherwise may arise in the chain's definition: in this case, the path law is fully determined by the chain's jump rates $\lambda:=(\lambda(x))_{x\in\s}$ and jump matrix $P:=(p(x,y))_{x,y\in\s}$ (or, equivalently, its rate matrix) and its initial distribution $\gamma$. Thus, for a given $\lambda$, $P$, and $\gamma$, it does not matter how we build the chain as long as its vector of jump rates is $\lambda$, its jump matrix is $P$, and its initial distribution is $\gamma$. The remainder of this section is dedicated to formalising this fact in Theorem~\ref{thrm:otherchains} below (see the ensuing material for the terminology used in the theorem's statement). To not overly complicate the exposition, we will focus on the case of chains without absorbing states.
\begin{theorem}[Uniqueness of the path law of minimal right-continuous chains]\label{thrm:otherchains}Suppose that there are no absorbing states (i.e., that~\eqref{eq:ninfty} below holds). The path law 
$$\mathbb{L}_\gamma(A):=\Pbl{\{X\in A\}}\quad\forall A\in\cal{E}$$
of any minimal right-continuous chain $X$ with state space $\s$ and initial distribution $\gamma$ is the only probability measure on $(\cal{P},\cal{E})$  satisfying~\eqref{eq:pathlawct} for all cylinder sets (i.e., sets of the form~\eqref{eq:nfanfewuiafnueia}), where the jump rates $(\lambda(x))_{x\in\s}$ and jump matrix $(p(x,y))_{x,y\in\s}$ are given by
\begin{align}\label{eq:lambda2}&\lambda(x)=-\ln(\Pbx{\{T_1>1\}})\quad\forall x\in\s,\\
\label{eq:jumpmatrix2}&p(x,y)=\Pbx{\{X_{T_1}=y\}}\quad\forall x,y\in\s.\end{align}
with $T_1$ denoting the chain's first jump time~\eqref{eq:jumptimesother}. Moreover, $\lambda(x)$ is positive and finite for all $x$ in $\s$ and $P=(p(x,y))_{x,y\in\s}$ is a one-step matrix (i.e., satisfies~\eqref{eq:1step}) with $p(x,x)=0$ for all $x$ in $\s$. For these reasons,
\begin{equation}\label{eq:qmatrix2}q(x,y):=\left\{\begin{array}{ll}-\lambda(x)&\text{if }x=y,\\ \lambda(x)p(x,y)&\text{if }x\neq y\end{array}\right.\quad\forall x,y\in\s,\end{equation}
defines a stable and conservative rate matrix $Q:=(q(x,y))_{x,y\in\s}$ (i.e., $Q$ satisfies \eqref{eq:qmatrix}) and we can recover $\lambda$ and $P$ from $Q$:
$$\lambda(x)=-q(x,x)\enskip\forall x\in\s,\quad p(x,y)=(1_x(y)-1)q(x,y)/q(x,x)\enskip\forall x,y\in\s.$$
Thus, the path law is fully determined by the initial distribution $\gamma$ and rate matrix $Q$.
\end{theorem}
\subsubsection*{Continuous-time chains}In general, a `continuous-time chain taking values in countable set $\s$' is a model composed of
\begin{itemize}
\item an \emph{underlying space}: a measurable space  $(\Omega,\cal{F})$;
\item a set of \emph{underlying probability measures}: a collection $(\Pb_x)_{x\in\s}$ of probability measures $\Pb_x$ on $(\Omega,\cal{F})$ indexed by \emph{states} $x$ in $\s$;
\item a \emph{final time} $T_f$: an $\cal{F}/\cal{B}(\r_E)$-measurable function mapping from $\Omega$ to $(0,\infty]$;
\item a \emph{chain} $X$: a collection $(X_t)_{t\geq0}$ of functions $X_t$ indexed by \emph{times} $t$ in $[0,\infty)$ such that 
\begin{enumerate}
\item $X$ is defined up until $T_f$: for each $t$ in $[0,\infty)$, $X_t$ maps from $\{t<T_f\}$ to $\s$,
\item $X$ satisfies an appropriate measurability requirement: for each $x$ in $\s$ and $t$ in $[0,\infty)$, $\{t<T_f,X_t=x\}$ belongs to $\cal{F}$,
\item under $\Pb_x$, $X$ starts at $x$: for each $x$ in $\s$, $\Pbx{\{X_0=x\}}=1$ (note that $\{0<T_f\}=\Omega$ by $T_f$'s definition), 
\item immediately prior $T_f$, $X$ cannot be in one particular state:
$$\liminf_{t\uparrow T_f(\omega)}1_{x}(X_t(\omega))=0\quad\forall x\in\s,\enskip\omega\in\{T_f<\infty\}$$
where $t\uparrow T_f(\omega)$ means that we are taking the limit infimum from below (that is, we are only considering $t$s strictly less than $T_f(\omega)$).

\item the collection $X$ satisfies the Markov property:
\begin{align*}&\Pbz{A\cap\{t<T_f,X_t=x\}\cap\{t+s<T_f,X_{t+s}=y\}}\\
&=\Pbz{A\cap\{t<T_f,X_t=x\}}\Pbx{\{s<T_f,X_s=y\}}\quad\forall x,y,z\in\s,\enskip t,s\in[0,\infty),\end{align*}
where $A$ denotes any event whose occurrence we can deduce from observing the chain up until $t$, i.e.  $A$ belongs to the sigma-algebra $\cal{F}_t$ generated by the events
\begin{equation}\label{eq:ndeyadnyuaenyda}\{X_{s}=x,s< T_f\}\quad\forall s\in[0,t],\enskip x\in\s.\end{equation} 
%
%
\end{enumerate}
\end{itemize}
Just as for the chains of Section~\ref{sec:FKG}, if the chain's starting position is sampled from a probability distribution $\gamma=(\gamma(x))_{x\in\s}$, then we use $\Pb_\gamma:=\sum_{x\in\s}\gamma(x)\Pb_x$ to describe the chain's statistics.

\subsubsection*{Chains with right-continuous paths}\index{right-continuous paths} The chain is said to have \emph{right-continuous} paths if for any $t$ in $[0,\infty)$ and $\omega$ in $\{t<T_f\}$, there exists an $0<\varepsilon<T_f(\omega)$ such that
$$X_{t+\varepsilon'}(\omega)=X_t(\omega)\quad\forall \varepsilon'\in[0,\varepsilon].$$
In this case, the \emph{jump times}
\begin{equation}\label{eq:jumptimesother}T_0(\omega):=0,\quad T_{n+1}(\omega):=\inf\{t\in[T_n(\omega),T_f(\omega)):X_t(\omega)\neq X_{T_n(\omega)}(\omega)\}\quad\forall n\in\n,\enskip \omega\in\Omega,\end{equation}
form an increasing sequence $(T_n(\omega))_{n\in\n}$. We use $N$ to denote the number of the chain's last jump:
$$N(\omega):=\sup\{n\in\n:T_n(\omega)<\infty\}\quad\forall \omega\in\Omega.$$
Due to these definitions,
$$T_{N(\omega)}<T_f(\omega)\quad\text{and}\quad\quad X_t(\omega)=X_{T_{N(\omega)}}(\omega)\quad\forall t\in[T_{N(\omega)}(\omega),T_f(\omega))$$
if the chain stops jumping ($N(\omega)<\infty$). In this case, 4 in the definition of $X$ above implies that the final time must be infinite as
\begin{align*}\{N<\infty,T_f<\infty\}&=\{N<\infty,X_t=X_{T_N}\enskip\forall t\in[T_N,T_f),T_f<\infty\}\\
&=\bigcup_{x\in\s}\{N<\infty,X_t=X_{T_N}=x\enskip\forall t\in[T_N,T_f),T_f<\infty\}\\
&\subseteq \bigcup_{x\in\s}\left\{\liminf_{t\uparrow T_f} 1_x(X_t)=1,T_f<\infty\right\}=\emptyset.\end{align*}
Putting the above together we find that the chain stops jumping if only if it gets stuck in some \emph{absorbing state}:
$$\{N<\infty\}=\{N<\infty,T_f=\infty,X_t=X_{T_N}\enskip\forall t\in[T_N,\infty)\}.$$
%
%
%
To not overly complicate the exposition of this section, we assume that this never happens:
\begin{equation}\label{eq:ninfty}N(\omega)=\infty\quad\forall \omega\in\Omega.\end{equation}
In this case, the jump times' definition and the right-continuity of the paths imply that the \emph{waiting times},
\begin{equation}\label{eq:waittimesgen}S_n(\omega):=T_{n}(\omega)-T_{n-1}(\omega)\quad\forall \omega\in\Omega,\enskip n\in\zp,\end{equation}
are all well-defined and positive; and that
\begin{equation}\label{eq:chaingendef}X_t(\omega)=Y_n(\omega)\quad\forall t\in[T_n(\omega),T_{n+1}(\omega)),\enskip  \omega\in\Omega,\enskip n\in\n,\end{equation}
where $Y=(Y_n)_{n\in\n}$ is the discrete-time process obtained by sampling the chain at the jump times:
\begin{equation}\label{eq:jumpchaingen}Y_n(\omega):=X_{T_n(\omega)}(\omega)\quad\forall \omega\in\Omega,\enskip n\in\n.\end{equation}
\subsubsection*{The explosion time and minimal chains with right-continuous paths}We refer to the limit of the jump times as the \emph{explosion time} $T_\infty$:
$$T_\infty(\omega)=\lim_{n\to\infty}T_n(\omega).$$
Our assumption~\eqref{eq:ninfty} that the chain never stops jumping implies that $T_n(\omega)<\infty$ for all $n$ and $\omega$. The definition of the jump times~\eqref{eq:jumptimesother} shows that this can only be the cases if $T_n(\omega)<T_f(\omega)$ for all $n$ and $\omega$. Taking the limit $n\to\infty$ then shows that the explosion time cannot be greater than the final time:
$$T_\infty(\omega)\leq T_f(\omega)\quad\forall \omega\in\Omega.$$
The chain is said to be \emph{minimal}\index{minimal chains} if the explosion time is the final time:
\begin{equation}\label{eq:tfti}T_\infty(\omega)=T_f(\omega)\quad\forall\omega\in\Omega.
\end{equation}

\subsubsection*{The strong Markov property}For any minimal chain $X$ with right-continuous paths, \eqref{eq:chaingendef} shows that the chain is fully determined by its waiting times $(S_n)_{n\in\zp}$ in~\eqref{eq:waittimesgen} and jump chain $Y:=(Y_n)_{n\in\n}$ in~\eqref{eq:jumpchaingen}. That is, we can view the chain $X$ as a function  mapping from $\Omega$ to  path space $\cal{P}$ introduced in Section~\ref{sec:pathspacect}. It is not too difficult to show that $X$ is $\cal{F}/\cal{E}$-measurable (do it!), where $\cal{E}$ denotes the cylinder sigma algebra on $\cal{P}$ also introduced in Section~\ref{sec:pathspacect}. Moreover, it follows from 5 in the chain's definition that, for any $A$ in $\cal{F}_t$ and set $B$ of the type in \eqref{eq:gensetsalt},
\begin{align*}&\Pbz{A\cap\{t<T_\infty,X_t=x\}\cap\{X^t\in B\}}\\
&=\Pbz{A\cap \{t<T_\infty,X_t=x\}\cap\{t+t_n<T_\infty,X_{t+t_1}=z_1,\dots,X_{t+t_n}=z_n\}}\\
&=\Pbz{A\cap \{t<T_\infty,X_t=x\}}\Pbx{\{t_n<T_\infty,X_{t_1}=z_1,\dots,X_{t_n}=z_n\}}\\
&=\Pbz{A\cap\{t<T_\infty,X_t=x\}}\Pbx{\{X\in B\}},\end{align*}
where $X^t$ denotes the $t$-shifted chain (as in Section~\ref{sec:markovprop}). Because the sets $B$ of this type generate $\cal{E}$ (Exercise~\ref{exe:measpedantry}), the same arguments as those given in the proofs of Theorems~\ref{thrm:markovprop} and~\ref{thstrmk} then show that $X$ possesses the strong Markov property: that is, Theorem~\ref{thstrmk} also applies to $X$ (here, it is important notice that the filtration $(\cal{F}_t)_{t\geq0}$ defined by~\eqref{eq:ndeyadnyuaenyda} coincides with that in Definition~\ref{def:filt3}, see Exercise~\ref{ex:filtrd}). 

\subsubsection*{The waiting times, the jump chain, and the path law}It follows from Exercise~\ref{ex:filtrd} that the jump times are $(\cal{F}_t)_{t\geq0}$-stopping times (Definition~\ref{def:stopct}). The strong Markov property~\eqref{eq:strmkvct} 
then implies that
\begin{align*}\Pbz{A\cap\{Y_n=x,Y_{n+1}=y,S_{n+1}>s\}}&=\Pb_z(A\cap\{X_{T_n}=x,Y_{1}^{T_n}=y,S_{1}^{T_n}>s\})\\
&=\Pbz{A\cap\{X_{T_n}=x\}}\Pbx{\{Y_1=y,S_{1}>s\}}\\
&=\Pbz{A\cap\{Y_n=x\}}\Pbx{\{Y_1=y,S_{1}>s\}}\end{align*}
for all natural numbers $n$, events $A$ in the pre-$T_n$ sigma-algebra $\cal{F}_{T_n}$, states $x,y,z$ in $\s$, and $s\geq0$.
Because Proposition~\ref{prop:ftngn} shows that the pre-$T_n$ sigma-algebra contains the sigma-algebra $\cal{G}_n$ generated by $Y_0,\dots, Y_n$ and $S_1,\dots,S_n$ (Definition~\ref{def:filt2}), the above implies that
\begin{equation}\label{eq:chaingenprop1}\Pbx{\{Y_{n+1}=y,S_{n+1}>s\}|\cal{G}_n}=\Pb_{Y_n}(\{Y_1=y,S_{1}>s\})\quad\Pb_x\text{-almost surely},\end{equation}
for all $n$ in $\n$,  states $x,y$ in $\s$, and $s$ in $[0,\infty)$. Thus,
\begin{align}\label{eq:chaingenprop2}\Pbx{\{Y_{n+1}=y\}|\cal{G}_n}&=\Pb_{Y_n}(\{Y_1=y\})\quad\Pb_x\text{-almost surely},\enskip\forall x,y\in\s,\enskip n\in\n,\\
\Pbx{\{S_{n+1}>s\}|\cal{G}_n}&=\Pb_{Y_n}(\{S_{1}>s\})\quad\Pb_x\text{-almost surely},\enskip\forall x\in\s,\enskip s\in[0,\infty),\enskip n\in\n.\label{eq:chaingenprop3}
\end{align}

Now, applying the strong Markov property~\eqref{eq:strmkvct}, we find that
\begin{align*}\Pbx{\{Y_1=y,S_1>s\}}&=\Pbx{\{s<T_1,X_{s}=x,Y_1^{s}=y\}}=\Pbx{\{s<T_1,X_{s}=x\}}\Pbx{\{Y_1=y\}}\\
&=\Pbx{\{s<S_1\}}\Pbx{\{Y_1=y\}}\quad\forall x,y\in\s,\enskip s\in[0,\infty)\end{align*}
and that
\begin{align*}\Pbx{\{S_{1}>t+s\}}&=\Pbx{\{t<T_1,X_{t}=x,s<S^t_1\}}=\Pbx{\{t<T_1,X_{t}=x\}}\Pbx{\{s<S_1\}}\\
&=\Pbx{\{t<S_1\}}\Pbx{\{s<S_1\}}\quad\forall  x\in\s,\enskip t,s\in[0,\infty).
\end{align*}
That is, the waiting time distribution is \emph{memoryless} and it follows that  $\lambda(x)$ in~\eqref{eq:lambda2} is finite and that $S_1$, under $\Pb_x$, is exponentially distributed with parameter $\lambda(x)$ \citep[Theorem~2.3.1]{Norris1997}. Thus, we can re-write \eqref{eq:chaingenprop1}--\eqref{eq:chaingenprop3} as
\begin{align}
&\Pbx{\{Y_{n+1}=y,S_{n+1}>s\}|\cal{G}_n}=\Pbx{\{Y_{n+1}=y\}|\cal{G}_n}\Pbx{\{S_{n+1}>s\}|\cal{G}_n}\quad\Pb_x\text{-a.s.},\label{eq:prop1}\\
&\Pbx{\{Y_{n+1}=y\}|\cal{G}_n}=p(Y_n,y)\quad\Pb_x\text{-a.s.},\\
&\Pbx{\{S_{n+1}>s\}|\cal{G}_n}=e^{-\lambda(Y_n)s}\quad\Pb_x\text{-a.s.},\label{eq:prop3}
\end{align}
for all $n$ in $\n$, states $x,y$ in $\s$, and $s$ in $[0,\infty)$, where $(p(x,y))_{x\in\s}$ is as in~\eqref{eq:jumpmatrix2}.
%

%
%
Given \eqref{eq:prop1}--\eqref{eq:prop3}, the same argument as that in the proof of Theorem~\ref{thrm:pathlawunict} yields Theorem~\ref{thrm:otherchains}.
\begin{exercise}[The case with absorbing states]By expanding $(\Omega,\cal{F})$, introduce a sequence of i.i.d. unit-mean exponentially distributed random variables $(\tilde{S}_n)_{n\in\n}$ independent of (the old) $\cal{F}$. Replace $(S_n)_{n\in\zp}$ in \eqref{eq:waittimesgen} with
$$S_n:=\left\{\begin{array}{cl}T_n(\omega)-T_{n-1}(\omega)&\text{if }n\leq N(\omega)\\
\tilde{S}_n(\omega)&\text{if }n>N(\omega)\end{array}\right.\quad\forall \omega\in\Omega,\enskip n\in\zp$$
and redefine the jump times as
$$T_0(\omega):=0,\quad T_n(\omega):=\sum_{m=1}^n S_n(\omega)\quad\forall n\in\zp.$$
Now, define the jump chain $(Y_n)_{n\in\n}$ via \eqref{eq:jumpchaingen}. Emulating the steps taken above, show that Theorem~\ref{thrm:otherchains} holds with the following modifications: $p(x,x)$ may be zero and $Q=(q(x,y))_{x,y\in\s}$ in \eqref{eq:qmatrix2} is replaced by
$$q(x,y):=\left\{\begin{array}{ll}0&\text{if }p(x,x)>0\\-\lambda(x)&\text{if }x=y,\text{ }p(x,x)=0\\ \lambda(x)p(x,y)&\text{if }x\neq y,\text{ }p(x,x)=0\end{array}\right.\quad\forall x,y\in\s.$$

\end{exercise}

\newpage
\thispagestyle{premain}
\sectionmark{\MakeUppercase{Continuous-time chains II: the long-term behaviour}}
\section*{Continuous-time chains II: the long-term behaviour}\label{chap:long-termct}
\addcontentsline{toc}{section}{\protect\numberline{}Continuous-time chains II: the long-term behaviour}

With Sections~\ref{sec:FKG}--\ref{sec:otherchains} out of the way, we now focus on formalising for continuous-time chains the discussion regarding their long-term behaviour given in the introduction to Part~\ref{part:theory}. The treatment here follows its discrete-time counterpart (Sections~\ref{sec:entrance}--\ref{sec:flgeodt}) very closely and
%
the chapter overview given for the discrete-time case applies here almost unchanged. 
The only real differences are:

\begin{itemize}[leftmargin=*] 
\item The possibility of explosions throws up a couple of complications: the characterisation of the stationary equations in terms of linear equations is a bit more subtle in the continuous-time case than in the discrete-time one (compare Theorems~\ref{pstateqs}~and~\ref{Qstateq}) and, to rule out explosions in practice, we require a further Foster-Lyapunov criterion (Section~\ref{sec:flreg}).
\item We need to spend a little bit of time on deriving a series of results connecting continuous-time chains with the discrete-time \emph{skeleton chains} obtained by sampling the former at regular time intervals (Section~\ref{sec:skeleton}). These will prove key in porting results from the discrete-time case to the continuous-time one.
\item Continuous-time chains are always aperiodic (Section~\ref{sec:entct}). Thus, their time-varying law always converges (Theorem~\ref{thrm:timevarlimexist}).
\end{itemize}

\subsection[Entrance times; accesibility; closed communicating classes]{Entrance times, accesibility, and closed communicating classes} \label{sec:entct}
\pagestyle{main}
To figure out which states the chain will visit in the long run, we must first look at which states it visits at all. For this, we need \emph{entrance times}:
\begin{definition}[Entrance times]\label{def:entrancect} Given\index{entrance time} any set $A\subseteq\s$ and   positive integer $k$, the $k$th entrance time $\phi^k_A$ is the point in time that the chain enters the set $A$ for the $k$th time:\glsadd{varphiA}\glsadd{varphiAk}
$$\varphi^0_A(\omega):=0\enskip\forall\omega\in\Omega,\quad\varphi^k_A(\omega):=\inf\{t\in(\varphi^{k-1}_A(\omega),T_{\infty}(\omega)):X_t\in A\}\enskip\forall\omega\in\Omega,\enskip k>0,$$
For the first entrance time, we write $\varphi_A$ instead of $\varphi_A^1$. Additionally, we use the shorthand $\varphi_x:=\phi_{\{x\}}$ (resp. $\varphi_x^k:=\varphi_{\{x\}}^k$) to denote the first (resp. $k$th) entrance time of a state $x$ in $\s$.
\end{definition}
A bit inconsistently with our terminology for exit times (Section~\ref{moh1}), we do not view starting inside $A$ as `entering' $A$ (hence, $\varphi_A>0$ by definition): a small but important detail. If the chain starts in $x$, then $\varphi_x$ is the first time that the chain returns to $x$ and $\varphi_x^k$ the $k$th time it does. Hence, $\varphi_x$ is often referred to as the \emph{first return time} (or, simply, the \emph{return time}\index{return time}) and $\varphi_x^k$ as the  $k$th \emph{return time}. Because we are to  deduce whether the chain has entered a given set yet  by continuously monitoring it and because it can only enter a set by jumping into it, entrance times are jump-time-valued stopping time.
\begin{exercise}\label{ex:entdc} Using similar arguments to those given in the proof of Proposition~\ref{prop:hitdc} and induction, show that the $k$th entrance time $\varphi^k_A$ to $A$ is a jump-time-valued $(\cal{F}_t)_{t\geq0}$-stopping time, where $(\cal{F}_t)_{t\geq0}$ denotes the filtration generated by the chain (Definition~\ref{def:filt2}). Furthermore, prove that \eqref{eq:stopdc}--\eqref{eq:etatinfinf} hold if we replace $\eta$ with $\varphi^k_A$ and $\varsigma$ with $\phi^k_A$, where $\phi^k_A$ denotes the $k$th entrance time to $A$ of the jump chain $Y$ (obtained by replacing $X$ with $Y$ in Definition~\ref{def:entrance}).  
\end{exercise}
%
%
%
%
\subsubsection*{Accessibility}Entrance times allow us to formalise the idea the chain can reach, or \emph{access}\index{accessible}, one region of the state space from another:
\begin{definition}[Accessibility]\label{def:accesiblect}A state\glsadd{accessibility}\glsadd{accessibility2} $y$ is accessible from a state $x$ (or $x\to y$ for short) if and only if $\Pb_x\left(\{\varphi_y<\infty\}\right)>0$. Similarly, we say that a subset $B$ of the state space is accessible from another subset $A$ (or $A\to B$) if for every $x$ in $A$ there exists a $y$ in $B$ such that $x\to y$ (equivalently, if $\Pbl{\{\varphi_B<\infty\}}>0$ for all initial distributions $\gamma$ with support contained in $A$). 
\end{definition}
We can characterise this relationship in multiple ways:
%
\begin{theorem}[Equivalent definitions of accesibility]\label{thrm:accesibility} Suppose that $x\neq y$ and let $(p(x,y))_{x,y\in\s}$ and $(p_t(x,y))_{x,y\in\s}$ denote the jump matrix~\eqref{eq:jumpmatrix} and $t$-transition matrix~\eqref{eq:transprob}, respectively. The following statements are equivalent:
\begin{enumerate}[label=(\roman*),noitemsep]
\item $x\to y$ for $X$.
\item $x\to y$ for the jump chain $Y$.
\item There exists some $x_1,\dots,x_l$ in $\cal{S}$ such that $q(x,x_1)q(x_1,x_2)\dots q(x_l,y)>0$.
\item There exists some $x_1,\dots,x_l$ in $\cal{S}$ such that $p(x,x_1)p(x_1,x_2)\dots p(x_l,y)>0$.
\item $p_t(x,y)>0$ for some $t$ in $(0,\infty)$.
\item $p_t(x,y)>0$ for all $t$ in $(0,\infty)$.
\end{enumerate}
\end{theorem}
\begin{proof}Because Exercise~\ref{ex:entdc} shows that $\varphi_y$ is finite if and only if the first entrance time $\phi_y$ of the jump chain is finite, $(i)\Leftrightarrow (ii) \Leftrightarrow (iv)$ follows directly from Lemma~\ref{lem:accdt}. The definition of the jump matrix in \eqref{eq:jumpmatrix} implies that $(iii)\Leftrightarrow(iv)$. The chain's definition in \eqref{eq:cpathdef}--\eqref{eq:cpathdef3} and Tonelli's theorem imply that
$$p_t(x,y)=\Pbx{\{X_t=y,t<T_\infty\}}=\sum_{n=1}^\infty\Pbx{\{Y_n=y,T_n\leq t<T_{n+1}\}}\leq\sum_{n=1}^\infty\Pbx{\{Y_n=y\}}.$$
Because the definition of $\phi_y$ (Definition~\ref{def:entrance}) implies that $\{Y_n=y\}\subseteq\{\phi_y<\infty\}$ for all $n>0$, it follows from the above that $(v)\Rightarrow (ii)$. Trivially $(vi)\Rightarrow (v)$ and it suffices to show that $(iv)\Rightarrow(vi)$. Theorem~\ref{thrm:pathlawunict} implies that
\begin{align*}p_s(z,w)=\Pbz{\{X_t=w,t<T_\infty\}}&\geq \Pbz{\{Y_1=w,T_1\leq s<T_2\}}\geq \Pbz{\{Y_1=w,S_1\leq s, s<S_2\}}\\&=p(z,w)(1-e^{-\lambda(z)s})e^{-\lambda(w)s}>0\quad\forall  s\in(0,\infty)\end{align*}
for all $z$ and $w$ such that $p(z,w)>0$. The semigroup property (Theorem~\ref{thrm:findim}) then shows that
$$p_t(x,y)\geq p_{t/(l+1)}(x,x_1)p_{t/(l+1)}(x_1,x_2)\dots p_{t/(l+1)}(x_l,y)>0\forall t\in(0,\infty).$$
\end{proof}
Theorem~\ref{thrm:accesibility}$(iii)$ leaves clear that $\to$ is a transitive relation on both  $\cal{S}$ and its power set. The theorem has another more unexpected consequence:
\subsubsection*{Continuous-time chains are always aperiodic}If $x\neq y$, Theorem~\ref{thrm:accesibility}$(vi)$ ensures that $p_{t}(x,y)$ is either zero for all $t\geq0$ or positive for all $t>0$. This precludes any type of periodicity (compare the  discrete-time case in  Section~\ref{sec:periodicity}): a chain is either unable to transition from $x$ to $y$ or it is able to transition in any amount of time.
\subsubsection*{Communicating classes and closed sets}
%
%
%
Two sets are said to \emph{communicate} if each is accessible from the other. A set to be a \emph{communicating class}\index{communicating class} if all its subsets communicate:
\begin{definition}[Communicating class]A subset $\cal{C}$ of $\cal{S}$ is a communicating class if and only if $x\to y$ for all  $x,y$ in $\cal{C}$ (equivalently $A\to B$ for all $A,B$ contained in $\cal{C}$).
\end{definition}
A subset of the state space is said to be \emph{closed}\index{closed set} if no state outside the set is accessible from a state inside the set. 
\begin{definition}[Closed set]A subset $\cal{C}$ of $\s$ is closed if and only if for each $x$ in $\cal{C}$ and $y$ in $\s$, $x\to y$ implies that $y$ belongs to $\cal{C}$.\end{definition}
Closed sets are those from which the chain may never escape:
\begin{proposition}\label{prop:closedct}If the chains starts in a closed set $\cal{C}$ (i.e., the initial distribution $\gamma$ satisfies $\gamma(\cal{C})=1$), then the chain remains in $\cal{C}$:
$$\Pbl{\{X_t\in\cal{C}\enskip\forall t\in[0,T_\infty)\}}=1.$$
\end{proposition}
\begin{proof}The chain's definition in \eqref{eq:cpathdef}--\eqref{eq:cpathdef3} implies that
$$\{X_t\in\cal{C}\enskip\forall t\in[0,T_\infty)\}=\{Y_n\in\cal{C}\enskip\forall n\in\n\},$$
where $Y=(Y_n)_{n\in\n}$ denotes the jump chain. Because Theorem~\ref{thrm:accesibility}$(ii)$ shows that $\cal{C}$ is closed for $Y$, the proposition then follows directly from the above and its discrete-time counterpart (Proposition~\ref{prop:closed}).
\end{proof}
Given the above, the following should come as no surprise:
\begin{proposition}[The chain visits at most one of two disjoint closed sets]\label{prop:closedisct} If $\cal{C}_1$ and $\cal{C}_2$ are two disjoint closed sets, then
$$\Pbl{\{\varphi_{\cal{C}_1}<\infty,\varphi_{\cal{C}_2}<\infty\}}=0.$$
\end{proposition}
\begin{proof}Because Exercise~\ref{ex:entdc} shows that the first entrance time entrance of the chain is finite if and only if the corresponding entrance time of the jump chain is finite, this proposition follows directly from its discrete-time counterpart (Proposition~\ref{prop:closedis}) and Theorem~\ref{thrm:accesibility}.
\end{proof}
Communicating classes that are closed play a very important role in the theory of continuous-time chains: they are the irreducible sets mentioned in the introduction to Part~\ref{part:theory}. For this reason, the chain and its rate matrix are said to be \emph{irreducible}\index{irreducible} if the entire state space is a single closed communicating class. 

\subsubsection*{Non-explosivity is a class property}
As we will see throughout this chapter, states in the same communicating class share many properties (e.g, transience, recurrence, etc.) and we refer to these as \emph{class properties}\index{class property}. One example is non-explosivity: the inability of the chain to explode when it starts in a \emph{non-explosive} state.
\begin{theorem}\label{thrm:expclprop} The explosion probabilities $(\Pbx{\{T_\infty<\infty\}})_{x\in\s}$ satisfy the equations
\begin{equation}\label{eq:expclprop}\sum_{z\in\s}q(x,z)\Pb_z(\{T_\infty<\infty\})=0,\qquad\forall x\in\s.\end{equation}
In particular, if $\Pbx{\{T_\infty<\infty\}}=0$ and $x\to y$, then $\Pby{\{T_\infty<\infty\}}=0$.
\end{theorem}

\begin{proof} Applying the  Strong Markov Property (Theorem~\ref{thstrmk}) to $\eta:=T_1$ (with $A:=\Omega$ and $F(x):=1_{[0,t]}(c^T_n(x))$ where $c^T_n$ is as in~\eqref{eq:coord}), we find that
$$\Pbx{\{X_{T_1}=z,T_{n+1}\leq t\}}=\Pbx{\{X_{T_1}=z\}}\Pb_z(\{T_{n}\leq t\})=p(x,z)\Pb_z(\{T_{n}\leq t\}),$$
for all $z$ in $\s$, $n$ in $\n$, and $t$ in $[0,\infty)$. Given the definition of the jump matrix in~\eqref{eq:jumpmatrix}, taking the limits $n\to\infty$ and $t\to\infty$ (in that order), applying monotone convergence, summing over $z$ in $\s$, and re-arranging yields \eqref{eq:expclprop}. Suppose that $x\to y$ and let $x_1,\dots,x_l$ be as in Theorem~\ref{thrm:accesibility}$(ii)$. Repeatedly applying \eqref{eq:expclprop} yields
\begin{align*}\Pbx{\{T_\infty<\infty\}}&=\frac{1}{q(x)}\sum_{z\in\s}q(x,z)\Pb_z(\{T_\infty<\infty\})\geq \frac{q(x,x_1)}{q(x)}\Pb_{x_1}(\{T_\infty<\infty\})\\
&=\dots\geq \frac{q(x,x_1)\dots q(x_l,y)}{q(x)\dots q(x_l)}\Pb_{y}(\{T_\infty<\infty\}),\end{align*}
and it follows that $\Pbx{\{T_\infty<\infty\}}=0$ only if $\Pby{\{T_\infty<\infty\}}=0$.
\end{proof}

\ifdraft

\subsubsection*{Notes and references} Call \eqref{thrm:accesibility} Levy's theorem $p_t(x,y)$ non-zero for one $t>0$ if and only if non-zero for all $t>0$?
\fi
\subsection{Recurrence and transience}

The starting point in understanding the long-term behaviour of chains are the notions of \emph{recurrence} and \emph{transience}:
\begin{definition}[Recurrent and transient states]\label{def:rectransct} A state $x$ is said to be recurrent\index{recurrent state} if the chain has probability one of returning to it (i.e., $\Pbx{\{\phi_x<\infty\}}=1$). Otherwise, the state is said to be transient\index{transient state}. 
\end{definition}
Recurrent states are those the chain will keep revisiting (as long as it visit them at least once) while transient ones are those the chain will stop visiting at some point:
\begin{theorem}\label{thrm:rectransct} For any initial distribution $\gamma$ and state $x$, 
\begin{equation}\label{eq:returnprob}\Pb_\gamma(\{\varphi^k_x<\infty\})=\Pb_\gamma(\{\varphi_x<\infty\})\Pb_x(\{\varphi_x<\infty\})^k\quad\forall k>0.\end{equation}
Moreover, having started at $x$, the chain will return to $x$ infinitely many times if and only if $x$ is recurrent:
$$\Pb_x(\{\varphi^k_x<\infty\enskip\forall k\in\zp\})=\left\{\begin{array}{ll}1&\text{if $x$ is recurrent}\\0&\text{if $x$ is transient}\end{array}\right.,\quad\forall x\in\s.$$
\end{theorem}
\begin{proof}Because, for any set $A$, the entrance time to $A$ of  $X$ is finite if and only if the entrance time to $A$ of its jump chain $Y$ is finite (Exercise~\ref{ex:entdc}), the theorem follows directly its discrete-time counterpart (Theorem~\ref{thrm:rectrans}, note also~\eqref{eq:mdw78ahdnw78andwa}).
\end{proof}
Recurrence and transience are class properties.
\begin{theorem}\label{thrm:recclassproct} If $x$ is recurrent and $x\to y$, then $y$ is also recurrent and 
$$\Pbx{\{\varphi_y<\infty\}}=\Pby{\{\varphi_x<\infty\}}=1.$$
For this reason, a) if $y$ is transient and $x\to y$, then $x$ is also transient and b) a state in a communicating class is recurrent (resp. transient) if and only if all states in the class are recurrent (resp. transient). 
\end{theorem}

\begin{proof}Because, for any set $A$, the entrance time to $A$ of  $X$ is finite if and only if the entrance time to $A$ of its jump chain $Y$ is finite (Exercise~\ref{ex:entdc}), the theorem follows directly from its discrete-time counterpart (Theorem~\ref{thrm:recclasspro}).
\end{proof}
Due to the above theorem, we say that a communicating class is \emph{recurrent} (resp. \emph{transient}) if any one (and therefore all) of its states is recurrent (resp. transient).\index{transient class}\index{recurrent class}  For chains, we use the following slightly more involved classification:

\begin{definition}[Recurrent and transient chains]\label{def:recurrentchainsct}If the set $\cal{R}$ of all recurrent states\glsadd{rec} is empty, then the chain is said to be transient\index{transient chain}. Otherwise, the chain is said to be Tweedie recurrent\index{Tweedie recurrent chain} if, regardless of the initial distribution, the chain has probability one of entering $\cal{R}$:
$$\Pbl{\{\varphi_\cal{R}<\infty\}}=1\enskip\text{for all initial distributions}\enskip\gamma,$$
where $\varphi_\cal{R}$ denotes the first entrance time to $\cal{R}$ (Definition~\ref{def:entrancect}). If, additionally, there exists only one recurrent communicating class, then the chain is said to be Harris recurrent\index{Harris recurrent chain}. If the class is the entire state space (i.e., if $\cal{R}=\s$), then the chain is simply called recurrent.\index{recurrent chain}
\end{definition}

We conclude the section with the following useful proposition:
\begin{proposition}\label{prop:phixphic} If the chain ever enters a recurrent closed communicating class $\cal{C}$, it will eventually visit every state in $\cal{C}$:
$$1_{\{\varphi_{\cal{C}}<\infty\}}=1_{\{\varphi_x<\infty\}}\enskip\forall x\in\cal{C},\quad\Pb_\gamma\text{-almost surely}.$$
\end{proposition}
\begin{proof}Because, for any set $A$, the entrance time to $A$ of $X$ is finite if and only if the entrance time to $A$ of its jump chain $Y$ is finite (Exercise~\ref{ex:entdc}), the proposition follows directly from its discrete-time counterpart (Corollary~\ref{cor:phixphic}).
\end{proof}

\subsubsection*{Notes and references} The reasons behind our choice of nomenclature in this section are the same as those given in its discrete-time counterpart (Section~\ref{sec:rec}, notes and references).

\subsection{The regenerative property}
For the same reasons as those given in Section~\ref{sec:regen} for the discrete-time case, the Markov property implies that the chain \emph{regenerates} every time it visits a recurrent state $x$ in the sense that upon each visit it forgets its entire past and starts afresh. More formally, the segments of paths between consecutive visits to $x$ form an i.i.d. sequence:
\begin{theorem}[The regenerative property]\label{thrm:rec-iidct} For\index{regenerative property} any non-negative function  $f:\s\to\r_E$, state $x$, and positive integer $k$,
\begin{equation}\label{eq:Ikdefct}I_k:=1_{\{\varphi_x^{k-1}<\infty\}}\int_{\varphi_x^{k-1}}^{\varphi_x^{k}\wedge T_\infty}f(X_t)dt\end{equation}
defines an $\cal{F}_{\varphi_x^{k}}/\cal{B}(\r_E)$-measurable function, where $\cal{F}_{\varphi_x^{k}}$ denotes the pre-$\varphi^k_x$ sigma algebra (Definition~\ref{def:stopct}). In the above, we are using our convention for partially defined random variables~\eqref{eq:partdef}. If the state $x$ is recurrent, the sequence $(I_k)_{k\in\zp}$ is i.i.d. under $\Pb_x$.\end{theorem}
Before proving the theorem, we  quickly apply it to that the chain cannot explode if it starts in a recurrent state:
\begin{proposition}[The chain cannot explode if it starts in a recurrent state]\label{prop:recnoexpl} If $x$ is a recurrent state, $\Pbx{\{T_\infty=\infty\}}=1$.
\end{proposition}
\begin{proof}Because Theorem~\ref{thrm:rectransct} shows that all entrance times to $x$ are finite with $\Pb_x$-probability one, the definition of the entrance times (Definition~\ref{def:entrancect}) implies that
\begin{equation}\label{eq:ndw78a9dfeafahh7w8and7wa}\varphi_x^k<T_\infty\quad\forall k>0,\enskip \Pb_x\text{-almost surely.}\end{equation}
The law of large numbers (Theorem~\ref{thrm:klln}) and the regenerative property (Theorem~\ref{thrm:rec-iidct} with $f:=1$) imply that
\begin{align*}&\lim_{N\to\infty}\frac{\varphi^N_x}{N}=\lim_{N\to\infty}\frac{1}{N}\sum_{k=1}^N(\varphi^k_x-\varphi^{k-1}_x)=\Ebx{\varphi_x}>0\quad\Pb_x\text{-almost surely},\\
&\Rightarrow \lim_{k\to\infty}\varphi^k_x=\infty\quad\Pb_x\text{-almost surely}.\end{align*}
For this reason, taking the limit $k\to\infty$ in \eqref{eq:ndw78a9dfeafahh7w8and7wa} completes the proof.
\end{proof}

\subsubsection*{A proof of Theorem~\ref{thrm:rec-iidct}}This proof  is entirely analogous to that of the theorem's discrete-time counterpart (Theorem~\ref{thrm:rec-iid}). We also do it in two steps, beginning with the measurability of $I_k$ and moving on to showing that the sequence is i.i.d.

\begin{proof}[Step 1: $I_k^f$ is $\cal{F}_{\varphi_x^{k}}/\cal{B}(\r_E)$-measurable]Note that
$$I_k=\lim_{N\to\infty}\sum_{n=1}^N\sum_{m=0}^{n-1}\sum_{l=m}^{n-1}1_{\{\varphi^k_x=T_n\}}1_{\{\varphi^{k-1}_x=T_m\}}S_{l+1}f(Y_l).$$
Because limits, finite products, and finite sums of measurable functions are measurable functions, it suffices to show that each term in the sum is $\cal{F}_{\varphi_x^{k}}/\cal{B}(\r_E)$-measurable. Because $\varphi_{k-1}\leq \varphi_{k}$ by definition, Lemma~\ref{lem:etatheta}$(i,iii)$ shows that $1_{\{\varphi^{k-1}_x=T_m\}}$ is $\cal{F}_{\varphi_x^{k}}/\cal{B}(\r_E)$-measurable. Similarly, Lemma~\ref{lem:etatheta}$(i,ii)$ and Proposition~\ref{prop:ftngn} yields that $1_{\{\varphi^k_x=T_n\}}S_{l+1}f(Y_l)$ is  $\cal{F}_{\varphi_x^{k}}/\cal{B}(\r_E)$-measurable for all $l<n$ (as $T_{l+1}\leq\varphi^k_x$ for any such $l$) and the result follows.\end{proof}
%

\begin{proof}[Step 2: $(I_k^f)_{k\in\zp}$ is i.i.d.] To simplify the notation, we write $I_k$ for $I^f_k$ throughout this proof. Suppose that $x$ is recurrent, choose any $a_1,\dots,a_k\geq0$, and let
$$A_k:=\{I_1\geq a_1,\dots,I_{k-1}\geq a_{k-1},I_k\geq a_k\}.$$
The return times to $x$ are all finite with $\Pb_x$-probability one (Theorem~\ref{thrm:rectransct}). Furthermore, Exercise~\ref{ex:entdc} and the chain's definition~\eqref{eq:cpathdef}--\eqref{eq:cpathdef3} imply that $X_{\varphi^{k-1}_x}=x$ on $\{\varphi^{k-1}_x<\infty\}$ and that $I_k=G(X^{\varphi^{k-1}_x})$, where $X_{\varphi^{k-1}_x}$ denotes the $\varphi^{k-1}_x$-shifted chain (Section~\ref{sec:markovprop}) and $G$ is the $\cal{E}/\cal{B}(\r_E)$-measurable function in Lemma~\ref{lem:pathspmeasct}$(iii)$ (with $A$ and $k$ therein set to $\{x\}$ and $1$). For these reasons,
\begin{align*}
\Pbx{A_k}&=\Pb_x(\{I_1\geq a_1,\dots,I_{k-1}\geq a_{k-1},\varphi^{k-1}_x<\infty,X_{\varphi^{k-1}_x}=x,I_k\geq a_k\})\\
&=\Pb_x(\{I_1\geq a_1,\dots,I_{k-1}\geq a_{k-1},\varphi^{k-1}_x<\infty,X_{\varphi^{k-1}_x}=x,G(X^{\varphi^{k-1}_x})\geq a_k\}).
\end{align*}
%
%
Because $\varphi^1_x\leq\varphi^2_x\leq \dots\leq\varphi^{k-1}_x$ by definition, Lemma~\ref{lem:etatheta} and Step 1 imply that $A_{k-1}$ belongs to $\cal{F}_{\varphi^{k-1}_x}$. Thus, the strong Markov property (Theorem~\ref{thstrmk}) implies that
$$\Pbx{A_k}=\Pb_x(A_{k-1}\cap\{\phi^{k-1}_x<\infty,X_{\phi^{k-1}_x}=x\})\Pbx{\{G(X)\geq a_k\}}=\Pbx{A_{k-1}}\Pbx{\{I_1\geq a_k\}}.$$
%
%
Iterating the above backwards, we obtain 
\begin{equation}\label{eq:nfhudasfnsha2}\Pbx{A_k}=\Pbx{\{I_1\geq a_1\}}\Pbx{\{I_1\geq a_2\}}\dots\Pbx{\{I_1\geq a_k\}}.\end{equation}
Setting $a_1=\dots=a_{k-1}=0$ we find that 
$$\Pbx{\{I_k\geq a_k\}}=\Pbx{\{I_1\geq a_k\}}.$$
Because $k$ and $a_1,\dots,a_k$ were arbitrary, and $\{[a,\infty]:a\in\r_E\}$ is a $\pi$-system that generates $\cal{B}(\r_E)$, the above and Lemma~\ref{lem:dynkinpl} show that the sequence $(I_k)_{k\in\n}$ is identically distributed. For this reason, we can rewrite \eqref{eq:nfhudasfnsha2} as 
$$\Pbx{A_k}=\Pbx{\{I_1\geq a_1\}}\Pbx{\{I_2\geq a_2\}}\dots\Pbx{\{I_k\geq a_k\}}.$$
The desired independence then also follows from Lemma~\ref{lem:dynkinpl} as $\{[a_1,\infty],\dots, [a_k,\infty]:a_1,\dots,a_k\in\r_E\}$ is a $\pi$-system that generates $\cal{B}(\r_E^k)$.
\end{proof}

\subsection{The empirical distribution and positive recurrent states}\label{sec:empdistct}Consider the  \emph{empirical distribution}\index{empirical distribution}\glsadd{ept} $\epsilon_T:=(\epsilon_T(x))_{x\in\s}$ tracking the fraction of time the chain spends in each state, where
\begin{equation}\label{eq:timeavedefct}\epsilon_T(x):=\frac{1}{T}\int_0^{T\wedge T_\infty}1_x(X_t)dt\quad \forall x\in\s,\end{equation}
denotes is the fraction of the interval $[0,T]$ that the chain spends in state $x$. From Theorem~\ref{thrm:rectransct}, we already know that the chain eventually stops visiting any given transient state $x$, and it follows that 
$$\lim_{T\to\infty}\epsilon_T(x)=0\quad\Pb_\gamma\text{-almost surely.}$$
In the case of recurrent states $x$, the visits do not cease. However, the regenerative property implies that the sequences $(\eta_k)_{k\in\zp}$ of time-$\eta_k$-spend-in-$x$-between-two-consecutive-visits and $(\varphi^{k+1}-\varphi^{k})_{k\in\zp}$ of time-$(\varphi^{k+1}-\varphi^{k})$-elapsed-between-consecutive-visits are i.i.d. if the chain starts from $x$. In this case, the law of large numbers (Theorem~\ref{thrm:klln}) then implies that, 
\begin{align*}\epsilon_T(x)&\approx\frac{1}{\varphi^K_x}\int_{n=0}^{\varphi^K_x}1_x(X_t)dt=\left(\sum_{k=0}^{K-1}(\varphi^{k+1}-\varphi^{k})\right)^{-1}\left(\sum_{k=0}^{K-1}\eta_k\right)\\
&=\left(\frac{1}{K}\sum_{k=0}^{K-1}(\varphi^{k+1}-\varphi^{k})\right)^{-1}\left(\frac{1}{K}\sum_{k=0}^{K-1}\eta_k\right)\approx\frac{\Ebx{\eta_1}}{\Ebx{\varphi_x}},\end{align*}
for all $T$ large enough that the number $K$ of the return time $\varphi^K_x$ closest to $T$ is large and $|\varphi^K_x-T|/T$ is small, where $\lambda(x)$ denotes the jump rate in~\eqref{eq:lambda}. Because $\eta_1$ equals the first waiting time $S_1$ and $S_1$'s mean is $1/\lambda(x)$ under $\Pb_x$, it follows that $\epsilon_T(x)\approx (\lambda(x)\Ebx{\varphi_x})^{-1}$ for all large enough $T$. If the chain does not start at $x$, then noting that $\epsilon_T(x)$ is non-zero only if the chain ever visits $x$. The  Markov property then implies following:
\begin{theorem}\label{thrm:empdistlimsct}For all states $x$ in $\s$,
$$\epsilon_\infty(x):=\lim_{T\to\infty}\epsilon_T(x)=1_{\{\varphi_x<\infty\}}\frac{1}{\Ebx{\varphi_x}}\quad\Pb_\gamma\text{-almost surely.}$$
\end{theorem}
\begin{exercise}Prove Theorem~\ref{thrm:empdistlimsct} by following the steps taken in the proof of its discrete-time counterpart (Theorem~\ref{thrm:empdistlims}). To do so, note that $F_-,F_+$ in Lemma~\ref{lem:pathspmeasct}$(iv)$ (with $f:=1_x$) are such that
$$F_-(X)=\liminf_{T\to\infty}\epsilon_T(x),\qquad F_+(X)=\limsup_{T\to\infty}\epsilon_T(x),$$
and replace $R_N$ with
$$R_T:=\left\{\begin{array}{cc}0&\text{if }\varphi_x^0\leq T <\varphi_x^1\\1&\text{if }\varphi_x^1\leq T <\varphi_x^2\\\vdots&\vdots\end{array}\right.\quad\forall T\in[0,\infty)$$
and Theorems~\ref{thrm:strmkvpath},~\ref{thrm:rectrans},~and~\ref{thrm:rec-iid} with their continuous-time counterparts (Theorems \ref{thstrmk}, \ref{thrm:rectransct}, and \ref{thrm:rec-iidct}, repesctively).
\end{exercise}
The theorem shows that the fraction $\epsilon_\infty(x)$ of all time that the chain spends in a state $x$ is negligible  unless returns to $x$ occur quickly enough that the mean return time is finite. Or, in other words, unless $x$ is \emph{positive recurrent}:
\begin{definition}[Positive and null recurrent states]\label{def:posrecct}
A state $x$ is said to be positive recurrent\index{positive recurrent state} if its mean return time is finite ($\Ebx{\varphi_x}<\infty$). If $x$ is recurrent but its mean return time is infinite ($\Ebx{\varphi_x}=\infty$), we say that it is null recurrent\index{null recurrent state}.
\end{definition}

\subsection{Skeleton chains and the Croft-Kingman lemma}\label{sec:skeleton}As the previous sections illustrate, we are able to prove results describing the long-term behaviour of the empirical distribution in the continuous-time case by emulating the proofs the results' discrete-time counterparts. To prove the results that characterise the long-term behaviour of the time-varying law of continuous-time chains, we can directly use their discrete-time counterparts. Key to unlocking this approach are the so-called\glsadd{skeleton}\index{skeleton chain} \emph{$\delta$-skeleton chains $X^\delta=(X_n^\delta)_{n\in\n}$}: discrete-time processes  obtained by sampling our continuous-time chain $X$ at regular time intervals of length $\delta>0$, i.e.,
\begin{equation}\label{eq:skeletondef}X_n^\delta(\omega):=\left\{\begin{array}{ll}X_{\delta n}(\omega)&\text{if }\delta n<T_\infty(\omega)\\ \Delta&\text{if }\delta n\geq T_{\infty}(\omega)\end{array}\right.\quad\forall\omega\in\Omega,\enskip n\in\n,\end{equation}
where $\Delta\not\in\s$ denotes an auxiliary dummy (or \emph{dead} or \emph{infinity}) state in which we leave $X^\delta$ for all sampling times past the explosion time of $X$. It is not difficult to use the Markov property of $X$ to show that 
\begin{equation}\label{eq:skeletonmarkov}\Pb_\gamma(\{X_{n+1}^\delta=x|X_{n}^\delta\})=p^\delta(X_{n}^\delta,x)\quad\forall x\in\s_E,\enskip n\in\n,
\end{equation}
where $\s_E$ denotes the \emph{extended state space} $\s\cup\{\Delta\}$ and $P^\delta=(p^\delta(x,y))_{x,y\in\s_E}$ the one-step matrix given by
\begin{equation}\label{eq:skeletonone-step}p^\delta(x,y):=\left\{\begin{array}{ll}\Pbx{\{X_\delta=y,\delta<T_\infty\}}&\text{if }x,y\in\s\\ \Pbx{\{T_\infty\leq \delta \}}&\text{if }x\in\s,\enskip y=\Delta\\ 1_\Delta(y)&\text{if }x=\Delta\end{array}\right.\quad\forall x,y\in\s_E.
\end{equation}
Theorem~\ref{samedef} implies that $ X^\delta$ is a discrete-time chain with state space $\s_E$, one-step matrix $P^\delta$, and initial distribution $\gamma_E:=(1_{\s}(x)\gamma(x))_{x\in\s_E}$ (here, we are using our convention in~\eqref{eq:partdef} for partially-defined functions). Moreover, because $\Delta$ is an absorbing state for $ X^\delta$, it is straightforward to check that 
\begin{equation}\label{eq:skeletonoriginal}p^\delta_n(x,y)=p_{\delta n}(x,y)\quad\forall x,y\in\s,\enskip n\in\n,\end{equation}
where $P^\delta_n=(p^\delta_n(x,y))_{x,y\in\s_E}$ denotes the $n$-step matrix of the skeleton chain (defined by replacing $P$ and $\s$ in~\eqref{eq:nstepdef} with $P^\delta$ and $\s_E$) and $(p_t(x,y))_{x,y\in\s}$ denotes $t$-transition matrix in~\eqref{eq:transprob} of our continuous-time chain $X$.
\begin{exercise}Prove \eqref{eq:skeletonmarkov} and \eqref{eq:skeletonoriginal}, and show that the $P^\delta$ defined in \eqref{eq:skeletonone-step} is indeed a one-step matrix (i.e., that $\sum_{y\in\s_E}p^\delta(x,y)=1$ for all $x$ in $\s_E$).\end{exercise}

It follows easily from \eqref{eq:skeletonoriginal} that the skeleton chain and the original chain share the same communicating classes and closed sets and that the former is aperiodic:
\begin{exercise}\label{exe:skeletonperiodic}Pick any $\delta\in(0,\infty)$. Use \eqref{eq:skeletonoriginal} and Theorem~\ref{thrm:accesibility}$(vi)$ to prove that  $X^\delta$ is aperiodic and that $y$ is accessible from $x$ for $X^\delta$ if and only if it is for $X$, for any states $x,y\in\s$.\end{exercise}

\subsubsection*{The Croft-Kingman lemma and the existence of limits for the time-varying law}Equation~\eqref{eq:skeletonoriginal} spells out the relationship between the $n$-step matrix of $X^\delta$ and the transition probabilities of $X$ evaluated at the sampling times $t=0,\delta,2\delta,\dots$ To relate these two for other times, the following lemma proves very useful.\index{Croft-Kingman lemma}
\begin{lemma}[The Croft-Kingman lemma,~\citep{Kingman1963}]\label{lem:croftkingman}Let $f$ be a continuous real-valued function on $[0,\infty)$. Suppose that for each $\delta$ in $(0,\infty)$, the limit
$$L_\delta:=\lim_{n\to\infty}f(\delta n)$$
exists and is finite. In this case, the limit $L_\delta$ is independent of $\delta$. Moreover, $\lim_{t\to\infty}f(t)$ exists and equals $L_\delta$.\end{lemma}

Before we prove the lemma we use it to prove that the time-varying $p_t$ converges pointwise as $t$ approaches infinity:

\begin{theorem}[The time-varying law has pointwise limits]\label{thrm:timevarlimexist}For all initial distributions $\gamma$ and states $x$, the limit $\lim_{t\to\infty}p_t(x)$ exists and is finite, where $p_t$ denotes the time-varying law in~\eqref{eq:ctlawdef} of $X$.
\end{theorem}

\begin{proof}As we showed in the proof of the differential equation satisfied by the time-varying law (Theorem~\ref{thrm:forward}), $t\mapsto p_t(x)$ is a continuous function on $[0,\infty)$ for every $x$ in $\s$.  Fix any $\delta$ in $(0,\infty)$. Because all skeleton chains are aperiodic (Exercise~\ref{exe:skeletonperiodic}), the pointwise convergence of the time-varying law of aperiodic discrete-time chains (Theorem~\ref{thrm:pointlims}), bounded convergence, and \eqref{eq:skeletonoriginal} imply that
\begin{equation}\label{eq:skeletonlims}\lim_{n\to\infty}p_{n\delta}(x)=\lim_{n\to\infty}\sum_{x'\in\s}\gamma(x')p_{n\delta}(x',x)=\lim_{n\to\infty}\sum_{x'\in\s_E}\gamma_E(x')p_{n}^\delta(x',x)=\lim_{n\to\infty}p_{n}^\delta(x)=\pi_\gamma(x)
\end{equation}
for all $x$ in $\s$, where $(p_{n}^\delta)_{n\in\n}$ denotes the time-varying  law \eqref{eq:timevardt} of the skeleton chain and $\pi_\gamma$ is its limit in \eqref{eq:reclims2} (with $X^\delta$ replacing $X$ therein). Because $\delta$ was arbitrary, the Croft-Kingman lemma tells us that the limit $\pi_\gamma$ is the same for all $\delta$ and that 
\begin{equation}\label{eq:nf8ahbfw678bfhafwa}\lim_{t\to\infty}p_t(x)=\pi_\gamma(x)\quad\forall x\in\s.\end{equation}
\end{proof}
Without realising it, we have almost characterised the limiting behaviour of the time-varying law. As we will show in the next section, continuous-time chains and their skeletons share the same stationary distributions. For this reason, we are able to rewrite the limits in~\eqref{eq:nf8ahbfw678bfhafwa} in terms of the continuous-time chain's stationary distributions and its entrance times to positive recurrent classes just as we did in Theorem~\ref{thrm:pointlims} for the aperiodic discrete-time case.

\subsubsection*{A proof of the Croft-Kingman lemma}Here, we follow the steps taken in~\citep{Kingman1963}.

\begin{theorem}\label{thrm:kingmansthrm}Let $A$ be any  unbounded open subset of $(0,\infty)$. There exists a $\delta>0$ such that $\delta n$ belongs to $ A$ for infinitely many integers $n$.
\end{theorem}

\begin{proof}For any $\alpha\in(0,\infty)$ and set $B\subseteq[0,\infty)$, let $\alpha B$ denote the set $\{\alpha x:x\in B\}$. Consider the open set
$$A_{m}:=\bigcup_{n=m}^\infty \frac{1}{n} A\quad\forall  m\in\zp.$$
Suppose that, for some $m$, there exists a non-empty open interval $I\subseteq(0,\infty)$ disjoint from $A_{m}$. In this case,  $nI\cap A=\emptyset$ for all $n\geq m$. Thus, $A$ is disjoint from $B:=\cup_{n=m}^\infty nI$. If you think about it a little bit, you'll see that $B$ contains all large enough numbers (hint: compare the distance between the mid-points of $(n-1)I$ and $nI$ with the length of $n I$), contradicting our assumption that $A$ is unbounded. Thus, $A_{m}$ intersects with all open intervals contained in $(0,\infty)$ or, in other words, $A_{m}$ is dense in $(0,\infty)$.  For this reason,
$$\bigcap_{m=1}^\infty A_{m}$$
is the intersection of countably many open dense subsets of $(0,\infty)$ and Baire's theorem tells us this intersection is dense in $(0,\infty)$. In particular, it is non-empty and picking any $\delta$ belonging to the intersection we have that $\delta n$ belongs to $A$ for infinitely many $n$ in $\zp$.
\end{proof}
The key to the proof of Lemma~\ref{lem:croftkingman} is the following corollary of Theorem~\ref{thrm:kingmansthrm}:\index{Croft's theorem}
\begin{corollary}[Croft's theorem,~\citep{Croft1957}]\label{cor:croft} Let $f$ be a continuous real-valued function on $[0,\infty)$. If 
$$\lim_{n\to\infty}f(\delta n)=0$$
for all $\delta$ in $(0,\infty)$, then
$$\lim_{t\to\infty}f(t)=0.$$
\end{corollary}
\begin{proof}Suppose that $f(t)$ does not tend to zero as $t$ approaches infinity and pick any $c$ in $(0,\infty)$ such that $\limsup_{t\to\infty}\mmag{f(t)}>c$. For any $n$ in $\n$, there exists a $t_n>n$ such that $\mmag{f(t_n)}>c$ and so by continuity there exists an open interval $I_n$ containing $t_n$ such that $\mmag{f(t)}> c$ for all $t$ in $I_n$. Thus, $A:=\cup_{n=0}^\infty I_n$ is an unbounded open set such that $\mmag{f(t)}>c$ for all $t$ in $A$. Theorem~\ref{thrm:kingmansthrm} then implies that there exists a $\delta$ in $(0,\infty)$ such that $\mmag{f(\delta n)}>c$ for all $n$ in $\n$ contradicting our premise. Hence, it must be the case that $f(t)\to0$ as $t\to\infty$.
\end{proof}
The proof of the Croft-Kingman lemma is now downhill:
\begin{proof}[Proof of Lemma~\ref{lem:croftkingman}]
For any $\delta$ in $(0,\infty)$ and positive integer $m$, the sequence $(nm\delta)_{n\in\zp}$ is a subsequence of $(n\delta)_{n\in\zp}$ and so 
$$L_{m\delta}=\lim_{n\to\infty}f(nm\delta)=\lim_{n\to\infty}f(n\delta)=L_{\delta}.$$
Thus, for any rational number $r=m/l$, we have that $L_{r\delta}=L_{lr\delta}=L_{m\delta}=L_{\delta}$. Because $\delta\mapsto L_\delta$ is the pointwise limit of the sequence $(\delta\mapsto f(\delta n))_{n\in\zp}$ of continuous functions, there exists at least one $\delta_*$ such that $\delta\mapsto L_\delta$ is continuous at $\delta_*$. Fix any $\delta$ in $(0,\infty)$ and let $(r_n)_{n\in\zp}$ be a sequence of rational numbers converging to $\delta_*/\delta$. Then, 
$$L:=L_{\delta_*}=\lim_{n\to\infty}L_{r_n\delta}=\lim_{n\to\infty}L_{\delta}=L_{\delta}$$
proving that $L$ is independent of $\delta$ and the result follows by applying Corollary~\ref{cor:croft} to the function $t\mapsto f(t)-L$.
\end{proof}
\subsection[Stationary and ergodic distributions; Doeblin decomposition]{Stationary distributions, ergodic distributions, and a Doeblin-like decomposition}\label{sec:statct}
A probability distribution $\pi$\glsadd{pi} on $\s$ is said to be a stationary distribution\index{stationary distribution} of the chain $X$ if sampling its starting state from $\pi$ makes $X$ a stationary process. By a stationary process, I mean that the chain's path law is invariant to time shifts:
\begin{equation}\label{eq:statprocct}\Pbp{\left\{X^t\in A\right\}}=\Lbp{A}\quad\forall A\in\cal{E}, \end{equation}
where $\cal{E}$ denotes the sigma-algebra generated by the cylinder sets of the path space $\cal{P}$, $\mathbb{L}_\pi$ the path law when the initial distribution is $\pi$,  and $X^t$ the $t$-shifted chain (see Sections~\ref{sec:pathspacect} and \ref{sec:markovprop}). Setting 
$$A:=\{(y,s)\in\cal{P}:y_0=x\}$$
in \eqref{eq:statprocct} for each state $x$, we find that initialising the chain by sampling a stationary distribution $\pi$ ensures that the chain remains with law $\pi$ for all time:
\begin{equation}\label{eq:statc}\Pbp{\{X_t=x,t<T_\infty\}}=\pi(x)\quad\forall x\in\s,\enskip t\in[0,\infty).\end{equation}
It then follows from \eqref{eq:tstep} that $\pi$ is a stationary distribution  only if it is a fixed point of the transition probabilities in~\eqref{eq:transprob}:
\begin{equation}\label{eq:statfixedpoint}\pi(x)=\Pbp{\{X_t=x,t<T_\infty\}}=\sum_{x'\in\cal{S}}\pi(x')\Pb_{x'}(\{X_t=x,t<T_\infty\})=\pi P_t(x)\quad\forall x\in\s,\end{equation}
or $\pi=\pi P_t$ in matrix notation, for any given $t$ in $[0,\infty)$. The converse is also true:
\begin{theorem}[The fixed points of the transition probabilities are the stationary distributions]\label{thrm:statfixedpoint}A probability distribution $\pi$ on $\s$ is a stationary distribution of $X$ if and only if it satisfies \eqref{eq:statc} for at least one $t$ in $(0,\infty)$, in which case it holds for all $t$ in $[0,\infty)$.\end{theorem}
Using Theorem~\ref{pstateqs}, we can rephrase Theorem~\ref{thrm:statfixedpoint} using the jargon of the previous section: $\pi=(\pi(x))_{x\in\s}$ is a stationary distribution of $X$ if and only if $\pi_E:=(1_{\s}(x)\pi(x))_{x\in\s_E}$ is a stationary distribution of at least one skeleton chain $X^\delta$, in which case $\pi_E$ is a stationary distribution of every skeleton chain $X^\delta$.
\begin{proof}[Proof of Theorem~\ref{thrm:statfixedpoint}]Given the discussion preceding the theorem's statement, we need only to prove that if $\pi$ is a probability distribution satisfying $\pi=\pi P_t$ for any given $t$ in $(0,\infty)$, then $\pi = \pi P_s$ holds for all $s$ in $[0,\infty)$ and $\pi$ is a stationary distribution. We begin with the former: repeatedly applying Tonelli's theorem and the semigroup property~\eqref{eq:semigroup}, we find that
\begin{align*}\sum_{x'\in\s}\pi(x')p_{nt}(x',x)&=\sum_{x''\in\s}\left(\sum_{x'\in\s}\pi(x')p_{t}(x',x'')\right)p_{(n-1)t}(x'',x)\\
&=\sum_{x''\in\s}\pi(x')p_{(n-1)t}(x',x)=\dots=\pi(x)\quad\forall x\in\s,\enskip n>0.\end{align*}
Theorem~\ref{thrm:timevarlimexist} shows that the limits
$$k(x',x):=\lim_{s\to\infty}p_s(x',x)\quad\forall x,x'\in\s,$$
exist and are finite. Bounded convergence then shows that
$$\sum_{x'\in\s}\pi(x')k(x',x)=\lim_{n\to\infty}\sum_{x'\in\s}\pi(x')p_{nt}(x',x)=\lim_{n\to\infty}\pi(x)=\pi(x)\quad\forall x\in\s.$$
Another application of bounded convergence and the semigroup property yields
$$\sum_{x''\in\s}k(x',x'')p_s(x'',x)=\lim_{r\to\infty}\sum_{x''\in\s}p_r(x',x'')p_s(x'',x)=\lim_{r\to\infty}p_{r+s}(x',x)=k(x',x)\quad\forall x,x'\in\s.$$
Combining the above two equations then proves the desired $\pi =\pi P_s$ for any $s$ in $[0,\infty)$:
$$\sum_{x'\in\s}\pi(x')p_s(x',x)=\sum_{x''\in\s}\pi(x'')\sum_{x'\in\s}k(x'',x')p_s(x',x)=\sum_{x''\in\s}\pi(x'')k(x',x)=\pi(x)\quad\forall x\in\s.$$
Marginalising over the equation in \eqref{thrm:findim} and exploiting that $\pi =\pi P_t$ for all $t$ yields \eqref{eq:statprocct} for sets of the form \eqref{eq:gensetsalt}. Because these sets generate $\cal{E}$ (Exercise~\ref{exe:measpedantry}), \eqref{eq:statprocct} holds for all $A$ in $\cal{E}$.
\end{proof}
\subsubsection*{The rate matrix's left nullspace  and the role of explosions}
Summing the \eqref{eq:statc} over all states $x$, taking the limit $t\to\infty$, and applying monotone convergence, we find that the chain is non-explosive when initialised with law $\pi$:
\begin{equation}\label{eq:statnotexpl} \Pbp{\{T_\infty=\infty\}}=1,\end{equation}
The above equation plays a starring role in the relationship between the rate matrix $Q$ and the stationary distributions:
\begin{theorem}[The stationary distributions and the rate matrix]\label{Qstateq} A probability distribution $\pi$ is a stationary distribution if and only if the chain is non-explosive when initialised with law $\pi$~(that is, \eqref{eq:statnotexpl} holds) and
\begin{equation}\label{eq:stateqsthect}\pi Q(x)=0\quad\forall x\in\s,\end{equation}
or $\pi Q=0$ in matrix notation.
%
\end{theorem}
\begin{proof}The definitions in \eqref{eq:jumpmatrix}--\eqref{eq:lambda} of the jump matrix $(p(x,y))_{x,y\in\s}$ and jump rates $(\lambda(x))_{x\in\s}$ imply that $\pi$ satisfies \eqref{eq:stateqsthect} if and only it satisfies

\begin{equation}\label{eq:jge798hga87gha79ga}\pi(x)\lambda(x)=\sum_{z\in\s}\pi(z)\lambda(z)p(z,x)\quad\forall x\in\s.\end{equation}
Suppose that $\pi$ is a stationary distribution. We have already argued that $\pi$ satisfies \eqref{eq:statnotexpl}. Showing it satisfies \eqref{eq:stateqsthect} (equivalently, \eqref{eq:jge798hga87gha79ga}) is also straightforward: multiplying the integral version \eqref{eq:forwardweak} of the forward equations by $\pi(x')$, summing over $x'$ in $\s$, applying Tonelli's theorem and \eqref{eq:statfixedpoint}, we find that 
\begin{align*}\pi(x)=\sum_{x'\in\s}\pi(x')p_t(x',x)&=\pi(x)e^{-\lambda(x)t}+\int_0^t\sum_{z\in\s}\left(\sum_{x'\in\s}\pi(x')p_s(x',z)\right)\lambda(z)p(z,x)e^{-\lambda(x)(t-s)}ds\\
&=\pi(x)e^{-\lambda(x)t}+\int_0^t\sum_{z\in\s}\pi(z)\lambda(z)p(z,x)e^{-\lambda(x)(t-s)}ds\\
&=\pi(x)e^{-\lambda(x)t}+\sum_{z\in\s}\pi(z)\lambda(z)p(z,x)\frac{1-e^{-\lambda(x)t}}{\lambda(x)}\quad\forall x\in\s,\enskip t\in[0,\infty).\end{align*}
Fixing any $t>0$, multiplying through by $\lambda(x)(1-e^{-\lambda(x)t})^{-1}$, and re-arranging we obtain~\eqref{eq:jge798hga87gha79ga}.

Conversely, suppose that $\pi$ is a probability distribution satisfying~\eqref{eq:statnotexpl} and \eqref{eq:stateqsthect} and let $p_t^n(x,y)$ be as in~\eqref{eq:pnt}. By definition,
\begin{align*}\sum_{x'\in\s}\pi(x')p_t^0(x',x)&=\sum_{x'\in\s}\pi(x')\Pb_{x'}(\{X_t=x,t<T_1\})=\sum_{x'\in\s}\pi(x')\Pb_{x'}(\{X_0=x,t<T_1\})\\
&=\pi(x)\Pb_{x}(\{t<T_1\})\leq\pi(x)\quad\forall x\in\s,\enskip t\geq0.\end{align*}
Induction, the forward integral recursion in~\eqref{eq:fir}, and \eqref{eq:jge798hga87gha79ga} then imply that
\begin{align*}\sum_{x'\in\s}\pi(x')p_t^n(x',x)&=\pi(x)e^{-\lambda(x)t}+\int_0^t\sum_{z\in\s}\left(\sum_{x'\in\s}\pi(x')p_s^{n-1}(x',z)\right)\lambda(z)p(z,x)e^{-\lambda(x)(t-s)}ds\\
&\leq\pi(x)e^{-\lambda(x)t}+ \int_0^t\sum_{z\in\s}\pi(z)\lambda(z)p(z,x)e^{-\lambda(x)(t-s)}ds\\
&=\pi(x)e^{-\lambda(x)t}+ \int_0^t\pi(x)\lambda(x)e^{-\lambda(x)(t-s)}ds=\pi(x)\quad\forall x\in\s,\enskip t\geq0,\enskip n\geq0.\end{align*}
Taking the limit $n\to\infty$ and applying monotone convergence we find that 
\begin{equation}\label{eq:fn7aw8fhwyafhhfdvvnuwfa}\sum_{x'\in\s}\pi(x')p_t(x',x)\leq \pi(x)\quad\forall x\in\s.\end{equation}
Suppose that the inequality is strict for at least one $x$. Summing both sides we would have that
$$\Pbp{\{t<T_\infty\}}=\sum_{x\in\s}\Pbp{\{X_t=x,t<T_\infty\}}=\sum_{x\in\s}\sum_{x'\in\s}\pi(x')p_t(x',x)<\sum_{x\in\s}\pi(x)=1$$
which violates \eqref{eq:statnotexpl}. Hence, it must be the case that \eqref{eq:fn7aw8fhwyafhhfdvvnuwfa} holds with equality.
\end{proof}
It is not difficult to find examples of probability distributions $\pi$s satisfying \eqref{eq:stateqsthect} that are not stationary distributions:
\begin{example}[Miller's example]\label{ex:miller}Consider\index{Miller's example} the chain $X$ with state space $\n$ and  rate matrix $(q(x,y))_{x,y\in\n}$ defined by
\begin{align*}&q(x,x-1)=4^x/2\quad\forall x>0,\quad q(x,x)=-(4^x+4^x/2),\quad q(x,x+1)=4^x\quad\forall x\geq0,\\
&q(x,y)=0\quad\text{for all other }x,y\geq0.\end{align*}
It is easy to verify that $\pi:=(2^{-(x+1)})_{x\in\n}$ is a probability distribution satisfying~\eqref{eq:stateqsthect}. However, $\pi$ cannot be a stationary distribution for otherwise every state would be recurrent (see Theorem~\ref{thrm:posreccharc} in the next section), something impossible given that the chain is twice as likely to jump one state up than one state down (if you're not following this, plug $Q$ into~\eqref{eq:jumpmatrix} and look at the definition of the jump chain in Algorithm~\ref{gilalg}). 

For a more formal argument, note that a state is recurrent for $X$ if and only if it is recurrent for its jump chain $Y$ (Exercise~\ref{ex:entdc}). Let's examine the return probability $\Pb_0(\{\phi_0<\infty\})$ to zero of $Y$, with $\phi_k$ denote the first entrance time of $Y$ to a set $k$ (Definition~\ref{def:entrance}). Because the chain necessarily jumps to $1$ if it starts at zero,
$$\Pb_0(\{\phi_0<\infty\})=\Pb_1(\{\sigma_0<\infty\}),$$
where $\sigma_k$ denotes the hitting time of $Y$ to a state~$k$~\eqref{eq:hitd}. Because $\sigma_k\to\infty$ as $k\to\infty$ with $\Pb_1$-probability one (Lemma~\ref{lem:dtexitrinf}), monotone convergence implies that 
$$\Pb_1(\{\sigma_0<\infty\})=\lim_{k\to\infty}\Pb_1(\{\sigma_0<\sigma_k\}).$$
Because  $Y$ behaves identically to the gambler's ruin chain in Section~\ref{sec:gamblers} with $a:=2/3$ up until the moment it hits $0$, \eqref{eq:fn8eanaf8wnfwe8a2} implies that
$$\Pb_1(\{\sigma_0<\sigma_k\})=1-\frac{1}{2-2^{-(k-1)}}\quad\forall k>0.$$
Thus, $\Pb_0(\{\phi_0<\infty\})=1/2$ showing that $0$ is transient for $Y$ (and, consequently, for $X$) and that $\pi$ cannot be a stationary distribution. Moreover, Theorem \ref{Qstateq} then implies that $\Pbp{\{T_\infty<\infty\}}>0$. Because the chain is irreducible, it follows from Theorem~\ref{thrm:expclprop} that the chain is explosive regardless of the initial distribution:
$$\Pbl{\{T_\infty<\infty\}}>0\text{ for all probability distributions }\gamma.$$
\end{example}

\subsubsection*{A stationary distribution has support on a state if and only if the state is positive recurrent}If the chain is a stationary process, the average fraction of time it spends in any given state remains constant throughout time. In particular, sampling the starting location from a stationary distribution $\pi$, taking expectations of the empirical distribution $\epsilon_T$~(Section~\ref{sec:empdistct}), and applying~\eqref{eq:statfixedpoint} we find that
\begin{equation}\label{eq:pttimeave}\Ebp{\epsilon_T(x)}=\Ebp{\frac{1}{T}\int_0^{T\wedge T_\infty}1_x(X_t)dt}=\frac{1}{T}\int_0^{T}\Pbp{\{t<T_\infty,X_t=x\}}dt=\pi(x)\end{equation}
for all $x$ in $\s$ and $T\geq0$. Given Theorem~\ref{thrm:empdistlimsct}, taking the limit $T\to\infty$ and applying bounded convergence we find that
\begin{equation}\label{eq:empdistmeanct}\pi(x)=\lim_{T\to\infty}\Ebp{\epsilon_T(x)}=\frac{\Pbp{\{\varphi_x<\infty\}}}{\Ebx{\varphi_x}}\quad\forall x\in\s.\end{equation}
It follows that a stationary distribution $\pi$ with support on $x$ (i.e., $\pi(x)>0$) exists only if $x$ is positive recurrent~(Definition~\ref{def:posrec}). The intuition here is that states $x$ satisfying $\pi(x)>0$ are those the chain will keep visiting frequently enough that the fraction of time $\epsilon_T(x)$ it spends in $x$ does not decay to zero whenever the initial position is sampled from $\pi$. Because, regardless of the starting location, visits to all  states $x$ not positive recurrent eventually become so rare that $\epsilon_T(x)$ collapses to zero (Theorem~\ref{thrm:empdistlimsct}) it follows that for $\pi(x)$ to be non-zero, $x$ must be positive recurrent.

The converse is also true and the trick in arguing it involves the stopping distributions and occupation measures of Section~\ref{sec:morehitdt}. In particular, let $\mu_S$ and $\nu_S$ be the space marginals~\eqref{eq:musc}--\eqref{eq:nusc} of the stopping distribution and occupation measure associated with the return time $\varphi_x$ of a recurrent state $x$ and suppose that we started the chain in $x$. Exercise~\ref{ex:entdc} implies that the chain is in $x$  at the moment of return. Because $x$ is recurrent, it follows that $\mu_S$ is the point mass $1_x$ at $x$. Furthermore, as we fixed the initial distribution $\gamma$ to also be $1_x$, the same proposition and Lemma~\ref{eqnsc} imply that 
\begin{equation}\label{eq:dnw78andw7ya8hdbwaudhnawu}q(x)\nu_S(x)=\sum_{x'\neq x}\nu_S(x')q(x',x)\quad\forall x\in\s.\end{equation}
Because the state is recurrent, $\varphi$'s definition (Definition~\ref{def:entrancect}) implies that it less than $T_\infty$,  $\Pb_x$ almost surely. For this reason, \eqref{eq:numass} shows that the  mass of $\nu_S$ is the mean return time $\Ebx{\varphi_x}$. If $x$ is positive recurrent, then the mass and, consequently, all entries of $\nu_S$ are finite. Thus, we can rearrange \eqref{eq:dnw78andw7ya8hdbwaudhnawu} into $\nu_SQ=0$ and  we obtain a probability distribution $\pi$ satisfying $\pi Q=0$ by normalising $\nu_S$ (i.e., $\pi:=\nu_S/\nu_S(\s)$). Moreover, 
\begin{align*}\pi(x)&=\frac{\nu_S(x)}{\nu_S(\s)}=\frac{\Ebx{\int_0^{\varphi_x\wedge T_\infty}1_x(X_s)ds}}{\Ebx{\varphi_x}}\geq\frac{\Ebx{\int_0^{T_1}1_x(X_s)ds}}{\Ebx{\varphi_x}}=\frac{\Ebx{S_11_x(X_0)}}{\Ebx{\varphi_x}}\\
&=\frac{1}{\lambda(x)\Ebx{\varphi_x}}>0,\end{align*}
where $S_1$ denotes the first waiting time (that is exponentially distributed with mean $\lambda(x)$ under $\Pb_x$, see Section~\ref{sec:FKG}). If we can only show that~\eqref{eq:statnotexpl} is satisfied, Theorem~\ref{eq:statfixedpoint} and the above then show that $\pi$ a stationary distribution with support on $x$. To argue~\eqref{eq:statnotexpl} note that Proposition~\ref{prop:recnoexpl} shows that $\Pbx{\{T_\infty=\infty\}}=1$. Because
\begin{equation}\label{eq:fme9anfe78gnayu8gna}\Ebx{\varphi_x}\pi(y)=\nu_S(y)=\Ebx{\int_0^{\varphi_x\wedge T_\infty}1_x(X_s)ds}\leq \Ebx{\int_0^{T_\infty}1_x(X_s)ds}=\int_0^\infty p_t(x,y)ds\end{equation}
for all $y$ in $\s$, Theorem~\ref{thrm:accesibility}$(v)$ implies that $\pi(y)>0$ only if $y$ is accessible from $x$ and it follows from Theorem~\ref{thrm:expclprop} that
$$\Pby{\{T_\infty=\infty\}}=1\quad\forall y\in\s:\pi(y)>0.$$
Multiplying the above by $\pi(y)$, summing over all states $y$, and comparing with \eqref{eq:pl} then yields the missing~\eqref{eq:statnotexpl}. To make a long story short:
\begin{theorem}\label{thrm:posreccharc} A state $x$ is positive recurrent if and only if there exists a stationary distribution $\pi$ such that $\pi(x)>0$. In this case, one such stationary distribution is given by
\begin{equation}\label{eq:statexprsc}\pi(y)=\frac{1}{\Ebx{\varphi_x}}\Ebx{\int_0^{\varphi_x\wedge T_\infty}1_y(X_s)ds}\quad\forall y\in\s.\end{equation}
\end{theorem}
As a freebie, the above also gives us that positive recurrence is a class property:
\begin{corollary}\label{cor:posrecclassc} If $x$ is positive recurrent and $x\to y$, then $y$ is also positive recurrent. Moreover, a state in a communicating class is positive (resp. null) recurrent if and only if all states in the class are positive (resp. null) reccurent.
\end{corollary}
\begin{proof} Given Theorems \ref{thrm:posreccharc},~\ref{thrm:accesibility}, and \ref{thrm:posrecchar}, we can argue this by following the states taken in the proof of the corollary's discrete-time counterpart (Corollary \ref{cor:posrecclass}). Alternatively, given that $\pi$ is a stationary distribution for $X$ if and only if $(1_{\s}(x)\pi(x))_{x\in\s_E}$ is for any and all skeleton chains $X^\delta$ (Theorem~\ref{thrm:statfixedpoint}) and that  $y$ is accessible from $x$ for $X$ if and only it is accessible for any and all skeleton chains $X^\delta$ (Theorem~\ref{thrm:accesibility}$(v)$), the corollary follows by applying Corollary \ref{cor:posrecclass} to any skeleton chain $X^\delta$.
\end{proof}
For this reason, we say that a communicating class is \emph{positive recurrent} (resp. \emph{null recurrent}) if each one (or, equivalently, all) of its states is positive recurrent (resp. null recurrent).\index{positive recurrent class}\index{null recurrent class}

\subsubsection*{Ergodic distributions; Doeblin-like decomposition; set of stationary distributions}

To wrap up and our treatment of stationary distributions and characterise the set of these, consider the following \emph{Doeblin-like} decomposition\index{Doeblin-like decomposition} of the state space:
\begin{equation}\label{eq:ddc}\s=\left(\bigcup_{i\in\cal{I}} \C_i\right)\cup \cal{T}=\cal{R}_+\cup\cal{T},\end{equation}
where $\{\C_i:i\in\cal{I}\}$ denotes the (necessarily countable) set of positive recurrent closed communicating classes,\glsadd{ci} which we index with some  set $\cal{I}$, $\cal{R}_+:=\cup_{i\in\cal{I}}\C_i$ the set of all positive recurrent states,\glsadd{recp} and $\cal{T}$ that of all other states.\glsadd{nrec} As we have already seen in Theorem~\ref{thrm:posrecchar}, no stationary distribution has support in $\cal{T}$ and there is at least one stationary distribution per $\cal{C}_i$. Much more can be said:
%
%
\begin{theorem}[Characterising the set of stationary distributions]\label{doeblinc}  Let $\{\C_i:i\in\cal{I}\}$ be the collection of positive recurrent closed communicating classes.
\vspace{5pt}
\begin{enumerate}[label=(\roman*),noitemsep]
\item For each $\cal{C}_i$, there exists a single stationary distribution $\pi_i$ with support contained in $\cal{C}_i$ (i.e., with $\pi_i(\cal{C}_i)=1$); $\pi_i$ is  known as the ergodic distribution\index{ergodic distribution}\glsadd{pii} associated with $\cal{C}_i$. 
It has support on all of $\C_i$ ($\pi_i(x)>0$ for all states $x$ in $\cal{C}_i$) and can be expressed as
\begin{equation}\label{eq:ergdistcharc}\pi_i(y)=\frac{1_{\cal{C}_i}(y)}{\Eby{\varphi_y}}=\frac{1}{\Ebx{\varphi_x}}\Ebx{\int_0^{\varphi_x\wedge T_\infty}1_y(X_t)dt}\quad\forall y\in\s,\end{equation}
where $x$ is any state in $\cal{C}_i$.
\item A probability distribution $\pi$ is a stationary distribution of the chain if and only if it is a convex combination of the ergodic distributions:
$$\pi=\sum_{i\in\cal{I}}\theta_i\pi_i$$
for some collection $(\theta_i)_{i\in\cal{I}}$ of non-negative constants satisfying $\sum_{i\in\cal{I}}\theta_i=1$. For any stationary distribution $\pi$, the weight $\theta_i$ featuring in the above is the mass that $\pi$ awards to $\cal{C}_i$ or, equivalent, the probability that the chain ever enters $\cal{C}_i$ if its starting location was sampled from $\pi$:
$$\theta_i=\pi(\cal{C}_i)=\Pbp{\{\varphi_{\cal{C}_i}<\infty\}}\quad\forall i\in\cal{C}_i.$$
\end{enumerate} 
\end{theorem}
\begin{proof} $(i)$ For any state $x$ in a recurrent closed communicating class $\cal{C}_i$, Proposition~\ref{prop:phixphic} implies that
$$\Pbx{\{\varphi_y<\infty\}}=1_{\cal{C}_i}(y)\quad\forall y\in\s.$$
Plugging the above into \eqref{eq:empdistmeanct} we find that there can only exist one stationary distribution $\pi_i$ with support contained in  $\cal{C}_i$ (i.e., $\pi_i(x)=0$ for all $x\not\in\cal{C}_i$), as for any such $\pi_i$
$$\pi_i(x)=\frac{\Pb_{\pi_i}(\{\varphi_x<\infty\})}{\Ebx{\varphi_x}}=\frac{\sum_{x'\in\cal{C}_i}\pi_i(x')\Pb_{x'}(\{\varphi_x<\infty\})}{\Ebx{\varphi_x}}=\frac{1_{\cal{C}_i}(x)}{\Ebx{\varphi_x}}\quad\forall x\in\s.$$
On the other hand, Theorem~\ref{thrm:posreccharc} shows that the right-hand side of \eqref{eq:ergdistcharc} defines such a  stationary distribution $\pi_i$ because Theorem~\ref{thrm:accesibility}$(v)$, \eqref{eq:fme9anfe78gnayu8gna}, and the closedness of $\cal{C}_i$ imply that $\pi_i$ has support contained $\cal{C}_i$.

$(ii)$ Given Proposition~\ref{prop:phixphic}, Theorem~\ref{thrm:statfixedpoint}, and \eqref{eq:empdistmeanct}, this proof is entirely analogous to that of its discrete-time counterpart (Theorem~\ref{doeblind}$(ii)$).
\end{proof}

%
%
%
%
%
%

\subsubsection*{Notes and references} Theorem~\ref{Qstateq} traces back to \citep{Kendall1957} and \citep{Miller1963} where it is shown that, for regular $Q$, \eqref{eq:statfixedpoint} and \eqref{Qstateq} are equivalent. The minor extension in Theorem~\ref{Qstateq} seems to be due to \citep{Kuntzthe}---although I still find this suspicious---and the proof given here closely follows~\citep{Miller1963}.

\subsection[Limits of the empirical distribution; positive recurrent chains]{Limits of the empirical distribution and positive recurrent chains}\label{sec:timeave}Combining the results of the previous sections we obtain a complete description  of the empirical distribution $\epsilon_T$~(Section~\ref{sec:empdistct}) that tracks the fraction of time that the chain spends in each state:
\begin{theorem}[The pointwise limits]\label{thrm:pointwiselimc} Let $\epsilon_T$ denote the empirical distribution in~\eqref{eq:timeavedefct}, $\{\cal{C}_i:i\in\cal{I}\}$ denote the collection of positive recurrent closed communicating classes $\cal{C}_i$~(Section~\ref{sec:statct}), and, for each $i\in\cal{I}$, $\pi_i$ denote the ergodic distribution of $\cal{C}_i$~(Theorem~\ref{doeblinc}). For any initial distribution $\gamma$, we have that\glsadd{epinf}
\begin{equation}\label{eq:reclimsct}\epsilon_\infty:=\lim_{T\to\infty}\epsilon_T=\sum_{i\in\cal{I}}1_{\{\varphi_{\cal{C}_i}<\infty\}}\pi_i,\quad\Pb_\gamma\text{-almost surely},\end{equation}
where  the convergence is pointwise and $\varphi_{\C_i}$ denotes the time of first entrance to $\C_i$~(Definition~\ref{def:entrancect}).
\end{theorem}
\begin{proof}Because, with $\Pb_\gamma$-probability one, the chain will enter a state $x$ (i.e., $\varphi_x<\infty$) in a recurrent closed communicating class $\cal{C}_i$ if and only if enters the class (i.e., $\varphi_{\cal{C}_i}<\infty$) at all (Proposition~\ref{prop:phixphic}), this follows directly from Theorems~\ref{thrm:empdistlimsct} and \ref{doeblinc}$(i)$.
\end{proof}
It is simple to describes the circumstances under which the limit \eqref{eq:reclimsct} holds in total variation:
\begin{corollary}[The convergence in total variation]\label{cor:timeavetv}The chain enters the set $\cal{R}_+$ of positive recurrent states with probability one (i.e., $\Pbl{\{\varphi_{\cal{R}_+}<\infty\}}=1$) if and only if the  limit \eqref{eq:reclimsct} holds in total variation $\Pb_\gamma$-almost surely: with $\norm{\cdot}$ as in~\eqref{eq:tvnorm},
$$\lim_{T\to\infty}\norm{\epsilon_T-\epsilon_\infty}=0,\quad\Pb_\gamma\text{-almost surely},$$
\end{corollary}

\begin{proof}If the chain is non-explosive, then $\epsilon_T(\s)=1$ with $\Pb_\gamma$-probability one, for all $T$ in $[0,\infty)$. Thus, if the limit \eqref{eq:reclimsct} holds in total variation, then 
$$1=\lim_{n\to\infty}\epsilon_n(\s)=\epsilon_\infty(\s)=\sum_{i\in\cal{I}}1_{\{\varphi_{\cal{C}_i}<\infty\}}\pi_i(\s)=\sum_{i\in\cal{I}}1_{\{\varphi_{\cal{C}_i}<\infty\}}\quad\Pb_\gamma\text{-almost surely}.$$
Because the positive recurrent classes are closed and disjoint, it follows from Proposition~\ref{prop:closedisct}   that
\begin{equation}\label{eq:varcivarrp}\sum_{i\in\cal{I}}1_{\{\varphi_{\cal{C}_i}<\infty\}}=1_{\{\varphi_{\cup_{i\in\cal{I}}\cal{C}_i}<\infty\}}=1_{\{\varphi_{\cal{R}_+}<\infty\}}\quad\Pb_\gamma\text{-almost surely}.\end{equation}
Putting the above two together, we have that $\Pbl{\{\varphi_{\cal{R}_+}<\infty\}}=1$.

Conversely, if $\Pbl{\{\varphi_{\cal{R}_+}<\infty\}}=1$, then Proposition~\ref{prop:recnoexpl} and the strong Markov property (Theorem~\ref{thrm:markovprop}) imply that the chain is non-explosive:
\begin{align}\Pbl{\{T_\infty=\infty\}}&=\Pbl{\{\varphi_{\cal{R}_+}<T_\infty,T_\infty=\infty\}}\label{eq:j8a7w0ef9hg87ea0nguapngea}\\
&=\sum_{x\in\cal{R}_+}\Pb_\gamma(\{\varphi_{\cal{R}_+}<T_\infty,X_{\varphi_{\cal{R}_+}}=x,T_\infty^{\varphi_{\cal{R}_+}}-\varphi_{\cal{R}_+}=\infty\})\nonumber\\
&=\sum_{x\in\cal{R}_+}\Pb_\gamma(\{\varphi_{\cal{R}_+}<T_\infty,X_{\varphi_{\cal{R}_+}}=x\})\Pbx{\{T_\infty=\infty\}}\nonumber\\
&=\sum_{x\in\cal{R}_+}\Pb_\gamma(\{\varphi_{\cal{R}_+}<T_\infty,X_{\varphi_{\cal{R}_+}}=x\})=\Pbl{\{\varphi_{\cal{R}_+}<\infty\}}=1,\nonumber\end{align}
where $T_\infty^{\varphi_{\cal{R}_+}}$ is as in \eqref{eq:tshift} and the second and penultimate equations follow from Exercise~\ref{ex:entdc}. In other words, the empirical distribution $\epsilon_T$ has mass one (i.e., $\epsilon_T(\s)=1$) for all $T$ in $[0,\infty)$, with $\Pb_\gamma$-probability one. Furthermore, \eqref{eq:varcivarrp} implies that $\epsilon_\infty(\s)=1$ $\Pb_\gamma$-almost surely. Thus, the pointwise convergence (Theorem~\ref{thrm:pointwiselimc}) and  Scheffe's lemma (Lemma~\ref{lem:scheffe}) imply that
\begin{equation}\label{eq:fdmue9famn893martwga}\lim_{N\to\infty}\norm{\epsilon_{\delta N}-\sum_{i\in\cal{I}}1_{\{\varphi_{\cal{C}_i<\infty\}}\pi_i}}=0\quad\Pb_\gamma\text{-almost surely},\end{equation}
for every  $\delta$ in $(0,\infty)$. That the limit \eqref{eq:reclimsct} holds in total variation then follows from Croft's lemma (Corollary~\ref{cor:croft}) assuming that we are able to show that 
$$T\mapsto \norm{\epsilon_{T}-\sum_{i\in\cal{I}}1_{\{\varphi_{\cal{C}_i<\infty\}}\pi_i}}$$
is a continuous function on $[0,\infty)$, with $\Pb_\gamma$-probability one. As the total variation distance between two probability distributions is half the $\ell^1$-distance (see \eqref{eq:tvl1p}), we need only show that 
\begin{equation}\label{eq:fmew98anfea9wuaf}T\mapsto \sum_{x\in\s}\mmag{\epsilon_{T}(x)-\sum_{i\in\cal{I}}1_{\{\varphi_{\cal{C}_i<\infty\}}\pi_i(x)}}\end{equation}
is a continuous function on $[0,\infty)$, with $\Pb_\gamma$-probability one. 

Let $(\s_r)_{r\in\zp}$ be an increasing  sequence of finite truncations approaching the state space (i.e., such that $\cup_{r=1}^\infty\s_r=\s$). The function in \eqref{eq:fmew98anfea9wuaf} is the pointwise limit of the sequence 
$$\left\{T\mapsto \sum_{x\in\s_r}\mmag{\epsilon_{T}(x)-\sum_{i\in\cal{I}}1_{\{\varphi_{\cal{C}_i<\infty\}}\pi_i(x)}} \right\}_{r\in\zp}$$
of functions. Because, for all $x$, $T\mapsto \epsilon_T(x)$ is continuous by its definition in~\eqref{eq:timeavedefct} and the truncations are  finite, all of the functions in the sequence are continuous. For this reason, to show continuity of \eqref{eq:fmew98anfea9wuaf} we need only to argue that the sequence converges uniformly on $[0,T^*]$, for every $T^*\in(0,\infty)$. Fix such a $T^*$ and note that
\begin{align}\nonumber\sum_{x\not\in\s_r}\mmag{\epsilon_{T}(x)-\sum_{i\in\cal{I}}1_{\{\varphi_{\cal{C}_i}<\infty\}}\pi_i(x)}&\leq \sum_{x\not\in\s_r}\left(\epsilon_{T}(x)+\sum_{i\in\cal{I}}1_{\{\varphi_{\cal{C}_i}<\infty\}}\pi_i(x)\right)\\
&=\frac{1}{T}\int_0^{T\wedge T_\infty}1_{\s_r^c}(X_t)dt+\sum_{i\in\cal{I}}1_{\{\varphi_{\cal{C}_i}<\infty\}}\pi_i(\s_r^c).\label{eq:mf8u9aefm498amfua}\end{align}
The rightmost term does not depend on $T$. Given that the chain enters only one positive recurrent class (Proposition~\ref{prop:closedisct}), this term can be made arbitrarily small by choosing large enough $r$. Furthermore, \eqref{eq:j8a7w0ef9hg87ea0nguapngea} implies that the set $S(\omega)$ of states that the chain's path $t\mapsto X_t(\omega)$ visits over the interval $[0,T^*\wedge T_\infty(\omega)]$ is finite (i.e., $T_n(\omega)\geq T^*$ for some sufficiently large $n$), for $\Pb_\gamma$-almost every $\omega$ in $\Omega$. Because the truncations are increasing and approach the entire state space ($\cup_{r=1}^\infty\s_r=\s$), it follows that $S(\omega)$ is contained in $\s_r$ for all sufficiently large $r$ and the first term in the right-hand side of \eqref{eq:mf8u9aefm498amfua} (evaluated at $\omega$) is zero for all such $r$s and $T$ in $[0,T^*]$.
%
\end{proof}

\subsubsection*{Positive recurrent chains}

Corollary~\ref{cor:timeavetv} motivates the following definitions:
%
\begin{definition}[Positive recurrent chains]\label{def:posTweediect} A chain is positive Tweedie  recurrent\index{positive Tweedie recurrent chain} if 
$$\Pbl{\{\varphi_{\cal{R}_+}<\infty\}}=1\text{ for all initial distributions }\gamma,$$
where $\cal{R}_+$ denotes the set of positive recurrent states. If, additionally, there is only one closed communicating class, then the chain is said to be Harris recurrent\index{positive Harris recurrent chain}. If this class is the entire state space, then the chain is simply said to be positive recurrent.\index{positive recurrent chain}
\end{definition}
%
%
The corollary shows that the empirical distribution converges in total variation, $\Pb_\gamma$-almost surely, for every initial distribution $\gamma$ if and only if the chain is positive Tweedie recurrent. Using analogous arguments to those given in Section~\ref{sec:timeaved} for the discrete-time case, it then follows that: 
\begin{enumerate}
\item If the chain is positive Tweedie recurrent, then, for any given initial distribution $\gamma$ and $\varepsilon$ in $(0,1]$, we are always able to find a finite set $F$ such that the chain spends at least $(1-\varepsilon)\times100\%$ of all time inside $F$:
$$\epsilon_\infty(F):=\lim_{T\to\infty}\frac{1}{T}\int_0^{T\wedge T_\infty}1_F(X_t)dt\geq 1-\varepsilon\quad\Pb_\gamma\text{-almost surely}.$$
\item Otherwise, there exists at least one initial distribution $\gamma$ for which a \emph{non-negligible fraction of $X$'s paths will spend infinitely more time outside any given finite set $F$ than inside the set:}
$$\Pbl{\{\epsilon_\infty(F)=0\text{ for all finite }F\subseteq\s\}}\geq \Pbl{\{\phi_{\cal{R}_+}=\infty\}}>0.$$
\end{enumerate}
For these reasons, I believe that positive Tweedie recurrence precisely captures what most practitioners think of when they hear the words `a stable chain'. 


%
%
\ifdraft
\subsubsection*{The law of large numbers}Putting Corollary~\ref{cor:timeavetv} together with Theorems~\ref{}~and~\ref{} we find that Positive Harris recurrence implies that
\begin{equation}\label{eq:llnmarkovct}\lim_{T\to\infty}\frac{1}{T}\int_0^Tf(X_t)dt=\pi(f)\quad\Pb_\gamma\text{-almost surely},\end{equation}
for all bounded functions $f$ and initial distributions $\gamma$, where $\pi$ denotes the chain's stationary distribution. This generalisation of the law of large numbers holds for a much wider class of functions. As in the discrete-time case, we will restrict ourselves here to initial conditions with support on the positive recurrent states for the sake of brevity (see Exercise~\ref{ex:nf8eaw9ng7ae8whgaw} for more on the general case).
\begin{theorem}\label{thrm:llnct}Let $\{\cal{C}_i\}_{i\in\cal{I}}$ denote the set of positive recurrent classes and $\{\pi_i\}_{i\in\cal{I}}$ the corresponding set of stationary distributions. If the initial distribution $\gamma$ has support contained in the set $\cup_{i\in\cal{I}}\cal{C}_i$ of positive recurrent states, then
\begin{equation}\label{eq:llnmarkovgenct}\lim_{T\to\infty}\epsilon_T(f)=\lim_{T\to\infty}\frac{1}{T}\int_0^{T\wedge T_\infty}f(X_t)dt=\sum_{i\in\cal{I}}1_{\{\varphi_{\cal{C}_i}<\infty\}}\pi_i(f)=\epsilon_\infty(f)\quad\Pb_\gamma\text{-almost surely},\end{equation}
for all non-negative functions $f$ and all  $f$ such that $\epsilon_\infty(|f|)<\infty$ with $\Pb_\gamma$-probability one.
\end{theorem}

\begin{exercise}Prove Theorem~\ref{thrm:llnct} by following the steps taken in the proof of its discrete-time counterpart. To do so, note that $F_-$ and $F_+$ in Lemma~\ref{lem:pathspmeasct}$(iv)$ are such that
$$F_-(X)=\liminf_{T\to\infty}\frac{1}{T}\int_0^Tf(X_t)dt\quad F_+(X)=\liminf_{T\to\infty}\frac{1}{T}\int_0^Tf(X_t)dt,$$
and replace Theorems~\ref{thrm:strmkvpath},~\ref{thrm:rectrans},~and~\ref{thrm:rec-iid} with their continuous-time counterparts (Theorems~\ref{thstrmk},~\ref{thrm:rectransct},~and~\ref{thrm:rec-iidct}), Proposition~\ref{prop:closedis} with its counterpart (Proposition~\ref{prop:closedisct}), and $R_N$ with
$$R_T:=\left\{\begin{array}{cc}0&\text{if }\varphi_x^0\leq T <\varphi_x^1\\1&\text{if }\varphi_x^1\leq T <\varphi_x^2\\\vdots&\vdots\end{array}\right.\quad\forall T\in[0,\infty).$$
\end{exercise}
\fi
\ifdraft
\begin{exercise}[The general case]\label{ex:nf8eaw9ng7ae8whgaw} If you are in the mood for a challenge: Use Theorem~\ref{thrm:llnct}, skeleton chains, and a harmonic function approach of the type in \citep[Prop.~17.1.6]{Meyn2009} to show that \eqref{eq:llnmarkovgenct} holds for all any initial distribution $\gamma$ as long as $f$ satisfies $\epsilon_\infty(|f|)<\infty$ with $\Pb_\gamma$-probability one.\end{exercise}

\subsubsection*{Notes} {\color{red}\eqref{thrm:pointwiselimd} provides the justification for the naive monte-carlo approach of estimating an ergodic distribution.} Explain choice of jargon Tweedie. Mention the ratio limits as exercises?

\fi

\subsection{Limits of the time varying-law, tightness, and ergodicity}\label{sec:limsct}

The main aim of this section is to prove the theorem below spelling out the long-term behaviour of the time-varying law $(p_t)_{t\geq0}$ (introduced in Section~\ref{sec:forward}).

\begin{theorem}[The pointwise limits]\label{thrm:pointlimsct} Let $\{\cal{C}_i:i\in\cal{I}\}$ denote the collection of positive recurrent closed communicating classes and let $\pi_i$ be the ergodic distribution of $\cal{C}_i$ for each $i\in\cal{I}$ (c.f.~Theorem~\ref{doeblinc}).  For any initial distribution $\gamma$,\glsadd{pigamma} 
\begin{equation}\label{eq:reclims2ct}\lim_{t\to\infty}p_t=\sum_{i\in\cal{I}}\Pbl{\{\varphi_{\cal{C}_i}<\infty\}}\pi_i=:\pi_\gamma,\end{equation}
where  the convergence is pointwise and $\varphi_{\C_i}$ denotes the time of first entrance time to $\C_i$.
\end{theorem}

\begin{proof} See the end of the section.\end{proof}

\subsubsection*{Convergence in total variation and tightness}The question of when does the limit in~\eqref{eq:reclims2ct} hold in total variation has a straightforward answer involving the notion of \emph{tightness}\footnote{Here, we are topologising the state space using the discrete metric so that the compact sets are the finite sets.}:\index{tightness of the time-varying law}
\begin{definition}[Tightness]\label{def:tightct}A set $\{\rho_t:t\in[0,\infty)\}$ of probability distributions on $\s$ indexed by $t\in[0,\infty)$  is tight if and only if for every $\varepsilon>0$ there exists a finite set $F$ such that $\rho_t(F)\geq 1-\varepsilon$ for all $t$ in $[0,\infty)$.
\end{definition}
We then have the following corollary of Theorem~\ref{thrm:pointlimsct}:
\begin{corollary}\label{cor:tvtightct}The following are equivalent:
\begin{enumerate}[label=(\roman*),noitemsep] 
\item The chain enters the set of positive recurrent states with probability one: $\Pbl{\{\phi_{\cal{R}_+}<\infty\}}=1$.
\item The chain is non-explosive ($\Pbl{\{T_\infty=\infty\}}=1$) and the time-varying law  is tight.
\item The chain is non-explosive ($\Pbl{\{T_\infty=\infty\}}=1$) and the limit~\eqref{eq:reclims2ct} holds in total variation: with $\norm{\cdot}$ as in~\eqref{eq:tvnorm},
\begin{equation}\label{eq:pttv}\lim_{t\to\infty}\norm{p_t-\pi_\gamma}=0.\end{equation}
\end{enumerate}
\end{corollary}

\begin{proof}See the end of the section.
\end{proof}

\subsubsection*{Positive Tweedie recurrence revisited} In Section~\ref{sec:timeave}, we characterised positive Tweedie recurrent chains (Definition~\ref{def:posTweediect}) in terms of the limits of the time-varying law $\epsilon_T$ in~\eqref{eq:timeavedefct}. Armed with Corollary~\ref{cor:tvtightct}, it is straightforward to improve this characterisation:\index{positive Tweedie recurrent chain}
\begin{corollary}[Characterising positive Tweedie recurrent chains]\label{cor:tweedierecct}The following conditions are equivalent:
\begin{enumerate}[label=(\roman*),noitemsep] 
\item The chain is positive Tweedie recurrent.
\item The chain is non-explosive ($\Pbl{\{T_\infty=\infty\}}=1$) and the empirical distribution converges in total variation to $\epsilon_\infty$ in~\eqref{eq:reclimsct} with $\Pb_\gamma$-probability one, for all initial distributions $\gamma$.
\item The chain is non-explosive ($\Pbl{\{T_\infty=\infty\}}=1$) and the time-varying law is tight, for all initial distributions $\gamma$.
\item The chain is non-explosive ($\Pbl{\{T_\infty=\infty\}}=1$) and the time-varying law converges in total variation (i.e.,~\eqref{eq:pttv}), for all initial distributions $\gamma$.
\end{enumerate}
%
\end{corollary}
\begin{proof}This follows directly from Corollaries~\ref{cor:timeavetv} and \ref{cor:tvtightct} and the definition  of positive Tweedie recurrence (Definition~\ref{def:posTweediect}).
\end{proof}
In other words, a chain is positive Tweedie recurrent if and only if for each initial distribution $\gamma$ and $\varepsilon$ in $(0,1]$, we can find a large enough finite set $F$ such that the chain has at least $1-\varepsilon$ probability of being at $F$ at any given time (i.e., $p_t(F)\geq 1-\varepsilon$ for all $t$): further reinforcing the idea discussed in Section~\ref{sec:timeave} that a chain is stable if and only if it is positive Tweedie recurrent.

\subsubsection*{Ergodicity and positive Harris recurrence}In the case of a positive Harris recurrent  chain with a single closed communicating class, Theorem~\ref{doeblinc} and Corollary \ref{cor:tweedierecct} show that the  the chain has a unique stationary distribution $\pi$ and that, regardless of the initial distribution, both the time-varying law  and  the empirical distribution  converge  to it:
$$\lim_{T\to\infty}\epsilon_T=\lim_{t\to\infty}p_t=\pi\quad\Pb_{\gamma}\text{-almost surely},$$
for all initial distributions $\gamma$, where the convergence is in total variation. That is, the \emph{time averages} $\epsilon_T$ converge to the \emph{space averages} $p_t$ and the chain is said to be \emph{ergodic}.\index{ergodic}

\subsubsection*{A proof of Theorem~\ref{thrm:pointlimsct}}  In~\eqref{eq:nf8ahbfw678bfhafwa}, we showed that 
\begin{equation}\label{eq:jfje78nfay8fenafa}\lim_{t\to\infty}p_t(x)=\sum_{i\in\cal{I}^\delta}\Pb_\gamma(\{\phi^\delta_{\cal{C}_i^\delta}<\infty\})\pi_i^\delta(x)\quad\forall x\in\s,\end{equation}
for any $\delta$ in $(0,\infty)$, where $\{\cal{C}_i^\delta:i\in\cal{I}^\delta\}$ denote the collection of positive recurrent closed communicating classes of the skeleton chain $X^\delta$ (Section~\ref{sec:skeleton}), $\{\pi_i^\delta:i\in\cal{I}^\delta\}$ the corresponding collection of $X^\delta$'s ergodic distributions, and $\phi^\delta_{\cal{C}_i^\delta}$ the first entrance time to $\cal{C}_i^\delta$ of $X^\delta$:
$$\phi^\delta_{A}(\omega):=\inf\{n>0:X^\delta_{n}(\omega)\in A\}\quad\forall \omega\in\Omega,\enskip A\subseteq\s.$$
Lemma~\ref{lem:accdt}, Theorem~\ref{thrm:accesibility}$(vi)$, and \eqref{eq:skeletonone-step} imply that $X$ and $X^\delta$ have the same closed communicating classes in $\s$. Similarly, Theorems~\ref{thrm:posrecchar}, \ref{doeblind}$(i)$, \ref{thrm:statfixedpoint},  \ref{thrm:posreccharc}, and \ref{doeblinc}$(i)$ imply that $X$ and $X^\delta$ have the same positive recurrent states in $\s$ and that $\pi=(\pi(x))_{x\in\s}$ is an ergodic distribution for $X$ if and only if $\pi_E=(1_{\s}(x)\pi(x))_{x\in\s_E}$ is an ergodic distribution for $X^\delta$. Thus, we can rewrite the \eqref{eq:jfje78nfay8fenafa} as
$$\lim_{t\to\infty}p_t(x)=\sum_{i\in\cal{I}}\Pb_\gamma(\{\phi^\delta_{\cal{C}_i}<\infty\})\pi_i(x)\quad\forall x\in\s.$$
Consequently, all we have left to show is that
\begin{equation}\label{eq:fmeia90mfei9fmau9fmeuaf}\Pb_\gamma(\{\phi^\delta_{\cal{C}}<\infty\})=\Pbl{\{\varphi_{\cal{C}}<\infty\}}\end{equation}
for any given positive recurrent closed communicating class $\cal{C}\subseteq\s$ of $X$ (or, equivalently, of $X^\delta$). Because $X^\delta$ is obtained by sampling $X$ every $\delta$ units of time (c.f.~\eqref{eq:skeletondef}), the above follows from $X$ being unable to leave $\cal{C}$ or explode once inside $\cal{C}$. In particular, letting $X^{\varphi_\cal{C}}$ denote the $\varphi_{\cal{C}}$-shifted chain (Section~\ref{sec:markovprop}), we have that
\begin{align*}
\Pb_\gamma(\{\varphi_{\cal{C}}<\infty,T_\infty=\infty,&X_t\in\cal{C}\enskip \forall t\in[\varphi_{\cal{C}},\infty)\})\\
&=\sum_{x\in\cal{C}}\Pb_\gamma(\{\varphi_{\cal{C}}<T_\infty,X_{\varphi_{\cal{C}}}=x,T_\infty^{\varphi_{\cal{C}}}=\infty,X_t^{\varphi_{\cal{C}}}\in\cal{C}\enskip \forall t\in[0,\infty)\})\\
&=\sum_{x\in\cal{C}}\Pb_\gamma(\{\varphi_{\cal{C}}<T_\infty,X_{\varphi_{\cal{C}}}=x\})\Pbx{\{T_\infty=\infty,X_t\in\cal{C}\enskip \forall t\in[0,\infty)\}}\\
&=\sum_{x\in\cal{C}}\Pb_\gamma(\{\varphi_{\cal{C}}<T_\infty,X_{\varphi_{\cal{C}}}=x\})=\Pbl{\{\varphi_{\cal{C}}<\infty\}},
\end{align*}
where the first equation follows from Exercise~\ref{ex:entdc}, the second from the strong Markov property (Theorem~\ref{thrm:markovprop}), the third from Propositions~\ref{prop:closedct}~and~\ref{prop:recnoexpl},
 and the fourth from Exercise~\ref{ex:entdc}. It then follows from the definition of the skeleton chain in~\eqref{eq:skeletondef} that
\begin{equation}\label{eq:mfeu9awnfeaunfeauy}\Pbl{\{\varphi_{\cal{C}}<\infty\}}=\Pb_\gamma(\{\varphi_{\cal{C}}<\infty,T_\infty=\infty,X_{\delta n}\in\cal{C}\enskip \forall n\geq \varphi_{\cal{C}}/\delta\})\leq\Pb_\gamma(\{\phi_{\cal{C}}^\delta<\infty\}).\end{equation}
Similarly, Proposition~\ref{prop:closedct}'s discrete-time counterpart (Proposition~\ref{prop:closed}) and the discrete-time strong Markov property (Theorem~\ref{thrm:strmkvpath}) imply that
\begin{align*}
\Pb_\gamma(\{\phi^\delta_{\cal{C}}<\infty,X^\delta_n\in\cal{C}\enskip\forall n\geq\phi^\delta_{\cal{C}}\})&=\sum_{x\in\cal{C}}\Pb_\gamma(\{\phi^\delta_{\cal{C}}<\infty,X^\delta_{\phi^\delta_{\cal{C}}}=x,X^\delta_n\in\cal{C}\enskip\forall n\geq\phi^\delta_{\cal{C}}\})\\
&=\sum_{x\in\cal{C}}\Pb_\gamma(\{\phi^\delta_{\cal{C}}<\infty,X^\delta_{\phi^\delta_{\cal{C}}}=x\})\Pb_\gamma(\{X^\delta_n\in\cal{C}\enskip\forall n\geq0\})\\
&=\sum_{x\in\cal{C}}\Pb_\gamma(\{\phi^\delta_{\cal{C}}<\infty,X^\delta_{\phi^\delta_{\cal{C}}}=x\})=\Pb_\gamma(\{\phi^\delta_{\cal{C}}<\infty\}).
\end{align*}
Thus, the definition of the skeleton chain in~\eqref{eq:skeletondef} implies that
$$\Pb_\gamma(\{\phi^\delta_{\cal{C}}<\infty\})=\Pb_\gamma(\{\phi^\delta_{\cal{C}}<\infty\}\cap\{X_{\delta n}\in\cal{C},\enskip \delta n<T_\infty,\enskip\forall n\geq\phi^\delta_{\cal{C}}\})\leq\Pbl{\{\varphi_{\cal{C}}<\infty\}}.$$
Putting the above together with \eqref{eq:mfeu9awnfeaunfeauy} then yields \eqref{eq:fmeia90mfei9fmau9fmeuaf}.

\subsubsection*{A proof of Corollary~\ref{cor:tvtightct}}We do this proof in parts:

$(i)\Rightarrow (iii)$  Taking expectations of \eqref{eq:varcivarrp}, we find that 
\begin{equation}\label{eq:pigmass}\pi_\gamma(\s)=\sum_{i\in\cal{I}}\Pbl{\{\varphi_{\cal{C}_i}<\infty\}\}}\pi_i(\s)=\sum_{i\in\cal{I}}\Pbl{\{\varphi_{\cal{C}_i}<\infty\}\}}=\Pbl{\{\varphi_{\cal{R}_+}<\infty\}}.\end{equation}
Suppose that $(i)$ holds. Downwards monotone convergence, Fatou's lemma, and \eqref{eq:pigmass}  imply that the chain is non-explosive:
\begin{equation}\label{eq:f,80ea-wjtkga0[-43g}\Pbl{\{T_\infty=\infty\}}=\lim_{n\to\infty}\Pbl{\{n<T_\infty\}}=\lim_{n\to\infty}p_n(\s)\geq \pi_\gamma(\s)=\Pbl{\{\varphi_{\cal{R}_+}<\infty\}}=1.\end{equation}
Thus, $(p_{\delta n})_{n\in\n}$ is a sequence of probability distributions, for any given $\delta$ in $(0,\infty)$. For this reason Theorem~\ref{thrm:pointlimsct} and Scheffe's lemma (Lemma~\ref{lem:scheffe}) imply that the sequence converges to $\pi_\gamma$ in total variation. Thus, if we are able to show that 
\begin{equation}\label{eq:nf8wnfay83a4f}t\mapsto \norm{p_t-\pi_{\gamma}}\end{equation}
is a continuous function on $[0,\infty)$, the desired~\eqref{eq:pttv} then follows from Croft's theorem (Corollary~\ref{cor:croft}). For this, it suffices to show that \eqref{eq:nf8wnfay83a4f} is a continuous function on $[0,T]$ for any given $T>0$. Because the total variation distance between two probability distributions is half the $\ell^1$-distance (see~\eqref{eq:tvl1p}), we need to show that
\begin{equation}\label{eq:nf8wnfay83a4f2}t\mapsto \sum_{x\in\s}\mmag{p_t(x)-\pi_{\gamma}(x)}\end{equation}
is a continuous function on $[0,T]$. Let $(\s_r)_{r\in\zp}$ be a sequence of increasing finite sets approaching the state space (i.e.,~$\cup_{r=1}^\infty\s_r=\s$) and note that \eqref{eq:nf8wnfay83a4f2}  is the pointwise limit of the sequence of functions
$$\left\{t\mapsto \sum_{x\in\s_r}\mmag{p_t(x)-\pi_{\gamma}(x)}\right\}_{r\in\zp}.$$
Because the sets are finite and because $t\mapsto p_t(x)$ is a continuous function on $[0,\infty)$ (see the proof of Theorem~\ref{thrm:forward}), arguing that the convergence is uniform over $t\in[0,T]$ completes the proof. However,
\begin{align*}
\sum_{x\not\in\s_r}\mmag{p_t(x)-\pi_{\gamma}(x)}&\leq\sum_{x\not\in\s_r} p_t(x)+\sum_{x\not\in\s_r} \pi_\gamma(x)=\Pbl{\{X_t\not\in\s_r\}}+\pi_\gamma(\s_r^c)\\
&\leq \Pbl{\{\tau_r\leq t\}}+\pi_\gamma(\s_r^c)\leq \Pbl{\{\tau_r\leq T\}}+\pi_\gamma(\s_r^c)\quad\forall t\in[0,T],
\end{align*}
and, the uniform convergence follows from Theorem~\ref{tautin} and \eqref{eq:f,80ea-wjtkga0[-43g}.

$(iii)\Rightarrow (i)$ Suppose that $(iii)$ holds. 
Non-explosiveness implies that $p_n(\s)=\Pbl{\{T_\infty>n\}}=1$ for all positive integers $n$. For this reason, the total variation convergence implies  that the limit  $\pi_\gamma$ is a probability distribution and it follows from \eqref{eq:pigmass} that  $\Pbl{\{\varphi_{\cal{R}_+}<\infty\}}=1$.

$(i)\Rightarrow (ii)$ Suppose that $(i)$ holds and fix any $\varepsilon>0$.
%
%
%
%
%
Because $(iii)$ also holds (see above), we can find a $T$ such that 
\begin{equation}\label{eq:ffneu9anfdysb8a7bdy7abydwa}\norm{p_{t}-\pi_\gamma}\leq \frac{\varepsilon}{2}\quad\forall t\in(T,\infty).\end{equation}
Suppose for now that $t\mapsto p_t$ is a continuous function from $[0,T]$ to the space $\ell^1$ of absolutely summable sequences  indexed by states $x$ in $\s$,
$$\ell^1:=\left\{\rho:\s\to\r:\sum_{x\in\s}\mmag{\rho(x)}<\infty\right\},$$ 
topologised by the total variation norm (here, technically, I'm tacitly using $\rho(A)=\sum_{x\in A}\rho(x)$ to identify $\ell^1$ with the space of signed  measures on $\s$ with finite total variation). Because the Heine-Borel theorem shows that continuous functions between metric spaces (like $[0,\infty)$ and $\ell^1$) are uniformly continuous on compact sets (like $[0,T]$), there exists a $\delta$ such that 
\begin{align}\label{eq:dnw8and7w8andawudaw}\norm{p_{t}-p_{\delta n}}&\leq \frac{\varepsilon}{2}\quad\forall t\in[\delta n,\delta (n+1)]\enskip n=0,1,\dots,\lfloor T/\delta\rfloor-1,\\
\norm{p_{t}-p_{\delta \lfloor T/\delta\rfloor}}&\leq \frac{\varepsilon}{2}\quad \forall t\in[\delta\lfloor T/\delta\rfloor,T].\nonumber\end{align}
Because \eqref{eq:pigmass}--\eqref{eq:f,80ea-wjtkga0[-43g} imply that $\pi_\gamma$ and $p_t$ are probability distributions for all $t$ in $[0,\infty)$, we can find a finite set $F$ large enough that 
$$p_{\delta n}(F)\geq 1-\frac{\varepsilon}{2}\quad\forall n=0,1,\dots,\lfloor T/\delta\rfloor,\qquad\pi_\gamma(F)\geq 1-\frac{\varepsilon}{2},$$
and it follows from~\eqref{eq:dnw8and7w8andawudaw} that
$$p_{t}(F)\geq 1-\varepsilon\quad\forall t\in[0,T].$$
Combining the above with \eqref{eq:ffneu9anfdysb8a7bdy7abydwa}, we obtain
\begin{align*}p_t(F)&\geq 1_{[0,T]}(t)p_t(F)+1_{(T,\infty)}(t)p_t(F)\\
&\geq 1_{[0,T]}(t)(1-\varepsilon)+1_{(T,\infty)}(t)(\pi_\gamma(F)+\norm{p_{t}-\pi_\gamma})\geq 1-\varepsilon\quad\forall t\in[0,\infty).\end{align*}
Because the $\varepsilon$ was arbitrary, we have that $\{p_t:t\in[0,\infty)\}$ is tight.

To complete the proof we need to show that $t\mapsto p_t$ is a continuous function from $[0,\infty)$ to $\ell^1$.  To do so, we need only show that $t\mapsto p_t$ is a continuous function on $[0,T]$ for any given $T>0$. Let $(\s_r)_{r\in\zp}$ be any sequence of increasing finite sets approaching the state space (i.e.,~$\cup_{r=1}^\infty\s_r=\s$) and note that $t\mapsto p_t$ is the pointwise limit of $(t\mapsto 1_{\s_r}p_t)_{r\in\zp}$:
$$p_t(x)=\lim_{r\to\infty}1_{\s_r}(x)p_t(x)\quad\forall x\in\s,\enskip t\in[0,\infty).$$
Because $t\mapsto p_t(x)$ is a continuous function from $[0,T]$ to $[0,1]$ for each $x$ in $\s$ (see the proof of Theorem~\ref{thrm:forward}), $t\mapsto 1_x p_t(x)$ is a continuous function from $[0,T]$ to $\ell^1$ for each $x$ in $\s$. Because the truncations are finite, it follows that each of the functions in the sequence $(t\mapsto 1_{\s_r}p_t)_{r\in\zp}$ is a continuous function from $[0,T]$ to $\ell^1$ and we only need to argue that the convergence is uniform over $t$ in $[0,T]$. However,
\begin{align*}\norm{p_t-1_{\s_r}p_t}&=\sup_{A\subseteq\s}\mmag{\sum_{x\in A}p_t(x)-\sum_{x\in A}1_{\s_r}(x)p_t(x)}=\sup_{A\subseteq \s}\sum_{x\in A\cap\s_r^c}p_t(x)=\sum_{x\in\s_r^c}p_t(x)\\
&=\Pbl{\{t<T_\infty,X_t\not\in\s_r\}}\leq\Pbl{\{\tau_r\leq t\}}\leq\Pbl{\{\tau_r\leq T\}} \quad\forall t\in[0,T],\end{align*}
where $\tau_r$ denotes the exit time from $\s_r$ (c.f.~\eqref{eq:taurexit}). Because these exit times approach $T_\infty$ as $r\to\infty$ (Theorem~\ref{tautin}), the uniform convergence follows from the above and \eqref{eq:f,80ea-wjtkga0[-43g}.

$(ii)\Rightarrow (i)$ Suppose that $(ii)$ holds. Theorem~\ref{thrm:pointlimsct} and \eqref{eq:pigmass} imply that
$$\lim_{t\to\infty}p_{t}(F)\leq \Pbl{\{\varphi_{\cal{R}_+}<\infty\}}$$
for any finite set $F$. Because the chain is non-explosive, $p_t$ has mass one (i.e., $p_t(\s)=1$) for all $t$ in $[0,\infty)$ and the above implies that the time-varying law is tight only if $\Pbl{\{\varphi_{\cal{R}_+}<\infty\}}=1$.

\subsection{Exponential recurrence and convergence*}\label{sec:kendallct}

The aim of this section is to prove the continuous-time analogue of Kendall's Theorem~\ref{thrm:kendall}. It shows that, in the non-explosive case, $p_t(x,x)$ converges exponentially fast if and only if the state $x$ is \emph{exponentially recurrent}\index{exponential convergence and recurrence}. That is, if and only if the tails of the return time distribution to $x$ are light:
\begin{theorem}\label{thrm:kendallct} Let $x$ denote any state belonging to a  positive recurrent class $\cal{C}$ and let $\pi$ be the ergodic distribution associated with $\cal{C}$.
\begin{enumerate}[label=(\roman*),noitemsep] 
\item $p_t(x,x)$ converges geometrically fast to $\pi(x)$: $\mmag{p_t(x,x)-\pi(x)}=\mathcal{O}(e^{-\alpha t})$ for some $\alpha>0$.
\item $x$ is geometrically recurrent: $\Ebx{e^{\beta \varphi_x}}<\infty$ for some $\beta>0$, where $\varphi_x$ denotes the first entrance time to $x$.
\end{enumerate}
\end{theorem}
\subsubsection*{An open question}Are there necessary and sufficient conditions in terms of the return time $\varphi_x$ to $x$ and the initial distribution $\gamma$ for the $x$-entry of the time-varying law to converge exponentially fast? That is, conditions for
$$\mmag{p_t(x)-\pi_\gamma(x)}=\cal{O}(e^{-\alpha t})\text{ to hold for some }\alpha>0,$$
where $\pi_\gamma$ denotes $p_t$'s limit in \eqref{eq:reclims2ct}.
\subsubsection*{A proof of Theorem~\ref{thrm:kendallct}}We use skeleton chains make the jump from  Kendall's Theorem to its continuous-time counterpart (Theorem~\ref{thrm:kendallct}). The bulk of the work we need to do here consists of showing that  a state is exponentially recurrent if and only if it is geometrically recurrent (Section~\ref{sec:kendall}) for at least one skeleton chain:
\begin{lemma}\label{lem:skerettime}For any given $\delta$ in $(0,\infty)$ and $x$ in $\s$, let
$$\phi^\delta_x(\omega):=\inf\{n\in\zp:X^\delta_{n}(\omega)=x\}\quad\forall \omega\in\Omega$$
denote the first entrance time to $x$ of the skeleton chain $X^\delta$ in~\eqref{eq:skeletondef}. The  following statements are equivalent:
\begin{enumerate}[label=(\roman*),noitemsep] 
\item The return time distribution  of $X$ to $x$  has light tails: there exists an $\beta>0$ such that $\Ebx{e^{\beta\varphi_x}}<\infty$.
\item There exists a $\delta$ in $(0,\infty)$ such that the return time distribution of $X^\delta$ to $x$  has light tails: there exists an $\theta>1$ such that $\Ebx{\theta^{\varphi_x^\delta}}<\infty$.
\end{enumerate} 
\end{lemma}
%
%
For the lemma's proof, we require the following fact:
\begin{exercise}\label{ex:soindep}Recall Exercise~\ref{ex:entdc} showing that, for any natural number $k$ and state $x$, the $k$th entrance time $\varphi^k_x$ to $x$ of $X$  is finite if and only the $k$th entrance time $\phi^k_x$ to $x$ of the jump chain is finite in which case $\varphi^k_x=T_{\phi^k_x}$ . For any positive integer $k$, and assuming that the chain starts at $x$, let
$$U_{k}:=\left\{\begin{array}{ll}S_{\phi^{k-1}_x+1}&\text{on }\{\phi^{k-1}_x<\infty\}\\0&\text{on }\{\phi^{k-1}_x=\infty\}\end{array}\right.\quad\text{and}\quad W_{k+1}:=\left\{\begin{array}{ll}\varphi^{k}_x-\varphi^{k-1}_x-S_{\phi^{k-1}_x+1}&\text{on }\{\phi^{k-1}_x<\infty\}\\0&\text{on }\{\phi^{k-1}_x=\infty\}\end{array}\right.$$
denote the time spent in state $x$ on the $(k-1)$th visit and the time elapsed between the $(k-1)$th and $k$th visits, respectively (with the convention that the these times are zero if the $k$th visit does not occur). By adapting the proof of the chain's regenerative property (Theorem~\ref{thrm:rec-iidct}), show that, under $\Pb_x$, the sequence 
$$(U_1,W_1,U_2,W_2,\dots)$$
is independent and that the sojourn times $U_k$ are identically distributed with
$$\Pbx{\{U_k< t\}}=\Pbx{\{U_1< t\}}=\Pbx{\{S_1< t\}}=1-e^{-\lambda(x)t}\quad\forall t\in[0,\infty),\enskip k\in\zp.$$
Hint: 
note that $U_k$ may be re-written as $I_{k}^f$ in \eqref{eq:Ikdefct} with $f:=1_x$ and $W_k$ may be re-written as $I_{k}^f$ with $f:=1_{\s\backslash \{x\}}$.
\end{exercise}

\begin{proof}[Proof of Lemma~\ref{lem:skerettime}]The lemma follows trivially if $x$ is an absorbing state. Assume otherwise. 

$(i)\Rightarrow(ii)$:  Set $\theta:=e^\beta$ (so that $\Ebx{\theta^\varphi_x}=\Ebx{e^{\beta \varphi_x}}<\infty$) and let $(U_k)_{k\in\zp}$ and $(W_k)_{k\in\zp}$ be the sojourn times and inter-visit times defined in Exercise~\ref{ex:soindep}. Throughout this proof fix a $\delta$ in $(0,\infty)$  small enough that 
$$\Pbx{\{S_1< \delta\}}\theta^{\delta}\Ebx{\theta^{\varphi_x}}<1.$$

The skeleton chain $X^\delta$ can skip visits that the continuous-time chain $X$ to $x$ and, so, it might be the case that $\delta\phi^\delta_x>\varphi_x$. To deal with this possibility, let
$$\sigma(\omega):=\inf\{k\in\n:\{n\in\zp:\varphi^k_x(\omega)\leq\delta n<\varphi^k_x(\omega)+U_{k+1}(\omega)\}\neq\emptyset\}\quad\forall\omega\in\Omega$$
denote the visit of $X$ to $x$ during which $X^\delta$ first returns to $x$. The definition of the skeleton chain $X^\delta$ in~\eqref{eq:skeletondef} implies that $\sigma$ is finite if and only if $X^\delta$ returns to $x$ at some point, i.e., if and only if $\phi^\delta_x$ is finite. Because, by our premise, $x$ is a positive recurrent state for $X$ and because $X$ and $X^\delta$ share the same positive recurrent states (c.f.~Theorems~\ref{thrm:posrecchar},~\ref{thrm:statfixedpoint},~and~\ref{thrm:posreccharc}), it follows that $\sigma $ is finite $\Pb_x$-almost surely. Thus, 
$$\delta \phi^\delta_x\leq 1_{\{\sigma<\infty\}}\varphi^{\sigma}_x+\delta,\quad\Pb_x\text{-almost surely}.$$
Given that we may rewrite $\varphi^{k}_x$ as $\sum_{k'=1}^{k}U_{k'}+W_{k'}$ for all $k>0$, it follows that
%
$$\Ebx{\theta^{\delta\phi^\delta_x}}\leq \theta^\delta+\theta^\delta\Ebx{1_{\{0<\sigma<\infty\}}\theta^{\delta+\sum_{k=1}^{\sigma}U_{k}+W_{k}}}=\theta^\delta+\theta^\delta\sum_{n=1}^\infty\Ebx{1_{\{\sigma=n\}}\theta^{\sum_{k=1}^{n}U_{k}+W_{k}}}.$$
If the $n$th sojourn time $U_{n+1}$ is at least $\delta$, then the skeleton chain must necessarily return to $x$ during the continuous-time chain's $n$th visit (recall that $X^\delta$ is obtained by sampling $X$ every $\delta$ units of time~\eqref{eq:skeletondef}) and it must be the case that $\sigma\leq n$. In short, $\{U_{n+1}\geq \delta\}\subseteq\{\sigma\leq n\}$. Because $\{\sigma=n\}\subseteq\{\sigma>n-1\}$, the above implies that
\begin{align*}\Ebx{\theta^{\delta\phi^\delta_x}}&\leq \theta^\delta+\theta^{\delta}\Ebx{\theta^{\varphi_x}}+\theta^{\delta}\sum_{n=2}^\infty\Ebx{1_{\bigcap_{k=1}^{n-1}\{U_{k+1}<\delta\}}\theta^{\sum_{k=1}^{n}U_{k}+W_{k}}}\\
&\leq \theta^\delta+\theta^{\delta}\Ebx{\theta^{\varphi_x}}+\sum_{n=2}^\infty \theta^{\delta n}\Ebx{1_{\bigcap_{k=1}^{n-1}\{U_{k+1}<\delta\}}\theta^{U_{1}+\sum_{k=1}^{n}W_{k}}}\end{align*}
Because the $U_k$s and $W_k$s  are independent and identically distributed (Exercise~\ref{ex:soindep}), we find that
\begin{align*}\Ebx{\theta^{\delta\phi^\delta_x}}&\leq \theta^\delta+\theta^{\delta}\Ebx{\theta^{\varphi_x}}+\sum_{n=2}^\infty\theta^{\delta n}\Pbx{\{S_1<\delta\}}^{n-1}\Ebx{\theta^{S_{1}}}\Ebx{\theta^{W_{1}}}^{n}\\
&\leq\theta^\delta+\theta^{\delta}\Ebx{\theta^{\varphi_x}}+\frac{\Ebx{\theta^{S_{1}}}}{\Pbx{\{S_1<\delta\}}}\sum_{n=2}^\infty(\theta^\delta \Pbx{\{S_1<\delta\}}\Ebx{\theta^{\varphi_x}})^n\\
&=\theta^\delta+\theta^{\delta}\Ebx{\theta^{\varphi_x}}+\frac{\Ebx{\theta^{S_{1}}}\theta^{2\delta}\Pbx{\{S_1<\delta\}}\Ebx{\theta^{\varphi_x}}^2}{1-\theta^\delta \Pbx{\{S_1<\delta\}}\Ebx{\theta^{\varphi_x}}}<\infty\end{align*}
and $(ii)$ follows.

$(ii)\Rightarrow(i)$: The skeleton chain $X^\delta$ having returned to $x$ does not necessarily imply that $X$ has as well: $X$ may not have left $x$ in the first place. To rule out this possibility, let
$$\phi_{\not x}^\delta(\omega):=\inf\{n>0:X^\delta_n(\omega)\neq x\},\quad\tilde{\phi}_{x}^\delta(\omega):=\inf\{n\geq \phi_{\not x}^\delta(\omega):X^\delta_n(\omega)=x\},\quad\forall \omega\in\Omega,$$
denote the first that $X^\delta$ leaves $x$ and the first time that $X^\delta$ returns to $x$ after having first left it, respectively. Because the $X$ must have left and returned to $x$ if the skeleton chain has left and return,  $\delta\tilde{\phi}_{x}^\delta$ is greater than the return time $\varphi_x$ of $X$ and, so, it suffices to show that there exists a $\beta>0$ such that
\begin{equation}\label{eq:mf79eahf7n8eawneawffa}\Ebx{e^{\beta\delta\tilde{\phi}_{x}^{\delta}}}<\infty.\end{equation}
To do so, note that Theorem~\ref{thrm:accesibility} implies that $p_{\delta}(x,x)<1$ (because we are assuming that $x$ is not absorbing) and pick a $0<\beta<\ln{(\theta)}$ close enough to zero that $p_{\delta}(x,x)e^{\beta\delta}<1$. It is then straightforward to check that $\phi^{\delta}_{\not x}$ is $\Pb_x$-almost surely finite and it follows by applying the discrete-time strong Markov property (Theorem~\ref{thrm:strmkvpath}, use Lemma~\ref{lem:pathspmeas}$(ii)$ to check the theorem's measurability requirement), that 
\begin{align}
\Ebx{e^{\beta\delta\tilde{\phi}_{x}^{\delta}}}&=\sum_{z\neq x}\Ebx{1_{\{\phi_{\not x}^{\delta}<\infty,X^{\delta}_{\phi_{\not x}^{\delta}}=z\}}e^{\beta\delta\phi_{\not x}^{\delta}}e^{\beta\delta(\tilde{\phi}_{x}^{\delta}-\phi_{\not x}^{\delta})}}\nonumber\\
&=\sum_{z\neq x}\Ebx{1_{\{\phi_{\not x}^{\delta}<\infty,X^{\delta}_{\phi_{\not x}^{\delta}}=z\}}e^{\beta\delta\phi_{\not x}^{\delta}}}\Ebz{e^{\beta\delta\phi_x^{\delta}}},\label{eq:dnwe8fan8fsdaadsaeny8afn}
\end{align}
where the above sums are taken over all $z$ in the extended state space $\s_E$ (Section~\ref{sec:skeleton}) except for $x$.
But,
\begin{align*}\Ebx{1_{\{\phi_{\not x}^{\delta}<\infty,X^{\delta}_{\phi_{\not x}^{\delta}}=z\}}e^{\beta\delta\phi_{\not x}^{\delta}}}
&=\sum_{n=1}^\infty e^{\beta n\delta}\Pb_x(\{\phi_{\not x}^{\delta}=n,X^{\delta}_{n}=z\})\\
&=\sum_{n=1}^\infty e^{\beta n\delta}\Pb_x(\{X^{\delta}_{0}=x,\dots,X^{\delta}_{n-1}=x,X^{\delta}_{n}=z\})\\
&=\sum_{n=1}^\infty e^{\beta n\delta}p_{\delta}(x,x)^{n-1}p_{\delta}(x,z)=\frac{ e^{\beta\delta}p_{\delta}(x,z)}{1-p_{\delta}(x,x)e^{\beta\delta}}\forall z\neq x.\end{align*}
Plugging the above into~\eqref{eq:dnwe8fan8fsdaadsaeny8afn} and applying the strong Markov property one more time, we have that
\begin{align*}
\Ebx{e^{\beta\delta\tilde{\phi}_{x}^{\delta}}}&=\frac{ e^{\beta\delta}}{1-p_{\delta}(x,x)e^{\beta\delta}}\sum_{z\neq x}p_{\delta}(x,z)\Ebz{e^{\beta\delta\phi_x^{\delta}}}\\
&= \frac{ e^{\beta\delta}}{1-p_{\delta}(x,x)e^{\beta\delta}}\sum_{z\neq x}\Pb_x(\{X_1^{\delta}=z\})\Ebz{e^{\beta\delta\phi^{\delta}_x}}\\
&\leq \frac{ e^{\beta\delta}}{1-p_{\delta}(x,x)e^{\beta\delta}}\Ebx{e^{\beta\delta(\phi_x^{\delta}-1)}}\leq\frac{ e^{\beta\delta}}{1-p_{\delta}(x,x)e^{\beta\delta}}\Ebx{e^{\beta\delta\phi_x^{\delta}}}.
\end{align*}
%
Because we have chosen $\beta$ to be smaller than $\ln(\theta)$, $\Ebx{e^{\beta\delta\phi_x^{\delta}}}\leq\Ebx{\theta^{\phi_x^\delta}}<\infty$  and \eqref{eq:mf79eahf7n8eawneawffa} follows from the above.
\end{proof}

The remainder of Theorem~\ref{thrm:kendallct}'s proof is downhill:
\begin{proof}[Theorem~\ref{thrm:kendallct}]Lemma~\ref{lem:skerettime} shows that $x$ is exponentially recurrent for $X$  if and only if there exists a $\delta$ such $x$ is geometrically recurrent for the skeleton chain $X^\delta$. Kendall's Theorem (Theorem~\ref{thrm:kendall}), \eqref{eq:skeletonoriginal}, and Theorem~\ref{thrm:statfixedpoint} imply that this is the case if and only if
\begin{equation}\label{eq:djw9ahdnyw8bdaw}\mmag{p_{\delta n}(x,x)-\pi(x)}\leq C_\delta\kappa^{-n}\quad\forall n\in\n,\end{equation}
for some $C_\delta\geq0$ and  $\kappa>1$. Because this is clearly the case if there exists some $C\geq0$ and  $\alpha>0$ such that
\begin{equation}\label{eq:ndwu8aondw8uaondu8wnad}\mmag{p_t(x,x)-\pi(x)}\leq Ce^{-\alpha t}\quad\forall t\in[0,\infty),\end{equation}
we only have left to show that the above holds for some $C\geq0$ and $\alpha>0$ if  \eqref{eq:djw9ahdnyw8bdaw} holds for some $C_\delta\geq0$ and $\kappa>1$. In this case, examining the proof of Theorem~\ref{thrm:fostersgeo} (in particular, \eqref{eq:mfwah9w44}--\eqref{eq:mfwah9w43}) and taking supremums over $A\subseteq \s$ in \eqref{eq:mfwah9w4end} we find that, after tweaking $C_\delta$ and $\kappa$, the inequality \eqref{eq:djw9ahdnyw8bdaw} can be strengthened to 
$$\norm{p_{\delta n}(x,\cdot)-\pi}\leq C_\delta\kappa^{-n}\quad\forall n\in\n,$$
where $\norm{\cdot}$ denotes the total variation norm in \eqref{eq:tvnorm}. Now, fix any $t$ and re-write $t$ as $\delta n+s$, where $n$ is an integer and $0\leq s <\delta$. Proposition~\ref{prop:recnoexpl} implies that $p_t(x,\cdot)$, $p_{\delta n}(x,\cdot)$, $p_s(x,\cdot)$ are probability distribution as $x$ is recurrent for $X$. Because the total variation distance between two probability distributions is half the $\ell^1$-distance (see~\eqref{eq:tvl1p}), the semigroup property~\eqref{eq:semigroup} and the equations $\pi P_s=\pi$ (Theorem~\ref{thrm:statfixedpoint}) imply that
\begin{align*}
\norm{p_t(x,\cdot)-\pi}&=\frac{1}{2}\sum_{x'\in\s}\mmag{p_t(x,x')-\pi(x')}=\frac{1}{2}\sum_{x'\in\s}\mmag{\sum_{x''\in\s}(p_{n\delta}(x,x'')-\pi(x''))p_s(x'',x')}\\
& \leq\frac{1}{2} \sum_{x''\in\s}\mmag{p_{n\delta}(x,x'')-\pi(x'')}\left(\sum_{x'\in\s}p_s(x'',x')\right)= ||p_{\delta n}(x,\cdot)-\pi||\leq C_\delta\kappa^{-n}\\
&=C_\delta\kappa^{-(t-s)/\delta}=C_\delta\kappa^{s/\delta}\kappa^{-t/\delta}\leq C_\delta\kappa e^{-(\ln(\kappa)/\delta)t}.
\end{align*}
Given that $\mmag{p_t(x,x)-\pi(x)}\leq\norm{p_t(x,\cdot)-\pi}$, it follows that \eqref{eq:ndwu8aondw8uaondu8wnad} holds with $C:= C_\delta\kappa$ and $\alpha:=\ln(\kappa)/\delta$.
\end{proof}

%
%
%
%

\subsection{Foster-Lyapunov criteria I: the criterion for regularity*}\label{sec:flreg}
The regularity of the rate matrix $Q$ (Definition~\ref{def:regular}) plays an important role in many aspects of the theory. For instance, it implies  that the only solutions to the forward and backward equations are the transition probabilities (Section~\ref{sec:forwardbackward}), guarantees that all Markov chains with rate matrix $Q$ and initial distribution $\gamma$ have the same path law (Section~\ref{sec:otherchains}),  ensures that the stationary distributions are characterised by the equations $\pi Q=0$ (Section~\ref{sec:statct}), and  is a basic requirement for the chain to be stable (Sections~\ref{sec:timeave}--\ref{sec:limsct}). For these reasons, it is important to have a practical test for it:\index{Foster-Lyapunov criteria}
\begin{theorem}\label{thrm:flreg} The rate matrix $Q$ in~\eqref{eq:qmatrix} is regular if and only if there exists a non-negative real-valued norm-like (Definition~\ref{def:normlike}) function $v$ on $\s$ and a constant $c$ in $\r$ such that
\begin{equation}\label{eq:flreg1}Qv(x)\leq cv(x)\quad\forall x\in\s.\end{equation}
In this case, the time-varying law $p_t$ in~\eqref{eq:ctlawdef} satisfies 
\begin{equation}\label{eq:flreg2}p_t(v)\leq \gamma(v)e^{ct}\quad\forall t\in[0,\infty)\end{equation}
where $\gamma$ denotes the initial distribution.
\end{theorem}

A practically useful  fact  is that the existence of a $(v,c)$ satisfying Theorem~\ref{thrm:flreg}'s premise is equivalent to a seemingly weaker condition: the existence of a real-valued norm-like function $u$ on $\s$ and real numbers $c_u$ and $d$ such that
$$Qu(x)\leq c_uu(x)+d\quad\forall x\in\s.$$
To see this, note that \eqref{eq:qmatrix} implies that  $Qe=0$ for any constant function $e=(e)_{x\in\s}$. Thus, the above inequality holds   if we replace $u$ with $u-u_m$ and $d$ with $d+c_uu_m$ where $u_m:=inf_{x\in\s}u(x)$ (this number is finite because $u$ is norm-like):
$$Q(u-u_m)(x)=Qu(x)-Qu_m=Qu(x)\leq c_uu(x)+d=c_u(u(x)-u_m)+d+c_uu_m.$$
In other words, we can assume without loss of generality that $u$ is non-negative. If $d\leq0$, then \eqref{eq:flreg1} holds with $v:=u$ and $c:=c_u$. Otherwise, if $c_u\leq 0$,  \eqref{eq:flreg1} holds with $v:=u+d$ and $c:=1$ and, if $c_u>0$, it instead holds with $v:=u+d/c_u$ and $c:=c_u$.

\subsubsection*{A proof of Theorem~\ref{thrm:flreg}}This is a bit delicate and we do it in steps. We start with \eqref{eq:flreg2}:

\begin{proof}[Proof of \eqref{eq:flreg2}]Pick any increasing sequence $(\s_r)_{r\in\zp}$ of finite subsets of the state space satisfying $\cup_{r=1}^\infty\s_r=\s$ and let $(\tau_r)_{r\in\zp}$ denote the corresponding sequence of exit times~\eqref{eq:taurexit}. Because these  are jump-time-valued $(\cal{F}_t)_{t\geq0}$-stopping times (Proposition~\ref{prop:hitdc}), Dynkin's formula (Theorem~\ref{Dynkin}, with $\eta:=\tau_r$, $g(t):=e^{-ct}$, and $f:=v$) shows that 
$$\Ebl{e^{-ct\wedge\tau_r\wedge T_n}v(X_{t\wedge\tau_r\wedge T_n})}=\gamma(v)+\Ebl{\int_0^{t\wedge\tau_r\wedge T_n}e^{-cs}Qv(X_s)-ce^{-cs}v(X_s)ds},$$
for all $t$ in $[0,\infty)$ and $n,r$ in $\zp$. Inequality \eqref{eq:flreg1} then shows that 
$$\Ebl{e^{-ct\wedge\tau_r\wedge T_n}v(X_{t\wedge\tau_r\wedge T_n})}\leq \gamma(v)\quad\forall t\in[0,\infty),\enskip n,r\in\zp.$$
As we will show later on, the chain is regular (in particular, $\Pbl{\{T_\infty=\infty\}}=1$). For this reason, Theorem~\ref{tautin} and Fatou's lemma imply that
\begin{align*}e^{-ct}p_t(v)=\Ebl{e^{-ct}V(X_t)1_{\{t<T_\infty\}}}&\leq\liminf_{r\to\infty}\liminf_{n\to\infty}\Ebl{e^{-ct\wedge\tau_r\wedge T_n}v(X_{t\wedge\tau_r\wedge T_n})1_{\{t<T_\infty\}}}\\
&=\liminf_{r\to\infty}\liminf_{n\to\infty}\Ebl{e^{-ct\wedge\tau_r\wedge T_n}v(X_{t\wedge\tau_r\wedge T_n})}\leq\gamma(v)\end{align*}
for all $t$ in $[0,\infty)$. Multiplying through by $e^{ct}$ then yields \eqref{eq:flreg2}.
\end{proof}

For the remainder of this proof, we follow the steps taken in \citep{Spieksma2015}. To do so, we require a discrete-time chain $Y^\alpha=(Y^\alpha_n)_{n\in\n}$ known as the \emph{$\alpha$-jump chain}. It takes values in the \emph{extended state space} $\s_E:=\s\cup\{\Delta\}$ obtained by appending an extra state $\Delta\not\in\s$ to the state space $\s$. The chain is the defined $Y^\alpha$ by running Algorithm~\ref{dtmcalg} in Section~\ref{sec:dtdef} using the one-step matrix $P^\alpha=(p^\alpha(x,y))_{x,y\in\s_E}$ on $\s_E$ given by
\begin{equation}\label{eq:pa}p^\alpha(x,y):=\left\{\begin{array}{ll}\dfrac{\strut q(x,y)}{q(x)+\alpha}&\text{if }x,y\in\s,\enskip x\neq y\\\dfrac{\strut \alpha}{q(x)+\alpha}&\text{if }x\in\s, y=\Delta\\\strut  p(y)&\text{if }x=\Delta,y\in\s\end{array}\right.\quad\forall x,y\in\s_E,\end{equation}
where $\alpha>0$ and the \emph{restart distribution} $(p(x))_{x\in\s}$ is any probability distribution on $\s$ satisfying
\begin{equation}\label{eq:paa}p(x)>0\quad\forall x\in\s.\end{equation}

This $\alpha$-jump chain $Y^\alpha$ behaves as follows: if $Y^\alpha$ lies at a state $x$ within the original state space $\s$, it jumps to $\Delta$ with probability $\alpha/(q(x)+\alpha)$. Otherwise, $Y^\alpha$ updates it states as the jump chain $Y$ does: it samples $p(x,\cdot)$ in \eqref{eq:jumpmatrix}. Whenever $Y^\alpha$ leaves $\s$, it spends a single step in $\Delta$ and then returns to $\s$ by sampling a state from the restart distribution $(p(x))_{x\in\s}$. Because $\Delta$ is accessible from all states in $\s$ and \eqref{eq:paa} implies that every state in $\s$ is accessible from $\Delta$,  $Y^\alpha$ is irreducible. For this reason, applying the necessary and sufficient criterion for recurrence of irreducible discrete-time chains yields the following:
\begin{lemma} \label{lem:dn78an3a83a}For any $\alpha>0$, the $\alpha$-jump chain $Y^\alpha$ is recurrent if and only if there exists a norm-like function $v:\s\to[0,\infty)$ satisfying \eqref{eq:flreg1} with $c:=\alpha$.
\end{lemma}

\begin{proof}Suppose that $Y^\alpha$ is recurrent. Because $Y^\alpha$ is irreducible, Corollary~\ref{cor:phixphic} and Lemma~\ref{lem:fosrecfinal} (with $F:=\Delta$) tell us that there exists a non-negative norm-like function $u:\s_E\to[0,\infty)$ satisfying 
\begin{equation}\label{eq:nf8a7n387hfa783a}P^\alpha u(x)\leq u(x)\quad\forall x\in\s.\end{equation}
Given that $u(\Delta)\geq 0$, the above implies that
$$\sum_{y\neq x}\frac{q(x,y)u(y)}{q(x)+\alpha}\leq \sum_{y\neq x}\frac{q(x,y)u(y)}{q(x)+\alpha}+\frac{\alpha u(\Delta)}{q(x)+\alpha}=P^\alpha u(x)\leq u(x)\quad\forall x\in\s,$$
where, in the above, we are summing over all $y$ in $\s$ except $x$. Multiplying through by $q(x)+\alpha$ and re-arranging we find that $v:=(u(x))_{x\in\s}$ is  a non-negative real-valued norm-like function on $\s$ satisfying \eqref{eq:flreg1} with $c:=\alpha$.

Similarly, if $v:\s\to[0,\infty)$ is norm-like and satisfies \eqref{eq:flreg2}, then
$$u(x):=v(x)\quad\forall x\in\s,\quad u(\Delta):=0,$$
is a non-negative real-valued norm-like function on $\s_E$ satisfying \eqref{eq:nf8a7n387hfa783a} and Theorem~\ref{thrm:lyarec} implies that $Y^\alpha$ is recurrent.
\end{proof}

We now need to show that $Y^\alpha$ is recurrent for some $\alpha>0$ if and only if $Q$ is regular. Because $Y^\alpha$ is irreducible, Corollary~\ref{cor:phixphic} implies that $Y^\alpha$ is recurrent if and only if the state $\Delta$ is recurrent: $\Pb_{\Delta}(\{\phi_\Delta<\infty\})=1$, where $\phi_\Delta$ denotes the first entrance time to $\Delta$ of $Y^\alpha$  (defined by replacing $X$ with $Y^\alpha$ in Definition~\ref{def:entrance}). Because Proposition~\ref{prop:hitprobeqs} shows that
$$\Pb_{\Delta}(\{\phi_\Delta<\infty\})=p^\alpha(\Delta,\Delta)+\sum_{x\in\s}p^\alpha(\Delta,x)\Pbx{\{\phi_\Delta<\infty\}}=\sum_{x\in\s}p(x)\Pbx{\{\phi_\Delta<\infty\}},$$
our assumption~\eqref{eq:paa} ensures that $\Delta$ is recurrent if and only if 
\begin{equation}\label{eq:je79a0fna3uia}\Pbx{\{\phi_\Delta<\infty\}}=1\quad\forall x\in\s.\end{equation}
Applying \eqref{eq:nstepthe}, we have that
\begin{align}\Pbx{\{\phi_\Delta<\infty\}}&=\sum_{n=0}^\infty\Pbx{\{\phi_\Delta=n+1\}}=\sum_{n=0}^\infty\Pbx{\{Y^\alpha_1\in\s,\dots,Y^\alpha_{n}\in\s,Y^\alpha_{n+1}=\Delta\}}\nonumber\\
&=\sum_{n=0}^\infty\sum_{x_1\in\s}\dots\sum_{x_n\in\s}p^\alpha(x,x_1)\dots p^\alpha(x_{n-1},x_n)p^\alpha(x_n,\Delta)\nonumber\\
&=\sum_{n=0}^\infty\sum_{y\in\s} \prescript{}{\Delta}p^\alpha_n(x,y)p^\alpha(y,\Delta)=\sum_{n=0}^\infty\sum_{y\in\s} \frac{\alpha\prescript{}{\Delta}p^\alpha_n(x,y)}{q(y)+\alpha},\label{eq:je79a0fna3uia1}
\end{align}
where $\prescript{}{\Delta}P^\alpha_n=(\prescript{}{\Delta}p^\alpha_n(x,y))_{x,y\in\s}$ denotes the $n$th power of the \emph{taboo matrix} $\prescript{}{\Delta}P^\alpha=(p^\alpha(x,y))_{x,y\in\s}$ obtained by removing the $\Delta$-row and column from $P^\alpha$.
%

To finish the proof, we are only missing one final ingredient: the \emph{$\alpha$-resolvent matrix} $R^\alpha=(r^\alpha(x,y))_{x,y\in\s}$ defined by
$$r^\alpha(x,y):=\int_0^\infty e^{-\alpha t}p_t(x,y)dt\quad\forall x,y\in\s.$$
Tonelli's theorem implies that
$$\sum_{y\in\s}r^\alpha(x,y)=\int_0^\infty e^{-\alpha t}\sum_{y\in\s}p_t(x,y)dt\quad\forall x\in\s,$$
and it follows from Proposition~\ref{prop:nonexptimevar} that $Q$ is regular if and only if $\alpha R^\alpha$ is a stochastic matrix,
$$\alpha\sum_{y\in\s}r^\alpha(x,y)=\alpha\frac{1}{\alpha}=1\quad\forall x\in\s,$$
for any and all $\alpha>0$. As we will show below,
\begin{equation}\label{eq:fah90h473ahf74}r^\alpha(x,y)=\sum_{n=0}^\infty\frac{\prescript{}{\Delta}p_n^\alpha(x,y)}{q(y)+\alpha}\quad\forall x,y\in\s,\enskip\alpha>0.\end{equation}
Because \eqref{eq:je79a0fna3uia}--\eqref{eq:fah90h473ahf74} imply that $\alpha R^\alpha$ is stochastic (equivalently, $Q$ is regular) if and only if $Y^\alpha$ is recurrent, Theorem~\ref{thrm:flreg} follows from Lemma~\ref{lem:dn78an3a83a}.
\begin{lemma} Equation \eqref{eq:fah90h473ahf74} holds.
\end{lemma}
\begin{proof}The equation holds trivially if $x$ is an absorbing state. Suppose otherwise, fix $\alpha>0$, and let
$$s_n^\alpha(x,y):=\sum_{m=0}^n\frac{\prescript{}{\Delta}p_m^\alpha(x,y)}{q(y)+\alpha}\quad\forall x,y\in\s,\enskip n\in\n,$$
and
$$r^\alpha_n(x,y):=\int_0^\infty e^{-\alpha t}p_t^n(x,y)dt\quad\forall x,y\in\s,\enskip n\in\n,$$
where $p^n_t(x,y)$ as in \eqref{eq:pnt}. Because  monotone convergence implies that
$$\lim_{n\to\infty}s_n^\alpha(x,y)=\sum_{m=0}^\infty \frac{\prescript{}{\Delta}p^\alpha_m(x,y)}{q(y)+\alpha},\qquad \lim_{n\to\infty}r^\alpha_n(x,y)=r^\alpha(x,y),\quad\forall x,y\in\s,$$
it suffices to show that
\begin{equation}\label{eq:fj94nmuafnaungua}r^\alpha_n(x,y)=s_n^\alpha(x,y)\quad\forall x,y\in\s,\enskip n\in\n.
\end{equation}
Multiplying the backward iterated recursion \eqref{eq:bir} by $e^{-\alpha t}$, integrating over $t$, and applying Tonelli's theorem yields the Laplace transform version of this recursion (note that $x$ is not absorbing and so $\lambda(x)=q(x)$ and $\lambda(x)p(x,y)=q(x,y)$ for all $y$ in $\s$):
$$r^\alpha_{n+1}(x,y)=\frac{1}{q(x)+\alpha}\left(1_x(y)+\sum_{z\neq x}q(x,z)r^\alpha_{n}(z,y)\right)\quad\forall y\in\s,\enskip n\in\n.$$
Thus, if \eqref{eq:fj94nmuafnaungua} holds for some $n$ in $\n$,
\begin{align*}r^\alpha_{n+1}(x,y)&=\frac{1_x(y)}{q(y)+\alpha}+\sum_{m=0}^{n}\sum_{z\neq x}\frac{q(x,z)}{q(x)+\alpha}\frac{\prescript{}{\Delta}p_m^\alpha(z,y)}{q(y)+\alpha}=\frac{\prescript{}{\Delta}p_{0}^\alpha(x,y)}{q(y)+\alpha}+\sum_{m=0}^{n}\frac{\prescript{}{\Delta}p_{m+1}^\alpha(z,y)}{q(y)+\alpha}\\
&=s^\alpha_{n+1}(x,y)\quad\forall x,y\in\s.\end{align*}
Because \eqref{eq:fm7824h8a7bgn789a49a3j8g4aq00}--\eqref{eq:fm7824h8a7bgn789a49a3j8g4aq0} imply that
$$r^\alpha_0(x,y)=\int_0^\infty e^{-\alpha t}p_t^0(x,y)dt=\int_0^\infty e^{-\alpha t}1_x(y)e^{-q(x)t}dt=\frac{1_x(y)}{q(y)+\alpha}=s_0^\alpha(x,y)\quad\forall x,y\in\s,$$
the result follows by induction.
\end{proof}

%

\subsubsection*{Notes and references} The theorem's sufficiency was first shown in \citep{Chen1986}. Other early proofs can be found in \citep{Anderson1991} and \citep{Meyn1993b}. 
The necessity was only shown recently in \citep{Spieksma2015}.  

\subsection{Geometric trial arguments*}\label{sec:geometrictrialct}

In the proofs of the Foster-Lyapunov criteria presented in the ensuing sections, we require the continuous-time versions of the geometric trials arguments in Section~\ref{sec:geometrictrial}. Just as in the discrete-time case, these arguments allow us to morph information regarding the chain's visits to a finite set $F$ into information on its visits to any set $B$ that is accessible from $F$ (Definition~\ref{def:accesiblect}). The intuition is the same: because $B$ is accessible from $F$  and $F$ is finite, the chance that a visit to $F$ results in a visit to $B$ may be bounded below by a constant independent of which state in $F$ the chain visits. The existence of this constant implies that the probability that the chain has not yet visited $B$ decays  geometrically with the  number of visits to $F$:\index{geometric trial arguments}
%
\begin{lemma}[Geometric trials property]\label{lem:AFtail1ct} If $F$ is finite and $F\to B$ for a second set $B$, then there exists constants $m$ and $\varepsilon>0$ independent of the initial distribution $\gamma$ such that
$$\Pbl{\{\varphi^{nm}_F<\varphi_B\}}\leq (1-\varepsilon)^{n-1}\quad\forall n\in\zp,$$
where $ \varphi_B$ denotes the first entrance time to $B$ and $\varphi^k_F$ the $k$th entrance time to $F$ (Definition~\ref{def:entrancect}).
%
\end{lemma}
\begin{proof}The event $\{\varphi^{k}_F<\varphi_B\}$ that the continuous-time chain  enters any $F$ for at least $k$ times before entering $B$ equals the event $\{\phi^{k}_F<\phi_B\}$ that the jump-chain  does (this a consequence of Exercise~\ref{ex:entdc}). For this reason, the lemma follows immediately from its discrete-time counterpart (Lemma~\ref{lem:AFtail1}).\end{proof}
Lemma~\ref{lem:AFtail1ct} lays the foundation for the following three geometric-trials-type arguments that  we use in the following sections to turn information on the chain's visits to $F$ into information on its visits to $B$:\index{geometric trial arguments}
\begin{lemma}[Geometric trials arguments]\label{lem:geotrialargct} Given a finite subset $F$ of $\s$, suppose that $F\to B$ for some subset $B$ and let $\varphi_F$ and $\varphi_B$ denote the first entrance time to $F$ and $B$~(Definition~\ref{def:entrance}).
\begin{enumerate}[label=(\roman*),noitemsep] 
\item If $\Pbl{\{\varphi_F<\infty\}}=1$ and $\Pbx{\{\varphi_F<\infty\}}=1$ for all $x$ in $F$, then $\Pb_\gamma(\{\varphi_B<\infty\})=1$.
\item If $\Ebl{\varphi_F}<\infty$ and  $\Ebx{\varphi_F}<\infty$ for all $x$ in $F$, then $\Ebl{\varphi_B}<\infty$.
\item If there exists a constant $\alpha>1$ such that $\Eb_\gamma[e^{\alpha\varphi_F}]<\infty$ and  $\Eb_x[e^{\alpha\varphi_F}]<\infty$ for all $x$ in $F$, then there exists second constant $\beta>1$ such that $\Eb_\gamma[e^{\beta\varphi_A}]<\infty$.
\end{enumerate}
\end{lemma}
Because Exercise~\ref{ex:entdc} implies that
$$\{\varphi_F<\infty\}=\{\phi_F<\infty\},\quad \{\varphi_B<\infty\}=\{\phi_B<\infty\},$$
where $\phi_F$ and $\phi_B$ denote the first entrance times to $F$ and $B$ of the jump chain, Lemma~\ref{lem:geotrialargct}$(i)$ follows directly its discrete-time counterpart (Lemma~\ref{lem:geotrialarg}$(i)$). For Lemma~\ref{lem:geotrialargct}$(ii)$--$(iii)$, we need to emulate the proofs of their discrete-time counterparts$\dots$
\begin{exercise}Following the steps we took in the proof of Lemma~\ref{lem:abkl}, show that, for any subsets $A$ and $B$ of the state space, constant $\alpha>0$, and natural numbers $k$ and $l$,
\begin{align*}
&\Pb_\gamma\left(\left\{\varphi^{k+l}_A<\varphi_B\right\}\right)=\sum_{x\in A}\Pbl{\left\{\varphi_A^k<\varphi_B,X_{\varphi_A^k}=x\right\}}\Pbx{\left\{\varphi^{l}_A<\varphi_B\right\}},\\
&\Ebl{1_{\{\varphi_A^{k}<\varphi_B\}}(\phi_A^{k+l}-\phi_A^{k})}=\sum_{x\in A}\Pbl{\left\{\varphi_A^{k}<\varphi_B,X_{\varphi_A^k}=x\right\}}\Ebx{\varphi_A^{l}},\\
&\Ebl{1_{\{\varphi_A^{k}<\varphi_B\}}e^{\alpha\varphi_A^{k+l}}}=\sum_{x\in A}\Ebl{1_{\{\varphi_A^{k}<\varphi_B,X_{\varphi_A^k}=x\}}e^{\alpha\varphi_A^{k}}}\Ebx{e^{\alpha\varphi_A^{l}}},
\end{align*}
where $\varphi_B$ denotes the first entrance time to $B$ and $\varphi^k_A$ the $k$th entrance time to $A$ (Definition~\ref{def:entrancect}). To do so, replace  Lemma~\ref{lem:2stop}, Theorem~\ref{thrm:strmkvpath}, Exercise~\ref{ex:entstop}, and $G^j_S$ therein with Lemma~\ref{lem:etatheta}, Theorem~\ref{thrm:markovprop}, Exercise~\ref{ex:entdc}, and $G$ in Lemma~\ref{lem:pathspmeasct}$(iii)$ (with $A:=S$, $k:=j$, and $f:=1$). Next, prove Lemma~\ref{lem:geotrialargct}$(ii)$--$(iii)$ by using the above equations and tweaking the proof of Lemma~\ref{lem:geotrialarg}$(ii)$--$(iii)$.
\end{exercise}

\subsection[Foster-Lyapunov criterion for positive recurrence*]{Foster-Lyapunov criteria II: the criterion for positive recurrence*}\label{sec:fostersct}

The continuous-time counterpart of Foster's theorem is often proves indispensable for establishing positive Tweedie recurrence in practice:\index{Foster-Lyapunov criteria}
\begin{theorem}[The criterion for positive recurrence]\label{thrm:fostersct}If the rate matrix is regular and there exists a non-negative real-valued function $v$ on $\s$, a finite subset $F$ of $\s$, and a constant $b$ in $\r$ such that 
\begin{equation}\label{eq:fostersct}Qv(x)\leq-1+b1_{F}(x)\quad \forall x\in\s.\end{equation}
then the chain is positive Tweedie recurrent and has a finite number of closed communicating classes. Conversely, if the chain is positive Tweedie recurrent and there are no transient states and only a finite number of closed communicating classes, then there exists $(v,F,b)$ satisfying \eqref{eq:fostersct} with $v$ real-valued and non-negative, $F$ finite, and $b$ real.
\end{theorem}
The criterion is a consequence of Theorem \ref{thrm:fosrecct} further down which shows that, for a fixed $F$, the existence of such a $v$ and $b$ satisfying~\eqref{eq:fostersct} is equivalent to $F$ having a finite mean return time for all deterministic starting locations.


The importance of a finite number of closed communicating classes is easy to see: the set $F$ must contain at least one state from each closed communicating class. Otherwise, the chain would not be able to reach $F$ whenever it starts inside of a closed communicating class that does not intersect with $F$ and the return time to $F$ would be infinite. As an example, the chain with one-step matrix $Q$ being the matrix of zeros is trivially positive Tweedie recurrent, however $Qv(x)=0\geq-1$ for all states $x$ and functions $v$. Consequently, if the state space is infinite, then \eqref{eq:fostersct} will never be satisfied for for a finite $F$.

To see the importance of the no transient states requirement, consider the continuous-time analogue of Example~\ref{ex:recnofos}:
%
%
%
\begin{example}Consider a chain with state space $\n$ and rate matrix $Q=(q(x,y))_{x,y\in\n}$ defined by
$$q(0,y)=0\quad\forall y\geq0,\qquad q(x,y)=\frac{1_{x-1}(y)}{2}-1_x(y)+\frac{1_{x+1}(y)}{2}\quad\forall x>0,\enskip y\geq0.$$
In other words, a chain $X$ that waits a unit mean exponential amount of time at its starting state $x$, jumps to $y=x-1$ with $50\%$ probability and to $y=x+1$ with the remaining $50\%$ probability, waits a unit mean exponential amount of time at $y$, jumps to $z=y-1$ with $50\%$ probability and to $z=y+1$ with the remaining $50\%$ probability, $\dots$, up until the moment it hits $0$ where it remains for all time. In other words, $X$ is just like the gambler's ruin chain $Y$ (Section~\ref{sec:gamblers}) except that the time waited between jumps is exponentially distributed. Indeed, $X$'s jump chain is $Y$ which is positive Harris recurrent (Example~\ref{ex:recnofos}).  
 Because the waiting times of $X$ all have unit means, Exercise~\ref{ex:entdc} then implies that $X$ and $Y$ have the same mean return times. Thus, $X$ is also positive Harris recurrent. However, inequality~\eqref{eq:fostersct} implies that
$$\frac{1}{2}v(x-1)+\frac{1}{2}v(x+1)\leq v(x)-1\quad\forall x\not\in F.$$
In Example~\ref{ex:recnofos}, we showed that any function $v$ satisfying the above must tend to $-\infty$ as its argument approaches $\infty$. Thus,  even though the chain is positive Harris recurrent there do not exist any $v$, $F$, and $b$ satisfying the premise of Theorem~\ref{thrm:fostersct}.
\end{example}

\subsubsection*{A proof of Theorem~\ref{thrm:fostersct}}Given any $F$, the existence of such a $v$ and $b$ satisfying~\eqref{eq:fostersct} is equivalent to $F$ having a finite mean return time for every deterministic initial condition:
\begin{theorem}\label{thrm:fosrecct}Suppose that the rate matrix is regular. Given any finite set $F$, there exists $v:\s\to[0,\infty)$ and $b$ in $\r$ satisfying \eqref{eq:fostersct} if and only if $\Ebx{\varphi_F}<\infty$ for all $x$ in $\s$. In this case, $\Ebl{\varphi_F}<\infty$ for all initial distributions $\gamma$ satisfying $\gamma(v)<\infty$.\end{theorem}
The key to proving the above is the  following lemma:
\begin{lemma}[\cite{Tweedie1981}]\label{lem:fosminct}Let $F$ be a subset of the state space and $\alpha$ be a real number satisfying $\alpha <q(x)$ for all $x$ outside $F$. The minimal (possibly infinite-valued) non-negative solution $v:=(v(x))_{x\in\s}$ to the inequality
\begin{align}
\label{eq:dlyain1ct}Qv(x)&\leq-\alpha v(x)-1\qquad\forall x\not\in F
\end{align}
is given by $u(x)=0$ for all $x$ in $F$ and
\begin{equation}\label{eq:uminct}u(x)=\int_0^\infty \Pbx{\{\varphi_F>t,t<T_\infty\}}e^{\alpha t}dt\qquad\forall x\not\in F.\end{equation}
\end{lemma}

While, $u$ in \eqref{eq:uminct} may come across as rather esoteric at first glance, it is not difficult to re-write it in terms that make the connection with Theorem~\ref{thrm:fosrecct} obvious. In particular, if $\alpha=0$, then Tonelli's theorem implies that
\begin{align}u(x)&=\int_0^\infty\Pbx{\{\varphi_F>t,t<T_\infty\}}dt=\int_0^\infty\Pbx{\{t<\varphi_F\wedge T_\infty\}}dt=\Ebx{\int_0^\infty1_{[0,\varphi_F\wedge T_\infty)}(t)dt}\nonumber\\
&=\Ebx{\int_0^{\varphi_F\wedge T_\infty}1dt}=\Ebx{\varphi_F\wedge T_\infty}\quad\forall x\not\in F.\label{eq:bdw6a7bd67aww}\end{align}
Thus, in the regular case, $u(x)$ is the mean entrance time to $F$ for any $x$ outside $F$: a fact that will be key for the proof of Lemma~\ref{lem:fosminct}. We will also require the following generalisations of the FIR~\eqref{eq:fir} and BIR~\eqref{eq:bir}: for  any  subset $F$  of the state space and positive integer $n$,
\begin{align}
\label{eq:firgen}f_t^n(x,y)&=1_x(y)e^{-\lambda(x)t}+\int_0^t\sum_{z\not\in F}f_s^{n-1}(x,z)\lambda(z)p(z,y)e^{-\lambda(y)(t-s)}ds\quad\forall x,y\not\in F,\enskip t\geq0,\\
\label{eq:birgen}f_t^n(x,y)&=1_x(y)e^{-\lambda(y)t}+\int_0^t\lambda(x)e^{-\lambda(x)(t-s)}\sum_{z\not\in F}p(x,z)f_s^{n-1}(z,y)ds\quad\forall x,y\not\in F,\enskip t\geq0,
\end{align}
where 
\begin{equation}\label{eq:ftndef}f_t^n(x,y):=\Pbx{\{\varphi_F>t,X_t=y,t<T_{n+1}\}}\quad\forall x,y\not\in F,\enskip t\geq0,\enskip n\in\n.\end{equation}
\begin{exercise}Prove \eqref{eq:firgen} by  tweaking the proof of Lemma~\ref{lem:forwardweak} and use \eqref{eq:firgen} to adapt the proof of  Lemma~\ref{lem:bir} and obtain \eqref{eq:birgen}. Hint: for any $y$ outside $F$, we can re-write the event 
$$\{Y_m=y,T_m\leq t<T_{m+1},\varphi_F>t\}$$
as 
$$\{Y_m=y,T_m\leq t<T_{m+1},Y_1\not\in F,\dots, Y_{m-1}\not\in F\}.$$
\end{exercise}

\begin{proof}[Proof of Lemma~\ref{lem:fosminct}]
Notice that \eqref{eq:dlyain1ct} does not have any non-negative solutions if there exists an absorbing state outside of $F$. Suppose otherwise and note that the generalised BIR \eqref{eq:birgen} reduces to
\begin{equation}\label{eq:birgen2}f_t^n(x,y)=1_x(y)e^{-q(y)t}+\int_0^t q(x)e^{-q(x)(t-s)}\sum_{z\not\in F}p(x,z)f_s^{n-1}(z,y)ds\quad\forall x,y\not\in F,\enskip  t\in[0,\infty),\end{equation}
for any positive integer $n$. 
%
Let 
$$u_n(x):=\left\{\begin{array}{ll}\int_0^\infty\sum_{y\not\in F}f_t^n(x,y)e^{\alpha t}dt&\text{if }x\not\in F\\0&\text{if }x\in F\end{array}\right.\quad\forall x\in\s,\enskip n\in\n,$$
and note that 
$$\lim_{n\to\infty}u_n(x)=u(x)\quad\forall x\in\s$$
by monotone convergence. However, applying \eqref{eq:birgen2} and Tonelli's Theorem, we have that
\begin{align}u_n(x)&=\int_0^\infty \sum_{y\not\in F}f_t^n(x,y)e^{\alpha t}dt\nonumber\\
&= \int_0^\infty e^{(\alpha-q(x))t}dt+\int_0^\infty \left(\int_s^\infty e^{(\alpha-q(x))t}dt\right)e^{q(x)s}\sum_{z\not\in F}q(x)p(x,z)\left(\sum_{y\not\in F}f_s^{n-1}(z,y)\right)ds\nonumber\\
&=\frac{1}{q(x)-\alpha}+\frac{1}{q(x)-\alpha}\sum_{z\not\in F}(q(x)p(x,z))\left(\int_0^\infty f_s^{n-1}(z,y)e^{\alpha s}ds\right)\nonumber\\
&=\frac{1+\sum_{z\neq x}q(x,z)u_{n-1}(z)}{q(x)-\alpha}\quad\forall x\not\in F,\enskip n\in\zp.\label{eq:fne8awnfyea}\end{align}
Taking the limit $n\to\infty$ and re-arranging shows that $u$ satisfies \eqref{eq:dlyain1ct}. 

To prove the minimality of $u$, we use induction. In particular, let $v=(v(x))_{x\in\s}$ be any other non-negative solution of \eqref{eq:dlyain1ct}. By definition,
$$v(x)\geq 0=u(x)\quad\forall x\in F.$$
Re-arranging \eqref{eq:dlyain1ct}, we find that $(q(x)-\alpha)v(x)\geq 1+\sum_{z\neq x}q(x,z)v(z)$ for all $x\not\in F$. Consequently,
$$v(x)\geq\frac{1}{q(x)-\alpha}=\int_0^\infty \sum_{y\not\in F}1_x(y)e^{-(q(x)-\alpha)t}dt=u_0(x)\quad\forall x\not\in F.$$
Moreover, if $v(x)\geq u_{n-1}(x)$ for all $x$ in $\s$,  it follows from \eqref{eq:dlyain1ct} and \eqref{eq:fne8awnfyea} that
$$u_n(x)\leq \frac{1+\sum_{z\neq x}q(x,z)v(z)}{q(x)-\alpha}\leq v(x)\quad\forall x\not\in F.$$
By induction, we have that $u_n(x)\geq v(x)$ for all $x$ in $\s$ and $n$ in $\n$. Taking the limit $n\to\infty$ completes the proof.
\end{proof}

\begin{proof}[Proof of Theorem \ref{thrm:fosrecct}] Suppose that \eqref{eq:fostersct} is satisfied. Because we are assuming that the rate matrix is regular, setting $\alpha:=0$ in Lemma~\ref{lem:fosminct} and making use of \eqref{eq:bdw6a7bd67aww}, we find that
$$\Ebx{\varphi_F}\leq v(x)<\infty\quad\forall x\not\in F.$$
%
%
For states $x$ inside $F$, $\Ebx{\varphi_F}$ is trivially finite if $x$ is absorbing. Otherwise, an application of the strong Markov property similar to that in the proof of Theorem~\ref{thrm:expclprop} shows that:
\begin{align*}\Ebx{\varphi_F}&=\Ebx{T_1}+\sum_{z\not\in F}\frac{q(x,z)}{q(x)}\Ebz{\varphi_F}=\frac{1}{q(x)}+\sum_{z\not\in F}\frac{q(x,z)}{q(x)}\Ebz{\varphi_F}\leq\frac{1}{q(x)}\left(1+\sum_{z\not\in F}q(x,z)v(z)\right)\\
&\leq\frac{1}{q(x)}\left(1+\sum_{z\neq x}q(x,z)v(z)\right)\leq v(x)+ \frac{b}{q(x)}<\infty\quad\forall x\in F.\end{align*}
%
%
Given that $\Pb_\gamma=\sum_{x\in\s}\gamma(x)\Pb_x$, multiplying the above two inequalities by $\gamma$ and summing over $x$ in $\s$ then yields $\Ebl{\varphi_F}\leq \gamma(v)+b\gamma(q^{-1}1_F)$ which is finite as long as $\gamma(v)$ is finite.

Conversely, suppose that $\Ebx{\phi_F}<\infty$ for all $x$ in $\s$. Lemma \ref{lem:fosminct} and \eqref{eq:bdw6a7bd67aww} show that $u$ in \eqref{eq:uminct} (with $\alpha=0$) satisfies the inequality \eqref{eq:fostersct} for all states $x$ not in $F$. For states inside $F$ that are absorbing the inequality holds trivially. For the non-absorbing ones, we apply the strong Markov property as before:
$$\frac{1}{q(x)}Qu(x)=\sum_{z\not\in F}\frac{q(x,z)}{q(x)}u(z)=\sum_{z\not\in F}\frac{q(x,z)}{q(x)}\Ebz{\varphi_F}=\Ebx{\varphi_F}-\frac{1}{q(x)}\quad\forall x\in F.$$
In other words, $u$ satisfies \eqref{eq:fostersct} with $b:=\max\{q(x)\Ebx{\varphi_F}:x\in F\}$.
\end{proof}

To the make the jump from Theorem~\ref{thrm:fosrecct} to the criterion, we need the following lemma. It tells us that we lose nothing by assuming that the set $F$ in Theorem \ref{thrm:fosrecct} contains no transient states:
\begin{lemma}\label{lem:Fnotransct}Let $F$ be a finite set such that $\Ebx{\phi_F}<\infty$ for all $x$ in $\s$. The same inequality holds if we remove all transient states from $F$.\end{lemma}
\begin{proof}Replace Lemma~\ref{lem:geotrialarg}$(i)$, Definition~\ref{def:accesible}, and the $\phi$s in the proof of Lemma~\ref{lem:Fnotrans0} with Lemma~\ref{lem:geotrialargct}$(ii)$, Definition~\ref{def:accesiblect}, and $\varphi$s.
\end{proof}

Given the above, the proof of Theorem~\ref{thrm:fostersct} is completely analogous to that of its discrete-time counterpart:
\begin{proof}[Proof of Theorem~\ref{thrm:fostersct}]Replace Proposition~\ref{prop:closedis}, Lemma~\ref{lem:geotrialarg}$(ii)$, Theorem~\ref{thrm:fosrec}, Lemma~\ref{lem:Fnotrans}, \eqref{eq:fosters}, and the $\phi$s in the proof Theorem~\ref{thrm:fosters} with Proposition~\ref{prop:closedisct}, Lemma~\ref{lem:geotrialargct}$(ii)$, Theorem~\ref{thrm:fosrecct}, Lemma~\ref{lem:Fnotransct}, \eqref{eq:fostersct}, and $\varphi$s, respectively.
\end{proof}
\ifdraft
\subsubsection*{Notes and references} An example of a non-regular process (in particular, with $\Ebx{\varphi_F}=\infty$) can be found in p.134 Tweedie1975 Sufficient conditions for regularity, recurrence and
ergodicity of Markov processes. 

\fi
\subsection[Foster-Lyapunov criterion for exponential convergence*]{Foster-Lyapunov criteria III: the criterion for exponential convergence*}\label{sec:expcri}

Lemma~\ref{lem:fosminct} instructs us to study the inequality
\begin{equation}\label{eq:fostersexp}Qv(x)\leq -\alpha v(x)-1+b1_{F}(x)\quad \forall x\in\s.\end{equation}
instead of focusing only on the special case $\alpha=0$ considered in Section~\ref{sec:fostersct}.  For any $\alpha\neq0$, the inequality's minimal (possibly infinite-valued) non-negative solution is given by $u(x)=0$ for all states $x$ in $F$ and
\begin{align}\label{eq:ugeo1ct}u(x)&=\int_0^\infty \Pbx{\{\varphi_F>t,t<T_\infty\}}e^{\alpha t}dt=\Ebx{\int_0^\infty1_{[0,\varphi_F\wedge T_\infty)}(t)e^{\alpha t}}=\Ebx{\int_0^{\varphi_F\wedge T_\infty}e^{\alpha t}}\\
&\frac{1}{\alpha}(\Ebx{e^{\alpha \varphi_F\wedge T_\infty}}-1)\quad\forall x\not\in F.\nonumber\end{align}
The $\alpha>0$ version of Theorem \ref{thrm:fosrecct} follows easily:
\begin{theorem}\label{thrm:fosexprec}Suppose the chain is regular. Given any finite set $F$ and $\alpha> 0$, there exists $v:\s\to[0,\infty)$ and $b\in\r$ satisfying \eqref{eq:fostersexp} if and only if $\Ebx{e^{\alpha\varphi_F}}<\infty$ for all $x$ in $\s$. In this case, $\Ebl{e^{\alpha\varphi_F}}<\infty$ for all initial distributions $\gamma$ satisfying $\gamma(v)<\infty$.
\end{theorem}
\begin{proof}Given \eqref{eq:ugeo1ct}, the proof is entirely analogous to that of Theorem \ref{thrm:fosrecct}.\end{proof}

The theorem shows that  \eqref{eq:fostersexp} is satisfied if and only if the entrance time distribution of $F$ has an exponentially decaying tail (i.e., a \emph{light tail}) for any deterministic starting state.  For this reason, we can relate~\eqref{eq:fostersexp} to the exponential convergence of the time-varying law using    the continuous-time analogue of Kendall's theorem (Theorem~\ref{thrm:kendallct}):\index{Foster-Lyapunov criteria}
\begin{theorem}[The geometric criterion]\label{thrm:fostersexp}Suppose the chain is regular. If there exists a real-valued non-negative function $v$ on $\s$, a finite set $F$, and a real number $b$ satisfying \eqref{eq:fostersexp} for some $\alpha>0$, then the chain is positive Tweedie recurrent. Moreover, if the initial distribution $\gamma$ satisfies $\gamma(v)<\infty$, then  the time varying law converges geometrically fast: there exists some $\beta>0$ such that
\begin{equation}\label{eq:tvgeoconvct}\norm{p_t-\pi_\gamma}=\cal{O}(e^{-\beta t}),\end{equation}
where $\pi_\gamma$ denotes the limiting stationary distribution in \eqref{eq:reclims2ct} and $\norm{\cdot}$ the total variation norm in \eqref{eq:tvnorm}.

Conversely, if the chain is positive Tweedie recurrent, there exist no transient states and only a finite number of closed communicating classes, and \eqref{eq:tvgeoconvct} holds whenever the chain starts at a fixed state (i.e., whenever $\gamma=1_x$ for some $x$ in $\s$), then there exists $(v,F,b,\alpha)$ satisfying \eqref{eq:fostersexp} with $v$ real-valued and non-negative, $F$ finite, $b$ real, and $\alpha>0$.
\end{theorem}

To see why  the finitely-many-closed-communicating-classes requirement in the converse is necessary, consider a chain on an infinite state space whose rate matrix is the matrix of zeros.
\subsubsection*{An open question} To the best of my knowledge, whether or not the no-transient-states requirement in the converse is necessary remains an open question. Similarly as in the discrete-time case, I believe that the answer here lies in that of the open question discussed in Section~\ref{sec:kendallct}.

\subsubsection*{Proving Theorem~\ref{thrm:fostersexp}}
To prove Theorem~\ref{thrm:fostersexp}, we require the exponential analogue of Lemma~\ref{lem:Fnotransct} which tells us that we lose nothing by assuming that the finite set $F$ does not contain any transient states:
\begin{lemma}\label{lem:Fnotransexp}Let $F$ be a finite set such that $\Eb_\gamma[e^{\alpha\varphi_F}]<\infty$ and $\Eb_x[e^{\alpha\varphi_F}]<\infty$ for all $x\in\s$ and some $\alpha>0$. There exists a $\beta>0$ such that, after removing all transient states from $F$,  $\Eb_\gamma[e^{\beta\varphi_{F}}]<\infty$ and  $\Eb_x[e^{\beta\varphi_{F}}]<\infty$ for all $x$ in $\s$.
\end{lemma}

\begin{proof}Replace Lemma~\ref{lem:geotrialarg}$(i)$ with Lemma~\ref{lem:geotrialargct}$(iii)$ in the proof of Lemma~\ref{lem:Fnotrans0}.
\end{proof}
%
%
Given the above, the proof of Theorem~\ref{thrm:fostersexp} is entirely analogous to that of its discrete-time counterpart:
\begin{exercise}Prove Theorem~\ref{thrm:fostersexp} by adapting the proof of its discrete-time counterpart (Theorem~\ref{thrm:fostersgeo}). To do so, replace Theorems \ref{thrm:fosters}, \ref{thrm:fosgeorec}, and \ref{thrm:kendall}, Lemmas~\ref{lem:geotrialarg} and \ref{lem:Fnotransgeo}, and Proposition~\ref{prop:closedis} with Theorems \ref{thrm:fostersct}, \ref{thrm:fosexprec}, and \ref{thrm:kendallct}, Lemmas~\ref{lem:geotrialargct} and \ref{lem:Fnotransexp}, and Proposition~\ref{prop:closedisct}, respectively.  Hint: to prove the continuous-time equivalent of \eqref{eq:fnmuy832nyfeas255} proceeds as follows:

\begin{align}\label{eq:fn738an738ahnw3f80}\{\varphi_{F_i}\leq t, X_t\in A,t<T_\infty\}
&=\bigcup_{n=1}^\infty\{\varphi_{F_i}\leq T_n, Y_n\in A,T_n\leq t<T_{n+1}\}\\
&=\bigcup_{x\in F_i}\bigcup_{n=1}^\infty\bigcup_{m=1}^n\{\varphi_{F_i}=T_m,Y_m=x, Y_n\in A,T_n\leq t<T_{n+1}\}\nonumber\\
&=\bigcup_{x\in F_i}\bigcup_{m=1}^\infty\bigcup_{n=m}^\infty\{\varphi_{F_i}=T_m, Y_{m}=x, Y_n\in A,T_n\leq t<T_{n+1}\}.\nonumber
\end{align}
If $n>m$, 
\begin{align}
&\{\varphi_{F_i}=T_m, Y_m=x, Y_n\in A,T_n\leq t<T_{n+1}\}\nonumber\\
&=\left\{\varphi_{F_i}=T_m,T_m\leq t, Y_{m}=x, Y_n\in A,\sum_{k=m+1}^nS_k\leq t-T_m<\sum_{k=m+1}^{n+1}S_k\right\}\nonumber\\
&=\bigcup_{x_1\in\s}\dots\bigcup_{x_{n-m-1}\in\s}\bigcup_{x_{n-m}\in A}\bigg\{\varphi_{F_i}=T_m,T_m\leq t, Y_m=x_0, Y_{m+1}=x_1,\dots, \label{eq:nfe8wa79nbe8wa7fyna8nfa}\\
&\qquad\qquad Y_{n-1}=x_{n-m-1},Y_{n-m}=x_{n-m},\sum_{k=1}^{n-m}S_{m+k}\leq t-T_m<\sum_{k=1}^{n-m+1}S_{m+k}\bigg\}\nonumber
\end{align}
Using Theorem~\ref{thrm:condind}, Lemma~\ref{lem:etatheta}, Theorem~\ref{thrm:stopdc}, Exercise~\ref{ex:entdc}, and the definitions of the jump chain $Y$ and of the waiting times $(S_n)_{n\in\zp}$ in the Kendall-Gillespie algorithm (Algorithm~\ref{gilalg}), it is not too difficult to show that
\begin{align}\Pb_\gamma\bigg(\bigg\{&\varphi_{F_i}=T_m,T_m\leq t,Y_m=x,Y_{m+1}=x_1,\dots,  Y_{n-1}=x_{n-m-1},Y_{n-m}=x_{n-m},\label{eq:nfe8wa79nbe8wa7fyna8nfa2}\\
&\sum_{k=1}^{n-m}S_{m+k}\leq t-T_m<\sum_{k=1}^{n-m+1}S_{m+k}\bigg|\cal{G}_m\bigg\}\bigg)=1_{\{\varphi_{F_i}=T_m,T_m\leq t, Y_m=x\}}g^{n,m,x,t}_{x_1,\dots,x_{n-m}}(T_m)\nonumber
\end{align}
$\Pb_\gamma$-almost surely, where $(\cal{G}_n)_{n\in\n}$ denotes the filtration generated by the jump chain and jump times (Definition~\ref{def:filt2}),
$$g^{n,m,x,t}_{x_1,\dots,x_{n-m}}(u):=p(x,x_1)\dots p(x_{n-m-1},x_{n-m})\int_0^{t-u}f_x*f_{x_1}*\dots*f_{x_{n-m-1}}(s)\int_{t-u-s}^\infty f_{x_{n-m}}(r)drds$$
for all $u$ in $[0,t]$, $f_z$ denotes the pdf of an exponential random variable with mean $1/\lambda(z)$ and $*$ denotes the convolution operator. Using Theorem~\ref{thrm:condind} (or Theorem~\ref{thrm:pathlawunict}), we have that
\begin{align*}
g^{n,m,x,t}_{x_1,\dots,x_{n-m}}(u)=\Pb_{x}\bigg(\bigg\{&Y_{1}=x_1,\dots,  Y_{n-m-1}=x_{n-m-1},Y_{n-m}=x_{n-m},\\
&\sum_{k=1}^{n-m}S_{m}\leq t-u<\sum_{k=1}^{n-m+1}S_{m}\bigg\}\bigg)\quad\forall u\in[0,t].
\end{align*}
Summing over $x_1,\dots,x_{n-m}$ in $A$, we find that
$$g^{n,m,x,t}_A(u):=\sum_{x_1\in\s}\dots\sum_{x_{n-m-1}\in\s}\sum_{x_{n-m}\in A}g^{n,m}_{x_1,\dots,x_{n-m}}(u)=\Pbx{\{X_{t-u}\in A, T_{n-m}\leq t-u<T_{n-m+1}\}},$$
for all $m>n$. Moreover it follows from \eqref{eq:nfe8wa79nbe8wa7fyna8nfa}--\eqref{eq:nfe8wa79nbe8wa7fyna8nfa2} that
\begin{equation}\label{eq:fn738an738ahnw3f8}\Pbl{\{\varphi_{F_i}=T_m, Y_m=x, Y_n\in A,T_n\leq t<T_{n+1}\}|\cal{G}_m}=1_{\{\varphi_{F_i}=T_m,T_m\leq t, Y_m=x\}}g^{n,m,x,t}_A(T_m),\end{equation}
$\Pb_\gamma$-almost surely, for all $m>n$. Using the same kind of argument, we also find that the above two also hold for $n=m$. Because
$$g^{x,t}_A(u):=\sum_{n=m}^\infty g^{n,m,x,t}_A(u)=\Pbx{\{X_{t-u}\in A,  t-u<T_{\infty}\}}\quad\forall u\in[0,t],\enskip x\in F_i,\enskip i\in\cal{I},$$
it follows from ~\eqref{eq:fn738an738ahnw3f80}~and~\eqref{eq:fn738an738ahnw3f8} that
\begin{align*}\Pb_\gamma(\{\varphi_{F_i}\leq t, X_t\in A,t<T_\infty\})&=\sum_{x\in F_i}\sum_{m=1}^\infty\Ebl{1_{\{\varphi_{F_i}=T_m,T_m\leq t, Y_m=x\}}g^{x,t}_A(T_m)}\\
&=\sum_{x\in F_i}\sum_{m=1}^\infty\Ebl{1_{\{\varphi_{F_i}=T_m,\varphi_{F_i}\leq t, X_{\varphi_{F_i}}=x\}}g^{x,t}_A(\varphi_{F_i})}\\
&=\sum_{x\in F_i}\Ebl{1_{\{\varphi_{F_i}\leq t, X_{\varphi_{F_i}}=x\}}g^{x,t}_A(\varphi_{F_i})}\quad\forall i\in\cal{I}.
\end{align*}
Applying  Lemma~\ref{lem:geotrialargct}$(iii)$ and Theorem~\ref{thrm:fosexprec} it is not difficult to show that all states in $F_i$ are exponentially recurrent. The argument given at the end of the proof of Theorem~\ref{thrm:kendallct} then shows that for any $x$ in $F_i$, 
$$\mmag{g^{x,t}_A(u)-\pi_i(A)}\leq\norm{p_{t-u}(x,\cdot)-\pi_i}\leq C_xe^{-\beta_x (t-u)}\quad\forall A\subseteq\s,\enskip u\in[0, t],\enskip t\in[0,\infty),$$
for some constants $C_x<\infty$ and $\beta_x>0$ depending on $x$. Because $F$ is finite and $F=\cup_{i\in\cal{I}}F_i$, the above implies that there exists some $C<\infty$ and $\beta>0$ such that
$$\mmag{g^{x,t}_A(u)-\pi_i(A)}\leq Ce^{-\beta (t-u)},\quad\forall A\subseteq \s,\enskip x\in F_i,\enskip i\in\cal{I},\enskip u\in[0, t],\enskip t\in[0,\infty).$$
Because Proposition~\ref{prop:closedisct} implies that $\varphi_{F_i}$ is finite if and only if $\varphi_F$ is and $\varphi_{F_i}=\varphi_{F}$ and Theorem~\ref{thrm:fostersexp} shows that $\varphi_F$ is $\Pb_\gamma$-almost surely finite, it follows from the above that
\begin{align*}
\sum_{i\in\cal{I}}|\Pb_\gamma(\{\varphi_{F_i}\leq t, X_t\in A,&t<T_\infty\})-\Pbl{\{\varphi_{F_i}\leq t\}}\pi_i(A)|\\
&\leq \sum_{i\in\cal{I}}\sum_{x\in F_i}\Ebl{1_{\{\varphi_{F_i}\leq t, X_{\varphi_{F_i}}=x\}}\mmag{g^{x,t}_A(\varphi_{F_i})-\pi_i(A)}}\\
&\leq \sum_{i\in\cal{I}}\sum_{x\in F_i}\Ebl{1_{\{\varphi_{F_i}\leq t, X_{\varphi_{F_i}}=x\}}Ce^{-\beta ({t-\varphi_{F_i}})}}\\
&=Ce^{-\beta t}\sum_{i\in\cal{I}}\Ebl{1_{\{\varphi_{F_i}\leq t\}}e^{\beta \varphi_{ F_i}}}=Ce^{-\beta t}\Ebl{e^{\beta \varphi_{ F}}}\quad\forall t\in[0,\infty).
\end{align*}
\end{exercise}

\subsection{Farewell}
%
%
%
I originally wrote the blurb below for my thesis. It feels right here too.
\\\\
\noindent$\dots$ there is yet much to be done to enable the quantitative analysis of chains with large state spaces $\dots$ The idiom `there is no rest for the wicked' comes to mind. However, I believe that quite the opposite is true given how much \emph{fun} these things can be. In the case that this is something you might like to get involved in (or already are), I wish you the best of luck: I am rooting for your success, \emph{especially} in the areas where I did not meet mine (or, at best, met it partially).  I hope that these questions and issues bring as much delight to your life as they have to mine.
\begin{quotation}``The things with which we concern ourselves in science appear in myriad forms, and with a multitude of attributes. For example, if we stand on the shore and look at the sea, we see the water, the waves breaking, the foam, the sloshing motion of the water, the sound, the air, the winds and the clouds, the sun and the blue sky, and light; there is sand and there are rocks of various hardness and permanence, color and texture. There are animals and seaweed, hunger and disease, and the observer on the beach; there may be even happiness and thought.''\end{quotation}
\begin{quotation}
\emph{Richard Feynman in the first volume of his lectures on physics.}\end{quotation}

\ifdraft

\newpage

\part{Practice}

\sectionmark{\MakeUppercase{}}

\newpage
\thispagestyle{premain}
\sectionmark{\MakeUppercase{State space truncation}}
\section*{State space truncation}
\addcontentsline{toc}{section}{\protect\numberline{}State space truncation}

In theory, computing the statistics of a chain (e.g., the time-varying law, exit time distributions, stationary distributions, or averages thereof) merely entails solving a system of linear equations of one type or another. The catch is that the system in question typically has as many equations and unknowns as there are states in the state space. If the state space is infinite, or finite but large, then these equations cannot be solved.

In this chapter, we discuss \emph{truncation-based} methods that are commonly used to overcome this issue. The general recipe most of them follow is very simple:
\begin{enumerate}
\item Choose a subset, or \emph{truncation}, $\s_r$ of the the state space $\s$ that is of manageable size.
\item Build a second chain $X^r$ that behaves similarly to our chain of interest $X$ but that never leaves $\s_r$.
\item Compute the statistics of $X^r$ (something straightforward if the truncation is small enough).
\item Use these as approximations of the statistics of $X$.\end{enumerate}
The $\s_r$-valued chain $X^r$ should be tailored to the particular statistics of interest. In some cases, it is even possible to construct $X$ and $X^r$ such that the statistics of one dominate the other and we are able to bound, if not outright compute, the error of the approximation produced: a significant practical advantage.

It is typically the case with these methods that adding states to the truncation used improves the approximation's quality. In exchange, larger truncations lead to greater computational costs. To explore these matters, we often view truncation-based methods as procedures that return an entire sequence of approximations corresponding to a sequence of increasing truncations 
$$\s_1\subseteq\s_2\subseteq\dots,$$
instead of a single approximation corresponding to a single truncation. 

A sanity check for the correctness of these methods is establishing their \emph{convergence}: their ability to produce arbitrarily accurate approximations given enough computational power. In other words, showing that, if the truncations approach the entire state space,
\begin{equation}\label{eq:trunc1}\lim_{r\to\infty}\s_r=\bigcup_{r=1}^\infty\s_r=\s,\end{equation}
then the corresponding sequence of approximations converges, in one sense or another, to statistic of interest.

Other aspects we typically aim to investigate are how to pick the truncation to obtain the best possible approximation for a given computational budget; whether there are efficient implementations of the method that allow us to further stretch our budget; and, when possible, the method's rate of convergence (for a given sequence of truncations) and its computational cost (for a given truncation) which, together, inform us on its limitations.


\subsection[The finite state projection algorithm]{The finite state projection (FSP) algorithm: truncations with absorbing complements for the time-varying law}\label{sec:fsp}
We begin with the the so-called \emph{finite state projection (FSP) algorithm} used to approximate time-varying laws of chains with large state spaces. In particular, time-varying laws of chains with state spaces large enough that we are unable to directly solve the chain's forward equations~\eqref{eq:master} or their discrete-time counterparts~\eqref{eq:dtlaw}, as appropriate.

In a nutshell, the algorithm consists of computing the time-varying law of the chain $X^r$ obtained by turning all states outside of a user-chosen truncation $\s_r$ into absorbing states and using it as an approximation to the time-varying law of the original chain $X$.

\subsubsection*{The discrete-time case}

Throughout this section, $X$ denotes a discrete-time chain with a large state space $\s$, one-step matrix $P=(p(x,y))_{x,y\in\s}$, and initial distribution $\gamma=(\gamma(x))_{x\in\s}$. Our aim is to obtain accurate approximations of the chain's distribution $p_n=(p_n(x))_{x\in\s}$ at some time $n$. The  discrete-time version of the FSP algorithm developed to this end entails choosing a finite truncation $\s_r$ of the state space, computing $(p^r_n(x))_{x\in\s_r}$ by running the recursion
\begin{equation}\label{eq:fspdtrec}p^r_{m+1}(x)=\sum_{y\in\s_r}p^r_m(y)p(y,x)\quad\forall x\in\s_r,\enskip m<n,\qquad p^r_0(x)=\gamma(x)\quad\forall x\in\s_r,\end{equation}
padding $(p^r_n(x))_{x\in\s_r}$ with zeros
\begin{equation}\label{eq:fspdtpad}p_n^r(x):=0\qquad \forall x\not\in\s_r,\end{equation}
and using the resulting measure $p^r_n:=(p_n^r(x))_{x\in\s}$ as an approximation of $p_n$. This approximation has appealing theoretical properties: 
\begin{theorem}[The finite state projection algorithm, discrete-time version]\label{thrm:fspdt} Let $(\s_r)_{r\in\zp}$ be an increasing sequence of finite subsets of $\s$, $(p^r_n)_{r\in\zp}$ be the corresponding sequence of FSP approximations defined by~\eqref{eq:fspdtrec}--\eqref{eq:fspdtpad}, and $(\sigma_r)_{r\in\zp}$ that of exit times defined by~\eqref{eq:sigmar}.
\begin{enumerate}[label=(\roman*)]
\item \label{th:fspdt_i}The approximations form an increasing sequence of lower bounds on the time-varying law:
%
%
$$p_n^1(x)\leq p_n^2(x)\leq \dots\leq p_n(x)\quad\forall x\in\s,\enskip n\geq0.$$
\item  \label{th:fspdt_ii}The mass of the approximation is the probability that the chain has not yet exited the truncation:
$$p_n^r(\s)=p_n^r(\s_r)=\Pbl{\{\sigma_r>n\}}\quad\forall n\geq0,\enskip r>0.$$
\item \label{th:fspdt_iii}The total variation approximation error is the probability that $X$ has left the truncation:
\begin{equation}\label{eq:fsperrdt}\norm{p_n-p_n^r}= \Pbl{\{\sigma_r\leq n\}}=1-p_n^r(\s_r)=:\epsilon_r\quad\forall n\geq0,\enskip r>0,\end{equation}
where $\norm{\cdot}$ denotes the total variation norm in \eqref{eq:tvnorm}.
\item \label{th:fspdt_iv}The approximation error decreases with $r$:
$$\norm{p_n-p^r_n}\leq \norm{p_n-p^s_n}\quad \forall s\leq r, \enskip  r>0,\enskip n\geq0,$$
and increases with $n$:
$$\norm{p_n-p^r_n}\geq \norm{p_l-p^r_l} \quad \forall l\leq n, \enskip n\geq0,\enskip \forall r>0.$$
%
%
\item \label{th:fspdt_v}If $\cup_{r=1}^\infty\s_r=\s$, then the approximation converges to the time-varying law:
$$\lim_{r\to\infty}\norm{p_n-p_n^r}=0\quad\forall n\geq0.$$
\end{enumerate}
\end{theorem}
We leave the details of the proof to end of the section and instead focus here on the simple ideas underpinning the proof. Consider a second chain $X^r$ which is identical to $X$ except that every state outside of the truncation $\s_r$ has been turned into an absorbing state. In other words, $X^r$ has one-step matrix
\begin{equation}\label{eq:pr}p^r(x,y):=\left\{\begin{array}{ll}p(x,y)&\text{if }x\in\s_r\\ 0&\text{if }x\not\in\s_r\end{array}\right.\quad\forall x,y\in\s.\end{equation}
Following the same type of approach as that in Lemma \ref{samechainsh}, we are able to build $X^r$ such that $X$ and $X^r$ coincide up until (and including) the time-step $\sigma_r$ that they both simultaneously leave the truncation $\s_r$ for the first time. At this point, $X^r$ becomes trapped in whichever state outside of the truncation it just entered and never returns to $\s_r$. In contrast, $X$ may return to the truncation. For this reason, the probability $p_n(x)$ that $X$ is at any given state $x$ inside the truncation at time $n$ is at least the probability that $X^r$ is in the same state at the same time. Theorem~\ref{dtlawchar} tells us that the time-varying law of $X^r_n$ (restricted to $\s_r$) is the solution of \eqref{eq:fspdtrec} and we arrive at Theorem \ref{thrm:fspdt}$\ref{th:fspdt_i}$. 

Because $X^r$ never returns to the truncation once it leaves, the probability $p^r_n(\s_r)$ that $X^r$ is inside the truncation at time $n$ is the same as the probability $\Pbl{\{\sigma_r>n\}}$ that it has not yet left. Theorem \ref{thrm:fspdt}$\ref{th:fspdt_ii}$--$\ref{th:fspdt_iii}$ follow from this fact. If $X^r$ has not left the truncation by time $n$, then it has not left the \emph{larger} truncated space $\s_{r+1}$ by $n$. Similarly, if the chain has not left $\s_r$ by time $n$, it has not left by any earlier time $l\leq n$. For these reasons, Theorem~\ref{thrm:fspdt}$\ref{th:fspdt_iv}$ holds.

If the truncations approach the entire state space, then the probability $\Pbl{\{\sigma_r>n\}}$ that the chain has exited the $r$th truncation by time-step $n$ decays to zero as $r$ tends to infinity. Theorem~\ref{thrm:fspdt}$\ref{th:fspdt_v}$ then follows directly from  Theorem~\ref{thrm:fspdt}$\ref{th:fspdt_iii}$.


\subsubsection*{The continuous-time case}The continuous-time version of the FSP algorithm yields approximations of the distribution $p_t$ at time $t$ of a continuous-time chain $X$ with state space $\s$, rate matrix $Q=(q(x,y))_{x,y\in\s}$, initial distribution $\gamma=(\gamma(x))_{x\in\s}$, and explosion time $T_\infty$. It consists of picking a finite truncation $\s_r$ of the state space, computing $(p^r_t(x))_{x\in\s_r}$ by solving the ODEs
\begin{equation}\label{eq:fsprec}\dot{p}^r_{s}(x)=\sum_{y\in\s_r}p^r_s(y)q(y,x)\quad\forall x\in\s_r,\enskip s\leq t,\qquad p^r_0(x)=\gamma(x)\quad\forall x\in\s_r,\end{equation}
padding $(p^r_t(x))_{x\in\s_r}$ with zeros
\begin{equation}\label{eq:fsppad}p_t^r(x):=0\qquad \forall x\not\in\s_r,\end{equation}
and using the resulting measure $p^r_t:=(p_t^r(x))_{x\in\s}$ as an approximation of $p_t$. The approximation's properties are (almost!) entirely analogous to those of its discrete-time counterpart:
\begin{theorem}[The finite state projection algorithm, continuous-time version]\label{thrm:fsp} 
Let $(\s_r)_{r\in\zp}$ be an increasing sequence of finite subsets of $\s$, $(p^r_t)_{r\in\zp}$ be the corresponding sequence of FSP approximations defined by~\eqref{eq:fsprec}--\eqref{eq:fsppad}, and $(\tau_r)_{r\in\zp}$ that of exit times defined by~\eqref{eq:taurexit}.
\begin{enumerate}[label=(\roman*)]
\item The approximations form an increasing sequence of lower bounds on the time-varying law: \label{th:fsp_i}
%
%
$$p_t^1(x)\leq p_t^2(x)\leq \dots\leq p_t(x)\quad\forall x\in\s,\enskip  t\geq0.$$
\item \label{th:fsp_ii}
The mass of the approximation is the probability that the chain has not yet exited the truncation:
$$p_t^r(\s)=p_t^r(\s_r)=\Pbl{\{\tau_r>t\}}\quad\forall t\geq0,\enskip r>0.$$
\item\label{th:fsp_iii} The total variation approximation error is the probability that $X$ has left the truncation but has not yet exploded:
\begin{equation}\label{eq:fsperr}\norm{p_t-p_t^r}=\Pbl{\{\tau_r\leq t<T_\infty\}}\leq 1-p_t^r(\s_r)=:\epsilon_r\quad\forall t\geq0,\enskip r>0,\end{equation}
where $\norm{\cdot}$ denotes the total variation norm in \eqref{eq:tvnorm}. Equality holds if and only if  the chain is non-explosive (i.e., $\Pbl{\{T_\infty=\infty\}}=1$). 
\item\label{th:fsp_iv}
The approximation error $\norm{p_t-p_t^r}$ and its upper bound $\epsilon_r$ decrease with $r$:
\begin{equation}\label{eq:fspermon1}\norm{p_t-p^r_t}\leq \norm{p_t-p^s_t},\quad 1-p^r_t(\s_r)\leq 1-p^s_t(\s_s) \quad \forall s\leq r, \enskip r>0,\enskip  t\geq0, 
\end{equation}
and increase with $t$:
\begin{equation}\label{eq:fspermon2}\norm{p_t-p^r_t}\geq \norm{p_u-p^r_u} ,\quad  1-p^r_t(\s_r)\geq 1-p^r_u(\s_r) \quad \forall u\leq t, \enskip t\geq0, \enskip r>0.
\end{equation}
%
%
\item\label{th:fsp_v}If $\cup_{r=1}^\infty\s_r=\s$, then the approximations converge to the time-varying law:
$$\lim_{r\to\infty}\norm{p_t-p_t^r}=0,\quad \forall t\geq0.$$
\end{enumerate}
\end{theorem}

We also leave the proof of the above until the end of the section. The ideas behind it are entirely analogous to those in the discrete-time case except that, here, $X^r$ is a chain with rate matrix $Q^r=(q^r(x,y))_{x,y\in\s}$ defined by
\begin{equation}
\label{eq:qr}
q^r(x,y):=\left\{\begin{array}{ll}q(x,y)&\text{if }x\in\s_r\\0&\text{if }x\not\in\s_r\end{array}\quad\forall x,y\in\s_r.\right.
\end{equation}
There is only one real difference between Theorem~\ref{thrm:fsp} and its discrete-time counterpart (Theorem~\ref{thrm:fspdt}): the possibility of explosion means that the practically-computable $\epsilon_r$ may not be the actual error but rather an upper bound thereof. In particular,
$$\epsilon_r=\Pbl{\{\tau_r\leq t\}}\geq \Pbl{\{\tau_r\leq t<T_\infty\}}=\norm{p_t-p^r_t}$$
and the inequality is strict unless the chain is non-explosive (i.e., $\Pbl{\{T_\infty=\infty\}}=1$).

\subsubsection*{Practical use}

While I personally don't tend to make a distinction, the FSP algorithm is usually presented as an iterated version of the approach described above: 
\begin{enumerate}
\item Choose a desired error tolerance $\varepsilon>0$ and initial truncation.
\item Compute the FSP approximation corresponding to the truncation.
\item Evaluate  $\epsilon_r$  in~\eqref{eq:fsperrdt} or \eqref{eq:fsperr}, as appropriate.
\item If $\epsilon_r\leq\varepsilon$ stop and return the approximation. Otherwise, add more states to the truncation and return to Step 2.
\end{enumerate}
For discrete-time chains or non-explosive continuous-time ones, this is a perfectly sensible approach: $\epsilon_r$ tends to zero as the truncation $\s_r$ approaches the entire state space $\s$ (c.f., Theorems~\ref{thrm:fspdt}$\ref{th:fspdt_iii}$,~$\ref{th:fspdt_v}$ and \ref{thrm:fsp}$\ref{th:fsp_iii}$,~$\ref{th:fsp_v}$) and the approach will terminate regardless of the error tolerance $\varepsilon$ chosen. 
However, if the chain is explosive, the procedure can be problematic: $\epsilon_r$ is no longer the total variation error but an upper bound thereof. Moreover, $\epsilon_r$ does not converge to zero as $\s_r$ approaches $\s$ but instead converges to the probability $\Pbl{\{t\leq T_\infty\}}$ that the chain has exploded by time $t$. Thus, if $\varepsilon$ is chosen smaller than $\Pbl{\{t\leq T_\infty\}}$, the procedure will never terminate.

In discrete-time, the cost of computing the FSP approximation corresponding to a given truncation $\s_r$ is $\cal{O}(n\mmag{\s_r}^2)$, where $\mmag{\s_r}$ denotes the size of the $\s_r$, or $\cal{O}(n\mmag{\s_r}w_c)$ if the one-step matrix $P$ is sparse and its columns only have $w_c$ non-zero entries. In continuous-time, the exact cost depends on the  method used to solve the ODEs~\eqref{eq:fsprec}. For instance, if we use Euler's method with a step size of $h$, then the cost is $\cal{O}(t\mmag{\s_r}^2/h)$ or $\cal{O}(t\mmag{\s_r}w_c/h)$ in the sparse case. In general, the cost will grow at least linearly with the truncation size $\mmag{\s_r}$ and often quadratically or higher. 
Not all truncations of the same size are born equal: Theorems~\ref{thrm:fsp}\ref{th:fspdt_iii} and~\ref{thrm:fsp}\ref{th:fsp_iii} show that the truncations that lead to the smallest error possible are those that minimise the probability $p$ that the chain has left the truncation by the time point of interest. For these reasons, the practical success of the FSP approach often will crucially depend on our truncation choice. Unfortunately, choosing a truncation that minimises $p$ is easier than done. However, a good starting point is to guide this selection using a few sample paths obtained from simulation.

\subsubsection*{Proofs of Theorems~\ref{thrm:fspdt}~and~\ref{thrm:fsp}}It's time to deal with the proofs of the FSP approximation's theoretical properties. As usual, we begin with the discrete-time case:
\begin{proof}[Proof of Theorem~\ref{thrm:fspdt}]$\ref{th:fspdt_i}$ Suppose that $p^{r+1}_n(x)\geq p^r_n(x)$  for each $x$ in $\s_r$ and note that
$$p^{r+1}_{n+1}(x)=\sum_{y\in\s_{r+1}}p^{r+1}_n(y)p(y,x)\geq \sum_{y\in\s_r}p^{r+1}_n(y)p(y,x)\geq \sum_{y\in\s_r}p^{r}_n(y)p(y,x)= p^{r}_{n+1}(x).$$
Since $p^{r+1}_0(x)=p^{r}_0(x)=\gamma(x)$ for each $x$ in $\s_r$, induction yields all inequalities in $\ref{th:fspdt_i}$ but the rightmost one. The rightmost one follows by replacing $p^{r+1}$ and $\s_{r+1}$ with $p$ and $\s$ throughout the above and applying Theorem \ref{dtlawchar}.

$\ref{th:fspdt_ii}$ The first equation follows immediately from \eqref{eq:fspdtpad}. The second is trivial if $n=0$. Otherwise,  setting $\cal{D}=\s_r$ in \eqref{eq:hatdtimelaw} and comparing with \eqref{eq:fspdtrec},  we find that $\hat{p}_n(x)$ in the former equals $p^r_n(x)$ in the latter for all $x$ in $\s_r$. For this reason, \eqref{eq:edisdefd}, \eqref{eq:eoed2t}, and \eqref{eq:hatdtimelaw} imply that
$$\Pbl{\{\sigma_r=m+1,X_{\sigma_r}=x\}}=\sum_{y\in\s_r}p^r_{m}(y)p(y,x)\quad\forall x\not\in\s_r,\enskip r>0,\enskip m\geq0.$$
Thus,
\begin{align*}\Pbl{\{\sigma_r=m+1\}}&=\sum_{x\not\in\s_r}\Pbl{\{\sigma_r=m+1,X_{\sigma_r}=x\}}=\sum_{x\not\in\s_r}\sum_{y\in\s_r}p^r_{m}(y)p(y,x)\\
&=\sum_{y\in\s_r}p^r_{m}(y)\left(1-\sum_{x\in\s_r}p(y,x)\right)=p^r_{m}(\s_r)-p^r_{m+1}(\s_r)\quad\forall r>0,\enskip m\geq0. \end{align*}
Summing the above over $m=n,n+1,\dots$, then yields the desired $p^r_{n}(\s_r)=\Pbl{\{\sigma_r>n\}}$.

$\ref{th:fspdt_iii}$ Because the total variation norm of an unsigned measure is its mass, this follows immediately from $\ref{th:fspdt_i}$--$\ref{th:fspdt_ii}$.

$\ref{th:fspdt_iv}$  Because the truncations increase with $r$, the definition in~\eqref{eq:sigmar} of the exit times  $\sigma_r$ imply that these also increase with $r$, and $\ref{th:fspdt_iv}$ follows directly from $\ref{th:fspdt_iii}$.

$\ref{th:fspdt_v}$ Because $\sigma_r\to\infty$ as $r\to\infty$ with $\Pb_\gamma$-probability one (Lemma~\ref{lem:dtexitrinf}), this follows directly from $\ref{th:fspdt_iii}$.
%
\end{proof}
Now, for the continuous-time case:
\begin{proof}[Proof of Theorem \ref{thrm:fsp}] Throughout this proof, let $X^r$ be as in Lemma~\ref{samechainsh} with $\hat{Q}=(q^r(x,y))_{x,y\in\s}$ in \eqref{eq:qr} replacing $\bar{Q}$ in the lemma's premise. Theorem~\ref{thrm:forward} implies that $p^r_t$ in~\eqref{eq:fsprec} and coincide with the time-varying law of $X^r$ when both are restricted to $\s_r$:
$$p^r_t(x)=\Pbl{\{t<T_\infty^r\}}\quad\forall x\in\s_r,\enskip t\geq0.$$

$\ref{th:fsp_i}$  The key facts here are that $X^r$ and $X$ are identical up until the moment that they simultaneously leave the truncation (Lemma~\ref{samechainsh}) and that $X^r$ gets stuck in an absorbing state the instant it does. In particular, the same reasoning as in \eqref{eq:stuck} shows that
$$\{\tau_r\leq t<T^r_\infty,X_t^r=x\}=\{\tau_r\leq t<T^r_\infty,X_{\tau_r}^r=x\}\quad\forall x\in\s,\enskip r>0,$$
where $T^r_\infty$ denotes the explosion time of $X^r$. Because $X^r$ lies outside the truncation $\s_r$ at the time of exit (Proposition~\ref{prop:hitdc}), the right-hand side set is empty if $x$ belongs to $\s_r$ and it follows that
\begin{equation}\label{eq:fn7ew8abfea8feua}\{t<T^r_\infty,X_t^r=x\}=\{t<\tau_r,t<T^r_\infty,X_t^r=x\}\quad\forall x\in\s_r,\enskip r>0.\end{equation}
It then follows from Lemma~\ref{samechainsh}~$\ref{samechainshi}$~and~$\ref{samechainshiv}$ that
\begin{equation}\label{eq:fn7ew8abfea8feua2}\{t<\tau_r,t<T^r_\infty,X_t^r=x\}=\{t<\tau_r,t<T_\infty,X_t=x\}\subseteq\{t<T_\infty,X_t=x\}\end{equation}
for all $x$ in $\s_r$ and $r>0$. Given~\eqref{eq:fsppad}, taking expectations of \eqref{eq:fn7ew8abfea8feua}--\eqref{eq:fn7ew8abfea8feua2} then shows that $p^r_t(x)\leq p_t(x)$ for all $x$ in $\s$. To instead show that $p^{r+1}_t(x)\leq p_t(x)$ for all $x$ in $\s$, replace $X$ with $X^{r+1}$ throughout the above argument.

$\ref{th:fsp_ii}$ The first equation follows trivially from~\eqref{eq:fsppad}. For the second, note that $X^r_t$ belongs to $\s_r$ for all $t<\tau_r$ because $\tau_r$ is also time of first exit from $\s_r$ for $X^r$ (Lemma~\ref{samechainsh}~\ref{samechainshi}). Thus, taking expectations and summing $x$ over $\s_r$ in \eqref{eq:fn7ew8abfea8feua} we find that
$$p^r_t(\s_r)=\Pbl{\{t<\tau_r,t<T^r_\infty\}}.$$
Because the rate matrix $Q^r$ is bounded, Lemma~\ref{lem:tinfsum} shows that $T^r_\infty=\infty$ with $\Pb_\gamma$-probability one and the desired result follows from the above.

$\ref{th:fsp_iii}$ Because the total variation norm of an unsigned measure is its mass and because Lemma~\ref{tautin} shows that $\tau_r$ is strictly smaller than $T_\infty$, $\ref{th:fsp_i}$--$\ref{th:fsp_ii}$ imply that
$$\norm{p_t-p^r_t}=p_t(\s)-p^r_t(\s_r)=\Pbl{\{t<T_\infty\}}-\Pbl{\{\tau_r<t\}}\quad\forall t\geq0,\enskip r>0.$$
The inequality also follows as $p_t(\s)\leq 1$. That the inequality is sharp if and only if $\Pbl{\{T_\infty=\infty\}}=1$ then follows from Proposition~\ref{prop:nonexptimevar} that tells us $\Pbl{\{T_\infty=\infty\}}=1$ if and only if $p_t(\s)=1$ for at least one $t\geq0$, in which case $p_t(\s)=1$ for all $t\geq0$.

$\ref{th:fsp_iv}$  Because the truncations increase with $r$, the definition in~\eqref{eq:taurexit} of the exit times  $\sigma_r$ imply that these also increase with $r$, and $\ref{th:fsp_iv}$ follows directly from $\ref{th:fsp_ii}$--$\ref{th:fsp_iii}$.

$\ref{th:fsp_v}$ Because $\tau_r\to T_\infty$ as $r\to\infty$ with $\Pb_\gamma$-probability one (Lemma \ref{tautin}), this follows directly from $\ref{th:fspdt_iii}$.

\end{proof}

\subsubsection*{Notes and references} The name \emph{finite state projection (FSP)} is due to the well-known paper~\citep{Munsky2006} that popularised this approach in the systems biology and chemical physics literature. This approach and its error bound was described in the `80s  queueing literature~\citep{Gross1984} (where it was simply referred to `state space truncation'), and perhaps earlier. Our particular treatment here is based on that given in \citep{Kuntzthe,Kuntz2019}.  As for formal proofs of the algorithm's theoretical properties, these lie scattered throughout the literature. In the continuous-time case, $\ref{th:fsp_i}$ and $\ref{th:fsp_v}$ were first shown in~\citep[Proposition~2.14]{Anderson1991} (however, there is a small mistake therein, see~\citep{Chen1996}); an argument for $\ref{th:fsp_iv}$ and the bound in $\ref{th:fsp_iii}$ was first given in~\cite{Munsky2006}; while the probabilistic interpretations of the approximation's mass and error in $\ref{th:fsp_ii}$--$\ref{th:fsp_iii}$ were given in \citep{Kuntzthe,Kuntz2019}.

\subsection[The exit time finite state projection algorithm]{The exit time finite state projection (ETFSP) algorithm: truncations with absorbing complements for exit distributions}\label{sec:etfsp}

In Sections~\ref{sec:exit} and \ref{moh1}, we saw how the exit distribution $\mu$ and occupation measure $\nu$ associated with the time of exit from a domain $\cal{D}$ of a Markov chain $X$ can be expressed in terms of the time-varying law of the chain $\hat{X}$ obtained by turning all states outside $\cal{D}$ into absorbing states. If $\cal{D}$ is large, then we are generally unable to solve for the time-varying law of $\hat{X}$ and, consequently, for $\mu$ and $\nu$. Instead, applying the FSP algorithm (Section~\ref{sec:fsp}) to $\hat{X}$, we obtain converging approximations $\mu^r$ and $\nu^r$ of $\mu$  and $\nu$. For a lack of a better name, I refer to the this approach as the \emph{exit time finite state projection (ETFSP)} algorithm.


\subsubsection*{The discrete-time case}
Throughout this section, $X$ denotes a discrete-time chain with a large state space $\s$, one-step matrix $P=(p(x,y))_{x,y\in\s}$, and initial distribution $\gamma=(\gamma(x))_{x\in\s}$. Here, we tackle the problem of computing the exit distribution  and occupation measure,
$$\mu(n,x):=\Pbl{\{\sigma=n,X_n=x\}},\quad \nu(n,x):=\Pbl{\{\varsigma>n,X_n=x\}},\quad\forall n\geq0,\enskip x\in\s,$$
associated with the time
$$\sigma(\omega):=\inf\{n\geq0: X_n(\omega)\not\in\cal{D}\}\quad\forall \omega\in\Omega,$$
that the chain $X$ exits a given domain $\cal{D}$ for the first time. If $\cal{D}$ is small enough that we are able to solve the recursion~\eqref{eq:hatdtimelaw} directly, then we can compute $\mu$ and $\nu$ using~\eqref{eq:eoed2t}. Otherwise we apply the ETFSP algorithm which entails choosing a state space truncation $\s_r$ of manageable size and final time $n^r_f$, computing $(\mu(n,x))_{n\leq n^r_f,x\in\cal{D}_r}$ and $(\nu(n,x))_{n\leq n^r_f,x\in\cal{D}_r}$ via
\begin{align*}
\nu^r(n+1,x)&=\sum_{y\in\cal{D}_r}\nu^r(n,y)p(y,x) \quad \forall x\in\cal{D}\cap\s_r,\enskip n<n^r_f,\qquad \nu^r(0,x)=\gamma(x),\quad \forall x\in\cal{D}\cap\s_r,\\
\mu^r(n+1,x)&=\sum_{y\in\cal{D}_r}\nu^r(n,y)p(y,x)\quad\forall x\in\s_r\cap\cal{D}^c,\enskip n<n^r_f,\qquad \mu^r(0,x)=\gamma(x)\quad \forall x\in\s_r\cap\cal{D}^c.\end{align*}
padding them with zeroes,
\begin{align*}\nu^r(n,x):=0  \enskip\text{if}\enskip x\not\in\cal{D}\cap\s_r\quad \text{or}\quad  n> t^r_f,\qquad \mu^r(n,x):=0\enskip \text{if}\enskip x\not\in\s_r\cap\cal{D}^c\quad \text{or}\quad n> n^r_f,\end{align*}
and using the resulting $\mu^r:=(\mu^r(n,x))_{n\in\n,x\in\s}$ and $\nu^r:=(\nu^r(n,x))_{n\in\n,x\in\s}$ as approximations of $\mu$ and $\nu$, respectively. These approximations have nice properties:
\begin{theorem}[The exit time finite state projection algorithm, discrete-time version]\label{etfspthrmd} Suppose that $\{\s_r\}_{r\in\n}$ is an increasing sequence of finite sets such that $\cup_{r}\s_r=\s$, $\{n^r_f\}_{r\in\n}$ is an increasing sequence of natural numbers with limit $\infty$, and pick any $r,n^r_f\in\n$. 
\begin{enumerate}[label=(\roman*)]
\item(Increasing sequence of lower bounds). It is the case that
\begin{align*}\mu^0(n,x)\leq \mu^1(n,x)\leq \dots \leq \mu(n,x)\qquad\forall x\in\s,\quad n\in\n,\\
\nu^0(n,x)\leq \nu^1(n,x)\leq \dots \leq \nu(n,x)\qquad\forall x\in\s,\quad n\in\n.\end{align*}
\item(Approximation error and the bound for the exit distribution case). For any $r\in\n$,
\begin{align*}\norm{\mu-\mu^r}_{TV}&=\Pbl{\{\sigma<\infty\}}-\sum_{x\in\cal{D}\cap\s_r}\sum_{n=0}^{n^r_f}\mu^r(n,x)\leq 1-\sum_{x\in\cal{D}\cap\s_r}\sum_{n=0}^{n^r_f}\mu^r(n,x),\\
\norm{\nu-\nu^r}_{TV}&=\Ebl{\sigma}-\sum_{x\in\cal{D}\cap\s_r}\sum_{n=0}^{n^r_f}\mu^r(n,x)\end{align*}
\item(Monotonicity of the error and of the error bound). The approximation errors $\norm{\mu-\mu^r}_{TV}$ (and its upper bound given in $(ii)$) and $\norm{\nu-\nu^r}_{TV}$ are non-increasing in $r$. 
\item(Convergence of bounds). The approximations of the exit distribution converge in total variation:
$$\lim_{r\to\infty}\norm{\mu-\mu^r}_{TV}=0.$$
If the exit time has finite mean, so do those of the occupation measure:
$$\lim_{r\to\infty}\norm{\nu-\nu^r}_{TV}=0.$$
\end{enumerate}

\begin{enumerate}[label=(\roman*)]
\item \label{th:etfspdt_i}The approximations form increasing sequences of lower bounds on $\mu$ and $\nu$:
\begin{align*}\mu^0(n,x)\leq \mu^1(n,x)\leq \dots \leq \mu(n,x)\qquad\forall n\geq0,\enskip x\in\s,\\
\nu^0(n,x)\leq \nu^1(n,x)\leq \dots \leq \nu(n,x)\qquad\forall n\geq0,\enskip x\in\s.\end{align*}
\item  \label{th:etfspdt_ii}The mass of $\mu^r$ is the probability that the chain exits the domain no later than it exits the truncation and the final time $n^r_f$: 
$$\mu^r([0,\infty),\s)=\Pbb{\{\sigma\leq \sigma_r\wedge n^r_f\}}\quad\forall r>0.$$
The mass of $\nu^r$ is given by
$$\nu^r([0,\infty),\s)=\Ebb{\{(\sigma\wedge n^r_f)1_{\{\sigma\leq \sigma_r\}}}\quad\forall r>0.$$

\item \label{th:etfspdt_iii}The total variation approximation error is the probability that $X$ has left the truncation:
\begin{align*}\norm{\mu-\mu^r}_{TV}&=\Pbl{\{\sigma<\infty\}}-\sum_{x\in\cal{D}\cap\s_r}\sum_{n=0}^{n^r_f}\mu^r(n,x)\leq 1-\sum_{x\in\cal{D}\cap\s_r}\sum_{n=0}^{n^r_f}\mu^r(n,x),\\
\norm{\nu-\nu^r}_{TV}&=\Ebl{\sigma}-\sum_{x\in\cal{D}\cap\s_r}\sum_{n=0}^{n^r_f}\mu^r(n,x)\end{align*}
where $\norm{\cdot}$ denotes the total variation norm in \eqref{eq:tvnorm}.
\item \label{th:etfspdt_iv}The approximation error decreases with $r$:
$$\norm{p_n-p^r_n}\leq \norm{p_n-p^s_n}\quad \forall s\leq r, \enskip  r>0,\enskip n\geq0,$$
and increases with $n$:
$$\norm{p_n-p^r_n}\geq \norm{p_l-p^r_l} \quad \forall l\leq n, \enskip n\geq0,\enskip \forall r>0.$$
%
%
\item \label{th:fspdt_v}If $\cup_{r=1}^\infty\s_r=\s$, then the approximation converges to the time-varying law:
$$\lim_{r\to\infty}\norm{p_n-p_n^r}=0\quad\forall n\geq0.$$
\end{enumerate}

\end{theorem}

\begin{proof}Given Theorem \ref{characttd}, the result follows from arguments similar to those in the proof of Theorem \ref{fspthrmd}. 
\end{proof}

\newpage
\thispagestyle{premain}
\section*{Uniformisation}
\addcontentsline{toc}{section}{\protect\numberline{}Uniformisation}
\sectionmark{\MakeUppercase{Uniformisation}}

\subsection{Uniformisation}

Step 1 (set up): Choose $q\geq\sup_{x\in\s}q(x)$. Build a DT chain $W$ and independent poisson process $N$ on the same $(\Omega,\cal{F},(\Pb_x)_{x\in\s})$ (in particular, choose waiting times of $N$ to be independent exponential RVs $R_1,R_2,\dots$ and define $N$ using them).  Set $W_t:=Z_{N_t}$ for all $t$ and $T_f:=\inf\{t:N_t=\infty\}$.

Step 2 (show that $W$ is a Markov chain with jump rates $\lambda$ and jump matrix $P$): Show that $W$ satisfies the defining properties in Section~\ref{sec:otherchains}. In particular, for the Markov property use something of the sort:
\begin{align*}\Pbz{\{W_t=x,W_{t+s}=y\}}&=\sum_{n=0}^\infty\sum_{m=0}^\infty \Pbz{\{Z_n=x,Z_{n+m}=y,N_t=n,N_{t+s}-N_t=m\}}\\
&=\sum_{n=0}^\infty\sum_{m=0}^\infty \Pbz{\{Z_n=x,Z_{n+m}=y\}}\Pbz{\{N_t=n,N_{t+s}-N_t=m\}}\\
&=\sum_{n=0}^\infty\sum_{m=0}^\infty \Pbz{\{Z_n=x\}}\Pbx{\{Z_{m}=y\}}\Pbz{\{N_t=n\}}\Pbz{\{N_s=m\}}\\
&=\sum_{n=0}^\infty\sum_{m=0}^\infty \Pbz{\{Z_n=x\}}\Pbx{\{Z_{m}=y\}}\Pbz{\{N_t=n\}}\Pbx{\{N_s=m\}}\\
&=\sum_{n=0}^\infty\sum_{m=0}^\infty \Pbz{\{Z_n=x,N_t=n\}}\Pbx{\{Z_{m}=y,N_s=m\}}\\
&=\Pbz{\{W_t=x\}}\Pbx{\{W_{s}=y\}}.\end{align*}

\fi
\ifdraft

\newpage

\section*{Foster-Lyapunov criteria}\label{flsec}

\subsubsection{The continuous-time case}
In contrast with discrete-time chains, continuous-time ones may explode. We begin with criterion often used in practice to rule this behaviour out
\begin{theorem}[Theorem 1.11 in \citep{Chen1991} or Theorem 2.1 in \citep{Meyn1993b}]\label{lyareg} Suppose that $u$ is a non-negative and norm-like function, $c>0$, $d\geq0$, and 
\begin{equation}\label{eq:lyaqreg}Qu(x)\leq cu(x)+d,\qquad\forall x\in\s.\end{equation}
Then $Q$ is regular.
\end{theorem}
\noindent For regular rate matrices, we can use the following to establish the existence of stationary distributions.
\begin{theorem}[Theorem 4.6 in \citep{Meyn1993b}]\label{lyafin} Suppose that $u:\s\to[0,\infty)$,  $f:\s\to[1,\infty)$, $c>0$, $d\geq0$, $A$ is a finite set, and
$$Qu(x)\leq d1_A(x)-cf(x),\qquad\forall x\in\s.$$
If $Q$ is regular, then each closed communicating class of $Q$ has an ergodic distribution and $\pi(f)\leq d/c$ for each stationary distribution $\pi$. 
\end{theorem}
\noindent Combining both of the above theorems (and exploiting that $u$ in the following is norm-like) we obtain the following third criterion.
\begin{corollary}\label{lyaboth}Suppose that $u$ is a non-negative and norm-like function, $c>0$, $d\geq0$, and 
$$Qu(x)\leq d-cu(x).$$
Then $Q$ is regular, each closed communicating class of $Q$ has an ergodic distribution, and $\pi(u)$ is finite for each stationary distribution $\pi$.
\end{corollary}
\noindent In the irreducible case, the above corollary can also be found in \citep{Chen1991} (the proof given in the reference is straightforward to adapt to the reducible case).
\begin{example}\label{schoglreg}Consider again Schl\"ogl's model with rate matrix $Q$ defined by \eqref{eq:bdqmatrx} and \eqref{eq:schoglbdr}. For any integer $d\geq 3$, fix $u(x):=x^{d-2}$ and note that
$$Qu(x)=g_{d-1}(x)-k_2(d-2)x^{d},$$
where $g_{d-1}$ is a polynomial of degree $d-1$. Thus, we have that
\begin{align*}Qu(x)&\leq \sup_{x\in\n}\left\{g_{d-1}(x)-\frac{k_2(d-2)}{2}x^{d}\right\}-\frac{k_2(d-2)}{2}x^{d}\\
&\leq \sup_{x\in\n}\left\{g_{d-1}(x)-\frac{k_2(d-2)}{2}x^{d}\right\}-\frac{k_2(d-2)}{2}x^{d-2}.\end{align*}
Corollary \ref{lyaboth} and Example \ref{schoglunique} then tell us that  $Q$ is regular, has a unique stationary distribution $\pi$, and the distribution's support if $\n$. The corollary further tells us that the first $d-2$ moments of $\pi$ are finite. Since $d$ was arbitrary, it follows that $\pi$ has all moments finite, a fact that useful when computing bounds on these moments as we do in Chapter \ref{mcmom}.
\end{example}
\begin{example}\label{togreg}Consider now the toggle switch model introduced in Example \ref{togintro} with rate matrix $Q$ defined by \eqref{eq:qmatrixsrn} and \eqref{toggle:prop}. Let $u(x):=(x_1+x_2)^d$, where $d$ is a positive integer, and note that
\begin{align*}Qu(x)&=(a_1(x)+a_3(x))((x_{1}+x_{2}+1)^d-(x_{1}+x_{2})^d) +(a_2(x)+a_4(x))((x_{1}+x_{2}-1)^d-(x_{1}+x_{2})^d) \\ 
&\leq (k_1+k_3)((x_{1}+x_{2}+1)^d-(x_{1}+x_{2})^d)+(a_2(x)+a_4(x))((x_{1}+x_{2}-1)^d-(x_{1}+x_{2})^d)\\ 
&\leq g_{d-1}(x_{1}+x_{2})-d(k_2x_{1}+k_4x_{2})(x_{1}+x_{2})^{d-1}\leq g_{d-1}(x_{1}+x_{2})-d\min\{k_2,k_4\}(x_1+x_2)^{d},\end{align*}
where $g_{d-1}$ is some polynomial of degree $d-1$. We then have that
$$Qu(x)\leq \sup_{x\in\n^2}\left(g_{d-1}(x_{1}+x_{2})-\frac{d\min\{k_2,k_4\}}{2}(x_1+x_2)^{d}\right) -\frac{d\min\{k_2,k_4\}}{2}u(x),$$
and it follows from Corollary \ref{lyaboth} that $Q$ is regular and has at least one stationary distribution $\pi$. Since each state $x$ is accessible from its horizontal and vertical neighbours (that is, from $x+(\pm1,0)$ and $x+(0,\pm1)$),  we have that $Q$ is irreducible and so Theorem \ref{doeblinc} tells us that $\pi$ is unique. Since the $d$ in the definition of $u$ was arbitrary, Corollary \ref{lyaboth} also tells us that $\pi$ has all moments finite. 
\end{example}
For the time of exit time $\tau$ from a domain $\cal{D}$ (see \eqref{eq:hitthec2}), we have the following exit-time analogue of Theorem \ref{lyareg}:
\begin{theorem}\label{lya1}Suppose that there exists a non-negative norm-like function $u$ and constants $c>0,d\geq0$ such that 
\begin{equation}\label{eq:lyar}Qu(x)\leq cu(x)+d\qquad \forall x\in\cal{D}.\end{equation}
Then $\Pbl{\{\tau\leq T_\infty\}}=1$. 
\end{theorem}
Having ruled out the possibility of exploding before exiting using the above, we can employ the following criterion to establish finiteness of mean exit time and to obtain quantitative bounds on unbounded integrals of $\mu_S$ and $\nu_S$ (defined in \eqref{eq:musc}--\eqref{eq:nusc}. Its proof is a straightforward modification of the proof of Theorem 4.3 in \citep{Meyn1993b}. 
\begin{theorem}\label{lya2}Suppose that $\Pbl{\{\tau\leq T_\infty\}}=1$ and that there exists a constant $c>0$, a non-negative $\lambda$-integrable function $u$, and a second function $f$ such that
\begin{equation}\label{eq:lyam}f(x)\geq c,\qquad Qu(x)\leq -f(x),\qquad \forall x\in\cal{D}.\end{equation}
Then 
$$\Ebl{\tau}\leq\lambda(u)/c,\qquad \mu_S(u)+\nu_S(f)\leq \lambda(u).$$
\end{theorem}
The same criterion as that in the above theorem can be found in \citep[Thrm.1.7]{Menshikov2014} under the further assumption that $\tau$ is $\Pb_\lambda$-almost surely finite. Finally, we are left with the continuous-time version of Theorem \ref{lya3d}.
\begin{theorem}\label{lya3}Suppose that Condition \ref{paths} is satisfied and that there exists constants $c>0$, $d\geq0$, and a non-negative, $\lambda$-integrable, and norm-like function $u$ such that
\begin{equation}\label{eq:lyadom}Qu(x)\leq d-c u(x)\qquad \forall x\in\cal{D},\end{equation}
Then $\Ebl{\tau}$, $\mu_S(u)$, and $\nu_S(u)$ are all finite.
\end{theorem}

\noindent The proof of the above is longer than those of the previous other two criteria; we do it in parts. Because $\tau=0$ on $\{X_0\not\in\cal{D}\}$, $u$ is $\lambda$-integrable, and
$$\Pb_\lambda=\sum_{x\in\s}\lambda(x)\Pb_x=\left(\sum_{x\in\cal{D}}\lambda(x)\Pb_x\right)+\left(\sum_{x\not\in\cal{D}}\lambda(x)\Pb_x\right),$$
it is enough to argue the result in the case that $\lambda$ has support contained in $\cal{D}$. The proof relies on another chain $\tilde{X}$ constructed by running Algorithm \ref{gilalg} with  the same $Z$, $\{\xi_n\}_{n=1}^\infty$, and $\{U_n\}_{n=1}^\infty$.  but with the rate matrix $\tilde{Q}:=(\tilde{q}(x,y))_{x,y\in\s}$ replacing $Q$, where
\begin{equation}\label{eq:qtilde}\tilde{q}(x,y):=\left\{\begin{array}{ll} q(x,y)&\text{if }x\in\cal{D}\\\lambda(y)-1_x(y)&\text{if }x\not\in\cal{D}\end{array}\right.,\qquad\forall x,y\in\s,\end{equation}
In other words, the $\tilde{X}$ is constructed as follows: initialise $\tilde{X}$ with law $\lambda$ and update it using the same rules as for those $X$ up until $\tilde{X}$ exits $\cal{D}$. At that point, wait a unit mean exponential amount of time and re-initialise $\tilde{X}$ with distribution $\lambda$. Sample the waiting time and the re-initialisation location independently of each other and of the chain's path up until (and including) the exit. Rinse and repeat. In what follows, we use $\tilde{Y}:=\{\tilde{Y}_n\}_{n\in\n}$, $\{\tilde{T}_n\}_{n\in\n}$, and $\tilde{T}_\infty$ denote the jump chain, jump times, and explosion time of this second chain $\tilde{X}$, respectively.

Inequality \eqref{eq:lyadom} implies that
\begin{equation}\label{eq:fdshjfsdj5}\tilde{Q}u(x)\leq (d\vee \lambda(u))-(1\wedge c)u(x),\qquad\forall x\in\s.\end{equation}
Corollary \ref{lyaboth} tells us that $\tilde{Q}$ is regular, has at least one stationary distribution, and that every stationary distribution $\pi$ of $\tilde{Q}$ satisfies
\begin{equation}\label{eq:pifin}\pi(u)<\infty.\end{equation}
If the mean exit time is finite, then
$$\pi_\tau(x):=\frac{\mu_S(x)+\nu_S(x)}{1+\Ebl{\tau}},\qquad\forall x\in\s,$$
is a probability distribution on $(\s,\tws)$ (recall \eqref{eq:numass} and note that the exit times is finite only if it is less than the explosion time). Using \eqref{eq:nueqsm}--\eqref{eq:mueqsm}, we have that
$$\pi_\tau\tilde{Q}(x)= \sum_{z\in\cal{D}}\pi_\tau(z)q(z,x)+\sum_{z\not\in\cal{D}}\pi_\tau(z)(\lambda(x)-1_z(x))=\frac{\nu_S Q(x)+\lambda(x)-\mu_S(x)}{1+\Ebl{\tau}}=0,$$
for each $x\in\s$. For this reason, Theorem \ref{Qstateq} then tells us that $\pi_\tau$ is a stationary distribution of $\tilde{Q}$. Finiteness of $\mu_S(u)$ and $\nu_S(u)$  then follows from \eqref{eq:pifin}. Alternatively, one can argue finiteness of these integrals (assuming that the mean exit time is finite) taking an approach similar to that in the proof of Theorem \ref{lya2}. In other words, all that remains to be shown is that $\tau$ has finite mean. The proof of this fact can be found in Appendix \ref{Markovpreproofs}. For it we will need the following two lemmas.
\begin{lemma}[Strong Markov Property and the mean hitting time equations]\label{dsnajdas}Let $\eta$ be any stopping time (as in \eqref{eq:stopdc}) of $X$ and
$$\tau_A^\eta:=\inf\{\eta\leq t<T_\infty:X_t\in A\}$$
be the first time after $\eta$ that the chain $X$ hits the set $A\subseteq \s$, where we are adhering to our convention that $\inf\emptyset=\infty$. 

\begin{enumerate}[label=(\roman*)]
\item Conditioned on $X_{\eta}$, the amount of time $\tau_A^\eta-\eta$ that elapses between $\eta$ and the first visit to $A$ after $\eta$ has the same law as the usual hitting time $\tau_A$ has if the chain starts from $X_{\eta}$:
\begin{equation}\label{eq:strmk2}\Pbl{\{\tau_A^\eta-\eta\leq t,\eta<\infty,X_\eta=x\}}=\Pbl{\{\eta<\infty,X_\eta=x\}}\Pbx{\{\tau_A\leq t\}},\end{equation}
for each $x\in\s$ and  $t\in[0,\infty)$.
\item If $\varphi_y$ is the first entrance time of a state $y$ (as in \eqref{eq:Ry}) such that $q_y>0$, then
\begin{equation}\label{eq:reteqs}\Eb_y[\varphi_y]=\frac{1}{q_y}+\sum_{x\neq y}\frac{q(y,x)}{q_y}\Ebx{\tau_y}.\end{equation}
\item If $Q$ has a stationary distribution with support on a closed communicating class $A$ of $Q$, then the mean time it takes the chain to transition between two states in $A$ is finite:
$$\Ebx{\tau_y}<\infty,\quad \forall x,y\in A.$$
\end{enumerate}
\end{lemma}
Equation \eqref{eq:strmk2} is a consequence of the strong Markov property (Theorem \ref{thstrmk}) while equations \eqref{eq:reteqs} are a minor variation of the usual mean hitting time equations \citep[Theorem 3.3.3]{Norris1997}. 
\begin{lemma}[Class structure of $\tilde{Q}$]\label{cstruct} Suppose that Condition \ref{paths} is satisfied. Let  $\cal{C}$ be the union of this support and the set of states accessible for $\tilde{Q}$ from the support:
$$\cal{C}:=\supp{\lambda}\cup\{x\in\s:\exists z\in\supp{\lambda}, \smallskip z\to x\}.$$
$\cal{C}$ is the only closed communicating class of $\tilde{Q}$. If  $\eta$ be a $\Pb_\lambda$-almost surely finite stopping time of $\tilde{X}$ (as in \eqref{eq:stopdc} with $\tilde{X}$ replacing $X$), then 
\begin{equation}
\label{eq:??}\Pb_\lambda(\{\tilde{X}_{\eta}\in\cal{C},\eta<\infty\})=1.
\end{equation}
\end{lemma}
\subsection{Concluding remarks}\label{Markovpreconc}
This chapter consisted of a relatively self-contained account of the Markov chain theory pertinent to the three questions posed in this thesis' introduction (Q1--3 in Chapter \ref{introthe}). One may wonder about Markov processes in general. Is the theory much harder than in the case countable state space case? Do more general processes exhibit a wider range of behaviours? Etc. On this matter, David Williams states in his elegant book \citep[p.233]{Rogers2000a} (in our experience, David Williams is \emph{always} elegant)
\begin{quotation}
``It is interesting that the full weight of Ray theory (or something like it) is needed to handle Markov chains.''
\end{quotation}
seemingly implying that the answer to these questions is ``no''. Indeed, the various analytical characterisations given in this chapter all have analogues in the general state space case, see \citep{Echeverra1982,Ethier1986,Kurtz1998} for instance. So do things like the Doeblin decomposition of the state space, its relationship with the stationary distributions, the ideas of recurrence, and the Foster-Lyapunov criteria, see \citep{Meyn1992,Meyn1993a,Meyn1993b,Meyn2009}, references therein, and references thereof.

However, we are being rather deceitful in our above use of Williams' quote, the reason being that we only consider \emph{Feller minimal} chains with \emph{totally stable rate matrices}. Non-minimal chains are the ones that have a life past an explosion: chains that come back from infinity in a manner that respects the Markov property. Chains whose rate matrices are not totally stable are chains that leave certain states immediately after entering them ($q_x$ will be infinity for such a state $x$). It is only in the case of chains that violate these caveats that one requires the same type of technical machinery as that required for general Markov processes. Indeed, Williams' has the following to say later on in his book \citep[p.285]{Rogers2000a}:
\begin{quotation}``In connection with question (25.1)(i), it is necessary to understand Chung's
statement that the then-prevailing assumptions covered only `trite' chains. This
relates to the following fact (discussed briefly and illustrated in a moment, and
explained fully later in this chapter). If a transition matrix function [$\{p_t(x,y)\}$ in our notation] on
a countable set [$\s$ in our notation] has the FD property relative to the discrete topology of $\s$, then
the associated chain $X$ is totally stable and Feller minimal. Though nearly all
chains that can serve as models for real-world phenomena are totally stable
and Feller minimal, such chains are `trite' from a pure-mathematical standpoint. (We shall attempt to provide a strong justification for the study of `non-trite'
chains later!).''\end{quotation}
The reason for the quotation marks around the word ``trite" is that Williams is quoting Chung who described these chains as trite. I have focused on these `trite' chains because, as Williams mentions above, there is a dearth of `non-trite' chains in applications (although I have been told that such instantaneous jump processes do feature in certain financial contexts) and I am an engineering student after all.

What we don't think is a trite matter is analysing quantitatively (in the sense of Chapter \ref{introthe}) Feller-minimal and totally stable chains. Our experience indicates the opposite: it is a difficult task hounded by the curse of dimensionality and plagued by theoretical nuances and numerical difficulties. The remaining chapters of this thesis aim to contribute to ongoing efforts to resolve, or at least mitigate, these issues by studying various numerical schemes with inbuilt error control that enable this analysis and dealing with theoretical subtleties that arise in these studies.

We pause here to point out that, although the starting point of our analysis of the continuous-time case was the construction of the chain via the Gillespie Algorithm (Algorithm~\ref{gilalg}), the results of this chapter (and of Chapters \ref{fspchap}--\ref{dists}) apply to all Feller-minimal Markov chains with rate matrix $Q$ regardless of their particular construction. The reason being that all these chains possess the same statistics (including time-varying law, exit distributions, and occupation measures) that our chain does (in the same spirit as Remark \ref{uniquechainlaw}), see \citep{Chung1967,Doob1942,Doob1945,Feller1940,Freedman1983,Norris1997}. These other Feller minimal chains include those generated through other common simulation algorithms such as the first reaction method \citep{Gillespie1976} or the method of Gibson and Bruck \citep{Gibson2000}. 

We conclude this chapter by discussing what we perceive (in view of the content of the ensuing chapters) as the biggest hole in our treatment of Markov chains: the Foster-Lyapunov criteria for exit times presented in Section \ref{flsec} do not allow us to establish integrability with respect to $\mu$ and $\nu$ of test functions that are unbounded in their time argument (notice that these results are all phrased in terms of the space marginals $\mu_S$ and $\nu_S$ of $\mu$ and $\nu$). In particular, they do not allow us to establish the finiteness of the moments of the exit time (aside from its mean), something that is important when using semidefinite programming approaches of the type discussed in Section \ref{mcmom} and \citep{Helmes2001,Lasserre2006,Eriksson2011} to compute quantitative bounds on these moments. To address this issue, we would start with the criteria in \citep{Aspandiiarov1999,Menshikov1996,Menshikov2014,Lamperti1963} and the material in \citep[Sec.I.11]{Chung1967}.

\fi

%

\clearpage


\pagestyle{premain}\sectionmark{\MakeUppercase{Bibliography}}

\bibliographystyle{plainnat}
\bibliography{./Biblio}

\begin{thebibliography}{58}
\providecommand{\natexlab}[1]{#1}
\providecommand{\url}[1]{\texttt{#1}}
\expandafter\ifx\csname urlstyle\endcsname\relax
  \providecommand{\doi}[1]{doi: #1}\else
  \providecommand{\doi}{doi: \begingroup \urlstyle{rm}\Url}\fi

\bibitem[Anderson(1991)]{Anderson1991}
W.~J. Anderson.
\newblock \emph{{Continuous-time Markov chains: an applications-oritented
  approach}}.
\newblock Springer-Verlag New York, 1991.

\bibitem[Asmussen(2003)]{Asmussen2003}
S.~Asmussen.
\newblock \emph{{Applied probability and queues}}.
\newblock Springer-Verlag New York, 2003.

\bibitem[Bell(1945)]{Bell1945}
E.~T. Bell.
\newblock \emph{{The development of mathematics}}.
\newblock McGraw-Hill Book Company, Inc., New York, second edition, 1945.

\bibitem[Bell(1952)]{Bell1952}
E.~T. Bell.
\newblock \emph{{Mathematics: queen and servant of science}}.
\newblock G. Bell \& Sons, Ltd., London, 1952.

\bibitem[Billingsley(1995)]{Billingsley1995}
P.~Billingsley.
\newblock \emph{{Probability and measure}}.
\newblock Wiley, 1995.
\newblock ISBN 0471007102.

\bibitem[Blumenthal and Getoor(2007)]{Blumenthal2007}
R.~M. Blumenthal and R.~K. Getoor.
\newblock \emph{{Markov processes and potential theory}}.
\newblock Dover Publications, Mineola, New York, 2007.

\bibitem[Chen(1986)]{Chen1986}
M.~F. Chen.
\newblock {Coupling for jump processes}.
\newblock \emph{Acta Math. Sin.}, 2\penalty0 (2):\penalty0 123--136, 1986.

\bibitem[Chung(1967)]{Chung1967}
K.~L. Chung.
\newblock \emph{{Markov chains with stationary transition probabilities}}.
\newblock Springer, second edition, 1967.

\bibitem[Croft(1957)]{Croft1957}
H.~T. Croft.
\newblock {A question of limits}.
\newblock \emph{Eureka}, 20:\penalty0 11--13, 1957.

\bibitem[Doob(1942)]{Doob1942}
J.~L. Doob.
\newblock {Topics in the theory of Markoff chains}.
\newblock \emph{Trans. Amer. Math. Soc.}, 52\penalty0 (1):\penalty0 37--64,
  1942.

\bibitem[Doob(1945)]{Doob1945}
J.~L. Doob.
\newblock {Markoff chains--denumerable case}.
\newblock \emph{Trans. Amer. Math. Soc.}, 58\penalty0 (3):\penalty0 455--473,
  1945.

\bibitem[Ethier and Kurtz(1986)]{Ethier1986}
S.~N. Ethier and T.~G. Kurtz.
\newblock \emph{{Markov processes: characterization and convergence}}.
\newblock John Wiley \& Sons, 1986.

\bibitem[Feller(1940)]{Feller1940}
W.~Feller.
\newblock {On the integro-differential equations of purely discontinuous
  Markoff processes}.
\newblock \emph{Trans. Amer. Math. Soc.}, 48\penalty0 (3):\penalty0 488--515,
  1940.

\bibitem[Feller(1971)]{Feller1971}
W.~Feller.
\newblock \emph{{An Introduction to Probability Theory and Its Applications:
  Volume 2}}.
\newblock John Wiley \& Sons, second edition, 1971.

\bibitem[Foster(1952)]{Foster1952}
F.~G. Foster.
\newblock {On Markov chains with an enumerable infinity of states}.
\newblock \emph{Math. Proc. Cambridge Philos. Soc.}, 48\penalty0 (4):\penalty0
  587--591, 1952.

\bibitem[Foster(1953)]{Foster1953}
F.~G. Foster.
\newblock {On the stochastic matrices associated with certain queuing
  processes}.
\newblock \emph{Ann. Math. Stat.}, 24\penalty0 (3):\penalty0 355--360, 1953.

\bibitem[Freedman(1983)]{Freedman1983}
D.~Freedman.
\newblock \emph{{Markov chains}}.
\newblock Springer-Velag, first edition, 1983.

\bibitem[Gillespie(1976)]{Gillespie1976}
D.~T. Gillespie.
\newblock {A general method for numerically simulating the stochastic time
  evolution of coupled chemical reactions}.
\newblock \emph{J. Comput. Phys.}, 22\penalty0 (4):\penalty0 403--434, 1976.

\bibitem[Gillespie(1977)]{Gillespie1977}
D.~T. Gillespie.
\newblock {Exact stochastic simulation of coupled chemical reactions}.
\newblock \emph{J. Phys. Chem. A}, 81\penalty0 (25):\penalty0 2340--2361, 1977.

\bibitem[Helmes(2002)]{Helmes2002}
K.~Helmes.
\newblock {Numerical methods for optimal stopping using linear and non-linear
  programming}.
\newblock In \emph{Stochastic Theory and Control}, pages 185--203. Springer,
  2002.

\bibitem[Helmes et~al.(2001)Helmes, R{\"{o}}hl, and Stockbridge]{Helmes2001}
K.~Helmes, S.~R{\"{o}}hl, and R.~H. Stockbridge.
\newblock {Computing moments of the exit time distribution for Markov processes
  by linear programming}.
\newblock \emph{Oper. Res.}, 49\penalty0 (4):\penalty0 516--530, 2001.

\bibitem[Kallenberg(2001)]{Kallenberg2001}
O.~Kallenberg.
\newblock \emph{{Foundations of Modern Probability}}.
\newblock Springer-Verlag, second edition, 2001.

\bibitem[Kendall(1950)]{Kendall1950}
D.~G. Kendall.
\newblock {An artificial realization of a simple ``birth-and-death" process}.
\newblock \emph{J. R. Stat. Soc. Ser. B Stat. Methodol.}, 12\penalty0
  (1):\penalty0 116--119, 1950.

\bibitem[Kendall(1951{\natexlab{a}})]{Kendall1951}
D.~G. Kendall.
\newblock {Some problems in the theory of queues}.
\newblock \emph{J. R. Stat. Soc. Ser. B. Stat. Methodol.}, 13\penalty0
  (2):\penalty0 151--185, 1951{\natexlab{a}}.

\bibitem[Kendall(1951{\natexlab{b}})]{Kendall1951a}
D.~G. Kendall.
\newblock {On non-dissipative Markoff chains with an enumerable infinity of
  states}.
\newblock \emph{Math. Proc. Cambridge Philos. Soc.}, 47\penalty0 (3):\penalty0
  633--634, 1951{\natexlab{b}}.

\bibitem[Kendall(1959)]{Kendall1959}
D.~G. Kendall.
\newblock {Unitary dilations of Markov transition operators and the
  corresponding integral representations for transition-probability matrices}.
\newblock In U.~Grenander, editor, \emph{Probability and statistics; the Harald
  Cram{\'{e}}r volume}. Stockholm: Almqvist \& Wiksell; New York: Wiley, 1959.

\bibitem[Kendall and Reuter(1957)]{Kendall1957}
D.~G. Kendall and G.~E.~H. Reuter.
\newblock {The calculation of the ergodic projection for Markov chains and
  processes with a countable infinity of states}.
\newblock \emph{Acta Math.}, 97:\penalty0 103--144, 1957.

\bibitem[Kingman(1961)]{Kingman1961}
J.~F.~C. Kingman.
\newblock {The ergodic behaviour of random walks}.
\newblock 48\penalty0 (3/4):\penalty0 391--396, 1961.

\bibitem[Kingman(1963)]{Kingman1963}
J.~F.~C. Kingman.
\newblock {Ergodic properties of continuous-time Markov processes and their
  discrete skeletons}.
\newblock \emph{Proc. Lond. Math. Soc.}, 13\penalty0 (3):\penalty0 593--604,
  1963.

\bibitem[Kuntz(2017)]{Kuntzthe}
J.~Kuntz.
\newblock \emph{{Deterministic approximation schemes with computable errors for
  the distributions of Markov chains}}.
\newblock PhD thesis, Imperial College London, 2017.

\bibitem[Kuntz et~al.(2019)Kuntz, Thomas, Stan, and Barahona]{Kuntz2019}
J.~Kuntz, P.~Thomas, G.~B. Stan, and M.~Barahona.
\newblock {The exit time finite state projection scheme: bounding exit
  distributions and occupation measures of continuous-time Markov chains}.
\newblock \emph{SIAM J. Sci. Comput.}, 41\penalty0 (2):\penalty0 A748--A769,
  2019.

\bibitem[Lasserre and Prieto-Rumeau(2004)]{Lasserre2004}
J.~B. Lasserre and T.~Prieto-Rumeau.
\newblock {SDP vs. LP relaxations for the moment approach in some performance
  evaluation problems}.
\newblock \emph{Stoch. Models}, 20\penalty0 (4):\penalty0 439--456, 2004.

\bibitem[Lasserre et~al.(2006)Lasserre, Prieto-Rumeau, and
  Zervos]{Lasserre2006}
J.~B. Lasserre, T.~Prieto-Rumeau, and M.~Zervos.
\newblock {Pricing a class of exotic options via moments and SDP relaxations}.
\newblock \emph{Math. Finance}, 16\penalty0 (3):\penalty0 469--494, 2006.

\bibitem[Mauldon(1957)]{Mauldon1957}
J.~G. Mauldon.
\newblock {On non-dissipative Markov chains}.
\newblock \emph{Math. Proc. Cambridge Philos. Soc.}, 53\penalty0 (4):\penalty0
  825--835, 1957.

\bibitem[Mertens et~al.(1978)Mertens, Samuel-Cahn, and Zamir]{Mertens1978}
J.-F. Mertens, E.~Samuel-Cahn, and S.~Zamir.
\newblock {Necessary and sufficient conditions for recurrence and transience of
  Markov chains, in terms of inequalities}.
\newblock \emph{J. Appl. Probab.}, 15\penalty0 (4):\penalty0 848--851, 1978.

\bibitem[Meyer(1976)]{Meyer1976}
P.~A. Meyer.
\newblock {Un cours sur les int{\'{e}}grales stochastiques}.
\newblock \emph{S{\'{e}}minaire de probabilit{\'{e}}s de Strasbourg},
  10:\penalty0 245--400, 1976.

\bibitem[Meyn and Tweedie(1992)]{Meyn1992}
S.~P. Meyn and R.~L. Tweedie.
\newblock {Stability of Markovian processes I: criteria for discrete-time
  chains}.
\newblock \emph{Adv. in Appl. Probab.}, 24\penalty0 (3):\penalty0 542--574,
  1992.

\bibitem[Meyn and Tweedie(1993{\natexlab{a}})]{Meyn1993a}
S.~P. Meyn and R.~L. Tweedie.
\newblock {Stability of Markovian processes II: continuous-time processes and
  sampled chains}.
\newblock \emph{Adv. in Appl. Probab.}, 25\penalty0 (3):\penalty0 487--517,
  1993{\natexlab{a}}.

\bibitem[Meyn and Tweedie(1993{\natexlab{b}})]{Meyn1993b}
S.~P. Meyn and R.~L. Tweedie.
\newblock {Stability of Markovian processes III: Foster-Lyapunov criteria for
  continuous-time processes}.
\newblock \emph{Adv. Appl. Probab.}, 25\penalty0 (3):\penalty0 518--548,
  1993{\natexlab{b}}.

\bibitem[Meyn and Tweedie(2009)]{Meyn2009}
S.~P. Meyn and R.~L. Tweedie.
\newblock \emph{{Markov chains and stochastic stability}}.
\newblock Cambridge University Press, second edition, 2009.

\bibitem[Miller(1963)]{Miller1963}
R.~G.~Jr. Miller.
\newblock {Stationarity equations in continuous time Markov chains}.
\newblock \emph{Trans. Amer. Math. Soc.}, 109\penalty0 (1):\penalty0 35--44,
  1963.
\newblock ISSN 00029947.

\bibitem[Norris(1997)]{Norris1997}
J.~R. Norris.
\newblock \emph{{Markov chains}}.
\newblock Cambridge University Press, 1997.

\bibitem[Nummelin(1978)]{Nummelin1978b}
E.~Nummelin.
\newblock {A splitting technique for Harris recurrent Markov chains}.
\newblock \emph{Zeitschrift f{\"{u}}r Wahrscheinlichkeitstheorie und Verwandte
  Gebiete}, 43\penalty0 (4):\penalty0 309--318, 1978.

\bibitem[Nummelin and Tuominen(1982)]{Nummelin1982}
E.~Nummelin and P.~Tuominen.
\newblock {Geometric ergodicity of Harris recurrent Markov chains with
  applications to renewal theory}.
\newblock \emph{Stochastic Process. Appl.}, 12\penalty0 (2):\penalty0 187--202,
  1982.

\bibitem[Nummelin and Tweedie(1978)]{Nummelin1978a}
E.~Nummelin and R.~L. Tweedie.
\newblock {Geometric ergodicity and R-positivity for general Markov chains}.
\newblock \emph{Ann. Probab.}, 6\penalty0 (3):\penalty0 404--420, 1978.

\bibitem[{\O}ksendal(2003)]{Oksendal2003}
B.~{\O}ksendal.
\newblock \emph{{Stochastic differential equations: an introduction with
  applications}}.
\newblock Universitext. Springer-Verlag, Berlin, Heidelberg, New York, fifth
  edition, 2003.

\bibitem[Pakes(1969)]{Pakes1969}
A.~G. Pakes.
\newblock {Some conditions for ergodicity and recurrence of Markov chains}.
\newblock \emph{Oper. Res.}, 17\penalty0 (6):\penalty0 1058--1061, 1969.

\bibitem[Popov(1977)]{Popov1977}
N.~N. Popov.
\newblock {Geometric ergodicity conditions for countable Markov chains}.
\newblock \emph{Dokl. Akad. Nauk SSSR}, 234\penalty0 (2):\penalty0 316--319,
  1977.

\bibitem[Rogers and Williams(2000{\natexlab{a}})]{Rogers2000a}
L.~C.~G. Rogers and D.~Williams.
\newblock \emph{{Diffusions, Markov processes and martingales: Volume 1.
  Foundations}}.
\newblock Cambridge University Press, second edition, 2000{\natexlab{a}}.

\bibitem[Rogers and Williams(2000{\natexlab{b}})]{Rogers2000b}
L.~C.~G. Rogers and D.~Williams.
\newblock \emph{{Diffusions, Markov processes and martingales: Volume 2. It\^o
  calculus}}.
\newblock Cambridge University Press, second edition, 2000{\natexlab{b}}.

\bibitem[Spieksma(2015)]{Spieksma2015}
F.~M. Spieksma.
\newblock {Countable state Markov processes: non-explosiveness and moment
  function}.
\newblock \emph{Probab. Eng. Inf. Sci.}, 29\penalty0 (04):\penalty0 623--637,
  2015.

\bibitem[Syski(1992)]{Syski1992}
R.~Syski.
\newblock \emph{{Passage times for Markov chains}}.
\newblock IOS Press, 1992.

\bibitem[Tao(2011)]{Tao2011}
T.~Tao.
\newblock \emph{{An introduction to measure theory}}.
\newblock American Mathematical Society, 2011.

\bibitem[Tweedie(1975{\natexlab{a}})]{Tweedie1975}
R.~L. Tweedie.
\newblock {Sufficient conditions for ergodicity and recurrence of Markov chains
  on a general state space}.
\newblock \emph{Stochastic Process. Appl.}, 3\penalty0 (4):\penalty0 385--403,
  1975{\natexlab{a}}.

\bibitem[Tweedie(1975{\natexlab{b}})]{Tweedie1975b}
R.~L. Tweedie.
\newblock {Sufficient conditions for regularity, recurrence and ergodicity of
  Markov processes}.
\newblock \emph{Math. Proc. Cambridge Philos. Soc.}, 78\penalty0 (1):\penalty0
  125--136, 1975{\natexlab{b}}.

\bibitem[Tweedie(1981)]{Tweedie1981}
R.~L. Tweedie.
\newblock {Criteria for ergodicity, exponential ergodicity and strong
  ergodicity of Markov processes}.
\newblock \emph{J. Appl. Probab.}, 18\penalty0 (01):\penalty0 122--130, 1981.

\bibitem[Tweedie(1983)]{Tweedie1983a}
R.~L. Tweedie.
\newblock {Criteria for rates of convergence of Markov chains, with application
  to queueing and storage theory}.
\newblock In J.~F.~C. Kingman and G.~E.~H. Reuter, editors, \emph{Probability,
  Statistics and Analysis}, pages 260--276. Cambridge University Press,
  Cambridge, 1983.

\bibitem[Williams(1991)]{Williams1991}
D.~Williams.
\newblock \emph{{Probability with martingales}}.
\newblock Cambridge University Press, 1991.

\end{thebibliography}

\addcontentsline{toc}{section}{\protect\numberline{}Bibliography}


\glsaddallunused
\clearpage

\pagestyle{premain}\sectionmark{\MakeUppercase{Symbol index}}

\pagestyle{premain}\sectionmark{\MakeUppercase{Symbol index}}
\addcontentsline{toc}{section}{\protect\numberline{}Symbol index}
\sectionmark{\MakeUppercase{Symbol index}}

\printglossary[title=\MakeUppercase{Symbol index},type=gen]
\sectionmark{\MakeUppercase{Symbol index}}

\vspace{-10pt}

\noindent{\large\textbf{Notation for discrete-time and continuous-time chains}}

\vspace{-30pt}
\printglossary[type=both]
\vspace{-20pt}

\noindent{\large\textbf{Notation for discrete-time chains}}

\vspace{-30pt}
%
\printglossary[type=dt]

\vspace{-20pt}

\noindent{\large\textbf{Notation for continuous-time chains}}

\vspace{-30pt}

\printglossary[type=ct]

\clearpage

\renewcommand\indexname{\MakeUppercase{Subject index}}
\addcontentsline{toc}{section}{\protect\numberline{}Subject index}\pagestyle{premain}
\sectionmark{\MakeUppercase{Subject index}}
\printindex\pagestyle{premain}\sectionmark{\MakeUppercase{Subject index}}
\end{document}